\documentclass[a4paper,12pt,times,numbered,print,index]{Classes/PhDThesisPSnPDF}

\usepackage[cmyk]{xcolor}

\input{Preamble/preamble}

\ifdefineAbstract
\fi

\ifdefineChapter
\fi

\makeatletter

\renewenvironment{cases}[1][l]{\matrix@check\cases\env@cases{#1}}{\endarray\right.}
\def\env@cases#1{%
  \let\@ifnextchar\new@ifnextchar
  \left\lbrace\def\arraystretch{1.2}%
  \array{@{}#1@{\quad}l@{}}}

   \def\marginparright{\@mparswitchfalse}
   \def\marginparoutside{\@mparswitchtrue}
\makeatother

\begin{document}
\mainmatter
\marginparright

\FINALVERSION %KILL ALL REMARK-MARKUPS; other numbering for table of contents
% \ADVISORVERSION %KILL GOLDSTARS, STUDY-REMARKS and BLUE REMARKS, LEAVES RED REMARKS

%\pagenumbering{roman}

\begin{centering}
  \thispagestyle{empty}
   
    \vspace*{0.6in}
    \bgroup
    \Huge\bfseries Contributions to the Problems of Recognizing and Coloring Gammoids \par
    \egroup
    \vspace{0.5in}
    \bgroup
    \normalsize Dissertation \\[0.1in]
    zur Erlangung des Grades\\[0.1in]
    \textsc{Doktor der Naturwissenschaften}\\[0.1in]
    \textsc{(Dr. rer. nat.)}\\[0.1in]
    \egroup
    \vspace{0.5in}
    \bgroup
    \normalsize an der Fakultät für \\[0.1in]
    Mathematik und Informatik der \\[0.1in] FernUniversität in Hagen\\[0.1in]
    \egroup
    \vspace{0.5in}
    \bgroup
    \normalsize verfasst von Herrn\\[0.1in]
    \textsc{Diplommathematiker}\\[0.1in]
    \Large Immanuel Albrecht\\[0.1in]
    \normalsize aus Dresden\\[0.1in]
    \egroup
    \vspace{0.7in}
    Hagen, 2018\\[0.1in]
\end{centering}

\cleardoublepage
\begin{centering}
  \thispagestyle{empty}
   
    \vspace*{1.2in}
    
    \begin{flushleft}
    \textit{``That may be impossible, sir.'' \\ ~~~~ --- Data. \hfill{~}}
    \textit{``Things are only impossible until they're not!'' ~~~~~ \\ \hfill{~} --- Jean-Luc Picard.}
    \end{flushleft}
  
\end{centering}

% -*- root: ../thesis.tex -*-

\cleardoublepage
\thispagestyle{empty}
\phantomsection
\addcontentsline{toc}{chapter}{Abstract}

\bigskip
\newcommand{\mytwocols}{\vspace{2mm}\begin{parcolumns}[colwidths={1=.42\textwidth,2=.48\textwidth},nofirstindent,sloppy]{2}}
\mytwocols
\colchunk{\ignorespaces
\Large\ignorespaces \textbf{Abstract}}
\colchunk{\Large\textbf{Zusammenfassung}}
\end{parcolumns}
\mytwocols
\colchunk{This work provides a thorough introduction to the field of gammoids
and presents new results that are considered helpful for solving the problems of recognizing and coloring gammoids.
}
\colchunk{Diese Arbeit gibt eine gründliche Einführung in das Gebiet der Gammoide,
und legt neue Ergebnisse dar, die für die Probleme des Erkennens und des Färbens von Gammoiden dienlich sind.  }
\end{parcolumns}
\mytwocols
\colchunk{Matroids are set systems that generalize the concept of linear independence between sets of rows of a matrix over a field. Gammoids are those matroids that may be represented by directed graphs where the corresponding
independence is modeled as the existence of certain families of pair-wise vertex disjoint paths.
The seminal papers in gammoid theory have been written by J.H.~Mason \cite{M72}, A.W.~Ingleton and M.J.~Piff \cite{IP73}.
Natural applications of gammoids can be found within the realms of connectivity of both directed 
and undirected graphs.}
\colchunk{Matroide sind Mengensysteme, welche den Begriff der linearen Unabhängigkeit zwischen Mengen von Zeilen einer Matrix über einem Körper verallgemeinern. Gammoide sind jene Matroide, welche so durch gerichtete Graphen dargestellt werden können, dass ihre zugehörige Unabhängigkeit durch die Existenz gewisser Familien von paarweise knotendisjunkten Pfaden beschreibbar ist. Die grundlegenden Arbeiten zur Theorie der Gammoide wurden von J.H.~Mason \cite{M72}, A.W.~Ingleton und M.J.~Piff \cite{IP73}
verfasst. Natürliche Anwendung finden Gammoide im Bereich des Zusammenhangs sowohl von gerichteten als auch von ungerichteten
Graphen.}
\end{parcolumns}
\mytwocols
\colchunk{In this work, we introduce our concept of the complexity of a gammoid, which may
be used to define subclasses of the class of gammoids that inherit the most notable
properties of the class of gammoids: being closed under minors, duality, and direct sums.
Furthermore, we provide a comprehensive method for deciding
whether a given matroid is a gammoid. We give a new procedure for obtaining 
an $\Rbm$-matrix, that represents a gammoid given by the means of a directed graph,
which avoids using power series.
We present the first purely combinatorial way of obtaining orientations of gammoids.
We prove that every lattice path matroid is $3$-colorable.
}
\colchunk{In dieser Arbeit führen wir unseren Begriff der Komplexität eines Gammoids ein,
welcher verwendet werden kann, um Unterklassen der Klasse der Gammoide zu definieren.
Diese Unterklassen erben  die bedeutendsten Eigenschaften der Klasse der Gammoide,
nämlich die Abgeschlossenheit unter Minoren, unter Dualität sowie unter direkten Summen.
Des Weiteren stellen wir eine umfassende Methode bereit, mit der entschieden werden kann,
ob ein gegebenes Matroid ein Gammoid ist. 
Wir geben eine neue Vorgehensweise an,
die eine $\Rbm$-Matrix-Darstellung eines Gammoids, welches mittels eines gerichteten Graphen
gegeben ist, findet, ohne auf Potenzreihen zurückzugreifen. 
Wir stellen das erste rein kombinatorische Verfahren vor, dass Orientierungen eines Gammoids liefert.
Wir zeigen, dass alle Lattice-Path-Matroide $3$-färbbar sind.
}
\end{parcolumns}
\mytwocols
\colchunk{In Chapter~1 we give a brief introduction to matroid theory:
we present axiomatizations of matroids most relevant to this work,
the concepts of minors and duality as well as representability over fields
and properties of extensions.
The same chapter also contains a brief introduction to the theory of transversals,
including the Theorems of Hall, Rado, Ore, and Perfect,
and an introduction to transversal matroids. 
Also, we provide a short introduction to directed graphs, we introduce
the concept of a routing in a directed graph and we close the chapter with Menger's Theorem
and its consequences.
}
\colchunk{Im ersten Kapitel geben wir eine kurze Einführung in die Matroidtheorie:
Wir stellen die Axiomatisierungen von Matroiden, welche am besten zu dieser Arbeit passen,
den Minorenbegriff, die Dualität sowie die Darstellbarkeit über Körpern und Eigenschaften von Erweiterungen vor.
Das Kapitel enthält außerdem eine Einführung in die Transversaltheorie, welche die Sätze von Hall, Rado, Ore und Perfect
sowie eine Einführung der Transversalmatroide umfasst.
Außerdem stellen wir gerichtete Graphen kurz vor, erläutern den Begriff des Routings in gerichteten Graphen und
beenden das Kapitel mit dem Satz von Menger sowie Schlussfolgerungen aus diesem.
} 
\end{parcolumns}
\mytwocols
\colchunk{In Chapter~2 we define gammoids as matroids that may be obtained from routings in directed graphs.
We explore the properties of their directed graph representations and along that we define our notion of a duality respecting representation
which correlates the duality-like notion of opposite directed graphs with the notion of duality with respect to gammoids.
Furthermore, we introduce our three complexity measures for gammoids that yield
subclasses of gammoids which are closed under minors and duality.
We present Mason's $\alpha$-criterion for strict gammoids, and we examine the properties of
strict gammoids and transversal matroids. We analyze the problem of recognizing gammoids,
we develop the notion of an $\alpha$-violation,
and we present our best approach for deciding instances of the recognition problem.
At the end of Chapter~2, we present our method for determining an $\Rbm$-matrix representing a gammoid from
a given representation in terms of a directed graph.
}
\colchunk{Im zweiten Kapitel definieren wir Gammoide als Matroide, die durch Routings in gerichteten Graphen beschrieben werden können. Wir untersuchen die Eigenschaften ihrer Darstellungen mit gerichteten Graphen und definieren dabei unseren Begriff einer dualitätsachtenden Darstellung, welche den dualitätsnahen Begriff gegenläufig gerichteter Graphen und den Dualitätsbegriff von Gammoiden in Wechselbeziehung stellt. Weiterhin stellen wir drei Komplexitätsmaße für Gammoide vor, welche Unterklassen von Gammoiden liefern, die unter Minoren und Dualität abgeschlossen sind. Wir stellen Mason’s $\alpha$-Kriterium für strikte Gammoide vor, und wir untersuchen die Eigenschaften von strikten Gammoiden und von Transversalmatroiden. Wir analysieren das Problem des Erkennens von Gammoiden, wir entwickeln den Begriff der $\alpha$-Verletzung
und wir präsentieren unseren besten Ansatz zum Entscheiden, ob ein Matroid ein Gammoid ist. Zum Schluß des zweiten Kapitels stellen wir unsere Methode vor, eine $\Rbm$-Matrix-Darstellung eines Gammoids aus einer Darstellung vermöge gerichteter Graphen zu erhalten. 
}
\end{parcolumns}
\mytwocols
\colchunk{In Chapter~3 we shortly introduce oriented matroids and their associated concept of colorings. We show that all orientations of
lattice path matroids have $3$-colorings. Then we introduce our concept of a heavy arc orientation of a gammoid that
yields a purely combinatorial way to obtain representable orientations of gammoids. In Chapter~4 we summarize our new results and give an
overview of new and old open problems.
}
\colchunk{Im dritten Kapitel geben wir eine kurze Einführung in orientierte Matroide und den damit verbundenen Begriff der Färbung. Wir zeigen, dass alle Orientierungen von Lattice-Path-Matroiden eine $3$-Färbung besitzen. Danach stellen wir unseren Begriff der Heavy-Arc-Orientierung eines Gammoids vor, welcher eine rein kombinatorische Vorgehensweise,
eine repräsentierbare Orientierung eines Gammoids zu finden, liefert. Im vierten Kapitel resümieren wir unsere neuen Resultate und geben einen Überblick über neue und alte offene Fragestellungen.
}
\end{parcolumns}
% -*- root: ../thesis.tex -*-

\cleardoublepage
\phantomsection
\thispagestyle{empty}
\addcontentsline{toc}{chapter}{Acknowledgments}

\section*{Acknowledgments}

I would like to offer my sincere gratitude to my advisor Prof. Dr. Winfried Hochstättler for
always providing me with the necessary input and feedback during all stages of my studies and research. 
His profound knowledge and open-hearted support have been a great resource throughout the whole process of preparing this work.

Furthermore, I would like to thank all my coworkers and colleagues both at the chair for Discrete Mathematics and Optimization
and at the Department of Mathematics and Computer Science of the FernUniversität in Hagen, as well as those doing related research
who are scattered around the world, for the fruitful discussions with them and for sharing their knowledge with me. 

I am indebted to the FernUniversität that granted me a scholarship for the last six months, which allowed me to focus 
entirely on my thesis.

Finally, I would like to thank my wife, my parents, my family, and my friends for their moral support,
especially during the more intensive days of preparation.

%\include{Dedication/dedication}
%\include{Declaration/declaration}

%TODO: REMEMBER TO ADD ACKNOWLEDGEMENTS!!
%\include{Acknowledgement/acknowledgement}

%\include{Abstract/abstract}

% *********************** Adding TOC and List of Figures ***********************

\tableofcontents

%\listoffigures

%\listoftables

% \printnomenclature[space] space can be set as 2em between symbol and description
%\printnomenclature[3em]

%\printnomenclature

%\pagenumbering{arabic}
% ******************************** Main Matter *********************************
%\mainmatter

%\include{Chapter1/chapter1}
\cleardoublepage
% -*- root: ../thesis.tex -*-

\chapter{Preliminaries}

\PRFRA
In this chapter, we  introduce those aspects of matroid theory that
are most important to the comprehension of the later chapters. For a thorough
introduction to matroid theory, we would like to redirect the reader to the
following books, in no particular order.

\begin{itemize}
\item {\em Matroid Theory} 
		by J.G.~Oxley \cite{Ox11} is a comprehensive resource on matroid theory covering
		most of the current state of the art. Matroids are introduced using a variety
		of cryptomorphic axiom systems    starting from independence axioms and base
		axioms. This book is the authoritative standard    reference for matroid
		theory and we guarantee that all definitions made in this work are compatible
		with those found in J.G.~Oxley's book.
\item {\em Matroid Theory}
		by D.J.A.~Welsh \cite{We76} is an introduction to matroid theory that also covers
		the greedy algorithm, transversal theory, Menger's Theorem and gammoids,
		polymatroids, and infinite generalizations of matroids. Although this book is not the
		most recent one on this topic, it is the book that we would like to recommend 
		to anyone
		who wants to read {\em only one} book on matroid theory, as it presents the theory in
		remarkable clarity.
\item {\em On the Foundations of Combinatorial Theory: Combinatorial Geometries} 
		by H.H.~Crapo and G.-C.~Rota  \cite{CR70} is a remarkably well structured
		introduction to matroid theory with lattice theory as a starting point.
		Unfortunately, a regular edition never followed the preliminary edition.
%\item {\em Introduction to the Theory of Matroids} by W.~Tutte \cite{Tu66} starts from the circuit axioms and focuses particularly on graphical and regular matroids. % DO NOT RECOMMEND: MANY DEFINITIONS ARE SUBTLELY DIFFERENT FROM WHAT IS STANDARD NOW!
\end{itemize}

\needspace{4\baselineskip}
\section*{Notation}

\PRFRA
All notation used in this work is either standard mathematical notation,  or
declared in the corresponding definitions. We would like to point out one less
common notational detail: If we denote a set $X=\SET{a,b,c}$ we are stating
that the set $X$ consists of the elements $a$, $b$, and $c$; but we do not
require any two or all three of $a$,$b$,$c$ to be distinct elements. Thus
$|X|=1$, $|X|=2$, and $|X|=3$ are possibly true assertions with this
notation. But if we denote a set $Y=\dSET{a,b,c}$,\label{n:dset} then we are stating that
$Y$ consists of the elements $a$, $b$, and $c$; and that no two of these
elements are equal, therefore $|Y|=3$ is the only possibility here.

%\noindent
We will denote the set of non-negative integers by $\N=\SET{0,1,2,\ldots}$,
the set of integers by $\Z = \SET{0,1,-1,2,-2,\ldots}$, the field of the rational numbers by $\Qbm$, and
the field of the real numbers by $\R$.
 The cardinality of a set $X$ is denoted by $\left| X \right|$, the power set of $X$ is denoted by $2^X$.
 The set of subsets of $X$ with cardinality $n$ is denoted by $\binom{X}{n}$. The set of all maps $f\colon X \maparrow Y$ is denoted by $Y^X$.

If $f\colon X\maparrow Y$ is a map and $X'\subseteq X$, then we denote the set of images of $x'\in X'$ under $f$ by
$f[X'] = \SET{f(x')\mid x'\in X'}$.\label{n:fsquareX} We denote the restriction of $f$ to $X'$ by $f\restrict_{X'}$.\label{n:frestrictX'}

%\noindent 
Whenever $\Acal \subseteq 2^X$ is a family of sets, we denote
the union of all those sets by \( \bigcup \Acal = \bigcup_{A\in\Acal} A \).\label{n:bigcup}
If $\Acal\not= \emptyset$, we denote the intersection of all sets in $\Acal$ by
\(\bigcap \Acal = \bigcap_{A\in\Acal} A\). For $\Acal=\emptyset$, we set $\bigcap \Acal = \bigcap_X \emptyset = X$.\label{n:bigcap}

\PRFR{Mar 7th}
We use the $O$-notation in the usual way: If $f,g,h \colon \N \maparrow \R$ are maps,
we write $f = O(g)$ in order to denote that $\limsup_{x\rightarrow \infty} \left| \frac{f(x)}{g(x)}  \right| < \infty$.
We write $O(g) = O(h)$ if the implication $f = O(g) \Rightarrow f = O(h)$ holds for all $f\in \R^\N$. Please keep in mind that
$O(g) = O(h)$ is not
equivalent to $O(h) = O(g)$. (!) Instead, the $O$-notation is asymmetric and has to be read from left-to-right.
We also use the straight-forward generalization of the $O$-notation to several non-negative integer variables in
an informal way, for instance we would write $O(x^2 y^3) = O(2^x y^4)$. Similarly, we write $f = \Omega(g)$ in order to
denote that $\limsup_{x\rightarrow \infty} \left| \frac{f(x)}{g(x)}  \right| > 0$.

%\noindent
%BASIC CONVENTION GOES HERE

% -*- root: ../thesis.tex -*-

\section%\remred{Soll sowas wirklich rein??}
{Canonical Preliminaries}
\PRFRA

This section contains canonical definitions,
which are most unrelated to matroid theory. The authors know that it is quite uncommon
to have a canonical preliminaries section within the preliminaries of a work. We
are certain that any person who did study mathematics to some extent
knows the contents of this section by heart, yet we
include it in order to maintain a higher level of self-sufficiency of this work 
as well as to fix certain formal aspects of the common basic definitions.

\needspace{6\baselineskip}
\begin{definition}\PRFRA
	Let $X$ be any set. The \deftext[multi-sets]{multi-sets over $\bm X$}
	are the elements of the set\label{n:multiset}
	\[ \N^X = \SET{f\colon X\maparrow \N}. \]
	The \deftextX{finite multi-sets over $\bm X$}
	are defined to be\label{n:fin-multiset}
	\[ \N^{(X)} = \SET{\left. f\in \N^X \quad\right|\quad \left| \SET{x\in X\mid f(x) \not= 0} \right| < \infty}.\qedhere\]
\end{definition}

\begin{notation}\PRFRA
	Let $X$ be a set, $\Kbm$ be a field. The \deftext{vectors} of the $X$-dimensional vector space $\Kbm^X$ over $\Kbm$ 
	are identified with the maps $v\colon X\maparrow \Kbm$. If $X$ is finite, then 
	the canonical basis of $\Kbm^X$ is the set $\SET{e_i\mid i\in X}$
	where \[ e_i\colon X\maparrow \Kbm, \quad x\mapsto \begin{cases} 1 & \quad \text{if~} x=i,\\ 0& \quad \text{otherwise.} \end{cases}\]
	For $\alpha\in \Kbm$ and $v\in \Kbm^X$, we shall denote the scalar multiplication of $\alpha$ and $v$ both by $\alpha\cdot v$
	and by
	\[ \alpha v\colon X\maparrow \Kbm, \quad x\mapsto \alpha \cdot v(x).\]
	For $X$ finite and $\alpha,\beta\in \Kbm^X$ we denote the scalar product of $\alpha$ and $\beta$ by
	\[ \langle \alpha,\beta \rangle = \sum_{x\in X} \alpha(x)\cdot \beta(x). \qedhere\]
\end{notation}

\begin{definition}\PRFRA
	Let $K$, $R$, and $C$ be any sets.
	An \index{matrix over K@matrix over $K$}\deftext[RxC-matrix over K@$R\times C$-matrix over $K$]{$\bm R\bm\times \bm C$-matrix over $\bm K$} is a map
    $ \mu \colon R\times C \maparrow K$.\label{n:matrix}
    Every $r\in R$ is a \deftextX{row-index of $\bm \mu$}, and every $c\in C$ is a \deftextX{column-index of $\bm \mu$}.
    For every $r\in R$, the map
    \[ \mu_r \colon C\maparrow K,\,c\mapsto \mu(r,c) \]
    is the \deftext[row of mu@row of $\mu$]{$\bm r$-th row of $\bm \mu$}.
    Analogously, for every $c\in C$, the map
    \[ \mu^\top_c \colon R\maparrow K,\,r\mapsto \mu(r,c) \]
    is the \deftext[column of mu@column of $\mu$]{$\bm c$-th column of $\bm \mu$}.
    The class of $R\times C$-matrices over $K$ shall be denoted by $K^{R\times C}$.
    If $R=\SET{1,2,\ldots,n}\subseteq \N$ and $C=\SET{1,2,\ldots,m}\subseteq \N$,
    then we also write $K^{n\times m}$ for $K^{R\times C}$.
    For every matrix $\mu\in K^{R\times C}$, we define the \deftext[transposed matrix $\mu^\top$]{transposed matrix $\bm \mu^{\bm \top}$} to be the map\label{n:transposed}
    \( \mu^\top\colon C\times R\maparrow K,\,(c,r)\mapsto \mu(r,c) \).
\end{definition}

\begin{definition}\PRFRA
	Let $X$ be any set, $\Kbm$ be a field or ring with zero and one. The \deftext[identity matrix]{identity matrix for $\bm X$ over $\bm K$}
	is the map 
	 $$ \mathrm{id}_\Kbm(X) \colon X\times X\maparrow \Kbm, \, (r,c)\mapsto \begin{cases} 1 &\quad \text{if~} r=c,\\
	 																					0 &\quad \text{otherwise.} \end{cases}$$
\end{definition}

\begin{definition}\PRFR{Mar 7th}
	Let $X,Y,Z$ be sets, $Y$ finite, $\Rbm$ a ring. Let further $\mu\in \Rbm^{X\times Y}$ and $\nu \in \Rbm^{Y\times Z}$ be matrices.
	Then the \deftextX{matrix multiplication of $\bm \mu$ with $\bm \nu$} shall be the matrix 
	\[ \mu\ast\nu \colon X\times Z \maparrow \Rbm,\, (x,z) \mapsto \sum_{y\in Y} \mu(x,y)\cdot \nu(y,z).\]
	Let $\alpha\in \Rbm^Y$. Analogously, the \deftextX{vector-matrix multiplication of $\bm \alpha$ with $\bm \nu$}
	shall be the vector
	\[ \alpha\ast \nu \colon Z\maparrow \Rbm,\, z\mapsto \sum_{y\in Y} \alpha(y)\cdot\nu(y,z),\]
	and the \deftextX{matrix-vector multiplication of $\bm \mu$ with $\bm \alpha$}
	shall be
	\[ \mu\ast\alpha \colon X\maparrow \Rbm,\, x\mapsto \sum_{y\in Y} \mu(x,y)\cdot \alpha(y) . \qedhere\]
\end{definition}

\begin{definition}\PRFRA
	Let $\mu \in K^{R\times C}$ be an $R\times C$-matrix over $K$, $R_0\subseteq R$, and $C_0\subseteq C$.
	The \deftext[restriction of mu@restriction of $\mu$]{restriction of $\bm \mu$ to $\bm R_0$} is defined to be the map\label{n:matrestrict}
	\[ \mu\restrict R_0 \,\colon\,R_0\times C\maparrow K,\,(r,c)\mapsto \mu(r,c).\]
	The \deftextX{restriction of $\bm \mu$ to $\bm R_0 \bm\times \bm C_0$} is defined to be the map
	\[ \mu\restrict R_0\times C_0 \,\colon\,R_0\times C_0\maparrow K,\,(r,c)\mapsto \mu(r,c). \qedhere\]
\end{definition}

\needspace{6\baselineskip}
\begin{definition}\label{def:det}\PRFRA
	Let $\Kbm$ be a field or a commutative ring, $X=\dSET{x_1,x_2,\ldots,x_m}$ and $Y=\dSET{y_1,y_2,\ldots,y_m}$
	 be finite sets of equal cardinality that have implicit linear orders given by the indexes, 
	and let $\mu\in \Kbm^{X\times Y}$ be a square matrix over $\Kbm$.
	The \deftext[determinant of mu@determinant of $\mu$]{determinant of $\bm \mu$}
	is defined to be
	\[ \det \mu = \sum_{\sigma\in \mathfrak{S}_m} \,\, \sgn(\sigma) \prod_{i=1}^{m} \mu \left( x_i,y_{\sigma(i)}  \right)  \]
	where $\mathfrak{S}_m$ consists of all permutations $\sigma\colon \SET{1,2,\ldots,m}\maparrow \SET{1,2,\ldots,m}$.
\end{definition}

\begin{definition}\label{def:idet}\PRFRA
	Let $R$ and $C$ be finite sets, $\mu\in\Kbm^{R\times C}$, and $n = \min\SET{\left| R \right|,\left| C \right|}$.
	The \deftext[determinant-indicator of mu@determinant-indicator of $\mu$]{determinant-indicator of $\bm \mu$} is defined to be\label{n:idet}
	\[ \idet \mu = \begin{cases}
					1 & \quad \text{if~} n = 0,\\
					1 & \quad \text{if for some~} R_n \in \binom{R}{n},\,C_n\in \binom{C}{n}  \colon\,\,\det \left( \mu \restrict R_n\times C_n \right) \not= 0, \\
					0 & \quad \text{otherwise.}
	\end{cases} \]
\end{definition}

\begin{notation}\PRFR{Mar 7th}
	Let $\R$ be a commutative ring, $X$ be a set. \label{n:polynomring}The \deftext[polynomial ring]{polynomial ring over $\bm \R$ with variables $\bm X$} 
	shall be denoted by $\R[X]$. The \deftext[unit monomials]{unit monomials of $\bm \R \bm[\bm X\bm]$}, 
	i.e. polynomials of the form $x_1^{n_1} x_2^{n_2} \ldots x_k^{n_k}$ where $\dSET{x_1,x_2,\ldots,x_k}\subseteq X$,
	may be identified
	with the finite multi-sets $\N^{(X)}$ and thus they shall be denoted by $\N^{(X)}$, too. It is also customary to identify the
	polynomial ring $\R[\emptyset]$ with the ring $\R$ itself, and to write $\R[x_1,x_2,\ldots,x_k]$ for $\R[\SET{x_1,x_2,\ldots,x_k}]$.
	Furthermore, for every polynomial $p\in \R[X]$, $Y\subseteq X$, and every $\eta \in \R^Y$, 
	we obtain a polynomial $p[Y=\eta]\in \R[X\BS Y]$ by setting $y = \eta(y)$ in $p$ for every $y\in Y$.
	For $Y=\dSET{x_1,x_2,\ldots,x_i}$, we also write $p[x_1=\eta(x_1),x_2=\eta(x_2),\ldots,x_i=\eta(x_i)]$ in order to denote $p[Y=\eta]$.
	For $p\in \R[x]$ and $r\in \R$, we denote $p[x=r]$ by $p(r)$.
\end{notation}

\begin{definition}\PRFR{Mar 7th}\label{def:Zindependent}
	Let $X\subseteq \R$ be a set of reals. Then $X$ shall be called 
	\deftext[Z-independent@$\Z$-independent]{ $\Z$-independent},
	if for the injection $\xi\colon X\maparrow \R$ with $\xi(x) = x$ and for all $p\in \Z[X]$
	the equivalency \[ p[X=\xi] =\, 0 \,\,\, \Longleftrightarrow\,\,\,  p \,=\, 0 \]
	holds.
\end{definition}

%\studyremark{
%\begin{theoremX}[Lindemann-Weierstraß]
%	Let $U = \dSET{u_1,u_2,\ldots,u_n} \subseteq \R$ be algebraic numbers over $\Qbm$ such that $U$ is linear independent with respect to coefficients in $\Qbm$,
%	then the set $e^U = \SET{e^{u_1},e^{u_2},\ldots,e^{u_n}}$ is algebraically independent over the algebraic closure of $\Qbm$.
%\end{theoremX}

%\noindent
%A proof can be found in N.~Jacobson's {\em Basic Algebra I} on page 268 \cite{Ja74}. We are only interested in the following consquence of the theorem.}

\begin{lemma}\label{lem:enoughZindependents}\PRFR{Mar 7th}
	Let $n\in \N$. There is a set $X = \dSET{x_1,x_2,\ldots,x_n}\subseteq \R$ such that
	$X$ is $\Z$-independent, where $\R$ denotes the set of reals.%\footnote{In fact, every subset of a basis $B$ of $\R$ is $\Z$-independent,
	%where $\R$ is considered to be a vector space over the field of the rationals $\mathbb{Q}$.}
\end{lemma}
\begin{proof}\PRFR{Mar 7th}
	By induction on $\N$. The base case is clear. For the induction step, let $X' = \dSET{x_1,x_2,\ldots,x_{n-1}}\subseteq \R$ be $\Z$-independent.
	Then for $x\in \R$, $X'\cup\SET{x}$ is not $\Z$-independent, if
	and only if there is a non-zero polynomial $p \in \Z[x_1,x_2,\ldots,x_{n-1},x]$ 
	such that the polynomial $p_0$ has the root $p_0(x) = 0$,
	where $p_0\in \R[x]$ arises from $p$
	--- which can be interpreted as a polynomial over $\R$ ---
	by setting $$p_0 = p[x_1=x_1,x_2=x_2,\ldots,x_{n-1}=x_{n-1}] \in \R[x].$$
	In other words, $p_0$ arises from $p$ by identification of the monomials $x'\in X'$ with their natural real value.
	Since $X'$ is $\Z$-independent,
	we obtain $p_0 \not= 0$ unless $p = 0$. Thus each polynomial $p_0$ obtained in this way has only finitely many roots. Furthermore, the set $\Z[x_1,x_2,\ldots,x_{n-1},x]$ is countable, therefore there are only countably many real numbers $x\in \R$
	such that the set $X'\cup\SET{x}$ is not $\Z$-independent. But $\R$ is uncountably infinite, so there is some $x\in \R\BS X'$,
	such that $X'\cup\SET{x}$ is $\Z$-independent.
\end{proof}

\begin{definition}[\cite{Bi67}, p.1]\PRFR{Feb 15th}
	Let $(P,\leq)$ be a pair, where
	$P$ is any set -- called the \deftextX{support set of $\bm (\bm P\bm,\bm\leq\bm)$} -- and $\leq$ is a binary relation on $P$.
	Then $(P,\leq)$ is a \deftext{poset}, if the following properties hold for all $p,q,r\in P$:
	\begin{enumerate}\ROMANENUM
	\item $p \leq p$;
	\item if $p \leq q$ and $q\leq p$ holds, then $p = q$; and
	\item if $p \leq q$ and $q\leq r$ holds, then $p \leq r$ holds, too.
	\end{enumerate}
	If the poset $(P,\leq)$ is clear from the context,
	we also denote $(P,\leq)$ by its support set $P$, or by its binary relation symbol $\leq$.
	Furthermore, we shall write $p < q$ -- where we may use an analogue symbol corresponding to the symbol used to denote the binary relation of the poset in question -- whenever $p\leq q$ and $p\not= q$ holds.
	A poset $(P,\leq)$ is called \deftextX{finite}, if $P$ is finite.
	For every poset $(P,\leq)$ and every $y\in P$, the \deftextX{$\bm (\bm P\bm,\bm\leq\bm)$-down-set of $\bm y$} shall be the set
	\label{n:Pdownset}
	\[ \downarrow_{(P,\leq)} y = \SET{x\in P ~\middle|~ x \leq y} . \qedhere\]
\end{definition}

\begin{example}\PRFR{Feb 15th}
	Let $X$ be a finite set, and $P\subseteq 2^X$. Then $(P,\subseteq)$ is a poset, where $\subseteq$ denotes the usual set-inclusion.
\end{example}

\needspace{8\baselineskip}

\begin{definition}[\cite{Bi67}, pp.101f]\label{n:zeta}\label{def:zetaMatrix}\label{def:moebiusFunction}\PRFR{Feb 15th}
	Let $(P,\leq)$ be a finite poset. The \deftext[zeta-matrix of a poset]{zeta-matrix of $\bm(\bm P\bm,\bm\leq\bm)$} shall be the map
	\[ \zeta_{(P,\leq)} \colon P\times P\maparrow \Z,\,(p,q) \mapsto \begin{cases}[r] 1& \quad \text{if~} p \leq q,\\
																				 % -1& \quad \text{if~} p= q,\\
																   0& \quad \text{otherwise.} \end{cases} \]
	If the poset is clear from the context, we shall denote $\zeta_{(P,\leq)}$ by $\zeta_P$ or $\zeta$.
	The \deftext[Möbius function]{Möbius-function of $\bm(\bm P\bm,\bm\leq\bm)$} is defined as\label{n:moebius}
	\[ \mu_{(P,\leq)} \colon P\times P\maparrow \Z, \,(p,q) \mapsto \begin{cases}[r] 0 & \quad\text{if~} p\not\leq q,\\
				1 & \quad\text{if~} p=q,\\
				- \displaystyle\sum_{q'\in P,\,p\leq q' < q} \mu (p,q') & \quad \text{otherwise.}
				\end{cases}\]
	Again, if the poset is clear from the context, we shall denote $\mu_{(P,\leq)}$ by $\mu_P$ or $\mu$.
\end{definition}

\begin{lemma}[\cite{Ro64}, Proposition 1]\label{lem:moebiusInversion}\PRFR{Feb 15th}
	Let $(P,\leq)$ be a finite poset. Then \[\mu_P \ast \zeta_P = \id_\Z(P).\]
\end{lemma}

\noindent In other words, the Möbius-function of a poset is the inverse matrix of the zeta-matrix of that poset, and thus all $\zeta_P$
are invertible in the ring of integer matrices.

\begin{proof}
	Let $(P,\leq)$ be a finite poset, and let $\mu = \mu_P$ and $\zeta = \zeta_P$ be defined as in Definition~\ref{def:moebiusFunction}.
	Let $p,r\in P$, then we have
	\begin{align*}
		\left( \mu\ast \zeta \right)(p,r) = & \sum_{q\in P} \mu(p,q)\cdot\zeta(q,r) 
		 =  \sum_{q\in P,\, p \leq q\leq r} \mu(p,q)\cdot\zeta(q,r) 
	\end{align*}
	because if $p\not\leq q$, then $\mu(p,q) = 0$, and if $q\not\leq r$, then $\zeta(q,r) = 0$.
	Therefore
	we obtain that for all $p\in P$, $$\left( \mu\ast \zeta \right)(p,p) = \sum_{q\in P,\,p\leq q \leq p} \mu(p,q)\cdot\zeta(q,p) = \mu(p,p)\cdot\zeta(p,p) = 1\cdot 1 = 1 = \id_\Z(P)(p,p). $$
	Now let $p,r\in P$ with $p\not= r$.
	Since $\zeta(p,q) = 1$ whenever $p\leq q$, we have
	\begin{align*} \sum_{q\in P,\, p \leq q\leq r} \mu(p,q)\cdot\zeta(q,r) = & \left( \sum_{q\in P,\, p \leq q < r} \mu(p,q)\cdot\zeta(q,r)  \right) + \mu(p,r)\cdot\zeta(r,r) \\
	   = & \left( \sum_{q\in P,\, p \leq q < r} \mu(p,q)  \right) + \mu(p,r) \\
	   = & \left( \sum_{q\in P,\, p \leq q < r} \mu(p,q)  \right) - \left( \sum_{q\in P,\, p \leq q < r} \mu(p,q)  \right) \\ = & \hphantom{(}0 = \id_\Z(P)(p,r).
	\end{align*}
	Therefore $\mu\ast\zeta = \id_\Z(P)$.
\end{proof}

\begin{lemma}[Principle of Inclusion-Exclusion, \cite{Ro64}]\PRFR{Feb 15th}\label{lem:inclusionExclusion}
 	Let $X$ be a finite set.
 	Then for all $A,B\subseteq X$
 	\[ \mu_{\left(2^X,\subseteq\right)}(A,B) = \begin{cases}[r] (-1)^{\left| B \right| - \left| A \right|} &\quad \text{if~} A\subseteq B,\\
 												0 & \quad \text{otherwise.} \end{cases} \]
\end{lemma}

%\noindent 

\cleardoublepage
% -*- root: ../thesis.tex -*-

\section{Matroid Basics}
\PRFRA

In this section, we give a quick and incomplete review of some axiomatizations
of matroids.  A more complete picture as well as some proofs\footnote{Some
axiomatizations can be found in the exercise sections, where, of cause, the
proofs are left for the reader.} of cryptomorphy can be obtained from J.G.~Oxley's
book \cite{Ox11}.

% -*- root: ../thesis.tex -*-

\subsection{Independence Axioms}
\PRFRA

All definitions, lemmas, theorems, and proofs in this subsection are canonical and 
can be found in \cite{Ox11}.  %, for instance, in J.~Oxley's book \cite{Ox11}. 
Readers familiar with matroid theory may safely skip this section.

\begin{definition}\label{def:indepAxioms}\label{n:Is}\PRFRA
	Let $E$ be a finite set, $\Ical \subseteq 2^{E}$. Then the  pair $(E,\Ical)$\label{n:EI} is
	an \deftext{independence matroid}, or shorter \deftext{matroid}, if
	the following properties hold:
	\begin{enumerate}
		\item[(I1)] $\emptyset \in \Ical$,
		\item[(I2)] for $I\in\Ical$ and every $J\subseteq I$, we have $J\in\Ical$.
		\item[(I3)] If $J,I\in\Ical$ and $|J|<|I|$, then there is some $i\in I\BS J$, such that $J\cup\SET{i}\in \Ical$.
	\end{enumerate}
	Let $X\subseteq E$, we say that $X$ is \deftext{independent} in the matroid $M=(E,\Ical)$,
	 if $X\in\Ical$. Otherwise, we say that $X$ is \deftext{dependent} in $M$.
\end{definition}

\begin{example}\label{ex:freematroid}\PRFRA
	Let $E$ be any finite set, then the \deftext{free matroid} on the ground set $E$
	shall be the matroid $M=(E,\Ical)$ where all subsets of $E$ are independent, i.e. where $\Ical = 2^{E}$.
\end{example}

\noindent
Matroids have the natural concept of isomorphy.

\begin{definition}\PRFR{Feb 15th}
	Let $M=(E,\Ical)$ and $N=(E',\Ical')$ be matroids. A bijective map
	$$\phi\colon E\maparrow E'$$ is called \deftext[matroid isomorphism]{matroid isomorphism between $\bm M$ and $\bm N$}, if
	for all $X\subseteq E$
	\[ X\in \Ical \quad\Longleftrightarrow\quad \phi[X] \in \Ical' \]
	holds. As usual, an \deftext[M-automorphism@$M$-automorphism]{$\bm M$-automorphism} is a matroid isomorphism between $M$ and itself.
\end{definition}

\noindent For now, we will stick to the independence axioms of matroids and define the typical matroid
 concepts in terms of their independence systems.

\begin{definition}\label{def:directSum}\PRFR{Mar 7th}
	Let $M=(E,\Ical)$ and $N=(E',\Ical')$ be matroids such that $E\cap E' = \emptyset$.
	Then the \deftext[direct sum of matroids]{direct sum of $\bm M$ and $\bm N$} is the matroid $M \oplus N = (E\cup E', \Ical_\oplus)$ 
	where
	\[ \Ical_\oplus = \SET{X\cup X' ~\middle|~ X\in \Ical,\,X'\in \Ical'}. \qedhere\]
\end{definition}

\begin{lemma}\PRFR{Mar 7th}
	Let $M=(E,\Ical)$ and $N=(E',\Ical')$ be matroids such that $E\cap E' = \emptyset$. Then $M\oplus N$ is indeed a matroid.
\end{lemma}
\begin{proof}\PRFR{Mar 7th}
	Each matroid axiom may be easily deduced from the fact that every summand satisfies that axiom:
	$\emptyset\in \Ical_\oplus$ since $\emptyset \in \Ical$ and $\emptyset \in \Ical'$, {\em(I1)} holds.
	 Let $X\cup X' \in \Ical_\oplus$ for
	some $X\in \Ical$ and $X'\in \Ical'$. Let $Y\subseteq X\cup X'$, then $Y = (Y\cap X) \cup (Y\cap X')$, and since $(Y\cap X) \subseteq X$
	and $(Y\cap X')\subseteq X'$, we have $(Y\cap X) \in \Ical$ and $(Y\cap X')\in \Ical'$, therefore $Y\in \Ical_\oplus$, 
	{\em(I2)} holds. Let $X\cup X'\in \Ical_\oplus$ and $Y\cup Y'\in \Ical_\oplus$ with
	$\left| X\cup X' \right| < \left| Y\cup Y' \right|$, i.e. $X,Y\in \Ical$ and $X',Y'\in \Ical'$, and
	$\left| X \right| + \left| X' \right| < \left| Y \right| + \left| Y' \right|$. By symmetry we may assume without loss of generality that $\left| X \right| < 	\left| Y \right|$. Then there is some $y\in Y\BS X$ such that $X\cup\SET{y} \in \Ical$,
	therefore $X\cup\SET{y}\cup X' \in \Ical_\oplus$, thus {\em(I3)} holds.
\end{proof}

\begin{definition} \PRFRA
	Let $M=(E,\Ical)$ be a matroid. Every maximal element of $\Ical$ is called 
	a \deftext{base} of $M$. For $F\subseteq E$, every maximal element of
	$\SET{I\in\Ical \mid I\subseteq F}$ is called a \deftext[base of F@base of $F$]{base of $\bm F$} in $M$.
	The family of all bases of $M$ shall be denoted by \label{n:BcalM} $\Bcal(M)$, and the family of all bases of $F$ in $M$
	shall be denoted by\label{n:BcalMF} $\Bcal_M(F)$.
\end{definition}

\noindent  It is an important property of matroids, that for every
$F\subseteq E$, the bases of $F$ have the same cardinality; and that every
independent subset of $F$ can be augmented to a base of $F$. Likewise, any
set independent in a matroid $M$ can be augmented to a base of $M$.

\needspace{7\baselineskip}
\begin{lemma}\label{lem:augmentation}\PRFRA
	Let $M=(E,\Ical)$ be a matroid, and let $F\subseteq H\subseteq E$ with $F\in\Ical$.
	Then there is a subset $G\in \Ical$ with $F\subseteq G\subseteq H$, such that
	\( \left| G\right| = \max \SET{\vphantom{A^A} \left| I\right| ~\middle|~ I\in \Ical,\,I\subseteq H} \).
\end{lemma}
\begin{proof}\PRFRA
	Let $\Ical' = \SET{I\in \Ical\mid F\subseteq I \subseteq H}$. Clearly, $F\in \Ical'$ and $\Ical'$ is finite,
	 therefore there is an element $G\in \Ical'$ which is maximal with respect to set-inclusion $\subseteq$. Now assume that $\left| G \right| < \left| I \right|$ for some $I\in \Ical$ with $I\subseteq H$. By {\em (I3)} there is an element $i\in I\BS G$
	such that $G\cup\SET{i}\in \Ical$. But $i\in I\subseteq H$, therefore $G\cup\SET{i} \in \Ical'$, which contradicts the choice of $G$ as $\subseteq$-maximal element of $\Ical'$. Thus
	$\left|G\right| = \max \SET{\vphantom{A^A} \left| I\right| ~\middle|~ I\in \Ical,\,I\subseteq H}$.
\end{proof}

\needspace{4\baselineskip}
\begin{corollary}\label{cor:equicardinality}\PRFRA
Let $M=(E,\Ical)$ be a matroid, $H\subseteq E$. Let $F,G$ be maximal elements in $\SET{X\in\Ical\mid X\subseteq H}$ with respect to set-inclusion. Then $\left| F \right|=\left| G\right|$.
\end{corollary}
\begin{proof}\PRFRA
	If, without loss of generality, $\left| F\right| < \left| G\right|$, then $F$ cannot
	be maximal with respect to set-inclusion, because then Lemma~\ref{lem:augmentation} gives
	a proper independent superset of $F$ in $H$.
\end{proof}

\begin{corollary}\label{cor:basisexchange}\PRFRA
	Let $M=(E,\Ical)$ be a matroid, $F\subseteq E$ and $B_{1},B_{2}\subseteq F$ be bases of $F$ in $M$. Then the following property is satisfied:
	\begin{enumerate}[align=parleft,leftmargin=2cm,labelsep=1.5cm]\label{n:B3p}
	\item[(B3')] For every element $x\in B_{1}\BS B_{2}$ there is an element $y\in B_{2}\BS B_{1}$, such that
	$\left(B_{1}\BSET{x}\right)\cup\SET{y}$ is a base of $F$ in $M$.
\end{enumerate}
\end{corollary}
\begin{proof}\PRFRA

	Since $\left| B_1 \right| = \left| \left(B_{1}\BSET{x}\right)\cup\SET{y}
	\right|$ for any $x\in B_{1}\BS B_{2}$ and $y\in B_{2}\BS B_{1}$, it
	suffices to show, that for each such $x$, there is a corresponding $y$ with
	$\left(B_{1}\BSET{x}\right)\cup\SET{y}\in \Ical$. We give an indirect
	argument. Assume that for $x\in B_{1}\BS B_{2}$, there is no $y\in
	B_{2}\BS B_{1}$ with $\left(B_{1}\BSET{x}\right)\cup\SET{y}$ independent
	in $M$. Then $B_{1}\BSET{x}$ is a base of $B' =
	\left(B_1\BSET{x}\right) \cup \left(B_2\BS B_1 \right)$. Clearly, $B' =
	\left(B_{1} \cup B_{2}\right)\BSET{x}$, but $x\notin B_{2}$, therefore
	$B_{2} \subseteq B'$. Now $B_{2}\in \Ical$ together with $\left| B_2 \right| > \left| B_1\BSET{x}  \right|$ contradicts that
	$B_{1}\BSET{x}$ is a base of $B'$. Therefore, there is some $y\in
	B_{2}\BS B_{1}$ such that $\left(B_{1}\BSET{x}\right)\cup\SET{y}$ is a base of
	$F$ in $M$.
\end{proof}

\begin{lemma}\label{lem:basisexchangesymmetric}\PRFRA
	Let $M=(E,\Ical)$ be a matroid, $F\subseteq E$ and $B_{1},B_{2}\subseteq F$ be bases of $F$ in $M$. For every element $y\in B_{2}\BS B_{1}$ there is an element $x\in B_{1}\BS B_{2}$, such that
	$\left(B_{1}\BSET{x}\right)\cup\SET{y}$ is a base of $F$ in $M$.
\end{lemma}

\noindent D.J.A.~Welsh gives the following nice and short proof of
 this lemma in \cite{We76}.

\begin{proof}\PRFRA
  Let $y\in B_{2}\BS B_{1}$, thus $\SET{y}\in\Ical$. From Lemma~\ref{lem:augmentation} we
  obtain that there is a basis $B'$ of $F' = B_1\cup\SET{y}$ with $\SET{y}\subseteq B'$.
   Since $B_1$ is a base of $F$ and a proper subset of $F'\subseteq F$, $F'$ is dependent. Thus $B'$ is a proper subset of $F'$ and therefore there is an element
    $x\in B_{1}\BS B'$. Since $B_{1}$ and $B'$ are bases of $F'=B_{1}\cup\SET{y}=B'\cup\SET{x}$ in $M$,
    and $B_{1}$ and $B_{2}$ are bases of $F$ in $M$,
     we have $\left| B' \right| = \left| B_{1}\right| = \left| B_{2}\right|$,
     so $B'=(B_{1}\BSET{x})\cup\SET{y}$ is a base of $F$ in $M$, too.
\end{proof}

\begin{definition} \PRFRA
	Let $M=(E,\Ical)$ be a matroid. A set $C\subseteq E$ is called \deftext{circuit} of $M$, if $C$ is dependent, yet any proper subset of $C$ is independent in $M$. The set of circuits of $M$ is denoted by\label{n:CM}
	\[ \Ccal(M) = \SET{\vphantom{A^A}C\subseteq E ~\middle|~C\notin \Ical,\,\forall c\in C\colon \,C\BSET{c}\in \Ical}. \qedhere\]
\end{definition}

\noindent Obviously, we may restore $\Ical$ from $\Ccal(M)$ since the independent sets of $M$ are those subsets of $E$,
 which do not contain a circuit. The following property of $\Ccal(M)$
  is called \deftext{strong circuit elimination} and also plays a role
 in axiomatizing matroids using axioms governing its family of circuits.

\begin{lemma}[\cite{Ox11}, Proposition~1.4.12]\label{lem:strongCircuitElimination}
	Let $M=(E,\Ical)$ be a matroid, and let $C_1, C_2\in\Ccal(M)$ be circuits of $M$.
	Furthermore, let $e\in C_1\cap C_2$ and $f\in C_1\BS C_2$.
	Then there is a circuit $C'\in \Ccal(M)$ such that
	$f\in C'$ and $C' \subseteq \left( C_1\cup C_2 \right)\BSET{e}$.
\end{lemma}

\noindent For a proof, see \cite{Ox11}, p.29.

\begin{definition}\label{def:matroidLoop}\label{def:parallelEdgesMatroid}\PRFRA
	Let $M=(E,\Ical)$ be a matroid, $l\in E$. Then $l$ is called a \deftext[loop (matroid)]{loop in $\bm M$}, if
	the singleton
	$\SET{l}$ is a circuit of $M$.
	Let $p_1,p_2\in E$ such that $p_1\not= p_2$. Then $p_1$ and $p_2$ are called \deftext[parallel edges (matroid)]{parallel edges in $\bm M$},
	if $\SET{p_1,p_2}$ is a circuit of $M$.
	Let $c\in E$ such that for all bases $B$ of $M$, $c\in B$. Then $c$ is called a \deftext[coloop]{coloop in $\bm M$}.
\end{definition}

\begin{definition}\label{def:rank}\PRFRA
	Let $M=(E,\Ical)$ be a matroid. The \deftext{rank function} of $M$ shall be
	 the map\label{n:rkM}
	\[ \rk_M \colon 2^{E} \maparrow \N,\,X\mapsto \max \SET{\vphantom{A^A}\left|Y\right| ~\middle|~ Y\subseteq X,\,Y\in \Ical} .\]
	If the matroid $M$ is clear from the context, we denote $\rk_{M}$ by $\rk$.
\end{definition}

\noindent Again, $\Ical$ may be retrieved from $\rk_{M}$ since the independent sets are precisely those elements of the 
domain $2^E$ of $\rk_M$, for which the cardinality and the image under the rank function coincide.

\begin{lemma}\label{lem:rankMonotone}\PRFRA
	Let $M=(E,\Ical)$ be a matroid, and $X\subseteq Y\subseteq E$. Then $\rk(X) \leq \rk(Y)$.
\end{lemma}
\begin{proof}\PRFRA
	Since $\SET{I\in \Ical\mid I\subseteq X} \subseteq \SET{I\in \Ical\mid I\subseteq Y}$
	the maximum expression for $\rk(Y)$ ranges over a superset of the expression for $\rk(X)$
	and therefore cannot be smaller.
\end{proof}

\begin{definition}\label{def:clM}\label{def:FcalM}\PRFRA
	Let $M=(E,\Ical)$ be a matroid. A set $F\subseteq E$ is called \deftext{flat} of $M$,
	if for all $x\in E\BS F$, the equality $\rk(F\cup\SET{x}) = \rk(F) + 1$
	holds. The family of all flats of $M$ is denoted by\label{n:FM}
	\[ \Fcal(M) = \SET{ \vphantom{A^A}X\subseteq E ~\middle|~ \forall y\in E\BS X\colon\, \rk_M(X) < \rk_M(X\cup\SET{y})}.\]
	The \deftext[closure operator of M@closure operator of $M$]{closure operator of $\bm M$} is defined to be the map\label{n:clM}
	\[ \cl_{M}\colon 2^{E} \maparrow 2^{E},\, X\mapsto \bigcap\SET{F\in \Fcal(M)\mid X\subseteq F}.\]
	If the matroid $M$ is clear from the context, we denote $\cl_{M}$ by $\cl$.
\end{definition}

\noindent
	Clearly, for every matroid $M=(E,\Ical)$, the ground set $E\in\Fcal(M)$ is a flat, and therefore the defining expression of $\cl(X)$ is well-defined, as it is never an intersection of an empty family. %\noindent
 The following properties are easy consequences from the definition of
the closure operator.

\begin{lemma}\label{lem:clFlips}\PRFRA
   Let $M=(E,\Ical)$ be a matroid, $X\subseteq Y\subseteq E$.
   Then $X\subseteq \cl(X) \subseteq \cl(Y)$.
\end{lemma}

\begin{proof}
   Since $\emptyset \not= \SET{F\in \Fcal(M)\mid Y\subseteq F} \subseteq \SET{F\in \Fcal(M)\mid X\subseteq F}$, we have 
    \[ X\subseteq \cl(X) = \bigcap\SET{F\in \Fcal(M)\mid X\subseteq F} \subseteq \bigcap\SET{F\in \Fcal(M)\mid Y\subseteq F} = \cl(Y).
    \qedhere \]
\end{proof}

\begin{lemma}\label{lem:clKeepsRank}\PRFRA
	Let $M=(E,\Ical)$ be a matroid, $X\subseteq E$. Then $\rk(X) = \rk(\cl(X))$.
\end{lemma}
\begin{proof}
	By Lemma~\ref{lem:clFlips} we have $X\subseteq \cl(X)$ and by Lemma~\ref{lem:rankMonotone} we obtain that $\rk(X) \leq \rk(\cl(X))$. Now consider the
	family $\Ecal = \SET{Y\subseteq E\mid X\subseteq Y \txtand \rk(X) = \rk(Y)}.$
	Since $X\in \Ecal$ and $E$ is finite, there is a maximal element $F\in \Ecal$ with respect to set-inclusion. Since $F$ is maximal, we have that $F\in\Fcal(M)$.
	Thus $\cl(X) \subseteq F$ and so $\rk(\cl(X)) \leq \rk(F) = \rk(X)$ holds, and
	consequently $\rk(X) = \rk(\cl(X))$.
\end{proof}

\begin{lemma}\label{lem:flatsintersectinflat}\PRFRA
	Let $M=(E,\Ical)$ be a matroid, $X\subseteq E$. Then for every $\Fcal' \subseteq \Fcal(M)$, $\bigcap_E \Fcal' \in \Fcal(M)$. Furthermore, for all $X\subseteq E$
	 $$\cl(X) \in \Fcal(M) \text{ ~~and~~ } \cl(\cl(X)) = \cl(X).$$
\end{lemma}

\begin{proof}\PRFRA	
 Let $\Fcal' \subseteq \Fcal(M)$, and let     $F' = \bigcap_E
\Fcal' = \SET{x\in E\mid \forall F\in \Fcal'\colon\, x\in F}$.     Let $e\in
E\BS F'$, then there is some $F\in \Fcal'$ with $e\notin F$. Since
$\rk(F\cup\SET{e}) > \rk(F)$ holds, for every base $B$ of $F$, we must have
$B\cup\SET{e} \in \Ical$. Now let $B'\subseteq F'$ be a base of $F'$, then by
Lemma~\ref{lem:augmentation}, there is a base $B$ of $F$ with $B'\subseteq B$.
Since $B'\cup\SET{e}\subseteq B\cup\SET{e}$, we obtain that
$\rk(F'\cup\SET{e}) \geq \left| B'\cup\SET{e} \right| > \left| B' \right| =
\rk(F')$. Thus $F'\in \Fcal(M)$.

\noindent Let $X\subseteq E$,
since the closure operator $\cl$ is defined to be the intersection of a family of flats of $M$,
we have
$\cl(X)\in \Fcal(M)$. Therefore $\cl(X)$ is the unique minimal element of
 $\SET{F\in\Fcal(M)\mid X\subseteq F}$ with respect to set-inclusion $\subseteq$. Thus we have the following equality between subfamilies of $\Fcal(M)$ $$\SET{F\in\Fcal(M)\mid X\subseteq F}=\SET{F\in\Fcal(M)\mid \cl(X)\subseteq F},$$
which yields $\cl(\cl(X))=\cl(X)$.
\end{proof}

\needspace{4\baselineskip}

\begin{lemma}\label{lem:clofbase}\PRFRA
	Let $M=(E,\Ical)$ be a matroid, $X \subseteq Y \subseteq E$.
	Then $\cl(X) = \cl(Y)$ if and only if there is a base $B$ of $Y$ with
	$B\subseteq X$.
\end{lemma}
\begin{proof}\PRFRA
	Assume that $\cl(X) = \cl(Y)$, then $\rk(X) = \rk(\cl(X)) = \rk(\cl(Y)) = \rk(Y)$ by Lemma~\ref{lem:clKeepsRank}. Let $B$ be a base of $X$, then $\rk(B) = \rk(Y)$, so $B\subseteq X\subseteq Y$ is also a base of $Y$. Now assume that $\cl(X)\not=\cl(Y)$, thus there is some $y\in \cl(Y) \BS \cl(X)$ such that for some base $B$ of $\cl(X)$ in $M$, $B\cup\SET{y}\in\Ical$ is independent. Thus $\rk(Y) = \rk(\cl(Y)) > \rk(X)$ and therefore no base $B'$ of $Y$ is a subset of $X$.
\end{proof}

\studyremark{
\remred{I AM REMOVING THIS}

\begin{definitionX}
  Let $M=(E,\Ical)$ be a matroid, $Z\subseteq E$ is called \deftext{cyclic flat} of $M$,
  if $Z\in\Fcal(M)$ is a flat and if it is a union of circuits, i.e. if
  \( F = \displaystyle \bigcup_{C\in\Ccal(M),\,C\subseteq F} C\).
  The family of all cyclic flats of $M$ is denoted by\label{n:ZM}
  \(\Zcal (M)\).
\end{definitionX}

\begin{lemmaX}
  Let $M=(E,\Ical)$ be a matroid, $F\in\Fcal(M)$.
  Then $F$ is a cyclic flat if and only if for all $f\in F$, $\rk(F) = \rk(F\BSET{f})$.
\end{lemmaX}

\begin{proof}
	Let $F$ be a cyclic flat, and let $f\in F$. Then there is a circuit
	$C\subseteq F$ with $f\in C$. Therefore, $C\BSET{f}\in\Ical$ and by
	Lemma~\ref{lem:augmentation} there is a subset $X\subseteq F\BSET{f}$ with
	$C\BSET{f}\subseteq X$ and $\left| X \right| = \rk(F\BSET{f})$. Assume
	that $\rk(F\BSET{f}) < \rk(F)$, then again by Lemma~\ref{lem:augmentation}
	there is a set $X'\subseteq F\cup\SET{f}$ with $X\subseteq X'$ and $\left|
	X' \right| = \rk(F) > \left| X \right|$. But $X$ is maximally independent
	in $F\BSET{f}$, thus the only possibility is $X'=X\cup\SET{f}$, but then
	$C\subseteq X'$ which contradicts that $X'\in\Ical$. Thus $\rk(F) = \left|
	X \right| = \rk(F\BSET{f})$.  Conversely, let $F\in\Fcal(M)$ such that for
	all $f\in F$, $\rk(F) = \rk(F\BSET{f})$. Assume that for some $f\in F$,
	there is no circuit $C\subseteq F$ with $f\in C$. Let $X\subseteq
	F\BSET{f}$ be of maximal cardinality, such that $X\in\Ical$, then $\left|
	X \right| = \rk(F\BSET{f})$. Then $X\cup\SET{f}\in \Ical$ because
	$X\cup\SET{f}\subseteq F$ which has no circuit containing $f$. But then
	$\rk(F) \geq \left|X \right| + 1 > \left| X \right| = \rk(F\BSET{f})$
	yields a contradiction. Therefore, for every $f\in F$ there is a circuit
	$C\subseteq F$ with $f\in C$.
\end{proof}
}

% -*- root: ../thesis.tex -*-

\subsection{Rank Axioms}

There are at least two natural ways to axiomatize matroids through their corresponding rank functions.

\begin{theorem}\label{thm:rankAxioms}\PRFRA
  Let $E$ be a finite set, $\rho\colon 2^{E}\maparrow \N$ a map. The following are
  equivalent:
  \begin{enumerate}\ROMANENUM
  \item There is a matroid $M=(E,\Ical)$ with $\rk_{M} = \rho$,
  \item $\rho$ satisfies the properties {\em (R1') -- (R3')}, and
  \item $\rho$ satisfies the properties {\em (R1) -- (R3)};
  \end{enumerate}
  where
  \begin{itemize}\label{n:Rxp}
  	\item[(R1')] $\rho(\emptyset) = 0$,
    \item[(R2')] $\rho(X) \leq \rho(X\cup\SET{y}) \leq \rho(X) + 1$ for all $X\subseteq E$ and all $y\in E$,
    \item[(R3')] if $\rho(X) = \rho(X\cup\SET{y})=\rho(X\cup\SET{z})$, then $\rho(X) = \rho(X\cup\SET{y,z})$, for all $X\subseteq E$ and all $y,z\in E$;
  \end{itemize}
  \begin{itemize}\label{n:Rx}
  	\item[(R1)] $0 \leq \rho(X) \leq \left|X\right|$ for all $X\subseteq E$,
    \item[(R2)] if $X\subseteq Y$, then $\rho(X) \leq \rho(Y)$ for all $X,Y\subseteq E$,
    \item[(R3)] $\rho(X\cup Y) + \rho(X\cap Y) \leq \rho(X) + \rho(Y)$ for all $X,Y\subseteq E$.
  \end{itemize}
 \end{theorem}
\noindent We named the rank axioms coherent with J.G.~Oxley's book \cite{Ox11}; D.J.A.~Welsh's {\em Matroid Theory} \cite{We76} denotes {\em (R1)--(R3)} with {\em (R1')--(R3')}, and vice-versa, yet the proof is more along the lines of section~1.6 in D.J.A.~Welsh's book \cite{We76}.

\begin{proof}
\underline{The implication {\em (i) $\Rightarrow$ (ii)}.}\PRFRA
\goldstar{Marc, Oct 27, whole proof}
\bSep By {\em (I1)} we
obtain $\rk(\emptyset) = \left|\emptyset\right| = 0$, thus {\em (R1')} holds
for $\rk$. 

\bSep \PRFRA
Let $X'\in \Ical$ with $X'\subseteq X\cup\SET{y}$ such that
$\rk(X\cup\SET{y}) = \left|X'\right|$. By {\em (I2)} $X'\BSET{y}\in\Ical$,
therefore $\rk(X\cup\SET{y}) \leq \rk(X) + 1$. Since every subset of $X$ is a
subset of $X\cup\SET{y}$, too, we obtain {\em (R2')} for $\rk$: $\rk(X)\leq
\rk(X\cup\SET{y}) \leq \rk(X) + 1$.

\bSep \PRFRA
We prove {\em (R3')} via contraposition and show that $\rho(X) \not= \rho(X\cup\SET{x,y})$ implies that $\rho(X)\not=\rho(X\cup\SET{x})$
or $\rho(X)\not=\rho(X\cup\SET{y})$.
We may assume the non-trivial case $y,z\notin X$.
If $\rk(X\cup\SET{y,z}) > \rk(X)$, then every $X'\subseteq X\cup\SET{y,z}$,
which has maximal cardinality such that $X'\in\Ical$, must have a non-empty
intersection $X'\cap \SET{y,z}\not= \emptyset$, because $X'\not\subseteq X$.
Without loss of generality we may assume that $y\in X'$. If $y=z$ or $z\notin X'$ or $\rk(X) =
\rk(X\cup\SET{y,z})-2$, we obtain that $\rk(X\cup\SET{y})=\rk(X)+1$. The
remaining case is that $\dSET{y,z} \subseteq X'$ and $\rk(X) = |X'|-1$. Let
$\tilde{X}\subseteq X$ be a subset with maximal cardinality such that it is still independent, i.e. $\tilde{X}\in\Ical$.
Since $X'\BSET{y,z}\in\Ical$, {\em (I3)} yields that there is an $x\in\tilde{X}\BS X'$
such that $\left( X'\BSET{y,z}\right) \cup \SET{x}\in \Ical$. Applying {\em
(I3)} again yields that either $\left(X'\BSET{y}\right) \cup \SET{x}\in \Ical$
or $\left( X'\BSET{z}\right) \cup \SET{x} \in \Ical$, therefore either $\rk(X)
< \rk(X\cup\SET{y})$ or $\rk(X) < \rk(X\cup\SET{z})$. This establishes {\em (R3')}.

\bigskip
\noindent \underline{The implication {\em (ii) $\Rightarrow$ (iii)}:} \PRFRA

\bSep We show {\em (R1)} by
induction on $|X|$. From {\em (R1')} we obtain $0 \leq \rho(\emptyset) = 0
\leq \left|\emptyset\right|$. Now, let $X\subseteq E$ and $x\in X$. By
induction hypothesis, we have $0\leq \rho(X\BSET{x}) \leq
\left|X\BSET{x}\right|=\left|X\right| - 1$. {\em (R2')} yields
$\rho(X\BSET{x}) \leq \rho(X) \leq \rho(X\BSET{x}) + 1$, which combines with
the previous inequality to the desired $ 0 \leq \rho(X\BSET{x}) \leq \rho(X) \leq \left(
\left|X\right| - 1\right) + 1 = \left| X \right|$.

\bSep \PRFRA
In order to show {\em (R2)}
it suffices to consider $X\subseteq Y\subseteq E$. We prove $\rho(X) \leq
\rho(Y)$ by induction on $\left| Y\BS X \right|$. The base case implies $X=Y$
thus $\rho(X)\leq \rho(Y)$ holds trivially. Now let $y\in Y\BS X$. By
induction hypothesis, $\rho(X) \leq \rho(Y\BSET{y})$ holds. From {\em (R2')}
we obtain $\rho(Y\BSET{y}) \leq \rho(Y)$, and thus $\rho(X) \leq
\rho(Y\BSET{y}) \leq \rho(Y)$ holds.

\PRFRA
\bSep We prove that the following auxiliary property ...
\begin{enumerate}[align=parleft,leftmargin=2cm,labelsep=1.5cm]\label{n:R2pp}
  \item[\em (R2'')] If $\rho(X\cup\SET{y})=\rho(X)+ 1$ and $X'\subseteq X$, then $\rho(X'\cup\SET{y})=\rho(X')+ 1$; \\ for all $X\subseteq E$, $y\in E$.
\end{enumerate}
... follows from {\em (ii)} by induction on
 $\left| X\BS X'\right|$. The base case $X=X'$ is trivial. For the induction step, let $x\in X\BS X'$, and
 assume that the implication is not vacuously true. By induction hypothesis
 $\rho(X'\cup\SET{x,y}) = \rho(X'\cup\SET{x}) + 1$. Using {\em (R2')} we obtain the inequalities $\rho(X') \leq \rho(X'\cup\SET{x}) \leq \rho(X') + 1$, similarly
 $\rho(X') \leq \rho(X'\cup\SET{y}) \leq \rho(X') + 1$, and furthermore
 $\rho(X'\cup\SET{y})\leq \rho(X'\cup\SET{x,y})\leq \rho(X'\cup\SET{y}) + 1$. 
 We establish {\em (R2'')} by the following case analysis:
 \begin{enumerate}
  \item[{\em (a)}] $\rho(X'\cup\SET{x}) = \rho(X') + 1$, by induction hypothesis $\rho(X'\cup\SET{x,y})=\rho(X')+2$ and as a consequence of the last inequality $\rho(X'\cup \SET{y}) = \rho(X') + 1$.
  \item[{\em (b)}] $\rho(X'\cup\SET{x}) = \rho(X')$. If we assume that $\rho(X'\cup\SET{y})=\rho(X')$, we could use {\em (R3')} in order to deduce
  $\rho(X'\cup\SET{x,y})=\rho(X')$, which would contradict the induction hypothesis. Therefore,
  $\rho(X'\cup\SET{y})=\rho(X')+1$.
\end{enumerate} 

\PRFRA
\bSep
In order to show that {\em (R3)} holds for all $X,Y\subseteq E$, we may use an
inductive argument over $(\left|X\BS Y\right|,\left|Y\BS X\right|)$ with
respect to the well-founded natural coordinate-wise partial order. The base case
$\left|X\BS Y\right| = 0 = \left|Y\BS X\right|$ implies that $X=Y$ and
therefore $\rho(X\cap Y)+\rho(X\cup Y)= 2\rho(X) = \rho(X) + \rho(Y)$ holds.
Due to the commutativity of the operations $\cap$, $\cup$, and $+$, it suffices to proof the
induction step from $(X\BSET{x},Y)$ to $(X,Y)$ for $x\in X\BS Y$, as the step
from $(X,Y\BSET{y})$ to $(X,Y)$ for $y\in Y\BS X$ follows symmetrically. By
induction hypothesis, we may assume that $\rho\left(\left(X\BSET{x}\right)\cup
Y\right) + \rho\left(\left(X\BSET{x}\right)\cap Y\right) \leq \rho(X\BSET{x})
+ \rho(Y)$ holds. Since $x\in X\BS Y$, we see that $x\notin Y$ and thus
$\left(X\BSET{x}\right)\cap Y = X\cap Y$ as well as
$\left(X\BSET{x}\right)\cup Y = \left(X\cup Y\right)\BSET{x}$, so we may
write the induction hypothesis as $\rho\left(\left(X\cup Y\right)\BSET{x}
\right) + \rho(X\cap Y) \leq \rho\left(X\BSET{x}\right) + \rho(Y)$. Property
{\em (R2')}  implies that \linebreak $\rho(X\cup Y) = \rho\left(\left(X\cup
Y\right)\BSET{x} \right) + \alpha$ and $\rho(X) = \rho(X\BSET{x})+ \beta$ for
some $\alpha,\beta\in\SET{0,1}$. The desired inequality $\rho(X\cup
Y)+\rho(X\cap Y)\leq \rho(X)+ \rho(Y)$ follows from the fact that $\alpha \leq
\beta$, which is a consequence of {\em (R2'')} where $X\BSET{x}$ takes the
role of $X'$, $\left(X\BSET{x}\right)\cup Y$ takes the role of $X$ and $x$
takes the role of $y$.

\bigskip
\noindent \underline{The implication {\em (iii) $\Rightarrow$ (i)}:} 

\PRFRA
\bSep First, we prove that {\em (iii)} implies property {\em (R2')} that
$\rho$ is unit-increasing, let $X\subseteq E$ and $y\in E$. If $y\in X$ the
property holds trivially, let $y\notin X$. The first inequality $\rho(X)\leq
\rho(X\cup\SET{y})$ holds due to {\em (R2)}. With {\em (R3)} we obtain
$\rho(X\cup\SET{y}) + \rho(X\cap \SET{y}) \leq \rho(X) + \rho(\SET{y})$, and
since $X\cap\SET{y}=\emptyset$ we may use {\em (R1)} twice to obtain
$\rho(\SET{y}) \leq 1$ and $\rho(\emptyset) = 0$, from which we may infer the
second inequality of {\em (R2')}, namely $\rho(X\cup\SET{y}) \leq \rho(X) +
1$.

\PRFRA
\bSep We prove that {\em (iii)} implies property
\begin{enumerate}[align=parleft,leftmargin=2cm,labelsep=1.5cm]\label{n:R4}
  \item[\em (R4)] $\left(\forall y\in Y\colon\,\rho(X\cup\SET{y}) = \rho(X)\right) \Rightarrow \rho(X\cup Y)= \rho(X)$ for all $X,Y\subseteq E$.
\end{enumerate}
By induction on $|Y\BS X|$. The base cases $|Y\BS X|\in \SET{0,1}$ are trivial. Now
let $v,w\in Y\BS X$. By induction hypothesis, $\rho(X) = \rho(X\cup Y \BSET{v}) = \rho(X\cup Y \BSET{w}) = \rho(X\cup Y\BSET{v,w})$. Using {\em (R3)} we obtain
$\rho(X\cup Y\BSET{v,w}) + \rho(X\cup Y) \leq \rho(X\cup Y\BSET{v}) + \rho(X\cup Y\BSET{w})$. Together with the induction hypothesis we get $\rho(X\cup Y) \leq \rho(X)$ and the property {\em (R2)} that $\rho$ is isotone yields $\rho(X\cup Y)=\rho(X)$.

\PRFRA
\bSep
Next, we prove that {\em (iii)} also implies the following property:
\begin{enumerate}[align=parleft,leftmargin=2cm,labelsep=1.5cm]\label{n:R5}
\item[\em (R5)] For every $X\subseteq E$ there is a subset $X'\subseteq X$, such that
  $\left|X'\right| = \rho(X') = \rho(X)$.
  %\item[\em (R5)] There is a subset $X'\subseteq X$, such that
  %$\left|X'\right| = \rho(X') = \rho(X)$, for all $X\subseteq E$.
\end{enumerate}
By induction on $|X|$.  The base case $\rho(\emptyset) = 0 =
\left|\emptyset\right|$ is clear. Now let $x\in X$ and by induction
hypothesis, there is a subset $X'\subseteq X\BSET{x}$ such that
$\left|X'\right| = \rho(X') = \rho(X\BSET{x})$. From {\em (R2')} we conclude
that $\rho(X)=\rho(X\BSET{x}) + \alpha$ for some $\alpha\in\SET{0,1}$. The
case $\alpha=0$ is trivial. For the case $\alpha=1$ we give an indirect
argument: Assume that $\rho(X'\cup\SET{x}) = \rho(X') = \rho(X\BSET{x})$. Then
$\rho(X) = \rho(X\BSET{x})$ follows from {\em (R4)}, because for every $y\in
X\BS X'$ we have $\rho(X'\cup\SET{y}) = \rho(X)$. Yet, this is a
contradiction to $\rho(X) = \rho(X\BSET{x}) + 1$, therefore
$\rho(X'\cup\SET{x}) = \rho(X') + 1$ follows from {\em (R2')}, thus
$|X'\cup\SET{x}|=\rho(X'\cup\SET{x}) = \rho(X)$.

\PRFRA
\bSep From $\rho$, we define the set system $\Ical = \SET{X\subseteq E\mid \rho(X) = \left|X\right|}$. For now, let us assume that $M=(E,\Ical)$ is indeed a matroid. An immediate consequence of property {\em (R5)} is that $\rho(X) \leq \rk_{M}(X)$ for all $X\subseteq E$.
By definition of $\rk_{M}$, there is a subset $X'\subseteq X$ such that  $\rk_{M}(X) = \left|X'\right| = \rho(X') \leq \rho(X)$ due to {\em (R2)}. Thus $\rho=\rk_{M}$.

\PRFRA
\bSep By {\em (R1)} we have $\rho(\emptyset) = 0 = \left| \emptyset \right|$, thus $\emptyset \in \Ical$, so {\em (I1)} holds.

\PRFRA
\bSep Let $X\in \Ical$. We show that $X'\in\Ical$
for all $X'\subseteq X$ by induction on $\left|X\BS X'\right|$. The base case
$X'=X$ is trivial. Now let $x\in X\BS X'$. By induction hypothesis,
$X'\cup\SET{x}\in \Ical$, therefore $\rho(X'\cup\SET{x}) = \left|X'\right| + 1$.
 From {\em (R1)} we get the inequality $\rho(X') \leq \left| X' \right|$, and from
 {\em (R2')} we get the inequality $\rho(X'\cup\SET{x}) \leq \rho(X') + 1$. 
 Thus $\rho(X') = \left|X'\right|$ follows, consequently $X'\in \Ical$, so {\em (I2)} holds.

\PRFRA
 \bSep We give an indirect argument for {\em (I3)}. Let $X,Y\in \Ical$ with $\left|X\right| < \left|Y\right|$, and assume that for all $y\in Y$, $X\cup\SET{y}\notin \Ical$.
 Since $\left| X\right| = \rho(X)$ and by {\em (R2)} $\rho$ is isotone, we can infer that
 $\rho(X\cup\SET{y}) = \rho(X)$ for all $y\in Y$. With {\em (R4)} we see that
 $\rho(X\cup Y)=\rho(X)$, and together with {\em (R2)} we obtain $\rho(Y) \leq \rho(X\cup Y) = \rho(X) = \left| X \right| < \left| Y \right|$, a contradiction to $Y\in\Ical$.
 We may now conclude that $M=(E,\Ical)$ is a matroid.
\end{proof}

\studyremark{{\em (R1)}, {\em (R2)}, {\em (R2')}, {\em (R5)} is not sufficient for $\rho$ to
be the rank function of a matroid! $X\mapsto \max \SET{\left| X\cap A \right|, \left|X \cap B\right|}$ with $A = \SET{a,b}$, $B = \SET{c,d,e}$; $\SET{a,c}$ and $\SET{a,d}$ violate submodularity}
% -*- root: ../thesis.tex -*-

\needspace{4\baselineskip}
\subsection{Matroids Induced From Submodular Functions}

\begin{definition}\PRFRB
  Let $E$ be any set, $R\subseteq \R$, and let $f\colon 2^{E} \maparrow R$.
  We call the map $f$ \deftext{non-decreasing}, if for every $X\subseteq Y\subseteq E$,
  the inequality $f(X) \leq f(Y)$ holds.
  We call $f$ \deftext{submodular}, if for all $X,Y\subseteq E$
  the inequality
  \( f(X\cap Y) + f(X\cup Y) \leq f(X) + f(Y) \)
  holds. 
\end{definition}

\begin{example}
Let $M=(E,\Ical)$ be a matroid. Then $\rk_{M}\colon 2^{E}\maparrow \N$ is a non-decreasing and submodular function.
\end{example}

\needspace{3\baselineskip}

\noindent
The following theorem is the independent-sets version of Proposition~11.1.1 in \cite{Ox11},
which is attributed to J.~Edmonds and G.C.~Rota.

\begin{theorem}\label{thm:fromsubmodular}\PRFRB
  Let $E$ be a finite set, and let $f\colon 2^{E} \maparrow \Z$ be
  a non-decreasing, submodular function. Then $M=(E,\Ical)$
  where \[
      \Ical = \SET{X\subseteq E ~~\middle|~~ \forall X'\subseteq X\colon\, X'\not=\emptyset \Rightarrow f(X') \geq \left| X' \right|}
  \] is a matroid.
\end{theorem}

\begin{proof}\PRFRB
  From the definition, it is clear that $\emptyset \in \Ical$ {\em (I1)} as well as that 
  for every $X\in \Ical$ and every $Y\subseteq X$, $Y\in \Ical$ {\em (I2)}. We have to show that {\em (I3)} holds for $\Ical$, too. Let $X,Y\in \Ical$ with $\left| X \right| < \left| Y \right|$. 
  We give an indirect argument. 
  Assume that for all $y\in Y\BS X$,\linebreak $X\cup \SET{y}\notin \Ical$. Since $
   \left| X\BS Y \right| + \left|X\cap Y \right| = \left| X \right| < \left| Y \right| = \left| Y\BS X\right| + \left|X\cap Y\right|$, we have\linebreak $\left| Y \BS X \right| > \left| X \BS Y \right|$. For every $y\in Y\BS X$, there is a subset $X_{y}\subseteq X$ with minimal cardinality
   such that $f\left(X_y \cup \SET{y}\right) < \left| X_{y} \right| + 1$. Since $Y\in \Ical$
   we obtain that $X_{y} \cap \left(X\BS Y\right) \not= \emptyset$. Therefore by a simple counting argument, there are $y_{1},y_{2}\in Y\BS X$ such that there is some $x\in X\BS Y$ with  $x \in X_{y_1}\cap X_{y_2}$. 
   Below, we first use that $f$ is non-decreasing, then that $f$ is submodular, and then the fact that $X_{y_1}\cap X_{y_2}\in \Ical$ from {\em (I2)} and $X\in \Ical$; finally, we use the fact that neither $X\cup\SET{y_1}\in \Ical$ nor
   $X\cup\SET{y_2}\in\Ical$ and that $f$ is integer-valued:
    \begin{align*}
      f\left(
       \left(X_{y_1}\cup X_{y_2}\cup\SET{y_1,y_2}
       \right) \BSET{x}
       \right) & \leq f\left(X_{y_1}\cup X_{y_2}\cup\SET{y_1,y_2}
                      \right) \\
      & \leq f\left(X_{y_1}\cup \SET{y_1}
             \right) + f\left(X_{y_2}\cup \SET{y_2}
                       \right) 
      - f\left(X_{y_1}\cap X_{y_2}
         \right) \\
      & \leq f\left(X_{y_1}\cup \SET{y_1}\right) + f\left(X_{y_2}\cup \SET{y_2}\right) - \left| X_{y_1}\cap X_{y_2}\right| \\
      & \leq \left| X_{y_1} \right| + \left| X_{y_2} \right| - \left| X_{y_1} \cap X_{y_2}\right| \\ & = \left| X_{y_1} \cup X_{y_2} \right| 
     \\& = \left|\left(X_{y_1}\cup X_{y_2}\cup\SET{y_1,y_2}\right) \BSET{x}\right| - 1.
    \end{align*}
   Thus there must be a subset of minimal cardinality $C \subseteq \left( X_{y_1}\cup X_{y_2}\cup \SET{y_1,y_2} \right)\BSET{x}$ such that $f(C) < \left| C \right|$. Then $C \cap \SET{y_1,y_2} = \emptyset$ because otherwise $C$ would contradict the minimality of the cardinalities of $X_{y_1}$ and $X_{y_2}$, respectively. But then the fact that $C\subseteq X_{y_1}\cup X_{y_2}\subseteq X$ would contradict $X\in \Ical$. Therefore there must be some $y\in Y\BS X$ such that $X\cup\SET{y}\in\Ical$.
\end{proof}

%\studyremark{
%Nochmal prüfen ob das hier so korrekt ist/und ob es nicht überflüssig ist...
%\begin{remark}
%  In the previous proof we used the integrality of $f$ in the way that $f(X) < k$ holds, if and only if $f(X) \leq k - 1$ holds. The proof generalizes easily to all real valued non-decreasing and submodular functions $f\colon 2^{E}\maparrow \R$ with the property that for any pair $X,Y\in\Ical$ with $\left| X \right| < \left| Y \right|$ there
%  are $y_{1},y_{2}\in \Ical$ and subsets $X_{y_1},X_{y_2}\subseteq X$ of minimal cardinality 
%  with $f(X_{y_i}\cup\SET{y_i}) < \left| X_{y_i} \right|$ ($i\in\SET{1,2}$); which also have the additional property that $f(X_{y_1}\cup\SET{y_1}) + f(X_{y_2}\cup\SET{y_2})
%  - \left| X_{y_2} \right| - \left| X_{y_2} \right| \leq 1$. A sufficient condition for this would be $f(X) - \lfloor f(X) \rfloor \leq \frac 1 2$ for all $X\subseteq E$.
%\end{remark}}

\noindent If we restrict $f$ to be a map into the non-negative integers, we may simplify the 
expression that gives $\Ical$ analogously to Corollary~8.1 \cite{We76}.

\begin{theorem}\label{thm:submodularIndependent}\PRFRB
   Let $E$ be a finite set, and let $f\colon 2^{E} \maparrow \N$ be
  a non-decreasing, submodular function. Then $M=(E,\Ical)$
  where \[
      \Ical = \SET{\vphantom{A^A}X\subseteq E ~\middle|~ \forall X'\subseteq X\colon f(X') \geq \left| X' \right|}
  \] is a matroid. If furthermore $f(\emptyset) = 0$, then its rank function is given by
  \[ \rk(X) = \min \SET{\vphantom{A^A}f(Y) + \left| X \BS Y \right| ~\middle|~ Y\subseteq X}.\]
\end{theorem}

\begin{proof}\PRFRB
  Let $\Ical' =  \SET{X\subseteq E \mid \forall X'\subseteq X\colon X'\not=\emptyset \Rightarrow f(X') \geq \left| X' \right|}$ corresponding to Theorem~\ref{thm:fromsubmodular}.
  From the definitions, it is clear that $\Ical \subseteq \Ical'$. From inspection we obtain that if $X \in \Ical' \BS \Ical$, then $f(\emptyset) < \left| \emptyset \right| = 0$. But this is impossible for $f(\emptyset)\in\N$. Thus $\Ical = \Ical'$. 
%

 %\noindent
 The second part of the proof follows the ideas from \cite{Dun76} where a more general statement is proved.\footnote{Both D.J.A.~Welsh and F.D.J.~Dunstan cite a conference abstract of the {\em Waterloo Conference on Combinatorics 1968} by J.~Edmonds and G.C.~Rota who proved that for submodular, non-decreasing, integer-valued $f$ the rank function is given by $\rk(X) = \min\SET{\left| X \right|, f(Y)-\left| X\BS Y \right|~\middle|~ \vphantom{A^A} Y\subseteq X}$. Unfortunately, we were not able to get a copy of that abstract.}
Now let us assume that we have the further property $f(\emptyset) = 0$, we shall now prove the rank formula. We will denote the formula given in the statement of the theorem by $\rk$, whereas we
are denoting the rank formula from Definition~\ref{def:rank} by $\rk_{M}$.
First, we want to show that $\rk$ is non-decreasing. Let $X'\subseteq X\subseteq E$,
we do induction on $\left| X\BS X' \right|$. The base case is trivial. Now,
let $X\subseteq E$, and $x\in X$, the induction hypothesis yields that $\rk(X') \leq \rk(X\BSET{x})$. If there is a subset $Y\subseteq X$ with $x\in Y$, such that
 $\rk(X) = f(Y) + \left| X\BS Y \right|$,
 then since  $f(Y\BSET{x}) \leq f(Y)$ we obtain that $$\rk(X\BSET{x}) \leq f(Y\BSET{x}) + \left| (X\BSET{x})\BS(Y\BSET{x}) \right|
 \leq f(Y) + \left| X\BS Y \right| = \rk(X).$$
 Otherwise let $Y\subseteq X\BSET{x}$ be a subset such that
 $\rk(X) = f(Y) + \left| X\BS Y \right|$. 
 Then $$\rk(X\BSET{x}) \leq f(Y) + \left| (X\BSET{x})\BS Y \right| < f(Y) + \left| X\BS Y \right| = \rk(X),$$
 thus in any case $\rk(X\BSET{x}) \leq \rk(X)$, so $\rk$ is non-decreasing.
 Now, in order to show that $\rk_M(X) \leq \rk(X)$ for all $X\subseteq E$, it suffices to show that $\rk(X) = \left| X \right|$ for all independent $X\subseteq E$.
 Let $X\in\Ical$. By definition of $\Ical$, for all
$Y\subseteq X$, $\left| Y \right| \leq f(Y)$ holds. Thus for any $Y\subseteq X$, we have
\[ \left| X \right| = \left| Y \right| + \left|X\BS Y \right| \leq f(Y) + \left|X\BS Y\right|.\]
Therefore the minimum in the expression for $\rk(X)$ is attained for $Y=\emptyset$, i.e.
  \[ \rk(X) = \min
\SET{\vphantom{A^A}f(Y) + \left| X \BS Y \right| ~\middle|~ Y\subseteq X} = f(\emptyset) + \left|
X\BS \emptyset \right| = \left| X \right|.\] 
To complete the proof that $\rk = \rk_M$, we have to show that for every $X\subseteq E$,
there is a subset $Y\subseteq X$ such that $Y\in \Ical$ and $\rk(X) = \left| Y \right|$.
Let $Z\subseteq X$ such that $Z$ has maximal cardinality with $Z\in \Ical$.
We give an indirect argument and assume that $\rk_M(X) = \left| Z \right| < \rk(X) \leq \left| X \right|$. Since $Z$ is maximally independent in $X$, for every
$x\in X\BS Z$ there must be a subset $Z_x\subseteq Z$ such that
$f(Z_x\cup\SET{x}) < \left| Z_x\cup\SET{x} \right| = \left| Z_x \right| + 1$. From $Z\in\Ical$ we may infer that $f(Z_x) \geq \left| Z_x \right|$, thus we have
$f(Z_x\cup\SET{x}) = \left| Z_x \right| = f(Z_x)$ due to $f$ being non-decreasing and integer-valued.
We show the auxiliary claim that for all $X'\subseteq X\BS Z$, $f(Z) = f(Z\cup X')$, by
induction on $\left| X' \right|$. The base case is trivial. For the induction step, 
let $x'\in X'$, and by induction hypothesis we may assume that $f(Z) = f(Z\cup (X'\BSET{x'}))$. Let $Z_{x'}\subseteq Z$ such that $f(Z_{x'}\cup\SET{x'}) = \left| Z_{x'} \right| = f(Z_{x'})$ as above. Since $f$ is submodular, we obtain that $f(Z\cup(X'\BSET{x'})) + f(Z_{x'}\cup\SET{x'}) \geq f(Z_{x'}) + f(Z\cup X')$, and along with the previous equation this yields $f(Z\cup(X'\BSET{x'})) \geq f(Z\cup X')$. So, together with the property that $f$ is
 non-decreasing and with the induction hypothesis, 
 we obtain the desired equation $f(Z\cup X') = f(Z\cup(X'\BSET{x'})) = f(Z)$.
 But now, let $X'=X\BS Z$, then $Z\cup X' = X$. We obtain from the auxiliary claim above, that $f(Z)=f(X)$, so
 that, by construction as a minimum, $\rk(X) \leq f(X) + \left| X\BS X \right| = f(Z) = \left| Z \right|$. Yet this
 contradicts $\left| Z \right| < \rk(X)$. Thus there is an independent subset of $X$ with cardinality $\rk(X)$ for every $X\subseteq E$, and therefore $\rk = \rk_M$.
 \end{proof}

% -*- root: ../thesis.tex -*-

\subsection{Dual Matroids}

\begin{definition}\PRFRB
	Let $M=(E,\Ical)$ be a matroid. We call $X\subseteq E$ \deftext{spanning} in $M$, if
	there is a base $B$ of $M$, such that $B \subseteq X$.
\end{definition}

\begin{lemma}\label{lem:basesminimalspanning}\PRFRB
	Let $M=(E,\Ical)$ be a matroid, $X\subseteq E$.
	Then $X$ is a base if and only if $X$ is spanning in $M$, yet for all $x\in X$,
	$X\BSET{x}$ is not spanning in $M$.
\end{lemma}
\begin{proof}\PRFRB
	Let $B\in\Ical$ be a base of $M$,
	then $\rk(B)=\left| B \right|$ is maximal, so $\cl(B) = E$.
	On the other hand, for every $b\in B$, $\rk(B\BSET{b}) < \left| B \right|$,
	thus $b\notin \cl(B\BSET{b})$, so $\cl(B\BSET{b})\not= E$.
	Let $X\subseteq E$ such that $X\notin\Bcal(M)$.
	If further $\rk(X) < \rk(E)$, then $X$ clearly is not a spanning set in $M$. 
	Now assume that $\rk(X) = \rk(E)$, so $X$ is spanning in $M$, and because it is not a base, $X\notin \Ical$. But then there is a base $B\subsetneq X$ with $\cl(B) = \cl(X)$ (Lemma~\ref{lem:clofbase}). So there is some $x\in X\BS B$, such that $X\BSET{x}$ still contains the base $B$ and therefore $X\BSET{x}$ is spanning in $M$.
\end{proof}

\noindent Matroids allow to be axiomatized cryptomorphically by characterizing the set of bases of $M$. For full disclosure on this topic we would like to refer the reader the first chapters in \cite{We76} and \cite{Ox11}.

\needspace{6\baselineskip}

\begin{theorem}\label{thm:frombases}\PRFRB
	Let $E$ be a finite set, $\Ical\subseteq 2^E$. Let further
	\[ \Bcal = \SET{\vphantom{A^A} X\in\Ical ~\middle|~ \nexists Y\in \Ical\colon\, X \subsetneq Y} \]
	be the family of maximal elements of $\Ical$ with regard to set-inclusion.
	 If 
	\begin{enumerate}[align=parleft,leftmargin=2cm,labelsep=1.5cm]\label{n:Bx}
		\item[(B1)] $\Bcal\not= \emptyset$,
		\item[(B2)] $\forall X,Y\in\Bcal\colon\,\left| X \right| = \left| Y \right|$, and
		\item[(B3)] for all $X,Y\in \Bcal$ and all $x\in X\BS Y$, there is an element $y\in Y\BS X$, such that $(X\BSET{x})\cup\SET{y} \in \Bcal$
	\end{enumerate}
	holds, and if $\Ical = \SET{\vphantom{A^A}X\subseteq E~\middle|~ \exists B\in \Bcal\colon\,X\subseteq B}$, then $M=(E,\Ical)$ is a matroid.
\end{theorem}

\begin{proof}\PRFRB
	From {\em (B1)} we obtain $B\in\Bcal$, and clearly $\emptyset \subseteq B$, so
	$\emptyset\in \Ical$ {\em (I1)}. Let $X\in\Ical$, then there is some $B\in\Bcal$ with $X\subseteq B$. For $Y\subseteq X$ we have $Y\subseteq B$ and therefore $Y\in\Ical$ {\em (I2)}.
	Let $X,Y\in\Ical$ with $\left| X \right| < \left| Y \right|$. 
	There are $B_X\supseteq X$ and $B_Y\supseteq Y$ with $B_X,B_Y\in \Bcal$.
	\linebreak
	If $Y' = B_X\cap(Y\BS X)\not= \emptyset$, then let $y\in Y'$ be an arbitrary choice,
	and we obtain
	\linebreak
	 $X\cup\SET{y}\subseteq B_X$ therefore $X\cup\SET{y}\in\Ical$.
	If $Y'=\emptyset$, then let $\alpha(B_X) = \left| (B_Y\BS B_X) \BS (Y\BS X)\right|$,
	we prove that we can augment $X$ by induction on $\alpha(B_X)$.
	 Since $\left| X \right| < \left| Y \right| \leq \left| B_Y \right| = \left| B_X \right|$, there is an element $x'\in B_X\BS X$.
	 % Also $X\cap Y\subseteq B_X\cap B_Y$, therefore 
	 We may use {\em (B3)} in order to obtain the base
	$B_X'=(B_X\BSET{x'})\cup\SET{y} \in \Bcal$ where $y\in B_Y\BS B_X$. If $y\in Y\BS X$, then $X\cup\SET{y}\subseteq B_X'$ and therefore $X\cup\SET{y}\in \Ical$. Otherwise $y\in (B_Y\BS B_X)\BS(Y \BS X)$, then $\alpha(B_X') = \alpha(B_X) - 1$ and thus there is some $y\in Y\BS X$ with $X\cup\SET{y}\in \Ical$ by the induction hypothesis. Thus
	$\Ical$ has the property {\em (I3)}.
\end{proof}

\begin{definition}\PRFRB
	Let $M=(E,\Ical)$ be a matroid. The \deftext{dual matroid} of $M$ shall be the pair
	 \( M^{\ast} = (E,\Ical^\ast)\) where\label{n:Mdual}
	 \[ \Ical^{\ast} = \SET{\vphantom{A^A} E\BS X ~\middle|~ X\subseteq E,\, \text{such that~}X\text{~is spanning in~}M}. \qedhere\]
\end{definition}

\needspace{7\baselineskip}

\begin{lemma}\label{lem:spanningdual}\PRFRB
  Let $M=(E,\Ical)$ be a matroid. Then  $M^{\ast}=(E,\Ical^\ast)$ is indeed a matroid.
\end{lemma}

\begin{proof}\PRFRB
	First, observe that for $\Bcal^\ast = \SET{E\BS B ~\middle|~\vphantom{A^A} B\text{ is a base of } M}$
	we have the set equation
	\[ \Ical^\ast = \SET{X\subseteq E~\middle|~\vphantom{A^A} \exists B'\in\Bcal^\ast\colon\,X\subseteq B'}, 
	\]
	because the minimal spanning sets of $M$ are precisely the bases of $M$, which in turn have complements in $E$ with maximal cardinality. Since $\emptyset\in \Ical$ implies that $M$ has at least one base, 
	we have $\Bcal^\ast \not= \emptyset$ {\em (B1)}.
	From Corollary~\ref{cor:equicardinality} we obtain that for any two $B,B'\in \Bcal^\ast$,
	we have $B_0,B_0'$ that are bases of $M$ with $B=E\BS B_0$ and $B'=E\BS B_0'$, therefore
	 $\left| B \right| =  \left| E \right|-\left| B_0 \right| = \left| E \right| - \left| B_0' \right| = \left| B' \right|$, so {\em (B2)} holds.
	 Now let $x\in B\BS B' = (E\BS B_0)\BS(E\BS B_0') = B_0'\BS B_0$, then there is a
	 $y\in B_0\BS B_0' = (E\BS B_0')\BS(E\BS B_0) = B'\BS B$ such that $(B_0\BSET{y})\cup\SET{x}$ is a base of $M$ (Lemma~\ref{lem:basisexchangesymmetric}). But then
	 \begin{align*}
	  E\BS \left( (B_0\BSET{y})\cup\SET{x}\right) & = E\BS \left( \left( B_0\cup\SET{x} \right) \BSET{y} \right)\\  & =
	  \left( E\BS(B_0\cup\SET{x}) \right)\cup\SET{y} \\ & = (B\BSET{x})\cup\SET{y} \in \Bcal^\ast.
	  \end{align*}
	  So {\em (B3)} holds for $\Bcal^\ast$, too, and from Theorem~\ref{thm:frombases}
	  we obtain that $M^\ast=(E,\Ical^\ast)$ is a matroid.
\end{proof}

\begin{corollary}\label{cor:dualbase}\PRFRB
	Let $M=(E,\Ical)$ be a matroid, $B\subseteq E$. Then
	$B$ is a base of $M$ if and only if $E\BS B$ is a base of $M^\ast$.
\end{corollary}
\begin{proof}\PRFRB
	Let $(E,\Ical') = M^\ast$.
	If $B$ is a base of $M$, then for all $b\in B$, $B\BSET{b}$ is not spanning $M$ (Lemma~\ref{lem:basesminimalspanning}),
	 therefore $E\BS B\in \Ical'$, yet $(E\BS B)\cup\SET{b}\notin \Ical'$, therefore $E\BS B$ is maximally independent with respect to set-inclusion, 
	and thus it is an independent set of $M^\ast$ 
	with maximal cardinality (Corollary~\ref{cor:equicardinality}), so $E\BS B$ is a base of $M^\ast$.
	Conversely, if $E\BS B$ is a base of $M^\ast$, then $E\BS(E\BS B) = B$ must be minimally spanning in $M$, since otherwise
	$E\BS \left( B\BSET{x} \right) \in \Ical'$ for some $x\in B$ contradicting the maximality of $E\BS B$ in $\Ical'$.
	Thus $B$ is a base of $M$ (Lemma~\ref{lem:basesminimalspanning}).
\end{proof}

\begin{corollary}\label{cor:doubleast}\PRFRB
	Let $M=(E,\Ical)$ be a matroid. Then $M = \left(M^{\ast}\right)^{\ast}$.
\end{corollary}

\begin{proof}\PRFRB
	By property {\em (I2)}, the family of independent sets of a matroid is determined by its maximal elements, which are the bases of $M$.
	By Corollary~\ref{cor:dualbase}, $B$ is base of $M$, if and only if $E\BS B$ is a base of $M^\ast$, if and only if $E\BS(E\BS B) = B$
	is a base of $\left( M^\ast \right)^\ast$. Thus $M=\left( M^\ast \right)^\ast$.
\end{proof}

\noindent The next two lemmas can be found in J.G.~Oxley's book (\cite{Ox11}, p.67) and yield an elegant way to characterize the rank function of the dual matroid in terms of the rank function of the primal matroid.

\begin{lemma}\label{lem:dualAug}\PRFRB
	Let $M=(E,\Ical)$ be a matroid, $X,Y\subseteq E$ with $X\cap Y = \emptyset$ such that
	$X\in \Ical$ is independent in $M$ and $Y\in\Ical^\ast$ is independent in $M^\ast$.
	Then there is a base $B\subseteq E$ of $M$ such that $X\subseteq B$ and $Y\subseteq E\BS B$.
\end{lemma}
\begin{proof}\PRFRB
	Let $B$ be a base of $E\BS Y$ in $M$ such that $X\subseteq B$ (Lemma~\ref{lem:augmentation}).
	Then $Y\subseteq E\BS B$. It remains to show that $B$ is a base of $M$.
	 Assume that $B$ is not a base of $M$, then $\rk_M(E\BS Y) < \rk_M(E)$. 
	 But $Y\in\Ical^\ast$, therefore $E\BS Y$ is spanning in $M$ -- a contradiction. 
	 Thus $B$ is the desired base of $M$.
\end{proof}

\needspace{6\baselineskip}
\begin{lemma}\label{lem:rankDual}\PRFRB
	Let $M=(E,\Ical)$ be a matroid and $X\subseteq E$.
	Then
	\[ \rk_{M^\ast}(X) = \left| X\right| + \rk_M(E\BS X) - \rk_M(E) .\]
\end{lemma}
\begin{proof}\PRFRB
	Let $B'_X\subseteq X$ be a base of $X$ in $M^\ast$,
	and $B_{E\BS X}\subseteq E\BS X$ be a base of $E\BS X$ in $M$.
	Then $\rk_{M^\ast}(X) = \left| B'_X \right|$ and $\rk_{M}(E\BS X) = \left| B_{E\BS X} \right|$.
	Clearly $B'_X\cap B_{E\BS X} = \emptyset$, therefore there is a base $B$ of $M$ such that $B_{E\BS X} \subseteq B$ and $B'_X \subseteq E\BS B$ (Lemma~\ref{lem:dualAug}). Since $B_{E\BS X}$ is a base of $E\BS X $ in $M$, we have that $B\cap \left( E\BS X \right) = B_{E\BS X}$, and analogously,
	$\left( E\BS B \right)\cap X = B'_X$. We obtain $B\cap X = X\BS B'_X$ and therefore $B= B_{E\BS X} \disunion \left( X\BS B'_X \right)$, so
	\[ \rk_M(E) = \left| B \right| = \left| B_{E\BS X}  \right| + \left| X \right| - \left| B'_X \right| = \rk_M(E\BS X) + \left| X \right| - \rk_{M^\ast}(X), \]
	and as a consequence, $\rk_{M^\ast}(X) = \left| X \right| + \rk_M(E\BS X) - \rk_M(E)$.
\end{proof}

\noindent The following fact will be of interest for oriented matroids in Chapter~\ref{ch:OMs}.
It can be found as Proposition~2.1.11 in J.G.~Oxley's book (\cite{Ox11}, p.68), together with the proof we present here.

\begin{lemma}\label{lem:CircuitCocircuitOrthogonality}\PRFRB
	Let $M=(E,\Ical)$ be a matroid and $M^\ast=(E,\Ical^\ast)$ be its dual matroid. Then for every $C\in \Ccal(M)$ and $D\in \Ccal(M^\ast)$,
	we have $\left| C\cap D \right|\not= 1$.
\end{lemma}
\begin{proof}\PRFRB
	We give an indirect proof and assume that $\SET{x} = C\cap D$ for some $C\in\Ccal(M)$ and $D\in \Ccal(M^\ast)$.
	Since $D\in \Ccal(M^\ast)$, we have $\rk_{M^\ast}(D) = \left| D \right| - 1$.
	We set $H = E\BS D$, then by Lemma~\ref{lem:rankDual}, we get
 $$
		\rk_{M^\ast}(D) = \left| D \right| - 1  = \left| D \right| + \rk_M(H) - \rk_M(E),
	$$
	and therefore $\rk_M(H) = \rk(E) - 1$ follows. Clearly, $\cl_M(H) = H$, since otherwise there would be an element $d\in D$ such that $d\in \cl_M(H)\BS H$,
	which would imply that $$\rk_{M^\ast}(D\BSET{d}) = \left| D\BSET{d} \right| + \rk_M \left(  H\cup\SET{d} \right) - \rk_M(E) =  \left| D\BSET{d} \right| - 1,$$ contradicting that $D\in\Ccal(M^\ast)$ is a minimally dependent set of $M^\ast$ with respect to set-inclusion. But now we arrive at another contradiction:
	We have $x\in C\cap D$, $x\notin H=E\BS D$, and thus $C\not\subseteq H$, yet %\linebreak
	 $\left| C\cap H \right| = \left| C \right| - \left| C\cap D \right| = \left| C \right| - 1$, and therefore $\cl_M(C\cap H) = C$, so we obtain the contradiction $C \subseteq \cl_M(C\cap H) \subseteq \cl_M(H) = H$ (Lemma~\ref{lem:clFlips}). Therefore $\left| C\cap D \right|\not= 1$ must be the case.
\end{proof}

\begin{lemma}\label{lem:CircuitCocircuitIntersectInTwo}\PRFRB
Let $M=(E,\Ical)$ be a matroid and $M^\ast=(E,\Ical^\ast)$ be its dual matroid. Let further $C\in \Ccal(M)$ be a circuit and $c,d\in C$ with $c\not = d$.
There there is some $D\in \Ccal(M^\ast)$ such that $C\cap D = \SET{c,d}$.
\end{lemma}

\noindent The proof presented here is along the lines of the proof of Lemma 2.2.3 in \cite{BlV78}.

\begin{proof}\PRFRB
	Since $C\in\Ccal(M)$, we have $C\BSET{c}\in \Ical$. There is a base $B_c$ of $M$ with $C\BSET{c}\subseteq B_c$ (Lemma~\ref{lem:augmentation}),
	 and since $C\notin \Ical$, $c\notin B_c$. Then $B'_c = E\BS B_c$ is a base of $M^\ast$ with $c\in B'_c$ (Corollary~\ref{cor:dualbase}).
	 Let $D'= B'_c \cup\SET{d}$, then $\rk_{M^\ast}(D') = \rk_{M^\ast}(E) = \left| D' \right| - 1$, and therefore there is a unique circuit $D\subseteq D'$.
	 Clearly, $d\in D$ is an element of that circuit. Therefore $d\in C\cap D$. Furthermore $C\subseteq B_c\disunion\SET{c}$ and $D\subseteq B'_c\cup\SET{d}=(E\BS B_c)\cup\SET{d}$
	 yield $C\cap D\subseteq \SET{c,d}$. Since $\left| C\cap D \right|\not= 1$ (Lemma~\ref{lem:CircuitCocircuitOrthogonality}), we obtain that
	 $C\cap D = \SET{c,d}$.
\end{proof}
% -*- root: ../thesis.tex -*-

\needspace{3\baselineskip}
\subsection{Minors}

\PRFR{Jan 15th}
\noindent
In this section, we introduce the natural substructures for matroids.

\begin{definition}\label{def:Mrestriction}\PRFR{Jan 15th}
 Let $M=(E,\Ical)$ be a matroid, and let $R\subseteq E$.
 The \deftext[restriction of M to R@restriction of $M$ to $R$]{restriction of $\bm M$ to $\bm R$} is the pair\label{n:MR}
 $ M\restrict R = (R,\Ical')$ where
 \[ \Ical' = \SET{X\in \Ical\mid X\subseteq R }. \qedhere\]
\end{definition}

\begin{lemma}\PRFR{Jan 15th}
	Let $M=(E,\Ical)$ be a matroid, and let $R\subseteq E$. Then
	$ M\restrict R = (R,\Ical')$ is a matroid.
\end{lemma}

\begin{proof}\PRFR{Jan 15th}
   $\emptyset\subseteq R$ and $\emptyset \in \Ical$ thus $\emptyset \in \Ical'$
	{\em (I1)}. 
	Let $X\subseteq Y\in \Ical'$, then $Y\subseteq R$ and $Y\in \Ical$, therefore
	$X\subseteq R$ and
	$X\in \Ical$, so $X\in \Ical'$ {\em (I2)}.
	Let $X,Y\in\Ical'$ with $\left| X \right| < \left| Y \right|$. There is some
	$y\in Y\BS X$ with $X\cup\SET{y}\in \Ical$, and since $X\cup\SET{y}\subseteq R$,
	$X\cup\SET{y}\in\Ical'$ {\em (I3)}.
\end{proof}

\begin{corollary}\label{cor:rkRestrict}\PRFR{Jan 15th}
	Let $M=(E,\Ical)$ and $R\subseteq E$. Then for all $X\subseteq R$ we have
	$$ \rk_{M\restrict R}(X) = \rk_M(X).$$
\end{corollary}
\begin{proof}\PRFR{Jan 15th}
	Clear from Definition~\ref{def:rank}.
\end{proof}

\begin{lemma}\label{lem:directSumAndRestrictionCommute}\PRFR{Mar 7th}
	Let $M=(E,\Ical)$ and $N=(E',\Ical')$ be matroids with $E\cap E' = \emptyset$.
	Let $X\subseteq E\cup E'$, then
	\[ (M\oplus N)\restrict X = (M\restrict X\cap E)\oplus (N\restrict X\cap E').\]
\end{lemma}
\begin{proof}\PRFR{Mar 7th}
	Clear from Definitions~\ref{def:directSum} and \ref{def:Mrestriction}: the independent sets of the direct sum $\Ical_\oplus$ are
	disjoint unions of independent sets of its parts, therefore the restriction of the family $\Ical_\oplus$ to subsets of $X$
	consists of those
	disjoint unions of the subsets of $X$, that are independent with respect to its parts.
\end{proof}

\needspace{4\baselineskip}

%\noindent The dual operation of restriction is contraction.

\begin{definition}\label{def:Mcontraction}\PRFR{Jan 15th}
	Let $M=(E,\Ical)$ be a matroid, and let $C\subseteq E$.
	The \deftext[contraction of M to C@contraction of $M$ to $C$]{contraction of $\bm M$ to $\bm C$} is the pair\label{n:MC}
	$ M\contract C = (C,\Ical')$ where
	\[ \Ical' = \SET{X\subseteq C ~\middle|~\vphantom{A^A} \forall B\subseteq E\BS C\colon\,B\in \Ical \Rightarrow B\cup X \in \Ical }. \qedhere\]
\end{definition}

%\noindent
%In \cite{Ox11} and \cite{We76}, this operation is denoted by $M .\, C$ instead of $M\contract C$.

\begin{lemma}\label{lem:contractionBchoice}\PRFR{Jan 15th}
	Let $M=(E,\Ical)$ be a matroid, $C\subseteq E$, and let $B$ be a base of $E\BS C$ in $M$.
	If further
	\begin{align*}
	  \Ical_B & = \SET{X\subseteq C ~\middle|~\vphantom{A^A} B\cup X \in \Ical } \text{ and }\\
	  \Ical' & =  \SET{X\subseteq C ~\middle|~\vphantom{A^A} \forall B'\subseteq E\BS C\colon\,B'\in \Ical \Rightarrow B'\cup X \in \Ical }, 
	 \end{align*}
	  then $\Ical' = \Ical_B$.
\end{lemma}

\begin{proof}\PRFR{Jan 15th}
	From the definition it is clear that $\Ical' \subseteq \Ical_B$.
	First, we show that
	$\Ical_B$ does not depend on the choice of the base of $E\BS C$ in $M$.
	Let $B, B'\subseteq E$ be any two bases of $E\BS C$ in $M$, and let $\Ical_B$
	be defined as in the lemma, and let
	$\Ical_{B'} = \SET{X\subseteq C \mid B'\cup X \in \Ical }$.
	If $X\in\Ical_B$ then $B\cup X \in \Ical$. Let $F= B\cup B'\cup X$, then
	there is a base $B_X$ of $F$ with $B\cup X\subseteq B_X$ (Lemma~\ref{lem:augmentation}).
	Furthermore, we already have $B_X= B\cup X$, because both $B$ and $B'$ are independent subsets of $E\BS C$ with maximal cardinality, so any $\left| B \right| + 1$ elementary subset of $B\cup B'$ must be dependent and therefore cannot be a subset of $B_X$. Again by Lemma~\ref{lem:augmentation}, we obtain a base $B_X'$ of $F$ with $B'\subseteq B_X'$.
	Since $\left| B_X \right| = \left| B_X' \right|$ (Corollary~\ref{cor:equicardinality}) and the previous argument about subsets of $B\cup B'$, we have $B_X' = B'\cup X$, therefore $X\in \Ical_B'$. This proves $\Ical_B\subseteq \Ical_B'$ for any two bases $B$ and $B'$ of $E\BS C$ in $M$, and therefore $\Ical_B = \Ical_B'$ for any two such bases.

	\needspace{6\baselineskip}
	\noindent Let $X\subseteq E$ and let $I\subseteq E\BS C$ such that $I\in\Ical$. 
	Then there is a base $B'$ of $E\BS C$ in $M$ with $I\subseteq B'$.
	If $X\cup I\notin \Ical$, then clearly $X\cup B'\notin \Ical$. Therefore we may write
	\begin{align*}
		\Ical' \,\,\,\,& =\,\, \bigcap_{B'\subseteq E\BS C,\,B'\in\Ical} \SET{X\subseteq C~\middle|~\vphantom{A^A} X\cup B'\in \Ical} \\
		& =\,\, \bigcap_{B'\,\in\, \Bcal_M(E\BS C)} \SET{X\subseteq C~\middle|~\vphantom{A^A} X\cup B'\in \Ical} 
	%	~
	%\\
	%	& = \bigcap_{B'\text{ base of }E\BS C \text{ in } M} \SET{X\subseteq C\mid X\cup B\in \Ical}
		 \,\,\,\, = \,\,\,\,\Ical_B
	\end{align*}
	where $B$ is any fixed base of $E\BS C$ in $M$.
\end{proof}

\begin{lemma}\PRFR{Jan 15th}
	Let $M=(E,\Ical)$ be a matroid, and let $C\subseteq E$. Then
	$ M\contract C = (C,\Ical')$ is a matroid.
\end{lemma}

\begin{proof}\PRFR{Jan 15th}
	Let $B$ be an arbitrarily fixed base of $E\BS C$ in $M$, then
	$\Ical' = \SET{X\subseteq C\mid X\cup B\in \Ical}$ (Lemma~\ref{lem:contractionBchoice}).
	Clearly $B\cup\emptyset = B\in \Ical$, thus $\emptyset \in \Ical'$ {\em (I1)}. Furthermore,
	if $X\in\Ical'$,then $B\cup X \in \Ical$, therefore for any $Y\subseteq X$, we have
	$B\cup Y \in \Ical$ {\em (I2)}. Now let $X,Y\in\Ical'$ with $\left| X \right| < \left| Y \right|$. Thus $B\cup X \in \Ical$ and $B\cup Y \in \Ical$ with $\left| B\cup X \right| = \left| B \right| + \left| X \right| < \left| B \right| + \left| Y \right| = \left| B\cup Y \right|$. There is $y\in (B\cup Y)\BS (B\cup X) = Y\BS X$ such that
	$B\cup X\cup\SET{y} \in \Ical$, and therefore $X\cup\SET{y}\in \Ical'$ {\em (I3)}.
	Thus $M\contract C$ is a matroid.
\end{proof}

\begin{corollary}\label{cor:rkContract}\PRFR{Jan 15th}
	Let $M=(E,\Ical)$ be a matroid and $C\subseteq E$. Then for all $X\subseteq C$
	\[ \rk_{M\contract C} (X) = \rk_M(X\cup\left( E\BS C \right)) - \rk_M(E\BS C).\]
\end{corollary}
\begin{proof}
	Immediate consequence from Lemma~\ref{lem:contractionBchoice} and Definition~\ref{def:rank}.
\end{proof}

\begin{lemma}\label{lem:directSumAndContraction}\PRFR{Mar 7th}
	Let $M=(E,\Ical)$ and $N=(E',\Ical')$ be matroids with $E\cap E' = \emptyset$.
	Let $C \subseteq E\cup E'$. Then
	\[ (M\oplus N)\contract C = (M\contract C\cap E) \oplus (N\contract C\cap E').\]
\end{lemma}
\begin{proof}\PRFR{Mar 7th}
	Direct consequence of Definition~\ref{def:directSum} and Corollary~\ref{cor:rkContract}:
	Since the independent sets of $M\oplus N$ are the disjoint unions of the independent sets of
	$M$ and $N$, it is clear that $\rk_{M\oplus N}(X) = \rk_M(X\cap E) + \rk_N(X\cap E')$
	holds for all $X\subseteq E\cup E'$
	(Definition~\ref{def:rank}).
	Thus 
	\begin{align*}
		 \rk_{(M\oplus N)\contract C}(X)& \,\,\, =\,\,\, \rk_M\left(\left( X\cap E\right)\cup \left( E\BS C \right)\right)
	+ \rk_N\left(( X\cap E')\cup ( E'\BS C )\right) \\& \,\,\, \,\,\, \quad- \rk_M(E\BS C) - \rk_N(E'\BS C) 
	\\& \,\,\, =\,\,\, \rk_{M\contract C\cap E}(X\cap E) + \rk_{N\contract C\cap E}(X\cap E'). \qedhere
	\end{align*}
\end{proof}

\PRFR{Jan 15th}
\noindent The operations of restriction and contraction are related by duality,
if you do one of these operations on the dual of $M$ and then dualize the result,
you get the matroid you would have obtained from the other operation on $M$.

\begin{lemma}\label{lem:restrictcontractdual}\PRFR{Jan 15th}
	Let $M=(E,\Ical)$ be a matroid, and let $C\subseteq E$. Then
	$ M\contract C = \left( M^\ast \restrict C \right)^\ast $.
\end{lemma}

\begin{proof}\PRFR{Jan 15th}
	Clearly, 	$ M\contract C = \left( M^\ast \restrict C \right)^\ast $ holds if and only if
		$ \left( M\contract C \right)^\ast = M^\ast \restrict C$ holds (Corollary~\ref{cor:doubleast}).
	Since the family of independent sets of a matroid can be reconstructed from the values of its rank function, it suffices to show that for any $X\subseteq C$ the equation
	\[ \rk_{\left( M\contract C \right)^\ast} (X) = \rk_{M^\ast \restrict C} (X) \]
	holds. First observe that for $X\subseteq C \subseteq E$ the set equation
	$\left( E\BS C \right)\cup \left( C\BS X \right) = E\BS X$ holds. Now from Lemma~\ref{lem:rankDual}, and the Corollaries \ref{cor:rkContract} and \ref{cor:rkRestrict} we obtain
	\begin{align*}
		\rk_{\left( M\contract C \right)^\ast} (X) & = \left| X \right| + \rk_{M\contract C} \left( C\BS X \right) - \rk_{M\contract C} (C) \\
		& = \left| X \right| + \rk_M\left( \left( E\BS C \right) \cup \left( C\BS X \right)  \right) -
		\rk_M \left( E\BS C \right) - \rk_M (E) + \rk_M \left( E\BS C \right) \\
		& = \left| X \right| + \rk_M \left( E\BS X \right) - \rk_M(E) = \rk_{M^\ast} (X) = \rk_{M^\ast \restrict C} (X). \qedhere
	\end{align*}
\end{proof}

\begin{lemma}\label{lem:contractrestrictcommutes}\PRFR{Jan 15th}
	Let $M=(E,\Ical)$ be a matroid, and let $C\subseteq E$ and $R\subseteq E$ such that $\left( E\BS C \right)\cap \left( E\BS R \right) = \emptyset$. Then
	\[ \left( M\restrict R \right)\contract \left( C\cap R \right) = \left( M\contract C \right)\restrict (C\cap R).\]
\end{lemma}
\begin{proof}\PRFR{Jan 15th}
	First, we want to establish the fact that $R\BS C = E\BS C$. Since $R\subseteq E$, it remains to show that
	$\left( E\BS C \right)\BS \left( R\BS C \right) = \emptyset$.
	For all $x\in \left(E\BS C\right) \BS \left( R\BS C \right)$ we have $x\in E$, $x\notin C$ and $x\notin R$, thus $x\in E\BS C$ and $x\in E\BS R$. Since $\left( E\BS C \right)\cap \left( E\BS R \right)= \emptyset$, we conclude $\left( E\BS C \right)\BS \left( E\BS R \right)=\emptyset$, so $E\BS C = R\BS C$. Furthermore, it is clear that $R\BS \left( C\cap R \right) = R\BS C$ for all sets $C$ and $R$.
	We give a proof of the statement of the lemma using the rank formulae from Corollaries~\ref{cor:rkRestrict} and \ref{cor:rkContract}.
	Let $X\subseteq C\cap R$, then
	\begin{align*}
		\rk_{ \left( M\restrict R \right)\contract \left( C\cap R \right)} & = \rk_{M\restrict R} \left( X\cup \left( R\BS C \right) \right) - \rk_{M\restrict R}(R\BS C) \\
		& = \,\,\,\,\,\,\,\rk_{M} \left( X\cup \left( R\BS C \right) \right) - \rk_{M}(R\BS C) \\
		& = \,\,\,\,\,\,\,\rk_{M} \left( X\cup \left( E\BS C \right) \right) - \rk_{M}(E\BS C) \\
		& = \,\,\,\,\,\,\,\rk_{M\contract C} (X) \quad \quad \,\,\,\,\,\,\, =  \rk_{ \left( M\contract C \right)\restrict (C\cap R)}(X). \qedhere
	\end{align*}
\end{proof}

\begin{definition}\PRFR{Mar 7th}
	Let $M=(E,\Ical)$ and $N=(E',\Ical')$ be matroids.
	We shall call $N$ a \deftext[minor of $M$]{minor of $\bm M$},
	if there are sets $X \subseteq Y \subseteq E$ such that
	\[ N = \left( M \contract Y \right)\restrict X \]
	holds.
\end{definition}

\begin{remark}\label{rem:contractRestrictCommutingFormula}\PRFR{Mar 7th}
	For $M=(E,\Ical)$ and $X\subseteq Y \subseteq E$ we have $Y \cap \left( E\BS\left( Y\BS X \right) \right)=X$ and
	 $\left( E\BS Y \right) \cap \left( Y\BS X \right) = \emptyset$, 
	so Lemma~\ref{lem:contractrestrictcommutes} yields that
	\[ \left( M \contract Y \right)\restrict X  = \left( M \restrict E\BS\left( Y\BS X \right) \right)\contract X
	\,\,\txtand\,\,
	\left( M \restrict Y \right)\contract X  = \left( M \contract E\BS\left( Y\BS X \right) \right)\restrict X
	. \qedhere
	 \] 
\end{remark}

\begin{definition}\PRFR{Jan 15th}
	Let $\Mcal$ be a class of matroids. Then $\Mcal$ shall be called a \deftext{minor-closed class},
	if for every $M=(E,\Ical)\in \Mcal$ and every $X\subseteq E$,
	also $M\restrict X \in \Mcal$ and $M\contract X\in \Mcal$ holds.
\end{definition}

\begin{example}\PRFR{Jan 15th}
	Let $\Mcal$ be the class where $M\in \Mcal$ if and only if $M=\left( E,2^E \right)$ for any set $E$,
	i.e. $\Mcal$ is the class of all free matroids. Clearly, for every $M=\left( E,2^E \right)$ and every $X\subseteq E$
	we have
	$M\restrict X = M\contract X = \left(X,2^X\right) \in \Mcal$.
\end{example}

\needspace{5\baselineskip}
\begin{definition}\PRFR{Mar 7th}
	Let $\Mcal$ be a minor-closed class of matroids. A matroid $M=(E,\Ical)$ is called
	\deftext[excluded minor]{excluded minor for $\bm \Mcal$}
	if $M\notin \Mcal$ and if for every $X\subsetneq E$ we have both $M\restrict X \in \Mcal$ and $M\contract X \in \Mcal$.
	Furthermore, a minor-closed class of matroids $\Mcal$ is called \deftextX{characterized by finitely many excluded minors}
	if there are only finitely many pair-wise non-isomorphic excluded minors for $\Mcal$.
\end{definition}

\begin{example}\PRFR{Mar 7th}
A matroid is representable over the $2$-elementary field $\Fbb_2$ (Definition~\ref{def:representableMoverK})
 if and only if it has no minor isomorphic to
the rank-$2$ uniform matroid $\left( \SET{a,b,c,d},\SET{X\subseteq \SET{a,b,c,d}~\middle|~\vphantom{A^A}\left| X \right|\leq 2} \right)$.
(Theorem 6.5.4 \cite{Ox11}, p.193). Thus the class of all matroids representable over $\Fbb_2$ is characterized by finitely many,
or in this case, a single excluded minor.
\end{example}

\begin{remark}
	If $\Mcal$ is a minor-closed class of matroids with the property that $M\in \Mcal \Leftrightarrow M^\ast \in \Mcal$ holds for all matroids $M$,
	i.e. $\Mcal$ is closed under duality;
	then $N$ is an excluded minor of $\Mcal$ if and only if $N^\ast$ is an excluded minor of $\Mcal$ (see also Lemma~\ref{lem:contractrestrictcommutes}).
\end{remark}

\PRFR{Mar 7th}
\noindent The excluded minors for matroids representable over fields with $2$, $3$, and $4$ elements are known (\cite{Ox11}, p.193),
and the famous {\em Rota's Conjecture} states that for every finite field $\Fbb$, the class of matroids representable over $\Fbb$ is
characterized by finitely many excluded minors. J.~Geelen, B.~Gerards and G.~Whittle claim to have proven Rota's Conjecture and published
 an overview of their proof in \cite{GGW14}. Furthermore, it has been shown
  that both the class of matroids representable over the field of the
  reals $\R$ and the class of gammoids have the property, that
   every matroid $M$ in each respective class is a minor of an excluded minor of that class, therefore those classes cannot be characterized by
    finitely many excluded minors, because both classes are non-empty and closed under direct sums, thus they are infinite.
     The result for the class of matroids representable over $\R$ 
    has been proven by D.~Mayhew, M.~Newman, and G.~Whittle in \cite{MNW09},
    and the result for the class of gammoids can be found in a paper by D.~Mayhew \cite{Ma16}, where the excluded minor constructed for
    an arbitrary gammoid is also an excluded minor for the class of matroids representable over $\R$.
    
% -*- root: ../thesis.tex -*-

\subsection{Matroids Representable Over a Field}

\PRFR{Jan 15th}
\noindent A quite natural class of matroids arises from the notion of linear independence. We only give 
a short introduction here. Those readers, who are interested in the classes of matroids representable over some given field $\Fbm$, shall hereby be referred to J.G.~Oxley's book \cite{Ox11}.

\begin{definition}\label{def:Mmu}\PRFR{Jan 15th}
	Let $\Kbm$ be a field, $E$ and $C$ be finite sets. Let $\mu\in \Kbm^{E\times C}$ be an $E\times C$-matrix over $\Kbm$. The \deftext[matroid represented over a field]{matroid represented by $\bm \mu$ over $\bm\Kbm$} is the pair\label{n:matMmu} $M(\mu) = (E,\Ical)$ where
	\[ \Ical = \SET{X\subseteq E~\middle|~\vphantom{A^A} \idet \left( \mu \restrict X \right) = 1}. \qedhere\]
\end{definition}

\begin{lemma}\PRFR{Jan 15th}
Let $\Kbm$ be a field, $E$ and $C$ be finite sets. Let $\mu\in \Kbm^{E\times C}$. Then
$M(\mu)$ is a matroid.
\end{lemma}

\noindent The proof is essentially elementary linear algebra.

\begin{proof} Let $(E,\Ical) = M(\mu)$. \PRFR{Jan 15th}
	It is clear from Definition~\ref{def:idet} that for $X\subseteq E$, the equality $\idet \left( \mu\restrict X \right) = 1$ holds if and only if the set
	$V_X = \SET{\mu_x\mid x\in X}$ is linear independent in the vector space $\Kbm^C$ with the further property
	 $\left| V_X \right| = \left| X \right|$. Thus $\emptyset\in \Ical$ {\em (I1)}. 
	 For every $Y\subseteq X\in\Ical$, we
	 have that $V_Y = \SET{\mu_y\mid y\in Y}$ is linear independent in $\Kbm^C$ with $\left| V_Y \right| = \left| Y \right|$, thus $Y\in\Ical$ {\em (I2)}. Let $X,Y\in\Ical$ with $\left| X \right| < \left| Y \right|$, and let $V_X$, $V_Y$ be defined as above. Since $\left| V_Y \right| > \left| V_X \right|$ and $V_Y$ is linear independent in $\Kbm^C$, we have that
	 $\SPAN_{\Kbm^C}(V_X) \subsetneq \SPAN_{\Kbm^C}\left( V_X\cup V_Y \right)$. Therefore, there is some $\mu_y\in V_Y$ with $\mu_y\notin \SPAN_{\Kbm^C}(V_X)$, and consequently, $V' = V_X\cup\SET{\mu_y}$
	 is linear independent in $\Kbm^C$ with $\left| V' \right| = \left| X \right| + 1$, thus
	 $X\cup\SET{y}\in \Ical$ {\em (I3)}.
\end{proof}

\begin{corollary}\PRFR{Jan 15th}
	Let $M=(E,\Ical)$ be a matroid, $\Kbm$ be a field, $E$ and $C$ be finite sets, and
	$\mu\in \Kbm^{E\times C}$ be a matrix. For all $R\subseteq E$,
	\[ M(\mu)\restrict R = M\left( \mu \restrict R\right).\]
\end{corollary}

\begin{remark}\label{rem:colops}\PRFR{Jan 15th} %ColOps sind keine co-loops :) 
	It is a well-known fact from linear algebra that the following operations on $\mu\colon E\times C\maparrow \K$ do not
	change linear dependency between rows, and therefore do not alter the matroid $M(\mu)$:
	\begin{enumerate}\ROMANENUMEM
		\item Interchanging two columns $c_1,c_2\in C$, i.e. if $\nu(e,c) = \begin{cases} \mu(e,c) &\quad \text{if~} c\notin\SET{c_1,c_2},\\
																						\mu(e,c_2) &\quad \text{if~} c=c_1,\\
																						\mu(e,c_1) &\quad \text{if~} c=c_2,\end{cases}$
					then $M(\mu) = M(\nu)$.
		\item Adding a multiple of one column to another column, i.e. for $c_1,c_2\in C$ with $c_1\not= c_2$ and $\alpha\in \K$, i.e.
		\hfill
		if  $\nu(e,c) = \begin{cases} \mu(e,c) &\quad \text{if~} c\not=c_2,\\
									   \mu(e,c_2) + \alpha\cdot\mu(e,c_1) &\quad \text{if~} c=c_2,\end{cases}^{\vphantom{X}}$\linebreak
		then $M(\mu) = M(\nu)$.
		\item Multiplying a column $c_1\in C$ with $\alpha\in \K\BSET{0}$, i.e. \\
		if  $\nu(e,c) = \begin{cases} \mu(e,c) &\quad \text{if~} c\not=c_1,\\
									  \alpha\cdot\mu(e,c_1) &\quad \text{if~} c=c_1,\end{cases}^{\vphantom{X}}\quad\quad$
		then $M(\mu) = M(\nu)$.
	\end{enumerate}
	Furthermore, if $B\subseteq E$ is a base of $M(\mu)$, then we can use Gauß-Jordan elimination steps\footnote{This procedure is commonly refered to as ''pivoting in the ordered basis $B$`` in the context of linear programming. Careful pivoting is the foundation of the simplex algorithm for solving linear optimization problems.}  in order to obtain an injective map
	$\iota\colon B\maparrow C$ and a matrix $\nu$, which has the properties $M(\nu) = M(\mu)$ and for all $b\in B$ and all $c\in C$,
	$\nu(b,c) = \begin{cases} 1 & \text{if~} c = \iota(b),\\ 0 & \text{otherwise}. \end{cases}^{\vphantom{X}}_{\vphantom{X}}$

	\noindent
	From the matrix $\nu$, we can easily read some important properties of $M(\mu)$. Let $e\in E\BS B$, 
	then the unique circuit contained in $B\cup\SET{e}$ consists of $e$ and the elements $b'\in B$ where $\nu(e,\iota(b'))\not= 0$.
	If $B'\subseteq B$, then $\cl_{M(\nu)}(B')$ consists of $B'$ and all $e\in E\BS B$ which have the property that $\nu(e,c)=0$ holds for all $c\in C\BS\left(  \iota[B'] \right)$.
\end{remark}

\begin{definition}\label{def:representableMoverK}\PRFR{Jan 15th}
	Let $M=(E,\Ical)$ be a matroid, $\Kbm$ be a field. We say that $M$ is \deftext[matroid representable over a field]{representable over $\bm \Kbm$}, if there is a finite set $C$
	and a matrix $\mu\in \Kbm^{E\times C}$, such that $M = M(\mu)$.
\end{definition}

\begin{lemma}\label{lem:contractequalspivot}\PRFR{Jan 15th}
 Let $E,C$ be finite sets, and $\mu\in \Kbm^{E\times C}$ be a matrix.
 Let $e\in E$ and $c\in C$, such that $\mu(e,c) \not= 0$.
 Let
 \[ \nu \colon \left(  E\BSET{e} \right)\times \left( C\BSET{c} \right) \maparrow \K,\,(f,d)\mapsto \mu(f,d) - \frac{\mu(e,d)}{\mu(e,c)}\cdot \mu(f,c)\]
 be the matrix obtained by carrying out a Gauß-Jordan elimination step with the pivot index $(e,c)$ and then deleting the corresponding row and column.
 Then $$M(\mu)\contract\left( E\BSET{e}  \right) = M(\nu).$$
\end{lemma}
\begin{proof}
	Let $\nu'\in \K^{E\times C}$ where for all $(f,d)\in E\times C$
	\[ \nu'(f,d) = \begin{cases} \mu(f,d) - \frac{\mu(e,d)}{\mu(e,c)}\cdot \mu(f,c) & \quad \text{if~} d\not= c, \\
	 							\frac{\mu(f,c)}{\mu(e,c)} & \quad \text{if~} d=c. \\
					\end{cases}
	\]
	Since $\nu'$ arises from $\mu$ by elementary column operations, we have $M(\mu) = M(\nu')$ (Remark~\ref{rem:colops}). Furthermore,
	$\nu'(e,c) = 1$ and $\nu'(e,d) = 0$ for all $d\in C\BSET{c}$. Let $E'\subseteq E\BSET{e}$ and $C'\subseteq C\BSET{c}$ with $\left| E' \right| = \left| C' \right|$. Then
	\begin{align*}
	\det \left( \nu'\restrict \left( E'\cup\SET{e} \right)\times \left( C'\cup\SET{c} \right) \right)&\,\, =\,\,
		\sigma \cdot \det \left( \nu' \restrict E'\times C' \right)
		\,\, = \,\, \sigma \cdot \det \left( \nu \restrict E'\times C' \right)
	\end{align*}
	for some $\sigma\in \SET{-1,1}$. Thus for $X\subseteq E\BSET{e}$ $$\idet \left( \nu'\restrict \left( X\cup\SET{e} \right)\right) = \idet \left( \nu \restrict X \right),$$
	and consequently $X$ is independent in $M(\nu)$, if and only if $X\cup\SET{e}$ is independent in $M(\nu') = M(\mu)$. Therefore
	$M(\nu) = M(\mu) \contract \left( E\BSET{e} \right)$.
\end{proof}

\begin{remark}\label{rem:contraction}\PRFR{Feb 15th}
	Let $M(\mu)$ be a matroid for some $\mu\in \R^{E\times C}$. A straightforward consequence of Lemma~\ref{lem:contractequalspivot} is
	that for $X\subseteq E$, we can pivot in a base $B$ of $M(\mu)$ with the property that $B\BS X$ is a base of $E\BS X$ --
	which exists due to Lemma~\ref{lem:augmentation} -- and then restrict the resulting matrix $\nu$ to $X\times C_0$ where
	$C_0 = \SET{c\in C\mid \forall b'\in B\BS X\colon\, \nu(b',c) = 0}$. Then $M(\mu)\contract X = M\left( \nu\restrict X\times C_0 \right)$.
\end{remark}

\begin{lemma}\label{lem:standardMatrixNu}\PRFR{Feb 15th}
	Let $M=(E,\Ical)$ be a matroid that is representable over $\Kbm$, such that for some $n,r\in \N$, $E=\dSET{e_1,e_2,\ldots,e_n}$
	and $B_0 = \dSET{e_1,e_2,\ldots,e_r}$ is a base of $M$. Then there is a matrix $\nu\in \Kbm^{E\times B_0}$ such that
	$\nu\restrict B_0\times B_0$ is the identity matrix for $B_0$ over $\Kbm$.
\end{lemma}
\begin{proof}\PRFR{Feb 15th}
	This is basic linear algebra.
	Let $\mu\in \Kbm^{E\times C}$ be a matrix with $M=M(\mu)$. Then the row vectors $\SET{\mu_b \mid b\in B_0}$ form 
	a basis of a sub-vector space $V\subseteq \Kbm^C$, and since $B_0$ is a base of $M$, we have that $\mu_e\in V$ for all
	$e\in E$. Thus every $\mu_e$ has a unique representation as linear combination of vectors from  $\SET{\mu_b \mid b\in B_0}$,
	and we can set $\nu(e,b)$ to be the coefficient of $\mu_{b}$ with respect to the linear combination representing $\mu_e$,
	for all $e\in E$ and $b\in B_0$. Since a change of basis in a vector space does not affect linear dependency, we have
	$M(\mu) = M(\nu)$.
\end{proof}

% \remred{TODO}

% \remblue{ Brauchen wir eigentlich nirgendwo....
% \begin{lemmaX}\label{lem:RepresentableClosedUnderDuality}
% 	Let $M=(E,\Ical)$ be a matroid, \remred{Muss nicht mit jeder Charakteristik gehen, oder?} $\Kbm$ be a field, $E$ and $C$ be finite sets, and
% 	$\mu\in \Kbm^{E\times C}$ be a matrix. For all $D\subseteq E$,
% 	the contraction $M(\mu)\contract D$ is representable over $\Kbm$.
% 	Furthermore, $\left( M(\mu) \right)^\ast $ is representable over $\Kbm$.
% \end{lemmaX}
% \begin{proof}
% 	The first statement follows from the second statement and $M\contract D = \left( M^\ast \restrict D \right)^\ast$ (Lemma~\ref{lem:restrictcontractdual}).
% 	\remred{TODO, auch die Darstellung machen; mit orthogonalem Komplement}
% \end{proof}
% }

% \remred{TODO: Minor closed class, etc.}

\begin{remark}\label{rem:stdRep}\PRFR{Feb 15th}
	An immediate consequence of Lemma~\ref{lem:standardMatrixNu} is, that
	every matroid $M=(E,\Ical)$ representable over $\Kbm$ with $r=\rk_M(E)$ has a matrix $\mu\in \Kbm^{E\times \SET{1,2,\ldots,r}}$
	 such that $M=M(\mu)$ and such that for some base $B\subseteq E$ of $M$,
	the matrix $\mu\restrict(B\times \SET{1,2,\ldots,r})$ resembles an identity matrix --- up to renaming of the rows.
	Thus we may consider $\mu^\top$ to be that identity matrix in apposition with a matrix $A^\top \in \Kbm^{\SET{1,2,\ldots,r}\times (E\BS B)}$,
	i.e. that $\mu = \left( I_r ~~ A^\top \right)^\top$. A matrix of this form is called \deftext[standard representation over $\Kbm$]{standard representation}. If $\mu$ is a standard representation, then
	 $\nu = \left( -A ~~ I_{|E|-r}  \right)^\top$ has the property that $M^\ast = M(\nu)$
	and further, that for all $e,f\in E$, $\langle \mu_e, \nu_f \rangle = 0$ (Corollary~1, \cite{We76}, p.~143).
	Thus for every field $\Kbm$ the family of matroids representable over $\Kbm$ is closed under duality.
\end{remark}
\clearpage
% -*- root: ../thesis.tex -*-

\section{Single Element Extensions}

H.H.~Crapo has exhaustively studied extensions of matroids by single elements in \cite{C65}.

\begin{definition}\PRFR{Feb 15th}
  Let $M=(E,\Ical)$ be a matroid, $A,B\subseteq E$.
  Then $A$ and $B$ are a \deftext[modular pair in M@modular pair in $M$]{modular pair in $\bm M$},
  whenever $$\rk(A\cap B) + \rk(A\cup B) = \rk(A) + \rk(B)$$ holds.
  A modular pair is called \deftext[trivial modular pair]{trivial},
  if $A\subseteq B$ or $B\subseteq A$.
\end{definition}

\begin{example}\label{ex:indepModPairs}\PRFR{Feb 15th}
	Let $M=(E,\Ical)$ be a matroid and $A,B\subseteq E$ such that $A\cup B \in \Ical$.
	Then $A$ and $B$ are a modular pair in $M$, since
	\[ \rk(A\cup B) + \rk(A\cap B) = \left| A\cup B \right| + \left| A\cap B \right| = \left| A \right| + \left| B \right| = \rk(A) + \rk(B).
	\qedhere \]
\end{example}

\begin{lemma}
	Let $M=(E,\Ical)$ be a matroid, $A,B\subseteq E$.
	Then $A$ and $B$ are a modular pair in $M$ if and only if there is a base $X$ of $A\cup B$
	such that $X\cap A$ is a base of $A$ and $X\cap B$ is a base of $B$. 
	%\remred{Ist dann $X\cap A\cap B$ schon eine Basis von $A\cap B$?} Ja.
\end{lemma}
\begin{proof}
	Let $X$ be a base of $A\cup B$ such that $X\cap A$ is a base of $A$ and $X\cap B$ is a base of $B$.
	Then $X\cap A \cap B$ is a base of $A\cap B$: Since $\rk(A) + \rk(B) \geq \rk(A\cup B) + \rk(A\cap B)$
	holds by submodularity of the rank function,
	and since $\left| X\cap A\cap B \right| \leq \rk(A\cap B)$ holds by {\em (I2)} and Definition~\ref{def:rank},
	we obtain
	\begin{align*}
	\left| X\cap A \cap B \right| & \leq \rk(A\cap B) \leq \rk(A) + \rk(B) - \rk(A\cup B) \\
	& = \left| X\cap A \right| + \left| X\cap B \right| - \left| X \right| = \left| X\cap A\cap B \right|.
	\end{align*}
	It has been shown in Example~\ref{ex:indepModPairs} that $X\cap A$ and $X\cap B$ are a modular pair in $M$.
	Thus $A$ and $B$ are a modular pair in $M$, because $\rk(A) = \rk(X\cap A)$, $\rk(B) = \rk(X\cap B)$,
	$\rk(A\cap B) = \rk(X\cap A\cap B)$, and $\rk(A\cup B) = \rk(X)$ holds.
	Now let $A,B\subseteq E$ such that there is no base $X$ of $A\cup B$, for which both $X\cap A$ is a base of $A$
	and $X\cap B$ is a base of $B$. Then for all bases $X$ of $A\cup B$, for which $X\cap A$ is a base of $A$,
	and for which $X\cap A \cap B$ is a base of $A\cap B$,
	there is some $b\in B\BS \cl(X\cap B)$, i.e. $\rk(B) > \rk(X\cap B)$. Lemma~\ref{lem:augmentation} guarantees that there is
	a base $X$ of $A\cup B$ with $\rk(X\cap A) = \rk(A)$ and $\rk(X\cap A\cap B) = \rk(A\cap B)$. Thus we obtain that
	\( \rk(A) + \rk(B) > \left| X\cap A \right| + \left| X\cap B \right| = \left| X \right| + \left| X\cap A\cap B \right| = \rk(A\cup B) + \rk(A\cap B) \)
	holds, which
	implies that $A$ and $B$ are not a modular pair in $M$.
\end{proof}

\begin{definition}\PRFR{Jan 15th}
	Let $M=(E,\Ical)$ be a matroid, and let $C \subseteq \Fcal(M)$ be a set of flats of $M$.
	We call $C$ a \deftext[modular cut of M@modular cut of $M$]{modular cut of $\bm M$}, if $C$ has the properties
	that
	\begin{enumerate}\ROMANENUM
	\item for all $A,B\in\Fcal(M)$ the implication
	$$\rk(A) + \rk(B) = \rk(A\cap B) + \rk(A\cup B) \quad \Longrightarrow \quad A\cap B\in C$$
	holds, and
	\item for all $X,Y\in \Fcal(M)$ with $X\subseteq Y$, the implication $X\in C \Rightarrow Y\in C$ holds.
\end{enumerate}
    The \deftextX{class of all modular cuts of $\bm M$} shall be denoted by\label{n:MM}
    \[ \Mcal(M) = \SET{C\subseteq \Fcal(M)\mid C\text{ is a modular cut of }M}. \qedhere\]
\end{definition}

\begin{definition}\PRFR{Jan 15th}
	Let $M=(E,\Ical)$ be a matroid and let $e\notin E$. The \deftextX{class of single element extensions of $\bm M$ by $\bm e$}
	 is defined to be\label{n:XMe}
	\[ \Xcal(M,e) = \SET{(E\cup\SET{e},\Ical') ~\middle|~ \Ical' \subseteq 2^{E\cup\SET{e}}\colon\,\,
	\Ical'\cap 2^E = \Ical 
	\txtand (E\cup\SET{e},\Ical')\text{ is a matroid}}.\]
	Let $N=(F,\Jcal)$ be a matroid. $N$ shall be called \deftext[single element extension of M@single element extension of $M$]{single element extension of $\bm M$}, if $F\BS E = \SET{f}$ and $N\in \Xcal(M,f)$.
\end{definition}

\noindent We convince ourselves that $N$ is indeed an extension of $M$ in the sense that $N$ behaves exactly like $M$ with respect to subsets of $E$.

\needspace{3\baselineskip}
\begin{lemma}\label{lem:extension_rk}\PRFR{Jan 15th}
	Let $M=(E,\Ical)$ be a matroid, $e\notin E$, $N\in \Xcal(M,e)$, and let $X\subseteq E$.
	Then $\rk_{M}(X) = \rk_{N}(X)$.
\end{lemma}
\begin{proof}
	Let $N=(E\cup\SET{e},\Ical')$. Since $\Ical' \cap 2^{E} = \Ical$,  $\SET{Y\subseteq X\mid Y\in \Ical} = \SET{Y\subseteq X\mid Y\in \Ical'}$. Thus
	\[ \rk_{M}(X) = \max\SET{\,\left| Y \right|\,\, \vphantom{A^A}~\middle|~ Y\subseteq X,\, Y\in\Ical} =
	\max\SET{\,\left| Y \right|\,\, \vphantom{A^A}~\middle|~ Y\subseteq X,\, Y\in\Ical'} = \rk_N(X). \qedhere\]
\end{proof}

\noindent Now we have all the definitions that we need in order to present H.H.~Crapo's results
 \cite{C65}:

 \needspace{6\baselineskip}

\begin{theorem}\label{thm:crapo}\PRFR{Jan 15th}
	Let $M=(E,\Ical)$ be a matroid and $e\notin E$. Then there is a bijection
	\[ \phi \colon \Xcal(M,e) \maparrow \Mcal(M) \]
	which maps the single element extension $N$ to
	the modular cut 
	\[ \phi(N) = \SET{F\in\Fcal(M)\mid e \in \cl_N(F)}. \]
\end{theorem}

\begin{proof}\PRFR{Jan 15th}
	First, we show that $\phi$ is well-defined.
	Let $N \in \Xcal(M,e)$ be a single element extension of $M$. 
	We have to prove that $\phi(N)$ is indeed a modular cut of $M$.
	Let $F\in \phi(N)$, and $G\in \Fcal(M)$ with $F\subseteq G$, then $e\in \cl_{N}(F) \subseteq \cl_{N}(G)$, thus $G\in \phi(N)$.
	Now, let $A,B\in \phi(N)$ such that
	\[ \rk_M(A) + \rk_M(B) = \rk_M(A\cap B) + \rk_M(A\cup B). \]
	We give an indirect argument for $A\cap B\in \phi(N)$.
	Assume that $A\cap B \notin \phi(N)$. Then $e\notin \cl_{N}(A\cap B)$ thus
	$\rk_{N}( (A\cap B) \cup \SET{e}) > \rk_{N}(A\cap B)$. By Lemma~\ref{lem:extension_rk},
	we have the equation \[ \rk_N(A) + \rk_N(B) = \rk_N(A\cap B) + \rk_N(A\cup B). \]
	Furthermore, $A,B\in \phi(N)$,
	therefore $e\in \cl_{N}(A)$, $e\in \cl_{N}(B)$, and $e\in \cl_{N}(A\cup B)$, so
	$\rk_{N}(A\cup\SET{e}) = \rk_{N}(A)$, $\rk_{N}(B\cup\SET{e}) = \rk_{N}(B)$, and
	$\rk_{N}(A\cup B\cup\SET{e}) = \rk_{N}(A\cup B)$. This yields
	\begin{align*}
  \rk_N(\left(A\cap B\right) \cup \SET{e}) & >
	 \rk_{N}(A\cap B) \\&
	 	= \rk_{N}(A) + \rk_N(B) - \rk_{N}(A\cup B)
	  \\ & = \rk_{N}(A\cup\SET{e}) + \rk_{N}(B\cup\SET{e}) - \rk_{N}(A\cup B\cup\SET{e}),\\
	\end{align*}
	which contradicts {\em (R3)}, the submodularity of $\rk_{N}$, which guarantees
	\[ \rk_N\left(A\cup\SET{e}\right) + \rk_N\left(B\cup\SET{e}\right) \geq \rk_{N}(A\cup B\cup\SET{e}) + \rk_{N}(\left(A\cap B\right)\cup\SET{e}).\]
	 Thus $A\cap B\in \phi(N)$, so $\phi(N)$ is indeed a modular cut of $M$.

	\bSep\PRFR{Jan 15th}
	Now, we show that $\phi$ is injective. Let $N,N'\in \Xcal(M,e)$ with $N\not= N'$. 
	Without loss of generality we may assume that there is a set $X\subseteq E\cup\SET{e}$ 
	which is independent in $N$, yet dependent in $N'$.
	 Since $N$ coincides 
	with $M$ on $2^{E}$, we obtain that $e\in X$ and that $X\BSET{e}\in \Ical$ is independent in $N$, $N'$ and $M$.
	Let $F=\cl_{M}(X\BSET{e})\supseteq X\BSET{e}$.
	 Then $e\in\cl_{N'}(X\BSET{e}) \subseteq \cl_{N'}(F)$ so $F\in\phi(N')$, but
	 $e\notin \cl_{N}(F) = \cl_{M}(F) = F$, so $F\notin\phi(N)$. Thus $\phi(N)\not= \phi(N')$.

    \bSep \PRFR{Jan 15th}
        It remains to show that $\phi$ is surjective. Let $C$ be a
modular cut of $M$. We define $N=(E\cup\SET{e},\Ical')$ such that     \[
\Ical' = \Ical \cup \SET{X\cup \SET{e} \mid X\in\Ical,\,\cl_M(X)\notin C}.\]
Assume for a moment that $N$ is a matroid, then    $\phi(N) = C$: Let $F\in
\Fcal(M)$ and let $X\subseteq F$ be a base of $F$ in $M$.     If $F\in C$,
then $X\cup\SET{e}\notin \Ical'$, thus $e\in \cl_{N}(F)$, and therefore
$F\in\phi(N)$. If $F\notin C$, then $X\cup\SET{e}\in \Ical'$ and $e\notin
\cl_{N}(F)$, therefore $F\notin \phi(N)$.

\bSep \PRFR{Jan 15th}
We show that $N$ is indeed a
matroid by explicating D.J.A.~Welsh's sketch on p.~319 \cite{We76}: Observe that the map \[ \rho\colon 2^{E\cup\SET{e}} \maparrow \N,\,X\mapsto \max\SET{ \left| Y \right| ~\middle|~ Y\subseteq X,\, Y\in \Ical'} \]
satisfies the equation 
\[
	\rho(X) = \begin{cases}
			\rk_M(X) & \quad \text{if } e\notin X,\\
			\rk_M(X\BSET{e}) + 1 & \quad \text{if } e\in X,\,\cl_M(X\BSET{e})\notin C,\\
			\rk_M(X\BSET{e}) & \quad \text{if }e\in X,\,\cl_M(X\BSET{e})\in C.
	\end{cases} 
\] Furthermore, we see that \[
\Ical' = \SET{X\subseteq E\cup\SET{e}\vphantom{A^A}~\middle|~ \forall Y\subseteq X\colon \rho(Y) \geq \left| Y \right|}, \]
which is the family of independent sets of a matroid obtained from $\rho$ by Theorem~\ref{thm:submodularIndependent}, whenever $\rho$ is non-negatively integer-valued, non-decreasing, and submodular. Clearly, $\rho$ is non-negatively integer-valued. Furthermore, $\rho$ restricted to $2^E$ 
is $\rk_M$, thus $\rho$ is non-decreasing and submodular on $2^E$. Let $X,Y\subseteq E\cup\SET{e}$ such that $e\in X$ and $e\in Y$.
 If $\cl_M(X\BSET{e})\in C \Leftrightarrow \cl_M(Y\BSET{e})\in C$, then
$\rho(X) = \rk_M(X\BSET{e}) + \alpha$ and $\rho(Y) = \rk_M(Y\BSET{e}) + \alpha$ for the same value of $\alpha\in \SET{0,1}$, thus $\rho$ is non-decreasing because $\rk_M$ is non-decreasing. Otherwise, if $X\subseteq Y$
 and $\cl_M(X\BSET{e})\in C \not\Leftrightarrow \cl_M(Y\BSET{e})\in C$, then $\cl_M(X)\notin C$ whereas $\cl_M(Y)\in C$, because $C$ is closed under super-flats and $\cl_M$ preserves set-inclusion 
(Lemma~\ref{lem:clFlips}). 
 But then $\cl_M(X\BSET{e})\not= \cl_M(Y\BSET{e})$, so $\rk_M(\cl_M(X)) < \rk_M(\cl_M(Y))$,
thus $\rho(X) = \rk_M(X\BSET{e}) + 1 \leq \rk_M(Y\BSET{e}) = \rho(Y)$.
Therefore $\rho$ is non-decreasing on its whole domain.
 Now let $A,B\subseteq E\cup\SET{e}$, we have to show that the submodular inequality 
 \begin{align}
  \rho(A) + \rho(B) \geq \rho(A\cap B) + \rho(A\cup B)\label{eq:rhoSubmod}
 \end{align}
 holds. Clearly $$\rho(A) + \rho(B) = \rk_M(A\BSET{e}) + \rk_M(B\BSET{e}) + \alpha$$
for some $\alpha\in\SET{0,1,2}$ and analogously $$\rho(A\cap B) + \rho(A\cup B) = \rk_M((A\cap B)\BSET{e}) + \rk_M((A\cup B)\BSET{e}) + \beta$$ for some $\beta\in \SET{0,1,2}$. 
Since $\rk_M$ is a submodular function,
we may deduce inequality~(\ref{eq:rhoSubmod}) from $\alpha \geq \beta$
as well as from $$\beta - \alpha \leq \rk_M(A\BSET{e}) + \rk_M(B\BSET{e}) - \rk_M((A\cap B)\BSET{e}) - \rk_M((A\cup B)\BSET{e}).$$
 If $\beta=2$, then $\cl_M((A\cup B )\BSET{e})\notin C$, therefore $\cl_M(A\BSET{e})\notin C$ and $\cl_M(B\BSET{e})\notin C$ because $C$ is closed under super-flats. So in this case, $\alpha = 2 \geq \beta$. If $\beta = 0$ then clearly $\alpha \geq \beta$, too.
 So the submodular inequality (\ref{eq:rhoSubmod}) holds due to $\alpha \geq \beta$ unless
 $\beta=1$ and $\alpha = 0$.
 In this case, if $$\rk_M(A\BSET{e}) + \rk_M(B\BSET{e}) - \rk_M((A\cap B)\BSET{e}) - \rk_M((A\cup B)\BSET{e}) \geq 1,$$ then 
(\ref{eq:rhoSubmod}) follows as mentioned above.
We close our argumentation by showing that for $\beta = 1$ and $\alpha = 0$,
 $$\rk_M(A\BSET{e}) + \rk_M(B\BSET{e}) - \rk_M((A\cap B)\BSET{e}) - \rk_M((A\cup B)\BSET{e}) = 0$$ can never be the case.
 There are two ways that lead to $\beta =1$. Assume that $e\notin A\cap B$, then $\cl_M((A\cup B)\BSET{e})\notin C$.
 If $e\in A$, then $\cl_M(A\BSET{e})\notin C$ follows, thus $\alpha \geq 1$; similarly if $e\in B$. Thus $e\in A\cap B$
 is implied by $\beta = 1$ and $\alpha = 0$. Consequently, $\cl_M((A\cap B) \BSET{e})\notin C$ is necessary for $\beta = 1$.
 Furthermore, for $\alpha = 0$ it is necessary that $\cl_M(A\BSET{e})\in C$
 and $\cl_M(B\BSET{e})\in C$. But then $\cl_M(A\BSET{e})$ and $\cl_M(B\BSET{e})$ cannot be a modular pair in $M$,
 because $C$ is closed under intersections of modular pairs yet $C$ does not contain the intersection of these two flats.
 This yields $$\rk_M(A\BSET{e}) + \rk_M(B\BSET{e}) - \rk_M((A\cap B)\BSET{e}) - \rk_M((A\cup B)\BSET{e})\not= 0 .$$
%
%
% If $e\notin A\cap B$, then
%
% Then $\cl_M((A\cap B)\BSET{e})\notin C$
% as well as $\cl_M(A\BSET{e})\notin C$ and $\cl_M(B\BSET{e})\notin C$  . If $A\BSET{e}$ and $B\BSET{e}$ is not a modular pair in $M$, then $\rk_M(A\BSET{e}) + \rk_M(B\BSET{e}) \geq 1 + \rk_M((A\cap B)\BSET{e}) + \rk_M((A\cup B)\BSET{e})$ and therefore $\rho(A) + \rho(B) \geq \rho(A\cap B) + \rho(A\cup B)$ still holds. If otherwise $A\BSET{e}$ and $B\BSET{e}$ are a modular pair in $M$, then $\cl_M((A\cap B)\BSET{e})\in C$ because $C$ is closed under intersections of modular pairs, which implies the contradiction $\beta = 0$. Finally, if $\beta=0$ then
% the submodularity of $\rk_M$ carries over.
So all premises for Theorem~\ref{thm:submodularIndependent} are satisfied, $N$ is a matroid and therefore $\phi$ is surjective.
\end{proof}

\noindent The next lemma summarizes how the family of flats behaves when a matroid is extended.

\begin{figure}[htb]
\begin{center}
	%\vspace*{1cm}
%
	\includegraphics[width=\textwidth]{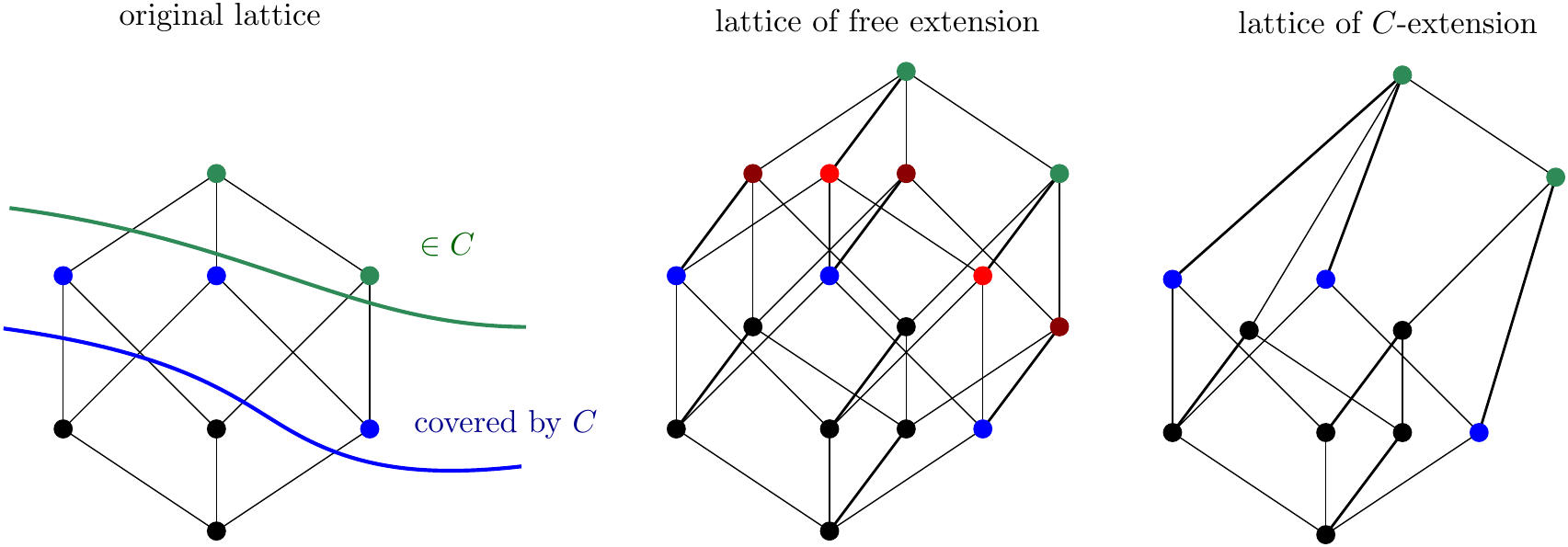}
\end{center}
\caption{Construction of the lattice of flats of a single element extension from the lattice of flats of the original matroid and the corresponding modular cut (Lemma~\ref{lem:flatsOfExtension}).} 
\end{figure}

\begin{lemma}\label{lem:flatsOfExtension}\PRFR{Feb 15th}
	Let $M=(E,\Ical)$ be a matroid, $e\notin E$, and $N\in \Xcal(M,e)$.
	Furthermore, let $C\in\Mcal(M)$ be the modular cut of $M$ where
	\[ C = \SET{F\in \Fcal(M) ~\middle|~ e\in \cl_N(F)}. \]
	Then 
	\begin{align*}
	 \Fcal(N)\,\,  =\,\, &\left( \Fcal(M) \BS C \right) \,\, \cup \,\, \SET{F\cup\SET{e} ~\middle|~ F\in C} 
	 		\,\,
	 		\\ & \cup \,\, \SET{F\cup\SET{e} ~\middle|~ F\in\Fcal(M)\BS C,\, \forall x\in E\BS F\colon\,\cl_M(F\cup\SET{x})\notin C}
	% \\& \hphantom{\left( \Fcal(M) \BS C \right) \,} \,\, \cup \,\, \SET{F\cup\SET{e} ~\middle|~ F\in C} \,\, \cup \,\, \SET{E\cup\SET{e}}.
		\\ = \,\, & \left( \Fcal(M) \BS C \right) \,\, \cup \,\, \SET{F\cup\SET{e} ~\middle|~ F\in\Fcal(M),\,F\in C \Leftrightarrow 
		\cl_N(F\cup\SET{e})\BSET{e} \in C} 
	.\end{align*}
%	In other words, the flats of the extension $N$ are consist of
%	 flats $F\in\Fcal(M)$ for which $e$ is a coloop of $N\restrict F\cup\SET{e}$
%	 and flats $F\cup\SET{e}$ for $F\in\Fcal(M)$ if either $F\in C$ or $F\notin C$ and $\cl_N(F\cup\SET{e}) = \cl_M(F)\cup\SET{e}$
\end{lemma}

\needspace{3\baselineskip}
\begin{proof}\PRFR{Feb 15th}
	First, we show that the second equation holds, which is implied by the equation
	\begin{align*} \SET{F\cup\SET{e} ~\middle|~ F\in C} 
	 		\,\,
	 		& \cup \,\, \SET{F\cup\SET{e} ~\middle|~ F\in\Fcal(M)\BS C,\, \forall x\in E\BS F\colon\,\cl_M(F\cup\SET{x})\notin C}
	 		\\&
	 	\!= \, \SET{F\cup\SET{e} ~\middle|~ F\in\Fcal(M),\,F\in C \Leftrightarrow 
		\cl_N(F\cup\SET{e})\BSET{e} \in C} . 
	\end{align*}
	If $F\in C$, then $F\in\Fcal(M)$ and $e\in \cl_N(F)$, so $\cl_N(F\cup\SET{e}) = F\cup\SET{e}$ 
	and therefore $\cl_N(F\cup\SET{e})\BSET{e} = F \in C$.
	If $F\notin C$ and for all $x\in E\BS F$ we have $\cl_M(F\cup\SET{x})\notin C$, i.e. whenever $F\notin C$ is not covered by a flat $G\in C$ in
	 $\Fcal(M)$, then $\cl_N(F\cup\SET{e}) = F\cup\SET{e}$, since otherwise $G = \cl_N(F\cup\SET{e})\BSET{e}$ would be
	  a maximal subset of $E$ with rank $\rk_N(G) = \rk_N(F) + 1 = \rk_M(F) + 1$, and therefore we would have $G\in\Fcal(M)$.
	 Furthermore, there would have to be some element
	 $g\in G\BS F$, so we would have found a flat $G\in C$ and with $\cl_M(F\cup\SET{g}) = G$, contradicting the assumption.
	 Thus we have $\cl_N(F\cup\SET{e})\BSET{e} = \left( F\cup\SET{e} \right)\BSET{e} = F \notin C$.
	 Therefore we obtain
	 \begin{align*} \SET{F\cup\SET{e} ~\middle|~ F\in C} 
	 		\,\,
	 		& \cup \,\, \SET{F\cup\SET{e} ~\middle|~ F\in\Fcal(M)\BS C,\, \forall x\in E\BS F\colon\,\cl_M(F\cup\SET{x})\notin C}
	 		\\&
	 	\!\subseteq  \, \SET{F\cup\SET{e} ~\middle|~ F\in\Fcal(M),\,F\in C \Leftrightarrow 
		\cl_N(F\cup\SET{e})\BSET{e} \in C} . 
	\end{align*}
	Now let $F'\in \SET{F\cup\SET{e} ~\middle|~ F\in\Fcal(M),\,F\in C \Leftrightarrow 
		\cl_N(F\cup\SET{e})\BSET{e} \in C}$ and $F= F'\BSET{e}$.
	If $F\in C$, then clearly $F'\in \SET{F\cup\SET{e} ~\middle|~ F\in C}$.
	If $F\notin C$, we give an indirect argument and assume that
	$$F'\notin\SET{F\cup\SET{e} ~\middle|~ F\in\Fcal(M)\BS C,\, \forall x\in E\BS F\colon\,\cl_M(F\cup\SET{x})\notin C}.$$
	So there is some $x\in E\BS F$ such that $\cl_M(F\cup\SET{x}) \in C$. Let $G = \cl_M(F\cup\SET{x})$,
	then $\cl_N(G) = G\cup\SET{e}$, thus $\rk_N(G\cup\SET{e}) = \rk_N(G) = \rk_M(G) = \rk_M(F) + 1 = \rk_N(F) + 1$.
	Thus $\cl_N(F\cup\SET{e}) = \cl_N(G) = G\cup\SET{e}$, so $\cl_N(F\cup\SET{e})\BSET{e} = G \in C$, contradicting the assumption
	that $F'$ is a member of the right-hand side set. Consequently, the second equation of the lemma holds.

	\noindent
	We show the inclusion of the left-hand side of the first equation in the right-hand side of the first equation.
	Let $X\in\Fcal(N)$, we have to treat two cases. If $e\notin X$, then the defining property of $X\in\Fcal(N)$ is that
	the strict inequality
	 $\rk_N(X\cup\SET{y}) > \rk_N(X)$ holds for all $y\in \left( E\cup\SET{e} \right)\BS X$. 
	 This together with $\rk_M = \rk_N\restrict_{2^E}$
	 implies that $X\in\Fcal(M)$. Furthermore, $X = \cl_N(X)$ implies $e\notin X$ and therefore $X\notin C$, so $X\in\Fcal(M)\BS C$.
%
%	\noindent
	Now assume that $e\in X$, and let $X' = X\BSET{e}$. If $e\in \cl_N(X')$, then clearly $X'\in C$ and
	so \linebreak
	 $X\in \SET{F\cup\SET{e} ~\middle|~ F\in C}$.
	  If otherwise $e\notin \cl_N(X')$, then we must have $\rk_N(X') = \rk_N(X) - 1$,
	  and $X'\in \Fcal(N)$ because $\cl_N(X') \subsetneq \cl_N(X) = X = X'\cup\SET{e}$.
	  As a consequence, we obtain that $X'\in \Fcal(M)$ and $X'\notin C$. Assume that for some $y\in E\BS X'$,
	  $\cl_M(X'\cup\SET{y})\in C$, then $e\in \cl_N(X'\cup\SET{y})$ so $X = \cl_N(X')\cup\SET{e}$ would be a proper subset of the flat
	  $\cl_N(X'\cup\SET{y})$ of rank $\rk_N(X') + 1$, but $\rk_N(X) = \rk_N(X') + 1$, which is impossible. Therefore, 
	  for all $y\in E\BS X'$ we have $\cl_M(X'\cup\SET{y})\notin C$. Thus we obtain
	  $$ X\in \SET{F\cup\SET{e} ~\middle|~ F\in\Fcal(M)\BS C,\, \forall x\in E\BS F\colon\,\cl_M(F\cup\SET{x})\notin C}.$$

	 \noindent
	 Finally, we show the inclusion of the right-hand side of the first equation in the left-hand side of the first equation.
	 Let $X\in \Fcal(M) \BS C$, then $e\notin \cl_N(X)$, so $\cl_N(X) = \cl_M(X) = X$, thus $X\in\Fcal(N)$.
	 Let $X\in C$, and let $X' = X\cup\SET{e}$. Then $\cl_N(X) =$\linebreak$ \cl_M(X) \cup\SET{e} = X\cup\SET{e}$, 
	 and therefore $X\cup\SET{e}\in\Fcal(N)$. Now let $X\notin C$ and for all $y\in E\BS X$,
	 $\cl_M(X\cup\SET{y})\notin C$. Let $G = \cl_N(X\cup\SET{e})$.
	  Assume that we have the proper inclusion $X\cup\SET{e} \subsetneq G$,
	 then $\rk_N(G) = \rk_N(X) + 1$ yields that there is some \linebreak $g\in G\BS\left( X\cup\SET{e} \right)$ such that
	 $\cl_N(X\cup\SET{g}) = G$. This leads us to the contradiction $G\BSET{e} = \cl_M(X\cup\SET{g}) \in C$. 
	 Therefore we must have $X\cup\SET{e} = G\in\Fcal(N)$.
\end{proof}

\clearpage
% -*- root: ../thesis.tex -*-

\section{Theorems of Hall, Rado, Ore, and Perfect}

\PRFR{Jan 15th}
D.J.A.~Welsh gives the following very elegant generalization of the 
theorems of Rado and Hall in \cite{We71}. From this generalization, the theorems of Hall, Rado, Ore, and Perfect follow as an easy corollary each. Before we present the theorem, we need some definitions.

\begin{definition}\PRFR{Jan 15th}
	Let $I$ and $E$ be sets. A \deftext{family of subsets} of $E$ indexed by $I$ is
	a map $A_{\bullet}\colon I\maparrow 2^{E}$ with domain $I$, such that for every
	$i\in I$ the image $A_{i}$ is a subset of $E$.
	We denote such a family by writing $(A_i)_{i\in I} \subseteq E$, or shorter\label{n:Afam}
	$(A_i)_{i\in I}$ whenever $E$ is clear from the context. We call $(A_i)_{i\in I}$
	\deftext{finite} if $I$ is finite. Further, we call $(A_i)_{i\in I}$ a \deftext{family of non-empty subsets}, if for all $i\in I$, $A_{i} \not= \emptyset$.
\end{definition}

\begin{definition}\PRFR{Jan 15th}
	Let $I$, $E$ be sets, and let $\Acal = (A_i)_{i\in I} \subseteq E$ be a family of subsets of $E$. A \deftext{system of representatives} is a map $x_{\bullet}\colon I\maparrow E$ such that there is a bijection $\sigma \colon I\maparrow I$ with 
	$x_{i} \in A_{\sigma(i)}$ for all $i\in I$. We will denote such a family
	by writing 
	\linebreak
	$(x_i)_{i\in I} \in \Acal$. A system of representatives is called
	\deftext{system of distinct representatives}, if $x_{\bullet}$ is an injective map.
	A \deftext{transversal} of $\Acal$ is a subset $T\subseteq E$ such that 
	there is a bijection $\sigma\colon T\maparrow I$ with $t\in A_{\sigma(t)}$ for all $t\in T$. A \deftext{partial transversal} of $\Acal$ is a subset $P\subseteq E$ such that
	there is an injection $\iota \colon P\maparrow I$ with $t\in A_{{\iota(t)}}$ for all $t\in P$.
	If $P$ is a partial transversal of $\Acal$, we define the \deftext[defect of a partial transversal]{defect} of $P$ to be $\left| I \right| - \left| P \right|$,
	i.e. the cardinality of those indices in $I$ that are not in the image of the corresponding $\iota$.
\end{definition}

\begin{theorem}\label{thm:radohall}\PRFR{Jan 15th}
 Let $\Acal=(A_i)_{i\in I} \subseteq E$ be a finite family of non-empty subsets of $E$,
 and let $\mu\colon 2^{E} \maparrow \N$ be a map with the properties that
 \begin{enumerate}\ROMANENUM
 	\item for all $X \subseteq Y\subseteq E$, $\mu(X) \leq \mu(Y)$, and
 	\item for all $X,Y\subseteq E$, $\mu(X) + \mu(Y) \geq \mu(X\cap Y) + \mu(X\cup Y)$.
 \end{enumerate}
 Then there is a system of representatives $(x_i)_{i\in I} \in \Acal$ with the property that
 \begin{enumerate}
 	\item[(1)] for all $J\subseteq I$, $\mu\left(\SET{x_i \mid i\in J}\right) \geq \left| J\right|$
\end{enumerate}
if and only if $\Acal$ has the property that
\begin{enumerate}
	\item[(2)] for all $J\subseteq I$, $\mu \left(\bigcup_{i\in J} A_i\right) \geq \left| J\right|$.
\end{enumerate}
\end{theorem}

\noindent This proof of the theorem follows the course of \cite{We71} --- a very nice version of which can be found on p.100 of \cite{We76} --- and it focuses more on details than brevity.

\begin{proof}\PRFR{Jan 15th}
	Let $(x_{i})_{i\in I} \in \Acal$ be such a system of representatives, that
	{\em(1)}  holds, and let $\sigma\colon I\maparrow I$ be a permutation
	that has the property $x_{i}\in A_{\sigma(i)}$ for all $i\in I$. Let $J\subseteq I$, then
	$\SET{x_{\sigma^{-1}(i)} ~\middle|~ i \in J} \subseteq \bigcup_{{i\in J}} A_{i}$.
	By {\em (i)} $\mu$ is non-decreasing, therefore
	\[ \left| J \right| = \left| \sigma^{{-1}}[J] \right| \leq 
	\mu\left(\SET{x_i  ~\middle|~  i\in \sigma^{-1}[J]}\right) \leq \mu \left(\bigcup_{i\in J} A_i\right).\]
	For the converse implication, we employ induction on the integer vector $v = \left( \left|A_{i}\right| \right)_{i\in I}$. The base case is $v_{i} = 1$ for all $i\in I$ where every $A_{i}$ is a singleton set, thus for any system of representatives $(x_i)_{i\in I} \in \Acal$, we have $A_{i} = \SET{x_{\sigma^{-1}(i)}}$ for all $i\in I$.
	Therefore, $\SET{x_i ~\middle|~  i\in \sigma^{-1}[J]} = \bigcup_{i\in J} A_{i}$ and the equivalence is obvious. For the induction step, let $i'\in I$ such that $\left| A_{i'} \right| > 1$. In this case, we claim that there is some $x\in A_{i'}$, such that
	the derived family
	$\Acal' = (A'_i)_{i\in I}$ where $A'_{i} = A_{i}$ if $i\not= i'$, and $A'_{i'} = A_{i'}\BSET{x}$
	still has the property {\em (2)}. Assume that this claim is false, then for any
	$\dSET{x,y}\subseteq A_{i'}$ there are $J_{x},J_{y}\subseteq I\BSET{i'}$ such that
	\begin{align*}
		\mu\left( \left(A_{i'}\BSET{x}\right) \cup \bigcup_{i\in J_x} A_i \right) &
		\leq \left|J_x \right| < \left|J_x\right| + 1 \text{, and}\\
		\mu\left( \left(A_{i'}\BSET{y}\right) \cup \bigcup_{i\in J_y} A_i \right) &
		\leq \left|J_y \right| <  \left|J_y\right| + 1 .\\
	\end{align*}
	We use the submodularity  {\em (ii)} of $\mu$ in order to obtain that
	\begin{align*}
		\mu\left( \left(A_{i'}\BSET{x}\right) \cup \bigcup_{i\in J_x} A_i \right) + \mu\left( \left(A_{i'}\BSET{y}\right) \cup \bigcup_{i\in J_y} A_i \right) & \geq  \mu(B_{\cap}) + 
		\mu\left( A_{i'} \cup \bigcup_{i\in J_x\cup J_y} A_i \right)
	\end{align*}
	where \[ B_{\cap} = \left( \left(A_{i'}\BSET{x}\right) \cup \bigcup_{i\in J_x} A_i \right) \cap 
	\left( \left(A_{i'}\BSET{y}\right) \cup \bigcup_{i\in J_y} A_i \right).\]
	Clearly, $\bigcup_{i\in J_x \cap J_y} A_{i} \subseteq B_{\cap}$, and since $\mu$ is non-decreasing due to property {\em (i)},
	we obtain that
	\begin{align*}
	\mu(B_{\cap}) + 
		\mu\left( A_{i'} \cup \bigcup_{i\in J_x\cup J_y} A_i \right) & \geq
		\mu\left(\bigcup_{i\in J_x\cap J_y} A_i\right) + \mu\left( A_{i'} \cup \bigcup_{i\in J_x\cup J_y} A_i \right).
	\end{align*}
	We now may use property {\em (2)} with $J = J_{x}\cup J_{y} \cup \SET{i'}$, and  $J=J_{x}\cap J_{y}$, respectively. We add the respective inequalities and obtain
	\begin{align*}
	\mu\left( A_{i'} \cup \bigcup_{i\in J_x\cup J_y} A_i \right) + \mu\left( \bigcup_{i\in J_x\cap J_y} A_i \right) & \geq \left( \left| J_x \cup J_y \right| + 1\right) +
	\left| J_x \cap J_y \right| = \left| J_x \right| + \left| J_y \right| + 1.
	\end{align*}
	Yet, this yields
	\begin{align*}
	\mu\left( \left(A_{i'}\BSET{x}\right) \cup \bigcup_{i\in J_x} A_i \right) + \mu\left( \left(A_{i'}\BSET{y}\right) \cup \bigcup_{i\in J_y} A_i \right) & \geq \left| J_x \right| + \left|J_y\right| + 1
	\end{align*}
	which contradicts
	\begin{align*}
	\mu\left( \left(A_{i'}\BSET{x}\right) \cup \bigcup_{i\in J_x} A_i \right) + \mu\left( \left(A_{i'}\BSET{y}\right) \cup \bigcup_{i\in J_y} A_i \right) & \leq \left| J_x \right| + \left|J_y\right|.
	\end{align*}
	Thus the claim holds, and since $\left|A'_{i'}\right| < v_{i'}$, we may use the induction hypothesis on $\Acal'$ which guarantuees the existence of a system of representatives $(x_i)_{i\in I}$ with
	property {\em (1)}. Every such $(x_i)_{i\in I}$ is also a system of representatives of
	$\Acal$, therefore $(x_i)_{i\in I}$ with {\em (1)} exists.
\end{proof}

\begin{corollary}[Hall]\label{cor:Hall}\PRFR{Jan 15th}
 Let $\Acal = (A_i)_{i\in I}$ be a finite family of sets, then $\Acal$ has a transversal if and only if for all $J\subseteq I$,
$$\left| \bigcup_{{i\in J}} A_{i} \right| \geq \left| J \right|.$$
\end{corollary}

\begin{proof} \PRFR{Jan 15th}
	Apply Theorem~\ref{thm:radohall} with $\mu(X) = \left| X \right|$ and $E = \bigcup_{{i\in I}} A_{i}$.
\end{proof}

\begin{corollary}[Rado]\label{cor:Rado}\PRFR{Jan 15th}
 Let $M=(E,\Ical)$ be a matroid, and let $\Acal = (A_i)_{i\in I}$ be a finite family of subsets of $E$, then $\Acal$ has a transversal which is independent in $M$ if and only if for all $J\subseteq I$,
$$\rk_M \left( \bigcup_{{i\in J}} A_{i}\right) \geq \left| J \right|.$$
\end{corollary}

\begin{proof}\PRFR{Jan 15th}
	Apply Theorem~\ref{thm:radohall} with $\mu(X) = \rk_M(X)$.
\end{proof}

\begin{corollary}[Ore]\PRFR{Jan 15th} Let $\Acal = (A_i)_{i\in I}$ be a finite family of sets, and $d\in \N$, then $\Acal$ has a partial transversal $T$ with defect $\leq d$ if and only if for all $J\subseteq I$,
$$\left| \bigcup_{{i\in J}} A_{i} \right| \geq \left| J \right| - d.$$
\end{corollary}
\begin{proof}\PRFR{Jan 15th}
	Apply Theorem~\ref{thm:radohall} with $\mu(X) = \left| X \right| + d$ and $E = \bigcup_{{i\in I}} A_{i}$.
\end{proof}

\needspace{9\baselineskip}
\begin{corollary}[Perfect]\PRFR{Jan 15th} Let $M=(E,\Ical)$ be a matroid, $d\in \N$, and let $\Acal = (A_i)_{i\in I}$ be a finite family of subsets of $E$, then $\Acal$ has a partial transversal $T$ with defect $ \leq d$ which is independent in $M$ if and only if for all $J\subseteq I$,
$$\rk_M \left( \bigcup_{{i\in J}} A_{i}\right) \geq \left| J \right| - d.$$
\end{corollary}
\begin{proof}\PRFR{Jan 15th}
	Apply Theorem~\ref{thm:radohall} with $\mu(X) = \rk_M(X) + d$.
\end{proof}

% -*- root: ../thesis.tex -*-

\subsection{Matroids Induced by Bipartite Graphs}

\PRFR{Jan 15th}
In this section, we describe how matchings in bipartite graphs can be used to induce
a matroid on one color-class from a matroid given on the other color-class. The class of
transversal matroids consists of those matroids, which can be obtained from a free matroid by bipartite matroid induction.

\begin{definition}\PRFR{Jan 15th}
    Let $D = (V,A)$ be a digraph, and let $M\subseteq \binom{V}{2}$ be a set of unordered pairs of vertices of $D$. We call $M$ a \deftext[matching in D@matching in $D$]{matching in $\bm D$},
    if the sets in $M$ are pair-wise disjoint, and if for every $\dSET{x,y} \in M$, there is an arc
    $(x,y)\in A$ or $(y,x)\in A$.
\end{definition}

\begin{definition}\PRFR{Jan 15th}
	Let $A,B$ be sets with $A\cap B = \emptyset$ and $\Delta\subseteq A\times B$.
	We call the digraph \label{n:DABD}$D=(A\disunion B, \Delta)$ the \index{bipartite graph}\deftext{directed bipartite graph} for
	$\Delta$ from $A$ to $B$.
\end{definition}

\PRFR{Jan 15th}
\noindent If there are no isolated vertices in the directed bipartite graph
 $(A\disunion B,\Delta)$, then the partition of its vertices into $A$ and $B$ 
 can be deduced from $\Delta$. Thus it is reasonable to identify the directed bipartite
 graph  $(A\disunion B,\Delta)$ with its arcs $\Delta$.

 \begin{definition}\PRFR{Jan 15th}\label{def:arcSystemDelta}
 	Let $A,B$ be finite sets with $A\cap B = \emptyset$,
 	 and let $\Delta\subseteq A\times B$. The
 	\deftext[arc system of ABD@arc system of $(A\disunion B,\Delta)$]{arc system of $\bm{(A\disunion B,\Delta)}$} shall be 
 	denoted by $\Acal_{\Delta}$. It is defined to be the family\label{n:ArcSystem} $\Acal_{\Delta} = (A_i)_{i\in B} \subseteq A$ where
 	\[
 		A_b = \SET{a\in A \mid (a,b) \in \Delta}
 	\]
 	for every $b\in B$.
 \end{definition}

 \begin{theorem}\label{thm:bipartiteInduction}\PRFR{Jan 15th}
 	Let $D,E$ be finite sets with $D\cap E = \emptyset$, $M=(E,\Ical)$ be a matroid,
 	and let $\Delta\subseteq D\times E$. Furthermore, let $N=(D,\Ical')$ be 
 	such that $\Ical'\subseteq 2^{D}$ with the defining property
 	that for all $X\subseteq D$, $X\in \Ical'$ if and only if $X$ is a partial transversal of the arc system
 	$\Acal_\Delta$ such that there is an injective map $\iota\colon X\maparrow E$ with
 	$x\in A_{\iota(x)}$ for all $x\in X$ with the additional property that
 	$\SET{\iota(x)\mid x\in X}$ is independent in $M$.
 	Then $N$ is a matroid.
 \end{theorem}

 \noindent The following proof is based on the proof in \cite{We76}, p.119.
 \begin{proof}\PRFR{Jan 15th}
 	Let $\mu\colon 2^{D}\maparrow \N$ be the map where \[ \mu(X) =
 	\rk_{M}\left( \SET{e\in E\mid \exists x\in X\colon\, x\in A_e} \right) \]
 	for every $X\subseteq D$. Clearly, $\mu(\emptyset) =\rk_M (\emptyset) =
 	0$ and $\mu$ is non-decreasing and submodular. By
 	Theorem~\ref{thm:submodularIndependent}, the set \[ \Ical'' =
 	\SET{X\subseteq D \mid \forall X'\subseteq X\colon \mu(X') \geq \left| X'
 	\right|}\] defines a matroid $N' = (D,\Ical'')$. We show that $\Ical' =
 	\Ical''$. Let $X\in\Ical'$, and let $\iota\colon X\maparrow E$ an
 	injective map such that $\iota[X]\in \Ical$ and for all $x\in X$, $x\in
 	A_{\iota(x)}$. Then, clearly, for all $X'\subseteq X$, $\iota[X']
 	\subseteq  \SET{e\in E\mid \exists x'\in X'\colon\, x'\in A_e}$.
 	Therefore for all $X'\subseteq X$, $$ \left| X' \right| =
 	\rk_{M}\left(\iota[X']\right) \leq \rk_{M}\left(  \SET{e\in E\mid \exists
 	x'\in X'\colon\, x'\in A_e}\right) = \mu(X'), $$ thus $X\in \Ical''$.
 	Conversely, assume that $X\in\Ical''$.  We flip around the arc system and and consider the
 	following family of subsets of $E$: Let $\Bcal_{X} = (B_i)_{{i\in X}}
 	\subseteq E$ be the family of subsets of $E$ where for $x\in X$, the
 	subset \( B_{x} = \SET{e\in E\mid x\in A_e} \) consists of all elements
 	of $E$ that $x$ can pair with in $\Delta$. By Rado's Theorem
 	(Corollary~\ref{cor:Rado}), the family $\Bcal_{X}$  has a transversal
 	$Y\subseteq E$ that is independent in $M$, if and only if for all
 	$X'\subseteq X$ we have the inequality \[ \rk_M \left( \bigcup_{{x'\in X'}} B_{x'}\right) \geq
 	\left| X' \right|. \] But \(  \bigcup_{x'\in X'} B_{x'} = \bigcup_{x' \in
 	X'} \SET{e\in E\mid x'\in A_e}  = \SET{e\in E\mid \exists x'\in
 	X'\colon\, x'\in A_e} \), which gives that $\rk_M \left( \bigcup_{{x'\in
 	X'}} B_{x'}\right) = \mu(X')$.  By definition, $X\in \Ical''$ implies that
 	for all $X'\subseteq X$, $\mu(X') \geq \left| X' \right|$. Thus we may
 	infer that there is an $M$-independent transversal $Y$ of $\Bcal_{X}$.
 	This gives rise to a bijective map $\sigma\colon Y\maparrow X$ such that for every
 	$y\in Y$ we have $y\in B_{\sigma(y)}$. Yet $y\in B_{\sigma(y)}$ implies that
 	$\sigma(y) \in A_{y}$. Therefore there is an injective map
 	$\tilde\iota\colon X\maparrow E$ with $\tilde\iota(x) =
 	\sigma^{-1}(x)$ which witnesses that $X$ is a partial transversal of
 	$\Acal_{\Delta}$, and therefore $X\in \Ical'$.
 \end{proof}

 \begin{remark}\PRFR{Jan 15th}
 Obviously, the premise $D\cap E = \emptyset$ in Theorem~\ref{thm:bipartiteInduction} may be dropped,
 since we may give the elements of $D$ distinct names $D'$, apply the construction, and then rename the elements of the so obtained matroid back to $D$.
 \end{remark}

 \begin{definition}\PRFR{Jan 15th}
   Let $D,E$ be finite sets, $\Delta\subseteq D\times E$, and let $M_0=(E,\Ical)$ be a matroid. The \deftext[matroid induced by D from M@matroid induced by $\Delta$ from $M$]{matroid induced by $\bm\Delta$ from $\bm{M_{\bm 0}}$} shall be the pair\label{n:MDM0}
   $M(\Delta,M_0) = (D,\Ical_{\Delta,M_0})$ where $\Ical_{\Delta,M_0} \subseteq 2^{D}$
consists of all partial transversals $X$ of the family $\Acal_{\Delta}$ that can be admissibly injected into an independent set of $M_{0}$, i.e. there is an injective map $\iota\colon X\maparrow E$ such that for all $x\in X$, $x\in A_{\iota(x)}$ and $\iota[X] \in \Ical$.
 \end{definition}
% -*- root: ../thesis.tex -*-

\subsection{Transversal Matroids}\label{sec:shortTransversalMatroids}

\PRFR{Jan 15th}
\noindent In this section, we provide the definition of transversal matroids and forestall
those important properties of transversal matroids, that we need in order to develop the theory of
routings in directed graphs --- which, in turn, is essential for defining gammoids. Further properties of transversal matroids are treated in Section~\ref{sec:TransversalMatroids}.

\begin{definition}\PRFR{Jan 15th}
	Let $E,I$ be a finite sets, and $\Acal = (A_i)_{i\in I} \subseteq E$ be a family of subsets. The \deftext[transversal matroid presented by A@transversal matroid presented by $\Acal$]{transversal matroid presented by $\bm{\Acal}$} shall be the pair\label{n:MAEIA}
	$M(\Acal) = (E,\Ical_\Acal)$ where $\Ical_\Acal \subseteq 2^{E}$ with
	the property that for all $X\subseteq E$, $X\in \Ical_{\Acal}$ if and only if
	$X$ is a partial transversal of $\Acal$.
\end{definition}

\begin{corollary}\label{cor:transversalMatroid}\PRFR{Jan 15th}
	Let $E,I$ be a finite sets, and $\Acal = (A_i)_{i\in I} \subseteq E$ be a family of subsets.
	Then $M(\Acal) = (E,\Ical_\Acal)$ is a matroid.
\end{corollary}

\begin{proof}\PRFR{Jan 15th}
	Let $M_0 = (I,2^I)$ be the free matroid on $I$, and let
	$\Delta = \SET{(e,i)\in E\times I\mid e\in A_i}$.
	Then $M(\Acal) = M(\Delta,M_0)$ is the matroid induced by $\Delta$ from $M_{0}$,
	which is a matroid by Theorem~\ref{thm:bipartiteInduction}.
\end{proof}

\noindent
We just proved that the maximal partial transversals are bases of a matroid $M(\Acal)$.

\begin{corollary}\label{cor:maximalpartialtransversals}\PRFR{Jan 15th}
	Let $E,I$ be finite sets, and let  $\Acal = (A_i)_{i\in I} \subseteq E$ be a family of subsets.
	Two maximal partial transversals $S,T\subseteq E$ of $\Acal$ have
	$\left| S \right| = \left| T \right|$.
\end{corollary}

\begin{definition}\PRFR{Jan 15th}
	Let $M=(E,\Ical)$ be a matroid. We call $M$ \deftext{transversal matroid}, if
	there is a finite family of subsets $\Acal = (A_i)_{i\in I} \subseteq E$, such that
	$M = M(\Acal)$, i.e. $M$ is the transversal matroid presented by $\Acal$.
\end{definition}

\clearpage
% -*- root: ../thesis.tex -*-

\section{Directed Graphs}

\PRFR{Jan 22nd}
In this section, we present the basic definitions and properties of directed
graphs used in the course of this work. We aim to be consistent regarding terminology with 
the monograph {\em Digraphs: Theory, Algorithms, and Applications} by J.~Bang-Jensen
and G.~Gutin \cite{BJG09}, although we may divert from it in technical details since we do not need 
the full
generality of \cite{BJG09}. 

\begin{definition}\label{def:directedGraph}\PRFR{Jan 22nd}
    A pair $D = (V,A)$\label{n:DVA} is called \deftext{directed graph}, or shorter
    \deftext{digraph}, whenever $V$ is a finite set and $A\subseteq V\times V$.
    Every $v\in V$ is called \deftext{vertex} of $D$ and every $a = (u,v)\in A$ is
    called \deftext{arc} of $D$. Furthermore, $u$ is called the \deftext{tail} of
    the arc $a$ and $v$ is called the \deftext{head} of $a$. We also say that $a=(u,v)$ is an arc that \deftextX{leaves} $\bm u$ 
    and \deftextX{enters} $\bm v$, or shorter that $a$ \deftextX{goes from} $\bm u$ \deftextX{to} $\bm v$.
    Furthermore, $u$ and $v$ are the \deftext[end vertices of an arc]{end vertices} of $a$, and we say that $u$ and $v$ are \deftext{incident} with $a$.
    Two vertices that are incident with the same arc are called \deftext{adjacent}.
    An arc $a = (u,v)$
    with $u=v$ is called a \deftext[loop (digraph)]{loop}.
\end{definition}

% -*- root: ../../thesis.tex -*-

\vspace*{-\baselineskip} %Remove the line space created by the tilde below
\begin{wrapfigure}{l}{5cm}
\vspace{\baselineskip}
\begin{centering}~~%move the picture slightly to the right
\begin{tikzpicture}
\NodeB{U}{label={left:{$u$}}}  at (-1,1) {};
\NodeB{V}{label={below:{$v$}}} at ( 1,1) {};
\NodeB{W}{label={above:{$w$}}}  at ( 1,2) {};
\NodeB{S}{label={left:{$s$}}}  at ( 0,1) {};
\NodeB{T}{label={right:{$t$}}} at ( 2,2) {};
\draw [->] (S) .. controls (.5,2) and (.5,2) .. (W);
\draw [->] (S)->(V);
\draw [->] (V) .. controls (1.5,1) and (1.5, 1) .. (T);
\draw [->] (W)->(T);
\draw [<-] (V) .. controls (.7,1.2) and (.7,1.8) .. (W);
\draw [->] (V) .. controls (1.3,1.2) and (1.3,1.8) .. (W);
\path
    (U) edge [loop above,looseness=15,->,out=120,in=40] (U);
\end{tikzpicture}
\end{centering}%
\vspace*{-2\baselineskip} %make the picture more tightly cropped
\end{wrapfigure}
~ %The tilde creates a new dummy paragraph. WHY IS THAT NEEDED? -> would increase the space %
  % before the ex. environment. THE NEXT FREE LINE IS ESSENTIAL!

\begin{example} \label{ex:131}\PRFR{Jan 22nd}
	Consider $V=\dSET{u,v,w,s,t}$ and $A=\{(u,u),$ $(v,w),$ $(w,v),$ $(s,v),$
	$(s,w),$  $(w,t),$ $(v,t)\}$. Then $D = (V,A)$ is a directed graph. We can
	represent $D$ by a figure, where each vertex is represented by a small circle
	which may or may not have a name next to it,  and where each arc is
	represented by an arrow which points from the tail vertex circle of the arc to
	the head vertex circle of the arc. The figure on the left represents $D$ as
	above.
\end{example}

\PRFR{Jan 22nd}
\noindent Clearly, we can construct another directed graph $D^\opp$ from any
directed graph $D$ by swapping heads and tails of all arcs of $D$, thus
effectively reorienting all arcs to their opposite direction.

\begin{definition} \PRFR{Jan 22nd}
    Let $D = (V,A)$ be a digraph. The \deftext{opposite digraph} is defined to
    be the unique directed graph \( D^\opp = (V, A^\opp) \)\label{n:Dopp} with the property
    \[ (u,v)\in A^\opp \Leftrightarrow (v,u)\in A.\]
\end{definition}

\noindent It is easy to see that $(D^\opp)^\opp = D$.

% -*- root: ../../thesis.tex -*-

\needspace{4\baselineskip}

\vspace*{-\baselineskip} %Remove the line space created by the tilde below
\begin{wrapfigure}{r}{5cm}
\vspace{\baselineskip}
\begin{centering}~~%move the picture slightly to the right
\begin{tikzpicture}
\NodeB{U}{label={left:{$u$}}}  at (-1,1) {};
\NodeB{V}{label={below:{$v$}}} at ( 1,1) {};
\NodeB{W}{label={above:{$w$}}}  at ( 1,2) {};
\NodeB{S}{label={left:{$s$}}}  at ( 0,1) {};
\NodeB{T}{label={right:{$t$}}} at ( 2,2) {};
\draw [<-] (S) .. controls (.5,2) and (.5,2) .. (W);
\draw [<-] (S)->(V);
\draw [<-] (V) .. controls (1.5,1) and (1.5, 1) .. (T);
\draw [<-] (W)->(T);
\draw [->] (V) .. controls (.7,1.2) and (.7,1.8) .. (W);
\draw [<-] (V) .. controls (1.3,1.2) and (1.3,1.8) .. (W);
\path
    (U) edge [loop above,looseness=15,->,out=120,in=40] (U);
\end{tikzpicture}
\end{centering}%
\vspace*{-2\baselineskip} %make the picture more tightly cropped
\end{wrapfigure}
~ %The tilde creates a new dummy paragraph. WHY IS THAT NEEDED? -> would increase the space %
  % before the ex. environment. THE NEXT FREE LINE IS ESSENTIAL!

\begin{example} \PRFR{Jan 22nd}
	Consider the digraph $D$ from Example~\ref{ex:131}.  Its opposite digraph
	$D^\opp$ has the same vertex set as $D$, whereas the arcs are reversed to
	$A^\opp=\{(u,u),$ $(w,v),$ $(v,w),$ $(v,s),$   $(w,s),$  $(t,w),$ $(t,v)\}$.
\end{example}

\begin{definition} \PRFR{Jan 22nd}
    Let $D = (V,A)$ be a digraph, $x\in V$. We call $x$ a \deftext{source} in
    $D$ if $x$ is never the head of an arc in $D$. Analogously, we call $x$ a
    \deftext{sink} in $D$ if $x$ is never the tail of an arc in $D$.
\end{definition}

\noindent From this definition it is clear that $x$ is a source in $D$ if and
only if $x$ is a sink in $D^\opp$, and analogously, $x$ is a sink in $D$
if and only if $x$ is a source in $D^\opp$.

% -*- root: ../../thesis.tex -*-

\begin{example} \PRFR{Jan 22nd}
	Consider the digraph $D$ from Example~\ref{ex:131}.  The vertex $s$ is a
	source in $D$ and a sink in $D^\opp$, the vertex $t$ is a sink in $D$ and a
	source in $D^{\opp}$, whereas the vertices $u,v,w$ are neither sinks nor
	sources in both $D$ and $D^{\opp}$.
\end{example}

\begin{definition}\PRFR{Jan 22nd}
    Let $D = (V,A)$ be a digraph, the \deftextX{outer-extension operator in $\bm D$}
    shall be the map\label{n:DclD}
    \[ \DclD{\bullet}{D}\colon 2^V \maparrow 2^V,\,X\mapsto \DclD{X}{D} \]
    where
    \[ \DclD{X}{D} = X \cup \SET{v\in V\mid \exists x\in X\colon\,(x,v)\in A}.\]
    We call $\DclD{X}{D}$ the \deftext[outer extension of X in D@outer extension of $X$ in $D$]{outer extension of $\bm X$ in $\bm D$}.
    If the digraph is clear from the context, we write $\Dcl{\bullet}$ for $\DclD{\bullet}{D}$. The \deftextX{outer-margin operator in $\bm D$} is defined to be the map\label{n:partD}
    \[ \partial_D \bullet \colon 2^V \maparrow 2^V,\,X\mapsto \partial_D X\]
    where \[\partial_D X = \left.\DclD{X}{D} \right\BS X.\] We call $\partial_D X$ the \deftext[outer margin of X in D@outer margin of $X$ in $D$]{outer margin of $\bm X$ in $\bm D$}. Again, if no confusion can occur, we write $\partial \bullet$ as a shorthand for $\partial_D \bullet$.
\end{definition}

% -*- root: ../../thesis.tex -*-

\needspace{6\baselineskip}
\vspace*{-\baselineskip} %Remove the line space created by the tilde below
\begin{wrapfigure}{l}{8.3cm}
\vspace{\baselineskip}
\begin{centering}~~
\includegraphics[scale=1.2]{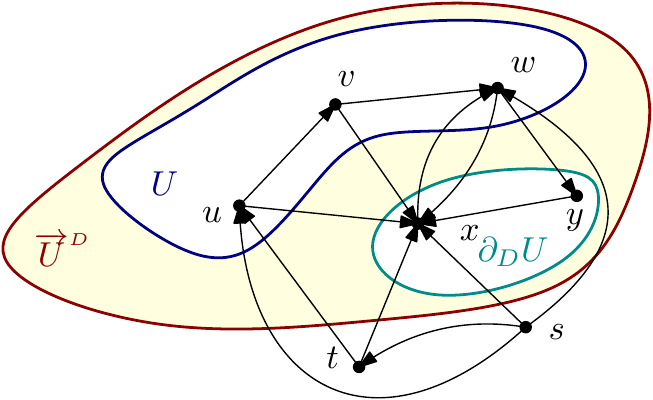}
\end{centering}%
\vspace*{-6\baselineskip} %make the picture more tightly cropped
\end{wrapfigure}
~ %The tilde creates a new dummy paragraph. WHY IS THAT NEEDED? -> would increase the space %
  % before the ex. environment. THE NEXT FREE LINE IS ESSENTIAL!

\begin{example} \PRFR{Jan 19th}
	Consider the digraph shown on the left. Let $U=\SET{u,v,w}$. The outer extension of $U$ is $\DclD{U}{D} = \SET{u,v,w,x,y}$ and the outer margin of $U$ is $\partial_D U = \SET{x,y}$.
\end{example}

~\vspace*{1.5cm}

\needspace{6\baselineskip}
\begin{definition}\PRFR{Jan 22nd}
    Let $D = (V,A)$ be a digraph, $\N \ni n>0$, and \label{n:walk}$w=(w_i)_{i=1}^n \in
    V^{n}$.  Then $w$ is a \deftext{walk} in $D$, if for all
    $i\in\SET{1,2,\ldots, n-1}$ there is an arc $(w_i,w_{i+1})\in A$. The
    \deftext{start vertex} --- or \deftext{initial vertex} --- of $w$ is $w_{1}$, and the \deftext[end vertex of a walk]{end vertex}
    --- or \deftext{terminal vertex} --- of
    $w$ is denoted by $w_{-1} = w_{n}$. The set of vertices \deftext[visited by w@visited by $w$]{visited
    by $\bm w$} is denoted by \( \left|w\right| = \SET{w_1,w_2,\ldots,w_n}.\) 
    The set of all arcs \deftext[traversed by w@traversed by $w$]{traversed by $\bm w$} is denoted by
    \( \left| w\right|_A = \SET{(w_i,w_{i+1})\mid i=1,2,\ldots,n-1}.\)
    The set of all
    walks in $D$ is denoted by\label{n:PbfD}
    \[ \Wbf(D) = 
      \SET{w \in \bigcup_{n=1}^{\infty} V^n ~\middle|~ w\text{~is a walk in~}D}.\]
    The \deftext[length of $w$]{length} of the walk $w=(w_i)_{i=1}^n$ is $n$. A walk $w$ is
    called \deftext[trivial walk]{trivial}, if its length is $1$.  We say that a walk $w$ is
    a \deftext[path]{path}, if no vertex is visited twice by $w$, i.e. if
    $w_{i}=w_{j}$ already implies $i=j$. The family of paths in $D$ is denoted by
    \label{n:simplePath}
    \[ \Pbf (D) = \SET{p\in \Wbf(D) \mid p \text{~is a path}}.\]
    Furthermore, for all $u,v\in V$, we shall denote the set of all walks from $u$ to $v$ in $D$ by\label{n:PathUV}
    \[ \Wbf(D; u,v) = \SET{w\in \Wbf(D) \mid w_1=u \txtand w_{-1}=v} \]
    and the set of all paths from $u$ to $v$ in $D$ by\label{n:SPathUV}
    \[ \Pbf(D; u,v) = \SET{p\in \Pbf(D) \mid p_1 = u \txtand p_{-1} = v}. \qedhere\]
\end{definition}

\noindent Instead of $w=(w_i)_{i=1}^n=(w_1,w_2,\ldots,w_n)$ we shall also write $w_{1}w_{2}\ldots w_{n}$. Furthermore, we set the convention that $(w_{1}w_{2}\ldots w_{n})^i$ shall denote the walk consisting of $i$-iterations of the visited-vertex sequence $w_{1}w_{2}\ldots w_{n}$, i.e. $(abc)^{3}$ shall denote the non-path walk $abcabcabc$. 

% FALSCH!! (a->b->c->d) hat (a->b), (b->c), (c->d), (a->b->c) ... mehr als 3 Pfade!
%\begin{lemma}\label{lem:nonTrivialPathsHaveSomeArcs}
%    Let $D=(V,A)$ be a digraph, and let $P\subseteq \Pbf(D)$ be a family of paths in $D$
%    such that $\left| p \right| > 1$ for every $p\in P$.
%    Then $\left| A \right| \geq \left| P \right|$.
%\end{lemma}
%\begin{proof}
%    Let $p\in \Pbf(D)$ be any path in $D$ with $\left| p \right| > 1$, then
%    $\left| \left| p \right|_A  \right| = \left| p \right| - 1 \geq 1$.
%    Clearly, we can reconstruct $p$ from $X = \left| p \right|_A$ whenever $X\not= \emptyset$: the first arc traversed by $p$
%    is the arc $(u,v)\in X$ such that $\left( V\times \SET{u} \right) \cap X = \emptyset$, the arc traversed by $p$ after
%    the arc $(u,v)\in X$ is the single element in the set $X_v = \left( \SET{v}\times V \right)\cap X$;
%    if $X_v = \emptyset$, then $v$ is the terminal vertex of $p$.
%\end{proof}

\begin{definition}\PRFR{Jan 22nd}
    Let $D= (V,A)$ be a digraph, and let $w=(w_i)_{{i=1}}^{n}\in \Wbf(D)$ and $q=(q_i)_{i=1}^{m}\in\Wbf(D)$ be walks.
    Then we say that \deftext[compatible walks]{$\bm w$ is compatible with $\bm q$}, if $w_{n}=q_{1}$.
    In that case, we define the \deftext[concatenation of w and q@concatenation of $w$ and $q$]{concatenation of $\bm w$ and $\bm q$} to be the walk\label{n:pdotq}
    \( w.q = w_{1}w_{2}\ldots w_{n}q_{2}q_{3}\ldots q_{m} \).
\end{definition}

% -*- root: ../../thesis.tex -*-

\needspace{4\baselineskip}

\vspace*{-\baselineskip} %Remove the line space created by the tilde below
\begin{wrapfigure}{r}{5cm}
\vspace{\baselineskip}
\begin{centering}~~%move the picture slightly to the right
\begin{tikzpicture}
\NodeB{U}{label={left:{$u$}}}  at (-1,1) {};
\NodeB{V}{label={below:{$v$}}} at ( 1,1) {};
\NodeB{W}{label={above:{$w$}}}  at ( 1,2) {};
\NodeB{S}{label={left:{$s$}}}  at ( 0,1) {};
\NodeB{T}{label={right:{$t$}}} at ( 2,2) {};
\draw [->] (S) .. controls (.5,2) and (.5,2) .. (W);
\draw [->] (S)->(V);
\draw [->] (V) .. controls (1.5,1) and (1.5, 1) .. (T);
\draw [->] (W)->(T);
\draw [<-] (V) .. controls (.7,1.2) and (.7,1.8) .. (W);
\draw [->] (V) .. controls (1.3,1.2) and (1.3,1.8) .. (W);
\path
    (U) edge [loop above,looseness=15,->,out=120,in=40] (U);
\end{tikzpicture}
\end{centering}%
\vspace*{-2\baselineskip} %make the picture more tightly cropped
\end{wrapfigure}
~ %The tilde creates a new dummy paragraph. WHY IS THAT NEEDED? -> would increase the space %
  % before the ex. environment. THE NEXT FREE LINE IS ESSENTIAL!

\begin{example}\PRFR{Jan 22nd}
	Consider again the digraph $D$ from Example~\ref{ex:131}.  The trivial walks
	in $D$ are $u$, $v$, $w$, $s$, and $t$. The paths in $D$ of length
	greater than $1$ are $sw$, $sv$, $vw$, $vt$, $wv$, $wt$, $svw$, $svt$, $swt$,
	$svwt$, $swvt$. The non-path walks in $D$ are $u^{i}$ $(i\geq 2)$,
	$w(vw)^{j}$, $v(wv)^{j}$, $w(vw)^{j}v$, $v(wv)^{j}w$, $sw(vw)^{j}$,
	$sv(wv)^{j}$, $sw(vw)^{j}v$, $sv(wv)^{j}w$, $w(vw)^{j}t$, $v(wv)^{j}t$,
	$w(vw)^{j}vt$, $v(wv)^{j}wt$, $sw(vw)^{j}t$, $sv(wv)^{j}t$, $sw(vw)^{j}vt$,
	and $sv(wv)^{j}wt$  $(j\geq 1)$.
\end{example}

\begin{definition}\PRFR{Jan 22nd}
    Let $D=(V,A)$ be a digraph. A walk $w=(w_i)_{i=1}^n \in \Wbf(D)$ is called \deftext{cycle walk}, or shorter, \deftextX{cycle},
    if $w_1 = w_n$ and $w_2 w_3 \ldots w_n$ is a path.
\end{definition}

\noindent Observe that there is no {\em ``empty walk''}, thus the trivial walks are not considered to be cycles.

\begin{definition}\PRFR{Jan 22nd}
    Let $D=(V,A)$ be a digraph. $D$ shall be called \deftext{acyclic digraph}, if every walk $w\in\Wbf(D)$ is a path, i.e. whenever
    \( \Wbf(D) = \Pbf(D) \).
\end{definition}

\begin{corollary}\PRFR{Jan 22nd}
    Let $D=(V,A)$ be a digraph. Then $D$ is acyclic if and only if there is no cycle
     walk $w\in\Wbf(D)$.
\end{corollary}

% -*- root: ../thesis.tex -*-

\subsection{Routings and Transversals}

\PRFR{Jan 22nd}
In this section, we introduce a correspondence between families of pair-wise vertex disjoint paths in digraphs and
transversals of certain families of sets, which will be valuable for the study of gammoids.

\needspace{6\baselineskip}
\begin{definition}\PRFR{Jan 22nd}
	Let $D = (V,A)$ be a digraph, and $X,Y\subseteq V$. A \deftext{routing} from $X$ to $Y$ in $D$ is a family of paths $R\subseteq \Pbf(D)$ such that
	\begin{enumerate}\ROMANENUM
		\item for each $x\in X$ there is some $p\in R$ with $p_{1}=x$,
		\item for all $p\in R$ the end vertex $p_{-1}\in Y$, and
		\item for all $p,q\in R$, either $p=q$ or $\left|p\right|\cap \left|q\right| = \emptyset$.%, and
        %\item all $p\in R$ are simple.
	\end{enumerate}
	We will write $R\colon X\routesto Y$ in $D$ as a shorthand for ``$R$ is a routing from $X$ to $Y$
    in $D$'', and if no confusion is possible, \label{n:routing}
    we just write $X\routesto Y$ instead of $R$ and $R\colon X\routesto Y$.
    A routing $R$ is called \deftext{linking} from $X$ to $Y$, if it is a routing onto $Y$, i.e. whenever $Y = \SET{p_{-1}\mid p\in R}$.
\end{definition}

% -*- root: ../../thesis.tex -*-

\begin{figure}[t]
\begin{center}
\includegraphics[scale=1]{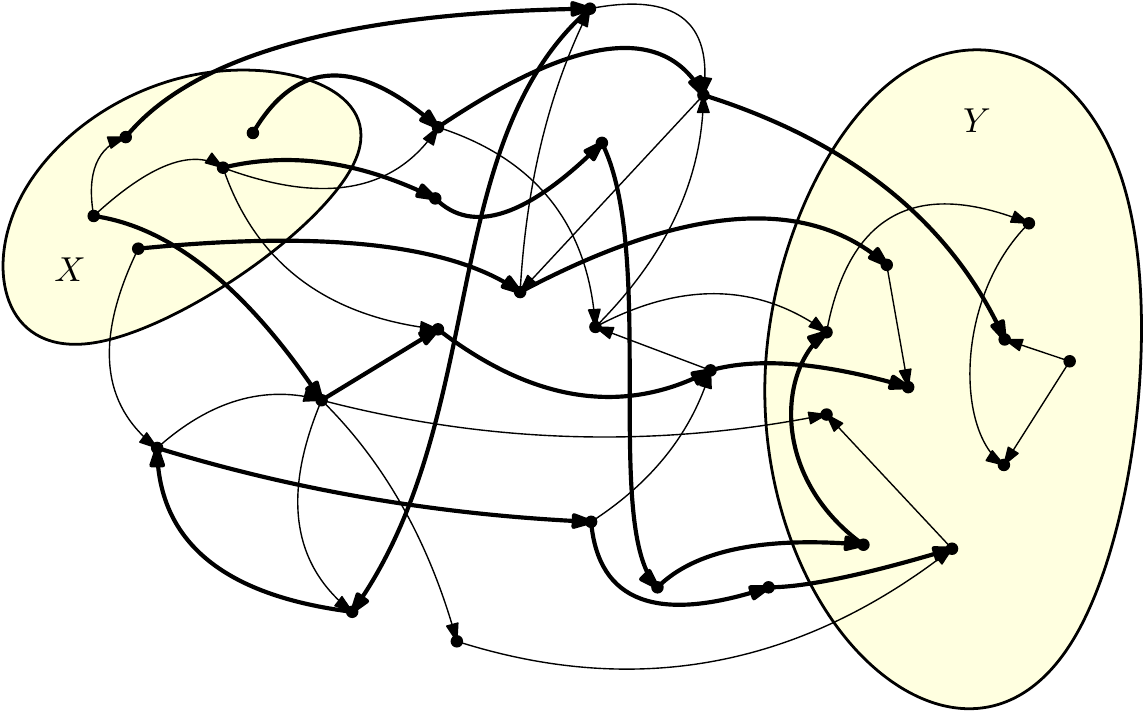}
\end{center}
\caption{Example of a routing $R\colon X\routesto Y$ in a digraph. The paths that belong to $R$ are depicted by bold arrows.}
\end{figure}

\begin{remark}\label{rem:straighteningWalks}\PRFR{Jan 22nd}
    We defined a routing to consist of paths only, but for most of what we are concerned with when using the concept of a routing in $D$,
     the property that the walks of $R$ are indeed paths is not crucial. Let $R'\subseteq \Wbf(D)$ be a family of walks such that 
     for all $w,q\in R'$ the implication $\left| w \right|\cap \left| q \right|\not= \emptyset \Rightarrow w=q$ holds.
      Now let $w\in R' \BS \Pbf(D)$
    be a non-path walk in $R'$. Then there is a vertex $v\in \left| w \right|$ such that
    $w = w_0 v p' v w_1$ where $w_0,w_1,p'\in \Wbf(D)$ such that $vp'v$ is a cycle walk.
     Clearly $\hat{w} = w_0 v w_1$ has less such cycle sub-walks than $w$ and $\left| w_0 v w_1 \right|\subseteq \left| w \right|$,
     thus we may iteratively straighten out any cycles in the walks from $R'$ without changing the 
     start and the end vertices. The property, that the family of walks consists of pair-wise vertex 
     disjoint walks, remains intact throughout the procedure.
      The result of this process is a family of paths which is a linking
       from $\SET{w_1\mid w\in R'}$ onto $\SET{w_{-1}\mid w\in R'}$ in $D$.
\end{remark}

\noindent The straightening out of cycle walks in routings is a special case of the following construction.

\needspace{4\baselineskip}
\begin{definition}\PRFR{Mar 7th}
    Let $D=(V,A)$ be a digraph, $w \in \Wbf(D)$ be a walk. Then $w$ is called \deftext[essential path in $D$]{essential path in $\bm D$}, if for all $w'\in \Wbf(D; w_1,w_{-1})$ with $\left| w' \right| \subseteq \left| w \right|$ we have
    $\left| w' \right| = \left| w \right|$.
    Let $R\subseteq \Pbf(D)$ be a routing, then $R$ is called \deftext[essential routing in $D$]{essential routing in $\bm D$},
    if  $p$ is an essential path in $D$ for all $p\in R$.
\end{definition}

\begin{lemma}\PRFR{Mar 7th}
    Let $D=(V,A)$ be a digraph and let $R\subseteq \Pbf(D)$ be a routing from $X$ to $Y$ in $D$. Then
    there is an essential routing from $X$ to $Y$ in $D$.
\end{lemma}
\begin{proof}\PRFR{Mar 7th}
    We show this by induction on the number of paths in $R$ that are not essential.
    In the base case, $R$ itself is an essential routing from $X$ to $Y$.
    Now let $p\in R$ be a path that is not essential in $D$. Then there is a path $p'\in \Pbf(D,p_1,p_{-1})$ 
    with $\left| p' \right|\subsetneq \left| p \right|$, such that $\left| p' \right|$ is $\subseteq$-minimal. 
    Such $p'$ is an essential path. 
    Then $\left( R\BSET{p} \right)\cup\SET{p'}$ is a routing from $X$ to $Y$ in $D$ with fewer non-essential paths,
    so by induction hypothesis there is an essential routing from $X$ to $Y$ in $D$.
\end{proof}

% -*- root: ../../thesis.tex -*-

\needspace{12\baselineskip}

\vspace*{-\baselineskip} %Remove the line space created by the tilde below
\begin{wrapfigure}{r}{4.5cm}
\vspace{\baselineskip}
\begin{centering}~~%move the picture slightly to the right
\begin{tikzpicture}
% TODO PICTURE OF BIPARTITE DIGRAPH WITH A LINKING IN IT!
\NodeB{A1}{label={left:{$a_{1}$}}}  at (-1,0) {};
\NodeB{A2}{label={left:{$a_{2}$}}}  at (-1,-1) {};
\NodeB{A3}{label={left:{$a_{3}$}}}  at (-1,-2) {};
\NodeB{A4}{label={left:{$a_{4}$}}}  at (-1,-3) {};
\NodeB{A5}{label={left:{$a_{5}$}}}  at (-1,-4) {};
\NodeB{A6}{label={left:{$a_{6}$}}}  at (-1,-5) {};
\NodeB{A7}{label={left:{$a_{7}$}}}  at (-1,-6) {};

\NodeB{B1}{label={right:{$b_{1}$}}}  at (1,-1) {};
\NodeB{B2}{label={right:{$b_{2}$}}}  at (1,-2) {};
\NodeB{B3}{label={right:{$b_{3}$}}}  at (1,-3) {};
\NodeB{B4}{label={right:{$b_{4}$}}}  at (1,-4) {};

\draw[->] (A1) -- (B1);

\draw[->] (A1) -- (B2);

\draw[->,ultra thick] (A1) -- (B3);

\draw[->] (A2) -- (B1);
\draw[->] (A2) -- (B2);

\draw[->,ultra thick] (A3) -- (B1);

\draw[->] (A3) -- (B3);

\draw[->] (A3) -- (B4);

\draw[->] (A5) -- (B1);

\draw[->] (A5) -- (B2);

\draw[->,ultra thick] (A5) -- (B4);

\draw[->] (A6) -- (B4);
\draw[->] (A7) -- (B4);

\end{tikzpicture}
\end{centering}%
\vspace*{-1\baselineskip} %make the picture more tightly cropped
\end{wrapfigure}
~ %The tilde creates a new dummy paragraph. WHY IS THAT NEEDED? -> would increase the space %
  % before the ex. environment. THE NEXT FREE LINE IS ESSENTIAL!

\begin{example}
\PRFR{Jan 22nd}
 	Let $A,B$ be finite disjoint sets, and let $\Delta \subseteq A\times B$.
 	Then $D=(A\disunion B, \Delta)$ is a directed bipartite graph.  Let
 	$R\colon X\routesto Y$ be a linking in $D$ with $X\subseteq A$ and $Y\subseteq B$.
    Then the set $M = \SET{\left|
 	p\right| \vphantom{A^A}~\middle|~ p\in R}$ is a matching in $D$. Conversely, if $M'$ is a
 	matching in $D$, we can construct an induced linking in $D$ from $M'$:
 	$R' = \SET{ ab \mid a\in A,\,b\in B,\,\SET{a,b}\in M}$. Furthermore, let
 	$\Acal_{\Delta} = (A_i)_{i\in B}$ be the family of subsets of $A$ where
 	$A_{b} = \SET{a\in A\mid (a,b)\in \Delta}$ for all $b\in B$. The linking $R$ then induces
 	a partial transversal $P=\SET{p_1 \mid p\in R}$ of $\Acal_{\Delta}$.
 	Conversely, if $P'$ is a partial transversal of $\Acal_{\Delta}$, then
 	there is an injective map $\iota\colon P\maparrow B$ such that for
 	all $p\in P$, $p \in A_{\iota(p)}$. Thus $R'' = \SET{ ab \mid a\in P,\,b=\iota(a)}$
 	is a linking in $D$.
\end{example}

\noindent The following connection between the linkings in directed graphs and the transversals of a set system, which define linkings in a bipartite graph that can be deduced from the digraph, has first been pointed out by A.W.~Ingleton and M.J.~Piff in \cite{IP73}. But first, we need to clarify how to deduce the correct family of sets given a digraph and a set of targets.

\needspace{7\baselineskip}
\begin{definition}\label{def:linkageSystem}\PRFR{Jan 22nd}
    Let $D=(V,A)$ be a digraph, and let $T\subseteq V$ be a set of vertices.
    The \deftext[linkage system of D to T@linkage system of $D$ to $T$]{linkage system of $\bm D$ to $\bm T$} -- denoted by
    $\Acal_{D,T}$ -- is defined to be the family\label{n:ADT}
    \[ \Acal_{D,T} = \left(A^{(D,T)}_i\right)_{i\in V\BS T} \subseteq V\]
    where for $v\in V\BS T$
    \[ A^{(D,T)}_{v} = \SET{w\in V\mid (v,w)\in A} \cup \SET{v}. \qedhere\]
\end{definition}

\begin{lemma}\label{lem:ADTtransversals}\PRFR{Jan 22nd}
        Let $D=(V,A)$ be a digraph, $T \subseteq V$.
        Every maximal partial transversal of $\Acal_{D,T}$ is a transversal of $\Acal_{D,T}$.
\end{lemma}

\begin{proof}\PRFR{Jan 22nd}
    Clearly, $V\BS T$ is a transversal of $\Acal_{D,T}$, and therefore %also a partial transversal of $\Acal_{D,T}$, so $V\BS T$ is independent in 
    $\rk_{M(\Acal_{D,T})} = \left| V\BS T \right|$.
    Let $P$ be a maximal partial transversal of $\Acal_{D,T}$,
    then $P$ is a base of $M(\Acal_{D,T})$ and thus $\left| P \right| = \left| V\BS T \right|$ due to the equicardinality of bases
     {\em(B2)}. Therefore
    % $\left| P \right| = \left| T \right|$ and since $V\BS T$ is finite, 
    every injective map $\iota\colon P\maparrow V\BS T$ with $p\in A_{\iota(p)}^{(D,T)}$ for all $p\in P$ is a bijection that witnesses that $P$ is a transversal of $\Acal_{D,T}$.
\end{proof}

\noindent The following lemma has been named {\em The Fundamental Lemma} by A.W.~Ingleton and M.J.~Piff \cite{IP73}, who used it as the key to proving that strict gammoids are precisely the duals of transversal matroids. We are going to use it in order to show augmentation properties of routings in digraphs, too.

\begin{lemma}\label{lem:linkage}\PRFR{Jan 22nd}
    Let $D=(V,A)$ be a digraph, $S,T \subseteq V$.
    Then there is a linking from $S$ to $T$ in $D$, if and only if
    $V\BS S$ is a transversal of the linkage system $\Acal_{D,T}$.
\end{lemma}

\noindent The proof presented here can be found on p.217 \cite{We76}, where the lemma is called
 {\em The Linkage Lemma}.

\begin{proof}\PRFR{Jan 22nd}
    Assume that $R\colon S\routesto T$ is a linking in $D$. We construct the bijective map
    $\sigma\colon V\BS S \maparrow V\BS T$ such that for $v\in V\BS S$, the image
    \[ \sigma(v) = \begin{cases} u & \quad \textit{(a)}
    \,\, \text{if }\exists p\in R\colon\,(u,v)\in\left| p\right|_A, \text{ and} \\
       v & \quad \textit{(b)} \,\, \text{otherwise.}
    \end{cases}\]
    The map $\sigma$ is well-defined because $R$ consists of pair-wise vertex
    disjoint paths in $D$; and whenever $v\in T$, then either $v\in S$
    in which case $v$ is not part of the domain of $\sigma$, or there is a
    non-trivial path $p\in R$ that ends in $v$. Then $\sigma(v) \notin T$
    since otherwise $R$ could not be onto $T$ as every path has precisely one
    end vertex. From the definition of $\Acal_{D,T}$ and the construction of
    $\sigma$ it is clear, that for every $v\in V\BS S$, $v \in
    A^{(D,T)}_{\sigma(v)}$. Assume that $\sigma$ is not injective, thus there
    are $v,w\in V\BS S$ with $v\not=w$, yet $\sigma(v) = \sigma(w)$. This is
    not possible if $v$ and $w$ are in the same case of $\sigma$. Thus without
    loss of generality we may assume that $\sigma$ maps $v$ through case {\em
    (a)} and $w$ through case {\em (b)}. Thus $\sigma(v) = \sigma(w) = w$, and
    $(w,v)\in \left| p\right|_{A}$ for some $p\in R$. Since for $w$ case {\em
    (b)} holds, we can infer that $w = p_{1}$ is the initial vertex of a path in $R$.
    But then $w\in S$ which is not part of the domain of $\sigma$. Therefore
    no such $v,w\in V$ exist and $\sigma$ is an injective map. Since $R$ is a
    linking, $\left| S \right| = \left| T \right|$ and $\left| V\BS S \right|
    = \left| V \BS T\right| < \infty$, thus $\sigma$ is a bijection and $V\BS
    S$ is indeed a transversal of $\Acal_{D,T}$.

\PRFR{Jan 22nd}
    \noindent 
    Conversely, assume that $V\BS S$ is a transversal of $\Acal_{D,T}$. Thus
    there is a bijection $\sigma\colon V\BS S \maparrow V\BS T$ such that for
    all $v\in V\BS S$, $v\in A^{(D,T)}_{\sigma(v)}$. We can construct a
    linking  $R\colon X\routesto Y$ from $\sigma$ in the following way:  for
    $v\in S\cap T$, we can let the trivial path $v \in R$.  For $v\in T\BS S$,
    there is some $k\in \N$ such that  $\sigma^{k}(v) \notin V\BS S$:
    assume that for every $k\in \N$, $\sigma^{k}(v)\in V\BS S$, then
    $\SET{\sigma^{k}(v)\vphantom{A^A}~\middle|~ k\in \N} \subseteq V\BS S$, yet $V\BS S$ is
    finite. Thus there must be some $k_0,k_1\in \N$ with $k_0 < k_{1}$ and
    $\sigma^{k_0}(v) = \sigma^{k_1}(v)$. Now let $k_{0},k_{1}\in \N$ be
    integers with $k_{0} < k_{1}$ and   $\sigma^{k_0}(v) = \sigma^{k_1}(v)$
    such that $k_{0}$ is smallest possible.    Clearly $v\notin V\BS T$, so
    $k_0 > 0$. But then $\sigma$ is a bijection, therefore   the pre-images of
    $\sigma^{k_0}(v)$ and $\sigma^{k_1}(v)$ coincide. Now we have
    $\sigma^{k_0-1}(v) = \sigma^{k_1-1}(v)$ which contradicts the minimality
    of $k_{0}$.   Thus the trajectory of $v$ under repetitions of $\sigma$ has
    no cycle and   therefore must be finite.   Let $k\in \N$ such that
    $\sigma^{k}(v) \notin V\BS S$. The range of $\sigma$ yields that $\sigma^{k}(v) \in V\BS
    T$ and therefore $\sigma^{k}(v)\in S\BS T$.   The construction of
    $\Acal_{D,T}$ guarantees that for every $i\in\SET{0,1,\ldots,k-1}$   there
    is an arc $(\sigma^i(v),\sigma^{i+1}(v))\in A$. Since $\sigma$-trajectories have no cycles, we can add the path
    $\sigma^{k}(v)\sigma^{k-1}(v)\ldots\sigma(v)v \in R$. All paths obtained
    from the above constructions are pair-wise vertex disjoint, % and simple
    because $\sigma$ is a bijection of finite sets, and so $R$ is indeed a
    linking from $S$ to $T$ in $D$.
\end{proof}

\needspace{2\baselineskip}
\noindent We can extend Lemma~\ref{lem:linkage} to routings in the natural way.

\begin{lemma}\label{lem:routage}\PRFR{Jan 22nd}
    Let $D=(V,A)$ be a digraph, $S,T \subseteq V$.
    Then there is a routing from $S$ to $T$ in $D$, if and only if
    there is some $T'\subseteq T$ such that
    $V\BS (S\cup T')$ is a transversal of the linkage system $\Acal_{D,T}$.
\end{lemma}
 
\begin{proof}\PRFR{Jan 22nd}
Every routing $R\colon S\routesto T$ in $D$ consists of a linking from
$S$ to $T_R = \SET{p_{-1}\mid p\in R}$ and a set of unused targets
$T'=T\BS T_{R}$, and thus for every $t'\in T'$, we may add the trivial path $t'$
to $R$ and obtain the linking $R'\colon S\cup T' \routesto T$ 
where $R' = R\cup \SET{t'\in \Pbf(D)\mid t'\in T'}$. Therefore, $R$ induces the 
transversal $V\BS \left( S \cup T' \right)$ of $\Acal_{D,T}$ by Lemma~\ref{lem:linkage}.
Conversely, let $T'\subseteq T$ such that $V\BS(S\cup T')$ is a transversal of $\Acal_{D,T}$. By Lemma~\ref{lem:linkage} there is a linking $R\colon S\cup T' \routesto T$ in $D$. Then
$R' = \SET{p\in R \mid p_1 \in S}$ is a routing from $S$ to $T$ in $D$.
\end{proof}

%\clearpage
% -*- root: ../thesis.tex -*-

\subsection{Menger's Theorem}

\PRFR{Jan 22nd}
F.~Göring published an intriguingly short and beautiful proof of Menger's Theorem \cite{Go00}. 
In this section,
we present a slightly more verbose variant of this proof, which is transformed 
into the context
of this work, along with two required yet straightforward definitions.

% REMARK: Sep 14 2017
%
%{Changed: $A$-$B$-separator to $S$-$T$-separator due to clash with $D=(V,A)$. Check!}
%
% We don't need to index separator, since S-T-separator and separator will go adjacent.
%

\begin{definition}\PRFR{Jan 22nd}
    Let $D=(V,A)$ be a digraph, $S,T\subseteq V$. 
    A set $X\subseteq V$ is called \deftext[S-T-separator@$S$-$T$-separator]{$\bm{S}$-$\bm{T}$-separator} in $D$, if for
    every $p\in\Pbf(D)$ with $p_{1}\in S$ and $p_{-1}\in T$, $\left|p\right|\cap X\not= \emptyset$.
\end{definition}

\PRFR{Jan 22nd}
\noindent It is easy to see that straightening out cycle paths from walks (Remark~\ref{rem:straighteningWalks}) yields paths using a subset of the original vertices,
thus if $X$ is an $S$-$T$-separator, then for all $w\in \Wbf$ with $w_1 \in S$ and $w_{-1}\in T$ we also have $\left| w \right|\cap X\not=\emptyset$.

\begin{example}\label{ex:outermarginseparates}\PRFR{Jan 22nd}
    Let $D=(V,A)$ be a digraph, $S\subseteq V$. Then $\partial S$ is a minimal
    $S$-$(V\BS S)$-separator in $D$: Since $S\cap (V\BS S) = \emptyset$, 
    any walk from $s\in S$ to $t\in V\BS S$ must use an
    arc that starts in $S$ but ends outside of $S$, and therefore it must visit an element of the outer margin $\partial S$. 
    Now let $v\in \partial S$, then there is some $u\in S$ such that $(u,v)\in A$. So $uv\in \Pbf(D)$ 
    is a path from $S$ to $V\BS S$, yet $\partial S \cap \left| uv \right| = \SET{v}$, therefore
     $\partial S\BSET{v}$ is not an $S$-$(V\BS S)$-separator; thus $\partial S$ is a minimal $S$-$(V\BS S)$-separator in $D$.
\end{example}

\PRFR{Jan 22nd}
\noindent Clearly, both $S$ and $T$ are $S$-$T$-separators in every digraph $D$. Furthermore,
every $S$-$T$-separator in $D$ is an $S'$-$T'$-separator for every $S'\subseteq S$ and $T'\subseteq T$.

\begin{definition}\PRFR{Jan 22nd}
    Let $D=(V,A)$ be a digraph, $S,T\subseteq V$. A routing $Y\routesto T$
    in $D$ is called \index{connector}\deftext[S-T-connector@$S$-$T$-connector]{$\bm{S}$-$\bm{T}$-connector} in $D$,
    whenever $Y\subseteq S$.
\end{definition}

\begin{theorem}[Menger's Theorem \cite{Me27,Go00}]\label{thm:MengerGoering}\index{Menger's Theorem}\PRFR{Jan 22nd}
Let $D=(V,A)$ be a digraph, $S,T\subseteq V$ subsets of vertices of $D$,
and $k \in \N$ the minimal cardinality of an $S$-$T$-separator in $D$.
There is an $S$-$T$-connector $R\colon Y\routesto T$ that consists of
$k$ paths.
\end{theorem}

\needspace{7\baselineskip}

\vspace*{-\baselineskip} %Remove the line space created by the tilde below
\begin{wrapfigure}{r}{7.5cm}
\vspace{1.25\baselineskip}
\begin{centering}~~%move the picture slightly to the right
\includegraphics{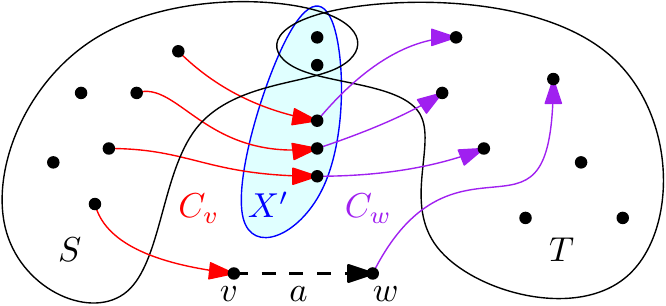}
\end{centering}%
\vspace*{-.5\baselineskip} %make the picture more tightly cropped
\end{wrapfigure}
~ %The tilde creates a new dummy paragraph. WHY IS THAT NEEDED? -> would increase the space %
  % before the ex. environment. THE NEXT FREE LINE IS ESSENTIAL!

\begin{proof}\PRFR{Jan 22nd}
By induction on $\left|A\right|$. If $A=\emptyset$, then there are only trivial paths in
$\Pbf(D)$. Thus $S\cap T$ is a minimal $S$-$T$-separator. Clearly,
$\SET{v \in \Pbf(D) ~\middle|~ v\in S\cap T}$ is a routing from $S\cap T$ to $T$ in $D$.

\PRFR{Jan 22nd}
\noindent
For the induction step, let $(v,w)=a\in A$. The theorem holds for
$D'=(V,A\BSET{a})$ by induction hypothesis, if $D'$ has no $S$-$T$-separator
$X'$ with $\left|X'\right|<k$, the claim follows directly from the induction hypothesis. Now assume
that there is an $S$-$T$-separator $X'$ with $\left|X'\right|<k$ in $D'$. Then
$X'\cup\SET{v}$ as well as $X'\cup\SET{w}$ are $S$-$T$-separators in $D$, therefore
$k\leq \left|X'\right|+1$. Furthermore, every $\left( X'\cup\SET{w} \right)$-$T$-separator in $D'$ and every
$S$-$\left( X'\cup\SET{v} \right)$-separator in $D'$ is an $S$-$T$-separator in $D$. 
By induction
hypothesis there is an $\left( X'\cup\SET{w} \right)$-$T$-connector $C_{w}\subseteq \Pbf(D')\subseteq \Pbf(D)$
with $\left|C_{w}\right|=k$ and $p\in C_{w}$ where $p_{1}=w$ and there is an $S$-$\left( X'\cup\SET{v} \right)$-connector $C_{v}\subseteq \Pbf(D')\subseteq\Pbf(D)$ with $\left|C_{v}\right|=k$ and $q\in C_{v}$ where $q_{-1}=v$.
Then \[ R = \SET{a.b \mid a\in C_{v}, b\in C_{w},\,a_{-1}=b_{1} } \cup \SET{qp}\]
is an $S$-$T$-connector in $D$ with $\left|R\right|=k$: For any two $r\in C_{v}$ and $s\in C_{w}$, we have $\left|r\right|\cap\left|s\right|\subseteq X'$, because otherwise there would be a walk from $S$ to $T$ in $D'$ that does not hit the $S$-$T$-separator $X'$ -- a contradiction. Therefore, for any two walks $x,y\in R$ with $x\not= y$, we obtain $\left|x\right|\cap \left|y\right|=\emptyset$. The walks in $R$ are paths because the concatenation $p.q$ of two compatible paths $p,q$ is a non-path walk if and only if $\left| p \right| \cap \left| q \right| \supsetneq \SET{q_{1}}$  and thus $R$ is indeed a routing.
\end{proof}

\PRFR{Jan 22nd}
\noindent It is immediate from the respective definitions that every
$S$-$T$-separator in $D$ must hit every path of  every $S$-$T$-connector in
$D$ at least once,  therefore Menger's Theorem is the non-trivial part of
the following strong duality\footnote{Strong duality is a notion from linear
programming, stating that if the primal and the dual linear optimization
problems are both feasible, then their optimal values are attained and equal.}
theorem.

\begin{corollary}\label{cor:MengerA}\PRFR{Jan 22nd}
    Let $D=(V,A)$ be a digraph, $S,T\subseteq V$. The maximal cardinality of an $S$-$T$-connector in $D$ equals the minimal cardinality of an $S$-$T$-separator in $D$.
\end{corollary}

\noindent Another immediate consequence is that every vertex of a minimal separator is hit by every maximal connector.

\begin{corollary}\label{cor:Menger}\PRFR{Jan 22nd}
    Let $D=(V,A)$ be a digraph, $S,T,X\subseteq V$, such that $X$ is an $S$-$T$-separator of minimal cardinality. Every $S$-$T$-connector $R$ with maximal cardinality in $D$ has the property
    \[ \forall x\in X\colon\,\exists p\in R\colon\,x\in \left| p\right| .\]
\end{corollary}

\subsection{Augmentation of $S$-$T$-Connectors}

\PRFR{Jan 22nd}
Menger's Theorem states that if $R\colon X\routesto Y$ is an $S$-$T$-connector in $D$ with $\left| R \right| < \left| C \right|$ for every $S$-$T$-separator $C$, then there must be some
bigger $S$-$T$-connector in $D$.
In this section, we  prove that Menger's Theorem is still true if 
we consider only those $S$-$T$-connectors 
$R'\colon X'\routesto Y'$ with $X\subsetneq X'$.

\begin{theorem}\label{thm:augmentationCons}\PRFR{Jan 22nd}
    Let $D=(V,A)$ be a digraph, $S,T\subseteq V$, and let $A,B$ be two $S$-$T$-connectors
    in $D$ with $\left| A \right| < \left| B \right|$.
    Then there is an $S$-$T$-connector $C$ in $D$, such that $\left| C \right| = \left| A \right| + 1$ and $\SET{p_1\mid p\in A} \subseteq \SET{p_1\mid p\in C} \subseteq \SET{p_1 \mid p\in A\cup B}$.
\end{theorem}

\noindent This proof is based on the argumentation found on p.~220 in \cite{We76}.

\begin{proof}
    \PRFR{Jan 22nd}
    Let $A_{1} = \SET{p_1\mid p\in A}$ and $B_{1} = \SET{p_1\mid p\in B}$ be
    the initial vertices of paths in $A$ and $B$, and let  $A_{-1} = \SET{p_{-1}\mid
    p\in A}$ and $B_{-1} = \SET{p_{-1}\mid p\in B}$ be the terminal vertices of paths in
    $A$ and $B$. By Lemma~\ref{lem:linkage} we see that $V\BS A_{1}$ is a
    transversal of the linkage system $\Acal_{D,A_{-1}}$ and $V\BS B_{1}$ is a
    transversal of the linkage system $\Acal_{D,B_{-1}}$. Consider the linkage
    system $\Acal' = \Acal_{D,A_{-1}\cup B_{-1}}$, clearly $\Acal'$ is both a
    subfamily of $\Acal_{D,A_{-1}}$ and $\Acal_{D,B_{-1}}$. Therefore, $V\BS
    A_{1}$ and $V \BS B_{1}$ both contain a maximal partial transversal of
    $\Acal'$, thus $V\BS A_{1}$ and $V\BS B_{1}$ each contain a base of
    $M(\Acal')$. By Lemma~\ref{lem:spanningdual}, $A_{1}$ and $B_{1}$ are
    independent sets of the dual matroid $M(\Acal')^\ast$, and by
    Lemma~\ref{lem:augmentation} there is a base $X$ of $A_{1}\cup B_{1}$ in
    $M(\Acal')^{\ast}$, such that $A_{1}\subseteq X \subseteq A_{1}\cup
    B_{1}$. But then $V\BS X$ is a spanning set of $M(\Acal')$, therefore it
    contains a maximal partial transversal $V\BS P$ of $\Acal'$ where
    $X\subseteq P$. By Lemma~\ref{lem:ADTtransversals} we obtain that $V\BS P$
    is also a transversal of $\Acal'$.  Again it follows from
    Lemma~\ref{lem:linkage} that there is a linking $L\colon P\routesto
    (A_{-1}\cup B_{-1})$ in $D$. Now $A_{1} \subseteq X$, furthermore $X\cap
    \left(B_{1}\BS A_1\right)\not=\emptyset$ and $X\subseteq P$, thus $A_1
    \subseteq P$ and  $P\cap (B_{1}\BS A_1)\not= \emptyset$. Therefore there is an element $b\in
    P\cap (B_{1}\BS A_{1})$ which can be used to filter the augmented $S$-$T$-connector from $P$:
    The linking $ C = \SET{p\in L\mid p_1\in A\cup\SET{b}}$
    is the desired augmented $S$-$T$-connector.
\end{proof}

\cleardoublepage
% -*- root: ../thesis.tex -*-

\chapter{Gammoids}
%\stepcounter{section}

\PRFR{Jan 22nd}
J.H.~Mason first introduced the notions of a gammoid and of a strict gammoid.
 Both are matroids that arise from free matroids through matroid
induction \cite{M72}: Given a digraph $D=(V,A)$ and a matroid $N=(E,\Ical)$, the set of
vertices, from which there is a routing onto some $T\subseteq V$  in $D$
with $T\in
\Ical$, forms a family of independent sets of a matroid on the ground set $V$.
The resulting matroid is called the matroid induced by $D$ from $N$. The general case of
matroids induced by $D$ from $N$ is connected to the special case of gammoids, where $N$ is a free matroid,
 through the following generalization:
 The augmentation theorem for $S$-$T$-connectors in $D$ still holds
if we restrict the class of all $S$-$T$-connectors in $D$ to the class of $S$-$T$-connectors in $D$ that
link onto an independent set of a given matroid on $T$ --- a proof may be obtained by replacing the transversal matroid $M(\Acal')$ 
presented by the linkage system $\Acal' = \Acal_{D,A_{-1}\cup B_{-1}}$ in the proof of Theorem~\ref{thm:augmentationCons}
 with a suitable matroid $M(\Delta',N)$ obtained through bipartite matroid induction with respect to the directed bipartite graph $\Delta'$ associated with the linkage system $\Acal'$ through Definition~\ref{def:arcSystemDelta}. But since we are most interested
in a certain special case of matroid induction by directed graphs, we omit this concept
for now and give a direct definition of gammoids instead.

\needspace{10\baselineskip}
\section{Definition and Representations}

\begin{definition}\label{def:gammoid}\PRFR{Jan 22nd}
    Let $D = (V,A)$ be a digraph, $E\subseteq V$,
    and $T\subseteq V$. 
    The \deftext[gammoid represented by DTE@gammoid represented by $(D,T,E)$]{gammoid represented by $\bm{(D,T,E)}$} is defined to be the matroid $\Gamma(D,T,E)=(E,\Ical)$\label{n:GTDE}
     where
    \[ \Ical = \SET{X\subseteq E \mid \text{there is a routing } X\routesto T \text{ in D}}. \]
    The elements of $T$ are usually called \deftextX{sinks} in this context, although they are not required to be actual sinks of the digraph $D$. To avoid confusion, 
    we shall call the elements of $T$ \deftext{targets} in this work. A matroid $M'=(E',\Ical')$ is called \deftextX{gammoid}, if there is a digraph $D'=(V',A')$ and a set $T'\subseteq V'$ such that $M' = \Gamma(D',T',E')$.
\end{definition}

\needspace{4\baselineskip}
\begin{lemma}\PRFR{Jan 22nd}
    Let $D = (V,A)$ be a digraph, $E\subseteq V$,
    and $T\subseteq V$. Then  $\Gamma(D,T,E)$ is a matroid.
\end{lemma}

\begin{proof}\PRFR{Jan 22nd}
	Let $\Gamma(D,T,E) = (E,\Ical)$.
	Clearly, the empty routing $\emptyset \subseteq \Pbf(D)$ routes $\emptyset$ to $T$,
	therefore $\emptyset \in \Ical$, so {\em (I1)} holds.
	Also, if $R\colon X\routesto T$ is a routing from $X$ to $T$ in $D$, and if $Y\subseteq X$, then $\SET{p\in R \mid p_1\in Y}$ is a routing from $Y$ to $T$ in $D$, therefore
	{\em (I2)} holds, too.
	Now let $X,Y\in \Ical$ with $\left| X \right| < \left| Y \right|$. Then there are routings $R\colon X\routesto T$ and $S\colon Y\routesto T$ in $D$. We may
	regard $R$ and $S$ as $(X\cup Y)$-$T$-connectors in $D$. Thus by 
	Theorem~\ref{thm:augmentationCons} there is a routing $C$ from $X'$ to $T$ such that
	$X\subsetneq X' \subseteq X\cup Y$, so there is an element $y\in X'\BS X \subseteq Y$ such that $X\cup\SET{y}\in \Ical$.
	Therefore {\em (I3)} holds and, consequently, $\Gamma(D,T,E)$ is a matroid.
\end{proof}

\begin{lemma} \label{lem:rkEqMaxConnector}\PRFR{Jan 22nd}
	Let $D=(V,A)$, $E\subseteq V$, $T\subseteq V$, $M = \Gamma(D,T,E)$, and $X\subseteq E$.
	Then $\rk_M(X)$ equals the size of a maximal $X$-$T$-connector in $D$.
\end{lemma}
\begin{proof}\PRFR{Jan 22nd}
	Let $(E,\Ical) = M$ and let $C\colon X_0\routesto T$ be a maximal-cardinality $X$-$T$-connector, then clearly for all $x\in X\BS X_0'$, there is no routing $X_0\cup\SET{x}\routesto T$ in $D$, therefore $X_0\cup\SET{x}\notin \Ical$ for all $x\in X\BS X_0$. Thus $\rk_M(X) < \left| C \right| + 1$. We have $X_0 \in \Ical$ since $\Gamma(D,T,E)$ is defined that way. Thus $\left| C \right| = \left| X_0 \right| \leq \rk_M(X)$, which yields $\rk_M(X) = \left| C \right|$.
\end{proof}

% -*- root: ../thesis.tex -*-

\subsection{Switching Between Representations}

\PRFR{Jan 22nd}
From the definition of a gammoid $M$, it is clear that any given representation $(D,T,E)$ of a gammoid
cannot be unique for $M$, because the number of vertices of $D$ is not constrained. Therefore every
gammoid has a myriad of representations, and some of these representations are nicer than others, also depending on the purpose.
In this section, we deal with operations on representations $(D,T,E)$ that leave the
represented gammoid fixed. 

\bigskip
\noindent 
Without loss of generality we may always assume that a gammoid is presented by some $(D,T,E)$ where $T$ consists only of sinks of $D$.

\begin{lemma}\label{lem:TwlogSinks}\PRFR{Jan 22nd}
	Let $D = (V,A)$ be a digraph, $E\subseteq V$, and $T\subseteq V$.
	Furthermore, let $D' = (V,A')$ where $A' = A \BS \left( T\times V \right)$.
	%i.e. $D'$ is obtained from $D$ by removing all arcs starting in some $t\in T$.
	Then $\Gamma(D,T,E) = \Gamma(D',T,E)$.
\end{lemma}
\begin{proof}\PRFR{Jan 22nd}
	Let $M = \Gamma(D,T,E) = (E,\Ical)$ and $M' = \Gamma(D',T,E) = (E,\Ical')$.
	Clearly every routing $R$ in $D'$ is also a routing in $D$, thus $\Ical' \subseteq \Ical$. Now let $X\in \Ical$ and let $R\colon X\routesto T$ be a routing in $D$.
	Then for every $p\in R$ there is a minimal integer $i(p)$ such that $p_{i(p)}\in T$.
	Let $R' = \SET{p_1 p_2 \ldots p_{i(p)} ~\middle|~ p\in R}$. $R'$ is a routing from $X$ to $T$ in $D'$. Thus $X\in \Ical'$ and therefore $\Ical \subseteq \Ical'$, so $M=M'$.
\end{proof}

\noindent Without loss of generality, we may always assume that the cardinality of the target set equals the rank of the gammoid.

\begin{lemma}\label{lem:rankequalsTcard}\PRFR{Jan 22nd}
	Let $M=(E,\Ical)$ be a gammoid. Then there is a digraph $D=(V,A)$ and a subset $T\subseteq V$, such that $\left| T \right|=\rk_M(E)$ and $M = \Gamma(D,T,E)$.
\end{lemma}

\needspace{8\baselineskip}
\vspace*{-\baselineskip} %Remove the line space created by the tilde below
\begin{wrapfigure}{r}{9.2cm}
\vspace{\baselineskip}
\begin{centering}~~%move the picture slightly to the right
\includegraphics[width=9cm]{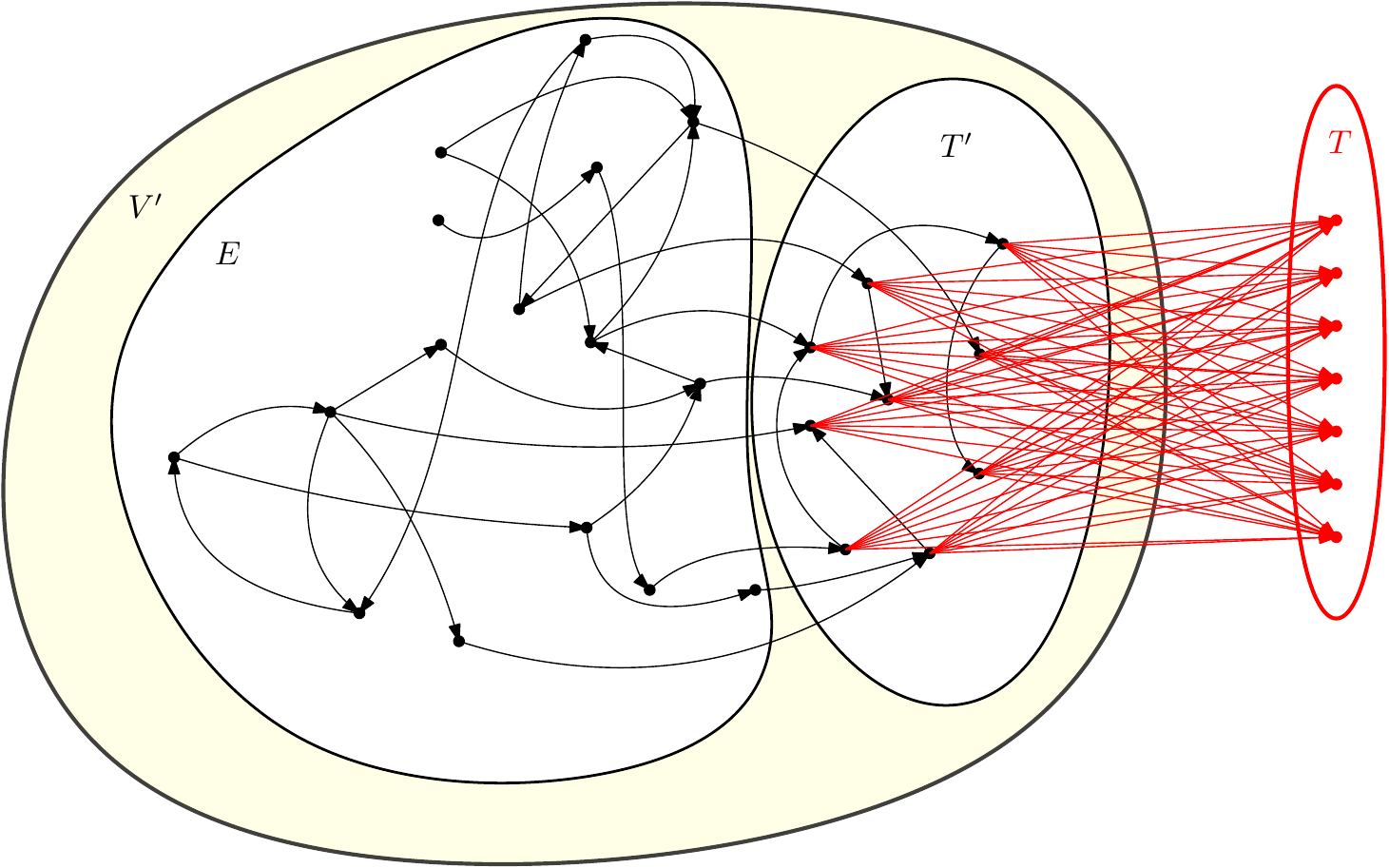}
\end{centering}%
\vspace*{-2\baselineskip} %make the picture more tightly cropped
\end{wrapfigure}
~ %The tilde creates a new dummy paragraph. WHY IS THAT NEEDED? -> would increase the space %
  % before the ex. environment. THE NEXT FREE LINE IS ESSENTIAL!

\begin{proof}\PRFR{Jan 22nd}
	Let $(D',T',E)$ be a representation of $M$ where \linebreak $D'=(V',A')$. % be such that $M = \Gamma(D',T',E)$.
	There is an easy construction that achieves the claim: 
	Remember that $X\subseteq E$ is independent in $\Gamma(D',T',E)$ 
	if and only if there is a routing \linebreak $X\routesto T'$ in $D'$. 
	Since $\left| X \right| \leq \rk(E)$, 
	at most $\rk(E)$ vertices of $T'$ are visited by paths that belong to $X\routesto T'$. 
	Thus we may extend the digraph $D'$ to a digraph $D=(V,A)$ 
	by adding $\rk(E)$ new vertices $T = \dSET{t_{1},\ldots,t_{\rk(E)}}$
	in such a way, that there is an arc $(v,t)\in A$ for $v\in V'$ and $t \in T$ 
	if and only if $v\in T'$. Formally, we let $V=V'\disunion T$ and $A=A'\cup (T'\times T)$.
	By construction,
	every routing $X\routesto T'$ in $D'$ can be extended to a routing $X\routesto T$ in $D$,
	as there are sufficient elements in $T$ and arcs between $T'$ and $T$ in D.
	On the other hand, a routing $X\routesto T$ in $D$ implies that there is a routing $X\routesto T'$
	because every non-trivial path ending in $T$ must visit some $t'\in T'$.
	Therefore, $X\subseteq E$ is independent in $\Gamma(D',T',E)$
	if and only if $X$ is independent in $\Gamma(D,T,E)$. 
	Thus $M = \Gamma(D,T,E)$, so $(D,T,E)$ represents $M$ with $\left| T \right| = \rk(E)$.
\end{proof}

\noindent We obtain the following from the previous proof:
\begin{corollary}\label{cor:wlogNiceDigraph}\PRFR{Jan 22nd}
	Let $M=(E,\Ical)$ be a gammoid. Then there is a digraph $D=(V,A)$ and a subset $T\subseteq V$ with $T\cap E =\emptyset$,
	 such that $\left| T \right|=\rk_M(E)$ and $M = \Gamma(D,T,E)$. Furthermore, every $t\in T$ is a sink in $D$.
\end{corollary}

\needspace{4\baselineskip}
\begin{definition}\label{def:rspivot}\PRFR{Jan 22nd}
	Let $D=(V,A)$ be a digraph, $s\in V$ be a vertex of $D$, and $r \in V$ be a vertex such that
	$(r,s)\in A$ is an arc of $D$. The \deftext[pivot of a digraph]{$\bm r$-$\bm s$-pivot of $\bm D$} 
	shall be the digraph $D_{r\leftarrow s} = (V,A_{r\leftarrow s})$\label{n:digraphpivot} where the arc set
	\[ A_{r\leftarrow s} = \left( A \BS \left( \SET{r}\times V \right) \right) \,\, \cup\,\, \left(  \left( \SET{s}\times \DclD{\SET{r}}{D} \right) \BSET{(s,s)} \right)\]
	consists of arcs leaving $s$ and entering $x$ for every 
	$x\in \DclD{\SET{r}}{D}\BSET{s}$, i.e. for every $x\not=s$ with either $x=r$ or such that there is an arc from $r$ to $x$ in $D$,
	and	all arcs $(u,v)\in A$ of $D$ which have a tail $u\not= r$.
\end{definition}

% -*- root: ../../thesis.tex -*-

\vspace*{-\baselineskip} %Remove the line space created by the tilde below
\begin{wrapfigure}{r}{6cm}
\vspace{1.6\baselineskip}
\begin{centering}%move the picture slightly to the right
\includegraphics{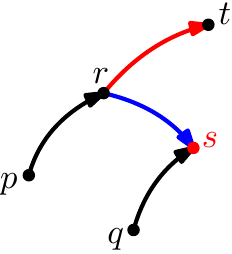} $\begin{array}{c}\leadsto \\~ \\~\\~\\ \end{array}$ \includegraphics{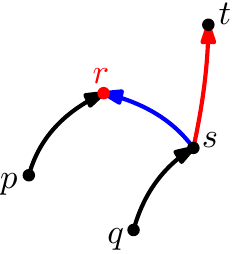}
\end{centering}%
\vspace*{-3\baselineskip} %make the picture more tightly cropped
\end{wrapfigure}
~ %The tilde creates a new dummy paragraph. WHY IS THAT NEEDED? -> would increase the space %
  % before the ex. environment. THE NEXT FREE LINE IS ESSENTIAL!

\begin{example} \PRFR{Feb 15th}
	Consider the digraph $D=(V,A)$ where $V=\dSET{p,q,r,s,t}$ and 
	$A = \{(p,r),$ $(q,s),$ $(r,s),$ $(r,t)\}$.
	Clearly, $s$ is a sink in $D$ and $(r,s)\in A$, and thus the $r$-$s$-pivot of $D$
	is 
	$D_{r\leftarrow s}=(V,A_{r\leftarrow s})$ with
	$A_{r\leftarrow s} = \{(p,r),$ $(q,s),$ $(s,r)$, $(s,t)\}$.
	Let us examine the paths in $D$ and $D_{r\leftarrow s}$:
	%
	%\noindent
	$\Pbf(D) = \{p,$ $pr,$ $prs,$ $prt,$ $q,$ $qs,$ $r,$ $rs,$ $rt,$ $s,$ $t\}$,
	whereas $\Pbf(D_{r\leftarrow s}) = \{p,$ $pr,$ $q,$ $qs,$ $qsr,$ $qst,$ $r,$ $s,$ $sr,$ $st,$ $t\}$.
	The maximal routings in $D$ with respect to set-inclusion,
	which are also maximal routings in $D_{r\leftarrow s}$ are
	$\SET{p,q,r,s,t},$ %{(('p',), ('q',), ('r',), ('s',), ('t',)),
%	$\SET{p,q,r,s,t},$ %%{(('p',), ('q',), ('r',), ('s',), ('t',)),
    $\SET{p,qs,r,t},$ and %(('p',), ('q', 's'), ('r',), ('t',)),
% 	$\SET{p,qs,r,t},$ %%(('p',), ('q', 's'), ('r',), ('t',)),
    $\SET{pr,q,s,t};$ %(('p', 'r'), ('q',), ('s',), ('t',)),
%	$\SET{pr,q,s,t}$. %%(('p', 'r'), ('q',), ('s',), ('t',))}	
	the maximal routings in $D$ which are not in $D_{r\leftarrow s}$ are
 	$\SET{p,q,rs,t},$ %(('p',), ('q',), ('r', 's'), ('t',)),
    $\SET{p,q,rt,s},$ %(('p',), ('q',), ('r', 't'), ('s',)),
    $\SET{prs,q,t},$ and %(('p', 'r', 's'), ('q',), ('t',)),
    $\SET{prt,q,s}$; %(('p', 'r', 't'), ('q',), ('s',))}
	and those only in $D_{r\leftarrow s}$ are
 	$\SET{p,q,r,st},$ %%(('p',), ('q',), ('r',), ('s', 't')),
 	$\SET{p,q,sr,t},$ %%(('p',), ('q',), ('s', 'r'), ('t',)),
 	$\SET{p,qsr,t},$ and %%(('p',), ('q', 's', 'r'), ('t',)),
 	$\SET{p,qst,r}.$ % and %%(('p',), ('q', 's', 't'), ('r',)),
 \end{example}

\noindent The next lemma is called {\em the fundamental theorem} by J.H.~Mason in \cite{M72}.
\begin{lemma}\label{lem:MasonsFundamental}\PRFR{Jan 22nd}
	Let $D=(V,A)$ be a digraph, $T\subseteq V$, $s\in T$ a sink of $D$, $r\in V\BS T$ with $(r,s)\in A$, and $X\subseteq V$.
	Then there is a routing $X\routesto T$ in $D$ if and only if there is a routing
	$X\routesto \left( T\BSET{s} \right)\cup\SET{r}$ in $D_{r\leftarrow s}$.
\end{lemma}

\begin{proof}\PRFR{Jan 22nd}
	First, we prove that a routing $X\routesto T$ in $D$ implies a routing
	$X\routesto \left( T\BSET{s} \right)\cup\SET{r}$ in $D_{r\leftarrow s}$.
	Let $R\colon X\routesto T$ be a routing in $D$.
	If $(r,s) \in \bigcup_{p\in R} \left| p \right|_A$, i.e. the routing $R$ has a path $p=(p_i)_{i=1}^n$ that
	traverses the arc $(r,s)$; then $n>1$ and $\SET{r,s}\cap \left| p' \right| = \emptyset$ for
	all $p'\in R\BSET{p}$. Let $q = p_1p_2\ldots p_{n-1}$ be the path that arises when the vertex $s$ is chopped off of $p$.
	 Then $R' = R\BSET{p}\cup\SET{q}$ is a routing from $X$ to $(T\BSET{s})\cup \SET{r}$  in $D_{r\leftarrow s}$.
	Otherwise, we have that $(r,s) \notin \bigcup_{p\in R} \left| p \right|_A$. Let $Q=\SET{r,s}\cap \left( \bigcup_{p\in R} \left| p \right| \right)$ be 
	the criterion for a case analysis. If $Q=\emptyset$, then $R$ is obviously a routing in $D_{r\leftarrow s}$, 
	because $D$ and $D_{r\leftarrow s}$ coincide on $V\BSET{r,s}$. 
	Then no path $p\in R$ has $p_{-1}=s$, thus $R$ even is a routing from $X$ 
	to $T\BSET{s}$ in $D_{r\leftarrow s}$. If $Q=\SET{s}$, then there is a path $p\in R$ with $p_{-1}=s$, yet no path of $R$ visits $r$, 
	therefore $R\BSET{p}\cup\SET{pr}$ is the desired routing in $D_{r\leftarrow s}$.
	If $Q=\SET{r}$, then no path in $R$ visits $s$, and there is a path $p=(p_i)_{i=1}^{n}$ that visits $r = p_j$ with $j\in\SET{1,2,\ldots,n}$. Then $R\BSET{p}\cup\SET{p_1p_2\ldots p_j}$
	is the desired routing in $D_{r\leftarrow s}$.
	If $Q=\SET{r,s}$, then there are two paths $p,q\in R$ with $p\not= q$ such that $s\in \left| p \right|$ and $r\in \left| q \right|$.
	Let $q=(q_i)_{i=1}^m$, and let $1\leq j \leq m$ such that $q_j = r$. 
	Since $s$ is a sink in $D$, we have $p_{-1}=s$. Let $p' = pq_{j+1}q_{j+2}\ldots q_m$ be
	the path in $D_{r\leftarrow s}$ that first follows $p$ and then follows the end of $q$. We have
	$\left| p' \right|_A\subseteq A_{r\leftarrow s}$ since $(r,q_{j+1})\in \left| q \right|_A \subseteq A$
	thus $(s,q_{j+1})\in A_{r\leftarrow s}$, and the digraphs $D$ and $D_{r\leftarrow s}$ have the same arcs on $V\BSET{r,s}$. 
	Furthermore, let $q'= q_1q_2\ldots q_j$, clearly $q'\in \Pbf(D_{r\leftarrow s})$, 
	thus $(R\BSET{p,q})\cup\SET{p',q'}$ is the desired routing in $D_{r\leftarrow s}$.

\noindent
	The second implication of the lemma follows from the first implication together with the fact, that in the situation of the lemma where the operand $s$ is a sink of $D$,
	$\left( D_{r \leftarrow s} \right)_{s\leftarrow r} = D$ holds.
\end{proof}

\begin{theorem}\label{thm:gammoidRepresentationWithBaseTerminals}\PRFR{Jan 22nd}
	Let $M=(E,\Ical)$ be a gammoid, and $B$ a base of $M$.
	Then there is a digraph $D=(V,A)$, 
	such that \[ M = \Gamma(D,B,E)\]
	and every $b\in B$ is a sink in $D$.
\end{theorem}
\begin{proof}\PRFR{Jan 22nd}
	Let $D'=(V,A')$ be a digraph and $T' \subseteq A'$ such that $M = \Gamma(D',T',E)$.
	We may assume that $\left| T' \right| = \left| B \right| = \rk_M(E)$ and that all $t'\in T'$ are
	sinks in $D'$ (Corollary~\ref{cor:wlogNiceDigraph}).
	Let $R\colon B \routesto T'$ a linking of $B$ onto $T'$ in $D'$.
	We prove the statement by induction on $\left| \bigcup_{p\in R} \left| p \right| \right|$.
	In the base case we have $\left| \bigcup_{p\in R} \left| p \right| \right| = \left| B \right|$, and
	therefore every path $p\in R$ is trivial. Thus $B = T$ and $D=D'$ is the desired digraph.
	Now let $\left| \bigcup_{p\in R} \left| p \right| \right| > \left| B \right|$,
	thus there is a non-trivial path $p=(p_i)_{i=1}^{n} \in R$ where $n>1$. Let $s = p_{n}$ and let
	$r = p_{n-1}$. The vertex $s$ is a sink in $D'$ since $s\in T'$, and clearly
	 $(r,s)\in \left| p \right|_A\subseteq A'$. Since \linebreak
	  $\left| B \right| = \left| T' \right|$, $r\notin T'$.
	   The proof of Lemma~\ref{lem:MasonsFundamental} yields
	    that $R' = R\BSET{p}\cup\SET{p_1p_2\ldots p_{n-1}}$ is
	a linking of $B$ onto $\left( T'\BSET{s}  \right)\cup\SET{r}$
	 in $D'_{r\leftarrow s}$ with $\left| \bigcup_{p\in R'} \left| p \right| \right| < \left| \bigcup_{p\in R} \left| p \right| \right|$.
	  Furthermore Lemma~\ref{lem:MasonsFundamental} implies that 
	  $\Gamma(D',T',E) = \Gamma(D'_{r\leftarrow s},\left( T'\BSET{s}  \right)\cup\SET{r}, E)$
	   and the existence of the digraph $D$ follows
	from the induction hypothesis for the linking $R'$ with respect to the representation 
	 $(D'_{r\leftarrow s}, \left( T'\BSET{s}  \right)\cup\SET{r},E)$.
\end{proof}

% -*- root: ../thesis.tex -*-
\needspace{8\baselineskip}

\subsection{Number of Vertices Needed to Represent a Gammoid}

\PRFR{Mar 7th}
\noindent
In the paper {\em Representative Sets and Irrelevant Vertices: New Tools for Kernelization} \cite{KW12}, S.~Kratsch and M.~Wahlström proved 
the following upper bound result regarding the number of vertices in a given digraph, that suffice to be considered in order to find certain $S$-$T$-separators of minimal cardinality.
This bound may be used to derive a bound on the number of vertices needed
in order to represent a gammoid on a ground set of given cardinality.

\needspace{4\baselineskip}

\begin{theorem}[\cite{KW12}, Theorem~3]\label{thm:upperBoundSizeOfV}\PRFR{Mar 7th}
	Let $D=(V,A)$ be a digraph, $E,T\subseteq V$, and $r>0$ be the cardinality of a minimal $E$-$T$-separator in $D$.
	There is a set $Z\subseteq V$ with $E\cup T \subseteq Z$ and
	 $\left| Z \right| = O(\left| E \right|\cdot\left| T \right|\cdot r)$
	such that for all $X\subseteq E$ and $Y\subseteq T$ there is a minimal $X$-$Y$-separator $S$ in $D$ with $S\subseteq Z$.
	The set $Z$ can be found in randomized polynomial time with failure probability $O(2^{-n})$.
\end{theorem}

\PRFR{Mar 7th}
\noindent
For the proof, see \cite{KW12}.\footnote{The actual proof starts on page 24.} 
The statement that $E\cup T \subseteq Z$ is not part of the original theorem in \cite{KW12}, as well as the condition $r>0$,
but it is easy to see that these modifications are valid, since $\left| E\cup T \right| = O(\left| E \right|\cdot\left| T \right|\cdot r)$ for $r\geq 1$.

\begin{remark}\label{rem:boundForZ}\PRFR{Mar 7th}
	 In \cite{KW12}, the authors only give the $O$-behavior of the size of $Z$ in Theorem~\ref{thm:upperBoundSizeOfV},
	 but it is possible to derive the factor hidden in the $O$-notation by inspecting their proof and the proof of
	 Lemma~4.1 \cite{Marx09} by D.~Marx.
	We obtain
	% In the proof, an auxiliary gammoid $M$ that is
	% the direct sum of a rank $r$ gammoid, a rank $\left| E \right|$ gammoid, and a rank $\left| T \right|$ gammoid
	% is constructed, along with a family of independent sets $\Tcal$ that consists of sets that contain exactly one element of each factor,
	% and there is a one-to-one correspondence between the elements $T\in\Tcal$ and the vertices $v\in V$, we denote the set in $\Tcal$
	% that corresponds to $v$ by $T(v)$.
	% For any $X,Y,R\subseteq V$, the question, 
	% whether $X$ may be linked onto $Y$ in $D$ without visiting any vertex from a set $R\subseteq V$, 
	% can be decided by testing whether a derived set $Q$ is independent in the auxiliary gammoid $M$.
	% Given $Q$ is independent in $M$, and let $v\in V$ such that $Q\cap T(v) = \emptyset$.
	% Then $Q\cup T(v)$ is independent in $M$ if there is a linking from $X$ onto $Y$ in $D$ that avoids all vertices 
	% from $R$ as well as $v$. If $Q\cup T(v)$ is dependent, then $v$ is part of a minimal $X$-$Y$-separator in the
	% digraph that arises from $D$ by deleting all vertices and incident arcs from $R$.

	% there is an independent set, a vertex $v\in V$ may be removed without
	% Lemma~4.1 from \cite{Marx09} is applied to each factor of the auxiliary matroid and
	 %yields that $Z\BS\left( E\cup T \right)$ has at most
	 \[ \left| Z \right| \leq \binom{r}{1} \cdot \binom{\left| E \right|}{1}  \cdot \binom{\left| T \right|}{1} + \left| E \right| + \left| T \right| = r\cdot \left| E \right|\cdot \left| T \right| + \left| E \right| + \left| T \right|. \qedhere \]
	 %elements.
\end{remark}

\begin{corollary}\label{cor:upperBoundOnV}\PRFR{Mar 7th}
	Let $M=(E,\Ical)$ be a gammoid.
	There is a representation $(D,T,E)$ of $M$ with $D=(V,A)$, such that
	\[ \left| V \right| = O\left( \left| E \right| \cdot \rk_M(E)^2 \right) \leq O\left(  \left| E \right|^3 \right).\]
\end{corollary}
\begin{proof}\PRFR{Mar 7th}
	Let $(D,T,E)$ be a representation of $M$ where $\left| T \right| = \rk_M(E)$ and $D=(V,A)$ (Lemma~\ref{lem:rankequalsTcard}).
	Let $Z'\subseteq V$ be a subset of $V$ as in the consequent 
	of Theorem~\ref{thm:upperBoundSizeOfV}.
	% Let $Z' = Z\cup E \cup T$, 
	%then $$\left| Z' \right| \leq \left| Z \right| + \left| E \right| + \rk_M(E) = O\left( \left| E \right| \cdot \rk_M(E)^2 \right).$$
	Let $D'=(Z',A')$ be the digraph,
	where for all $x,y\in Z'$, there is an arc
	\[ (x,y)\in A' \quad\Longleftrightarrow\quad \exists p\in \Pbf(D; x,y) \colon\, \left| p \right|\cap Z' = \SET{x,y}.\]
	Thus there is an arc leaving $y\in Z'$ and entering $z\in Z'$ in $D'$ if there is a path from $y$ to $z$ in $D$ that never visits
	another vertex of $Z'$.
	Let $p=(p_i)_{i=1}^n \in \Pbf(D)$ be a path of length $n$ from $p_1\in E$ to $p_{-1}\in T$.
	Let  $I' = \SET{i\in \N ~\middle|~ 1\leq i\leq n \txtand p_i \in Z'} = \dSET{i_1,i_2,\ldots,i_k}$ with $i_1 < i_2 < \ldots < i_k$.
	Then let $p' = p_{i_1}p_{i_2}\ldots p_{i_k}$, i.e. $p'$ is the path consisting of the vertices visited by $p$ that are in $Z'$.
	Observe that
	$p_1 = p'_1$, $p_{-1} = p'_{-1}$, and $p'\in \Pbf(D')$ holds.
	Let $R\colon X\routesto T$ be a routing from $X\subseteq E$ to $T$ in $D$,
	and let
	 $R' = \SET{p'\in \Pbf(D') ~\middle|~ p\in R}$ be the set of paths in $D'$ that consists of all
	$p'$ derived from $p\in R$ as described above. By construction of $D'$, we see that $R'$ is a 
	routing from $X$ to $T$ in $D'$. Thus every independent set $X\subseteq E$
	of $M = \Gamma(D,T,E)$ is also an independent set of $N = \Gamma(D',T,E)$.
	Now assume that there is some $X\subseteq E$ that is independent in $N$, but not in $M$.
	Then $D$ would have an $X$-$T$-separator $S$ with $\left| S \right| < \left| X \right|$, and by Theorem~\ref{thm:upperBoundSizeOfV}
	we may assume that $S\subseteq Z'$ holds. Thus $S$ would be an $X$-$T$-separator of $D'$, too, contradicting the assumption 
	that $X$ is independent with respect to $N$. Therefore every independent set of $N$ is also independent with respect to $M$.
	Consequently, $M = N$. Thus $(D',T,E)$ is a representation of $M$ using 
	only $O\left(  \left| E \right| \cdot \rk_M(E)^2  \right)$ vertices.
\end{proof}

\begin{remark}\label{rem:upperBoundForV}\PRFR{Mar 7th}
	In the light of Remark~\ref{rem:boundForZ} we obtain that if $M=(E,\Ical)$ is a gammoid,
	there is a representation $(D,T,E)$ where $D=(V,A)$ such that $\left| T \right| = \rk_M(E)$
	and such that \[ \left| V \right| \,\,\, \leq \,\,\, \rk_M(E)^2 \cdot \left| E \right| + \rk_M(E) + \left| E \right| \,\,\, \leq \,\,\, 2 \left| E \right|^3. \qedhere\]
\end{remark}

% -*- root: ../thesis.tex -*-

\needspace{6\baselineskip}
\subsection{Duality Respecting Representations}

\begin{definition}\label{def:dualityRespectingRepr}\PRFR{Jan 22nd}
	Let $(D,T,E)$ be a representation of a gammoid. We say that $(D,T,E)$ is a \deftext{duality respecting representation},
	if 
	\[ \Gamma(D^\opp,E\BS T,E) = \left( \Gamma(D,T,E) \right)^\ast . \qedhere \]
\end{definition}

\begin{figure}[h]
\begin{center}
\includegraphics[scale=1.1]{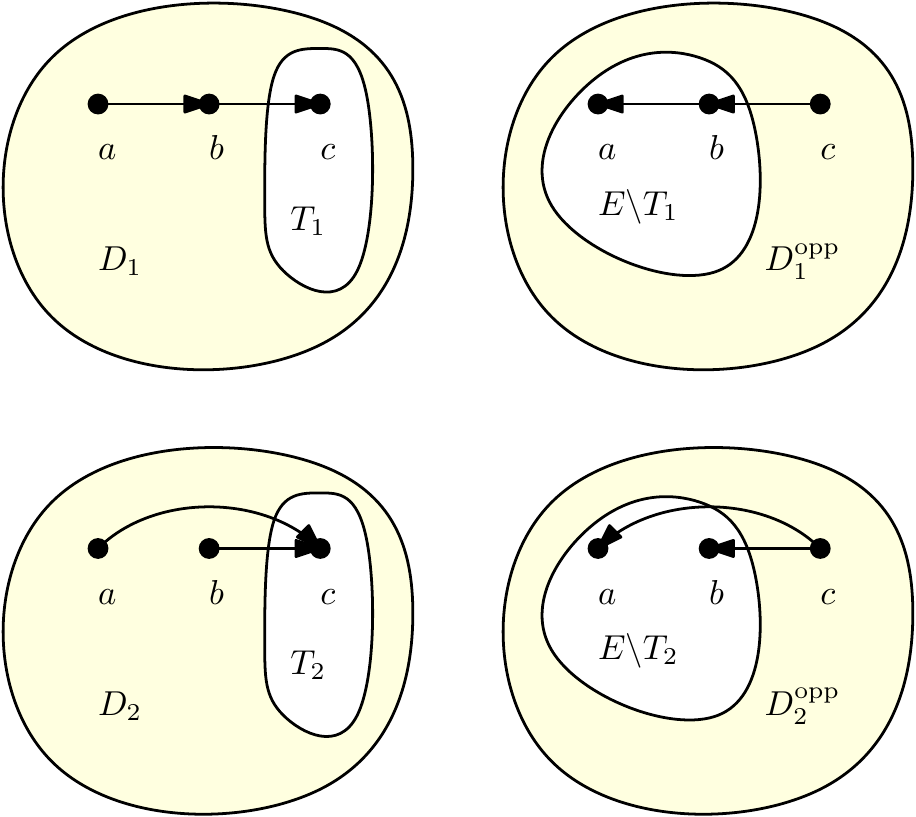}
\end{center}
\caption{Non-duality respecting and duality respecting representations of $U$. \label{fig:drnondr}}
\end{figure}
\begin{example}
Consider the uniform matroid $U=\left( \SET{a,b,c}, \SET{\emptyset,\SET{a},\SET{b},\SET{c}} \right)$
and the digraphs $D_1=	\left( \SET{a,b,c},\SET{(a,b),(b,c)} \right)$ and $D_2=\left( \SET{a,b,c},\SET{(a,c),(b,c)} \right)$
(Fig.~\ref{fig:drnondr}).
Clearly $\Gamma(D_1,\SET{c},\SET{a,b,c}) = U = \Gamma(D_2,\SET{c},\SET{a,b,c})$,
but $\Gamma(D_1^{\opp}, \SET{a,b}, \SET{a,b,c}) \not= U^\ast$, since there is no routing from $\SET{b,c}$ to $\SET{a,b}$ in $D_1^\opp$.
On the other hand, such a routing exists in $D_2^\opp$, and indeed we have $U^\ast = \Gamma(D_2^\opp,\SET{a,b},\SET{a,b,c})$. Therefore duality respecting
representations exist, but not all representations have this property.
\end{example}

\needspace{6\baselineskip}
\begin{lemma}\label{lem:sourcesinkrepresentation}\PRFR{Jan 22nd}
	Let $M=(E,\Ical)$ be a gammoid, and $B\subseteq E$ a base of $M$.
	There is a digraph $D=(V,A)$ such that the sinks of $D$ are precisely the elements of $B$,
	the sources of $D$ are precisely the elements of $E\BS B$, and
	such that $M = \Gamma(D,B,E)$.
\end{lemma}
\begin{proof}\PRFR{Jan 22nd}
	There is a digraph $D'=(V',A')$ such that $M=\Gamma(D',B,E)$ (Theorem~\ref{thm:gammoidRepresentationWithBaseTerminals}),
	and without loss of generality we may assume that the only sinks in $D'$ are the elements of $B$, and that all sources in $D'$ are
	elements of $E$ --- since sources not in $E$ and sinks not in $B$ cannot be part of a path that belongs to any routing 
	from $X\subseteq E$ to $B$ in $D'$, and
	therefore may be dropped from $D'$ without changing the represented gammoid.
	Clearly, we can give each $e \in V'\cap \left( E\BS B \right)$ a new name -- say $e''$ -- in $D'$, yielding a digraph \linebreak $D'' =(V'',A'')$
	where $V''\cap E = B$. Then we can add the elements $E\BS B$ back to $D''$ as isolated vertices, and after that,
	we add an arc leaving $e$ and entering its renamed copy $e''$ for every $e\in E\BS B$.
	Let us denote the digraph that we just constructed by \linebreak $D=(V,A)$
	where $V = V''\disunion \left( E\BS B \right)$ and $A = A'' \cup \SET{(e,e'')\mid e\in E\BS B}$.
	Clearly, each routing $R\colon X\routesto B$ with $X\subseteq E$ in $D'$ induces the routing
	$R'' = \SET{ p_1p_1''p_2\ldots p_n \mid p_1p_2\ldots p_n\in R}$ from $X$ to $B$ in $D$; and conversely,
	each routing $Q''\colon X\routesto B$ with $X\subseteq E$ in $D$ induces the routing
	$Q = \SET{p_1p_3p_4\ldots p_n\mid p_1p_2\ldots p_n\in Q''}$ from $X$ to $B$ in $D'$. Therefore, $\Gamma(D,B,E) = \Gamma(D',B,E) = M$.
\end{proof}

\begin{lemma}\label{lem:dualityrespectingrepresentation}\PRFR{Jan 22nd}
	Let $(D,T,E)$ be a representation of a gammoid with $T\subseteq E$,
	and such that every $e\in E\BS T$ is a source of $D$, and every $t\in T$ is a sink of $D$.
    Then $(D,T,E)$ is a duality respecting representation.
\end{lemma}
\begin{proof}\PRFR{Jan 22nd}
	We have to show that the bases of $N = \Gamma(D^\opp, E\BS T, E)$ are precisely the complements of the
	bases of $M = \Gamma(D,T,E)$ (Corollary~\ref{cor:dualbase}).
	%For the first case, we assume that $\rk_M(E) \leq \rk_{M^\ast}(E)$.
	Let $B\subseteq E$ be a base of $M$, then there is a linking $L\colon B\routesto T$ in $D$, and since $T$ consists of sinks,
	we have $\SET{x\in\Pbf(D)~\middle|~x\in T\cap B} \subseteq L$.
	Further, let $L^\opp = \SET{ p_n p_{n-1} \ldots p_1 \mid p_1 p_2 \ldots p_n \in L}$.
	Then $L^\opp$ is a linking
	from $T$ to $B$ in $D^\opp$ which routes $T\BS B$ to $B\BS T$. The special property of $D$, that $E\BS T$ consists of sources and that $T$ consists of sinks,
	implies, that for all $p\in L$, we have $\left| p \right|\cap E = \SET{p_1,p_{-1}}$.
	Observe that thus
 $$R = \SET{p\in L^\opp \mid p_{1}\in T \BS B} \cup \SET{x\in \Pbf(D^\opp)\mid x\in E\BS \left( T\cup B \right)}$$
	is a linking from $E\BS B=\left( T\disunion \left( E\BS T \right) \right)\BS B$ onto $E\BS T$ in $D^\opp$, thus $E\BS B$ is a base of $N$. An analog argument yields that for every base $B'$ of $N$,
	$E\BS B'$ is a base of $M$. Therefore $\Gamma(D^\opp, E\BS T, E) = \left( \Gamma(D,T,E) \right)^\ast$.
\end{proof}

\begin{corollary}\label{cor:dualityrespectingrepresentation}\PRFR{Jan 22nd}
	Let $M=(E,\Ical)$ a gammoid. Then there is a duality respecting representation $(D,T,E)$ with
	$\Gamma(D,T,E) = M$. Consequently, $M^\ast$ is a gammoid if and only if $M$ is a gammoid.
\end{corollary}
\begin{proof}
	Immediate consequence of Lemmas~\ref{lem:sourcesinkrepresentation} and \ref{lem:dualityrespectingrepresentation}.
\end{proof}

%\begin{definition}\label{def:srdualpivot}
%	Let $D=(V,A)$ be a digraph, $s\in V$ be a source of $D$, and $r \in V$ be a vertex such that
%	$(s,r)\in A$ is an arc of $D$. The \deftext[dual-pivot of a digraph]{$\bm s$-$\bm r$-dual-pivot of $\bm D$} 
%	shall be the digraph $D_{s\rightarrow r} = (V,A_{s\rightarrow r})$\label{n:digraphdualpivot} where the arc set
%	\[ A_{s\rightarrow r} = \left( A \BS \left( V\times\SET{r} \right) \right) \,\, \cup\,\, \left(  \left( \DclD{\SET{r}}{D^\opp}\times \SET{s} \right) \BSET{(s,s)} \right)\]
%	consists of arcs leaving $x$ and entering $s$ for every 
%	$x\in \DclD{\SET{r}}{D^\opp}\BSET{s}$, i.e. for every $x\not=s$ with either $x=r$ or such that there is an arc from $x$ to $r$ in $D$,
%	and	all arcs $(u,v)\in A$ of $D$ which have a head $v\not= r$.
%\end{definition}

%\begin{corollary}\label{cor:pivotdualpivot}
%	Let $D=(V,A)$ be a digraph, $s\in V$ be a sink of $D$, $r\in V$ be a vertex such that there is an arc $(s,r)\in A$.
%	Then 
%	\[ \left( D_{r\leftarrow s}  \right)^\opp = \left( D^\opp \right)_{s\rightarrow r}. \]
%\end{corollary}
%\begin{proof}
%	Immediate from Definitions~\ref{def:rspivot} and \ref{def:srdualpivot}.
%\end{proof}

%\noindent Thus, we may use digraph-duality in order to translate properties of the $r$-$s$-pivot operation into properties of the $s$-$r$-dual-pivot operation on digraphs.

\PRFR{Mar 7th}
\noindent Unfortunately, the property of a representation to be duality respecting is not preserved by the digraph pivot operation.
Thus we cannot take a duality respecting representation, pivot in a base as in the proof of Theorem~\ref{thm:gammoidRepresentationWithBaseTerminals} and then expect that the resulting representation is still duality respecting.

\needspace{9\baselineskip}

\vspace*{-\baselineskip} %Remove the line space created by the tilde below
\begin{wrapfigure}{r}{5cm}
\vspace{1.6\baselineskip}
\vspace*{-\baselineskip} %make the picture more tightly cropped
\begin{centering}%move the picture slightly to the right
\includegraphics[width=5cm]{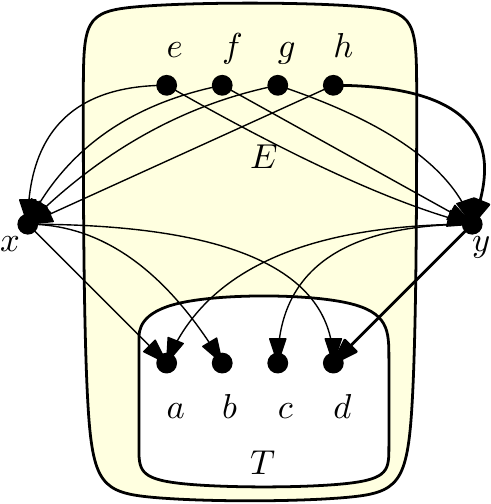}
\end{centering}%
\vspace*{-2\baselineskip} %make the picture more tightly cropped
\end{wrapfigure}
~ %The tilde creates a new dummy paragraph. WHY IS THAT NEEDED? -> would increase the space %
  % before the ex. environment. THE NEXT FREE LINE IS ESSENTIAL!

\begin{longexample}\label{ex:pivotBreaksDualityRespecting}\PRFR{Mar 7th}
Consider the gammoid $M$ on the ground set $E=\dSET{a,b,c,d,e,f,g,h}$ represented by the digraph $D=(V,A)$ with the vertex set
 $V=E\disunion\dSET{x,y}$
and the arcs $A = $ $\left( \SET{e,f,g,h}\times\SET{x,y} \right) \cup\left(  \SET{x}\times\SET{a,b,d} \right) \cup \left( \SET{y}\times \SET{a,c,d} \right)$
 together with the target set $T=\SET{a,b,c,d}$,
i.e. we have $M=\Gamma(D,T,E)$. The bases of $M$ are the set $T=\SET{a,b,c,d}$, the sets of the form $X\cup\SET{y}$ where
$X\subseteq T$ with $\left| X \right| = 3$ and $y\in\SET{e,f,g,h}$, and the sets of the form $X\cup Y$ where $X\subseteq T$ with
$\left| X \right|= 2$ and $Y\subseteq \SET{e,f,g,h}$ with $\left| Y \right| = 2$. Clearly, $(D,T,E)$ is duality respecting (Lemma~\ref{lem:dualityrespectingrepresentation}). Observe that there is only
one routing  that links $\SET{a,b,c,h}$ to $T$ -- up to symmetries of $D$ that stabilize $E$ -- namely $R=\SET{a,b,c,hxd} \subseteq \Pbf(D)$.
\end{longexample}

\needspace{9\baselineskip}

\vspace*{-\baselineskip} %Remove the line space created by the tilde below
\begin{wrapfigure}{l}{5cm}
\vspace{1.6\baselineskip}
\vspace*{-\baselineskip} %make the picture more tightly cropped
\begin{centering}%move the picture slightly to the right
\includegraphics[width=5cm]{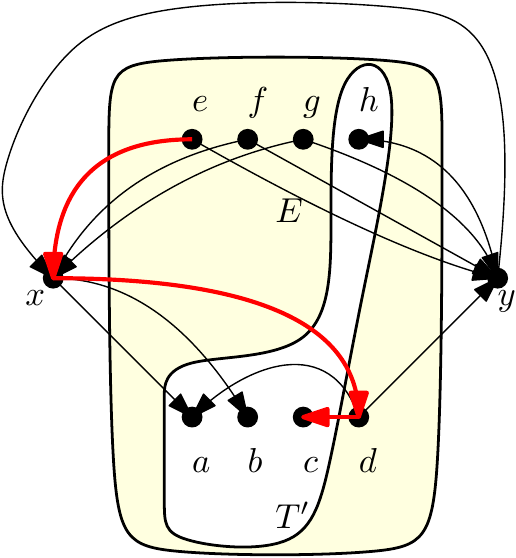}
\end{centering}%
\vspace*{-2\baselineskip} %make the picture more tightly cropped
\end{wrapfigure}
~ %The tilde creates a new dummy paragraph. WHY IS THAT NEEDED? -> would increase the space %
  % before the ex. environment. THE NEXT FREE LINE IS ESSENTIAL!

\noindent If we use the routing $R$ together with the procedure 
described in the proof of Theorem~\ref{thm:gammoidRepresentationWithBaseTerminals}, 
we obtain the digraph $D'$ depicted to the left. The vertex $h$ is now a sink in $D'$, but $d$ is not a source in $D'$, 
therefore Lemma~\ref{lem:dualityrespectingrepresentation} is not applicable to $D'$.
There are two routings that link the base $B' = \SET{a,b,e,h}$ to the target set $T'=\SET{a,b,c,h}$ in $D'$,
$R_1' = \SET{a,b,exdc,h}$ and $R_2' = \SET{a,b,eyxdc,h}$. Therefore every routing from $B$ to $T'$
uses the vertex $d$ as an inner vertex of some path, so the construction from the proof of 
Lemma~\ref{lem:dualityrespectingrepresentation} breaks at this point.
 Let $B^\ast = E\BS B' = \SET{c,d,f,g}$ and $T^\ast = E\BS T' = \SET{d,e,f,g}$. There is no routing from $B^\ast$ to $T^\ast$ in $\left( D' \right)^\opp$ because $c$ can only be linked to $d$, and 
 therefore $\rk_{\Gamma\left( \left( D' \right)^\opp, T^\ast, E \right)}\left( \SET{c,d} \right) = 1$, thus $B^\ast$ is 
 not independent in $\Gamma\left( \left( D' \right)^\opp, T^\ast, E \right)$. Consequently, $(D',T',E)$ is not a
 duality respecting representation of $M$. There are two obvious ways to modify $D'$ such that the resulting digraph is again duality 
 respecting, but both methods introduce another arc. If we would like to use Lemma~\ref{lem:dualityrespectingrepresentation} as it is 
 stated, we could rename $d$ with $x$, add a new $d$-vertex and the arc $(d,x)$ to $D'$, effectively forcing $d$ to be a source again.
 Or we could add the arc $(x,c)$ to $D'$ --- which corresponds to adding the arc $(x,c)$ to $D$ --- then
 $d$ is no longer on any essential path from $x$ to any $t\in T'$. This would imply that for every $X\subseteq E$ and
  every routing from $X$ to $T'$
  that uses $d$ as
 an inner vertex there is a routing $R_{-d}$ from $X$ to $T'$ that omits $d$ entirely. This routing $R_{-d}$ could be used
 in the construction from the proof of Lemma~\ref{lem:dualityrespectingrepresentation}, which yields a routing $R^\ast$
 linking $E\BS X$ to $T^\ast$
 in the opposite digraph.

\noindent ~
  \strut\hfill $\blacktriangle$

% -*- root: ../thesis.tex -*-

\needspace{6\baselineskip}
\subsection{Complexity-Bounded Classes of Gammoids}

\PRFR{Mar 27th}
\noindent
In this section, we introduce three measures of complexity for gammoids that are related to a class of certain
representations, and examine the corresponding classes of matroids
with a bounded complexity measure.

%\noindent In this section, we measure the complexity of a gammoid

\begin{definition}\label{def:standardRepresentation}\PRFR{Mar 27th}
	Let $M$ be a gammoid and $(D,T,E)$ with $D=(V,A)$ be a representation of $M$.
	Then $(D,T,E)$ is a \deftext[standard representation of a gammoid]{standard representation of $\bm M$},
	if $(D,T,E)$ is a duality respecting representation, $T\subseteq E$, every $t\in T$ is a sink in $D$,
	and every $e\in E\BS T$ is a source in $D$.
\end{definition}

\begin{remark}\PRFR{Mar 27th}\label{rem:standardRepresentation}
	Lemmas~\ref{lem:sourcesinkrepresentation} and \ref{lem:dualityrespectingrepresentation} guarantee that every gammoid $M$ has
	a standard representation.
\end{remark}

\begin{definition}\PRFR{Mar 27th}
	Let $M$ be a gammoid. The \deftext[arc-complexity of a gammoid]{arc-complexity of $\bm M$}
	is defined to be \label{n:ArcCompl}
	\[ \arcC(M) = \min\SET{\vphantom{A^A} {\left| A \right|} ~\middle|~ \left( (V,A), T, E \right)\text{~is a standard representation of~}M}.\]
	The \deftext[vertex-complexity of a gammoid]{vertex-complexity of $\bm M$}
	is defined to be \label{n:VertexCompl}
	\[ \vK(M) = \min\SET{\vphantom{A^A} {\left| V \right|} ~\middle|~ \left( (V,A), T, E \right)\text{~is a standard representation of~}M}.
	\qedhere \]
\end{definition}

\begin{lemma}\label{lem:arcCkVDualityAndMinors}\PRFR{Mar 27th}
	Let $M=(E,\Ical)$ be a gammoid, $X\subseteq E$. Then
	\[ \arcC(M\restrict X) \leq \arcC(M), \quad \arcC(M\contract X) \leq \arcC(M), \quad \arcC(M) = \arcC(M^\ast)\]
    as well as
    \[ \vK(M\restrict X) \leq \vK(M), \quad \vK(M\contract X) \leq \vK(M), \quad \vK(M) = \vK(M^\ast).\]
\end{lemma}
\begin{proof}\PRFR{Mar 27th}
	Let $(D,T,E)$ be a standard representation of $M$. Then $(D^\opp,E\BS T, E)$ is a standard representation of $M^\ast$:
	By Definition~\ref{def:dualityRespectingRepr} we have
	 $M^\ast = \Gamma(D^\opp, E\BS T, E)$, and since every sink of $D$ is a source of $D^\opp$
	and every source of $D$ is a sink of $D^\opp$, the set $E\BS T$ consists of sinks of $D^\opp$, 
	and the set $T= E\BS\left( E\BS T \right)$ consists of sources of $D^\opp$.
	Consequently, $\arcC(M^\ast) \leq \arcC(M)$ and $\vK(M^\ast) \leq \vK(M)$. It follows that $\arcC(M) = \arcC(M^\ast)$,
	as well as $\vK(M) = \vK(M^\ast)$, 
	since $M=\left( M^\ast \right)^\ast$ (Corollary~\ref{cor:doubleast}).

	\noindent
	Now let $(D,T,E)$ be a standard representation of $M$ where $D=(V,A)$ such that $\left| A \right| = \arcC(M)$.
	If $T\subseteq X$, then $(D,T,X)$ is a standard representation of $M\restrict X$ and therefore
	$\arcC(M\restrict X) \leq \arcC(M)$.
	Otherwise let $Y = T\BS X$, and let $B_0\subseteq X$ be a set of maximal cardinality such that
	there is a routing $R_0\colon B_0\routesto Y$ in $D$. Let $D'=(V,A')$ be the digraph that arises from $D$ 
	by a sequence of pivot operations as they are described in the proof of Theorem~\ref{thm:gammoidRepresentationWithBaseTerminals}
	with respect to the routing $R_0$.
	Observe that every $b\in B_0$ is a sink in $D'$ and that $\left| A' \right| = \left| A \right|$.
	We argue that $(D',(T\cap X) \cup B_0, X)$ is a standard representation of $M\restrict X$:
	Let $Y_0 = \SET{p_{-1}~\middle|~ p\in R_0}$ be the set of targets that are entered by the routing $R_0$.
	It follows from the proof of Theorem~\ref{thm:gammoidRepresentationWithBaseTerminals} that the triple
	$(D',(T\cap X) \cup B_0 \cup \left( Y\BS Y_0 \right), E)$ is a representation of $M$. The chain of pivot operations 
	we carried out on $D$ preserves all those sources and sinks of $D$, which are not visited by a path $p\in R_0$.
	So we obtain that every $e\in E\BS\left( T\cup B_0 \right)$ is a source in $D'$, and that every $t\in T\cap X$
	is a sink in $D'$. Thus the set $T' = (T \cap X)\cup B_0$ consists of sinks in $D'$, and the set
	$X\BS T' \subseteq E\BS\left( T\cup B_0 \right)$ consists of sources in $D'$. Therefore $(D',(T\cap X) \cup B_0, X)$ is a standard
	representation, and we give an indirect argument that $(D',(T\cap X) \cup B_0, X)$ represents $M\restrict X$.
	Clearly, $(D',(T\cap X) \cup B_0 \cup \left( Y\BS Y_0 \right), X)$ is a representation of $M\restrict X$.
	Since we assume that $(D',(T\cap X) \cup B_0, X)$ does not represent $M\restrict X$, there must be a set $X_0\subseteq X$
	such that there is a routing $Q_0\colon X_0\routesto (T\cap X) \cup B_0 \cup \left( Y\BS Y_0 \right)$ and such that
	there is no routing $X_0\routesto (T\cap X) \cup B_0$. Thus there is a path $q\in Q_0$ with $q_{-1} \in Y\BS Y_0$
	and $q_{1}\in X$. Consequently we have a routing $Q'_1 = \SET{q}\cup\SET{b\in \Pbf(D') ~\middle|~ b\in B_0}$ in $D'$.
	This implies that there is a routing $B_0\cup\SET{q_{1}} \routesto Y$ in $D$, a contradiction to the maximal cardinality of the choice of $B_0$
	above. Thus our assumption is wrong and $(D',(T\cap X) \cup B_0, X)$ is a standard representation of $M\restrict X$, so
	$\arcC(M\restrict X) \leq \arcC(M)$ holds.  
	%
	%\noindent
	Finally, let $(D,T,E)$ be a standard representation of $M$ with $D=(V,A)$ such that $\left| V \right| = \vK(M)$. By an analogue argument 
	we obtain that $\vK(M\restrict X) \leq \vK(M)$ holds.
	The previous results combined with Lemma~\ref{lem:restrictcontractdual} yield that the dual inequalities \linebreak $\arcC(M\contract X) \leq \arcC(M)$ and $\vK(M\contract X) \leq \vK(M)$ hold, too.
\end{proof}

\begin{lemma}\label{lem:kVgeqE}\PRFR{Mar 27th}
	Let $M=(E,\Ical)$ be a gammoid.
	Then $\vK(M) \geq \left| E \right|$.
\end{lemma}
\begin{proof}\PRFR{Mar 27th}
	Clear, since $E\subseteq V$ for every representation $(D,T,E)$ with $D=(V,A)$.
\end{proof}

\begin{remark}\label{rem:vKeasy}\PRFR{Mar 27th}
	Let $k\in \N$. Clearly, the class of gammoids $M$ with $\vK(M)\leq k$ is closed under duality and arbitrary minors,
	but Lemma~\ref{lem:kVgeqE} shows that this class has only a finite number of pair-wise non-isomorphic matroids. Thus such a class of gammoids is trivially characterized
	by a finite number of excluded minors,
	because there are only finitely many non-isomorphic matroids with $k+1$ elements.
\end{remark}

\begin{lemma}\label{lem:arcCSubAdditive}\PRFR{Mar 27th}
	Let $M=(E,\Ical)$ and $N =(E',\Ical')$ be gammoids with $E\cap E' = \emptyset$.
	Then $M\oplus N$ is a gammoid,
	\[ \arcC(M\oplus N) \leq \arcC(M) + \arcC(N), \txtand
	 \vK(M\oplus N) \leq \vK(M) + \vK(N). \]
\end{lemma}
\begin{proof}\PRFR{Mar 27th}
	Let $(D,T,E)$ and $(D',T',E')$ be standard representations of $M$ and $N$, respectively, such that $D =(V,A)$ and $D'=(V',A')$
	with $\left| A \right| = \arcC(M)$ and $\left| A' \right| = \arcC(N)$, and
	such that $V\cap V'=\emptyset$.
	Let $D_\oplus = (V\disunion V', A\disunion A')$. Then \mbox{$M\oplus N = \Gamma(D_\oplus, T\disunion T', E\disunion E')$}
	because there are no arcs in $D_\oplus$ connecting vertices from $V$ with $V'$ or vice versa.
	Thus every routing $R_\oplus\colon X_\oplus \routesto T\disunion T'$ in $D_\oplus$ is the disjoint union of the routings
	\linebreak $R = \SET{p\in R_\oplus ~\middle|~ \left| p \right|\subseteq V}$ 
	and $R' = \SET{p'\in R_\oplus ~\middle|~ \vphantom{A^A}\left| p' \right|\subseteq V'}$, and conversely every pair of routings
	$R\colon X\routesto T$ and $R'\colon X'\routesto T'$ yields a routing $R_\oplus = R\disunion R'$ since $V\cap V'=\emptyset$.
	Thus a set $X\subseteq E\disunion E'$ is independent in $\Gamma(D_\oplus, T\disunion T', E\disunion E')$ if and only if 
	$X\cap E$ is independent in $M$ and $X\cap E'$ is independent in $N$.  \mbox{$(D_\oplus, T\disunion T', E\disunion E')$} is a 
	standard representation of $M\oplus N$ with $\arcC(M) + \arcC(N)$ arcs, therefore
	 $$\arcC(M\oplus N) \leq \arcC(M) + \arcC(N)$$ holds.
	The same construction applied to representations $(D,T,E)$ and $(D',T',E')$ with $D=(V,A)$, $D'=(V',A')$,
	$V\cap V'=\emptyset$, $\left| V \right| = \vK(M)$, and $\left| V' \right| = \vK(N)$ yields that \(\vK(M\oplus N) \leq \vK(M) + \vK(N)\)
	holds, too.
\end{proof}

\begin{corollary}\label{cor:extWithLoopCoLoop}\PRFR{Mar 27th}
	Let $M=(E,\Ical)$ be a gammoid, $F$ and $L$ finite sets such that $F\cap L = E\cap F = E\cap L = \emptyset$.
	Then $M\oplus (F,2^F) \oplus (L,\SET{\emptyset})$ is a gammoid and
	\[ \arcC\left( M\oplus (F,2^F) \oplus (L,\SET{\emptyset}) \right) = \arcC(M) \]
\end{corollary}
\begin{proof}\PRFR{Mar 27th}
This is a direct consequence of Lemma~\ref{lem:arcCSubAdditive}
and the fact that \[\Gamma((F,\emptyset),F,F) = (F,2^F) \,\,\txtand\,\,\Gamma((L,\emptyset),\emptyset,L) = (L,\SET{\emptyset}).
\qedhere \]
\end{proof}

\begin{lemma}\label{lem:disunionArcCZero}\PRFR{Mar 27th}
	Let $M=(E,\Ical)$ be a gammoid with $\arcC(M) = 0$.
	Then there is a subset $X\subseteq E$ such that
		\[ M = (X,2^X) \oplus (E\BS X, \SET{\emptyset}).\]
\end{lemma}
\begin{proof}\PRFR{Mar 27th}
	Since there is a representation $(D,T,E)$ of $M$ with $D = (V,\emptyset)$, we obtain that the sets $X\subseteq E$ that
	are linked to $T$ are precisely the subsets of $T$. An element of $E\BS T$ can never be linked to $T$ since $\Pbf(D)$ only consists of
	trivial paths. Thus $\Ical = 2^T$ and obviously $M = (T,2^T) \oplus (E\BS T, \SET{\emptyset})$.
\end{proof}

\begin{theorem}\PRFR{Mar 27th}
	Let $\Gcal_0$ be the class of gammoids $M$ with $\arcC(M) = 0$. Then $\Gcal_0$ is closed under duality, minors, and direct sums;
	and $\Gcal_0$ is characterized by the excluded minor $U=(E,2^E\BSET{E})$ with $E=\dSET{a,b}$.
\end{theorem}
\begin{proof}\PRFR{Mar 27th}
	Lemma~\ref{lem:arcCkVDualityAndMinors} yields that $\Gcal_0$ is closed under duality and minors.
	Let $M_1,M_2\in \Gcal_0$ with disjoint ground sets.
	By Lemma~\ref{lem:arcCSubAdditive} we have $$\arcC(M_1\oplus M_2) \leq \arcC(M_1) + \arcC(M_2) = 0,$$
	so $\arcC(M_1\oplus M_2) = 0$, thus $\Gcal_0$ is closed under direct sums.
	%Then there are standard representation $(D_1,T_1,E_1)$ and $(D_2,T_2,E_2)$ of $M_1$ and $M_2$, respectively,
	%where $D_1=(V_1,\emptyset)$ and $D_2=(V_2,\emptyset)$, 
	%possibly renaming elements from $V_2$, we may assume that $V_1\cap V_2 = \emptyset$.
	%Now let $D=(V_1\disunion V_2,\emptyset)$, then $(D,T_1\disunion T_2,E_1\disunion E_2)$ is a representation of $M_1\oplus M_2$:
	%Let $X\subseteq E_1$ and $Y\subseteq E_2$ be independent. Then $X\subseteq T_1$ and $Y\subseteq T_2$. 
	%Thus $X\cup Y \subseteq T_1\cup T_2$. Conversely, if $Z$ is independent in $\Gamma(D,T_1\disunion T_2,E_1\disunion E_2)$,
	%then $Z\subseteq T_1\cup T_2$ because $D$ has no arcs. Thus $Z = (Z\cap E_1)\disunion(Z\cap E_2) = (Z\cap T_1)\disunion(Z\cap T_2)$
	%and so $Z\cap T_1$ is independent in $M_1$ and $Z\cap T_2$ is independent in $M_2$. 
	%Thus $\Gamma(D,T_1\disunion T_2,E_1\disunion E_2) = M_1\oplus M_2$, and so $\arcC(M_1\oplus M_2) = 0$.

	\noindent
	Now let $X\subsetneq E$, then $U\contract X = (X,\SET{\emptyset})$ and $U\restrict X = (X,\SET{X,\emptyset})$. 
	Clearly, $\arcC(U\contract X) = 0$ and $\arcC(U\restrict X) = 0$. Thus every proper minor of $U$ is in $\Gcal_0$.
	Now let $M\in \Gcal_0$, then $M=(F,2^F)\oplus(L,\SET{\emptyset})$ for some finite sets $F$ and $L$.
	Therefore $\Ccal(M) = \SET{\SET{l} ~\middle|~ l\in L}$, so every circuit of a matroid $M\in \Gcal_0$ has cardinality $1$.
	But $\Ccal(U) = \SET{\SET{a,b}}$, thus $U\notin \Gcal_0$.
	Now let $M=(Q,\Ical)$ be any matroid. If there is some $C\in\Ccal(M)$ with $\left| C \right| > 1$, then $M\notin\Gcal_0$
	and	$M\restrict C = (C,2^C\BSET{C})$ is a uniform matroid. Now, let $c_1,c_2\in C$ with $c_1\not= c_2$,
	then $(M\restrict C)\contract \SET{c_1,c_2} = (\SET{c_1,c_2},\SET{\emptyset,\SET{c_1},\SET{c_2}})$ is a rank-$1$ uniform matroid on a
	 $2$-elementary ground set. Therefore $(M\restrict C)\contract \SET{c_1,c_2}$ is a minor of $M$ that is isomorphic to $U$. 
	If there is no $C\in \Ccal(M)$ with $\left| C \right| > 1$,
	 then let $L_Q = \SET{q\in Q ~\middle|~ \SET{q}\in\Ccal(M)}$ and $F_Q = Q\BS L$.
	Clearly, $M = (F_Q,2^{F_Q})\oplus (L_Q,\SET{\emptyset})$ and it is easy to see that $\arcC(M) = 0$, thus $M\in\Gcal_0$.
	Therefore the class $\Gcal_0$ is characterized by the single excluded minor $U$.
\end{proof}

\begin{lemma}\label{lem:oplusReprOfGk}\PRFR{Mar 27th}
	Let $k\in \N$ and $M=(E,\Ical)$ be a gammoid with $\arcC(M) = k$.
	Then there is a partition $E_1\disunion E_2\disunion E_3$ of $E$ such that
	\[ M = \left( M\restrict E_1 \right) \oplus (E_2,2^{E_2}) \oplus (E_3,\SET{\emptyset}),\]
	and such that $\arcC(M\restrict E_1) = k$.
	Furthermore, $\left| E_1 \right| \leq 2k$, $\rk_M(E_1) \leq k$,
	 and there is a set $X_0\subseteq E_1$ with cardinality at most $\rk_M(E_1)$
	such that for every $X\subsetneq E_1$ with $X_0\subseteq X$
	\[ \arcC(M\restrict X) < k .\]
\end{lemma}
\begin{proof}\PRFR{Mar 27th}
	Let $(D,T,E)$ be a standard representation of $M$ with $D=(V,A)$ and $\left| A \right| = k$.
	We may partition $V$ into $V_1 = \SET{v\in V ~\middle|~ \exists u\in V\colon\,(u,v)\in A \txtor (v,u)\in A}$,
	the set of vertices incident with an arc,
	and $V_2 = V\BS V_1$, the set of isolated vertices. 
	There are no arcs in the induced digraph $D'=(V_2,A\cap \left( V_2\times V_2 \right))=(V_2,\emptyset)$,
	thus we obtain that 
	\[ M' = M\restrict \left( E\cap V_2 \right) = \Gamma(D,T,E\cap V_2) = \Gamma(D',T\cap V_2, E\cap V_2) \] and consequently
	we have $\arcC(M\restrict V_2) = 0$. Therefore there are disjoint $F,L\subseteq V_2$ such that
	$M\restrict V_2 = (F,2^F)\oplus (L,\SET{\emptyset})$ (Lemma~\ref{lem:disunionArcCZero}).
	Now let $X\subseteq V_1 \cap E$ with $X\in\Ical$. %, which implies $X\subseteq E$.
	Then there is a routing $R\colon X\routesto T$ in $D$,
	and since no arc of $D$ is incident with $v\in V_2$, we obtain that $\left| p \right|\subseteq V_1$ for all $p\in R$.
	Therefore we may conclude that $X$ is independent in $M'' = \Gamma(D'',T\cap V_1, E\cap V_1)$ for
	$D'' = (V_1,A)$, thus $M'' = M\restrict \left( E\cap V_1 \right)$. Therefore, for all $X\subseteq E$ with $X\in\Ical$, 
	we have that $X\cap V_1$ is independent in $M''$,
	and that $X\cap V_2$ is independent in $M'$. Thus
	\[ M = M''\oplus M' = \left(\vphantom{2^{T\cap V_2}} M\restrict \left( E\cap V_1 \right) \right)
		\oplus \left( T\cap V_2, 2^{T\cap V_2} \right) \oplus (\vphantom{2^{T\cap V_2}}V_2\BS T, \SET{\emptyset}).\]
	Assume that $\arcC(M\restrict \left( E \cap V_1 \right)) < \arcC(M)$, then we could take a standard representation of
	$M\restrict \left( E\cap V_1 \right)$ and augment it with isolated vertices in order to obtain a standard representation of $M$ 
	with fewer than $\arcC(M)$ arcs --- yielding a contradiction. Therefore $\arcC(M\restrict \left( E \cap V_1 \right)) = \arcC(M)$.
	Since every element of $E_1=E\cap V_1$ must be incident with at least one arc in $D$, and every arc is incident with two vertices,
	and since $\left| A \right|=k$, we obtain that $\left| E_1 \right| \leq \left| V_1 \right| \leq 2k$. Furthermore,
	every arc in $D$ is incident with at most one source and at most one sink, thus $\left| T\cap V_1  \right| \leq k$,
	and therefore $\rk_M(E_1) \leq k$.

	\noindent
	Now we show that  $\arcC(M\restrict X) < \arcC(M)$ holds for every $X\subsetneq E_1$ with $T\cap V_1 \subseteq X$ 
	by constructing a smaller representation.
	Let $X\subsetneq E_1$ such that $\arcC(M\restrict X) = \arcC(M)$ and $T\cap V_1 \subseteq X$, 
	and let $x\in E_1\BS X$.
	%Then there is a base $B\in \Bcal(M)$ such that $x\notin B$, because otherwise $x\in V_2$.
	Since $x\notin T\cap V_1$ we have
	 $x\notin T$. Let
	  $$D_x=\left(\vphantom{A^A}V_1\BSET{x},A\cap\left(\left( V_1\BSET{x} \right)\times \left( V_1\BSET{x} \right)\right)\right)$$
	be the digraph induced from $D$ by removing the source $x$.
	Clearly, $D_x$ has fewer arcs than $D$ because at least one arc in $D$ is incident with $x$.
	But then the contraction of $M$ to $X$ satisfies the equation $M\restrict X = \Gamma(D_x,T,X)$, 
	which implies $\arcC(M\restrict X) < \arcC(M)$. % --- a contradiction to our assumption $\arcC(M\restrict X) = \arcC(M)$.
\end{proof}

\begin{theorem}\label{thm:arcCquiteEasy}\PRFR{Mar 27th}
	Let  $k\in \N$, $k \geq 1$, and let $\Gcal_k$ be the class of gammoids $M$ with $\arcC(M) \leq k$. 
	Then $\Gcal_k$ is closed under duality and minors, but not under direct sums;
	and $\Gcal_k$ is characterized by finitely many excluded minors.
\end{theorem}
\begin{proof}\PRFR{Mar 27th}
	Let $k\in \N$ be arbitrarily fixed from now on.
	Lemma~\ref{lem:arcCkVDualityAndMinors} yields that $\Gcal_k$ is closed under duality and minors.
	Now let $M_i = (\dSET{a_i,b_i},\SET{\emptyset,\SET{a_i},\SET{b_i}})$ for
	\linebreak
	 $i\in \SET{1,2,\ldots,k+1}$,
	such that $\dSET{a_i,b_i}\cap \dSET{a_j,b_j}= \emptyset$ for all $i,j\in \SET{1,2,\ldots,k+1}$ with $i\not= j$.
	Then $\arcC(M_i) = 1$, because $M_i$ is neither free nor does $M_i$ consist of loops, and it can be represented
	by $((\SET{a_i,b_i},\SET{(a_i,b_i)}),\SET{b_i},\SET{a_i,b_i})$. 
	Now let $N = \bigoplus_{i=1}^{k+1} M_i$, and let $(D,T,E)$  be a standard representation of $N$ with $D=(V,A)$. 
	Then $\left| T \right| = \rk_N(E) = k+1$ and $\left| E \right| = 2k +2$. Now assume that $\left| A \right| \leq k$,
	i.e. that $N\in\Gcal_k$. There is some $e\in E\BS T$ such that $e$ is not incident with an arc from $A$, thus $\SET{e}$ cannot
	be linked to $T$ in $D$.
	But then $\rk_N(\SET{e}) = 0$ follows, which is a contradiction to the fact that $\rk_{M_i}(\SET{e})= 1$ for the appropriate
	index $i$. Thus $N\notin \Gcal_k$, and consequently, $\Gcal_k$ is not closed under direct sums.

	\noindent
	Now let $M =(E,\Ical)$ be a matroid.
	If $M\in\Gcal_k$, then Lemma~\ref{lem:oplusReprOfGk} yields that there is a partition $E_1\disunion E_2\disunion E_3=E$
	with\[ M = \left( M\restrict E_1 \right) \oplus (E_2,2^{E_2}) \oplus (E_3,\SET{\emptyset}) \]
	and $\left| E_1 \right| \leq 2k$ such that $\arcC(M\restrict E_1) \leq k$.
	%On the other hand, Lemma~\ref{lem:arcCSubAdditive} yields that if a matroid $M$ 
	%can be expressed as such a direct sum, then $M\in \Gcal_k$.
	Now let $M=(E,\Ical)$ be an excluded minor for $\Gcal_k$. Then for all $e\in E$ the restriction
	$M\restrict\left( E\BSET{e} \right)\in \Gcal_k$. Thus Lemma~\ref{lem:arcCSubAdditive}
	 yields that for all $e\in E$
	\[ M\restrict\left( E\BSET{e} \right)\oplus(\SET{e},\SET{\emptyset,\SET{e}}) 
	\not= M \not= M\restrict\left( E\BSET{e} \right)\oplus(\SET{e},\SET{\emptyset}), \]
	i.e. $M$ has neither a loop nor a coloop. In this case, 
	Lemma~\ref{lem:oplusReprOfGk} implies that $\left| E\BSET{e}  \right| \leq 2k$,
	so $\left| E \right| \leq 2k+ 1$, thus every excluded minor for $\Gcal_k$ has at most $2k+1$ elements. But up to isomorphism, 
	there are only finitely many
	matroids on ground sets with at most $2k+1$ elements, so $\Gcal_k$ is characterized by finitely many excluded minors.
\end{proof}

\noindent We have seen that subclasses of gammoids, that are defined by limiting the number of arcs or the number of vertices 
available in a standard representation, merely consist of a finite number of matroids which may be extended with an arbitrary amount of
loops and coloops. Moreover, except for $\Gcal_0$, those classes are not closed under direct sums.

\needspace{6\baselineskip}
\begin{definition}\label{def:arcWf}\PRFR{Mar 27th}
	Let $f\colon \N\maparrow \N\BSET{0}$ be a non-decreasing function, and let
	\linebreak
	 $M=(E,\Ical)$ be a gammoid. The \deftext[width of a gammoid]{$\bm f$-width of $\bm M$}
	shall be \label{n:arcWfM}
	\[ \arcW_f(M) = \max\SET{\frac{\arcC\left( \left( M\contract Y \right)\restrict X \right))}{f\left( \left| X \right|
	% - \rk_{M\contract Y}(X)
	 \right) }
		 ~\middle|~ X\subseteq Y\subseteq E}. \]
	Let $k\in \N$, then the \deftextX{$\bm k$-width of $\bm M$} shall be
	\[ \arcW^k(M) = \arcW_{f_k}(M) \] where 
	\[ f_k\colon \N\maparrow \N\BSET{0},\,n\mapsto \max\SET{1, k\cdot n} . \qedhere\]
\end{definition}

\noindent Clearly $\arcW^{0}(M) = \arcC(M)$ for all gammoids $M$.

\begin{corollary}\label{cor:WfClosedUnderMinors}\PRFR{Mar 27th}
	Let $M=(E,\Ical)$ be a gammoid, $X\subseteq Y\subseteq E$.
	Then \[ \arcW_f(M) = \arcW_f(M^\ast) \,\,\txtand\,\, \arcW_f\left( \left( M\contract Y \right)\restrict X \right) \leq \arcW_f(M) .\]
\end{corollary}

\begin{proof}\PRFR{Mar 27th}
	The second inequality is a direct consequence of the Definition~\ref{def:arcWf}.
		Let $M=(E,\Ical)$ be a gammoid and $X\subseteq Y\subseteq E$,
	  then 
	 \[ (M^\ast\contract Y)\restrict X = \left( (M\restrict Y)\contract X \right)^\ast = \left( (M\contract E\BS\left( Y\BS X \right))\restrict X \right)^\ast  \]
	 holds due to Lemmas~\ref{lem:restrictcontractdual} and \ref{lem:contractrestrictcommutes}, 
	 and Remark~\ref{rem:contractRestrictCommutingFormula}.
	 Since $N$ and $N^\ast$ share the same ground set and
	  $\arcC(N) = \arcC(N^\ast)$ for all gammoids $N$ (Lemma~\ref{lem:arcCkVDualityAndMinors}), we obtain that
	 \[ \arcW_f(M) = \arcW_f(M^\ast). \qedhere\]
\end{proof}

%\begin{remark}\PRFR{Mar 27th}
%	Let $f\colon \N \maparrow \N\BSET{0}$ where $f(n) = 4n^6 + n$ for $n>0$ and $f(0) = 1$,
%	and let $M$ be a gammoid.
%	Then $\arcW_f(M) \leq 1$, since  $M$ can be represented
%	using at most $2n^3$ vertices (Remark~\ref{rem:upperBoundForV}), 
%	and such digraphs can have at most $4n^6$ arcs. We might have to introduce at most $n$ arcs
%	in order to obtain a standard representation. Thus in this case $\Gcal_{f,1}$ is just the class of gammoids.%
%	\footnote{Corollary~2.5 in \cite{M72}
%	states that it is a direct consequence of the construction used in the proof of Lemma~\ref{lem:presentsStrictGammoid}, that
%	a strict gammoid $N=(E',\Ical')$ can be represented with at most
%	 $\left( \left| E' \right| - \rk_{N}(E')\right)\cdot (\left| E' \right|-1)$ 
%	 arcs. Apparently, this upper bound does not yield a map 
%	 $f'$ with $\Gcal_{f',1} = \Gcal_{f,1}$ that is significantly smaller than $4n^6 + n$
%	  when we plug in the same estimate $\left| E' \right| \leq 2n^3$.}
%\end{remark}

\begin{definition}\PRFR{Mar 27th}
	Let $f\colon \N\maparrow \N\BSET{0}$ be a non-decreasing function. We say that $f$ is \deftext{super-additive},
	if for all $n,m\in \N\BSET{0}$
	\[ f(n+m) \geq f(n) + f(m) \]
	holds.
\end{definition}

\needspace{3\baselineskip}
\begin{lemma}\PRFR{Mar 27th}
	Let $f\colon \N\maparrow \N\BSET{0}$ be a non-decreasing and super-additive function, let $k\in \N$,
	and let $\Wcal_{f,k}$ denote the class of gammoids $M$ with $\arcW_f(M) \leq k$.
	Then $\Wcal_{f,k}$ is closed under duality, minors, and direct sums.
\end{lemma}

\begin{proof}\PRFR{Mar 27th}
	It is clear from Corollary~\ref{cor:WfClosedUnderMinors} that $\Wcal_{f,k}$ is closed under minors and duality.
	 	 %\noindent
	 Now, let $M=(E,\Ical)$ and $N=(E',\Ical')$ with $E\cap E' =\emptyset$ and $M,N\in \Wcal_{f,k}$.
	 Furthermore, let $X\subseteq Y\subseteq E\cup E'$. Then, by Lemmas~\ref{lem:directSumAndRestrictionCommute} and \ref{lem:directSumAndContraction}, we have that
	 \begin{align*}
	 	 \left( (M\oplus N)\contract Y \right)\restrict X &  \,\,\,=\,\,\,
	 	 \left(\left( M\contract Y\cap E \right) \oplus ( N\contract Y'\cap E )\right)\restrict X
	 	 \\& \,\,\,=\,\,\,
	 	 	\left(\vphantom{A^A}\left( M\contract Y\cap E \right)\restrict X\cap E\right)  \oplus 
	 	 	\left( ( N\contract Y\cap E' )\restrict X\cap E'\right).
	 \end{align*}
%	 Furthermore, it is clear from Definition~\ref{def:directSum} that
%	 \[ \rk_{(M\oplus N)\contract Y}(X) = \rk_{M\contract Y\cap E}(X\cap E) + \rk_{N\contract Y\cap E'}(X\cap E')\]
	 holds.
	 With Lemma~\ref{lem:arcCSubAdditive} we obtain
	 \begin{align*}
	 	\arcC\left(\vphantom{A^A} \left( (M\oplus N)\contract Y \right)\restrict X  \right) & \,\,\,\leq \,\,\,
	 	\arcC\left(\vphantom{A^A} \left( M\contract Y\cap E \right)\restrict X\cap E\right) 
	 	+ \arcC\left(\vphantom{A^A} ( N\contract Y\cap E' )\restrict X\cap E'\right) .
	 \end{align*}
	 We use the super-additivity of $f$ in order to derive
	 \begin{align*}
	 	\frac {\arcC\left(\vphantom{A^A} \left( (M\oplus N)\contract Y \right)\restrict X  \right)}{f\left( \left| X \right|
	 	% - \rk_{(M\oplus N)\contract Y}(X) 
	 	\right)} &
	 	\,\,\, \leq\,\,\,
	 	\frac { \arcC\left(\vphantom{A^A} \left( M\contract Y\cap E \right)\restrict X\cap E\right) 
	 	+ \arcC\left(\vphantom{A^A} ( N\contract Y\cap E' )\restrict X\cap E'\right) }{f\left( \left| X \right|
	 	% - \rk_{(M\oplus N)\contract Y}(X) 
	 	\right) } \\
%	 	&
%	 	\,\,\, \leq\,\,\,
%	 	\frac { \arcC\left(\vphantom{A^A} \left( M\contract Y\cap E \right)\restrict X\cap E\right) 
%	 	+ \arcC\left(\vphantom{A^A} ( N\contract Y\cap E' )\restrict X\cap E'\right) }{f\left( \left| X\cap E \right| \right)
%	 	+ f\left( \left| X\cap E' \right| \right) } \\
	 	&% \!\!\!\!\!\!\!\!\!\!\!\!\!\!\!\!\!\!\!\!\!\!\!\!\!\!\!\!\!\!\!\!\!\!\!\!\!\!\!\!\!\!\!\!\!\!\!\!\!\!\!\!\!\!\!\!\!\!\!\!\!\!\!\!\!\!\!\!\!\!\!\!\!\!
	 	\,\,\, \leq\,\,\, \frac{k\cdot f\left( \left| X\cap E \right|
	 	% - \rk_{M\contract Y\cap E}(X\cap E)
	 	\right) + k\cdot f\left( \left| X\cap E' \right|
	 	% - \rk_{N\contract Y\cap E'}(X\cap E')
	 	 \right)}{f\left( \left| X \right|
	 	% - \rk_{(M\oplus N)\contract Y}(X)
	 	 \right) } \\
	&%\!\!\!\!\!\!\!\!\!\!\!\!\!\!\!\!\!\!\!\!\!\!\!\!\!\!\!\!\!\!\!\!\!\!\!\!\!\!\!\!\!\!\!\!\!\!\!\!\!\!\!\!\!\!\!\!\!\!\!\!\!\!\!\!\!\!\!\!\!\!\!\!\!\!
	 	\,\,\, =\,\,\, \frac{k\cdot \left(\vphantom{A^A} f\left( \left| X\cap E \right|
	 	%  - \rk_{M\contract Y\cap E}(X\cap E)
	 	\right) + f\left( \left| X\cap E' \right|
	 	% - \rk_{N\contract Y\cap E'}(X\cap E')
	 	\right) \right)}
	 	{f\left( \left| X \right|
	 	%- \rk_{(M\oplus N)\contract Y}(X)
	 	 \right)} \\ &
	 	%\!\!\!\!\!\!\!\!\!\!\!\!\!\!\!\!\!\!\!\!\!\!\!\!\!\!\!\!\!\!\!\!\!\!\!\!\!\!\!\!\!\!\!\!\!\!\!\!\!\!\!\!\!\!\!\!\!\!\!\!\!\!\!\!\!\!\!\!\!\!\!\!\!\!
	 	\,\,\, \leq\,\,\, k\cdot \frac{f\left( \left| X \right|
	 	% - \rk_{(M\oplus N)\contract Y}(X) 
	 	\right)}{f\left( \left| X \right|
	 	% - \rk_{(M\oplus N)\contract Y}(X) 
	 	\right)} \,\,\, = \,\,\, k	 	,
	 \end{align*}
	 where the second inequality follows from
	  the fact that \[% k \geq \arcW_f(G) \geq
	   \frac{\arcC(G)}{f\left( \left| F \right|
	  % - \rk_G(F)
	   \right)} \leq \arcW_f(G) \leq k\] holds
	 for every $G=(F,\Jcal)\in \Wcal_{f,k}$, thus it holds for all minors of $M$ and $N$ 
	 (Corollary~\ref{cor:WfClosedUnderMinors}).
	 As a consequence, $\arcW_f(M\oplus N) \leq k$, and therefore
	 $M\oplus N\in \Wcal_{f,k}$ holds.
\end{proof}

\PRFR{Apr 1st}
\noindent We may consider a class of matroids, that is closed under direct sums, and that contains a matroid, that is
neither trivial nor free, to be truly infinite, as opposed to a class that consists of matroids, that are 
direct sums of free matroids, trivial matroids, and one matroid that is isomorphic to a member of a finite family of
matroids.

\begin{theorem}\label{ref:infiniteChainOfSubclasses}\PRFR{Apr 1st}
	Let $\left(M_k\right)_{k\in \N}$ with $M_k = (E_k,\Ical_k)$ be a sequence of gammoids 
	with \[ \arcC(M_k) \geq k\cdot \left| E_k \right|.\]
	Then there is an infinite chain of strictly bigger classes of gammoids that are closed under
	duality, minors, and direct sums in the family of classes
	\[ \Wcal^\N = \SET{ \Wcal^k ~\middle|~ \Wcal^k \text{~is the class of all gammoids $M$ with~} \arcW^k(M) \leq 1,\,k\in\N }. \]
\end{theorem}
\begin{proof}\PRFR{Apr 1st}
	Clearly, we have that $\arcW^k(M) > \arcW^{k'}(M)$ and
	 $\arcW^k(M_{k'}) \geq \frac{k'}{k} > 1$ for all $k,k'\in \N$ with $k' > k$,
	so every class $\Wcal^k$ contains at most $k$ elements of the matroid sequence $\left( M_k \right)_{k\in \N}$,
	and every class $\Wcal^{k'}$ contains the class $\Wcal^k$ if $k' > k$.
	Furthermore, $M_{k'}$ is contained in $\Wcal^{\arcC(M_{k'})}$, therefore every matroid of the sequence is
	eventually contained in some $\Wcal^k$. Consequently, $\Wcal^\N$ must contain a countable chain of strictly
	bigger subclasses of gammoids.
\end{proof}

%%%TODO-UNIFORM CONJECTURE
\noindent Conjecture~\ref{conj:uniformArcs} would imply that there is a strict chain of truly 
infinite subclasses of gammoids that are closed under
minors and duality, and that  $\Wcal^i$ is a proper subclass of $\Wcal^{i+1}$ for all $i\in \N$.

\begin{lemma}\label{lem:uniformArcs}\PRFR{Apr 1st}
	Let $E$ be a finite set, $r\in \N$ with $r \leq \left| E \right|$, and let $$U = \left( E, \SET{X\subseteq E ~\middle|~
	\vphantom{A^A} \left| X \right| \leq r} \right)$$
	be the uniform matroid of rank $r$ on $E$.
	Then \[ \vK(U) = \left| E \right| \,\,\,\txtand\,\,\, \arcC(U) \leq r\cdot \left( \left| E \right| - r \right) .\]
\end{lemma}
\begin{proof}\PRFR{Apr 1st}
	Let $T \subseteq E$ with $\left| T \right| = r$ and let $D=(E,A)$ be the digraph on the vertex set $E$
	 where $A = \SET{(e,t)~\middle|~ e\in E\BS T,\, t\in T}$.
	Clearly, $(D,T,E)$ is a standard representation with $U = \Gamma(D,T,E)$. 
	Therefore $\vK(U) \leq \left| E \right|$ and $\arcC(U) \leq r\cdot \left( \left| E \right| - r \right)$.
	Obviously, the vertex complexity is bounded from below by the size of the ground set, thus $\vK(U) = \left| E \right|$.
\end{proof}

%\begin{definition}
%	Let $f\colon \N\maparrow \N\BSET{0}$ be a non-decreasing function.
%	The \deftext{characteristic of $\mathbf{W}_{\bm f}$} shall be \label{n:kWf}
%	\[ \charW(f) = \inf \SET{\frac{f(x+y)}{f(x)+f(y)} ~\middle|~ x,y\in \N } .\]
%\end{definition}

%\begin{corollary}
%	Let $k\in \N$. The characteristic of $\arcW^k$ equals $1$.
%\end{corollary}

\PRFR{Apr 1st}
\noindent The following kind of matroids is usually defined as matroids, whose ground sets consist of edges of undirected 
graphs, such that subsets of these edges are independent, if they contain no subgraph that consists of {\em (i)}
two cycles with a single common vertex ($\infty$-graph), {\em (ii)} two cycles which
share a common line segment ($\Theta$-graph),
or {\em (iii)} two cycles each of which has a special vertex and those special vertices are connected by a line (hand-cuffs graph).
We use L.R.~Matthews's characterization in order to define bicircular matroids.
\begin{definition}[\cite{Ma77}, Corollary~3.3 and Theorem~3.5]\PRFR{Apr 1st}
	Let $M=(E,\Ical)$ be a matroid. Then $M$ is a \deftext{bicircular matroid}, if there is a family $\Acal=(A_i)_{i=1}^{\rk_M(E)}$ of subsets of $E$ with the property that $\left| \SET{i\in I~\middle|~e\in A_i}\right| \in \SET{1,2}$ holds
	 for all $e\in E$, and such that $M = M(\Acal)$.
\end{definition}

\noindent It is clear that bicircular matroids are special gammoids.

\begin{lemma}\PRFR{Apr 1st}
	Let $M=(E,\Ical)$ be a bicircular matroid. Then
	\[ \arcC(M) \leq 2\cdot \left| E \right| \quad\txtand\quad \vK(M) \leq \left| E \right| + \rk_M(E) .\]
\end{lemma}
\begin{proof}\PRFR{Apr 1st}
	Let $I$ be a set with $\left| I \right| = \rk_M(E)$ and $\Acal =(A_i)_{i\in I}$ be a family of subsets of $E$
	such that $M= M(\Acal)$ and such that $\left| \SET{i\in I~\middle|~e\in A_i}\right| \in \SET{1,2}$ for all $e\in E$.
	For technical reasons, let us further assume that $I\cap E = \emptyset$. Let $D_0 = (V,A)$ with $V=E\disunion I$ and
%	\linebreak
	$A = \SET{(e,i)~\middle|~e\in E,\,i\in I,\,e\in A_i}$.
	 Then $M = \Gamma(D_0,I,E)$ and $\left| A \right| \leq 2\cdot \left| E \right|$.
	We obtain a standard representation of $M$ by pivoting in an 
	arbitrary base $T\in \Bcal(M)$ as it is done in the proof of Theorem~\ref{thm:gammoidRepresentationWithBaseTerminals}.
	This operation does not introduce any new arcs or vertices,
	 therefore $\arcC(M) \leq \left| A \right| \leq 2\cdot \left| E \right|$ 
	 and $\vK(M) \leq \left| E \right| + \left| I \right| = \left| E \right| + \rk_M(E)$ holds.
\end{proof}

%\remred{
%\begin{theorem}[\cite{Ma77}, Theorem~3.6]
%	Let $M$ be a bicircular matroid, and $N$ be a minor of $M$. If $N$ has no loops, then $N$ is a bicircular matroid.
%\end{theorem}
%\noindent For a proof, see \cite{Ma77}, p.218.
%
%\begin{corollary}
%	Let $M$ be a bicircular matroid, then $W^1(M) \leq 2$.
%\end{corollary}
%\begin{proof}
%	Let $N=(E,\Ical)$ be a minor of $M$, if $N$ has no loops, then $\arcC(M) \leq 2\cdot \left| E \right|$
%	Puh, doch nicht so einfach.
%\end{proof}
%}

% -*- root: ../thesis.tex -*-

\needspace{6\baselineskip}

\subsection{Essential Arcs and Vertices}

\PRFR{Apr 1st}
\noindent Let $(D,T,E)$ be a representation of a gammoid, and let $D=(V,A)$. In this section, we are concerned with the
question when an arc $a\in A$ or a vertex $v\in V$ is
essential for the representation of $\Gamma(D,T,E)$. It turns out that this kind of question may be answered by
inspection of the family of independent sets of
a derived gammoid.

\begin{definition}\label{def:ACD}\PRFR{Apr 1st}
	Let $(D,T,E)$ with $D=(V,A)$ be a representation of the gammoid $\Gamma(D,T,E)=(E,\Ical)$,
	and let $a\in A$ be an arc of $D$.
	The arc $a$ shall be called \deftext[essential arc of DTE@essential arc of $(D,T,E)$]{essential arc of $\bm(\bm D\bm, \bm T\bm, \bm E\bm)$},
	if there is some $X\in \Ical$ such that $X$ is not independent with respect to $\Gamma(D_a,T,E)$ where
	$D_a = (V,A\BSET{a})$.
\end{definition}

\begin{remark}\PRFR{Apr 1st}
 If $(D,T,E)$ with $D=(V,A)$ is a representation of $M= \Gamma(D,T,E)$ such that $\left| A \right| = \arcC(M)$,
 then every arc $a\in A$ is essential. Also, the converse is not true: Let $(D,T,E)$
 be a representation of a gammoid
 such that every arc of $D=(V,A)$ is essential. If we subdivide an arc of $D$ 
 with a newly introduced auxiliary vertex, 
 then the resulting digraph $D'$ still consists only of essential arcs
 with respect to $(D',T,E)$ --- but $(D',T,E)$ can no longer have an arc set of minimal cardinality.
 \end{remark}

\begin{lemma}\label{lem:essentialArcsC}\PRFR{Apr 1st}
	Let $M=(E,\Ical)$ be a gammoid, and let $(D,T,E)$ be a representation of $M$ with $D=(V,A)$.
	Let $(u,v)\in A$ be an essential arc of $(D,T,E)$,
	and let 
	\linebreak
	 $N = \Gamma(D,T,V)$ and $N' = \Gamma(D',T,V)$ where
	$D'=(V,A\BSET{(u,v)})$.
	There is a circuit $C\in \Ccal(N')$ with $u\in C$ such that $C$ is independent in $N$.
\end{lemma}
\begin{proof}\PRFR{Apr 1st}
	Clearly, if $N = N'$, then $(u,v)$ is not an essential arc of $(D,T,E)$.
	Therefore there is a subset $X\subseteq E \subseteq V$ that is independent in $N$ yet dependent in $N'$.
	Since every routing in $D$ is a routing in $D'$ unless it traverses the arc $(u,v)$,
	we observe that every routing $R\colon X\routesto T$ in $D$ must traverse the arc $(u,v)$.
	Since $X$ is dependent in $N'$, there is an minimum-cardinality $X$-$T$-separator
	 $S'$ in $D'$ with $\left| S' \right| < \left| X \right|$.
	With the previous observation we obtain that $S = S'\cup\SET{u}$ is
	an $X$-$T$-separator in $D$ with $\left| S \right| = \left| X \right|$.
	Furthermore, we see that $u\notin S'$, because otherwise $S'$ would be an $X$-$T$-separator in $D$,
	which would lead us to the contradiction $\rk_{N}(X) \leq \left| S' \right| < \left| X \right| = \rk_N(X)$ ---
	as $X$ is an independent set of $N$.
	Corollary~\ref{cor:Menger} yields that we may cut off the initial parts of the paths of a maximal $X$-$T$-connector in $D$
	and thereby obtain a routing from $S$ to $T$ in $D$, so $S$ is independent in $N$.
	We give an indirect argument that $S$ is dependent in $N'$.
	 Assume that $S$ is independent in $N'$. $S'$ is a minimal cardinality $X$-$T$-separator in $D'$,
	 thus $S'\subseteq \cl_{N'}(X)$ (Corollary~\ref{cor:Menger}).
	 If there is a path $p\in \Pbf(D)$ with %$p_1\in \cl_{N'}(X)$
	 $p_1\in X\BS S'$ and $p_{-1}=u$
	 that does not visit a vertex $s\in S'$, then $S'\cup\SET{p_{1}}$ is independent,
	 and so we obtain 
	 \[\rk_{N'}(X) = \rk_{N'}\left( \cl_{N'}(X) \right) 
	 \geq \rk_{N'}(S'\cup\SET{p_{1}}) = \left| S' \right| + 1 = \left| X \right|.\]
	 Thus $X$ would be independent in $N'$ --- a contradiction.
	 To avoid this contradiction, every path $p\in \Pbf(D)$ with $p_1 \in X$ and $p_{-1}=u$
	 must visit a vertex $s\in S'$. But then $S'$ is an $X$-$S$-separator in $D$,
	  and since $S$ is an $X$-$T$-separator in $D$,
	 we have that $S'$ is an \linebreak $X$-$T$-separator in $D$.
	 Again, this yields the contradiction
	 $\rk_N(X) \leq \left| S' \right| < \left| X \right| = \rk_N(X)$.
	 Therefore
	 we may dismiss our assumption and we conclude that $S$ is dependent in $N'$.
	 Remember that $S'$ is independent in $N'$ because it is a minimal-cardinality $X$-$T$-separator in $D'$,
	 thus there is a circuit $C\in\Ccal(N')$ with $C\subseteq S$ and $C\not\subseteq S'$, so $u\in C$;
	 and since $S$ is independent in $N$, we obtain that $C$ is independent in $N$, too.
\end{proof}

%\noindent We give a construction that may be considered vaguely
%similar to that used in the proof of Theorem~3 in \cite{KW12} (Theorem~\ref{thm:upperBoundSizeOfV}).
%In contrast to S.~Kratsch and 
%M.~Wahlström, we are not interested in a subset of vertices in digraphs which
%contain sufficiently many minimal separators, 
%but in the arcs of gammoid representations that have to be used to .

\begin{definition}\PRFR{Apr 1st}
	Let $D=(V,A)$ be a digraph
	with $V\cap \left( \left( V\times V  \right)\times \SET{1,2} \right) = \emptyset$.
	The \deftext[arc-cut digraph]{arc-cut digraph for $\bm D$} shall be the \label{n:arcCutDigraph}
	digraph $\mathrm{AC}(D) = (V_{D},A_{D})$ where
	\begin{align*}
	 V_D \,\,=\,\,\hphantom{\cup\,\,} & V\disunion \left( \SET{(u,v)\in V\times V~\middle|~ u\not= v}\times \SET{1,2} \right)
	  \quad\txtand\\
	 A_D \,\,=\,\, \hphantom{\cup\,\,}& \SET{\left(u, \left( (u,v),1 \right) \right),\left( \left( (u,v),1 \right), v \right)
	  ~\middle|~ (u,v)\in A,\,u\not=v} 
	\\  \cup\,\, &\SET{\left( \left( (u,v),1 \right),\left( (u,v),2 \right) \right) ~\middle|~ u,v\in V,\,u\not= v}.
	\end{align*}
	In other words, for all $u,v\in V$ with $u\not=v$ we do the following in order to obtain $\mathrm{AC}(D)$ from $D$:
	 If there is an arc $(u,v)$ in $D$, 
	we add two new vertices and turn it into a top-left-to-bottom-right-oriented $\top$-shaped-junction. If there is no arc $(u,v)$ in $D$,
	we add two new vertices and connect one with the other.
\end{definition}

\needspace{6\baselineskip}
\vspace*{-\baselineskip} %Remove the line space created by the tilde below
\begin{wrapfigure}{l}{8.3cm}
\vspace{\baselineskip}
\begin{centering}~~
\includegraphics{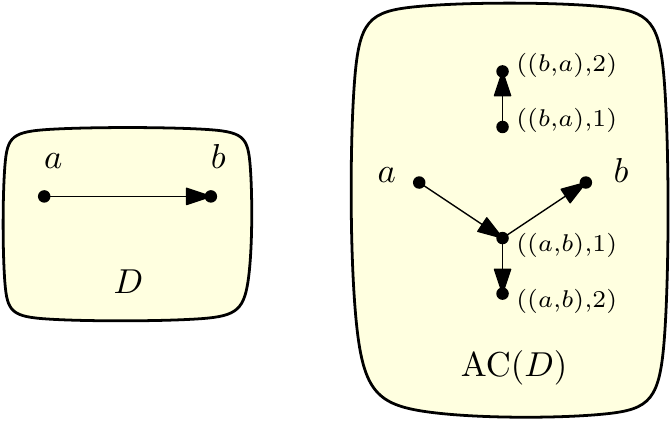}
\end{centering}%
\vspace*{-1\baselineskip}
\end{wrapfigure}
~ %The tilde creates a new dummy paragraph. WHY IS THAT NEEDED? -> would increase the space %
  % before the ex. environment. THE NEXT FREE LINE IS ESSENTIAL!
\begin{example}\PRFR{Apr 1st}
	Consider the digraph $D = \left( \SET{a,b},\SET{(a,b)} \right)$.
	Then $\mathrm{AC}(D)=(V_D,A_D)$
	 is the digraph where $V_D = \{a,$ $b,$ $\left( (a,b),1 \right),$ $\left( (a,b),2 \right),$ $\left( (b,a),1 \right),$
	$ \left( (b,a),2 \right)\}$ and where $A_D =$ $ \{ \left( a, \left( (a,b),1 \right)  \right),$ 
	 $\left(  \left( (a,b),1 \right), b \right),$ \hfill{~} \linebreak
	  $\hphantom{\{} \left( \left( (a,b),1 \right), \left( (a,b),2 \right) \right),$
	 \linebreak
	 $\hphantom{\{}\left(  \left( (b,a),1 \right), \left( (b,a),2 \right) \right)\}$.
\end{example}

\needspace{6\baselineskip}
\begin{definition}\label{def:ACDTE}\PRFR{Apr 1st}
	Let $(D,T,E)$ be a representation of a gammoid where $D=(V,A)$, and such that
	 $V\cap \left( \left( V\times V \right) \times \SET{1,2} \right) = \emptyset$. \label{n:ACDTE}
	 The \deftext[arc-cut matroid]{arc-cut matroid for $\bm(\bm D\bm, \bm T\bm, \bm E\bm)$}
	 shall be the matroid $\mathrm{AC}(D,T,E) = \Gamma(\mathrm{AC}(D),T',E')$
	 where $$E' = E \cup \SET{\left( (u,v), i \right) ~\middle|~ u,v\in V,\,u\not= v,\,i\in\SET{1,2}}$$
	 and where 
	 \[ T' = T\cup \SET{\left( (u,v), 2 \right) ~\middle|~ u,v\in V,\,u\not= v}. \qedhere\]
\end{definition}

\needspace{5\baselineskip}
\begin{lemma}\label{lem:indepACDTE}\PRFR{Apr 1st}
	Let $(D,T,E)$ be a representation of a gammoid where $D=(V,A)$, and such that
	 $V\cap \left( \left( V\times V \right) \times \SET{1,2} \right) = \emptyset$.
	 Then $X\subseteq E$ is independent with respect to $\Gamma(D,T,E)$, if and only if
	 \(X' = X\cup\SET{\left( (u,v),2 \right)~\middle|~ u,v\in V,\,u\not=v} \)
	 is independent with respect to $\mathrm{AC}(D,T,E)$.
\end{lemma}
\begin{proof}\PRFR{Apr 1st}
	Let $X$ be independent with respect to $M = \Gamma(D,T,E)$. There is a routing $R\colon X\routesto T$ in $D$.
	Thus we have a routing $R' = \SET{p'~\middle|~ p\in R} \cup \SET{\left( (u,v), 2 \right) ~\middle|~ u,v\in V,\,u\not= v}$ in $D'$
	where
	 $$p' = p_1\left( (p_1,p_2),1 \right)p_2\left( (p_2,p_3),1 \right) \ldots
	p_{n-1}\left( (p_{n-1},p_{n}),1 \right)p_n$$ denotes the path in $\mathrm{AC}(D)$ that is obtained from $p=(p_i)_{i=1}^n$ by
	subdividing every arc $(u,v)$ traversed by $p$ with $\left( (u,v),1 \right)$.
	Consequently, the derived set $X'$ is independent in $N = \mathrm{AC}(D,T,E)$.
	Now let $X$ be dependent in $M$, therefore there is no routing from $X$ to $T$ in $D$. Now assume that there is a routing $R'$
	from the derived set $X'$ to $T'=T\cup\SET{\left( (u,v),2 \right)~\middle|~ u,v\in V,\,u\not=v} $ in $\mathrm{AC}(D)$, i.e. that
	$X'$ is independent with respect to $N$. Then $R'$ routes every $x\in X$ to some element $t_x\in T'\BS \left( X' \BS X \right) = T$
	in $\mathrm{AC}(D)$. By omitting the subdivision vertices in the corresponding paths $p'\in R$, we obtain a routing from $X$ to
	$T$ in $D$ --- a contradiction. Therefore $X'$ is dependent in $N$ if $X$ is dependent in $M$.
\end{proof}

\begin{lemma}\PRFR{Apr 1st}
	Let $(D,T,E)$ be a representation of a gammoid where $D=(V,A)$, and such that
	 $V\cap \left( \left( V\times V \right) \times \SET{1,2} \right) = \emptyset$.
	 Furthermore, let $a\in A$. The arc $a$ is an essential arc of $(D,T,E)$
	if and only if there is a circuit $C\in \Ccal\left( \mathrm{AC}(D,T,E) \right)$ with 
	\[ \left( a,1 \right) \in C \subseteq E \cup \SET{\left( a,1 \right)} \cup \SET{\left( (u,v),2 \right)~\middle|~ u,v\in V,\,u\not= v,\,a\not=(u,v)} .\]
\end{lemma}
\begin{proof}\PRFR{Apr 1st}
First, let us assume that $a$ is an essential arc of $(D,T,E)$.
	Let $X\subseteq E$ be independent with respect to $\Gamma(D,T,E)$,
	 such that every routing $R\colon X\routesto T$ in $D$ traverses the arc $a$.
	Then every routing from $X$ to $T$ in $\mathrm{AC}(D)$ visits the vertex $(a,1)$.
	Therefore every routing from $X' = X\cup\SET{\left( (u,v),2 \right)~\middle|~ u,v\in V,\,u\not= v,\,a\not=(u,v)}$
	to $T'=T\cup\SET{\left( (u,v),2 \right)~\middle|~ u,v\in V,\,u\not=v} $  in $\mathrm{AC}(D)$ also
	has to visit the vertex $(a,1)$. This implies that $X'\cup\SET{(a,1)}$ must be dependent. From Lemma~\ref{lem:indepACDTE} we obtain that
	 $X'$ is independent 
	in $\mathrm{AC}(D,T,E)$, and consequently there is a circuit $C\subseteq X'\cup\SET{(a,1)}$ such that $\left( a,1 \right)\in C$.
	Now assume that $a$ is not an essential arc of $(D,T,E)$.
	Let $X\subseteq E$ be independent with respect to $\Gamma(D,T,E)$, then there is a routing $R\colon X\routesto T$ in $D$
	such that the arc $a$ is not traversed by $R$.
	Thus there is a routing $R'$ from $X' = X\cup\SET{\left( (u,v),2 \right)~\middle|~ u,v\in V,\,u\not= v,\,a\not=(u,v)}$ 
	to $T'$ in $\mathrm{AC}(D)$
	that does not visit the vertex $(a,1)$. It is clear from Definition~\ref{def:ACD} that such a routing $R'$ cannot visit $(a,2)$ either.
	Therefore $R'\cup\SET{(a,1)(a,2)}$ is a routing in $\mathrm{AC}(D)$ and $X'\cup\SET{(a,1)}$ is independent with
	respect to $\mathrm{AC}(D,T,E)$.  Consequently, if $ C \subseteq E \cup \SET{\left( a,1 \right)} \cup \SET{\left( (u,v),2 \right)~\middle|~ u,v\in V,\,u\not= v,\,a\not=(u,v)} $ is a circuit of $\mathrm{AC}(D,T,E)$, then $C\cap E$ is dependent,
	therefore $\left( a,1 \right) \notin C$.
\end{proof}

\noindent A.W.~Ingleton and M.J.~Piff showed the following nice theorem about representations of strict gammoids
where every arc is essential, which they call {\em minimal presentation of $\Gamma(D,T,V)$}.

\begin{theorem}[\cite{IP73}, Theorem~3.12]\label{thm:IPEssentialStars}\PRFR{Apr 1st}
	Let $(D,T,V)$ be a representation of a gammoid where $D=(V,A)$ and where all $a\in A$ are essential arcs of $(D,T,V)$,
	and let $u\in V\BS T$.
	Then 
	\[ S_u = \SET{v\in V~\middle|~ (u,v)\in A} \cup \SET{u} \in \Ccal(\Gamma(D,T,V)).\]
\end{theorem}

\noindent For a proof, see \cite{IP73} p.60.

\begin{corollary}\label{cor:arcEstimatesDTE}
	Let $D=(V,A)$ be a digraph, $T\subseteq V$, and $E\subseteq V$.
	Furthermore, let $M=\Gamma(D,T,E)$ and $N=\Gamma(D,T,V)$.
	Then \begin{align*}
	 	 \arcC(M) \leq \arcC(N)& \leq
		 \left| V\BS T \right| + \sum_{u\in V\BS T} \rk_N\left(\vphantom{A^A} \SET{v\in V~\middle|~(u,v)\in V} \right)
		\\& \leq \left( \left| V \right| - \rk_N(V) \right)\cdot \left( \rk_N(V) + 1 \right).
	\end{align*}
\end{corollary}
\begin{proof}
	Since $M$ is a minor of $N$, we have $\arcC(M) \leq \arcC(N)$ (Lemma~\ref{lem:arcCkVDualityAndMinors}).
	The last inequality follows from Lemma~\ref{lem:rankMonotone}.
	Let $D'=(V,A')$ be a digraph obtained from $D$ by successively removing one non-essential arc of $(D',T,V)$
	after another from $A'$ until every remaining arc $a\in A'$ is an essential arc of $(D',T,V)$. Let $u\in V\BS T$, then
	Theorem~\ref{thm:IPEssentialStars} yields that $S_u = \SET{v\in V~\middle|~ (u,v)\in A} \cup \SET{u} \in \Ccal(N)$,
	thus the process of removing non-essential arcs stops no sooner than
	when $O_u = \SET{v\in V~\middle|~ (u,v)\in A'}$ is independent 
	in $N$ for all $u\in V\BS T$. Clearly, no arc leaving a vertex $t\in T$ is essential for $(D',T,V)$. Thus
	$$\left| A' \right| = \sum_{u\in V\BS T} \rk_N\left(\vphantom{A^A} \SET{v\in V~\middle|~(u,v)\in V} \right)$$
	holds. We may obtain a standard representation of $N$ from $(D',T,V)$ by first renaming all $v\in V\BS T$
	to $v'$ and then adding a new source $v\in V\BS T$ and a new arc $(v,v')$ to $D'$. 
	Consequently, 
	\[ \arcC(N) \leq
		 \left| V\BS T \right| + \sum_{u\in V\BS T} \rk_N\left(\vphantom{A^A} \SET{v\in V~\middle|~(u,v)\in V} \right)
		 \leq \left| V\BS T \right| + \sum_{u\in V\BS T} \rk_N(V). \qedhere\]
\end{proof}

\noindent Corollary~\ref{cor:arcEstimatesDTE} together with Remark~\ref{rem:upperBoundForV} implies that
every gammoid $M=(E,\Ical)$ may be represented on a digraph with at most $k = \rk_M(E)^2 \cdot \left| E \right| + \rk_M(E) + \left| E \right|$ vertices and with at most $(k-\rk_M(E))\cdot \left( 1+\rk_M(E) \right)$ arcs. 

\begin{lemma}\label{lem:uniformStrictStdRep}\PRFR{Apr 1st}
	Let $r\in \N$, $U=(E,\Ical)$ be a uniform matroid with $r \leq \left| E \right|$,\linebreak i.e.
	$\Ical = \SET{\vphantom{A^A}X\subseteq E~\middle|~ \left| X \right| \leq r}$, and
	let $(D,T,E)$ with $D=(E,A)$ be a strict representation of $U$.
	Then $$\left| A \right| \geq r\cdot\left( \left| E \right| - r \right).$$
\end{lemma}
\begin{proof}\PRFR{Apr 1st}
	Without loss of generality we may assume that no digraph occurring in this proof contains a loop arc $(v,v)$.
	Let $(D,T,E)$ with $D=(E,A)$ be a strict representation of $U$ 
	with a minimal number of arcs among all such representations.
	Due to that minimality, every $t\in T$ is a sink of $D$, and every arc $a\in A$ is an essential arc.
	Observe that $\left| C \right| = r+1$ for all $C\in\Ccal(U)$.
	Thus we obtain from Theorem~\ref{thm:IPEssentialStars}
	that 
	\begin{align*}
	 \left| A \right| & = \left| \bigcup_{u\in E\BS T} \SET{(x,v)\in A~\middle|~ x=u}  \right| 
	 \\& = \sum_{u\in E\BS T}
	\left| \SET{(x,v) \in A~\middle|~ x=u}  \right| \\ 
		& = \sum_{u\in E\BS T} \left| \SET{v\in V~\middle|~ (u,v)\in A} \right| \\
		& \geq \left| E\BS T \right|\cdot \min \SET{\vphantom{A^A} \left| C \right| - 1 ~\middle|~ C\in\Ccal(U)} = \left( \left| E \right| - r \right) \cdot r.
		\end{align*}
	Thus, every strict representation of $U$ has at least $r\cdot \left( \left| E \right| - r \right)$ arcs.
	The strict standard representation constructed in Lemma~\ref{lem:uniformArcs} yields that this bound is attained.
\end{proof}
\noindent
	In general, strict gammoids that are not transversal matroids exist and such matroids cannot
	have a standard representation
	that is also a strict representation, because their duals do not have a strict representation.
%	Fortunately, the standard representation of $U$ given in Lemma~\ref{lem:uniformArcs} is a strict representation.

\needspace{4\baselineskip}
\begin{definition}\PRFR{Apr 1st}
	Let $(D,T,E)$ with $D=(V,A)$ be a representation of the gammoid $\Gamma(D,T,E)=(E,\Ical)$,
	and let $q\in V$ be a vertex of $D$.
	Then $q$ shall be called \deftext[essential vertex of DTE@essential vertex of $(D,T,E)$]{essential vertex of $\bm(\bm D\bm, \bm T\bm, \bm E\bm)$},
	if either $q\in E$ or
	if $q\in V\BS E$ and there is some $X\in \Ical$ such that $X$ is not independent with respect to $\Gamma(D_q,T,E)$ where
	\[ D_q = (V\BSET{q},\SET{(u,v)\in A~\middle|~u\not= q\txtand v\not=q}). \qedhere \]
\end{definition}

\needspace{6\baselineskip}

\vspace*{-\baselineskip} %Remove the line space created by the tilde below
\begin{wrapfigure}{r}{3cm}
\vspace{\baselineskip}
~~~\includegraphics{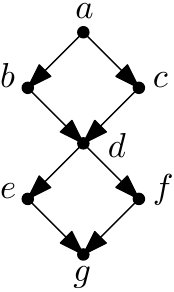}
\vspace*{-\baselineskip} %make the picture more tightly cropped
\end{wrapfigure}
~ %The tilde creates a new dummy paragraph. WHY IS THAT NEEDED? -> would increase the space %
  % before the ex. environment. THE NEXT FREE LINE IS ESSENTIAL!

\begin{remark}\PRFR{Apr 1st}
	Clearly, if $(u,v)$ is an essential arc of $(D,T,E)$, then $u$ and $v$ are essential vertices of $(D,T,E)$. 
	On the other hand, not every essential vertex $v$ of $(D,T,E)$ is incident with an essential arc of $(D,T,E)$.
	For instance, let\linebreak $D=(V,A)$ be the digraph where $V=\dSET{a,b,c,d,e,f,g}$ and
	\linebreak $A = \SET{(a,b),(a,c),(b,d),(c,d),(d,e),(d,f),(e,g),(f,g)}$.
	Then $d$ is an essential vertex of $(D,\SET{g},\SET{a})$, but $(D,\SET{g},\SET{a})$ has no essential arcs.
\end{remark}

\begin{lemma}\label{lem:vtxInMinReprs}\PRFR{Apr 1st}
	Let $M=(E,\Ical)$ be a gammoid, and let $(D,T,E)$ be a representation of $M$, and let  $D=(V,A)$.
	% such
	%that $\left| V \right| = \vK(M)$.
	Let $q\in V\BS E$ be an essential vertex of $(D,T,E)$,
	and let $N = \Gamma(D,T,V\BSET{q})$ and $N' = \Gamma(D_q,T,V\BSET{q})$
	where $$D_q = \left(\vphantom{A^A} V\BSET{q}, \SET{(u,v)\in A~\middle|~u\not= q \txtand v\not=q} \right).$$
	Then there is a circuit $C\in \Ccal(N')$ with $C\cap \SET{u \in V~\middle|~ (u,q)\in A} \not= \emptyset$
	such that \linebreak $C$ is independent in $N$.
\end{lemma}

\begin{proof}\PRFR{Apr 1st}
Without loss of generality we may assume
	that $(q,q)\notin A$, and we do induction on the number of arcs entering $q$ in $D$. 
	There has to be at least one arc $(u,q)\in A$ with $u\in V$,
	because there is a subset $X\subseteq E$, such that every routing $R\colon X\routesto T$
	visits the vertex $q\notin E$. So $q$ cannot be a source in $D$.
	Therefore, the base case of the induction is the case where precisely one arc $(u,q)$ enters $q$ in $D$.
	This arc is essential
	with respect to $(D,T,E)$,
	since it is traversed by every routing from $X$ to $T$ in $D$.
	Lemma~\ref{lem:essentialArcsC} yields a desired circuit $C$ with $u\in C$.
	If there is a non-essential arc entering $q$, then $\Gamma(D,T,E) = \Gamma(D',T,E)$
	where $D'=(V,A\BSET{(u,q)})$ for an arbitrarily chosen non-essential arc $(u,q)\in A$,
	and then the existence of $C$ follows by induction hypothesis on $(D',T,E)$.
	If all arcs entering $q$ are essential for $(D,T,E)$, then we pick an arbitrary choice $(u,q)\in A$,
	and throw away all other arcs entering $q$. Let $D'' = (V,A'')$
	 where $A'' = \SET{(x,y)\in A~\middle|~ y\not= q}\cup\SET{(u,q)}$,
	 and let $M' = \Gamma(D'',T,E)$. Clearly, every independent set of $M'$ is also independent in $M$,
	 and if we delete $q$ and all incident arcs from $D''$ we obtain $D_q$.
	 Furthermore, there is exactly one arc
	 entering $q$ in $D''$,
	 and therefore we obtain a circuit $C\in\Ccal(N')$ with $u\in C$ 
	 that is independent in $\Gamma(D'',T,V\BSET{q})$ -- and therefore independent in $N$ -- 
	 from the induction hypothesis applied to $(D'',T,E)$.
\end{proof}

% -*- root: ../thesis.tex -*-

\needspace{6\baselineskip}
\subsection{Digraphs as Black Boxes}

\PRFR{Mar 7th}
\noindent Lemma~\ref{lem:sourcesinkrepresentation} states that each gammoid $M$ may be represented by
a triple $(D,T,E)$, where $T\subseteq E$ is a base of $M$ and where $D=(V,A)$ is a digraph, such that
all $t\in T$ are sinks and all $e\in E\BS T$ are sources of $D$. Given such a representation, we may
disregard the structure of $D$. Instead, we may regard $D$ merely as a function, which assigns
to each pair $(X,S)$ --- where $X\subseteq E$ and where $S\subseteq T$ 
such that
$X\cap T \subseteq S$ holds --- the
minimal cardinality of an $X$-$S$-separator in $D$. Clearly, the value of this function with respect to $D$ equals the rank of $X$
with respect to the contraction $M\contract\left( V\BS T \right) \cup  S$, and therefore the function derived from $D$
does not depend on
the choice of the representation $(D,T,E)$ of $M$, it is already determined by $M$ alone.
In this section, we will elaborate this idea.

\begin{definition}\label{def:Mblackbox}\PRFR{Mar 7th}
	Let $M=(E,\Ical)$ be a matroid, $B\subseteq E$, $\rho\colon 2^E \times 2^B \maparrow \N$ a map.
	 The pair $(B,\rho)$\label{n:Brho} shall be called \deftext[M-black box@$M$-black box]{$\bm M$-black box},
	if $B$ is a base of $M$ and if for all $X\subseteq E$ and all $S\subseteq B$
	the equation
	\[ \rho(X,S) = \rk_{M\contract\left( E\BS B \right)\cup S} (X\BS (B\BS S)) \]
	is satisfied. If $B$ is clear from the context, we also denote the $M$-black box $(B,\rho)$ by $\rho$ alone.
\end{definition}

\noindent Clearly, for every $B\in\Bcal(M)$, there is a unique $M$-black box $(B,\rho)$.

\begin{definition}\label{def:digraphBlackbox}\PRFR{Mar 7th}
	Let $D=(V,A)$ be a digraph, and $X,Y\subseteq V$.
	The \deftext[black box for $(X,D,Y)$]{black box for $\bm (\bm X\bm,\bm D\bm,\bm Y\bm)$} shall be the map
	\[ \lambda_{(X,D,Y)} \colon 2^X \times 2^Y \maparrow \N \]
	where for all $S\subseteq X$ and all $T\subseteq Y$
	\[ \lambda_{(X,D,Y)}(S,T) = \min \SET{\left| C \right| ~\middle|~\vphantom{A^A} C\subseteq V \text{~s.t.~} C \text{~is an }S-T-\text{separator in~} D} .\]
	If $(X,D,Y)$ is clear from the context, we may denote $\lambda_{(X,D,Y)}$ by $\lambda$, too.
\end{definition}

\begin{definition}\PRFR{Mar 7th}
	Let $X,Y$ be finite sets with $X\cap Y = \emptyset$,
	and let
	 $\lambda\colon 2^X\times 2^Y\maparrow \N$ be a map.
	 Then $\lambda$ shall be called a \deftext[D-black box@$D$-black box]{$\bm D$-black box},
	if there is a digraph $D=(V,A)$ with $X\cup Y\subseteq V$ such that for all $X'\subseteq X$ and $Y'\subseteq Y$
	with $X'\cap Y \subseteq Y'$
	\[\lambda(X',Y') = \lambda_{(X,D,Y)}(X',Y').\] 
	%i.e. $\lambda$ equals the black box for $(X,D,Y)$. 
	In this case, we say that $\lambda$ is a $D$-black box represented by $(X,D,Y)$.
\end{definition}

\begin{corollary}\PRFR{Mar 7th}
	Let $M=(E,\Ical)$ be a matroid, $B\in\Bcal(M)$ a base of $M$, and let $(B,\rho)$ be the corresponding $M$-black box.
	 Then $M$ is a gammoid if and only if $\rho$ is a $D$-black box.
\end{corollary}
\begin{proof}\PRFR{Mar 7th}
	There is a standard representation $(D,T,E)$ of the gammoid $M$ such that $D=(V,A)$ and $T=B$ (Remark~\ref{rem:standardRepresentation}).
	Then $(E, D, T)$ 
	represents $\rho$ due to Menger's Theorem~\ref{thm:MengerGoering}, Definition~\ref{def:gammoid}, and the fact that 
	for all $T'\subseteq T$, we have the equality $\Gamma(D,T,E)\contract(E\BS T') = \Gamma(D',T\BS T',E\BS T')$ 
	where $D'=(V\BS T', A\cap\left( (V\BS T')\times  (V\BS T')\right)$ --- this is a special case of the construction used in the proof of
	Lemma~\ref{lem:contractionStrictGammoid}. Therefore every $M$-black box is a $D$-black box.
	Conversely, if the $M$-black box $\rho$ is represented by $(X,D,Y)$,
	then $M = \Gamma(D,Y,X)$ and so $M$ is a gammoid.
\end{proof}

\begin{definition}\label{def:cascadeDigraph}\PRFR{Mar 7th}
	Let $D=(V,A)$ be a digraph. Then $D$ shall be called \deftext{cascade digraph},
	if there is a partition $V_1\disunion V_2 \disunion \cdots \disunion V_k = V$
	such that $A\subseteq \bigcup_{i=1}^{k-1}\left(  V_i \times V_{i+1}  \right)$.
\end{definition}

\begin{definition}[\cite{M72}]\PRFR{Mar 7th}
	Let $M=(E,\Ical)$ be a matroid.
	Then $M$ is a \deftext{cascade},
	if there is a digraph $D=(V,A)$, such that there is a partition $V_1\disunion V_2 \disunion \cdots \disunion V_k = V$
	with  $A\subseteq \bigcup_{i=1}^{k-1} \left(  V_i \times V_{i+1}  \right)$, and such that
	\[ M = \Gamma(D,V_k,V_1). \qedhere\]
\end{definition}

\begin{proposition}[\cite{M72}, \cite{Ma70thesis}\textsuperscript{\ref{ftn:Ma70}}]\label{prop:cascadesNonDual}\PRFR{Mar 7th}
Let $\Ccal\Mcal$ be the class of all cascades. Then
\[ \SET{M^\ast ~\middle|~ M\in \Ccal\Mcal} \not\subseteq \Ccal\Mcal \quad\txtand\quad
   \SET{M\contract E' ~\middle|~ M=(E,\Ical)\in \Ccal\Mcal,\,E'\subseteq E} \not\subseteq \Ccal\Mcal.\]
In other words, the class of cascades is neither closed under taking duals nor under contraction.
\end{proposition}
\stepcounter{footnote}
\footnotetext[\thefootnote]{\label{ftn:Ma70}Unfortunately, we were not able to acquire a copy of J.H.~Mason's thesis from within Europe before the printing deadline. The thesis contains the proof that the dual of a certain cascade is not a cascade itself, which is only cited in \cite{M72}. It appears to be available at the Memorial Library,
UW Madison Theses Basement North
AWB M411 J655.}

\PRFR{Mar 7th}
\noindent Clearly, every cascade digraph is acyclic, and it is also clear that
 the \deftext{transitive triple digraph} $\left(\SET{x,y,z},\SET{(x,y),(y,z),(x,z)}\right)$ is an acyclic digraph but not a cascade digraph. But regarding $D$-black boxes, the class of cascade digraphs and the class of acyclic digraphs have the same expressiveness.

\needspace{6\baselineskip}
\begin{lemma}\label{lem:DblackBoxArcSubdivision}\PRFR{Mar 7th}
	Let $D=(V,A)$ be a digraph and let $X,Y\subseteq V$, and $a=(u,w) \in A$. Furthermore, let $v\notin V$ be a new element.
	Then % for all $X'\subseteq X$ and $Y'\subseteq Y$ with $X'\cap Y\subseteq Y'$ $$\lambda_{(X,D,Y)}(X',Y') = \lambda_{(X,D',Y)}(X',Y')$$ where
	$$\lambda_{(X,D,Y)} = \lambda_{(X,D',Y)}$$ where
	$$ D' = (V\disunion\SET{v}, A\BSET{a} \cup \SET{(u,v),(v,w)}) $$
	denotes the digraph obtained from $D$ by subdividing the arc $a$ with the new vertex $v$.
\end{lemma}
\begin{proof}\PRFR{Mar 7th} Clearly, $v\notin X\cup Y \subseteq V$.
	The statement of the lemma follows from the fact that for all $x\in X$ and $y\in Y$ there is an obvious bijection
	\[ \phi \colon \Pbf(D;x,y) \maparrow \Pbf(D';x,y), \, p\mapsto \begin{cases}[r]
				p & \quad \text{if~}a\notin \left| p \right|_A,\\
				qvr & \quad \text{otherwise,}
			\end{cases}.\]
	where $q = (p_1,p_2,\ldots, p_j)$ and $r = (p_j,p_{j+1},\ldots, p_n)$ for
	$p = (p_i)_{i=1}^n$ and $j\in \SET{1,2,\ldots,n}$ such that $p_j = u$, and consequently, $p_{j+1} = w$.
	Let $X'\subseteq X$ and $Y'\subseteq Y$. The map $\phi$ yields that
	 every $X'$-$Y'$-separator in $D$ is also an $X'$-$Y'$-separator in $D'$, as well as
	every $X'$-$Y'$-separator $S$ in $D'$ with $v\notin S$ is an $X'$-$Y'$-separator in $D$.
	Furthermore, if $S$ is an $X'$-$Y'$-separator in $D'$ with $v\in S$, then $S\BSET{v}\cup\SET{u}$
	is an $X'$-$Y'$-separator in $D$ of the same or less cardinality. Therefore $\lambda_{(X,D,Y)} = \lambda_{(X,D',Y)}$,
	since the values of those maps only depend on the cardinality of their respective minimal separators.
\end{proof}

\begin{lemma}\label{lem:acyclicToCascade}\PRFR{Mar 7th}
	Let $D=(V,A)$ be an acyclic digraph, and let $X,Y\subseteq V$.
	Then there is a cascade digraph $D'$ such that
	\[ \lambda_{(X,D,Y)} = \lambda_{(X,D',Y)}.\]
\end{lemma}
\begin{proof}\PRFR{Mar 7th}
	Without loss of generality we may assume that $X\cap Y =\emptyset$, since otherwise we could introduce a copy $v'$
	for every vertex $v\in X\cap V$ and add a single arc leaving $v$ and entering $v'$ to $D$, and then
	continue with $Y' = Y\BS X \cup \SET{v'~\middle|~v\in X\cap Y}$. Using similar constructions, 
	we may also assume without loss of generality, that $X$ consists
	of sources of $D$ and $Y$ consists of sinks of $D$, as well as that $D$ has no sources and no sinks in $V\BS\left( X\cup Y \right)$.
	Possibly renaming elements from $V$, we may further assume that $V\cap \left( A\times \N \right) = \emptyset$.
	Since $D$ is acyclic, there is a strict linear order\footnote{That is a binary relation, which is irreflexive, 
		antisymmetric, and transitive.} $\prec$ on $V$ such that $u \prec v$ holds for all $(u,v)\in A$.
	Let $D' = (V', A')$ be the digraph where
	 \begin{align*} V' = & V \disunion \SET{((u,v),i) \in A\times \N ~\middle|~ \vphantom{A^A}1 \leq i \leq \left| \SET{x\in V~\middle|~ u \prec x \prec v} \right| } \text{~and}\\
	 A' =& \hphantom{\cup}\,\,\SET{(u,v)\in A ~\middle|~\vphantom{A^A} \SET{x\in V~\middle|~ u \prec x \prec v} = \emptyset} \\
	 & \cup \SET{(u,((u,v),1)) ~\middle|~\vphantom{A^A}(u,v)\in A,\, \SET{x\in V~\middle|~ u \prec x \prec v} \not= \emptyset}\\
	 & \cup \SET{(((u,v),k),v) ~\middle|~\vphantom{A^A}(u,v)\in A,\, k = \left| \SET{x\in V~\middle|~ u \prec x \prec v}  \right| \not= 0}\\
	 & \cup \SET{(((u,v),k),((u,v),k+1)) ~\middle|~\vphantom{A^A}(u,v)\in A, k\in \N,\, 1 \leq k < \left| \SET{x\in V~\middle|~ u \prec x \prec v}  \right|}.
	 \end{align*}
	In words, every arc $(u,v)\in A$ is subdivided 
	by $k$ new vertices, where $k$ equals the number of vertices that the arc $(u,v)$ skips with
	respect to the strict linear order $\prec$ on $V$. For instance, if $(u,v)\in A$ and $u \prec x \prec y \prec v$ is a maximal chain
	connecting $u$ with $v$ in $\prec$, then the arc $(u,v)$ is subdivided by the new vertices $((u,v),1)$ and $((u,v),2)$.
    Repeated application of Lemma~\ref{lem:DblackBoxArcSubdivision} yields that
	\[ \lambda_{(X,D,Y)} = \lambda_{(X,D',Y)} .\]
	We define the map 
	\begin{align*} \phi\colon V' \maparrow &\SET{\vphantom{A^A}1,2,\ldots,\left| V\BS\left( X\cup Y \right) \right|+1},
	\\v'\mapsto &\begin{cases}[r]
					1 & \quad \text{if~}v'\in X,\\
					1 + \left| \SET{u\in V~\middle|~ u \prec v'} \right| & \quad \text{if~}v'\in V\BS X,\\
					1 + i + \left| \SET{x\in V~\middle|~ x \prec u} \right| & \quad \text{if~}v'=((u,v),i)\in A\times \N.\\
				\end{cases}
	\end{align*}
	Let $k = \left| V\BS\left( X\cup Y \right) \right|+1$, then there is a partition $V'_1,V'_2,\ldots,V'_k$ of $V'$
	with \[ V'_i = \SET{v'\in V'~\middle|~ \phi(v') = i} \] for all $i\in\SET{1,2,\ldots,k}$ that has the property that
	\[ A' \subseteq \bigcup_{i=1}^{k-1}\left(  V'_i \times V'_{i+1}  \right).\]
	Therefore, $D'$ is a cascade digraph where the set $X = V_1$ and the set $Y = V_k$.
\end{proof}

\needspace{6\baselineskip}
\begin{example}\PRFR{Mar 7th}
	The construction from Lemma~\ref{lem:acyclicToCascade} applied to
	\begin{center}
	\includegraphics{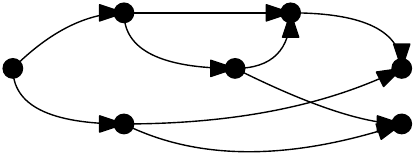}
	\end{center}
	yields
	\begin{center}
	\includegraphics{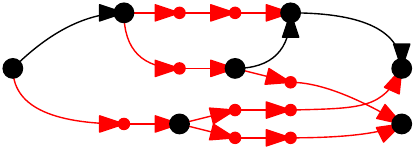},
	\end{center}
	where the vertices and arcs that do not belong to the original digraph are depicted red.
\end{example}

\needspace{4\baselineskip}
\begin{corollary}\label{cor:acyclicDigraphIsCascade}\PRFR{Mar 7th} Let $D=(V,A)$ be an acyclic digraph, $E,T\subseteq V$. Then there is a cascade digraph $D'=(V',A')$
with a partition $V'_1 \disunion V'_2\disunion\cdots\disunion V'_k = V'$ such that $A'\subseteq \bigcup_{i=1}^{k-1}\left(  V'_i\times V'_{i+1}  \right)$
and such that
	\[\Gamma(D,T,E) = \Gamma(D',V'_k, V'_1).\]
Every gammoid that can be represented using an acyclic digraph is a cascade.
\end{corollary}
\begin{proof}
	Direct consequence of the proof of Lemma~\ref{lem:acyclicToCascade}.
\end{proof}

\begin{remark}\label{rem:weNeedCycles}\PRFR{Mar 7th}
	Corollary~\ref{cor:acyclicDigraphIsCascade}, together with Proposition~\ref{prop:cascadesNonDual} stating that cascades are not closed under duality, implies that cycle walks are inevitable in representations of some gammoids, since the class of gammoids is closed under duality (Lemma~\ref{lem:dualityrespectingrepresentation}).
\end{remark}

\clearpage
% -*- root: ../thesis.tex -*-

\section{Strict Gammoids}\label{sec:StrictGammoids}

\begin{definition}\PRFR{Jan 22nd}
	Let $M=(E,\Ical)$ be a matroid. $M$ is a \deftext{strict gammoid} if there is a digraph $D=(V,A)$ and a set $T\subseteq V$ such that\label{n:GDTV} \( M = \Gamma(D,T,V) \).\footnote{Note that this implies that $V=E$, so $D=(E,A)$ is a digraph where the ground set of $M$ is the vertex set of $D$.}
\end{definition}

\PRFR{Jan 22nd}
\noindent It is clear from this definition, that every gammoid is a deletion-minor of a strict gammoid, as $\Gamma(D,T,E)$ for $D=(V,A)$ is a deletion-minor of $\Gamma(D,T,V)$.
Given a strict gammoid representation $(D,T,V)$, then $T\subseteq V$ is a base of $M$ and $\rk_M(V) = \left| T \right|$.
%
%\noindent
The following characterization of the rank function of a strict gammoid was given by C.~McDiarmid in \cite{McD72}, where it is used in order to proof that gammoids are indeed matroids.

\begin{theorem}\PRFR{Jan 22nd}
	Let $D=(V,A)$ and $T\subseteq V$. Let $M=\Gamma(D,T,V)$ be the strict gammoid represented by $(D,T,V)$. Then for $X\subseteq V$,
	\[ \rk(X) = \min_{U\subseteq V\BS T} \left(  \left| X\BS \Dcl{U} \right| + \left| \partial U \right|  \right). \]
\end{theorem}

\begin{proof}\PRFR{Jan 22nd}
	Let $X\subseteq V$ and $U\subseteq V\BS T$.
	Trivially, $\partial U$ separates $X\cap \Dcl{U}$ from $T$ %$\partial U$
	 in $D$.
	Let $S = \partial U \cup \left( X\BS \Dcl{U} \right)$, then $S$ is an
	$X$-$T$-separator in $D$: let $x\in X \BS S$, then $x\in \Dcl{U}$. Since $U\cap T=\emptyset$, every path $p\in\Pbf(D)$ with $p_1 = x$ and $p_{-1}\in T$ must leave the
	 set $U$ at some point. Consequently, it must visit a vertex from $\partial U$, 
	 so $\left| p \right|\cap S\not=\emptyset$. It follows that
	 \[ \rk(X) \leq \left| S \right| = \left| \partial U \cup \left( X\BS \Dcl{U} \right) \right|
	 \leq \left| X\BS \Dcl{U} \right| + \left| \partial U \right| ,\]
	 i.e. that the right-hand side of the equation in the lemma is an upper bound for the left-hand side of that equation.

	 \needspace{3\baselineskip}
   
	 \noindent
	 Now let $S\subseteq V$ be an $X$-$T$-separator in $D$ 
	 with $\left| S \right| = \rk(X)$, its existence is guaranteed by 
	 Menger's Theorem~\ref{thm:MengerGoering}. 
	 % -*- root: ../../thesis.tex -*-
\begin{figure}[!t]
	\begin{center}
	\includegraphics[scale=1.1]{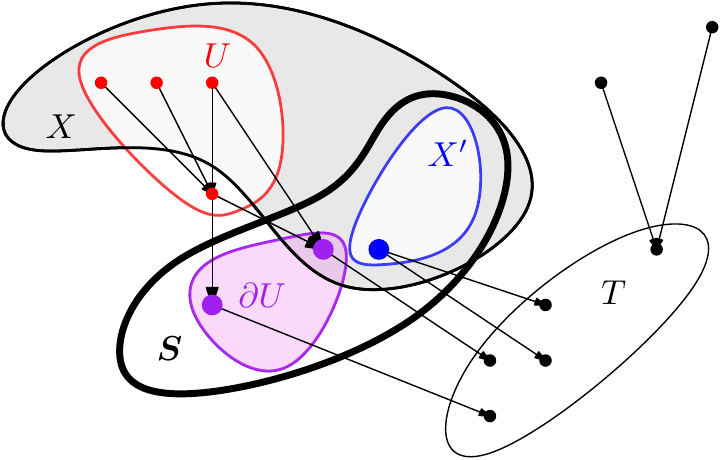}
	\end{center}
	\caption{Construction of $U$ from an $X$-$T$-separator $S$.}
\end{figure}	 
	 Let
	 \[ U = \SET{p_{-1} \in V ~\middle|~\vphantom{A^A}  p\in \Pbf(D)\colon\, p_1\in X \txtand \left| p \right|\cap S=\emptyset}\]
	 denote the set of vertices, that can be reached from any vertex $x\in X$ by a path not visiting $S$. 
	 Clearly, $U\subseteq V\BS T$ holds because $S$ is an $X$-$T$-separator in $D$, and the outer margin
	 $\partial U$ is a subset of $S$ whereas $S\cap U = \emptyset$ by construction. Let $S' = S\cap \partial U$ and let
	 $X' = S\BS \Dcl{U}$. $X'$ is indeed a subset of $X$: Let $s\in S\BS X$,
	 since $S$ is a minimal $X$-$T$-separator, every maximal $X$-$T$-connector $R$ has
	 a path $p\in R$ with $S\cap \left| p \right|=\SET{s}$.
	 We obtain that $s = p_k$ for some $k\in\N$,
	 and since $s\notin X$ but $p_1\in X$, we have $k > 1$. Therefore $p_{k-1}\in U$ and
	 then $s\in \partial U \subseteq \Dcl{U}$, so $s\notin X'$ --- this establishes $X'\subseteq X$. This yields $X' = X\cap X' = X\cap \left( S \BS \Dcl{U} \right) = \left( X\cap S \right)\BS \Dcl{U}$, and since $X\BS S\subseteq U \subseteq \Dcl{U}$, we have $X' = X\BS \Dcl{U}$.
	 Since $S\cap \Dcl{U} = S\cap\left( U\cup \partial U \right) = \left( S\cap U \right)\cup \left( S\cap\partial U \right) = S'$, we have
	 \( \left| S \right| = \left| X' \right| + \left| S' \right|
	 	 \).
	 Now assume that $S'\subsetneq \partial U$, then there is a vertex $u\in\partial U$
	 with $u\notin S$, such that there is a path $p\in \Pbf(D)$ with $p_1\in X$, $\left| p \right|\cap S =\emptyset$, and there is an arc $(p_{-1},u)\in A$. But then we obtain that $u\in U$, and therefore $u\notin \partial U$: 
	 Since $pu\in \Pbf(D)$ is a path with $\left| pu \right| = \left| p \right|\cup\SET{u}$,
	 and since $u\notin S$, we have $\left| pu \right|\cap S = \emptyset$, and so $u$ qualifies as a member of $U$.
	 Therefore $u\notin S$ cannot be the case and so $S' = \partial U$ holds.
	 Thus we obtain
	 \[ \rk(X) = \left| S \right| = \left| X' \right| + \left| S' \right| = \left| X\BS \Dcl{U}  \right| + \left| \partial U \right| \]
	 and therefore, on the right-hand side of the equation in the lemma, the minimum expression ranges over an upper bound that is equal to
	 $\rk(X)$, 
	 therefore both sides of the equation must be equal.
\end{proof}

\PRFR{Jan 22nd}
\noindent
J.H.~Mason gives a necessary and sufficient condition for when a matroid $M$ is a strict
gammoid in \cite{M72}. In order to present the proof, we need the Lemma~2.1 
from \cite{M72}. Here, we give a slightly more detailed version of J.H.~Mason's proof.
But first, we want to introduce the following notion of a special $X$-$T$-separator.

\begin{definition}\label{def:rightmost-separator}\PRFR{Jan 22nd}
	Let $D=(V,A)$ be a digraph, and let $X\subseteq V$ and $T\subseteq V$ be sets of vertices.
	The \deftext[barrier between X and T in D@barrier between $X$ and $T$ in $D$]{barrier between $\bm X$ and $\bm T$ in $\bm D$}
	is defined to be the set\label{n:barrier}
	\[ \delta_D(X,T) = \SET{x\in X~\middle|~\vphantom{A^A} \left( \partial_D \SET{x}\right) \cap \left( V\BS X \right) \not=\emptyset} \,\cup\, \left( X\cap T \right). \qedhere\]
\end{definition}

\begin{lemma}\label{lem:barrier}\PRFR{Jan 22nd}
	Let $D=(V,A)$ be a digraph,  and let $X\subseteq V$ and $T\subseteq V$ be sets of vertices.
	Then the barrier $\delta_D(X,T)$ is an $X$-$T$-separator in $D$.
\end{lemma}
\begin{proof}\PRFR{Jan 22nd}
 	Let $R$ be an $X$-$T$-connector, and let 
	$p=(p_i)_{i=1}^k \in R$ be a path that does not end in a vertex from $X\cap T$. Then 
	there is a maximal integer $1\leq j < k$ such that $p_j\in X$. Then $p_{j+1}\notin X$, 
	yet $(p_j,p_{j+1})\in A$, thus $p_{j+1}\in \partial\SET{p_j}$ and so $p_j\in \delta_D(X,T)$.
	 If otherwise $p\in R$ is a path that ends in $X\cap T$ we clearly have $p_{-1}\in \delta_D(X,T)$ --- thus $\delta_D(X,T)$
	  is an $X$-$T$-separator.
\end{proof}

\begin{lemma}\label{lem:minFTSeparatorInFStrictGammoid}\PRFR{Jan 22nd}
	Let $D=(V,A)$ be a digraph, $T\subseteq V$, $F\in \Fcal(\Gamma(D,T,V))$, and let $S\subseteq V$ be an $F$-$T$-separator in $D$
	with minimal cardinality. Then $S\subseteq F$.
\end{lemma}
\begin{proof}\PRFR{Jan 22nd}
	Let $M= \Gamma(D,T,V)$,
	and let $R\colon B_F\routesto T$ be a maximal $F$-$T$-connector in $D$ where
	$B_F$ is a base of $F$ in $M$.
	Then every $s\in S$ is visited by a path $p\in R$ (Corollary~\ref{cor:Menger}), thus any path
	$q\in \Pbf(D)\BS R$ with $q_1\in S$ has the property that $\SET{s}\subseteq \left| q \right|\cap \left| p \right|$,
	so $R\cup\SET{q}$ can never be a routing in $D$.
	 Therefore $R$ cannot be extended by a path starting in $s$, so $R$ is also a maximal $F\cup S$-$T$-connector in $D$.
	  Therefore $\rk(F) = \rk(F\cup S)$, so $S\subseteq \cl(F) = F$.
\end{proof}

\needspace{6\baselineskip}

\begin{lemma}\label{lem:flatofstrictgammoidisstrictgammoid}\PRFR{Jan 22nd}
	Let $M=(E,\Ical)$ be a strict gammoid, and $F\in\Fcal(M)$ be a flat of $M$.
	Then the restriction $M\restrict F$ is a strict gammoid.
	Furthermore, if $D=(E,A)$ and $T\subseteq E$ with $M=\Gamma(D,T,E)$, then
	the barrier $\delta_D(F,T)$ is an $F$-$T$-separator of minimal cardinality in $D$.
\end{lemma}

\begin{proof}\PRFR{Jan 22nd}
	Let $M=\Gamma(D,T,E)$ for suitable $D=(E,A)$ and $T\subseteq E$.
	\begin{figure}[t]
	\begin{center}
	\includegraphics[scale=1.2]{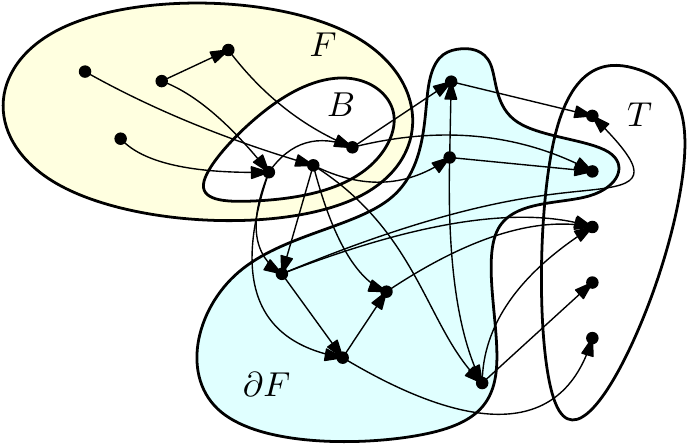}
	\end{center}
	\caption{Situation of the ``right-most'' $F$-$T$-separator $B$ in $D$.\label{fig:FTBD}}
	\end{figure}
	Now let
	\[ B = \delta_D(F,T) = \SET{f\in F~\middle|~ \vphantom{A^A}\left( \partial \SET{f} \right) \cap \left(E \BS F\right) \not= \emptyset} \cup \left( F\cap T \right) \]
	be the barrier between $F$ and $T$ in $D$ (Fig.~\ref{fig:FTBD}), i.e. the set
	which consists of those $f\in F$, that are either targets of the representation $(D,T,E)$, or that
	have an out-arc which leaves the flat $F$. Clearly, $B$ is an $F$-$T$-separator in $D$ (Lemma~\ref{lem:barrier}).
	We give an indirect argument that $B$ is a
	minimal $F$-$T$-separator in $D$:
	Assume that $B$ is not a minimal $F$-$T$-separator, then there is a set $S\subseteq F$, which is
	a minimal $F$-$T$-separator (Lemma~\ref{lem:minFTSeparatorInFStrictGammoid}), and there is an element $b\in B\BS S$, since $\left| B \right| > \left| S \right|$.
	 Clearly, $b\notin T$ since $F\cap T$ is a subset of every $F$-$T$-separator.
	  Further, there is an element $e\in E\BS F$ such that $(b,e)\in A$ according to the definition of $B$, 
	  and since $e\notin F=\cl_M(F)$, 
	  there is a maximal $F$-$T$-connector $R$ and a path $p\in \Pbf(D)$ with $p_1 = e$ 
	  and $p_{-1}\in T$, such that $R\cup\SET{p}$ is a routing in $D$. 
	  Thus the path $p$ does not visit any vertex that belongs to a minimal $F$-$T$-separator in $D$, 
	  and therefore the path $bp$ does not visit any 
	  vertex of $S$, too --- which contradicts the assumption that $S$ is an $F$-$T$-separator, 
	  thus $B$ must be a minimal $F$-$T$-separator.

	 \bSep \PRFR{Jan 22nd}
	%\noindent
	Let $D'=\left( F,A\cap\left( F\times F \right) \right)$ be the restriction of $D$ to $F$, and let $M' = \Gamma(D',B,F)$ be the strict gammoid presented by the 
	restriction of $D$ and the target set $B$.
	Let $R$ be a routing from $X_0\subseteq F$ to $T$ in $D$,
	then every path $p=(p_i)_{i=1}^k \in R$ has a smallest integer $1\leq j(p) \leq k$,
	such that $p_{j(p)} \in B$. By construction of $B$, we have that $\SET{p_1,p_2,\ldots,p_{j(p)}} \subseteq F$. Thus $R$ induces a routing $R'=\SET{p_1 p_2 \ldots p_{j(p)} ~\middle|~p\in R}$ in $D'$ which routes $X_0$ to $B$. So $\rk_{M'}(X_0) \geq \rk_M(X_0)$.
	We give an indirect argument that the inequality $\rk_{M'}(X_0) \leq \rk_M(X_0)$ holds, too. 
	Let $S\subseteq E$ be a minimal $X_0$-$T$-separator in $D$, then we have $S\subseteq \cl_M(X_0) \subseteq F$ (Lemma~\ref{lem:minFTSeparatorInFStrictGammoid}). 
	Assume that there is a routing $R'$ from $X_0$ to $B$ in $D'$, such that $\left| R' \right| > \left| S \right|$. 
	Then there must be some $p\in R'$ such that $\left| p \right|\cap S = \emptyset$ and $p_{-1}\in B$. If $p_{-1}\in T$, then $p$ 
	is a contradiction to $S$ being an $X_0$-$T$-separator in $D$. If otherwise $p_{-1}\in B\BS T$, we have again the situation 
	where there is some $e\in E\BS F$ with $(p_{-1},e)\in A$, such that there is a path $q\in \Pbf(D)$ with $q_1=e$ and
	 $q_{-1}\in T$, that avoids every $F$-$T$-separator. So, consequently, $\left| q \right|\cap F=\emptyset$.
	The path $pq\in \Pbf(D)$ contradicts $S$ being an $X_0$-$T$-separator in $D$. 
	Therefore $\left| R' \right| \leq \left| S \right|$, and
	we just proved, that for any $X\subseteq F$ the equation $\rk_M(X) = \rk_{M'}(X)$ holds. 
	Thus $M\restrict F = \Gamma(D',B,F)$ is a strict gammoid.
\end{proof}

\needspace{4\baselineskip}

\begin{corollary}\label{cor:gammoidrestrictionsamerank}\PRFR{Jan 22nd}
	Let $M=(E,\Ical)$ be a gammoid. Then there is a strict gammoid $M'=(V,\Ical')$ such that
	$\rk_M(E) = \rk_{M'}(V)$ and $M = M'\restrict E$.
\end{corollary}
\begin{proof} \PRFR{Jan 22nd}
	Let $(D,T,E)$ be a representation of $M$, where $D=(V,A)$. Let \linebreak $M_0=\Gamma(D,T,V)$ be the strict gammoid arising naturally from the representation of $M$, and let $F=\cl_{M_0}(E)$ be the smallest flat in $M_0$ that contains $E$. Then $\rk_{M_0}(F) = \rk_{M_0}(E) = \rk_{M}(E)$ since $M= M_0\restrict E$. Now, let $M'= M_0\restrict F$. 
	$M'$ is a strict gammoid (Lemma~\ref{lem:flatofstrictgammoidisstrictgammoid}), and since $E\subseteq F$, we have $M = M_0 \restrict E = M'\restrict E$, thus $M$ is the
	restriction of a strict gammoid of the same rank.
\end{proof}

\begin{lemma}\label{lem:contractionStrictGammoid}\PRFR{Jan 26th}
	Let $M=(E,\Ical)$ be a strict gammoid, $C\subseteq E$.
	 Then $M \contract C$ is a strict gammoid.
\end{lemma}
\begin{proof}\PRFR{Jan 26th}
	Let $B_0$ be a base of $E\BS C$ in $M$, and let $B$ be a base of $M$ with $B_0\subseteq B$ 
	(Lemma~\ref{lem:augmentation}). Let further $D=(E,A)$ be a digraph, such that $M=\Gamma(D,B,E)$
	and such that $B$ consists only of sinks in $D$
	 (Theorem~\ref{thm:gammoidRepresentationWithBaseTerminals}). We denote the family of independent sets of $M\contract C$ by $\Ical'$.
	Then for every $X\subseteq C$, we have $X\in \Ical'$ if and only if \linebreak $X\disunion B_0 \in \Ical$ 
	(Lemma~\ref{lem:contractionBchoice}). But $X\disunion B_0\in \Ical$
	 if and only if there is a routing \linebreak
	$R\colon X\disunion B_0 \routesto B$ in $D$. Since $B_0\subseteq B$ consists of sinks in $D$,
	for every $b_0\in B_0$, the trivial path $b_0\in\Pbf(D)$ is a member of $R$.
	We give an indirect argument, that for every $e\in \left( E\BS C \right)\BS B_0$ 
	and every $p\in R$, $e\notin \left| p \right|$ holds: If there would be such a path $p=(p_i)_{i=1}^n\in R$, 
	then for some $j\in \N$, $p_j = e$. But then the path $q = p_j p_{j+1}\ldots p_n \in \Pbf(D)$ yields 
	a routing $\SET{q}\cup\SET{b_0\in \Pbf(D)\mid b_0\in B_0}$ which implies that
	$\SET{e}\cup B_0\in \Ical$ -- a contradiction to the maximality of the 
	base $B_0$ of $E\BS C$ in $M$. Thus the routing $R'=\SET{p\in R\mid p_1\notin B_0}$ routes $X$ to $B\BS B_0$ 
	in $D'=(C,A\cap \left( C\times C \right))$, the sub-digraph of $D$ induced by $C$.
	Conversely, every routing $S\colon Y\routesto B\BS B_0$ in $D'$ induces the
	routing $S\cup\SET{b_0\in \Pbf(D)\mid b_0\in B_0}$ from $Y\disunion B_0$ to $B$ in $D$,
	so $M\contract C = \Gamma(D',B\BS B_0,C)$: the contraction is again a strict gammoid.
\end{proof}

% -*- root: ../thesis.tex -*-

\subsection{Mason's $\alpha$-Criterion}

\PRFR{Jan 30th}
\noindent In the proof of Lemma~\ref{lem:flatofstrictgammoidisstrictgammoid}
we have seen that the elements of a flat $F$ of a strict gammoid $M=\Gamma(D,T,V)$ fall into two disjoint categories: for some $f\in F$, we have $\partial \SET{f} \subseteq F$,
and for an independent subset $I\subseteq F$, we have $\partial \SET{i} \not\subseteq F$
for all $i\in I$ -- more precisely, there is a base $B$ of $F$ such that $I=B\BS T$. Before we present Mason's criterion, we need one last definition.

\begin{notation}\PRFR{Feb 15th}
	Let $M=(E,\Ical)$ be a matroid and $X\subseteq E$. The family of those flats of $M$, which are proper subsets of $X$, shall
	be denoted by\label{n:FcalMX}
	\[ \Fcal(M,X) = \SET{F\in\Fcal(M)~\middle|~ F\subsetneq X}. \qedhere\]
\end{notation}

\needspace{6\baselineskip}
\begin{definition}\label{def:alphaM}\PRFR{Jan 30th}
	Let $M=(E,\Ical)$ be a matroid. The \deftext[a-invariant of M@$\alpha$-invariant of $M$]{$\bm \alpha$-invariant of $\bm M$} shall be the map\label{n:alphaM}
	\[ \alpha_M\colon 2^E \maparrow \Z \]
	that is uniquely characterized by the recurrence relation
	\[ \alpha_M(X) = \left| X \right| - \rk_M(X) - \sum_{F\in \Fcal(M,X)} \alpha_M(F).\]
	If the matroid $M$ is clear from the context, we also write $\alpha(X)$ for $\alpha_M(X)$.
\end{definition}

\begin{remark}\label{ref:AlphaIndependent}\PRFR{Jan 30th}
Clearly, $\alpha(\emptyset) = 0$ for any matroid $M$. Furthermore,
the value $\alpha(X)$ for $X\subseteq E$ may be calculated from the values $\alpha(X')$
corresponding to proper subsets $X'\subsetneq X$ and the rank of $X$, so $\alpha$ is well-defined.
\end{remark}

\PRFR{Jan 30th}
\noindent Just like it is the case for the rank function, the family of bases, and the family of circuits of a matroid,
 we can use $\alpha_M$ to reconstruct the matroid $M$; thus $M$ is already uniquely determined by $\alpha_M$.

\begin{definition}\label{def:FcalAlpha}\PRFR{Jan 30th}
	Let $E$ be a finite set and let $\alpha\colon 2^E \maparrow \Z$ be a map. 
	The \deftext[zero-family of a@zero-family of $\alpha$]{zero-family of $\bm \alpha$} \label{n:Ialpha}
	shall be
	\[ \Ical_\alpha = \SET{X\subseteq E ~\middle|~\vphantom{A^A} \forall Y\subseteq X\colon\, \alpha(Y) = 0}.\]
	The family of \deftext[a-flats@$\alpha$-flats]{$\bm \alpha$-flats}
	shall be defined as\label{n:FcalAlpha}
	 \[ \Fcal(\alpha) = \SET{F\subseteq E ~\middle|~\vphantom{A^A} \forall e\in E\BS F,\,X\subseteq F\colon\,\,\, X\in \Ical_\alpha \txtand \SET{e}\in\Ical_\alpha \Rightarrow X\cup\SET{e}\in\Ical_\alpha}. \]
	Furthermore, we define the pair \label{n:MAlpha}
	\[ M(\alpha) = (E,\Ical_\alpha). \qedhere\]
	% Clearly $E\in\Fcal(\alpha)$, so we may define the \deftext[a-zero flat@$\alpha$-zero flat]{$\bm \alpha$-zero flat}
	% to be\label{n:Oalpha}
	% \[ O_\alpha = \bigcap \Fcal(\alpha) .\]
\end{definition}

\begin{lemma}\label{lem:alphaIndependent}\PRFR{Jan 30th}
	Let $M=(E,\Ical)$ be a matroid and let $\alpha=\alpha_M$ be its $\alpha$-invariant.
	Then $\Ical=\Ical_\alpha$, $\alpha(X) = 0$ for all $X\in \Ical$, and $\alpha(C) = 1$ for all $C\in\Ccal(M)$.
%	For all $X\subseteq E$ we have
%	\[ X\in \Ical \quad\Longleftrightarrow\quad \forall Y \subsetneq X\colon\,\,\, Y\cup O_\alpha \in \Fcal(\alpha).\]
\end{lemma}

%\noindent In other words, we can reconstruct the independent sets of a matroid by examining its $\alpha$-
%alone.

\begin{proof}\PRFR{Jan 30th}
	Let $X\in\Ical$, we show $\alpha(X) = 0$ by induction on $\left| X \right|$.
	In the base case, we have $$\alpha(\emptyset) = \left| \emptyset \right| - \rk(\emptyset) = 0 - 0 = 0.$$
	For the induction step, we may assume by induction hypothesis that for all $Y\subsetneq X$ the equality $\alpha(Y) = 0$ holds.
	Thus
	$$ \alpha(X) = \left| X \right| - \rk(X) - \sum_{F\in\Fcal(M,X)} \alpha(F) = \left| X \right| - \left| X \right| - \sum_{F\in\Fcal(M,X)} 0 = 0.$$
	Therefore we obtain $\Ical \subseteq \Ical_\alpha$.
	Now let $X\subseteq E$ with $X\notin \Ical$. Then there is a circuit $C\in\Ccal(M)$ such that $C\subseteq X$.
	For all $D\subsetneq C$, we have $D\in \Ical$, therefore $\alpha(D) = 0$.
	So clearly
	$$ \alpha(C) = \left| C \right| - \rk(C) - \sum_{F\in\Fcal(M,C)} \alpha(F) = \left| C \right| - \left( \left| C \right| - 1 \right) - \sum_{F\in\Fcal(M,C)} 0 = 1,$$
	which implies $X\notin \Ical_\alpha$, and we obtain that $\Ical = \Ical_\alpha$.
\end{proof}

\begin{corollary}\PRFR{Jan 30th}
	Let $M=(E,\Ical)$ and $N=(E,\Ical')$ be two matroids defined on the same ground set $E$.
	Then $M = N$ if and only if $\alpha_M = \alpha_N$.
\end{corollary}

\begin{remark}\PRFR{Jan 30th}
	We may express the rank of a matroid $M=(E,\Ical)$ in terms of the $\alpha$-invariant of $M$ in different ways.
	Let $X\subseteq E$, then
	\begin{align*}
		 \rk(X) & = \left| X \right| - \sum_{F\in\Fcal(M),\,F\subseteq X} \alpha(F) \\
		 		& = \max \SET{\left| I \right|\vphantom{A^A} ~\middle|~ I\subseteq X,\,\forall J\subseteq I\colon\, \alpha(J) = 0} \\
		 		& = \max \SET{\left| I \right|\vphantom{A^A} ~\middle|~ I\subseteq X,\,I\in \Ical_\alpha }.
	\end{align*}
	We may use this equation in order to give an axiomatization of matroids in terms of its $\alpha$-invariant,
	which admittedly appears to be not very helpful.
	\begin{enumerate}
		\item[\em (A1)] $\alpha(\emptyset) = 0$.
		\item[\em (A2)] For all $X,Y\subseteq E$ with $\left| X \right| < \left| Y \right|$ and for which
			the restrictions
			 $\alpha\restrict_{2^X}$ and $\alpha\restrict_{2^Y}$ are constantly zero, there is an element
			 $y\in Y\BS X$, such that $\alpha\restrict_{2^{X'}}$ is constantly zero, where $X' = X\cup\SET{y}$.
		\item[\em (A3)] For all $X\subseteq E$ 
			\[ \alpha(X) = \left| X \right| - \max \SET{\left| I \right| \vphantom{A^A}~\middle|~ I\subseteq X,\,I\in\Ical_\alpha} - \sum_{F\in \Fcal(\alpha),\,F\subsetneq X} \alpha(F) .\]
	\end{enumerate}
	Clearly, {\em (A2)} resembles the augmentation axiom {\em (I3)} for $\Ical_\alpha$ and {\em (A1)} 
	guarantees that $\emptyset\in\Ical_\alpha$. {\em (I2)} trivially holds for $\Ical_\alpha$  by construction,
	and so $M(\alpha) = (E,\Ical_\alpha)$ is a matroid. Then {\em (A3)} guarantees that $\alpha = \alpha_{M(\alpha)}$,
	i.e. that $\alpha$ behaves like the $\alpha$-invariant for $M(\alpha)$ on the dependent sets.
\end{remark}

\needspace{6\baselineskip}
\begin{theorem}[\cite{M72}, Theorem~2.2]\label{thm:AlphaCriterion}\PRFR{Jan 30th}
	Let $D=(V,A)$, $T\subseteq V$, and $M=\Gamma(D,T,V)$ be a strict gammoid.
	Then for all $X\subseteq V$, we have $\alpha(X) \geq 0$.
	Furthermore, if $F\in\Fcal(M)$ % and all $t\in T$ are sinks in $D$, 
	then $\alpha(F)$ is the number of elements of $F\BS T$
	with the property,
	that $\partial \SET{f}$ is a subset of $F$ but not of any proper sub-flat $F'\subsetneq F$.
\end{theorem}

 \noindent We present a slightly polished version of the proof in \cite{M72}.

\begin{proof}\PRFR{Jan 30th}
	For every $X\subseteq E$ we define the subsets
	\begin{align*}
	 X_1 & = \SET{x\in X\BS T~\middle|~\vphantom{A^A}\partial \SET{x} \not\subseteq X} \cup \left( X\cap T \right) = \delta_D(X,T),\\
	 X_2 & = \SET{x\in X\BS T~\middle|~\vphantom{A^A} \partial \SET{x} \subseteq X\txtand \forall F\in\Fcal(M)\colon\,F\subsetneq X \Rightarrow \partial \SET{x} \not \subseteq F},\text{ and} \\
	 X_3 & = \SET{x\in X\BS T~\middle|~\vphantom{A^A} \exists F\in\Fcal(M)\colon\,F \subsetneq X \txtand \partial\SET{x}\subseteq F}.
	 \end{align*}
	Then $X= X_1\disunion X_2 \disunion X_3$ is the disjoint union of $X_1$, $X_2$, and $X_3$.
	Furthermore $X_3$ is the disjoint union of the sets $F_2$, where $F$ ranges over all flats in $M$ that are proper subsets of $X$, 
	because if $\partial\SET{x}\subseteq F\in \Fcal(M)$ and $x\notin T$, then every path from $x$ to some $t\in T$ must visit a vertex 
	of $F$. Therefore every $F$-$T$-separator in $D$ is also an $\left( F\cup\SET{x} \right)$-$T$-separator in $D$, 
	so $\rk(F) = \rk(F\cup\SET{x})$, thus $x\in F$, so $x\in F_2$ for some flat $F\subsetneq X$.
	Now assume that $x\in F_2\cap G_2$ for some $F,G\in\Fcal(M)$. 
	Then $F\cap G\in \Fcal(M)$ (Lemma~\ref{lem:flatsintersectinflat}) and $x\in F\cap G$. But $x\notin F_3$, 
	so $F\cap G$ is not a proper subset of $F$, thus $F=F\cap G$. Analogously $G=F\cap G$, thus
	$F = G$ whenever $x\in F_2\cap G_2$. Therefore $F_2\cap G_2 = \emptyset$ for every $F,G\in \Fcal(M)$ with $F\not= G$.

	\bSep
	First, we prove the second claim by induction on the rank of the flat. Let $O=\cl(\emptyset)$
	be the unique rank $0$ flat of $M$. Then
	$\alpha(O) = \left| O \right| - \rk(O) = \left| O \right|$. We have
	$$O=\SET{v\in E~\middle|~\vphantom{A^A}\nexists p\in \Pbf(D)\colon\,p_1=v\txtand p_{-1}\in T},$$ because
	 $O$ must consist precisely of those vertices of $D$, which cannot reach any target $t\in T$.
	 Therefore $\partial O = \emptyset$, which implies that for every $o\in O$,
	  $\partial\SET{o} \subseteq O$. Consequently, $O = O_2$ as defined above, 
	  so $\alpha(O) = \left| O_2 \right|$ follows and the induction base is established.
	
	\bSep
	Now let $F\in\Fcal(M)$ be a flat, and by induction hypothesis we may assume that $\alpha(F') = \left| F'_2 \right|$ for all $F'\in\Fcal(M)$ with $F'\subsetneq F$.
	Thus we may assume the equation
	 $$\left| F_3 \right| = \sum_{F'\in\Fcal(M,F)} \alpha(F').$$
	Furthermore, $F_1 = \delta_D(F,T)$ is a minimal $F$-$T$-separator in $D$ (Lemma~\ref{lem:flatofstrictgammoidisstrictgammoid}), therefore $F_1$ is a base of $F$, and so
	 $\left| F_1 \right| = \rk(F)$. We obtain
	 \begin{align*}
	 	\left| F \right| & = \left| F_1 \right| + \left| F_2 \right| + \left| F_3 \right| \\
	 	& = \rk(F) + \left| F_2 \right| +  \sum_{F'\in\Fcal(M,F)} \alpha(F'), \text{ and so}\\
	 	\left| F_2 \right| & = \left| F \right| - \rk(F) -   \sum_{F'\in\Fcal(M,F)} \alpha(F') = \alpha(F).
	 \end{align*}

	 \bSep
	 Now let $X$ be a subset of $E$ that is not necessarily a flat of $M$. Then $X_1 = \delta_D(X,T)$ is still an $X$-$T$-separator in $D$
	  (Lemma~\ref{lem:barrier}), albeit not necessarily minimal.
	 Therefore $\left| X_1 \right| \geq \rk(X)$. Thus we obtain
	 \begin{align*}
	 	\alpha(X) & = \left| X \right| - \rk(X) -  \sum_{F\in\Fcal(M,X)} \alpha(F) \\
	 			    & \geq \left| X \right| - \left| X_1 \right| - \left| X_3 \right| \\
	 			     & = \left| X_2 \right| \geq 0. \qedhere
	 \end{align*}
\end{proof}

% -*- root: ../../thesis.tex -*-

%\needspace{6\baselineskip}

\vspace*{-\baselineskip} %Remove the line space created by the tilde below
\begin{wrapfigure}{r}{5cm}
\vspace{\baselineskip}
\begin{centering}~~%move the picture slightly to the right
\includegraphics{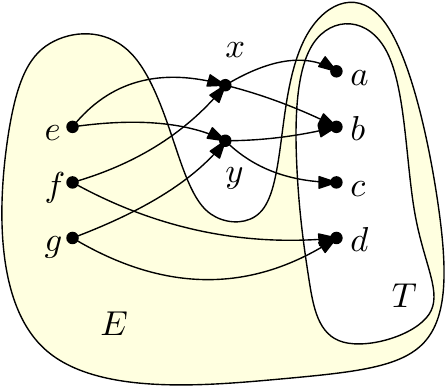}
\end{centering}%
\vspace*{-2\baselineskip} %make the picture more tightly cropped
\end{wrapfigure}
~ %The tilde creates a new dummy paragraph. WHY IS THAT NEEDED? -> would increase the space %
  % before the ex. environment. THE NEXT FREE LINE IS ESSENTIAL!

\begin{example} \label{ex:nonStrictGammoid}\PRFR{Jan 30th}
	Consider the digraph $D=(V,A)$ with $V=\dSET{a,b,c,d,e,f,g,x,y}$ and $A$ as depicted on the right.
	Let $E=\SET{a,b,c,d,e,f,g}$, $T=\SET{a,b,c,d}$ and $M=\Gamma(D,T,E)=(E,\Ical)$.
	We argue that $M$ -- which is obviously a gammoid -- is not a strict gammoid. In order to show this,
	we calculate some values of $\alpha$. Since $\alpha(X) = 0$ for every $X\in\Ical$,
	we only have
	to consider summands in the recurrence relation of $\alpha$, that correspond to dependent flats of $M$.
	Let \[\Fcal' = \Fcal(M) \BS \Ical = \SET{E, \SET{a,b,c,e}, \SET{a,b,d,f},\SET{b,c,d,g},\SET{d,e,f,g}}.\]
	For any $F\in\Fcal'\BSET{E}$, we have $\alpha(F) = \left| F \right| - \rk(F) = 4 - 3 = 1$.
	Therefore
	\[ \alpha(E) = \left| E \right| - \rk(E) - \sum_{F\in\Fcal'\BSET{E}}\alpha(F) = 7 - 4 - 4\cdot 1 = -1,\]
	so $M$ cannot be a strict gammoid (Theorem~\ref{thm:AlphaCriterion}).
\end{example}

\begin{definition}\label{def:alphaSystem}\PRFR{Jan 30th}
	Let $M=(E,\Ical)$ be a matroid.
	The \deftext[a-system of M@$\alpha$-system of $M$]{$\bm \alpha$-system of $\bm M$} is defined to be
	the family\label{n:alphaSystem} $\Acal_M = (A_i)_{i\in I} \subseteq E$, where 
	\[ I = \SET{(F,n)\in \Fcal(M)\times \N ~\middle|~\vphantom{A^A}  1\leq n \leq \alpha_M(F)} \]
	and $A_{(F,n)} = F$ for all $(F,n)\in I$.
\end{definition}

\noindent J.H.~Mason also proved that the condition $\alpha_M \geq 0$ is sufficient for $M$ to be a strict gammoid. First, we need a sufficient condition that allows us to recognize that a 
triple $(D,T,V)$ satisfies the equality $M=\Gamma(D,T,V)$.

\needspace{8\baselineskip}

\begin{lemma}[\cite{M72}, Lemma~2.3]\label{lem:presentsStrictGammoid}\PRFR{Jan 30th}
	Let $M=(V,\Ical)$ be a matroid, $D=(V,A)$ be a digraph and $T\subseteq V$.
	If for all $X\subseteq V$, the barriers $\delta_D(X,T)$ have the property
		\[ \rk_M\left( \delta_D(X,T) \right) = \rk_M(X) \] and
	if the barriers of flats are independent, i.e. for all $F\in \Fcal(M)$
	\[ \left| \delta_D(F,T) \right| = \rk_M(F), \]
	then $M = \Gamma(D,T,V)$.
\end{lemma}

\begin{proof}\PRFR{Jan 30th}
	Let $N = \Gamma(D,T,V)$ throughout this proof.
	Let $B\subseteq V$ be a base of $M$, and assume that $B$ is not independent in $N$, i.e. that 
	there is a $B$-$T$-separator $S$ in $D$,
	such that $\left| S \right| < \left| B \right|$. Let
	 $$ X = \SET{p_{-1}\in V~\middle|~\vphantom{A^A}
		 p\in \Pbf(D)\colon\,p_1\in B \txtand \left| p \right| \cap S = \emptyset} \cup S.$$
	By construction, $X$ consists of $S$ together with all vertices, that can be reached from some $b\in B\BS S$ 
	without traversing an element from $S$.
	Since $S$ separates $B$ from $T$ in $D$, $S$ is %an $X$-$\left( E\BS X \right)$-separator 
	an $X$-$T$-separator in $D$, too. By construction of $S$, we have $\partial \left(  X\BS S\right) \subseteq S$ 
	and $X\cap T = S\cap T$. 
	Therefore, for all $x\in X$ we have that $\partial\SET{x}\not\subseteq X$ implies that $x\in S$.
	Together with the fact that $S$ is a minimal $B$-$T$-separator and $B\subseteq X$, we arrive at
	 $\delta_D\left( X,T \right) = S$. Using the premise of the lemma we arrive at the contradiction to {\em (R1)},
	 $$
	\rk_M\left(\delta_D(X,T)\right) = \rk_M(X) \geq \rk_M(B) = \left| B \right| > \left| S \right| = \left| \delta_D(X,T) \right|.$$
	In other words, if we assume that the base $B$ of $M$ is dependent in $N$, we can construct a set $X$ that spans $M$,
	but that has a barrier in $D$ which is smaller than $\left| B \right|$, and therefore the barrier property
	$\rk_M\left( \delta_D(X,T) \right) = \rk_M(X)$ cannot hold. Consequently,
	 $B$ must be independent in $N$ and we have $\rk_N(X) \geq \rk_M(X)$.

	 \noindent
	 Now let $X\subseteq V$, then let $F = \cl_M(X)$ be the smallest flat containing $X$ in $M$.
	 By Lemma~\ref{lem:barrier}, $\delta_D(F,T)$ is an $F$-$T$-separator in $D$, therefore
	 $$\rk_N(X) \leq \rk_N(F) \leq \left| \delta_D(F,T) \right| = \rk_M(F) = \rk_M(X).$$
	 Consequently, $\rk_M = \rk_N$, and so $M = N = \Gamma(D,T,V)$.
\end{proof}

\begin{theorem}[\cite{M72}, Theorem~2.4]\label{thm:AlphaCriterion2}\PRFR{Jan 30th}
	Let $M=(E,\Ical)$ be a matroid.
	If $\alpha(X) \geq 0$ holds for all $X\subseteq E$, then $M$ is a strict gammoid.
\end{theorem}

\PRFR{Jan 30th}
\noindent J.H.~Mason's proof \cite{M72} uses the following line of arguments: 
First, observe that $\alpha \geq 0$ is a sufficient 
condition for the $\alpha$-system of $M$ to have a transversal
 $T_0$. Let $T_0$ be such a transversal, and the map $\sigma' \colon T_0 \maparrow \Fcal(M)$ shall be the projection on the first coordinate of the bijection 
  $\sigma\colon T_0\maparrow I$ witnessing the transversal property of $T_0$ with respect to $\Acal_M$. Then
let $T=E\BS T_0$ be the target set, and let $D=(E,A)$ be the digraph where $(u,v)\in A$ if and only if
$u\in T_0$ and
$v \in \sigma'(u)$. We have $M= \Gamma(D,T,E)$. Now, let us see this proof in detail.

\begin{proof}\PRFR{Jan 30th}
	Let $I$ and $\Acal_M$ be as in Definition~\ref{def:alphaSystem}.
It follows from Hall's Theorem (Corollary~\ref{cor:Hall}) that $\Acal_M$
	has a transversal $T_0$ if and only if for all $J\subseteq I$ the inequality
	\( \left| \bigcup_{i\in J} A_i \right| \geq \left| J \right| \) holds. For $\Acal_M$, this is the case if and only if for all $\Gcal \subseteq \Fcal(M)$, the inequality
	\begin{align}
	\left| \bigcup_{F\in \Gcal} F \right| \geq \sum_{F\in \Gcal} \alpha(F)\label{ineq:Thm24}
	\end{align}
	holds. 
	 For every $X\subseteq E$, the recurrence relation of $\alpha$ (Definition~\ref{def:alphaM}) can be written as the
	equation \[ \alpha(X) + \sum_{F'\in \Fcal(M,F)} \alpha(F') = \left| X \right| - \rk(X).\]
	Consequently, we obtain the inequality
	\[ \left| X \right| - \rk(X) =
	\alpha(X) + \sum_{F'\in \Fcal(M,X)} \alpha(F') \geq 
	 \sum_{F'\in \Fcal(M),\,F'\subseteq X} \alpha(F'),\]
	 where equality holds whenever $X\in\Fcal(M)$ or $\alpha(X)=0$.
	 Now let $\Gcal\subseteq \Fcal(M)$, then we can use the last inequality
	 together with the property that $\alpha \geq 0$, and we obtain
	 \begin{align*}
	 	\left| \bigcup_{G\in\Gcal} G \right| \,\,\geq\,\, \rk\left( \bigcup_{G\in\Gcal} G \right) 
	 	+ \sum_{F'\in \Fcal(M),\, F'\subseteq\,\bigcup\Gcal} \alpha(F') \,\,\geq\,\, \sum_{G\in\Gcal} \alpha(G),
	 \end{align*}
	  therefore the inequality~\ref{ineq:Thm24} holds for $\Gcal$. Consequently, $\Acal_M$ has a transversal. 
\bSep
	Let $T_0$ be a transversal of $\Acal_M$, and let $\sigma\colon T_0 \maparrow I$ be a bijective map
	with the property that for all $t\in T_0$, $t\in A_{\sigma(t)}$. We set $T= E\BS T_0$ and define the map
	  \[ \sigma' \colon T_0 \maparrow \Fcal(M),\, t\mapsto F_t \]
	  where $F_t\in \Fcal(M)$ such that there is some $i_t\in \N$ with $\sigma(t) = (F_t,i_t)$.
	Now, define $D=(E,A)$ to be the digraph on $E$ where
	$(u,v)\in A$ if and only if $u\in T_0$ and $v\in \sigma'(u)$.
	Let $N= \Gamma(D,T,E)$ be the strict gammoid represented by $(D,T,E)$. We want to use Lemma~\ref{lem:presentsStrictGammoid} in order to show that $M=N$.
	Let $X\subseteq E$ be a subset of $E$. We have to show that the set
	\[ B_X = \delta_D(X,T) = \SET{x\in X\BS T ~\middle|~\vphantom{A^A} \sigma'(x) \not\subseteq X} \cup \left( X\cap T \right) \]
	contains a base of $X$ with respect to $M$, i.e. that $\rk_M(B_X) = \rk_M(X)$; and further, if $X\in\Fcal(M)$, we
	have to show that $\rk_M(B_X) = \left| B_X \right|$ holds, too. 
	Assume for now, that $B_X$ contains a base of $X$, and that $X\in \Fcal(M)$. Then
	 $$\left| X \right| = \rk_M(X) + \sum_{F\in\Fcal(M),\,F\subseteq X} \alpha_M(F)$$ and the set
	 $X\BS B_X$ consists of all elements $t\in X \cap T_0$, such that the flat $\sigma'(t)$ is a subflat of $X$,
	 therefore $$\left| X\BS B_X \right| = \sum_{F\in\Fcal(M),\,F\subseteq X} \alpha_M(F),$$
	 and consequently $\left| B_X \right| = \rk_M(B_X)$.

	 \bSep We give an indirect argument that indeed
	 $\rk_M(B_X) = \rk_M(X)$, so let us assume that \linebreak $\rk_M(B_X) < \rk_M(X)$.
	 Let $Y\subseteq \cl_M(X)$ be a subset that is maximal with respect to set-inclusion among all
	 subsets of $\cl_M(X)$ with the property, that $\rk_M(B_Y) < \rk_M(Y)$,
	 where
	 $B_Y = \delta_D(Y,T)$.
	 We show that $\cl_M(B_Y)\subseteq Y$ holds for the maximal choice $Y$.
	 Let $Y' = Y\cup \cl_M(B_Y)$, then 
	 \[ B_{Y'} = \delta_D(Y',T) =
	  \SET{y'\in Y'\BS T~\middle|~\vphantom{A^A}\sigma'(y')\not\subseteq Y'}\cup 
	  \left( Y'\cap T \right) \subseteq \cl_M(B_Y),\]
	  because $Y'\cap T = \left( Y\cap T \right) \cup \left( \cl_M(B_Y) \cap T \right) \subseteq \cl_M(B_Y)$
	  for the reason that $$Y\cap T\subseteq \delta_D(Y,T) = B_Y;$$
	  and because
	  $$ \SET{y'\in Y' \BS T ~\middle|~ y'\notin \cl_M(B_Y) \txtand \sigma'(y')\not \subseteq Y'} \subseteq \SET{y\in Y \BS T ~\middle|~  \sigma'(y)\not \subseteq Y} \subseteq B_Y. $$
	  This holds since for every $y\in Y'\BS \cl_M(B_Y) \subseteq Y$,
	   the inequality  $\partial \SET{y} \cap \left(T_0\BS Y'\right) \not= \emptyset$
	   implies the inequality $\partial \SET{y} \cap \left(T_0\BS Y\right) \not= \emptyset$ due to $Y\subseteq Y'$ --- so the left-most
	   set above is actually empty, because $B_Y\subseteq \cl_M(B_Y)$.
	 Since $\cl_M$ does not change the rank, we obtain $$\rk_M(B_{Y'}) \leq \rk_M(B_Y) < \rk_M(Y) = \rk_M(Y'),$$
	 that means $Y'= Y\cup\cl_M(B_Y)$ also satisfies $\rk_M(B_{Y'}) < \rk_M(Y')$,
	 and consequently,
	  $\cl_M(B_Y)\subseteq Y$ for the $\subseteq$-maximal subset $Y$.

	  \bSep
	  Now, we want to show that for the $\subseteq$-maximal choice $Y$, the barrier $B_Y$ is independent in $M$.
	  We give an indirect argument.
	  Assume that $B_Y$ is not independent, therefore there is a circuit $C\subseteq B_Y$. Clearly,
	  $\cl_M(C) \subseteq \cl_M(B_Y) \subseteq Y$. Let $F\in\Fcal(M)$ such that $F\subseteq \cl_M(C)$.
	  From the definition of $B_Y$, it is clear, that any $e\in E$ with $\sigma'(e) = F$ has the property, that
	  $e\in F\BS B_Y \subseteq  Y\BS B_Y$. So $\cl_M(C)$ has at least as many elements as the sum of the $\alpha$-values of all (not necessarily proper)
	  subflats of $\cl_M(C)$ plus the number of elements of $B_Y\cap \cl_M(C)$.
	  Thus we obtain
	  \begin{align*}
	  	\left| \cl_M(C) \right| 
	  	 & \,\,\geq\,\,
	  	\left| \cl_M(C) \cap B_Y \right| + \sum_{F\in\Fcal(M), F\subseteq \cl_M(C)} \alpha_M(F) \\
	  	& \,\, \geq\,\, \hphantom{|} \rk_M(C) + 1 \,\,\,\,\,\,\,\, + \sum_{F\in\Fcal(M), F\subseteq \cl_M(C)} \alpha_M(F) \\
	  	& \,\,=\,\, \left| \cl_M(C) \right| + 1,
	  \end{align*}
	  and arrive at a contradiction, where the second inequality is due to the fact that $C\subseteq B_Y$ and $\rk_M(C) = \left| C \right| - 1$. Therefore $B_Y \in \Ical$ is independent in $M$.

	  \bSep Now, observe that with $\alpha_M \geq 0$, we obtain
	   \begin{align*} \rk_M(Y)  > \rk_M(B_Y) & =\, \left| B_Y \right| 
	   \\ & =\, \left| Y \right| - \sum_{F\in\Fcal(M),\,F\subseteq Y} \alpha_M(Y)
	   \\ & \geq\, \rk_M(Y) + \alpha_M(Y) 
	   \\ & \geq\, \rk_M(Y). \end{align*}
	   	This  contradiction yields, that the assumption, that there is a maximal subset $Y$ of $\cl_M(X)$ 
	   with $\rk_M(B_Y) < \rk_M(Y)$, is wrong. Consequently, $\rk_M(B_X) < \rk_M(X)$ cannot be the case.
	   Thus $\rk_M(B_X) = \rk_M(X)$ and all premises of Lemma~\ref{lem:presentsStrictGammoid} are met. We just established $M=N=\Gamma(D,T,E)$, so $M$ is a strict gammoid.
\end{proof}

\begin{corollary}\label{cor:MasonAlpha}\PRFR{Jan 30th}
	Let $M=(E,\Ical)$ be a matroid. Then $M$ is a strict gammoid if and only if for all $X\subseteq E$
	the inequality $\alpha(X) \geq 0$ holds.
\end{corollary}
\begin{proof}\PRFR{Jan 30th}
	Combine the Theorems \ref{thm:AlphaCriterion} and \ref{thm:AlphaCriterion2}.
\end{proof}

\noindent We just saw that we obtain a 
strict representation of a strict gammoid from every transversal of its $\alpha$-system.
The converse holds, too, in the sense that every representation of a strict gammoid yields a transversal
of its $\alpha$-system.

\needspace{6\baselineskip}
\begin{lemma}\label{lem:alphaTransversalFromD}\PRFR{Apr 4th}
	Let $D=(V,A)$ be a digraph, $T\subseteq V$, $M=\Gamma(D,T,V)$ be a strict gammoid,
	and $\Acal_M=(A_i)_{i\in I}$ be the $\alpha$-system of $M$.
	Let further $X = V\BS T$ and
	\[ \phi\colon X \maparrow \Fcal(M),\,u\mapsto \cl\left(\vphantom{A^A} \SET{v\in V~\middle|~ (u,v)\in A} \right) .\]
	Then $X$ is a transversal of $\Acal_M$, and there is a bijection $\psi\colon X\maparrow I$
	such that   $$x\in A_{\psi(x)} = \phi(x)$$ for all $x\in X$.
\end{lemma}
\begin{proof}\PRFR{Apr 4th} Without loss of generality we may assume that there is no loop arc $(v,v)\in A$ in $D$.
% and that all $t\in T$
%	are sinks of $D$.
	Let $M=(V,\Ical)$, $x\in V\BS T$, and let $S_x = \SET{y\in V~\middle|~(x,y)\in A}$.
	Clearly $S_x$ is an $\SET{x}\cup S_x$-$T$-separator in $D$, because every path from $x$ to $t\in T$
	must use an arc that leaves $x$, and thus this arc visits a vertex from $S_x$. Consequently,
	$x\in \cl(S_x)$. Since $x\notin S_x$, we obtain that 
	$$\rk(S_x) \leq \left| S_x \right| < \left| S_x\cup\SET{x} \right| \leq \left| \cl(S_x) \right|.$$
	Therefore $\phi(x) = \cl(S_x)$ is a dependent flat of $M$ with $x\in \phi(x)$.
	Let $F\in \Fcal(M)$, let $I_F = \SET{(F',k) \in I~\middle|~F' \subseteq F}$,
	and let $X_F = \SET{x\in X~\middle|~ \phi(x) \subseteq F}$.
	We show that $X_F$ is a partial transversal of the subfamily $\Acal_F = (A_i)_{i\in I_F}$ of $\Acal_M$,
	by induction on the nullity $\left| F \right| - \rk(F)$ of $F$.
	If $F\in \Ical$ then $X_F = \emptyset$, since $\cl(S_x)$ is dependent for all $x\in V\BS T$.
	We give an indirect argument for the induction step and assume that
	$\left| X_F  \right| > \left| F \right| - \rk(F)$.
	%By induction hypothesis and the fact that 
	%$\sum_{F'\in \Fcal(M),\,F'\subseteq F} \alpha_M(F') = \left| F \right| - \rk(F)$ we obtain that
	%$$F = \phi\left( f_1  \right) = \phi\left(f_2\right) = \ldots = \phi\left( f_{k+1} \right)$$
	%holds, where $k=\alpha_M(F)+1$. --- up to possibly renumbering the $f_i$'s.
	%Furthermore, $\cl(S_x) = F$ for all $x \in \SET{f_1,f_2,\ldots, f_k}$.
	%Consequently, for all $x\in X_F$, the set $S'_x = S_x\cup\SET{x}$
	%must contain a circuit $C_x \subseteq S'_x$ with $\left| C_x \right| = \rk(F) + 1$.
	There is an $F$-$T$-separator $S_F$ in $D$ with minimal cardinality $\left| S_F \right| = \rk(F)$.
	Clearly, $S_F\in \Ical$ and $S_F\subseteq \cl(F) = F$ (Lemma~\ref{lem:minFTSeparatorInFStrictGammoid}).
	Since $\left| X_F \right| > \left| F \right| - \rk(F)$ and $X_F \subseteq F$, we obtain that
	$X_F \cap S_F \not= \emptyset$. Let $f\in X_F\cap S_F$, 
	then $S_f = \SET{g\in V~\middle|~(f,g)\in A} = \phi(f) \subseteq F$. Since $f\in X_F \subseteq V\BS T$
	we have 
	$f\notin T$. Now let $f'\in F$ and $t\in T$, then every path $p\in\Pbf(D;f',t)$
	with $f\in \left| p \right|$ must also visit
	another element $f''\in S_f \subseteq F$ as it continues to $t$. Thus every such $p$ 
	must visit an element from $S_F\BSET{f}$ after visiting $f$
	 --- a contradiction to the fact that $S_F$ is an $F$-$T$-separator 
	with minimal cardinality in $D$. Thus $\left| X_F \right| \leq \left| F \right| - \rk(F)$.
	Consequently, $X_F$ is a partial transversal of $\Acal_F$.
%	We showed that for all $F\in \Fcal(M)$, $X_F = \SET{x\in X~\middle|~ \phi(x) \subseteq F}$ is a partial transversal
%	of the subfamily $\Acal_F$. 
	Observe that
	$$\left| X \right| = \left| V\BS T \right| = \left| V \right| - \rk(V) = \sum_{F\in \Fcal(M)} \alpha_M(F) = \left| I \right|.$$
	So $X$ is a transversal of $\Acal_M$ with the property, that
	$\left| \SET{x\in X~\middle|~\phi(x) = F}  \right| = \alpha_M(F)$ holds for all $F\in \Fcal(M)$,
	thus $X$ has the claimed property.
\end{proof}

%\begin{remark}
%	Lemma~\ref{lem:alphaTransversalFromD} and Theorem~\ref{thm:IPEssentialStars} together yield
%	the observation that every strict representation $(D,T,V)$ of a strict gammoid, where every arc of $D$
%	is essential with respect to $(D,T,V)$, corresponds to a transversal $T$ of its $\alpha$-system together with
%	an arbitrarily chosen witness $\iota \colon V\BS T\maparrow I$ of $T$ and an
%	arbitrarily chosen circuit $C_v$ with $\rk(C_v) = \rk(A_{\iota(v)})$ for every $v\in V\BS T$.
%\end{remark}

\clearpage
% -*- root: ../thesis.tex -*-

\section{Transversal Matroids}\label{sec:TransversalMatroids}

\PRFR{Jan 30th}
The notion of transversal matroids has been introduced in section \ref{sec:shortTransversalMatroids}.
In this section, we  develop the theory of transversal matroids a little further.

\begin{lemma}\label{lem:transersalMatroidsAreGammoids}\PRFR{Jan 30th}
	Let $E$ be a finite set, $\Acal=(A_i)_{i\in I}\subseteq E$ be a finite family of subsets of $E$,
	and $M=M(\Acal)$ be the transversal matroid presented by $\Acal$.
	Then $M$ is a gammoid.
\end{lemma}

\begin{proof}\PRFR{Jan 30th}
	Without loss of generality, we may assume that $E\cap I=\emptyset$.
	Let $D=(V,A)$ be the digraph where $V=E\disunion I$ and
	$(e,i)\in A$ if and only if $e\in E$, $i\in I$ and $e\in A_i$.
	Then $M(\Acal)= \Gamma(D,I,E)$: The routings $R\colon X_0\routesto I$ in $D$ with $X_0\subseteq E$
	are in correspondence to the injections $\iota\colon X_0\maparrow I$ that have the property $x\in A_{\iota(x)}$ for
	all $x\in X_0$, where $R(\iota)=\SET{x\iota(x) \in \Pbf(D)\mid x\in X_0}$ is the routing induced by a partial transversal $X_0$ of $\Acal$ with injective map $\iota$; and $P(R) = \SET{p_1\in E\mid p\in R}$ is the partial transversal of $\Acal$ induced from a routing $R\colon X_0\routesto I$ with $X_0\subseteq E$ in $D$.
\end{proof}

\begin{lemma}\label{lem:dualOfStrictGammoidIsTransversal}\PRFR{Jan 30th}
	Let $M=(E,\Ical)$ be a matroid.
	 If $M$ is a strict gammoid, then $M^\ast$ is a transversal matroid.
\end{lemma}
\begin{proof}\PRFR{Jan 30th}
	It is an immediate consequence from the fact that
	there is a linking from $X$ onto $T$ in $D$ if and only if $V\BS X$ is a transversal of
 $\Acal_{D,T}$ (Lemma~\ref{lem:linkage}), that the bases of $\Gamma(D,T,V)$ are precisely those subsets of $V$, for which their complement in $V$ is a base of the transversal matroid $M(\Acal_{D,T})$
defined by the linkage system of $D$ to $T$. Thus $M^\ast = M(\Acal_{D,T})$.
\end{proof}

\noindent The converse statement holds, too.

\begin{lemma}\label{lem:dualtransversalstrictgammoid}\PRFR{Jan 30th}
	Let $\Acal=(A_i)_{i\in I} \subseteq E$ be a family of sets, and 
	$M = M(\Acal)$ the transversal matroid presented by $\Acal$. Then
	$M^\ast$ is a strict gammoid.
\end{lemma}

\begin{proof}\PRFR{Jan 30th}
	Without loss of generality, we may assume that $E\cap I = \emptyset$.
	We define the family $\hat\Acal = (\hat A_i)_{i\in I} \subseteq E\disunion I$
	by setting $\hat A_i = A_i \disunion \SET{i}$ for all $i\in I$.
	Further, let $D=(E\disunion I, A)$ where 
	\[ A = \SET{(e,i)\in E\times I\mid e\in A_i}. \]
	It is easy to see that the linkage system $\Acal_{D,E}$ of $D$ to $E$ is
	precisely the family $\hat \Acal$. Therefore $M(\hat\Acal)^\ast = \Gamma(D,E,E\disunion I)$ is
	a strict gammoid.
	On the other hand, $M(\hat \Acal)\restrict E = M(\Acal)$ is evident from the construction.
	With Lemma~\ref{lem:restrictcontractdual} we obtain
	\[
		M(\Acal)^\ast = \left( M(\hat\Acal)\restrict E\right)^\ast = \left( M(\hat\Acal)^\ast \right)\contract E,
	\]
	where the last term is the contraction of a strict gammoid, therefore $M^\ast$ is a strict gammoid (Lemma~\ref{lem:contractionStrictGammoid}).
\end{proof}

\begin{corollary}\label{cor:transversalstrictdual}\PRFR{Jan 30th}
	Let $M$ be a matroid. Then $M$ is a transversal matroid if and only if $M^\ast$ is a strict gammoid.
\end{corollary}

\begin{corollary}\label{cor:transversalrepresentation}\PRFR{Jan 30th}
	Let $E,I$ be finite sets and $\Acal = (A_i)_{i\in I} \subseteq E$ be a family of subsets.
	If $M=M(\Acal)$ is the transversal matroid presented by $\Acal$ and $r = \rk_M(E)$,
	then there is a family of subsets $\Acal' = (A'_i)_{i=1}^r \subseteq E$, such that
	$M=M(\Acal')$.
\end{corollary}
\begin{proof}\PRFR{Jan 30th}
	By Lemma~\ref{lem:dualtransversalstrictgammoid} the dual $M^\ast$ is a strict gammoid
	of rank $\left| E \right| - r$. Thus there is a digraph $D=(E,A)$ and a base $T\subseteq E$ of $M^\ast$ such that $M^\ast = \Gamma(D,T,E)$. Let $\Acal_{D,T}$ be the linkage system of $D$ to $T$,
	then $M(\Acal_{D,T})^\ast = M^\ast$, and $\Acal_{D,T} = (A^{(D,T)}_i)_{i\in E\BS T}$ consists of $\left| E \right| - \left| T \right| = r$ sets $A_i$, which may be renumbered by the integers from $1$ through $r$ yielding the desired $\Acal'$.
\end{proof}

\begin{example}\label{ex:transversalMatroidDualityRespecting}\PRFR{Feb 15th}
It is particularly easy to obtain a duality respecting representation of a transversal matroid from representations that are in the form of Corollary~\ref{cor:transversalrepresentation}.
Let $M=(E,\Ical)$ be a transversal matroid, $r = \rk(E)$, and $\Acal = (A_i)_{i=1}^r \subseteq E$ be a representation of $M$, i.e. $M = M(\Acal)$.
Then there is a base $B\in\Bcal(M)$ and a bijective map $\phi \colon B\maparrow \SET{1,2,\ldots,r}$ such that 
$b\in A_{\phi(b)}$ holds for all $b\in B$. Furthermore, if $I=\dSET{i_1,i_2,\ldots,i_r}$ is a set with $I\cap E = \emptyset$, 
then the digraph $D=(E\disunion I, A)$ with
	 \begin{align*} A = \SET{\left(i_{\phi(b)},b\right) ~\middle|~ b\in B} & \cup \SET{\left(i_{\phi(b)},i_{k}\right) ~\middle|~b\in B,\, k\in\SET{1,2,\ldots,r}\BSET{\phi(b)}\colon\, b\in A_k} \\ &
	 \cup \SET{\left(e,i_k\right) ~\middle|~
		\vphantom{A^A} k\in \SET{1,2,\ldots,r},\,e\in A_k\BS B}
\end{align*}
has the property, that $M = \Gamma(D,B,E)$, because it is the digraph that arises from the 
digraph described in Lemma~\ref{lem:transersalMatroidsAreGammoids} and the construction from the proof 
of Theorem~\ref{thm:gammoidRepresentationWithBaseTerminals} with respect to the basis $B$.
 Since the premises of Lemma~\ref{lem:dualityrespectingrepresentation} are satisfied,
$(D,B,E)$ is a duality respecting representation of $M$.
\end{example}

\begin{corollary}\PRFR{Feb 15th}
	Let $M$ be a transversal matroid. There is a representation $(D,T,E)$ where $D=(V,A)$ with $M = \Gamma(D,T,E)$ and $\left| V \right| \leq \left| E \right| + \rk(E)$. There even is a representation that uses a digraph with $\left| V \right| < 2\cdot \left| E \right|$.
\end{corollary}
\begin{proof}\PRFR{Feb 15th}
	A representation with $\left| V \right| \leq \left| E  \right| + \rk(E)$ has been constructed in Example~\ref{ex:transversalMatroidDualityRespecting}. If $\rk(E) = \left| E \right|$, then $M = (E,2^E)$, i.e. $M$ is the free matroid on $E$, and therefore
	the digraph $D' = \left( E, \emptyset \right)$ yields a representation $M = \Gamma(D',E,E)$ with strictly fewer than $2\cdot \left| E \right|$ elements.
\end{proof}

\clearpage
% -*- root: ../thesis.tex -*-

\section{Constructions within the Class of Gammoids}

\noindent \PRFR{Mar 27th}
In this section, we explore methods of obtaining new gammoids from old ones. The main application of this section is the following:
If we know that a matroid $M$ may be constructed from a matroid $N$ using a construction that does not leave the class of gammoids,
then we may conclude that $M$ is a gammoid whenever $N$ is a gammoid.

\noindent
Let us start with a well-known result of J.H.~Mason \cite{M72}.

\begin{theorem}\label{thm:GammoidsClosedMinorsDuality}\PRFR{Mar 27th}
	The class consisting of all gammoids is closed under minors, duality, and direct sums.
\end{theorem}
\begin{proof}\PRFR{Mar 27th}
	Let $M=(E,\Ical)$ be a matroid. It is clear from Definition~\ref{def:gammoid}
	that the representation $(D,T,E)$ of $M$ yields the representation $(D,T,X)$ of
	$M\restrict X$ for all $X\subseteq E$. Thus the class of all gammoids is closed under restriction.
	Corollary~\ref{cor:dualityrespectingrepresentation} yields that if $M$ is a gammoid, then so is
	its dual $M^\ast$. Consequently, the class of all gammoids is closed under duality.
	It follows with Lemma~\ref{lem:restrictcontractdual} that the class of gammoids is also closed under
	contraction. We showed in Lemma~\ref{lem:arcCSubAdditive} that the class of gammoids is closed under direct sums.
\end{proof}

\noindent \PRFR{Mar 27th}
	Remember that Corollary~\ref{cor:extWithLoopCoLoop} established, that
	every extension of a gammoid $M$ by a loop or a coloop is again a gammoid.
	Therefore $M$ is a gammoid if and only if $M\restrict X$ is a gammoid, where $X$ consists of all elements of the ground set of $M$,
	that
	are neither loops nor coloops.

\begin{lemma}\label{lem:principalExtOfStrictGammoidIsStrict}\PRFR{Mar 27th}
	Let $M=(E,\Ical)$ be a gammoid, $e\notin E$, and let $N\in \Xcal(M,e)$ such that
	\[ C = \SET{F\in \Fcal(M)~\middle|~e\in \cl_N(F)} \]
	has at most one $\subseteq$-minimal element.
	Then $N$ is a gammoid. If $M$ is a strict gammoid, then $N$ is a strict gammoid.
\end{lemma}
\begin{proof}\PRFR{Mar 27th}
	Let $(D,T,E)$ with $D=(V,A)$ be a strict representation of $M$ if $M$ is a strict gammoid, otherwise let $(D,T,E)$
	be a representation of $M$.
	If $C=\emptyset$, then $e$ is a coloop in $N$. So $(D',T\cup\SET{e},E\cup\SET{e})$
	with
	$D' = (V\cup\SET{e},A)$ is a representation of $N$, which is a strict representation if $M$ is a strict gammoid.
	Otherwise let $F_0 = \bigcap C$ be the unique $\subseteq$-minimal element of $C$.
	 Then
	$D''=(V\cup\SET{e},A'')$ with
	\[ A'' = A \cup \left( \SET{e} \times F_0 \right)\]
	yields the representation $(D'',T,E\cup\SET{e})$ of $N$ -- which is strict if $M$ is a strict gammoid:
	Let $N' = \Gamma(D'',T,E\cup\SET{e})$, and let $X\subseteq E\cup\SET{e}$ be independent in $N'$.
	If $X\subseteq E$, then $X$ is independent in $N$, because by construction, 
	no path $p\in \Pbf(D'';x,t)$ for any $x\in E$ and any $t\in T$
	visits $e$. Thus every routing $X\routesto T$ in $D''$ is also a routing with respect to $D$. If $e\in X$ for $X$ independent in $N'$,
	 then the fact that $F_0\not\subseteq \cl_M(X\BSET{e})$  follows
	 from the way $D''$ is constructed from $D$: Every path from $e$ to any $t\in T$ visits some element from $f\in F_0$.
	 Thus every routing $X\routesto T$ in $D''$ induces a routing $(X\BSET{e})\cup\SET{f}\routesto T$ in $D$ for some $f\in F_0\BS X$.
	  Therefore there is some
	$f\in F_0\BS X$ such that $(X\BSET{e})\cup\SET{f}$ is independent in $M$, consequently $f\notin \cl_M(X\BSET{e})$.
%	 In this case $X=X'\cup\SET{e}$ with $X'= X\cap E$ independent in $M$
%	such that $F_0\not\subseteq \cl_M(X')$. 
	We obtain $\rk_N(X) = \rk_M(X\BSET{e}) + 1$ and so $X$ is independent in $N$.
	Now let $X\subseteq E\cup\SET{e}$ be independent in $N$. If $e\notin X$ holds,
	 then $X$ is independent in $M$. So $X$ is independent 
	in $N'$, too, because $A\subseteq A''$. If $e\in X$ and $X$ is independent in $N$, then $X' = X\BSET{e}$ is independent in $M$ 
	and $F_0\not\subseteq \cl_M(X')$.
	But then there is some $f\in F_0\BS X'$ such that $\rk_M(X'\cup\SET{f}) > \rk_M(X')$, thus $X'\cup\SET{f}$ is independent in $M$, too.
	Now let $R\colon X'\cup\SET{f}\routesto T$ be a corresponding routing in $D$, and let $p^{(f)}\in R$ be the path of that routing 
	where $p_1^{(f)} = f$. Then $R' = \left(R\BSET{p^{(f)}}\right)\cup\SET{ep^{(f)}}$ is a routing from $X'\cup\SET{e}$ to $T$ in $D''$. 
	It follows that $X$ is independent in $N'$, and consequently $N = N'$.
\end{proof}

\begin{definition}\label{def:deflateOfM}\PRFR{Mar 27th}
	Let $M=(E,\Ical)$ be a matroid, $X\subseteq E$. The restriction $N=M\restrict X$ shall be a \deftext[deflate of a matroid]{deflate of $\bm M$},
	if $E\BS X = \dSET{e_1,e_2,\ldots,e_m}$ can be ordered naturally,
	such that for all $i\in \SET{1,2,\ldots,m}$ the modular cut
	\[ C_i = \SET{F\in \Fcal\left( \vphantom{A^A} M\restrict \left( X\cup\SET{e_1,e_2,\ldots,e_{i-1}} \right)  \right)~\middle|~ e_i \in \cl_M(F)} \]
	has precisely one $\subseteq$-minimal element.
\end{definition}

\begin{definition}\label{def:deflated}\PRFR{Mar 27th}
	Let $M=(E,\Ical)$ be a matroid. $M$ shall be called \deftext{deflated}, if the only deflate of $M$ is $M$ itself.
\end{definition}

\needspace{4\baselineskip}

\begin{lemma}\label{lem:excludedMinorsForGammoidsAreDeflated}\PRFR{Mar 27th}
	Let $M$ be an excluded minor for the class of gammoids. Then $M$ is deflated.
\end{lemma}
\begin{proof}\PRFR{Mar 27th}
	We give an indirect proof and
	assume that $M=(E,\Ical)$ is an excluded minor for the class of gammoids, and $M$ is not deflated.
	Then there is an element $e\in E$ such that 
	\[ C = \SET{F\in \Fcal\left( M\restrict (E\BSET{e}) \right) \vphantom{A^A}~\middle|~ e\in \cl_M(F)} \]
	has the property that
	\[ C = \SET{F\in \Fcal\left( M\restrict (E\BSET{e}) \right) \vphantom{A^A}~\middle|~ F_0\subseteq F} \]
	where $F_0 = \bigcap C$.
	Since $M$ is an excluded minor, the restriction $N = M\restrict (E\BSET{e})$ is a gammoid.
%	Thus there is a strict gammoid $N'$ such that $N'\restrict (E\BSET{e}) = N$. Let $F_0' = \cl_{N'}(F_0)$.
%	Clearly, $C' = \SET{F'\in \Fcal(N') ~\middle|~ F_0' \subseteq F'}$ is a modular cut of $N'$.
%	Let $M' = \Xcal(N',e)$ such that $\SET{F'\in \Fcal(N')~\middle|~ e\in \cl_{N'}(F')} = C$, 
%	the existence and uniqueness of $M'$ is shown in Theorem~\ref{thm:crapo}. We obtain that $M'$ is a strict gammoid
%	from Lemma~\ref{lem:principalExtOfStrictGammoidIsStrict}.
%	We further have
%	\begin{align*} C & = \SET{F\in \Fcal(N)~\middle|~ F_0\subseteq F} \\
%					 & = \SET{\left( \cl_{N'}(F) \right)\cap \left( E\BSET{e} \right) ~\middle|~ F\in \Fcal(N),\,\cl_{N'}(F_0) \subseteq \cl_{N'}(F)}\\
%					  & = \SET{ F' \cap \left( E\BSET{e} \right) ~\middle|~ F\in \Fcal(N),\,F'\in\Fcal(N'),\,F_0' \subseteq F',\, F\subseteq F'}\\
%	 & = \SET{F'\cap \left( E\BSET{e} \right)~\middle|~F'\in C'},
%	\end{align*}
%	and therefore $M'\restrict E = M$, which contradicts our assumption that $M$ is not a gammoid.
	In this situation, Lemma~\ref{lem:principalExtOfStrictGammoidIsStrict} yields that $M$ is a gammoid --- a contradiction. Thus every
	excluded minor for the class of gammoids is deflated.
\end{proof}

\needspace{4\baselineskip}
\begin{lemma}\label{lem:deflationLemma}\PRFR{Mar 27th}
	Let $M=(E,\Ical)$ be a matroid, $X\subseteq E$ and let  $N = M\restrict X$ be a deflate of $M$.
	Then $M$ is a gammoid if and only if $N$ is a gammoid.
\end{lemma}
\begin{proof}\PRFR{Mar 27th}
	If $M$ is a gammoid, then $N$ is a gammoid, too (Theorem~\ref{thm:GammoidsClosedMinorsDuality}).
	Now let $N$ be a gammoid, and let $E\BS X = \dSET{e_1,e_2,\ldots,e_m}$ be implicitly ordered with the properties required in Definition~\ref{def:deflateOfM}.
	Lemma~\ref{lem:principalExtOfStrictGammoidIsStrict} yields that $M\restrict \left( X\cup\SET{e_1,e_2,\ldots,e_i} \right)$ is a gammoid whenever
	$M\restrict \left( X\cup\SET{e_1,e_2,\ldots,e_{i-1}} \right)$ is a gammoid. 
	Thus, by induction on $m$, we obtain that $M$ is a gammoid whenever $N$ is a gammoid -- a fact that we assumed.
\end{proof}

\noindent The former situation is a special case of the following situation:%, for which to see we need a result of J.H.~Mason \cite{M72}.

\begin{definition}\PRFR{Mar 27th}
	Let $T=(T_0,\Tcal)$ be a matroid, $D=(V,A)$ be a digraph with $T_0\subseteq V$, and let $E\subseteq V$ be any set.
	The \deftext[matroid induced by D from R@matroid induced by $D$ from $T$]{matroid on $\bm E$ induced by $\bm D$ from $\bm T$}
	shall be the pair \label{n:IDTE}
	\( I(D,T,E) = (E,\Ical), \)
	where $X\in \Ical$ if and only if there is a routing $R\colon X\routesto T_0$ in $D$, such that $\SET{p_{-1}~\middle|~p\in R}\in \Tcal$.
	In other words $X$ is independent in $I(D,T,E)$ if and only if there is a linking from $X$ onto an independent set of $T$ in $D$.
\end{definition}

\noindent It is a result of J.H.~Mason that this generalization of Definition~\ref{def:gammoid} always produces a matroid.

\begin{theorem}[\cite{M72}, Theorem 1.1]\label{thm:mason11}\PRFR{Mar 27th}
	Let $T=(T_0,\Tcal)$ be a matroid, $D=(V,A)$ be a digraph with $T_0\subseteq V$, and let $E\subseteq V$ be any set. Then $I(D,T,E)$ is indeed a matroid.
\end{theorem}

\noindent \PRFR{Mar 27th}
For a proof\footnote{We chose to omit the full proof, because for pure logical reasons,
we do not need this theorem for our purposes in this work: the construction in
Lemma~\ref{lem:digraphInducedGammoidIfTisGammoid} works if we define a triple $(M,D,T)$ to be a digraph induction whenever $M$ and $T$ are matroids, 
such that $X$ is independent in $M$ if and only if it can be linked to an independent set of $T$ in $D$. Theorem~\ref{thm:mason11} states that for every matroid $T$ and every digraph $D$, there is a matroid $M$ on every subset of the vertex set of $D$ such that $(M,D,T)$ is such a digraph induction.}, see \cite{M72}, p.58; J.H.~Mason constructs the linkage system with respect to two routings from $X$, and $Y$, respectively,
onto independent subsets of $T_0$
and then uses Theorem~\ref{thm:bipartiteInduction} in order to show that if $\left| X \right| < \left| Y \right|$ the augmentation axiom {\em (I3)} holds 
for $I(D,T,V)$. The axioms {\em (I1)} and {\em (I2)} follow easily from the definition, thus $I(D,T,V)$ is a matroid, 
and consequently $I(D,T,E) = I(D,T,V)\restrict E$ is a matroid, too. Now let us present the general form of the non-trivial implication of
 Lemma~\ref{lem:deflationLemma}.

\begin{lemma}\label{lem:digraphInducedGammoidIfTisGammoid}\PRFR{Mar 27th}
	Let $T=(T_0,\Tcal)$ be a matroid, $D=(V,A)$ be a digraph with $T_0\subseteq V$, and let $E\subseteq V$. %~ be a set of vertices such that $T_0\subseteq E$.
	If $T$ is a gammoid, then $I(D,T,E)$ is a gammoid.
\end{lemma}
\begin{proof}\PRFR{Mar 27th}
	%If $T$ is not a gammoid, then $I(D,T,E)$ cannot be a gammoid because we have $T = D(D,T,E)\restrict T_0$, 
	%i.e. $T$ is a minor of $I(D,T,E)$ and the class of gammoids is closed under minors (Theorem~\ref{thm:GammoidsClosedMinorsDuality}).
	If $T$ is a gammoid, then there is a representation $(D',S',T_0)$ with $D'=(V',A')$ 
	such that $T = \Gamma(D',S,T_0)$. Let $v\mapsto \tilde{v}$ denote a renaming scheme such that the renamed vertices are disjoint from $V$,
	i.e.
	$\tilde V' \cap V = \emptyset$ where $\tilde V' = \SET{\tilde{v} ~\middle|~ v\in V'}$. 
	Let $\tilde{X} = \SET{\tilde{x}~\middle|~ x\in X}$ for all $X\subseteq V'$.
    We define the digraph $D_I = (V\disunion \tilde{V'}, A_I)$ where 
    \[ A_I = A\disunion \SET{\left(\tilde u,\tilde v\right)~\middle|~ (u,v)\in A'} \disunion \SET{\left(t,\tilde t\right)~\middle|~t\in T_0}.
    \]
    We show that $I(D,T,E) = \Gamma(D_I,\tilde{S},E)$. First, observe that $T_0$ is a $V$-$\tilde S$-separator in $D_I$, because every arc between $V$ and
     $\tilde V'$ leaves $t$ and enters $\tilde t$ for some $t\in T_0$. Therefore there is a routing $R_I\colon X\routesto \tilde S$ in $D_I$ with $X\subseteq E$
     if and only if there are a linking $R\colon X\routesto T_X$ with $T_X \subseteq T_0$ in $D$ and a routing $R'\colon T_X\routesto S$ in $D'$ ---
     we may construct $R$ and $R'$ from $R_I$ by splitting every $p_I\in R_I$ into its $V$- and $\tilde V'$-components. Conversely, if we have a
     pair of routings $R$ and $R'$ such that $\SET{p_{-1}~\middle|~p\in R} = \SET{p'_1 ~\middle|~p'\in R'}$,
     then we may obtain $R_I$ by joining the corresponding paths in $D_I$.
      The latter routing $R'$ exists if and only if $T_X\in\Tcal$,
      therefore $X\subseteq E$ is independent in $\Gamma(D_I,\tilde S, E)$ if and only if $X$ may be linked onto an independent subset of $T_0$ with respect to $T$, i.e. if and only if $X$ is independent in $I(D,T,E)$.
\end{proof}

\noindent The following corollary is the corresponding generalization of Lemma~\ref{lem:deflationLemma}.

\begin{corollary}\PRFR{Mar 27th}
	Let $T=(T_0,\Tcal)$ be a matroid, $D=(V,A)$ be a digraph with $T_0\subseteq V$, such that every vertex $t\in T_0$ is a sink in $D$.
	 Let further $E\subseteq V$ such that $T_0\subseteq E$.
	Then $T$ is a gammoid if and only if $I(D,T,E)$ is a gammoid.
\end{corollary}
\begin{proof}\PRFR{Mar 27th}
	If $T$ is a gammoid, then $I(D,T,E)$ is a gammoid (Lemma~\ref{lem:digraphInducedGammoidIfTisGammoid}).
	Since $\Pbf(D;t,v) =\emptyset$ for all $t\in T_0$ and all $v\in V\BSET{t}$, we obtain that
	$X$ is independent in $I(D,T,E)$ if and only if $X$ is independent in $T$
	for all $X\subseteq T_0$. Thus $I(D,T,E)\restrict T_0 = T$, and, consequently, if $I(D,T,E)$ is a gammoid, then so is $T$ 
	(Theorem~\ref{thm:GammoidsClosedMinorsDuality}).
\end{proof}
\clearpage
% -*- root: ../thesis.tex -*-

\section{The Recognition Problem}

First, we give a formal definition of what we mean when we talk about the problem of {\em recognizing a gammoid}.

\begin{definition}\PRFR{Feb 15th}
	Let $\Mcal$ be a class of matroids.
	The \deftext{gammoid recognition problem for $\bm \Mcal$} -- or shorter $\mathrm{Rec}\Gamma_\Mcal$\label{n:RecGM} --
	is the problem of computing the image of $M\in \Mcal$ under the class-map\label{n:GammaM}
	\[ \Gamma_\Mcal \colon \Mcal \maparrow \SET{0,1},\quad M\mapsto \begin{cases}[r] 1 & \text{if~} M\text{~is a gammoid,}\\
																					 0 & \text{otherwise.} \end{cases} \]
	The elements $M\in \Mcal$ are called the \deftext[instances of RecGM@instances of $\mathrm{Rec}\Gamma_\Mcal$]{instances of $\bm{\mathbf{Rec}\Gamma_\Mcal}$}.
\end{definition}

\PRFR{Mar 7th}
\noindent In less formal words, the gammoid recognition problem is the problem that given an instance of a matroid $M$,
determine whether $\Gamma_\Mcal(M) = 1$ or $\Gamma_\Mcal(M) = 0$ by application of some algorithm. Thus we are interested in
algorithms that compute $\Gamma_\Mcal$ and naturally we are also interested in the run-time complexity of 
those algorithms as well as lower bounds for the complexity of these algorithms.
Obviously, there is no constant-time algorithm for the computation of $\Gamma_\Mcal(M)$, therefore we would like to fix
a certain way to encode a matroid $M$ --- up to renaming elements of its ground set,
yet preserving the implicit linear order of its ground set $E=\dSET{e_1,e_2,\ldots, e_n}$.

\begin{definition}\PRFR{Mar 7th}
	Let $n,r\in \N$ with $n\geq r$. We fix the bijection\label{n:kth}
	\begin{align*}
		 \kth(n,r)\colon \SET{1,2,\ldots,\binom{n}{r}} \maparrow &\binom{\SET{1,2,\ldots,n}}{r}
	\end{align*}
	with the defining property that for all $i,j\in \N$ with $1\leq i,j \leq \binom{n}{r}$ % with $i\not= j$
	we have
	\[ \min\left(\vphantom{A^A} \kth(n,r)(i) \bigtriangleup  \kth(n,r)(j)  \right) \,\,\in\,\, \kth(n,r)(i) \quad\Longleftrightarrow\quad i < j,\]
	where $\bigtriangleup$ denotes the symmetric difference of sets.
	In words, we enumerate all $r$-elementary subsets of $\SET{1,2,\ldots,n}$ in ascending order with respect to the linear order that
	relates a subset $A$ with a subset $B$ whenever the smallest element in $\left( A\cup B \right) \BS \left( A\cap B \right)$
	belongs to $A$.
\end{definition}

\noindent For $n \geq r+1$ we have $\kth(n,r)(1) = \SET{1,2,\ldots,r}$, $\kth(n,r)(2) = \SET{1,2,\ldots,r-1,r+1}$,
and $\kth(n,r)\left( \binom{n}{r} \right) = \SET{n-r, n-r+1,\ldots,n}$.

\needspace{4\baselineskip}

\begin{definition}\label{def:bitsM}\PRFR{Mar 7th}
	Let $M=(E,\Ical)$ be a matroid and let $E = \dSET{e_1,e_2,\ldots,e_n}$ bear an implicit linear order.
	We define the \deftext[binary encoding of $M$]{binary encoding of $\bm M$} to \label{n:bMenc}
	be the vector \[ \bits(M) = (\bits(M,i))_{i=1}^{N} \in \SET{0,1}^{N} \]
	where $N = \left| E \right| + 2 + \binom{\left| E \right|}{\rk_M(E)}$ is the encoding length of $M$
	and where
	\[ \bits(M,i) = \begin{cases}[r]
						1 & \quad \text{if~} i \leq \rk_M(E),\\
						0 & \quad \text{if~} i = \rk_M(E) + 1, \\
						1 & \quad \text{if~} \rk_M(E) + 1 < i \leq \left| E \right| + 1, \\
						0 & \quad \text{if~} i = \left| E \right| + 2, \\
						1 & \quad \text{if~} i > \left| E \right| + 2 \txtand \kappa\left( i - \left| E \right| - 2 \right)\in \Ical,\\
						0 & \quad \text{otherwise,}
					\end{cases}\]
	where $\kappa(k) = \SET{e_i\in E ~\middle|~ i\in \kth(\left| E \right|, \rk_M(E))(k)}$.
	In other words, $\bits(M)$ consists of a unary encoding of $\rk_M(E)$, followed by a unary encoding of $\left| E \right| - \rk_M(E)$,
	followed by $\binom{\left| E \right|}{\rk_M(E)}$ bits encoding which of the $\rk_M(E)$-elementary subsets of $E$ are bases of $M$,
	in the ascending order with respect to the implicit linear order on $E$.
	Furthermore, the \deftext[encoding length of $M$]{encoding length of $\bm M$} shall be denoted by \label{n:encM}
	\[ \Nbf(M) = \left| E \right| + 2 + \binom{\left| E \right|}{\rk_M(E)}. \qedhere \]
\end{definition}

\begin{remark}\PRFR{Mar 7th}
	Clearly, we can restore a matroid isomorphic to $M=(E,\Ical)$ from $\bits(M)$.
	Furthermore, the laws of the binomial coefficients yield that
	$$\Nbf(M) = \Nbf(M^\ast)$$ since $\binom{n}{k} = \binom{n}{n-k}$, and for all $X\subsetneq E$
	\[\Nbf(M\restrict X) = \Nbf(M\contract X) < \Nbf(M)\]
	since $\binom{n}{k} = \binom{n-1}{k-1} + \binom{n-1}{k}$.
	When $\left| E \right|\geq 4$, a rather rough estimate is \( \Nbf(M) \leq 2^{\left| E \right|}. \)
\end{remark}

\needspace{4\baselineskip}
\begin{remark}\PRFR{Mar 7th}
	R.~Pendavingh and J.~van~der~Pol give the following
	 lower bound for the number of matroids of rank $r$ on an $n$-elementary ground set in \cite{PP17}, 
	 let $s_{r,n}$ denote this lower bound. Then 
	 \[ \log \left( s_{n,r} \right) \geq \frac{1}{n-r+1}\cdot \binom{n}{r}\cdot \log\left( c^{1-r}(n-r+1)(1+o(1)) \right) \]
	 for some constant $c$ independent of $r$ and $n$
	 (Lemma~9~(3), \cite{PP17} p.4). Thus, if we write a big list of all matroids with $n$ elements and rank $r$, and then
	 use the corresponding list index, encoded as a binary number, in order to represent the base vector of the matroid, 
	 we would still have the binomial $\binom{\left| E \right|}{\rk_M(E)}$ as a factor in
	 the encoding length. Therefore our encoding $\bits(M)$ from Definition~\ref{def:bitsM} may be considered not excessively-bloated.
\end{remark}

% \needspace{3\baselineskip}
% \begin{algorithm}  \textbf{Verify Base Axioms}\\

% \noindent
% \begin{tabular}{rl}
% 	\textbf{Input}& {\em(1)} $r$ unary encoded, \\
% 	&{\em(2)} $e$ unary encoded,\\
% 	&{\em(3)} a bit-vector $b\in \SET{0,1}^{\binom{\SET{1,2,\ldots,r+e}}{r}}$.\\
% 	\textbf{Output}& $1$, if $b$ corresponds to the family of bases of a matroid with $r+e$ elements\\& $\quad $ and rank $r$,\\
% 	& $0$ otherwise.\\
% 	& \\
% 	& if $r = 0$ then return $1$.\\
% 	& if $r > 0$ and $b \equiv 0$ then return $0$.\\
% 	& $I_0 := $ any $r$-elementary subset of $\SET{1,2,\ldots,r+e}$ such that $b_{I_0} \not= 0$\\
% 	& initialize $l \in \SET{0,1}^{\binom{\SET{1,2,\ldots,r+e}}{r}}$ such that $l_{I_0} = 1$ and $l_J = 0$ for all $J\not= I_0$
% \end{tabular}

% \end{algorithm}

\PRFR{Mar 7th}
\noindent First, we shall examine how easy it is to extract matroid information from an encoded matroid.
Throughout this work, we assume that checking, whether a set of the correct cardinality is a base of $M$, can be
done in $O(1)$ time by reading the corresponding bit from $\bits(M)$.

\needspace{3\baselineskip}
\begin{algorithm}\PRFR{Mar 7th}\label{alg:indep}\index{algorithm!independence}  \textbf{Check For Independence}\\

\noindent 
\begin{tabularx}{\textwidth}{rl}
	\textbf{Input}& {\em(1)} A matroid $M=(E,\Ical)$ given by $\bits(M)$.\\
	&{\em(2)} A subset $X\subseteq E$, given by a vector of $2^{\SET{1,2,\ldots,\left| E \right|}}$.\\
	\textbf{Output}& $1$ if $X\in \Ical$, $0$ otherwise.\\
	& \\
	& \ttfamily for $i = 1 \ldots \binom{\left| E \right|}{\rk_M(E)}$ do\\*
	& \ttfamily  $\quad $ if $\bits(M,\left| E \right|+1+i) = 1 \text{\rmfamily ~and~} X \subseteq \kth\left( \left| E \right|, \rk_M(E) \right)(i)$ 
		then do \\*
	& \ttfamily  $\quad $ $\quad $ return $1$ and stop\\*
	& \ttfamily  $\quad $ end if $\bits(M,\left| E \right|+1+i) = 1 \text{\rmfamily ~and~} X \subseteq \kth\left( \left| E \right|, \rk_M(E) \right)(i)$ \\
	& \ttfamily $\quad $ next $i$ \\*
	& \ttfamily end for $i$ \\*
	& \ttfamily return $0$.
\end{tabularx}

\noindent
In order to check for independence we have to test whether $X$ is the subset of a base of $M$.
Therefore we iterate over at most $O(\Nbf(M))$ base candidates $B$ and for each  candidate we check whether $B\in\Bcal(M)$ and
$B \subseteq X$, which can be done in $O(\left| E \right|)$ bit-comparisons.
Thus the overall run-time is $O(\left| E \right|\cdot \Nbf(M)) = O(n^2)$ 
where $n= \Nbf(M) + \left| E \right|$ is the total bit-length of the input.
\end{algorithm}

\begin{proof}[Proof of correctness]\PRFR{Mar 7th}
Lemma~\ref{lem:augmentation} states that every independent set $X\in \Ical$ can be extended to a base $B\in \Bcal(M)$.
The execution invariant at the ``{\ttfamily next $i$}'' instruction is that
 none of the bases in $\Bcal(M)\cap \SET{\SET{e_k ~\middle|~ k\in \kth(\left| E \right|,\rk_M(E)(j)} ~\middle|~ j \leq i}$ contains $X$.
 Thus the invariant at the ``{\ttfamily end for $i$}''-instruction is
 that there is no $B\in\Bcal(M)$ with $X\subseteq B$. And then the return value is $0$.
 Otherwise, if the algorithm returns $1$, then the set $\SET{e_k ~\middle|~ k\in \kth\left( \left| E \right|, \rk_M(E) \right)(i)}$
 is a base of $M$ that proves $X\in \Ical$.
\end{proof}

\noindent \textbf{In order to reduce the technicalities of notation}, we identify the ground set $E=\dSET{e_1,e_2,\ldots,e_n}$ with
the set $\dSET{1,2,\ldots,n}$ through the bijection $i\mapsto e_i$; therefore we identify $\SET{e_i~\middle|~i\in \kth(n,r)(j)}$ with $\kth(n,r)(j)$ for the treatise of algorithms.

\needspace{3\baselineskip}
\begin{algorithm}\PRFR{Mar 7th}\index{algorithm!rank}  \textbf{Compute the Rank}\\

\noindent
\begin{tabularx}{\textwidth}{rl}
	\textbf{Input}& {\em(1)} A matroid $M=(E,\Ical)$ given by $\bits(M)$.\\
	&{\em(2)} A subset $X\subseteq E$, given by a vector $2^{\SET{1,2,\ldots,\left| E \right|}}$.\\
	\textbf{Output}& $\rk_M(X)$.\\
	& \\
	& $r := 0$ \\*
	& \ttfamily for $i = 1 \ldots \binom{\left| E \right|}{\rk_M(E)}$ do\\*
	& \ttfamily $ \quad $ $r := \max\SET{\vphantom{A^A}r,\,\bits(M, \left| E \right| + 1 + i)\cdot 
	\left| \kth\left( \left| E \right|, \rk_M(E) \right)(i) \cap X \right|}$\\*
	& \ttfamily $\quad $ next $i$ \\*
	& \ttfamily end for $i$\\*
	& \ttfamily return $r$.
\end{tabularx}

\PRFR{Mar 7th}
\noindent
In order to compute the rank, we have to do $O(\Nbf(M))$ iterations of the main loop which consists of one multiplication, 
one comparison of values bounded above by $\left| E \right|$,
 possibly a copy operation of these values,
 and possibly a calculation of the intersection between 
$\kth(\left| E \right|, \rk_M(E))(i)$ and $X$, which can be done in $O(\left| E \right|)$. 
Thus the overall run-time is $O(\left| E \right|\cdot \Nbf(M)) = O(n^2)$ 
where $n= \Nbf(M) + \left| E \right|$ is the total bit-length of the input.
\end{algorithm}

\begin{proof}[Proof of correctness]\PRFR{Mar 7th}
Lemma~\ref{lem:augmentation} and Definition~\ref{def:rank} yield that the rank of $X\subseteq E$ equals the maximum cardinality of
the intersection of $X$ with a base $B\in\Bcal(M)$.
Clearly, the invariant for the value of $r$ at the ``{\ttfamily next $i$}'' instruction is
 $$r = \max\SET{\vphantom{A^A}\left| X\cap B \right| ~\middle|~ B\in \Bcal(M) \cap \SET{\kth(\left| E \right|,\rk_M(E)(j) ~\middle|~ j \leq i}}.$$
Therefore the invariant at the ``{\ttfamily end for $i$}'' instruction is
$$r = \max\SET{\vphantom{A^A}\left| X\cap B \right| ~\middle|~ B\in \Bcal(M)},$$
 thus the returned value $r = \rk_M(X)$ is correct.
\end{proof}

\needspace{3\baselineskip}
\begin{algorithm}\PRFR{Mar 7th}\index{algorithm!closure}  \textbf{Compute the Closure}\\

\noindent
\begin{tabularx}{\textwidth}{rl}
	\textbf{Input}& {\em(1)} A matroid $M=(E,\Ical)$ given by $\bits(M)$.\\
	&{\em(2)} A subset $X\subseteq E$, given by a vector of $2^{\SET{1,2,\ldots,\left| E \right|}}$.\\
	\textbf{Output}& $\cl_M(X)$.\\
	& \\
	& $C := X$ \\*
	& $r := \rk_M(X)$ \\*
	& \ttfamily for $e \in E\BS X$ do \\*
	& \ttfamily $\quad $ if $\rk_M(X\cup\SET{e}) = r$ then $C := C \cup \SET{e}$ \\*
	& \ttfamily $\quad $ next $e$\\*
	& \ttfamily end for $e$\\*
	& \ttfamily return $C$.
\end{tabularx}

\PRFR{Mar 7th}
\noindent
In order to compute the closure, we have to compute the ranks of subsets of $E$ precisely $\left| E\BS X \right| + 1$ times,
therefore we accumulate a running time of $O(\left| E \right|\cdot n^2) = O(n^3)$ where $n= \Nbf(M) + \left| E \right|$ is the total bit-length of the input.
\end{algorithm}

\begin{proof}[Proof of correctness]\PRFR{Mar 7th}
	First, we show that $$\cl(X) = X\cup \SET{e\in E\BS X ~\middle|~\vphantom{A^A} \rk(X\cup\SET{e}) = \rk(X)}.$$
	The closure of $X$ is the intersection
	of all flats $F\in \Fcal(M)$ with $X\subseteq F$ (Definition~\ref{def:clM}),
	and flats are those sets $F\subseteq E$, such that $\rk(F\cup\SET{e}) > \rk(F)$
	holds for all $e\in E\BS F$ (Definition~\ref{def:FcalM}).
	Since $\rk(X) = \rk(\cl(X))$ (Lemma~\ref{lem:clKeepsRank}), we obtain that for all $x\in E\BS X$ with $\rk(X\cup\SET{x})=\rk(X)$,
	we have the implication $X \subseteq F \Rightarrow x\in F$ for all flats $F\in\Fcal(M)$. Consequently, $x\in \cl(X)$ whenever
	$\rk(X\cup\SET{x})=\rk(X)$. On the other hand, if $\rk(X\cup\SET{x}) > \rk(X)$, then $\cl(X\cup\SET{x})\in \Fcal(M)$ is the
	smallest flat that contains $X\cup\SET{x}$, but $\rk(\cl(X\cup\SET{x})) > \rk(\cl(X))$, thus $x\notin \cl(X)$ whenever 
	$\rk(X\cup\SET{x})\not= \rk(X)$.

	\noindent
	Let $E\BS X = \dSET{e_1,e_2,\ldots,e_k}$ with the implicit order of occurrence in the ``{\ttfamily for $e$}''-loop,
	and let $i\in \SET{1,2,\ldots,k}$ be the index of the iteration of the loop corresponding to $e$, i.e. $e_i = e$.
	 The invariant at the 
	``{\ttfamily next $e$}''-instruction with regard to $C$ is $C= \cl_M(X) \cap \left( X\cup\SET{e_1,e_2,\ldots,e_i} \right)$.
	Thus the invariant at the ``{\ttfamily end for $e$}''-instruction is
	 $C = \cl_M(X) \cap \left( X\cup \left( E\BS X \right) \right) = \cl_M(X)$.
\end{proof}

\needspace{3\baselineskip}
\begin{algorithm}\index{algorithm!test for strict gammoid}\label{alg:strictGammoidTest}\PRFR{Mar 7th}  \textbf{Naive Test for Strict Gammoids}\\

\noindent
\begin{tabularx}{\textwidth}{rl}
	\textbf{Input}& {\em(1)} A matroid $M=(E,\Ical)$ given by $\bits(M)$.\\
	%&{\em(2)} A subset $X\subseteq E$, given by a vector $2^{\SET{1,2,\ldots,\left| E \right|}}$.\\
	\textbf{Output}& $1$, if $\alpha_M \geq 0$, or $0$ otherwise.\\
	& \\
	& \ttfamily initialize $\alpha_M \in \Z^{2^E}$ with $\alpha_M \equiv 0$ \\
	& \ttfamily initialize $\Fcal \in \SET{0,1}^{2^E}$ with $\Fcal \equiv 0$ \\
	& \ttfamily for $k = 0 \ldots \left| E \right|$ do\\*
	& \ttfamily $\quad $ for $X \in \binom{E}{k}$ do\\*
	& \ttfamily $\quad $$\quad $ $a := k - \rk_M(X)$\\
	& \ttfamily $\quad $$\quad $ for $Y \subsetneq X$ do\\*
	& \ttfamily $\quad $$\quad $ $\quad $ if $\Fcal(Y) = 1$ then do\\*
	& \ttfamily $\quad $$\quad $ $\quad $ $\quad $ $a := a - \alpha_M(Y)$\\*
	& \ttfamily $\quad $$\quad $ $\quad $ $\quad $ if $a < 0$ then do \\*
	& \ttfamily $\quad $$\quad $ $\quad $ $\quad $ $\quad $ return $0$ and stop\\*
	& \ttfamily $\quad $$\quad $ $\quad $ $\quad $ end if $a < 0$\\*
	& \ttfamily $\quad $$\quad $ $\quad $ end if $\Fcal(Y) = 1$\\*
	& \ttfamily $\quad $$\quad $ $\quad $ next $Y$\\*
	& \ttfamily $\quad $$\quad $ end for $Y$\\
	& \ttfamily $\quad $$\quad $ $\alpha_M(X) := a$\\*
	& \ttfamily $\quad $$\quad $ if $X = \cl_M(X)$ then $\Fcal(X) := 1$\\*
	& \ttfamily $\quad $$\quad $ next $X$\\*
	& \ttfamily $\quad $ end for $X$\\*
	& \ttfamily $\quad $ next $k$\\*
	& \ttfamily end for $k$\\
	& \ttfamily return $1$.
\end{tabularx}

\noindent
The algorithm calculates $\alpha_M$ bottom-up using the recurrence relation and simultaneously keeping track of the family of flats of $M$.
If $\alpha_M < 0$ at some point, then the algorithm stops early with a negative answer. Otherwise we have to calculate all values of $\alpha_M$ in order to be sure that $M$ is a strict gammoid, therefore iterating $2^{\left| E \right|}$ different values of $X$.
All values that are assigned to $\alpha_M(X)$ are non-negative integers that are bounded by $\left| X \right| - \rk_M(X)$, because
the algorithm stops if $a<0$ before assigning the negative value to $\alpha_M(X)$. For the same reason we 
have $\left| a \right| \leq \left| X \right|$. Thus the calculation of the correct value of $\alpha_M(X)$ 
needs at most $2^{\left| X \right|}$ subtractions, each with a run-time in $O(\log(\left| X \right|))$, 
and $2^{\left| X \right|}$ tests of the flat property in $O(1)$. In order to 
determine the value of $\alpha_M(X)$ for a single instance $X$,
 we need $O\left( \left( \log (\left| X \right|) + 1  \right) \cdot 2^{\left| X \right|} \right) = O\left(\log(\left| E \right|)\cdot 2^{\left| E \right|}\right)$-time. 
 The flat book-keeping needs one closure operation that takes
$O(\left| E \right|^2\cdot \Nbf(M))$, and $\left| E \right|$ bit-comparisons, thus the book-keeping is in $O\left( \left| E \right|^2\cdot \Nbf(M) \right)$. If $M$ is a strict gammoid, 
this has to be done for all $2^{\left| E \right|}$ subsets $X\subseteq E$, 
thus the total run-time is 
\[O\left(\log(\left| E \right|)\cdot 2^{2 \left| E \right|} + \left| E \right|^2\cdot \Nbf(M)\cdot 2^{\left| E \right|}\right) = O\left(\log(n)\cdot 2^{2n}\right)\] where $n= \Nbf(M)$ is the bit-length of the input. 
\end{algorithm}

\begin{proof}[Proof of correctness]\PRFR{Mar 7th}
It suffices to show that the algorithm correctly computes the values $\alpha_M(X)$, and that the algorithm returns $1$, if $\alpha_M(X)\geq 0$ holds for all $X\subseteq E$, and $0$ otherwise (Corollary~\ref{cor:MasonAlpha}).
Let $\Xcal = \dSET{X_1,X_2,\ldots, X_{2^{\left| E \right|}}}$ be the family of all subsets of $E$ in the order of occurrence with respect to the
``{\ttfamily for $X$}''-instruction, and let $i\in \SET{1,2,\ldots,2^{\left| E \right|}}$ such that $X = X_i$.
 The invariant at the ``$ a := k - \rk_M(X)$''-instruction is, that for all $j\in\SET{1,2,\ldots,i-1}$ the value of $\alpha_M(X_j)$
 is correctly assigned and non-negative,
 and that we have that $\Fcal(X_j) = 1$ if and only if $X_j$ is a flat. Furthermore, $a = \left| X_i \right| - \rk_M(X_i)$.
 Now let $\Ycal = \dSET{Y_1,Y_2,\ldots,Y_K}$ be the proper subsets of $X$ in their order of occurrence with respect to the
``{\ttfamily for $Y$}''-instruction, and let $k\in\SET{1,2,\ldots,K}$ be the index such that $Y = Y_k$.
The invariant at the ``{\ttfamily next $Y$}''-instruction is
 $$a = \left| X_i \right| - \rk_M(X_i) - \sum_{F\in\Fcal(M)\cap \SET{Y_1,Y_2,\ldots,Y_k},\, F\subsetneq X} \alpha_M(F).$$
Thus the invariant at the ``{\ttfamily end for $Y$}''-instruction is $a = \alpha_M(X_i) \geq 0$,
and the invariant at the ``{\ttfamily return $1$}''-instruction is that $M$ is a strict gammoid.
The invariant at the ``{\ttfamily return $0$ and stop}''-instruction is that
$$ a = \left| X_i \right| - \rk_M(X_i) - \sum_{F\in\Fcal(M)\cap \SET{Y_1,Y_2,\ldots,Y_k},\, F\subsetneq X} \alpha_M(F) < 0,$$
and since $\alpha_M(Y)\geq 0$ for all $Y\subsetneq X$, we may conclude that $\alpha_M(X_i) \leq a < 0$, which implies that
$M$ is not a strict gammoid.
Therefore the output of the algorithm is $1$ if $M$ is a strict gammoid, and $0$ if $M$ is not a strict gammoid.
\end{proof}

\begin{corollary}\label{cor:decideTransversalOrStrictGammoid}\PRFR{Mar 7th}
	Given a matroid $M$ via its encoding $\bits(M)$, we can decide whether $M$ is a transversal matroid and whether $M$ is a strict gammoid in
	$O\left(\log(\Nbf(M))2^{2\Nbf(M)}\right)$ time.
\end{corollary}
\begin{proof}\PRFR{Mar 7th}
	We may use Algorithm~\ref{alg:strictGammoidTest} on $M$ to test whether $M$ is a strict gammoid,
	and on $M^\ast$ in order to test whether $M$ is a transversal matroid. The encoding $\bits(M^\ast)$ can be obtained in $O(\Nbf(M))$ time:
	the location of the first zero has to be moved from $\rk_M(E)+1$ to $\left| E \right| - \rk_M(E) + 1$, and the encoding of the bases in $\bits(M)$ has to be brought in reverse order to obtain an encoding of $M^\ast$. Then we can test whether $M^\ast$ is a strict gammoid to obtain the result (Lemma~\ref{lem:dualOfStrictGammoidIsTransversal}).
\end{proof}

% -*- root: ../thesis.tex -*-

\subsection{Special Cases}

\PRFR{Mar 7th}
In this section, we give a quick overview over some special classes of matroids
and gammoids, where there is an easy way to answer the question
whether a matroid, that exhibits the additional properties,
is a gammoid with other special properties --- or whether it is not.

\begin{proposition}[\cite{IP73}, Proposition~4.8 and Corollary~4.9]\label{prop:cornerCases}\PRFR{Feb 15th}
	Let $M=(E,\Ical)$ be a matroid. Then
	\begin{enumerate}\ROMANENUM 
		\item If $\rk_M(E) \leq 2$, then $M$ is a strict gammoid.
		\item If $\rk_M(E) = 3$, then $M$ is a gammoid if and only if $M$ is a strict gammoid.
		\item If $\rk_M(E) = \left| E \right| - 3$, then $M$ is a gammoid if and only if $M$ is a transversal matroid.
		\item If $\rk_M(E) \geq \left| E \right| - 2$, then $M$ is a transversal matroid.
	\end{enumerate}
\end{proposition}

\PRFR{Mar 7th}
\noindent We omit the proof here as it does not provide any further guidance for the unconstrained problem of deciding whether
a given matroid is a gammoid or not. The reader interested in a proof of this proposition should read A.W.~Ingleton and M.J.~Piff's paper \cite{IP73}.
Certainly, we should keep in mind the consequence of this proposition: We can expect the most general flavor of gammoids to unfold only
 with matroids $M=(E,\Ical)$ where $4 \leq \rk_M(E) \leq \left| E \right| - 4$. For matroids with
  $\rk_M(E) \in \SET{0,1,2,\left| E \right|-2,\left| E \right|-1,\left| E \right|}$, the answer is always that $M$ is a gammoid. For
  $\rk_M(E)\in\SET{3,\left| E \right|-3}$, we may use Mason's $\alpha$-criterion with respect to $M$, or $M^\ast$, respectively,
  in order to decide whether $M$ is a gammoid (Corollary~\ref{cor:MasonAlpha}) in $O\left(\log(n)2^{2n}\right)$ time (Corollary~\ref{cor:decideTransversalOrStrictGammoid}).

\bigskip 

\PRFR{Mar 7th}
\noindent
 Instead of limiting the class of allowed input instances, we also might consider the related problem of determining
whether a given matroid $M$ belongs to some subclass $\Gcal_P$ of the class of gammoids, where $\Gcal_P$ consists of all gammoids
that have the additional property $P$. It is folklore, that the problem of deciding whether a given matroid $M$ belongs to 
a minor-closed class of matroids $\Mcal_P$, which is characterized by finitely many excluded
minors, has a solution algorithm that runs in polynomial time. First, we present the following theorem, that gives us a hint where
such an excluded minor must appear in any $M\notin\Mcal_P$.

\needspace{4\baselineskip}

\begin{theorem}[Scum Theorem, \cite{Ox11} p.113]\label{thm:scum}\PRFR{Mar 7th}
	Let $M=(E,\Ical)$ be a matroid and let $N=(E',\Ical')$ be a minor of $M$. There is a subset $Z\subseteq E\BS E'$
	such that $M\contract \left( E\BS Z \right)$ has the same rank as $N$, and such that 
	\[ N =  \left( M\contract \left( E\BS Z \right) \right)\restrict E' .\]
	If $N$ has no loop, then we may choose $Z\in\Fcal(M)$.
\end{theorem}

\noindent
For the proof, please refer to J.G.~Oxley's book \cite{Ox11}. It is easy to see that given such a $Z\subseteq E\BS E'$,
every base $Z' \in \Bcal_M(Z)$ also has the property that $N =  \left( M\contract \left( E\BS Z' \right) \right)\restrict E'$
(Lemma~\ref{lem:contractionBchoice}).

\needspace{3\baselineskip}
\begin{algorithm}\label{alg:isMinor}\PRFR{Mar 7th}\index{algorithm!minor}  \textbf{Test for Minor}\\

\noindent
\begin{tabularx}{\textwidth}{rl}
	\textbf{Input}& {\em(1)} A matroid $M=(E,\Ical)$ given by $\bits(M)$.\\
	& {\em(2)} A matroid $N=(E',\Ical')$ given by $\bits(N)$.\\
	\textbf{Output}& $1$, if $N$ is isomorphic to a minor of $M$,\\
	& $0$, otherwise.\\
	& \\
	& \ttfamily if $\rk_N(E',\Ical') > \rk_M(E,\Ical)$ then return $0$ and stop\\
	& \ttfamily for $Z\in \binom{E}{\rk_M(E)-\rk_N(E')}$ do\\*
	& \ttfamily $\quad$ if $Z\notin \Ical$ then next $Z$\\
	& \ttfamily $\quad$ for every $\phi\colon E' \maparrow E\BS Z \text{\rmfamily~injective map}$ do\\*
	& \ttfamily $\quad$ $\quad$ for $X\in \binom{E'}{\rk_N(E')}$ do\\*
	& \ttfamily $\quad$ $\quad$ $\quad$ if not $X\in \Bcal(N) \Leftrightarrow \phi[X]\cup Z \in \Bcal(M)$ then next $\phi$\\
	& \ttfamily $\quad$ $\quad$ $\quad$ next $X$\\*
	& \ttfamily $\quad$ $\quad$ end for $X$ \\
	& \ttfamily $\quad$ $\quad$ return $1$ and stop\\*
	& \ttfamily $\quad$ end for $\phi$ \\
	& \ttfamily $\quad$ next $Z$ \\*
	& \ttfamily end for $Z$\\
	& \ttfamily return $0$.
\end{tabularx}

\PRFR{Mar 7th}
\noindent
Let $n = \Nbf(M) + \Nbf(N)$, $n_M = \Nbf(M)$, and $n_N = \Nbf(N)$ be the respective encoding lengths.
The ``{\ttfamily for $Z$}''-instruction loops through at most $\binom{\left| E \right|}{\rk_M(E) - \rk_N(E')}$ iterations.
For $k,m\in \N$ with $k \leq m$, we may estimate
\( \binom{m}{k-1} \leq k \cdot \binom{m}{k}\), because we obtain a $(k-1)$-elementary subset of an $m$-elementary set $X$ by first
choosing an $k$-elementary subset of $X$ and then choosing one element to drop. This way we obtain all $(k-1)$-elementary subsets
provided that there is some $k$-elementary subset of $X$. Consequently, we may estimate the number of $Z$-iterations by
\begin{align*}
	\binom{\left| E \right|}{\rk_M(E) - \rk_N(E')} & \leq \left| E \right|^{\rk_N(E')} \cdot \binom{\left| E \right|}{\rk_M(E)}
	= O\left( {\left( n_M  \right)}^{n_N+1} \right).
\end{align*}
The test whether $Z\in\Ical$ can be done in $O\left( \left( n_M  \right)^2 \right)$ (Algorithm~\ref{alg:indep}).
The generation of the injective maps $\phi$ for a single instance of $Z$ as lookup 
tables has a combined run-time in $$O\left(\left| E \right|^{\left| E' \right|}\cdot \log\left( \left| E \right| \right)\right)
=O\left( \left( n_M \right)^{n_N}\cdot \log\left( n_M \right) \right).$$
The \mbox{``{\ttfamily for $\phi$}''}-instruction loops through at 
most $\left| E \right|^{\left| E' \right|}= O\left( \left( n_M \right)^{n_N} \right)$ iterations,
and the \mbox{``{\ttfamily for $X$}''}-instruction loops through at most $O(n_N)$ iterations.
Calculating $\phi[X]$ can be done with $\left| X \right| = \rk_N(E') = O\left( n_N \right)$ table lookups and corresponding bit-set
operations,
calculating $\phi[X]\cup Z$ can be achieved with an $\left| E \right|$-bit bitwise-or operation in $O\left( n_M \right)$-time.
Checking whether $X\in \Bcal(N) \Leftrightarrow \phi[X]\cup Z\in \Bcal(M)$ can be done in $O(1)$-time by bit comparison.
This yields a combined run-time of the algorithm in 
\[ O\left( \left( n_M \right)^{2n_N + 2}\cdot \left( n_N \right)^2 \right) = O\left( n^{2n+4} \right). \qedhere \]
\end{algorithm}

\noindent Thus deciding whether $M$ has a minor isomorphic to $N$ can be done in polynomial time with respect to $\Nbf(M)$ for a
fixed matroid $N$.

\PRFR{Mar 7th}
\begin{proof}[Proof of correctness]
Assume that $M=(E,\Ical)$ has a minor $L=(D,\Jcal)$ that is isomorphic to $N=(E',\Ical')$, 
then there is a set $Z_L\subseteq E\BS D$ such 
that $L = \left( M\contract \left(E\BS Z_L\right) \right)\restrict D$ as guaranteed by the 
Scum Theorem~\ref{thm:scum}. Thus for every base $B_L$ of $Z_L$ in $M$,
we have the property, that a set $X\subseteq D$ is a base of $L$ if and only if $B_L\cup X$ 
is a base of $M$ (Lemma~\ref{lem:contractionBchoice}). Furthermore, 
the minor $L$ is isomorphic to $N$, if and only if there is a matroid isomorphism
$\phi'\colon E' \maparrow D$ between $L$ and $N$, i.e. a bijective map $\phi'$ with the property that
 $\phi'[X]\in \Bcal(L) \Leftrightarrow X\in \Bcal(N)$ holds. Assume that $M$ has a minor $L$ isomorphic to $N$. Let further $\phi'$
 be the corresponding matroid isomorphism, and let $Z_L\subseteq E\BS D$ and $B_L\subseteq Z_L$ 
 be derived from the Scum Theorem~\ref{thm:scum} 
 as above. Since $B_L\subseteq Z_L \subseteq E\BS D$, we have $B_L\cap D = \emptyset$.
 By extension of the codomain of $\phi'$ we obtain an injective map $\hat\phi\colon E'\maparrow E\BS B_L$,
 where $\hat\phi(e') = \phi'(e')$ for all $e'\in E'$. Either the algorithm returns $1$ early, 
 or at some point, the ``{\ttfamily for $Z$}''-instruction starts an iteration where
 $Z = B_L$ since $\left| B_L \right| = \rk_M(Z_L) = \rk_M(E) - \rk_L(D) = \rk_M(E) - \rk_N(E')$. In this iteration, we have
 $Z=B_L\in \Ical$, therefore we enter the ``{\ttfamily for $\phi$}''-loop. Again, either the algorithm returns $1$ early, or
 we reach the iteration where $\phi = \hat\phi$. In this iteration, we have the equivalence $X\in \Bcal(N) \Leftrightarrow \phi[X]\cup Z =
 \phi'[X] \cup B_L \in \Bcal(M)$, therefore we reach the ``{\ttfamily end for $X$}''-instruction. 
 In the next instruction, we return the correct value $1$.

 \noindent Now assume that $M$ has no minor isomorphic to $N$. If the algorithm reaches the ``{\ttfamily return $0$}''-instruction
 the result is correct. We give an indirect argument for this to happen, and assume that the algorithm reaches the 
 ``{\ttfamily return $1$}''-instruction.
 But then the ``{\ttfamily for $X$}''-loop must have finished without reaching the ``{\ttfamily next $\phi$}''-instruction.
 This, together with the property {\em (B2)} that all bases of a matroid have the same cardinality,  implies that
 $X\in \Bcal(N) \Leftrightarrow \phi[X]\cup Z\in \Bcal(M)$ holds for all $X\subseteq E'$.
 Thus $\left( M\contract \left( E\BS Z \right) \right) \restrict \phi[E']$
 is a minor of $M$ isomorphic to $N$, contradicting our assumption that $M$ has no minor isomorphic to $N$.
 Therefore we may conclude that the algorithm returns $0$ if $M$ has no minor isomorphic to $N$. 
\end{proof}

\begin{theorem}\PRFR{Mar 7th}
	Let $\Gcal$ be a minor-closed class of matroids that is characterized by the excluded minors $N_1,N_2,\ldots,N_k$,
	and let $K = \max\SET{\Nbf(N_1),\Nbf(N_2),\ldots,\Nbf(N_k)}$ be the maximal encoding length of the excluded minors.
	If $M=(E,\Ical)$ is an arbitrary matroid and $n = \Nbf(M)$ is its encoding length,
	 then we may decide whether $M\in \Gcal$ in $ O\left( n^{K+2} \right) $-time.
\end{theorem}
\begin{proof}\PRFR{Mar 7th}
	For each $i\in\SET{1,2,\ldots,k}$ we may use Algorithm~\ref{alg:isMinor} in order to test whether $M$ has a minor isomorphic to $N_i$
	in $O\left( n^{2\Nbf(N_i)+2} \right)$-time. On the other hand, $M\in \Gcal$ if and only if $M$ has no minor isomorphic to one of
	the matroids $N_1,N_2,\ldots,N_k$. Thus if Algorithm~\ref{alg:isMinor} returns $1$ for any $N_i$, then $M\notin \Gcal$,
	and if Algorithm~\ref{alg:isMinor} returns $0$ for all $N_i$, then $M\in \Gcal$. Therefore we have to run at most $k$ tests 
	in $O\left( n^{2K+2} \right)$-time in order to decide whether $M\in\Gcal$.
\end{proof}

\noindent
A consequence of this theorem is the following: Let $k\in \N$, then we may 
decide in polynomial time whether a matroid $M$ 
is {\em (a)} a gammoid with $\vK(M) \leq k$ and {\em (b)} a gammoid with $\arcC(M) \leq k$ (Remark~\ref{rem:vKeasy} and Theorem~\ref{thm:arcCquiteEasy}).
\clearpage
% -*- root: ../thesis.tex -*-

\subsection{The General Recognition Problem}

\PRFR{Mar 7th}
\noindent
For the rest of this chapter, \label{sec:generalCase} we let $\Mcal$ be the class of all matroids. Now, let us investigate the problem $\mathrm{Rec}\Gamma_\Mcal$.
In order to present the most obvious algorithm that computes $\Gamma_\Mcal(M)$,
 we need a way to verify whether a given vector $b\in \SET{0,1}^{\binom{n}{r}}$
codes the bases of a rank-$r$ matroid on an $n$-elementary ground set.

\needspace{3\baselineskip}
\begin{algorithm}\PRFR{Mar 7th}\label{alg:matroidTest}\index{algorithm!test base axioms}  \textbf{Test Base Axioms}\\

\noindent
\begin{tabularx}{\textwidth}{rl}
	\textbf{Input}& {\em(1)} $r\in \N$, given as unary encoded bit-stream.\\
	& {\em(2)} $(e-r)\in \N$, given as unary encoded bit-stream.\\
		&{\em(3)} $B\in \SET{0,1}^{\binom{\SET{1,2,\ldots,e}}{r}}$ family of $r$-elementary sets, as a vector of $2^{\binom{e}{r}}$.\\
	\textbf{Output}& $1$, if $B$ is the characteristic vector of a family of bases of a matroid\\
	& ~~~~~ with rank $r$,\\
	& $0$, otherwise.\\
	& \\
	& \ttfamily $g := 0$ \\
	& \ttfamily for $X \in \binom{\SET{1,2,\ldots,e}}{r}$ do \\*
	& \ttfamily $\quad $ if $B(X) = 0$ then next $X$\\*
	& \ttfamily $\quad $ $g := 1$\\
	& \ttfamily $\quad $ for $Y \in \binom{\SET{1,2,\ldots,e}}{r}$ do \\*
	& \ttfamily $\quad $ $\quad $ if $X=Y$ or $B(Y) = 0$ then next $Y$\\*
	& \ttfamily $\quad $ $\quad $ for $x\in X\BS Y$ do\\
	& \ttfamily $\quad $ $\quad $ $\quad $ for $y \in Y\BS X$ do \\*
	& \ttfamily $\quad $ $\quad $ $\quad $ $\quad $ if $B\left( \left(X\BSET{x}  \right)\cup\SET{y} \right) = 1$ then next $x$\\*
	& \ttfamily $\quad $ $\quad $ $\quad $ $\quad $ next $y$\\*
	& \ttfamily $\quad $ $\quad $ $\quad $ end for $y$\\*
	& \ttfamily $\quad $ $\quad $ $\quad $ return $0$ and stop\\
	& \ttfamily $\quad $ $\quad $ end for $x$\\*
	& \ttfamily $\quad $ $\quad $ next $Y$\\
	& \ttfamily $\quad $ end for $Y$\\*
	& \ttfamily $\quad $ next $X$\\
	& \ttfamily end for $X$\\*
	& \ttfamily return $g$.
\end{tabularx}

\PRFR{Mar 7th}
\noindent
Observe that the input resembles the format of an encoding of a rank-$r$ matroid defined on an $e$-elementary ground set,
the total bit-length of the input is $n = 2 + e + \binom{e}{r}$.
A rough estimate of the run-time is the following: the ``{\ttfamily for $X$}''-instruction iterates over $\binom{e}{r}$ sets,
the ``{\ttfamily for $Y$}''-instruction iterates over $\binom{e}{r}$ sets, too, the ``{\ttfamily for $x$}''-instruction iterates over $\leq r$ elements of $X$,
the ``{\ttfamily for $y$}''-instruction iterates over $\leq r$ elements of $Y$. In total, we have to do less than
\[ \binom{e}{r}\cdot \binom{e}{r} \cdot r^2 + \binom{e}{r}\cdot \binom{e}{r} + \binom{e}{r} \]
bit-comparisons involving the vector $B$. Thus the run-time is in 
 $O\left( r^2\cdot \binom{e}{r}^2 \right) = %O\left(r^2 n^2\right) =
  O\left( n^3 \right)$, since $r^2 = O\left( \binom{e}{r} \right) = O(n)$.
\end{algorithm}
\begin{proof}[Proof of correctness]\PRFR{Mar 7th}
Clearly, the input format guarantees that {\em (B2)} holds for all inputs.
For every matroid $M$ of rank $r$ on the ground set $\SET{1,2,\ldots,e}$,
at least one set $X\in \binom{\SET{1,2,\ldots,e}}{r}$ must be a base of $M$, 
and the variable $g$ obviously keeps track of the existence of this set, and consequently, whether {\em (B1)} holds.
In other words, upon reaching the ``{\ttfamily return $g$}''-instruction, $g = 0$ if and only if $B \equiv 0$,
i.e. $B$ is the zero vector.
 Let $X,Y\in\binom{\SET{1,2,\ldots,e}}{r}$. The ``{\ttfamily for $x$}''-instruction is reached for $X$ and $Y$
if and only if $B(X)=B(Y)=1$, i.e. $X$ and $Y$ are supposed to be bases of the input matroid candidate.
The loops ``{\ttfamily for $x$}'' and ``{\ttfamily for $y$}'' test whether $\left( X\BSET{x}\right)\cup{y}$ is a base. If for a given $x\in X\BS Y$
an exchange partner $y\in Y\BS X$ is found, the ``{\ttfamily next $x$}''-instruction is reached. Otherwise, if no $y\in Y\BS X$
has this property for a given $x\in X\BS Y$, the ``{\ttfamily end for $y$}''-instruction is reached. In this case, $B$ violates
the base exchange axiom {\em (B3)} and therefore the input candidate does not correspond to a matroid. In this case,
the output of the algorithm is $0$. When the algorithm reaches the ``{\ttfamily end for $X$}''-instruction,
then it is established that the axiom {\em (B3)} holds for the input candidate. In this case, the input candidate is a matroid
if and only if $B\not\equiv 0$, which is correctly reflected by the value of $g$. Thus the output of the algorithm is $1$ if the input
base vector candidate is a base vector of a matroid of rank $r$ with $e$ elements, and $0$ otherwise.
\end{proof}

\PRFR{Mar 7th}
\noindent Given any matroid $M=(E,\Ical)$, we can combine Remark~\ref{rem:upperBoundForV} and the Algorithms~\ref{alg:matroidTest} 
and \ref{alg:strictGammoidTest}
with the brute-force exhaustive search algorithm
in order to compute $\Gamma_\Mcal(M)$: We generate all candidate families of subsets of $\binom{E'}{\rk_M(E)}$
with respect to a set $E'$ of
cardinality $\rk_M(E)^2\cdot \left| E \right| + \rk_M(E) + \left| E \right|$ with $E\subseteq E'$, that
coincide with $\Bcal(M)$ when intersected with $2^E$. Then we use Algorithm~\ref{alg:matroidTest} in order to determine
whether the generated family corresponds to an actual matroid $M'$ on $E'$. If this is the case, we test whether the
generated matroid $M'$ is a strict gammoid. If so, then $M$ is a gammoid, and $M'$ certifies this.
Otherwise we continue until we exhausted all possibilities to generate candidate families. If we have not found any strict gammoid among
the candidates, then $M$ is not a gammoid. 

\needspace{3\baselineskip}
\begin{algorithm}\PRFR{Mar 7th}\label{alg:gammoidTest}\index{algorithm!brute-force $\Gamma_{\Mcal}(M)$}
\textbf{Compute $\Gamma_{\Mcal}(M)$ (Brute-Force Search)}\\

\noindent
\begin{tabularx}{\textwidth}{rl}
	\textbf{Input}& {\em(1)} $M=(E,\Ical)$ matroid, given by its encoding $\bits(M)$.\\
	\textbf{Output}& $1$ if $M$ is a gammoid,\\
	& $0$ otherwise.\\
	& \\
%	& \ttfamily if $\rk_M(E) \leq 2$ return $1$ and stop\\
	& \ttfamily let $E' := \SET{1,2,\ldots,  \rk_M(E)^2 \cdot \left| E \right| + \rk_M(E) + \left| E \right|}$\\
	& \ttfamily let $\Ecal := \binom{E'}{\rk_M(E)}$\\
	& \ttfamily let $\Ycal := \SET{Y\in 2^{\Ecal} ~\middle|~ Y\cap 2^{\SET{1,2,\ldots,\left| E \right|}} = \Bcal(M)}$\\
	& \ttfamily for $B \in \Ycal$ do \\
	& \ttfamily \quad if {\rmfamily $B$ satisfies the base axioms} then do \\
	& \ttfamily \quad \quad let \rmfamily
		 $N := (E',\SET{I\subseteq E'~\middle|~ \exists X\in \Ecal\colon\,I\subseteq X\txtand B(X) = 1})$\\
	& \ttfamily \quad \quad  if {\rmfamily $N$ is a strict gammoid} then return $1$ and stop\\
	& \ttfamily \quad end if {\rmfamily $B$ satisfies the base axioms}\\
	& \ttfamily end for $B$ \\
	& \ttfamily return $0$. 
\end{tabularx}

\PRFR{Mar 7th}
\noindent
In the worst case $M$ is not a gammoid and we have to iterate over
 $$\displaystyle 2^{\left| \Ecal \right| - \binom{\left| E \right|}{\rk_M(E)}} = O\left( 2^{\left| E' \right|^{\rk_M(E)}} \right)$$
possible values for $B$ in the ``{\ttfamily for $B$}''-loop.
The test whether $B$ corresponds to a matroid can be carried out by Algorithm~\ref{alg:matroidTest} and is possible within
 $$O\left( \rk_M(E)^2\cdot \binom{|E'|}{\rk_M(E)}  \right)\text{-time}.$$
The test whether $N$ is a strict gammoid may be done with Algorithm~\ref{alg:strictGammoidTest} and therefore can be done
in $$O\left(\log\left( \left| E' \right| \right)\cdot 2^{2 \left| E' \right|} + 
\left| E' \right|^2\cdot \Nbf(N)\cdot 2^{\left| E' \right|}\right) \text{-time}.$$
Let $n=\Nbf(M)$ be the bit-length of $\bits(M)$, clearly $\left| E \right| \leq n$ and $\rk_M(E) \leq n$.
%Furthermore $$2^{\left| E'\BS E \right|} = O\left( 2^{n^3+n} \right),$$ 
%since $\left| E'\BS E \right|=\rk_M(E)^2\cdot \left| E \right| + \rk_M(E) \leq n^3 + n$. 
We have
$\left| E' \right| \leq n^3 + 2n$. Thus one test of the base axioms can be done in
$$ O\left( n^2\cdot \binom{n^3 + 2n}{n} \right)=O\left( n^{2.001n} \right) \text{-time.}$$
Since the bit-length $\Nbf(N) \leq 2^{\left| E' \right|} \leq 2^{n^3 + 2n}$, we obtain that 
each strict gammoid test can be
carried out in 
$$ O\left( n^6\cdot 2^{2n^3 + 4n} \right) \text{-time}.$$
R.~Pendavingh and J.~van~der~Pol give the following upper bound for the number of matroids $m(k,r)$  on $k$-elementary ground sets
with rank $r$ in \cite{PP17}
\[ \log(m(k,r)) \leq  \frac{1}{k-r+1} \binom{k}{r} \cdot \log(c\cdot(k-r+1))\]
under the mild condition that $r \geq 3$ and $k \geq r + 12$, and where $c$ denotes a constant factor that does not depend on $k$ or $r$.
We can use this bound in order to determine how often we have to decide whether $N$ is a strict gammoid or not,
let $i$ denote the number of strict gammoid tests, then 
\begin{align*}
   \log(i) & \leq \log(m(n^3 + 2n, \rk_M(E)))
   \\i &  = O\left( 2^ {n^{\left( 3\rk_M(E)-3 \right)}\cdot\log\left( c\cdot\left(  n^3+2n \right) \right)} \right)\\
 %  &  = O\left( n^{3\rk_M(E)-2} \right) \quad = 
 & = O\left( 2^{\left( n^{3\rk_M(E)} \right)} \right)
 = O\left( 2^{n^{3n}} \right)
 .
   \end{align*}
 A naive upper bound for the strict gammoid tests can derived from the number of $B$-iterations, it is
 \[ O\left( 2^{{(n^3+ 2n)}^{3n}} \right) = O\left( \prod_{i=0}^{3n} 2^{{\binom{3n}{i}\cdot(n^{2i + 3n})}} \right), \]
 and it obviously is a looser upper bound than the one derived from \cite{PP17}, since it has a factor $2^{\left( n^{9n} \right)}$.
Thus we have to account
\[ O\left( 2^{{(n^3+ 2n)}^{3n}}\cdot n^{2.001n}  \right) \]
for the base exchange axiom tests and
\[ O\left( n^6\cdot 2^{n^{3n}+2n^3 + 4n}   \right) \]
for the strict gammoid tests. Clearly, 
$ O\left( n^6\cdot 2^{n^{3n}+2n^3 + 4n}   \right) = O\left( 2^{{(n^3+ 2n)}^{3n}}  \right)$,
so we may estimate the total-run time of the algorithm to be in
\[ O\left( 2^{{(n^3+ 2n)}^{3n}}\cdot n^{2.001n}  \right) = O\left( 2^{\left( n^{9n+1}  \right)} \right) . \qedhere \]
\end{algorithm}

\begin{proof}[Proof of correctness]\PRFR{Mar 7th}
%Every matroid with rank $\leq 2$ is a transversal matroid and therefore a gammoid. Now assume that $\rk_M(E) > 2$.
If $M$ is a gammoid, then there is a representation $(D,T,E)$ with $D=(V,A)$ 
such that $\left| V \right| \leq \rk_M(E)^2 \cdot \left| E \right| + \rk_M(E) + \left| E \right| = \left| E' \right|$ 
(Remark~\ref{rem:upperBoundForV}). But then $N = \Gamma(D,T,V)$ is a strict gammoid with the property $N\restrict E = \Gamma(D,T,E) = M$,
thus $\SET{X\subseteq E~\middle|~X\in \Bcal(N)} = \Bcal(M)$. The ``{\ttfamily for $B$}''-instruction iterates through all 
possible and impossible $B = \Bcal(N)$ with that restriction-property, then tests whether $B$ is indeed a matroid base family, and then tests whether the
corresponding matroid is a strict gammoid. If so, the algorithm gives the truthful output $1$. If no such $B$ is found, the algorithm
returns $0$, and by the above consideration we may conclude that in this case $M$ is not a gammoid.
\end{proof}

\noindent
No one would expect that the brute-force method would be of any practical use for determining whether a given matroid is a gammoid,
and it apparently is not.
One obvious problem with Algorithm~\ref{alg:gammoidTest} is that it does not make use of any of the structural results for
matroid extensions, instead it guesses matroid extensions, and this takes so much time that the actual testing for the
strict gammoid property in $O\left( 2^{n^{3n+1}} \right)$ does not have a significant impact on the estimation. The other obvious
problem is that we are not using any information from the $\alpha_M$-vector in order to guide our search for a strict gammoid
extensions of $M$. Furthermore, it seems to be excessive to compute all $\alpha$-vectors of strict gammoid extension candidates
from scratch.\footnote{%%TODO-SCHRANKEN
 There is a second brute-force search method for finding the strict gammoid extension of a gammoid $M=(E,\Ical)$,
which guesses the arcs of a digraph, then calculates the ranks for all $\rk_M(E)$-elementary
subsets of $E$ of the gammoid represented by the candidate digraph:
 if all of these values are correct, then $M$ is a gammoid represented by the candidate digraph, if otherwise we run out of candidates,
 then $M$ is not a gammoid. This method is clearly better than Algorithm~\ref{alg:gammoidTest} as it does eliminate the
 check whether a candidate is indeed a matroid. With the currently known bounds for arcs and vertices there are still
 $\left(\sum_{k=0}^{r}\binom{n^3+2n}{k}\right)^{n^3+n}$ candidate digraphs on $n^3+2n$ vertices (Remark~\ref{rem:upperBoundForV}) with at most $r=O(n)$ leaving arcs per non-target vertex (Theorem~\ref{thm:IPEssentialStars}),
 %and if we calculate the maximal routings in the digraph by chasing arcs from the sink set,
 %then we have to process at most $n^6 + 4n^4 + 4n^2 = O\left( n^7 \right)$ arcs for each candidate. Per discovered maximal routing,
 %we only have to store the vertices used by the paths and which subset is the head of the routing. In the worst case, 
 %we double the number of
 %discovered routing sets when processing an arc --- if we do not account for that fact, we have to process at most $2^{n^3+2n}$
 %sets in arc step, and we have to do at most $n$ bit copies and two status bit flips per processed set,
 %so processing one set
 %and one arc is in $O(n)$.
 % Then we have the following rough upper bound for the run-time
 %\( O\left(n^8 \cdot 2^{\left(  n^6+n^3+4n^4+4n^2+2n\right)}  \right) \) is
 %better than the bound for Algorithm~\ref{alg:gammoidTest}, but since the factor hidden in the $O$ is nowhere near zero,
 so this brute-force method is still not a practical solution --- and the computation of the bases of a given strict gammoid
 involves a more complicated algorithm, therefore we will not give more details for this method. Of course, 
 the refined brute-search method still wastes a lot of
time because usually a considerable amount of digraphs represent the same gammoid.}
We give a straight-forward back-tracking algorithm that also computes $\Gamma_{\Mcal}(M)$. 
%% TODO-SCHRANKEN

\needspace{3\baselineskip}
\begin{algorithm}\PRFR{Mar 7th}\label{alg:gammoidBackTrack}\index{algorithm!backtracking $\Gamma_{\Mcal}(M)$}\textbf{Compute $\Gamma_{\Mcal}(M)$ (Digraph Backtracking)}\\

\noindent
\begin{tabularx}{\textwidth}{rl}
	\textbf{Input}& {\em(1)} $M=(E,\Ical)$ matroid, given by its encoding $\bits(M)$.\\
	\textbf{Output}& $1$ if $M$ is a gammoid,\\
	& $0$ otherwise.\\
	& \\
	& \ttfamily let $V$ be \rmfamily a set with $E\subseteq V$ and $\left| V \right| = \rk_M(E)^2\cdot \left| E \right| + \rk_M(E) + \left| E \right|$\\
	& \ttfamily let $\dSET{(u_1,v_1), (u_2,v_2),\ldots, (u_{\left| V \right|^2-\left| V \right|},v_{\left| V \right|^2-\left| V \right|})} 
	    = V\times V \BS \SET{(v,v)~\middle|~v\in V}$\\
	& \ttfamily let $K := (V, V\times V)$\\
	& \ttfamily let $B$ be \rmfamily an arbitrary base of $M$\\
	& \ttfamily declare state variable $A \subseteq V\times V$\\
	& \ttfamily declare state variable $i\in\SET{1,2,\ldots,\left| V \right|^2-\left| V \right|}$\\
	& \ttfamily declare state variable $\Pcal \subseteq \Pbf\left( K \right)$\\
	& \ttfamily declare state variable $\Rcal \subseteq 2^{\Pbf\left( K \right)}$\\
	& \ttfamily declare state variable $\Bcal \subseteq \binom{V}{\rk_M(E)}$\\
	& \ttfamily $A := \emptyset$\\
	& \ttfamily $i := 1$\\
	& \ttfamily $\Pcal := \SET{v \in \Pbf(K)~\middle|~\vphantom{A^A}v\in V}$\\
	& \ttfamily $\Rcal := \SET{ \SET{b\in \Pbf(K)~\middle|~\vphantom{A^A} b\in B} }$\\
	& \ttfamily $\Bcal := \SET{ B }$\\
	& \ttfamily push state to stack\\
	& \ttfamily $d := 1$\\
	& \ttfamily while $d > 0$ do \\
	& \ttfamily $\quad$ if $\Bcal = \Bcal(M)$ return $1$ and stop\\
	& \ttfamily $\quad$ if $i > \left| V \right|^2 - \left| V \right|$ or $\Bcal \not \subseteq \Bcal(M)$ then do \\
	& \ttfamily $\quad$ $\quad$ pop state from stack\\
	& \ttfamily $\quad$ $\quad$ $d := d - 1$\\
	& \ttfamily $\quad$ $\quad$ $i := i + 1$\\
	& \ttfamily $\quad$ else do\\
	& \ttfamily $\quad$ $\quad$ push state to stack\\
	& \ttfamily $\quad$ $\quad$ $d := d + 1$\\
	& \ttfamily $\quad$ $\quad$ $\Pcal' := \SET{lr ~\middle|~\vphantom{A^A} l,r\in \Pcal,\,l_{-1}=u_i,\,r_{1}=v_i,\,\left| l \right|\cap \left| r \right|=\emptyset}$\\
	& \ttfamily $\quad$ $\quad$ $\Rcal' := \SET{\left( R\BSET{r} \right)\cup\SET{l.r} ~\middle|~ 
									\begin{array}{l} R\in \Rcal,\, r\in R,\,l\in \Pcal',\,l_{-1} = r_1,\\
									\left| l \right|\cap \left( \bigcup_{p\in R} \left| p \right| \right) = \SET{r_1},\, l_1\in E \end{array}}$\\
	& \ttfamily $\quad$ $\quad$ $\Bcal' := \SET{ \SET{p_1 ~\middle|~ p\in R} ~\middle|~ \vphantom{A^A}R\in \Rcal'}$\\
	& \ttfamily $\quad$ $\quad$ $A := A\cup\SET{(u_i,v_i)}$\\
	& \ttfamily $\quad$ $\quad$ $\Pcal := \Pcal \cup \Pcal'$\\
	& \ttfamily $\quad$ $\quad$ $\Rcal := \Rcal \cup \Rcal'$\\
	& \ttfamily $\quad$ $\quad$ $\Bcal := \Bcal \cup \Bcal'$\\
	& \ttfamily $\quad$ end if\\
	& \ttfamily end while $d > 0$\\
	& \ttfamily return $0$.
\end{tabularx}

\PRFR{Mar 7th}
\noindent
We give a rough estimate of the worst-case run-time behavior of this algorithm relative to the run-time of two major blocks of instructions.
First, let $\phi(d)$ denote a worst-case run-time estimate for the instructions inclusively between the ``{\ttfamily push state stack}''-instruction
and the ``{\ttfamily $\Bcal := \Bcal\cup \Bcal'$}''-instruction. This is the time it takes to update the paths, maximal routings, and bases of the digraph
when adding the arc $(u_i,v_i)$, and this operation clearly depends on the number of arcs in $A\BSET{(u_i,v_i)}$, which is a function of the value of $d$ at
the start of the instruction block. We would expect $\phi(d)$ to grow with $\left| \Pcal \right|$ and $\left| \Rcal \right|$. Clearly, we have the
very loose upper bound
$\left| \Pcal \right| \leq \left| V \right| + d!$, since every non-trivial path consists of a non-repeating sequence of arcs with further constraints.\footnote{P.~Seymour and B.D.~Sullivan give an upper bound for the number of $4$-vertex paths in digraphs without a cycle walk of length $\leq 4$, which is $\frac{4}{75}\left| V \right|^4$ \cite{SS10}.} We further have $\left| \Rcal \right| \leq d!$ because we may associate a routing $R\in\Rcal$ with
a non-repeating sequence of arcs obtained from its paths: if $R = \dSET{p_1,\ldots,p_r}$, 
we first list all arcs of $p_1$ in the order of appearance, then the arcs of $p_2$, and so on until we reach $p_r$, 
and then we list all arcs from $A$ that are not traversed by any path $p\in R$. Since all $R\in \Rcal$ route onto $B$, we can reconstruct $R$ from the
arc sequence we just constructed. Again, this bound is very loose.
Furthermore, let $\psi(d)$ denote a worst-case run-time estimate for the instructions inclusively between the ``{\ttfamily pop state stack}''-instruction
and the ``{$i:=i+1$}''-instruction. For the worst-case analysis, we assume that the backtracking method does traverse 
every digraph candidate $(V,A)$ with $A\subseteq V\times V \BSET{(v,v)~\middle|~ v\in V}$ --- which clearly is impossible for any input matroid $M$.
With this assumption we obtain a run-time in
\[ O\left( \sum_{i=0}^{2n^3-1} \binom{2n^3}{i} \left( \phi(i) + \psi(i+1)  \right)  \right). \]
It is clear that this estimation is overly pessimistic and does not convey a realistic picture of the run-time behavior of the digraph backtracking algorithm.
Therefore we implemented a version of this algorithm in {\ttfamily SageMath} 
and measured its performance on a few sample inputs  (see Listing~\ref{lst:isGammoidSage}).
It is an open research task to
find conditions and estimates for how often the above algorithm
does prune a large chunk of candidate solutions, as well as good implementations of the update procedure,
that exceeds the scope of this work.
\end{algorithm}

\begin{proof}[Proof of correctness]\PRFR{Mar 7th}
We have the following invariants at the ``{\ttfamily while $d > 0$}''-instruction: $d = \left| A \right|$,
let $D=(V,A)$, then
$\Pcal = \Pbf(D)$, $\Rcal$ is the family of all linkings from a subset of $E$ onto $B$ in $D$, and $\Bcal$ is the family of subsets of $E$
that can be linked onto $B$ by a routing $R\in\Rcal$. 
Furthermore, the stack consists of $d$ sets of previously pushed assignments of the
variables $A, i, \Pcal, \Rcal, \Bcal$.
The instructions in the ``{\ttfamily while $d > 0$}''-loop recursively test or dismiss all digraphs $D'=(V,A')$
for the property $\Gamma(D',B,E) = M$. First, the algorithm tests whether there is a digraph $D'$ representing $M$ with $(u_i,v_i)\in A$;
if it can be ruled out that there is such a digraph $D'$, the algorithm tests whether there is a digraph $D'$ representing $M$ with $(u_i,v_i)\notin A$.
On the other hand, if the loop finds a representation, it returns $1$ and the algorithm ends.
Therefore,
if we reach the ``{\ttfamily end while $d > 0$}''-instruction, we may conclude that there is no digraph on $V$ that represents $M$. 
With Remark~\ref{rem:upperBoundForV} and Theorem~\ref{thm:gammoidRepresentationWithBaseTerminals} we then may conclude that $M$ is not a gammoid,
and the next instruction correctly returns $0$.

Now let us show in detail that the ``{\ttfamily while $d > 0$}''-loop indeed has the property stated above.
Clearly, if $\Bcal = \Bcal(M)$, then $\Gamma(D,B,E) = M$ and therefore we may safely return $1$.
First, if there is some $X\in\Bcal$
which is not a base of $M$, then there is a routing $X\routesto B$ in every digraph $D'=(V,A')$ with $A\subseteq A'$,
consequently $X$ is independent in $\Gamma(D',B,E)$ for all such $D'$.
Thus we may dismiss all such candidate digraphs. The same holds if $i > \left| V \right|^2 - \left| V \right|$, in this case
we are out of arcs that we may add to $A$, but $\Bcal \subsetneq \Bcal(M)$, i.e. there is still a base $Y$ of $M$ which is not a base of $\Gamma(D,B,E)$.
In other words, if $M$ is a gammoid, then there is an arc $a\in A$ that obstructs the addition of some other arcs, one of which is needed to represent $M$.
In both cases, we have to undo our last addition of an arc. We achieve this by popping the assignments of $A, i, \Pcal, \Rcal, \Bcal$ from the stack,
which were pushed before the last arc had been added. Afterwards, we have to decrease $d$ in order to reflect the new stack size, and increase $i$ in order to
prevent adding the same arc again.
Now assume that neither $i>\left| V \right|^2 - \left| V \right|$ nor $\Bcal \not\subseteq \Bcal(M)$, thus we enter the ``{\ttfamily else do}''-branch of
the second {\ttfamily if}-instruction in 
the ``{\ttfamily while $d>0$}''-loop. In this case there is a base $Y$ of $M$ that is
not a base of $\Gamma(D,B,E)$. 
We try to fix this by adding the arc $(u_i,v_i)$ to $A$, after we pushed the current state to the stack and adjusted $d$ accordingly.
Let $D=(V,A)$ denote the digraph before adding the arc, i.e. $(u_i,v_i)\notin A$, and let $D'=(V,A\cup\SET{(u_i,v_i)})$.
Clearly $\Pbf(D') \BS \Pbf(D)$ consists of all paths $p\in \Pbf(D')$ with $(u_i,v_i)\in \left| p \right|_A$. But every such $p$ can be
written as $lr$ where $l,r\in \Pbf(D)$ with $\left| l \right|\cap \left| r \right|=\emptyset$ and such that $l$ ends in $u_i$ and $r$ starts in $v_i$.
Remember that $\Pcal = \Pbf(D)$, thus $\Pcal' = \Pbf(D') \BS \Pbf(D)$. Now let $R\colon X\routesto B$ with $X\subseteq E$ be a routing in $D'$ 
which is not a routing in $D$,
then $R\cap \Pcal' \not= \emptyset$. If we cut off the path of $R$ that uses the new arc $(u_i,v_i)$ at $u_i$, we obtain a routing that is also a routing in
$D$. Therefore, all routings of $D'$ that start in a subset $X\subseteq E$
and that are not routings of $D$ are members of the family $\Rcal'$. Consequently, all bases of $\Gamma(D',B,E)$
that are not bases of $\Gamma(D,B,E)$ are members of $\Bcal'$. Thus the above invariants at the ``{\ttfamily while $d > 0$}''-instruction
hold after the update operations on $A,\Pcal,\Rcal,\Bcal$. The correctness of the algorithm is therefore established.
\end{proof}

\begin{remark}\label{rem:backTrackSlow}\PRFR{Mar 7th}
Algorithm~\ref{alg:gammoidBackTrack} is obviously faster than the brute-force search method, 
as it does not generate non-matroid solution candidates.
 Yet, it still has to dismiss all
possibilities of arranging arcs in a big digraph in order to determine that $M$ is not a gammoid. The dismissal of a chunk of solution candidates
may only take place as soon
as it can be proven, that every gammoid corresponding to a digraph that contains a certain set of arcs
has some independent set $X\subseteq E$ which is dependent in $M$.\footnote{A tempting modification of Algorithm~\ref{alg:gammoidBackTrack}
would be to check whether $\SET{ \SET{p_1 ~\middle|~ p\in R} \cap E ~\middle|~ R\in \Rcal'} \not\subseteq \Ical$
holds instead of $\Bcal \not\subseteq \Bcal(M)$, but the Augmentation Lemma~\ref{lem:augmentation} implies that this does not occur any earlier than
 $\Bcal\not\subseteq \Bcal(M)$.}
 Alas, this happens quite late in the process: At least as long as none of the auxiliary vertices
in $V\BS E$ has a leaving arc that enters any $e\in E$, there is no way to detect excess connectivity in partial solutions.
 Therefore Algorithm~\ref{alg:gammoidBackTrack}
traverses all candidate arc sets $A$ that cover 
more than $2^{\left( \left| V \right| - 1  \right)\cdot \left| V\BS E \right|}$
% all
% $2^{(\rk_M(E)^2\cdot \left| E \right| + \rk_M(E) + \left| E \right|)\cdot (\rk_M(E)^2\cdot \left| E \right| + \rk_M(E) - 1)}$
digraphs without loop-arcs on $V$ with $A\cap \left( 	\left( V\BS E \right)\times E \right) = \emptyset$.
This implies a lower bound of the run-time for all inputs $M=(E,\Ical)$ with $M$ not a gammoid\footnote{
For instance, we may use the arbitrary large non-gammoids $M(K_4)\oplus \left( \SET{X}, 2^X \right)$
for growing finite sets $X$  to approach this run-time bound (see also Example~\ref{ex:MK4}).} 
in
\( \Omega \left( 2^{\left| E \right|^{5.999}} \right)\).

We implemented a less naive version of Algorithm~\ref{alg:gammoidBackTrack} (see Listing~\ref{lst:isGammoidSage}),
where we use an implicit linear order on $V=\dSET{\hat v_1,\hat v_2,\ldots,\hat v_m}$
such that $E = \SET{\hat v_1,\hat v_2,\ldots, \hat v_{\left| E \right|}}$ holds. 
We keep track of the smallest index $i_0$ that belongs to a vertex $\hat v_{i_0}\in V\BS E$
that is not entered by any arc $a\in A$. Let $i,j\in \SET{1,2,\ldots, \left| V \right|^2 - \left| V \right|}$
and let $(u_i,v_i) = (\hat v_{i_1}, \hat v_{i_2})$ and $(u_j,v_j) = (\hat v_{j_1}, \hat v_{j_2})$.
We require that the implicit linear order on the arcs has the property
that  we have  $\max\SET{i_1,i_2} < \max\SET{j_1,j_2}$ or 
($\max\SET{i_1,i_2} = \max\SET{j_1,j_2}$ holds, and either $i_2 = j_1 = \max\SET{i_1,i_2}$ or $\min\SET{i_1,i_2} < \min\SET{j_1,j_2}$
holds), if and only if $i < j$ holds.
 In other words, we enumerate $V\times V\BSET{(v,v)~\middle|~v\in V}$ in the following order:
 $(\hat v_1, \hat v_2),$ $ (\hat v_2, \hat v_1),$ $ (\hat v_1, \hat v_3),$
 $ (\hat v_2, \hat v_3),$ $ (\hat v_3, \hat v_1),$ $ (\hat v_3, \hat v_2),$ $ (\hat v_1, \hat v_4),\ldots$
 Now we may implement a shortcut and backtrack as soon as $i_2 > i_0$ holds for $(u_i,v_i) = (\hat v_{i_1}, \hat v_{i_2})$.
 The rationale behind this is that if we have to add a new arc that enters a previously unentered vertex, we can always choose the
 vertex entered to be the one with the lowest index among all unentered vertices. Although the improvement corresponding to this adjustment is
 quite measurable in practice,
  the algorithm still has to try more than 
  $2^{\left| V\BS E \right|^2 - \left| V\BS E \right|}$
  %$2^{(\rk_M(E)^2\cdot \left| E \right| + \rk_M(E))\cdot(\rk_M(E)^2\cdot \left| E \right| + \rk_M(E) - 1)}$
  --- e.g. the number of digraphs $D'=(V',A')$ on $V'=V\BS E$ 
   with $A'\subseteq \left( V'\times V' \right)\BSET{(v,v) ~\middle|~ v\in V'}$ ---
  different candidate arc sets that have the property $A\cap \left( \left(V\BS E\right) \times E \right) = \emptyset$ 
  %--- which correspond to all digraphs $D'=(V',A')$ on $V'=V\BS E$ 
  % with $A'\subseteq \left( V'\times V' \right)\BSET{(v,v) ~\middle|~ v\in V'}$ ---
  for every  $Q = A\cap \left( E \times \left( V\BS E \right) \right)$ 
 with $Q_A\subseteq Q$ 
 %where $Q_A\subseteq A$ 
 where we let $Q_A =\SET{ (\hat v_1, w) ~\middle|~w\in V\BS E}$. Since there 
 are $2^{(\left| E \right| - 1)\cdot \left| V\BS E \right|}$ different possibilities for 
such $Q$, 
 % we obtain that the adjusted algorithm still has to work through at least
 %\( 2^{(\rk_M(E)^2\cdot \left| E \right| + \rk_M(E))\cdot(\rk_M(E)^2\cdot \left| E \right| + \rk_M(E) + \left| E \right| - 2)} \) candidate digraphs.
 %Thus 
 the adjusted algorithm still exposes the \( \Omega \left( 2^{\left| E \right|^{5.999}} \right)\) behavior.

 \noindent
 In theory, there is a possibility to speed up the algorithm a little further, since Theorem~\ref{thm:IPEssentialStars}
 guarantees that no vertex has more than $\rk_M(E)$ leaving arcs. Clearly, the problem that the backtracking information
 is only available late in the process is not remedied by limiting the number of arcs leaving each vertex.
 The number of digraph candidates,
 that have to be processed before any target is connected, still is at least
 \[ \left( \sum_{k=0}^{\rk_M(E)} \binom{\rk_M(E)^2\cdot \left| E \right| + \left| E \right|}{k} \right)^
 {\rk_M(E)^2\cdot \left| E \right| } = \Omega\left( 2^{\rk_M(E)^2\cdot \left| E \right|} \right) .\]
 Thus such an adjusted algorithm still exposes \( \Omega \left( 2^{\left| E \right|^{2.999}} \right)\) behavior.
 When we implemented forced bounds on the number of leaving arcs, the run-time actually increased.
   %%%TODO-Schranke
\end{remark}

\PRFR{Mar 7th}
\noindent
Therefore it is clearly indicated that we examine how Mason's criterion and matroid extensions play along with each other
in order to gain better understanding of the problem. This understanding is an essential milestone for
the research in better algorithms for determining $\Gamma_{\Mcal}(M)$.
Before we devote ourselves to that, we want to make a remark on potentially more easy subclasses of gammoids.

\begin{remark}
	Let $k\in \N$. For the subclasses $\Wcal^k$ that consists of all gammoids $G$ with $\arcW^k(G) \leq 1$,
	the problem of deciding class membership
	appears to be dramatically more easy. Let $M=(E,\Ical)$ be a matroid. If $M\in \Wcal^k$ there is a representation
	using at most $k\cdot \left| E \right|$ arcs. Therefore there is a representation with at most $2k\cdot\left| E \right|$
	auxiliary vertices. So if $M\in \Wcal^k$, then $M=\Gamma(D,T,E)$ where $D=(V,A)$ with
	 $\left| V \right| \leq (2k+1)\cdot	\left| E \right|$
	 and $\left| A \right| \leq k\cdot \left| E \right|$.
	 Thus there are at most
	 \[ \sum_{i=0}^{k\cdot\left| E \right|} \binom{9k^2\cdot\left| E \right|^2}{i} \]
	 candidate digraphs for $M$ that we may have to regard.
	 Furthermore, 
	 $M$ is not in $\Wcal^k$ as soon as $M\restrict X$ cannot be represented with a digraph on $3k\cdot \left| X \right|$
	 vertices with at most $k\cdot \left| X \right|$ arcs, this may open up possibilities for effective divide and conquer approaches.	 
\end{remark}

\subsection{Violations of Mason's $\alpha$-Criterion}

Some of the ideas and results presented in this section have been published in I.~Albrecht's 
{\em On Finding New Excluded Minors for Gammoids} \cite{Al17}, 
where an equivalent yet different version of Mason's $\alpha$-criterion is used.

\begin{definition}\label{def:aViolation}\PRFR{Feb 15th}
	Let $M=(E,\Ical)$ be a matroid, $V\subseteq E$. $V$ shall be an
	 \deftext[aM-violation@$\alpha_M$-violation]{$\bm \alpha_{\bm M}$-violation},
	if $\alpha_M(V) < 0$ and for all $V'\subsetneq V$, $\alpha_M(V') \geq 0$.
	The \deftextX{family of all $\bm \alpha_{\bm M}$-violations} is denoted by \label{n:VM}
	\[ \Vcal(M) = \SET{ V\subseteq E ~\middle|~ \alpha_M(V) < 0 \txtand\, \forall V'\subsetneq V\colon\, \alpha_M(V') \geq 0}. \qedhere\]
\end{definition}

\noindent Clearly, an $\alpha_M$ violation is an inclusion minimal set $X$, for which the inequality $\alpha_M(X) \geq 0$ 
does not hold.

\begin{corollary}\PRFR{Feb 15th}
	Let $M=(E,\Ical)$ be a matroid. Then $M$ is a strict gammoid, if and only if $\Vcal(M) = \emptyset$.
\end{corollary}
\begin{proof}\PRFR{Feb 15th}
Immediate from Corollary~\ref{cor:MasonAlpha} and Definition~\ref{def:aViolation}.
\end{proof}

\needspace{4\baselineskip}

\vspace*{-\baselineskip} %Remove the line space created by the tilde below
\begin{wrapfigure}{r}{5cm}
\vspace{\baselineskip}
~~\includegraphics{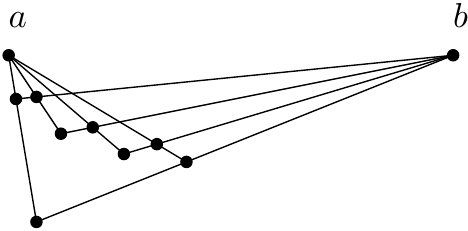}
%\vspace*{-1\baselineskip} %make the picture more tightly cropped
\end{wrapfigure}
~ %The tilde creates a new dummy paragraph. WHY IS THAT NEEDED? -> would increase the space %
  % before the ex. environment. THE NEXT FREE LINE IS ESSENTIAL!

\begin{example}\label{ex:BiApexMatroid}\PRFR{Apr 4th}
	Let $k\in \N\BSET{0,1}$ be an arbitrary choice, let
	 $$X= \SET{\vphantom{A^A}(i,j)\in \N\times\N ~\middle|~ i,j < k,\, j\in \SET{i,(i+1)\,\,\mathrm{mod}\, k}},$$
	and let $E = X\cup\dSET{a,b}$ where $a,b\notin X$.
	Furthermore, let $$\Hcal_a = \SET{\vphantom{A^A}\SET{a,\left( i,i \right),\left( i,(i+1)\,\,\mathrm{mod}\, k \right)}~\middle|~ i\in \N,\,i<k}$$
	and $$\Hcal_b = \SET{\vphantom{A^A}\SET{b,\left( (j-1)\,\,\mathrm{mod}\, k,j \right),\left( j,j \right)}~\middle|~ j\in \N,\,j<k}.$$
	Let $M=(E,\Ical)$ be the matroid of rank $3$ 
	such that $\Hcal = \Hcal_a \cup\Hcal_b$ is the family of its dependent hyperplanes.\footnote{$M$ is a paving matroid, see \cite{We76}, Section~2.3. 
	The figure
	depicts the affine configuration of $M$ for $k=4$.}
	We show that the matroid $M$ exists by postulating that
	$$\Ccal(M) = \SET{C\in\binom{E}{4}~\middle|~ \nexists H\in \Hcal\colon \,H\subseteq C} \cup \Hcal.$$ Clearly, $\Ccal(M)\not=\emptyset$,
	and for $C_1,C_2\in\Ccal(M)$ we have $C_1\subseteq C_2$ if and only if $C_1 = C_2$.
	Let $C_1,C_2\in\Ccal(M)$ with $C_1\not= C_2$. Then $C_1\cup C_2$ has at least $5$ elements, thus for every $e\in C_1\cap C_2$, the set
	$\left( C_1\cup C_2 \right)\BSET{e}$ has four elements. So, by construction of $\Ccal(M)$, the set $\left( C_1\cup C_2 \right)\BSET{e}$
	contains a circuit. Consequently, $\Ccal(M)$ satisfies the circuit elimination axiom.
	Therefore $\Ccal(M)$ satisfies all circuit axioms for matroids (\cite{Ox11}, Theorem~1.1.4, p.9) and the matroid $M$ with the above
	properties exists. Since every dependent hyperplane is a circuit, we obtain that $\alpha_M(H) = 1$ for all $H\in \Hcal$.
	A flat $F\in\Fcal(M)$ is either independent, a hyperplane, or the whole ground set of $M$.
	For independent $F\in\Fcal(M)$ we have $\alpha_M(F) = 0$. Therefore
	\[ \alpha_M(E) = \left| E \right| - \rk_M(E) - \sum_{F\in\Fcal(M),\,F\not= E} \alpha_M(F) = 2k - 1 - \left| \Hcal \right| = -1.\]
	Since there are at least $1+\left| W \right|$ dependent hyperplanes $H\in \Hcal$ with $H\cap W \not= \emptyset$ for all $W\subseteq E$ with
	$W\not=\emptyset$,
	we obtain that $\alpha_M(E\BS W) \geq 1 + \alpha_M(E) \geq 0$ for all $\emptyset\not= X\subseteq E$.
	Consequently $E$ is an $\alpha_M$-violation and $\Vcal(M) = \SET{E}$.
	Let us fix an arbitrary element $e\in E$ for now and let $N = M\restrict\left( E\BSET{e} \right)$.
	 Then at least two dependent hyperplanes of $M$ are no longer hyperplanes of $N$.
	All dependent flats of $N$ are still hyperplanes of $N$, consequently $\alpha_N \geq 0$ and $N$ is a strict gammoid
	(Corollary~\ref{cor:MasonAlpha}). $M\contract\left( E\BSET{e} \right)$ has rank $2$ 
	and therefore is a strict gammoid (Proposition~\ref{prop:cornerCases}). Furthermore, Proposition~\ref{prop:cornerCases} yields that $M$ is
	not a gammoid, and therefore $M$ is an excluded minor of the class of gammoids. By Theorem~\ref{thm:GammoidsClosedMinorsDuality} we obtain that $M^\ast$
	is an excluded minor for the class of gammoids, too.
	The circuits of $M^\ast$ are the complements of hyperplanes of $M$, 
	and the hyperplanes of $M^\ast$ are the complements of circuits of $M$.
	Thus every $D\in \Ccal(M^\ast)$ has at least $\left| E \right| - 3$ elements.
	Furthermore $\rk_{M^\ast}(E) = \left| E \right| - \rk_M(E) = \left| E \right| - 3$, therefore all dependent hyperplanes of $M^\ast$ are circuits,
	and these hyperplane circuits of $M^\ast$ are the complements of the dependent hyperplanes of $M$.\footnote{We admit that this situation might be slightly confusing. In both matroids $M$ and $M^\ast$, every circuit has one of two cardinalities, and every hyperplane has one of two cardinalities. The hyperplanes with the higher cardinality are circuits and therefore dependent, whereas the hyperplanes with smaller cardinality are independent. The circuits with smaller cardinality are hyperplanes, the circuits with higher cardinality have full rank. The hyperplanes with higher cardinality -- the dependent hyperplanes -- are therefore complements of smaller cocircuits, which happen to be cohyperplanes. Therefore the dependent hyperplanes are complements of codependent cohyperplanes, a situation that might appear very special, but it is rather not: Matroids with this property are called {\em sparse paving matroids}, and R.~Pendavingh and J.~van~der~Pol showed that they are quite abundant \cite{PvdP15}.} Therefore $M^\ast$ has
	$\left| \Hcal \right| = 2k$ dependent hyperplanes of the form $H' = E\BS H$ for $H\in\Hcal$. In this situation, we have $\alpha_{M^\ast}(H') = 1$.
	Furthermore, if $F\in\Fcal(M^\ast)$ such that $F\not= E\BS H$ for all $H\in \Hcal$, then either $F= E$ or $F$ is independent.
	Thus we may calculate
	\[ \alpha_{M^\ast}(E) = \left| E \right| - \rk_{M^\ast}(E) -\sum_{F\in\Fcal(M^\ast),\,F\not=E}\alpha_{M^\ast}(F) = 3 - \left| \Hcal \right| = -\left( 2k- 3 \right).\]
	Now, let $x\in X$, then there are precisely $2$ dependent hyperplanes $H$ of $M$ with $x\in H$.
	Consequently, there  are $\left| H \right| - 2 = 2k - 2$ dependent hyperplanes of $M^\ast$ which contain $x$.
	Similarly, there are $\left| H \right| - k = k$ dependent hyperplanes of $M^\ast$ which contain $y\in\SET{a,b}$.
	Also, for all $W\subseteq X$ with $\left| W \right| = 2$ there are at least $2k -1$ dependent hyperplanes of $M^\ast$ which have non-empty intersection 
	with $W$, and for all $W\subseteq X$ with $\left| W \right| \geq 3$ every dependent hyperplane of $M^\ast$ has non-empty intersection with $W$.
	Thus $\alpha_{M^\ast}(E\BS W) \geq 0$ for all $W\subseteq E$ with $X\cap W \not= \emptyset$. Furthermore
	$$\alpha_{M^\ast}(E\BSET{a}) = \alpha_{M^\ast}(E\BSET{b}) = -\left( 2k - 3 \right) + \left( k - 1 \right) = -\left( k - 2 \right),$$
	thus $\Vcal(M^\ast) = \SET{E}$ if $k=2$, and $\Vcal(M^\ast) = \SET{E\BSET{a},E\BSET{b}}$ whenever $k>2$.
	Let us consider the case where $k>2$, and let $N = \left( M^\ast  \right)\restrict\left( E\BSET{b} \right)$.
	Then we also have $\alpha_N(E\BSET{b}) = -\left( k - 2 \right)$, $\Vcal(N) = \SET{E\BSET{b}}$, and $N$ is the dual of a strict gammoid 
	--- so $N$ is a transversal matroid.
\end{example}

\begin{remark}\PRFR{Apr 4th}\label{rem:BiApex}
	Let $\Mcal$ be the class of matroids where $M\in\Mcal$ if
	and only if $M$ is the matroid constructed in Example~\ref{ex:BiApexMatroid} for some $k\in\N\BSET{0,1}$. Then $\Mcal$ is an infinite family of
	excluded minors of rank $3$ for the family of gammoids.
	The derived family $\Mcal^\ast = \SET{M^\ast ~\middle|~ M\in \Mcal}$ is also an infinite family of excluded minors and yields excluded minors with
	rank $2k - 1$. We see that excluded minors for the class of gammoids
	may have multiple $\alpha_M$-violations, and that the value of $\alpha_M(E)$ for $E\in\Vcal(M)$
	may become arbitrarily low for excluded minors of the class of gammoids as well as for gammoids that are non-strict.
	The matroids $N = \left( M^\ast  \right)\restrict\left( E\BSET{b} \right)$ and $M^\ast$ have very similar $\alpha$-violation structure:
	$M^\ast$ contains two violations, and if we restrict $M^\ast$ to any of these two violations, we obtain $N$ -- thus $M^\ast$ essentially 
	has two isomorphic copies of the unique $\alpha_N$-violation. Consequently, we cannot decide whether a matroid is a gammoid or not by just considering
	one violation of $M$ at a time, instead we have to consider the interaction between violations in $M$ as well.\footnote{We may consider the property that
	a matroid is a gammoid to be global with respect to the proper violation-restrictions $M\restrict X$ for $X\in\Vcal(M)$ in the same sense
	as the chromatic number of a graph is a global property with respect to proper induced sub-graphs.}
\end{remark}

\noindent Before we start developing the theory of $\alpha_M$-violations, we should familiarize ourselves some more with
the two different kinds of violations that arise in matroids --- violations in non-gammoids and violations in gammoids that are not strict.

\needspace{7\baselineskip}

\vspace*{-\baselineskip} %Remove the line space created by the tilde below
\begin{wrapfigure}{r}{5cm}
\vspace{\baselineskip}
\begin{centering}~~%move the picture slightly to the right
\includegraphics{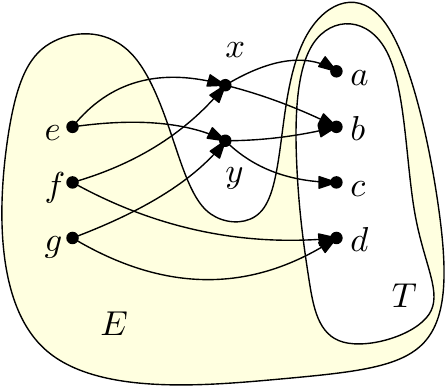}
\end{centering}%
\vspace*{-1\baselineskip} %make the picture more tightly cropped
\end{wrapfigure}
~ %The tilde creates a new dummy paragraph. WHY IS THAT NEEDED? -> would increase the space %
  % before the ex. environment. THE NEXT FREE LINE IS ESSENTIAL!

\begin{example}\label{ex:violationGammoid}\PRFR{Feb 15th}
	We examine the situation with respect to the gammoid $M=\Gamma(D,T,E)$ with the ground set $E=\dSET{a,b,c,d,e,f,g}$ as presented in Example~\ref{ex:nonStrictGammoid}.
	We have $\Fcal(M)\BS \Ical = \{\SET{a,b,c,e},$ $\SET{a,b,d,f},$ $\SET{b,c,d,g},$ $\SET{d,e,f,g},$ $E\}$ and, 
	clearly, $\Vcal(M) = \SET{E}$
	and we have $\alpha_M(E) = -1$.
	But since we know that $M=\Gamma(D,T,V)\restrict E$, we know that the violation $E$ can be resolved by adding the elements $x$ and $y$ to the matroid $M$. 
	Let $M_x = \Gamma(D,T,E\cup\SET{x}) = (E\cup\SET{x},\Ical_x)$. Then the new dependent flats of $M_x$ with respect to $M$ are
	\begin{align*} \left(  \Fcal(M_x)\BS \Ical_x\right) \BS \left( \Fcal(M)\BS \Ical\right)  = \{ &
	\SET{a,b,x},\SET{d,f,x},\SET{a,b,g,x},\SET{c,d,f,x},\\
	&\SET{a,b,c,e,x},\SET{a,b,d,f,x},\SET{d,e,f,g,x}, E\cup\SET{x}\},
	\end{align*}
	and the dependent flats of $M$ that vanish in $M_x$ are
	\[ \left( \Fcal(M)\BS \Ical\right) \BS  \left(  \Fcal(M_x)\BS \Ical_x\right)  = \SET{\SET{a,b,c,e},\SET{a,b,d,f},\SET{d,e,f,g},E}.\]
	We now have
	\begin{align*}\alpha_{M_x}\left(\SET{a,b,x}\right) = \alpha_{M_x}\left(\SET{d,f,x}\right) = \alpha_{M_x}\left(\SET{b,c,d,g}\right)
	= \alpha_{M_x}\left( \SET{a,b,g,x} \right)  & =\\
	 \alpha_{M_x}\left(\SET{a,b,c,e}\right) = \alpha_{M_x}\left(\SET{a,b,d,f}\right) = \alpha_{M_x}\left(\SET{d,e,f,g}\right)& = \hphantom{-}1,\\
	\alpha_{M_x}\left(\SET{a,b,c,e,x}\right) = 5 - 3 - \alpha_{M_x}\left( \SET{a,b,x} \right) & = \hphantom{-}1,\\
	\alpha_{M_x}\left(\SET{d,e,f,g,x}\right) = 5 - 3 - \alpha_{M_x}\left( \SET{d,f,x} \right) & = \hphantom{-}1,\\
	\alpha_{M_x}\left(\SET{a,b,d,f,x}\right) = 5 - 3 - \alpha_{M_x}\left( \SET{a,b,x} \right)  - \alpha_{M_x}\left( \SET{d,f,x} \right) & = \hphantom{-}0,\\
	\alpha_{M_x}(E) = 7 - 4 - \alpha_{M_x}\left( \SET{b,c,d,g} \right) & = \hphantom{-}2, \text{~and}\\
	\alpha_{M_x}\left( E\cup\SET{x} \right) = 8 - 4
	 - \alpha_{M_x}\left( \SET{a,b,x} \right) 
	 - \alpha_{M_x}\left( \SET{d,f,x} \right)
	 - \alpha_{M_x}\left( \SET{d,e,f,g,x} \right) &\\
	 -\, \alpha_{M_x}\left( \SET{a,b,c,e,x} \right)
	 - \alpha_{M_x}\left( \SET{b,d,f,g} \right)
	 &= -1.
	\end{align*}
	Thus $\Vcal(M_x) = \SET{E\cup\SET{x}}$. So we still have a violation if we just add $x$ to the ground set of the gammoid, and it is easy to tell from the symmetric design of $D$, that the same holds when we would just add $y$. Although $E$ is no longer a violation,
	we seem to just have shifted the problem to $E\cup\SET{x} = \cl_{M_x}(E)$. 
	Yet, we made some progress by adding $x$ to $M$: $M_x$ has a modular cut that is generated by three rank $2$ flats, namely $\Ccal_y = \SET{F\in \Fcal(M_x) ~\middle|~ \SET{b,c}\subseteq F \txtor \SET{d,g}\subseteq F \txtor \SET{e,x}\subseteq F}$,
	whereas $M$ has no such modular cut. So the violation $E\cup\SET{x}$ of $M_x$ is less rigid than the violation $E$ of $M$.
	Now let $N= \Gamma(D,T,V)$. Then $\Vcal(N) = \emptyset$, and $\alpha_N(X) = 1$, if $X\in \SET{\SET{a,b,x},\SET{b,c,y},\SET{d,f,x},\SET{d,g,y},\SET{e,x,y}}$, otherwise $\alpha_N(X) = 0$. So we see how the gammoid $M$ violates Mason's $\alpha$-criterion: by deleting $x$ and $y$, the nullity of the rank $2$ flats disappears, and so a common reason for nullity in the hyperplanes goes below the radar, resulting in excess negative terms for $\alpha_M(E)$, which then create an $\alpha_M$-violation.
\end{example}

\begin{example}\label{ex:violationNonGammoid}\label{ex:MK4}\PRFR{Feb 15th}
	Let us now consider the matroid $M(K_4) = (E,\Ical)$ which shall be defined on the ground set $E = \dSET{a,b,c,d,e,f}$ and which has
	the following circuits
	$$\Ccal(M(K_4)) = \SET{\vphantom{A^A}\SET{a,b,d}, \SET{a,c,e}, \SET{b,c,f}, \SET{d,e,f}}.$$
	Every circuit of $M(K_4)$ is also a hyperplane of $M(K_4)$, and therefore a flat.
	We calculate
	\begin{align*}
		\alpha_{M(K_4)}\left( \SET{a,b,d} \right)  = \alpha_{M(K_4)}\left( \SET{a,c,e} \right) &=\\ \alpha_{M(K_4)}\left( \SET{b,c,f} \right) =
		\alpha_{M(K_4)}\left( \SET{d,e,f} \right) &= 1\\
		\alpha_{M(K_4)}\left( E \right) = 6 - 3 - 4\cdot 1 &= -1.
	\end{align*}
	Thus $\Vcal(M(K_4)) = \SET{E}$ and by Proposition~\ref{prop:cornerCases} and Corollary~\ref{cor:MasonAlpha}, we obtain that $M(K_4)$ is
	not a gammoid. Unlike Example~\ref{ex:violationGammoid}, the $\alpha_{M(K_4)}$-violation $E$ does not allow any particular progress
	by adding elements to $M(K_4)$. First, consider an extension $N\in \Xcal(M(K_4),g)$, that corresponds to a modular cut
	 $\SET{F\in\Fcal(M)\mid g\in \cl_N(F)}$ which is the principal filter of a flat $F_g \in \Fcal(M(K_4))$ in $\Fcal(M(K_4))$.
	 Such an extension always has the violation
	 $E\cup\SET{g}\in \Vcal(N)$. Furthermore, we do not gain any headroom in the sense of allowing new qualities of modular cuts
	 that are not available with respect to $M(K_4)$:
	 If $A,B\in\Fcal(N)$
	 is not a modular pair in $N$, then $A\BSET{g}, B\BSET{g}$ is not a modular pair in $M(K_4)$.
	 The modular cuts of $M(K_4)$ that are not principal filters in $\Fcal(M(K_4))$ are the cuts of $\Fcal(M(K_4))$ generated by
	 the two- and three-elementary subsets of $Q=\SET{\SET{a,f},\SET{b,e},\SET{c,d}}$.
	 % Due to the symmetries of $M(K_4)$,
	 %we only have to consider $N_1\in \Xcal(M(K_4),g)$ corresponding to the filter generated by $\SET{\SET{a,f},\SET{b,e}}$
	 %and $N_2\in \Xcal(M(K_4),g)$ corresponding to the filter generated by $\SET{\SET{a,f},\SET{b,e},\SET{c,d}}$.
	 Now none of the hyperplanes of $M(K_4)$ belong to such cuts, because no hyperplane is a subset of any element of $Q$.
	 Therefore, the hyperplanes of $M(K_4)$ are still flats in the corresponding extension, and so $E$ is still a violation.
\end{example}

% -*- root: ../thesis.tex -*-

\needspace{8\baselineskip}
\subsection{The $\alpha$-Invariant and Single Element Extensions}\label{sec:alphaNExt}
\noindent In this section we take a look at how single element extensions of a matroid interact with the $\alpha$-invariant.
We start with an easy observation.

\begin{lemma}\label{lem:necPropForVanishingExt}
	Let $M=(E,\Ical)$ be a matroid, $e\notin E$, $N\in\Xcal(M,e)$ be a single-element extension of $M$ such that
	$ C = \SET{F\in\Fcal(M)~\middle|~ e\in \cl_N(F)}$. Let further $X\subseteq E$ such that
	$\alpha_N(X) \geq 0$.
	 Then there is a set $C_0 \subseteq \SET{F\in C~\middle|~ F\subsetneq X}$ such that
	\[ \forall F\in C_0\colon\, \alpha_M(F) > 0 \quad\txtand\quad \sum_{F\in C_0} \alpha_M(F) \geq -\alpha_M(X).\]
\end{lemma}
\begin{proof}
	Clearly, $\alpha_N(X) - \alpha_M(X) \geq - \alpha_M(X)$. It follows from Lemma~\ref{lem:flatsOfExtension} 
	and Definition~\ref{def:alphaM} that
	 $\alpha_N(Y) = \alpha_M(Y)$ for all $Y\subseteq E$
	with $\cl_M(Y)\notin C$. Furthermore, $\SET{F\in\Fcal(N)~\middle|~F\subsetneq X,\,F\in C} = \emptyset$. 
	\begin{align*}
		\alpha_N(X) -\alpha_M(X) & = \hphantom{-} \left| X \right| - \rk_N(X) -
			 \sum_{F\in\Fcal(N),\,F\subsetneq X} \alpha_N(F) \\ & \hphantom{ = }\,\, - \left| X \right| + \rk_M(X) + 
			 \sum_{F\in\Fcal(M),\,F\subsetneq X} \alpha_M(F) \\
					& = - \left( \sum_{F\in\Fcal(N),\,F\subsetneq X} \alpha_N(F)  \right)
					 + \left( \sum_{F\in\Fcal(M),\,F\subsetneq X} \alpha_M(F)  \right) \\
				& = \sum_{F\in C,\,F\subsetneq X} \alpha_M(F) \leq \sum_{F \in C_0} \alpha_M(F) \\
	\end{align*}
	where $C_0 = \SET{F\in C ~\middle|~,\,F\subsetneq X,\,\alpha_M(F) > 0}$ is a subset of $C$ with the desired property.
\end{proof}

\noindent Unfortunately, if $M$ is a matroid and $C$ is a modular cut that satisfies the consequent of
 Lemma~\ref{lem:necPropForVanishingExt}
with respect to every $X\subseteq E$, and  if  $N\in \Xcal(M,e)$ is the extension corresponding to $C$,
then $N$ may still not be a strict gammoid.
On the other hand, if there is a subset $X\subseteq E$ which violates the consequent of Lemma~\ref{lem:necPropForVanishingExt}, 
we know that $N$ is definitely not a strict gammoid. If we tried to extend a given matroid in order to obtain a strict gammoid, then
it would be quite natural to first try modular cuts which satisfy the consequent of Lemma~\ref{lem:necPropForVanishingExt}
for as many $X\subseteq E$ with $\alpha_M(X) < 0$ as possible.

\begin{definition}\label{def:AlphaPoset}\PRFR{Feb 15th}
	Let $M=(E,\Ical)$ be a matroid. We define the \deftext[aM-poset@$\alpha_M$-poset]{$\bm\alpha_{\bm M}$-poset}
	as the pair $(\Alpha_M, \sqsubseteq_M)$ \label{n:alphaposet} where $\Alpha_M = 2^E$ and where for all $X,Y\in \Alpha_M$
	\[ X \sqsubseteq_M Y \quad\Longleftrightarrow\quad X = Y \txtor X\in \Fcal(M,Y) \]
	holds. If $M$ is clear from the context, we also write $\Alpha$ for $\Alpha_M$ and $\sqsubseteq$ for $\sqsubseteq_M$.
\end{definition}
\begin{remark}\PRFR{Feb 15th}
	$(\Alpha_M,\sqsubseteq)$ is obviously a poset: for all $X\in \Alpha_M$ we have $X\sqsubseteq X$. Furthermore, if 
	$X \sqsubseteq Y$ and $Y\sqsubseteq X$ holds
	for $X,Y\in\Alpha_M$, then $X = Y$ must hold because all elements of $\Fcal(M,Y)$ are proper subsets of $Y$
	and therefore $X\in\Fcal(M,Y)$ and $Y\in\Fcal(M,X)$ contradict each other. Now let $X,Y,Z\in \Alpha_M$
	such that $X \sqsubseteq Y$ and $Y\sqsubseteq Z$. If $X=Y$ or $Y=Z$, there is nothing to show. Otherwise,
	$X \sqsubseteq Y \sqsubseteq Z$ implies $X,Y\in\Fcal(M)$. Since $X\subsetneq Y \subsetneq Z$ we obtain $X\in \Fcal(M,Z)$,
	thus $X\sqsubseteq Z$.
\end{remark}

\needspace{4\baselineskip}
\begin{lemma}\label{lem:alphaMoebius}\PRFR{Feb 15th}
	Let $M=(E,\Ical)$ be a matroid, and let
	\[ \nu\colon 2^E \maparrow \Z,\,X\mapsto \left| X \right| - \rk(X) .\]
	Then \[\alpha_M =  \nu \ast \mu_{\Alpha} \]
	where $\mu_{\Alpha}$ is the Möbius-function of the $\alpha_M$-poset $(\Alpha, \sqsubseteq)$.
\end{lemma}
\begin{proof}\PRFR{Feb 15th}\PRFR{ ~~ + sage}
	From the recurrence relation of the $\alpha$-invariant (Definition~\ref{def:alphaM}) 
	and the definition of the $\alpha$-poset (Definition~\ref{def:AlphaPoset}) we obtain
	\[ \nu(X) = \left| X \right| - \rk(X) = \alpha(X) + \sum_{F\in\Fcal(M,X)}\alpha(F) = \sum_{Y\sqsubseteq X} \alpha(Y) \]
	for all $X\subseteq E$. The zeta-matrix of $(\Alpha,\sqsubseteq)$ (Definition~\ref{def:zetaMatrix}) 
	allows us to write
	\[  \nu = \alpha \ast \zeta_\Alpha.\]
	We multiply with the Möbius-function of $(\Alpha,\sqsubseteq)$ and use Lemma~\ref{lem:moebiusInversion} in order to obtain
	\[
		\nu \ast \mu_\Alpha = \alpha \ast \zeta_\Alpha \ast \mu_\Alpha = \alpha \ast \id_\Z\left( 2^E \right) = \alpha.
		\qedhere
	\] 
\end{proof}

\begin{corollary}\label{cor:muAlphaNFromMuAlphaM}\PRFR{Feb 15th}
	Let $M=(E,\Ical)$ be a matroid, $e\notin E$, and $N\in\Xcal(M,e)$ a single element extension of $M$. Then
	\[\alpha_N\restrict_{2^E} \,\,\,= \,\,\, \alpha_M \ast \zeta_{\Alpha_M}\ast \left( \mu_{\Alpha_N} \restrict 2^E\times 2^E  \right) .\]
\end{corollary}
\begin{proof}\PRFR{Feb 15th}\PRFR{ ~~ + sage}
%\footnote{This proof is easily generalized to arbitrary extensions, i.e. matroids $N$ where $N\restrict E = M$.}
	Let $\nu_N\in\Z^{2^E}$ and $\nu_M\in \Z^{2^{E\cup\SET{e}}}$ be the maps where 
	$\nu_M(X) = \left| X \right| - \rk_M(X)$ and $\nu_N(X) = \left| X \right| - \rk_N(X)$ 
	holds for all $X\subseteq E$, or $X\subseteq E\cup\SET{e}$, respectively.
	 Then $\nu_N\restrict_{2^E} = \nu_M = \alpha_M\ast \zeta_{\Alpha_M}$ because $N$ is an extension of $M$. Furthermore, for $X\subseteq E$ and $Y\subseteq E\cup\SET{e}$
	 with $e\in Y$, we have $Y\not\sqsubseteq_N X$, and therefore $\mu_{\Alpha_N}(Y,X) = 0$ $(\ast)$,
	  thus we may restrict the equation from Lemma~\ref{lem:alphaMoebius}
	 in the following way:
	 \begin{align*}
	  \alpha_N\restrict_{2^E} \,\,\, &= \,\,\, (\nu_N \ast \mu_{\Alpha_N}  )\restrict_{2^E}
	  \quad\quad\quad\quad\quad \,%\\
	  \,\,\, =  \,\,\, \nu_N \ast \left( \mu_{\Alpha_N} \restrict 2^{E\cup\SET{e}} \times 2^E \right)  \\
	  & \stackrel{(\ast)}{=} \,\,\   \left(\nu_N\restrict_{2^E}\right) \ast \left( \mu_{\Alpha_N} \restrict 2^E\times 2^E  \right) %\\
	  \,\,\, = \,\,\ \alpha_M \ast \zeta_{\Alpha_M}\ast  \left( \mu_{\Alpha_N} \restrict 2^E\times 2^E  \right).  
	  \qedhere
	 \end{align*} 
\end{proof}

\noindent
Let us explain the above equations a little further.
	 Here, we interpret $\alpha_N\restrict_{2^E}$ as a vector in the $2^{\left| E \right|}$-dimensional $\Z$-module $\Z^{2^E}$.
	 The term $\nu_N \ast \mu_{\Alpha_N}$ denotes
	 a vector of the $\Z$-module $\Z^{2^{E\cup\SET{e}}}$, and for all $X\subseteq E\cup\SET{e}$,
	 $$\left(\nu_N \ast \mu_{\Alpha_N}\right)(X) = \sum_{W\subseteq E\cup\SET{e}} \nu_N(W)\cdot \mu_{\Alpha_N}(W,X) = \alpha_N(X)$$ 
	 by Lemma~\ref{lem:alphaMoebius}, therefore the equation also holds for the vector restricted to $\Z^{2^E}$.
	 The vector $\nu_N \ast \left( \mu_{\Alpha_N} \restrict 2^{E\cup\SET{e}} \times 2^E \right)$ on the right arises by first restricting
	 $\mu_{\Alpha_N}$ to $2^{E\cup\SET{e}}\times 2^E$, effectively dropping all rows ${\left( \mu_{\Alpha_N} \right)}_R$ from $\mu_{\Alpha_N}$ where
	 $e\in R\subseteq E\cup\SET{e}$, and only afterwards calculating the product. For all $X\subseteq E$, we still have to compute
	 \[ \left( \nu_N \ast \left( \mu_{\Alpha_N} \restrict 2^{E\cup\SET{e}} \times 2^E \right) \right)(X) = \sum_{W\subseteq E\cup\SET{e}} \nu_N(W)\cdot \mu_{\Alpha_N}(W,X).\]
	 For the next equation, we need the property $(\ast)$ that allows us to drop all the summands that belong to $W\subseteq E\cup\SET{e}$
	 with $e\in W$ on the left-hand side:
	 \begin{align*} \sum_{W\subseteq E\cup\SET{e}} \nu_N(W)\cdot \mu_{\Alpha_N}(W,X) 
	  & \stackrel{(\ast)}{=} \sum_{W\subseteq E} \nu_N(W)\cdot \mu_{\Alpha_N}(W,X)\\ & =
	 \left( \left(\nu_N\restrict_{2^E}\right) \ast \left( \mu_{\Alpha_N} \restrict 2^E\times 2^E  \right) \right)(X) .
	 \end{align*}

\begin{lemma}\label{lem:AlphaPosetDownsetsExtension}\PRFR{Feb 15th}
	Let $M=(E,\Ical)$ be a matroid, $e\notin E$, $N\in\Xcal(M,e)$ a single element extension of $M$, and 
	$C = \SET{F\in\Fcal(M)~\middle|~e\in\cl_N(F)}$ the corresponding modular cut.
	Further, let $(\Alpha_M,\sqsubseteq_M)$ be the $\alpha_M$-poset, and $(\Alpha_N,\sqsubseteq_N)$ 
	be the $\alpha_N$-poset.
	Then for all $X\subseteq E$ and all $Y\subseteq E\cup\SET{e}$ with $X\not= Y$
	\[ X\sqsubseteq_N Y \quad\Longleftrightarrow\quad X \sqsubseteq_M Y \txtand X\notin C. \]
\end{lemma}

\begin{proof}\PRFR{Feb 15th}\PRFR{ ~~ + sage}
	Lemma~\ref{lem:flatsOfExtension} yields $\Fcal(N)\cap 2^E = \Fcal(M)\BS C$ and the statement of this lemma follows from Definition~\ref{def:AlphaPoset}.
\end{proof}

\begin{lemma}\label{lem:AlphaPosetDownsetsExtensionWithE}\PRFR{Feb 15th}
	Let $M=(E,\Ical)$ be a matroid, $e\notin E$, $N\in\Xcal(M,e)$ be an extension of $M$,
	$C = \SET{F\in\Fcal(M)~\middle|~ e\in \cl_N(F)}$ the corresponding modular cut, and
	let $\Alpha_M$ and $\Alpha_N$ denote the $\alpha_M$- and $\alpha_N$-posets,
	respectively.
	Then for all $X,Y\subseteq E$, we have
	\begin{enumerate}\ROMANENUM
	\item \( X\sqsubseteq_{\Alpha_N} Y\cup\SET{e} \) holds if and only if $X\sqsubseteq_{\Alpha_M} Y$ and $X\notin C$.\\
	Furthermore, if $Y\sqsubseteq_{\Alpha_N} Y\cup\SET{e}$ then $Y\cup\SET{e}\in \Fcal(N)$.
	%, and whenever $X=Y$ holds, then $Y\in\Fcal(M)$;
	\item \( X\cup\SET{e} \sqsubseteq_{\Alpha_N} Y\cup\SET{e}\) holds if and only if $X\sqsubseteq_{\Alpha_M} Y$ and $X\notin \partial C$
	where \[ \partial C = \SET{F\in\Fcal(M)\BS C ~\middle|~ \exists x\in E\BS F\colon\,\cl_M(F\cup\SET{x}) \in C}. \]
	\end{enumerate}
\end{lemma}
\begin{proof}\PRFR{Feb 15th}
This is clear from Lemma~\ref{lem:flatsOfExtension} and Definition~\ref{def:AlphaPoset}, too.
\end{proof}

\begin{remark}\label{rem:moebiusAMstableBelowC}\PRFR{Feb 15th}
As we have just seen, the $\Alpha_N$-down-sets of subsets of $E$  are the corresponding down-sets 
of the $\alpha_M$-poset $\Alpha_M$ where the upper part, that corresponds to the modular cut $C$ of the 
single element extension $N$ of $M$, has been cut off. 
Since the values of the Möbius-function $\mu_{P}(X,Y)$ for an arbitrary poset $P$
only depend on the $P$-down-sets of elements of the $P$-down-set of $Y$ (Definition~\ref{def:moebiusFunction}), 
we see that for $X\subseteq E$ and $Y\subseteq E$ with $C\cap 2^Y \subseteq \SET{Y}$ 
we have $\mu_{\Alpha_M}(X,Y) = \mu_{\Alpha_N}(X,Y)$ and consequently
$\alpha_M(Y) = \alpha_N(Y)$.
\end{remark}

\begin{corollary}\label{cor:muAlphaAsSumWithDelta}\PRFR{Feb 15th}
Let $M=(E,\Ical)$ be a matroid, $e\notin E$, $N\in\Xcal(M,e)$ a single element extension of $M$, and 
	$C = \SET{F\in\Fcal(M)~\middle|~e\in\cl_N(F)}$ the corresponding modular cut.
	Then for all $X,Y\subseteq E$
	\[ \mu_{\Alpha_N}(X,Y) = \mu_{\Alpha_M}(X,Y) + \sum_{Z\in C,\,X\subseteq Z\subsetneq Y} \mu_{\Alpha_M}(X,Z).\]
\end{corollary}
\begin{proof}\PRFR{Feb 15th}\PRFR{~~ + sage}
	The first equation is a direct consequence of Lemma~\ref{lem:AlphaPosetDownsetsExtension} and Remark~\ref{rem:moebiusAMstableBelowC}:
	\begin{align*} - \sum_{X\sqsubseteq_{M} Z \sqsubset_{M} Y} \mu_{\Alpha_M}(X,Z)  \,\,\,=\,\,\, &
	- \left( \sum_{X\sqsubseteq_{N} Z \sqsubset_{N} Y} \mu_{\Alpha_N}(X,Z) \right)
	\\ & - 
	\left( \sum_{X\sqsubseteq_{M} Z \sqsubset_{M} Y, Z\in C} \mu_{\Alpha_M}(X,Z) \right)
	\end{align*}
	holds for all $X,Y\subseteq E$. Thus we may expand the terms $\mu_{\Alpha_N}(X,Y)$ and $\mu_{\Alpha_M}(X,Y)$ with Definition~\ref{def:alphaM}, and then cancel in the above equation.
\end{proof}

\begin{definition}\label{def:DeltaAlphaInvariant}\PRFR{Feb 15th}
	Let $M=(E,\Ical)$ be a matroid and let $\Mcal(M)$ be the class of all modular cuts of $M$.\label{n:Deltaalphainvariant}
	The \deftext[da-invariant of M@$\Delta\alpha$-invariant of $M$]{$\bm \Delta \bm \alpha$-invariant of $\bm M$}
	shall be defined as
	\[
		\Delta \alpha_M \colon \Mcal(M)\times 2^E \maparrow \Z,\]\[
		(C,X) \mapsto \sum_{Y\subsetneq X} \left( \left( \left| Y \right| - \rk(Y) \right) \cdot \sum_{Z\in C,\,Y\subseteq Z \subsetneq X} \mu_{\Alpha}(Y,Z)\right),
	\]
	where $\mu_\Alpha$ denotes the Möbius-function of the $\alpha_M$-poset.
	If the matroid $M$ is clear from the context, we will denote $\Delta\alpha_M$ simply by $\Delta\alpha$.
\end{definition}

\begin{lemma}\label{lem:alphaNXThroughDeltaalphaM}\PRFR{Feb 15th}
	Let $M=(E,\Ical)$ be a matroid, $e\notin E$, $N\in\Xcal(M,e)$ be an extension of $M$,
	$C = \SET{F\in\Fcal(M)~\middle|~ e\in \cl_N(F)}$ the corresponding modular cut, and
	$X\subseteq E$.
	Then \[ \alpha_N(X) = \alpha_M(X) + \Delta\alpha_M(C,X).\]
\end{lemma}
\begin{proof}\PRFR{Feb 15th}\PRFR{~~ + sage}
	Let $X\subseteq E$, and let $\Alpha_M$ and $\Alpha_N$ denote the $\alpha_M$- and $\alpha_N$-posets, respectively.
	From Corollary~\ref{cor:muAlphaNFromMuAlphaM} and Lemmas~\ref{lem:alphaMoebius} and \ref{lem:moebiusInversion} we obtain the equation
	\[ \alpha_N(X) = \sum_{Y\subseteq X} \left( \left( \left| Y \right| - \rk_M(Y) \right)\cdot \mu_{\Alpha_N}(Y,X)\right) .\]
	Corollary~\ref{cor:muAlphaAsSumWithDelta}  yields
	\[ \mu_{\Alpha_N}(Y,X) = \mu_{\Alpha_M}(Y,X) + \sum_{Z\in C,\,Y\subseteq Z \subsetneq X} \mu_{\Alpha_M}(Y,Z) \]
	and therefore applying the distributive law of $\Z$ together with Definition~\ref{def:DeltaAlphaInvariant} yields 
	the desired equation
	\begin{align*}
		 \alpha_N(X) \,\,\,=\,\,\, & \sum_{Y\subseteq X} 
		 \left( \left( \left| Y \right| - \rk_M(Y) \right)\cdot\left(\mu_{\Alpha_M}(Y,X)
		 + \sum_{Z\in C,\,Y\subseteq Z \subsetneq X} \mu_{\Alpha_M}(Y,Z)\right)  \right) 
		  \\
		  = \,\,\,& \alpha_M(X) + \Delta\alpha_M(C,X). \qedhere
	\end{align*}
\end{proof}

\needspace{6\baselineskip}

\begin{lemma}\label{lem:alphaNXbelowC}\PRFR{Feb 15th}
	Let $M=(E,\Ical)$ be a matroid, $e\notin E$, $N\in\Xcal(M,e)$ be an extension of $M$,
	$C = \SET{F\in\Fcal(M)~\middle|~ e\in \cl_N(F)}$ the corresponding modular cut, and
	$X\subseteq E$ such that % $\cl_M(X\cup\SET{y}) \notin C$ holds for all $y\in E\BS X$.
	 $\rk_M \left( F\cap X' \right) < \rk_M(F)$ for all $F\in C$ and all proper subsets $X'\subsetneq X$.
	Then $$ \alpha_N(X\cup\SET{e}) = \begin{cases}[r] 0 & \quad\text{if~} X\in \Fcal(M) \txtand \cl_M(X)\notin C, \\
													%  1 & \quad\text{if~} X\in \Fcal(M) \txtand \cl_M(X)\in C, \\
												     \alpha_M(X) & \quad\text{if~} X\notin \Fcal(M) \txtand \cl_M(X)\notin C,\\
												    1 + \alpha_M(X) & \quad\text{if~}  \cl_M(X)\in C.\\
												     \end{cases} $$
\end{lemma}
\begin{proof}\PRFR{Feb 15th}\PRFR{~~ + sage}
	Let $(\Alpha_M,\sqsubseteq_M)$ and $(\Alpha_N,\sqsubseteq_N)$ denote the $\alpha_M$- and $\alpha_N$-poset, respectively.
	Let $W\subseteq X$, then $W$ satisfies the premises of this lemma whenever $X$ satisfies the premises.
	Furthermore, if for some $F\in C$ the equality $\rk_M(F\cap X) = \rk_M(F)$ holds,
	then
	\linebreak
	 $F = \cl_M(X)$ and conversely, if $\cl_M(X)\notin C$, then $\rk_M(F\cap X) < \rk_M(F)$ for all $F\in C$.

	%\noindent
	Now, we prove the statement for all $X\in\Fcal(M)$ with $\cl_M(X)\notin C$ by induction on $\rk_M(X)$.
	Let $O=\cl_M(\emptyset)$ be the unique rank-$0$ flat of $M$.
	 Then
	 the down-set  $\downarrow_{\Alpha_N}\left( O\cup\SET{e} \right) = \SET{O,O\cup\SET{e}}$.
	Thus, by Definitions~\ref{def:alphaM} and \ref{def:AlphaPoset},
	we have \begin{align*}
		\alpha_N(O\cup \SET{e}) = & \left| O\cup\SET{e} \right| - \rk_N(O\cup \SET{e}) - \alpha_N(O)\\
	=& \left| O \right| + 1 - 1 - \left( \left| O \right| - \rk_N(O) \right) = 0.
	\end{align*}
	Now let $X\in\Fcal(M)$ be a flat with $\rk_M(X) > 0$. Lemma~\ref{lem:AlphaPosetDownsetsExtensionWithE}
	yields that
	$$\downarrow_{\Alpha_N} \left( X\cup\SET{e} \right) = \SET{F,F\cup\SET{e}~\middle|~ F\in\,\, \downarrow_{\Alpha_M} X}.$$
	Note that $X\cup\SET{e}$ may or may not be a flat in $N$, as we have $X\cup\SET{e}\notin\Fcal(N)$ if $X\in\Fcal(M)$ 
	and $X$ is covered by a flat from $C$ --- but $X\cup\SET{e}$ is still an element of the above down-set.
	The assumption, that  $\rk_M \left( F\cap X \right) < \rk_M(F)$  for all $F\in C$, guarantees that all $F\in\Fcal(M,X)$
	are flats of $N$, too.
	Furthermore, we have
	\[ \alpha_N(X\cup\SET{e}) = \left| X\cup\SET{e} \right| - \rk_N(X\cup\SET{e}) - \sum_{F\sqsubset_N X} \alpha_N(F).\]
	Using the induction hypothesis, we obtain
	\begin{align*}
	 \alpha_N(X\cup\SET{e}) = &
	  \left| X\cup\SET{e} \right| - \rk_N(X\cup\SET{e}) - \left( \sum_{F\sqsubset_N X} \alpha_N(F) \right) - \alpha_N(X) 
	 \\
	 = & \left| X \right| - \rk_N(X) -\left(  \sum_{F\sqsubset_N X} \alpha_N(F)  \right)
	 - \left( \left| X \right| - \rk_N(X) - \sum_{F\sqsubset_N X} \alpha_N(F) \right) \\ = & \,\,0. \\
	\end{align*}
	Now let $X\subseteq E$ with $\cl_M(X) \notin C$ and $X\notin \Fcal(M)$.
	Then
	$$\downarrow_{\Alpha_N} \left( X\cup\SET{e} \right) = \SET{F,F\cup\SET{e}~\middle|~ F\in\,\, \downarrow_{\Alpha_M} X}\BSET{X}.$$
	Analogously to the above calculation we obtain
	\begin{align*}
	 \alpha_N(X\cup\SET{e}) = & \left| X\cup\SET{e} \right| - \rk_N(X\cup\SET{e}) - \sum_{F\sqsubset_N X} \alpha_N(F)
	 \\ = &\,\, \alpha_N(X) = \alpha_M(X), \end{align*}
	 where the last equation is due to the fact that  $F\not\subseteq X$ holds for all $F\in C$,
	 which implies that $N\restrict X = M\restrict X$ and therefore
	  $\alpha_N(X) = \alpha_{N\restrict X}(X) = \alpha_{M\restrict X}(X) = \alpha_M(X)$
	 (Definition~\ref{def:alphaM}).	

	 \noindent
	 Now assume that $\cl_M(X)\in C$. 
	 If $X\in\Fcal(M)$, then $e\in\cl_N(X)$, thus $X\notin \Fcal(N)$. Otherwise $X\notin\Fcal(M)$ and therefore $X\notin\Fcal(N)$, too. 
	 In both cases we obtain that
	 $$\SET{F\subseteq E\cup\SET{e}~\middle|~\vphantom{A^A} F \sqsubset_N X\cup\SET{e}}
	  = \SET{F,F\cup\SET{e}~\middle|~ F\in\,\, \downarrow_{\Alpha_M} X} \BSET{\vphantom{A^A}X, X\cup\SET{e}}.$$
	 Furthermore, for all $X'\subsetneq X$ we have $\cl_M(X')\notin C$, 
	 because  $\rk_M \left( F\cap X' \right) < \rk_M(F)$ for all $F\in C$.
	 This implies that if $F\cup\SET{e} \sqsubset_{\Alpha_N} X$ for some $F\in\Fcal(M)$, then $\alpha_N(F\cup\SET{e}) = 0$.
	 Consequently, with Lemma~\ref{lem:flatsOfExtension}, we obtain
	 \[ \sum_{F\sqsubset_N X\cup\SET{e}} \alpha_N(F) = \sum_{F\sqsubset_N X\cup\SET{e},\,e\notin F} \alpha_N(F)
	 = \sum_{F\sqsubset_N X} \alpha_N(F) = \sum_{F\sqsubset_M X} \alpha_M(F).\]
	 Since $e\in\cl_N(X)$, we have $\rk_N(X\cup\SET{e}) = \rk_N(X)$. This yields the desired equation
	 \begin{align*}
	 	 \alpha_N(X\cup\SET{e}) & = \left| X\cup\SET{e} \right| - \rk_N(X\cup\SET{e}) - \sum_{F\sqsubset_N X\cup\SET{e}} \alpha_N(F) \\
	 	 & = 1 + \left| X \right| - \rk_M(X) - \sum_{F\sqsubset_M X} \alpha_M(F) = 1 + \alpha_M(X). \qedhere
	 	\end{align*}
\end{proof}

\noindent In order to determine the values of $\alpha_N(X)$ of the extension $N$ of $M$ by $e$ when $e\in X$ and $e\in\cl_N(X\BSET{e})$,
		we have to keep track of the flats $F$ of $M$ that are proper subsets $X$ with the additional property that $e\in\cl_N(F)$.
\begin{definition}\PRFR{Feb 15th}
	Let $M=(E,\Ical)$ be a matroid and let $C\in\Mcal(M)$ be a modular cut of $M$.
	We define the \deftext[extension poset of $C$ with respect to $M$]{extension poset of $\bm C$ with respect to $\bm M$}
	as the pair \label{n:BetaMC} $\left(\Beta_M^C, \sqsubseteq_M^C \right)$ where
	\( \Beta_M^C = 2^E \) and  where
	\[ X \sqsubseteq_M^C Y \quad \Longleftrightarrow\quad X = Y \txtor \left( X\in C \txtand X\subseteq Y\right)\]
	holds for all $X,Y\subseteq E$. If $M$ is clear from the context, we will denote $\Beta_M^C$ by $\Beta^C$ and
	$\sqsubseteq_M^C$ by $\sqsubseteq^C$, too.
\end{definition}
\begin{remark}\PRFR{Feb 15th}
	Clearly, $\sqsubseteq_M^C$ is reflexive, the anti-symmetry of $\Beta_M^C$ follows from the anti-symmetry of $\subseteq$.
	Let $X \sqsubset_M^C Y \sqsubset_M^C Z$. Then $X,Y\in C$ and $X\subsetneq Y\subsetneq Z$. Therefore $X\sqsubset_M^C Z$ holds,
	and $\Beta_M^C$ is indeed a poset.
\end{remark}

 \needspace{8\baselineskip}
\begin{definition}\PRFR{Feb 15th}
	\label{def:DeltaPAlphaInvariant}
	Let $M=(E,\Ical)$ be a matroid and let $\Mcal(M)$ be the class of all modular cuts of $M$.\label{n:DeltaPalphainvariant}
	The \deftext[da-invariantX of M@$\DeltaP\alpha$-invariant of $M$]{$\bm \DeltaP \bm \alpha$-invariant of $\bm M$}
	shall be defined as
	\[
		\DeltaP \alpha_M \colon \Mcal(M)\times 2^E \maparrow \Z,\]\[
		(C,X) \mapsto \begin{cases}[r]
							- \alpha_M(X) &\quad\text{if~} X\in \Fcal(M) \txtand \cl_M(X)\notin C, \\
							0 & \quad\text{if~} X\notin \Fcal(M) \txtand \cl_M(X)\notin C,\\
							1 - \displaystyle \sum_{F \sqsubset^C X} \DeltaP \alpha_M(C,F) & \quad\text{otherwise,}
					\end{cases}
	\]
	where $\left( \Beta^C, \sqsubseteq^C \right)$ denotes the extension poset of $C$ with respect to $M$.
	If the matroid $M$ is clear from the context, we will denote $\DeltaP\alpha_M$ simply by $\DeltaP\alpha$.
\end{definition}

\begin{lemma}\PRFR{Feb 15th}\label{lem:DPalpha}
	Let $M=(E,\Ical)$ be a matroid, $e\notin E$, $N\in\Xcal(M,e)$ be an extension of $M$,
	$C = \SET{F\in\Fcal(M)~\middle|~ e\in \cl_N(F)}$ the corresponding modular cut.
	Then $$ 
		\alpha_N(X\cup\SET{e}) = \alpha_M(X) + \DeltaP\alpha_M(C,X). $$
\end{lemma}
\begin{proof}\PRFR{Feb 15th}\PRFR{~~ + sage}
	Let $X\subseteq E$. The cases where $\cl_M(X)\notin C$ are covered by Lemma~\ref{lem:alphaNXbelowC}.
	Furthermore, if $X$ is $\subseteq$-minimal with the property that $\cl_M(X)\in C$, then $\downarrow_{\Beta_M^C} X = \SET{X}$
	and therefore $\DeltaP\alpha(C,X) = 1 = \alpha_N(X\cup\SET{e}) - \alpha_M(X)$ holds, too, by Lemma~\ref{lem:alphaNXbelowC}.
	For the general case where $\cl_M(X)\in C$, remember that we saw in the proof of Lemma~\ref{lem:alphaMoebius} that the
	equations
	\[ \left| X \right| - \rk_M(X) = \sum_{F\sqsubseteq_M X} \alpha_M(F) \]
	and
	\[ \left| X\cup\SET{e} \right| - \rk_N(X\cup\SET{e}) = \sum_{F\sqsubseteq_{N} X\cup\SET{e}} \alpha_N(F) \]
	hold. Thus we obtain
	\[ (\ast)\quad \left( \sum_{F\sqsubseteq_{M} X} \alpha_M(F)  \right) + 1 = \sum_{F\sqsubseteq_{N} X\cup\SET{e}} \alpha_N(F).\]
	We prove the missing part of the statement by induction on the length $k$ 
	of a maximal chain $C_1 \subsetneq C_2 \subsetneq \ldots \subsetneq C_k \subsetneq X$ with $C_1,\ldots,C_k\in C$.
	The base case with $k=0$ has been established above. 
%	
%	\needspace{6\baselineskip}\noindent
	Using Lemma~\ref{lem:flatsOfExtension} we obtain that
	\( \downarrow_{\Alpha_N} \left( X\cup\SET{e} \right) = Q \disunion R \disunion S \disunion T \)
	where \allowdisplaybreaks
	\begin{align*}
		Q & = \SET{Y ~\middle|~\vphantom{A^A} Y\in \Fcal(M)\BS C,\,Y\subseteq X}, \\
		R & = \SET{Y\cup\SET{e} ~\middle|~\vphantom{A^A} Y\in \Fcal(M)\BS C,\,Y\subsetneq X,\,\forall f\in E\BS Y\colon\,\cl_M(Y\cup\SET{f})\notin C}, \\
		S & = \SET{Y\cup\SET{e} ~\middle|~\vphantom{A^A} Y\in C,\,Y\subsetneq X}, \text{~and}\\
		T & = \SET{X\cup\SET{e}\vphantom{A^A}}.
	\end{align*}
	Clearly, $Q\subseteq\,\,\downarrow_{\Alpha_M} X$, and
	Lemma~\ref{lem:alphaNXThroughDeltaalphaM} and Definition~\ref{def:DeltaAlphaInvariant} yield that \[
		\sum_{F\in Q} \alpha_N(F) = \sum_{F\in Q} \alpha_M(F).
	\]
	Lemma~\ref{lem:alphaNXbelowC} yields that $ \sum_{F\in R} \alpha_N(F) = 0$. All $F\in S$ 
	have $F\BSET{e}\in C$ with $F\BSET{e}\subsetneq X$
	and therefore those sets $F\BSET{e}$ have shorter maximal descending
	chains in $C$ than $X$. 
	The induction hypothesis applied to each summand yields that
	\[ \sum_{F\in S} \alpha_N(F) = \sum_{F\in S} \left( \alpha_M(F\BSET{e}) + \DeltaP\alpha_M(C,F\BSET{e}) \right).\]
	Furthermore, observe that $X\notin Q$ because $\cl_M(X) \in C$ holds, and so we have the equivalence
	\[ F \sqsubset_M X \quad\Longleftrightarrow \quad F\in Q \txtor F\cup\SET{e} \in S \]
	for all $F\subseteq E$: Elements $F$ of the $\Alpha_M$-down-set of $X$ have either
	 $F\in \Fcal(M)\BS C$ or $F\in C$, thus either $F\in Q$ or $F\cup\SET{e} \in S$.
	 Therefore we may cancel the corresponding summands of $\downarrow_{\Alpha_M} X$ and drop the zero summands from $R$
	 in the equation $(\ast)$ and obtain
	 \begin{align*}
	   \alpha_M(X) + 1& =  \alpha_N(X) + \sum_{F\in S} \DeltaP\alpha_M(C,F\BSET{e}) .\\
	 \end{align*}
	 Since all $F\in S$ have $e\in F$, and since 
	 \[ \SET{F\BSET{e} \vphantom{A^A}~\middle|~ F\in S} = \SET{F\in C \vphantom{A^A}~\middle|~ F\subsetneq X} = 
	 \SET{F\subseteq E~\middle|~ \vphantom{A^A}F \sqsubset^C X} \]
	 we obtain the desired equation
	 \[ \alpha_N(X\cup\SET{e}) \,\,\,=\,\,\, \alpha_M(X) + 1 - \sum_{F\sqsubset^C X} \DeltaP\alpha_M(C,F) \,\,\,=\,\,\, \alpha_M(X) + \DeltaP\alpha_M(C,X). \qedhere \]
\end{proof}

\noindent
We implemented and tested the performance of determining the $\alpha_N$-invariant for single element extensions $N\in \Xcal(M,e)$ by means of the formulas
given in Lemmas~\ref{lem:alphaNXbelowC} and \ref{lem:DPalpha}.
For details, please refer to Listing~\ref{lst:measureDeltaAlpha}.

\clearpage
% -*- root: ../thesis.tex -*-

\needspace{8\baselineskip}
\section{Matroid Tableaux}

\PRFR{Mar 29th}
\noindent In this section, we present a general framework for the decision of $\mathrm{Rec}\Gamma_{\Mcal}$ instances\footnote{Remember that in this chapter starting from Section~\ref{sec:generalCase}, $\Mcal$ denotes the class of all matroids.}
by searching the domain of matroids defined on ground sets with bounded cardinality by the means of tableaux and derivations. 

%Our method may be considered a generalization of depth first search.
%The main idea of the Matroid Construction Method is the following: starting with the matroid $M$ for which we would like to
%know the value of $\Gamma_\Mcal(M)$, we construct new matroids using various methods --- single element extension,
%deflation, dualization, for instance --- with the property that if any matroid constructed is a gammoid, then so is $M$.
%We maintain an equivalency relation of all matroids encountered with the property that two
%matroids from the same equivalency class are either both gammoids or both non-gammoids. We also maintain a special implication structure
%on the equivalency classes with the property that the the consequent class is a gammoid if and only if at least one of the antecedent classes is a gammoid.
%Furthermore, we maintain a set of active equivalency classes, which consists of all classes of matroids that have not been proven to contain a
%non-gammoid so far.
%As soon as we encounter a matroid $N$ which cannot be a gammoid, we remove the corresponding class of matroids from the active class set and update the 
%implication structure accordingly. This may lead to further classes of matroids to be ruled out, too. If the equivalency class of the initial matroid
%is among the deactivated classes, we have determined $\Gamma_\Mcal(M)$. On the other hand if we encounter a matroid $N$ which is a gammoid, then
%we may immediately give a positive answer.
% Let us elaborate the details now.

\needspace{5\baselineskip}
\begin{definition}\PRFR{Mar 29th}
	A \deftext{matroid tableau} is a tuple \label{n:mattab} $\Tbf = (G,\Gcal,\Mcal,\Xcal,\simeq)$ where
	\begin{enumerate}\ROMANENUM
		\item $G$ is a matroid, called the \deftextX{goal of $\bm \Tbf$},
		\item $\Gcal$ is a family of matroids, called the \deftextX{gammoids of $\bm \Tbf$},
		\item $\Mcal$ is a family of matroids, called the \deftextX{intermediates of $\bm \Tbf$},
		\item $\Xcal$ is a family of matroids, called the \deftextX{excluded matroids of $\bm \Tbf$}, and where
		\item $\simeq$ is an equivalence relation on $\SET{G'~\middle|~ G'\text{~is a minor of~}G}\cup\Gcal \cup \Mcal \cup \Xcal$, called the \deftextX{equivalence of $\bm \Tbf$}. \qedhere
	\end{enumerate}
\end{definition}

\needspace{5\baselineskip}
\begin{definition}\label{def:validTableau}\PRFR{Mar 29th}
	Let $\Tbf = (G,\Gcal,\Mcal,\Xcal,\simeq)$ be a matroid tableau.
	$\Tbf$ shall be \deftext[valid matroid tableau]{valid},
	\begin{enumerate}\ROMANENUM
	\item  if
	all matroids in $\Gcal$ are indeed gammoids,
	\item if no matroid in $\Mcal$ is a strict gammoid,
	\item if all matroids in $\Xcal$ are indeed matroids which are not gammoids, and
	\item  if for every equivalency classes $[M]_\simeq$ of $\simeq$ we have that either $[M]_\simeq$ is fully contained in the class of gammoids
	or $[M]_\simeq$ does not contain a gammoid. \qedhere
\end{enumerate}
\end{definition}

\begin{definition}\label{def:decisiveTableau}\PRFR{Mar 29th}
	Let $\Tbf = (G,\Gcal,\Mcal,\Xcal,\simeq)$ be a matroid tableau. $\Tbf$ shall be \deftext[decisive matroid tableau]{decisive},
	if $\Tbf$ is valid and 
	if either of the following holds:
	\begin{enumerate}\ROMANENUM 
	\item There is a matroid $M\in \Gcal$ such that $G \simeq M$.
	\item There is 
	a matroid $X\in \Xcal$ that is isomorphic to a minor of $G$.
	\item For every extension $N=(E',\Ical')$ of $G=(E,\Ical)$ 
	with $$\left| E' \right| = \rk_G(E)^2\cdot \left| E \right| + \rk_G(E) + \left| E \right|$$ there is a matroid $M\in\Mcal$ that is
	isomorphic to $N$. \qedhere
\end{enumerate}
\end{definition}

\needspace{3\baselineskip}
\begin{lemma}\label{lem:decisiveTableau}\PRFR{Mar 29th}
	Let $\Tbf = (G,\Gcal,\Mcal,\Xcal,\simeq)$ be a decisive matroid tableau. Then $G$ is a gammoid if and only if there is a matroid $M\in \Gcal$
	such that $G\simeq M$.
\end{lemma}
\begin{proof}\PRFR{Mar 29th}
	Assume that such an $M\in \Gcal$ exists. From Definition~\ref{def:validTableau}
	we obtain that $M$ is a gammoid, and that in this case $G\simeq M$ implies that $G$ is a gammoid, too.
	Now assume that no $M\in \Gcal$ has the property $G\simeq M$. Since $\Tbf$ is decisive, either case {\em (ii)} or {\em (iii)} of
	Definition~\ref{def:decisiveTableau} holds. If case {\em (ii)} holds, then $G$ cannot be a gammoid since it has a non-gammoid minor,
	but the class of gammoids is closed under minors (Theorem~\ref{thm:GammoidsClosedMinorsDuality}).
	If case {\em (iii)} holds but not case {\em (ii)}, then no extension of $G=(E,\Ical)$ with $k=\rk_G(E)^2\cdot \left| E \right| $ $+\, \rk_G(E) + \left| E \right|$
	elements is a strict gammoid. Now assume that $G$ is a gammoid, then there is a digraph $D=(V,A)$ with $\left| V \right| \leq k$ vertices, such that
	$G = \Gamma(D,T,E)$
	for some $T\subseteq V$ (Remark~\ref{rem:upperBoundForV}). 
	Let $N' = \Gamma(D,T,V)\oplus (\SET{\left| V \right|,\left| V \right|+1,\ldots,k},\SET{\emptyset})$. 
	Clearly, $N'$ is an extension of $G$ on a ground set with $k$ elements, which is also a strict gammoid, a contradiction to the assumption that $N'$ is 
	isomorphic to some $N\in \Mcal$, since $\Mcal$ is a family which consists of matroids that are not strict gammoids. Therefore we may conclude that in case {\em (iii)} the matroid $G$ is not a gammoid.
\end{proof}

\subsection{Valid Derivations}

\PRFR{Mar 29th}
\noindent A \deftext{derivation} is an operation on a finite number of input tableaux and possible additional parameters with constraints
that produces an output tableau. Furthermore,
a derivation is \deftext[valid derivation]{valid}, if the output tableau is valid for all sets of valid input tableaux and possible additional parameters that
satisfy the constraints. The valid derivations presented here are fairly straight-forward consequences of the concepts presented earlier in this work.

\begin{definition}\PRFR{Mar 29th}
	Let $\Tbf_i = (G_i,\Gcal_i,\Mcal_i,\Xcal_i,\simeq^{(i)})$ be matroid tableaux for $i\in \SET{1,2,\ldots,n}$.
	The \deftext{joint tableau} shall be the matroid tableaux \label{n:jointTableau}
	\[\bigcup_{i=1}^{n} \Tbf_i = (G_1,\Gcal,\Mcal,\Xcal,\simeq)\]
	where \[ \Gcal = \bigcup_{i=1}^n \Gcal_i,\,\,\, \Mcal = \bigcup_{i=1}^n \Mcal_i,\,\,\, \Xcal = \bigcup_{i=1}^n \Xcal_i, \]
	and where $\simeq$ is the smallest equivalence relation such that $M \simeq^{(i)} N$ implies $M \simeq N$ for all $i\in \SET{1,2,\ldots,n}$.
	In other words, $\simeq$ is the equivalence relation on the family of matroids
	 $\SET{G'~\middle|~ G'\text{~is a minor of~}G}\cup\Gcal \cup \Mcal \cup \Xcal$ which is
	generated by the relations $ \simeq^{(1)}, \simeq^{(2)}, \ldots, \simeq^{(n)}$.
\end{definition}

\needspace{2\baselineskip}
\begin{lemma}\PRFR{Mar 29th}
	The derivation of the joint tableau is valid.
\end{lemma}
\begin{proof}\PRFR{Mar 29th}
	Clearly, $\Gcal$, $\Mcal$, and $\Xcal$ inherit their desired properties of Definition~\ref{def:validTableau} from
	the valid input tableaux $\Tbf_i$  where $i\in\SET{1,2,\ldots, n}$. Now let $M \simeq N$ with $M\not= N$.
	Then there are matroids $M_1,M_2,\ldots,M_{k}$ and indexes $i_0,i_1,\ldots,i_k\in \SET{1,2,\ldots,n}$ such that
	there is a chain of $\simeq^{(i)}$-relations
	\[ M \simeq^{(i_0)} M_1 \simeq^{(i_1)} M_2 \simeq^{(i_2)} \cdots \simeq^{(i_{k-1})} M_k \simeq^{(i_k)} N. \]
	The assumption that the input tableaux are valid yields that $M$ is a gammoid if and only if $M_1$ is a gammoid,
	if and only if $M_2$ is a gammoid, and so on. Therefore it follows that $M$ is a gammoid if and only if $N$ is a gammoid,
	thus $\simeq$ has the desired property of Definition~\ref{def:validTableau}. Consequently, $\bigcup_{i=1}^n \Tbf_i$ is a valid tableau.
\end{proof}

\begin{definition}\PRFR{Mar 29th}
	Let $\Tbf = (G,\Gcal,\Mcal,\Xcal,\simeq)$ and $\Tbf' = (G,\Gcal',\Mcal',\Xcal',\simeq')$ be matroid tableaux.
	We say that $\Tbf$ is a \deftext[sub-tableau]{sub-tableau of $\Tbf\bm'$} if $\Gcal \subseteq \Gcal'$, $\Mcal \subseteq \Mcal'$, and
	$\Xcal \subseteq \Xcal'$ holds, and if $M \simeq N$ implies $M \simeq' N$.
\end{definition}

\needspace{2\baselineskip}
\begin{lemma}\PRFR{Mar 29th}
	The derivation of a sub-tableau is valid.
\end{lemma}
\begin{proof}\PRFR{Mar 29th}
	Clearly $\Tbf$ inherits the properties of Definition~\ref{def:validTableau} from the validity of $\Tbf'$.
\end{proof}

\begin{definition}\PRFR{Mar 29th}
	Let $\Tbf = (G,\Gcal,\Mcal,\Xcal,\simeq)$  be a matroid tableau. We shall call the
	\label{n:expTab}
	 tableau $[\Tbf]_\simeq= (G,\Gcal',\Mcal,\Xcal',\simeq)$   \deftext[expansion tableau]{expansion tableau of $\Tbf$} 
	whenever \[ \Gcal' = \bigcup_{M\in\Gcal} [M]_\simeq \quad\txtand\quad \Xcal' = \bigcup_{M\in\Xcal} [M]_\simeq. \qedhere\]
\end{definition}

\needspace{2\baselineskip}
\begin{lemma}\PRFR{Mar 29th}
	The derivation of the expansion tableau is valid.
\end{lemma}
\begin{proof}\PRFR{Mar 29th}
	If $M'\in \Gcal'$, then there is some $M\in \Gcal$ such that $M\simeq M'$. Since we assume $\Tbf$ to be valid, we may infer that
	$M'$ is a gammoid if and only if $M$ is a gammoid, and the latter is the case since $M\in\Gcal$. Therefore $M'$ is a gammoid.
	An analogous argument yields that if $M'\in \Xcal'$, then $M'$ is not a gammoid.
\end{proof}

\needspace{5\baselineskip}
\begin{definition}\PRFR{Mar 29th}
  Let $\Tbf = (G,\Gcal,\Mcal,\Xcal,\simeq)$  be a matroid tableau. We shall call the
  \label{n:extTab}
	 tableau $[\Tbf]_{\equiv} = (G,\Gcal',\Mcal',\Xcal',\simeq')$ \deftext[extended tableau]{extended tableau of $\bm \Tbf$}
	 whenever $$\Gcal' = \Gcal \cup \SET{M^\ast~\middle|~M\in \Gcal},\,\,\,
	 		  \Xcal' = \Xcal \cup \SET{M^\ast~\middle|~M\in \Xcal},\,\,\, 
	   \Mcal' = \Mcal \cup \Xcal',$$ and when 
	   $\simeq'$ is the smallest equivalence relation that contains the 
	   relations $\simeq$ and $\sim$; where $M\sim N$ if and only if $N$ is
	   isomorphic to $M$ or $M^\ast$.
\end{definition}

\needspace{2\baselineskip}
\begin{lemma}\PRFR{Mar 29th}
	The derivation of the extended tableau is valid.
\end{lemma}
\begin{proof}\PRFR{Mar 29th}
	By Theorem~\ref{thm:GammoidsClosedMinorsDuality} the class of gammoids is closed under duality, therefore a matroid $M$ is a gammoid if and only if $M^\ast$ is a gammoid. So $\Gcal'$ and $\Xcal'$ inherit their desired properties of Definition~\ref{def:validTableau} from the validity of $\Tbf$.
	If $M\in \Mcal'\BS \Mcal$, then $M\in \Xcal'$, therefore $M$ cannot be a strict gammoid.
\end{proof}

\begin{definition}\label{def:decisionTableau}\PRFR{Mar 29th}
	Let $\Tbf = (G,\Gcal,\Mcal,\Xcal,\simeq)$ be a decisive matroid tableau. The tableau
	\label{n:concTab}
	 $\Tbf! = (G,\Gcal',\Mcal,\Xcal',\simeq)$ shall be the \deftext[conclusion tableau]{conclusion tableau for $\Tbf$} if either
	\begin{enumerate}\ROMANENUM
	\item $\Gcal' = \Gcal \cup \SET{G'~\middle|~ G'\text{~is a minor of~} G}$, $\Xcal' = \Xcal$, and the tableau $\Tbf$ 
	satisfies case {(i)} of Definition~\ref{def:decisiveTableau}; or
	\item $\Gcal' = \Gcal$, $\Xcal' = \Xcal\cup\SET{G}$, and $\Tbf$ satisfies case {(ii)} or {(iii)} of Definition~\ref{def:decisiveTableau}. \qedhere
	\end{enumerate}
\end{definition}

\needspace{2\baselineskip}
\begin{corollary}\PRFR{Mar 29th}
	The derivation of the conclusion tableau is valid.
\end{corollary}
\begin{proof}\PRFR{Mar 29th}
	Easy consequence of Lemma~\ref{lem:decisiveTableau}.
\end{proof}

\begin{definition}\PRFR{Mar 29th}
	Let $\Tbf = (G,\Gcal,\Mcal,\Xcal,\simeq)$  be a matroid tableau, let $M_1=(E_1,\Ical_1)$ and $M_2=(E_2,\Ical_2)$ be matroids of the
	tableau, i.e.
	\[ \SET{M_1,M_2}\subseteq  \SET{G'~\middle|~ G'\text{~is a minor of~}G}\cup\Gcal \cup \Mcal \cup \Xcal .\]
	Furthermore, let $E_1'$ and $E_2'$ be finite sets,
	$D_1=(V_1,A_1)$ and $D_2=(V_2,A_2)$ be digraphs such that $E_1 \cup E_2' \subseteq V_1$ 
	and $E_1'\cup E_2\subseteq V_2$, and such that
	the induced matroid \( I(D_1,M_1,E_2')\) is isomorphic to \( M_2 \) 
	and the the induced matroid \( I(D_2,M_2,E_1') \) is isomorphic to \(M_1\).
	The tableau
	\label{n:idTab}
	 $$\Tbf(M_1\simeq M_2) = (G,\Gcal,\Mcal,\Xcal,\simeq')$$ is called 
	\deftext[identified tableau]{identified tableau for $\Tbf$ with respect to $\bm M_1$ and $\bm M_2$} if
	the relation
	$\simeq'$ is the smallest equivalence relation, such that $M_1\simeq' M_2$ holds, and such that  $M'\simeq N'$  implies $M'\simeq' N'$.
\end{definition}

\needspace{2\baselineskip}
\begin{lemma}\PRFR{Mar 29th}
	The derivation of an identified tableau is valid.
\end{lemma}
\begin{proof}\PRFR{Mar 29th}
	From Lemma~\ref{lem:digraphInducedGammoidIfTisGammoid} we obtain that $M_2' = I(D_1,M_1,E_2')$ is a gammoid if $M_1$ is a gammoid,
	and that $M_1' = I(D_2,M_2,E_1')$ is a gammoid if $M_2$ is a gammoid. Therefore $M_1$ is a gammoid if and only if $M_2$ is a gammoid.
	Consequently, $\simeq'$ satisfies the properties of Definition~\ref{def:validTableau}, and thus the identified tableau $\Tbf(M_1 \simeq M_2)$ is valid
	for every valid input tableau $\Tbf$.
\end{proof}

\subsection{Valid Tableaux}

\PRFR{Mar 29th}
\noindent In this section we present a variety of valid tableaux which may be used as inputs for valid derivation operations.
Trivially, if $M$ is a gammoid, then $(M,\SET{M},\emptyset,\emptyset,\langle\,\rangle)$ is a valid tableau,
and if $M$ is not a gammoid, then $(M,\emptyset,\emptyset,\SET{M},\langle\,\rangle)$ is a valid tableau;
 where $\langle .\rangle$ denotes the generated equivalence relation defined on the set of matroids occurring in the respective tableau. Thus exactly one of these two tableaux
is valid. Unfortunately, in order to know which one is valid, we have to decide whether $M$ is a gammoid first ---
in general this is not easier than determining $\Gamma_\Mcal(M)$, but there are special cases which we should not ignore.

\begin{corollary}\label{cor:strictGammoidTableau}\PRFR{Mar 29th}
	Let $M=(E,\Ical)$ be a matroid with $\alpha_M \geq 0$. Then the matroid tableau
	$\Tbf$ is valid, where
	$\Tbf = (M,\Gcal,\Mcal,\Xcal,\simeq)$ with
	\( \Gcal = \SET{M,M^\ast}\),  $\Mcal= \emptyset$, $\Xcal=\emptyset$, and $M\simeq N \Leftrightarrow M=N$.
\end{corollary}
\begin{proof}\PRFR{Mar 29th}
See Corollary~\ref{cor:MasonAlpha}.
\end{proof}

\begin{corollary}\PRFR{Mar 29th}
	Let $M=(E,\Ical)$ be a matroid with $\rk_M(X) = 3$, $X\subseteq E$ with $\alpha_M(X) < 0$. Then the matroid tableau
	$\Tbf$ is valid, where
	$\Tbf = (M,\Gcal,\Mcal,\Xcal,\simeq)$ with
	\( \Gcal = \emptyset\),  $\Mcal= \emptyset$, $\Xcal=\SET{M,M^\ast}$, and $M\simeq N \Leftrightarrow M=N$.
\end{corollary}
\begin{proof}\PRFR{Mar 29th}
See Corollary~\ref{cor:MasonAlpha} and Proposition~\ref{prop:cornerCases}.
\end{proof}

\begin{remark}\label{rem:nonStrictGammoidTableau}\PRFR{Mar 29th}
	Let $M=(E,\Ical)$ be a matroid, $X\subseteq E$ with $\alpha_M(X) < 0$. Then the matroid tableau
	$\Tbf$ is valid, where
	$\Tbf = (M,\Gcal,\Mcal,\Xcal,\simeq)$ with
	\( \Gcal = \emptyset\),  $\Mcal= \SET{M}$, $\Xcal=\emptyset$, and $M\simeq N \Leftrightarrow M=N$.
\end{remark}

\needspace{4\baselineskip}
\begin{theorem}[\cite{In77}, Theorem~13; \cite{Brylawski1971}, \cite{Brylawski1975}, \cite{Ingleton1971}]\label{thm:graphicGammoidsAreSeriesParallel}\PRFR{Mar 29th}
	Let $\Fbb_2$ be the two-elementary field, $E,C$ finite sets, and let $\mu \in \Fbb_2^{E\times C}$ be a matrix.
	Then $M(\mu)$ is a gammoid if and only if there is no minor $N$ of $M(\mu)$ which is isomorphic to $M(K_4)$.
	The latter is the case if and only if $M(\mu)$ is isomorphic to the polygon matroid of a series-parallel network.
\end{theorem}

\noindent For proofs of a sufficient set of implications which establish the equivalency stated, refer to
 \cite{Brylawski1971}, \cite{Brylawski1975}, and \cite{Ingleton1971}.

 \begin{theorem}[\cite{Ox11}, Theorem~6.5.4]\label{thm:binaryMatroids}\PRFR{Mar 29th}
 	Let $M=(E,\Ical)$ be a matroid. Then $M$ is isomorphic to $M(\mu)$ for some matrix $\mu \in \Fbb_2^{E\times C}$
 	if and only if $M$ has no minor isomorphic to the uniform matroid $U_{2,4} = \left( E',\Ical' \right)$,
 	where $E'= \dSET{a,b,c,d}$ and \linebreak $\Ical' = \SET{X\subseteq E'~\middle|~ \left| X \right| \leq 2}$.
 \end{theorem}
 \noindent See \cite{Ox11}, pp.193f, for a proof.

\begin{corollary}\PRFR{Mar 29th}
	Let $M=(E,\Ical)$ be a matroid.
	If $M$ has no minor isomorphic to $M(K_4)$ and no minor isomorphic to $U_{2,4}$, then the matroid tableau
	$\Tbf$ is valid, where
	$\Tbf = (M,\Gcal,\Mcal,\Xcal,\simeq)$ with
	\( \Gcal = \SET{M,M^\ast}\),  $\Mcal= \emptyset$, $\Xcal=\emptyset$,  and $M\simeq N \Leftrightarrow M=N$.
\end{corollary}
\begin{proof}\PRFR{Mar 29th}
Direct consequence of Theorems~\ref{thm:graphicGammoidsAreSeriesParallel} and \ref{thm:binaryMatroids}.
\end{proof}

\begin{definition}\PRFR{Mar 29th}
	Let $M=(E,\Ical)$ be a matroid. Then $M$ shall be \deftext{strongly base-orderable}, if for every pair of bases
	$B_1,B_2\in\Bcal(M)$ there is a bijection $\phi\colon B_1\maparrow B_2$ such that
	\[ \left( B_1\BS X  \right) \cup \phi[X] \in \Bcal(M) \]
	holds
	for all $X\subseteq B_1$. This property is also referred to as \deftext{full exchange property}.
\end{definition}

\begin{lemma}[\cite{M72}, Corollary~4.1.4]\label{lem:GammoidsAreStronglyBaseOrderable}\PRFR{Mar 29th}
	Let $M=(E,\Ical)$ be a gammoid. Then $M$ is strongly base-orderable.
\end{lemma}
\begin{proof}\PRFR{Mar 29th}
	Let $B_1,B_2\in \Bcal(M)$ be any two bases of $M$,
	and let $(D,B_1,E)$ be a representation of $M$ (Theorem~\ref{thm:gammoidRepresentationWithBaseTerminals}).
	Since $B_2\in \Ical$ and $\left| B_1 \right| = \left| B_2 \right|$, there is a linking $R\colon B_2\routesto B_1$ in $D$.
	Let $\phi\colon B_1\maparrow B_2$ be the unique bijection with the property that $p_{1} = \phi(p_{-1})$ for all $p\in R$.
	Let $X\subseteq B_1$, then the derived linking $$R_X = \SET{p\in R ~\middle|~ p_1 \in \phi[X]} \cup \SET{b\in B_1~\middle|~b\notin X}$$
	proves that
	$\left( B_1\BS X \right)\cup \phi[X] \in \Bcal(M)$. Thus $M$ is strongly base-orderable.
\end{proof}

\begin{corollary}\PRFR{Mar 29th}
	Let $M=(E,\Ical)$ be a matroid, $B_1,B_2\in \Bcal(M)$ be bases of $M$ such that for every bijection
	$\phi\colon B_1\BS B_2 \maparrow B_2\BS B_1$ there is a set $X\subseteq B_1\BS B_2$
	with the property $ \left( B_1\BS X \right)\cup \phi[X] \notin \Bcal(M)$. Then the matroid tableau
	$\Tbf$ is valid, where
	$\Tbf = (M,\Gcal,\Mcal,\Xcal,\simeq)$ with
	\( \Gcal = \emptyset\),  $\Mcal= \emptyset$, $\Xcal=\SET{M,M^\ast}$, and $M\simeq N \Leftrightarrow M=N$.
\end{corollary}
\begin{proof}\PRFR{Mar 29th}
	Direct consequence of the proof of Lemma~\ref{lem:GammoidsAreStronglyBaseOrderable}.
\end{proof}

\needspace{4\baselineskip}
\vspace*{-\baselineskip} %Remove the line space created by the tilde below
\begin{wrapfigure}{r}{4.5cm}
\vspace{\baselineskip}
\begin{centering}~~
\includegraphics{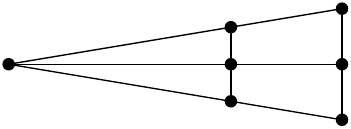}
\end{centering}%
\vspace*{-\baselineskip}
\end{wrapfigure}
~ %The tilde creates a new dummy paragraph. WHY IS THAT NEEDED? -> would increase the space %
  % before the ex. environment. THE NEXT FREE LINE IS ESSENTIAL!

\begin{example}\PRFR{Mar 29th}
	Consider the matroid $P_7$ (\cite{Ox11}, p.644), its affine configuration is depicted on the right.
	It is strongly base-orderable
	but it is not a strict gammoid. 
	Since $P_7$ has rank $3$, it follows that $P_7$ is a strongly base-orderable non-gammoid (Proposition~\ref{prop:cornerCases}).
\end{example}

\begin{example}\PRFR{Mar 29th}
	The Vámos matroid (\cite{Ox11}, p.649) is strongly base-orderable, but not representable over the reals $\R$.
\end{example}

\begin{theorem}[\cite{Ingleton69}, \cite{MNW09}]\label{thm:IngeltonsCondition}\PRFR{Mar 29th}
	Let $M=(E,\Ical)$ be a matroid such that there is a field $\Fbb$ and a matrix $\mu\in\Fbb^{E\times C}$
	with $M= M(\mu)$. Let further $W,X,Y,Z\subseteq E$. Then
	\begin{align*}
		\rk(W) \,+\, &\rk(X) + \rk(W\cup X\cup Y) + \rk(W\cup X\cup Z) + \rk(Y\cup Z)  \\
		& \leq \rk(W\cup X) + \rk(W\cup Y) + \rk(W\cup Z) + \rk(X\cup Y) + \rk(X\cup Z).
	\end{align*}
\end{theorem}

\PRFR{Mar 31st}
\noindent
For a proof, see \cite{Ingleton69}. We mention A.W.~Ingleton's theorem here
because it has been used by D.~Mayhew in \cite{Ma16} in order to prove
 that certain matroids are excluded minors of the class of gammoids.
P.~Nelson and J.~van~der~Pol \cite{NvdP17} showed that A.W.~Ingleton's necessary condition for representability over any field
is rather weak: It is quite improbable for matroids on large ground sets that
a matroid which satisfies this condition is indeed representable over any field, because there are double-exponentially many
matroids satisfying the condition with respect to the cardinality of the ground set, 
yet there are only exponentially many representable
matroids with respect to the cardinality of the ground set. Furthermore, L.~Guillé, T.~Chan, and A.~Grant found a unique
minimal subset of $\frac{6^n}{4} - O(5^n)$ inequalities that imply A.W.~Ingleton's condition if satisfied \cite{GCG11}.

\needspace{5\baselineskip}
\begin{corollary}\PRFR{Mar 29th}\label{cor:ingletonTab}
	Let $M=(E,\Ical)$ be a matroid, and let $W,X,Y,Z\subseteq E$
	such that
	\begin{align*}
		\rk(W) \,+\, &\rk(X) + \rk(W\cup X\cup Y) + \rk(W\cup X\cup Z) + \rk(Y\cup Z)  \\
		& > \rk(W\cup X) + \rk(W\cup Y) + \rk(W\cup Z) + \rk(X\cup Y) + \rk(X\cup Z).
	\end{align*}
	Then the matroid tableau
	$\Tbf$ is valid, where
	$\Tbf = (M,\Gcal,\Mcal,\Xcal,\simeq)$ with
	\( \Gcal = \emptyset\),  $\Mcal= \emptyset$, $\Xcal=\SET{M,M^\ast}$, and $M\simeq N \Leftrightarrow M=N$.
\end{corollary}
\begin{proof}\PRFR{Mar 29th}
	Consequence of Theorems~\ref{thm:IngeltonsCondition} and \ref{thm:gammoidOverR}.
\end{proof}

\PRFR{Mar 31st}
\noindent This strict inequality dualizes to
\begin{align*}
	\nu(W') \,+\, & \nu(X') + \nu\left( W'\cap X'\cap Y' \right) + 
	\nu\left(  W'\cap Y'\cap Z' \right) + \nu\left( Y'\cap Z' \right)
	\\
	 < \,& \hphantom{+\,} \nu\left(  W'\cap X' \right)+ \nu\left( W'\cap Y' \right)+ 
	\nu\left( W'\cap Z' \right) + \nu\left( X'\cap Y' \right)+ \nu\left(  X'\cap Z'\right)
\end{align*}
where $\nu(X) = \left| X \right| - \rk(X)$. A.~Cameron showed that $M$ satisfies A.W.~Ingleton's condition
if and only if $M^\ast$ satisfies it (\cite{Ca14Msc}, Lemma~4.5, p.26), therefore the dualized inequality
does not provide
any valid matroid tableaux that cannot be derived using Corollary~\ref{cor:ingletonTab}.

\subsection{Derivation of a Decisive Tableau}

\PRFR{Mar 29th}
\noindent First of all, it is clear that we may derive a decisive matroid tableau for any given matroid $G=(E,\Ical)$
by simply determining all extensions of $G$ with $\rk_G(E)^2\cdot \left| E \right| + \rk_G(E) + \left| E \right|$ 
elements (Remark~\ref{rem:upperBoundForV}). For each such extension $N$, there is a valid tableau, which depends on whether $N$
is a strict gammoid (Corollary~\ref{cor:strictGammoidTableau}) or not (Remark~\ref{rem:nonStrictGammoidTableau}).
Thus we may derive the joint tableau of all valid tableaux of the extensions of $G$. It is clear from Definition~\ref{def:decisiveTableau} 
that this joint tableau is decisive: either case {\em (i)} or case {\em (iii)} holds. Thus we may always decide $\Gamma_\Mcal(G)$ using the
matroid tableau method. Unfortunately, we cannot guarantee that there is no excluded minor $X$ for the class of gammoids, where the only
feasible way to refute, that $X$ is a gammoid, requires to employ the tiresome case {\em (iii)}. Now, let us provide a glimpse of the art
of employing matroid tableaux.

\begin{example}\label{ex:tab}\PRFR{Mar 29th}
	Consider the matroid $G = G_{8,4,1} = (E,\Ical)$ where $E=\SET{1,2,\ldots,8}$ and where $\Ical = \SET{\vphantom{A^A}X\subseteq E ~\middle|~ \left| X \right|\leq 4,\,X\notin \Hcal}$ with
	\[\Hcal = \SET{\vphantom{A^A} \SET{1, 3, 7, 8},
  \SET{1, 5, 6, 8},
   \SET{2, 3, 6, 8},
  \SET{4, 5, 6, 7},
  \SET{2, 4, 7, 8}}.\]
  Clearly, $\alpha_G(H) = 1$ for all $H\in \Hcal$, and consequently $\alpha_G(E) = 4-5 = -1$.
  The dual matroid $G^\ast = (E,\Ical^\ast)$ has a similar structure:
  $\Ical^\ast = \SET{\vphantom{A^A}X\subseteq E ~\middle|~ \left| X \right|\leq 4,\,X\notin \Hcal^\ast}$ with
	\[\Hcal^\ast = \SET{\vphantom{A^A} \SET{1, 2, 3, 8},
  \SET{1, 3, 5, 6},
   \SET{1, 4, 5, 7},
  \SET{2, 3, 4, 7},
  \SET{2, 4, 5, 6}}.\]
  Thus $\alpha_{G^\ast}(H') = 1$ for all $H'\in\Hcal^\ast$, and so $\alpha_{G^\ast}(E) = 4-5 = -1$, too.
  It turns out that neither $G$ nor $G^\ast$ have any minors of rank $3$ which are not strict gammoids.
  Furthermore, both $G$ and $G^\ast$ are strongly base-orderable, and both $G$ and $G^\ast$ have a $U_{2,4}$ minor.
  For the rest of this example, we will refer to `single-element extensions of the same rank' simply by the word `extension'.
   There are
  $11962$ different isomorphism classes of extensions of $G$, and $11495$ different isomorphism classes of
  extensions of $G^\ast$. No extension of $G$ or $G^\ast$ is a strict gammoid or a transversal matroid.
  $8643$ isomorphism classes of $G$-extensions either have non-gammoid rank-$3$ minors, or they are not strongly base-orderable, the same holds for
  $7892$ isomorphism classes of $G^\ast$-extensions. This leaves $3319$ classes of $G$-extensions and $3603$ classes of $G^\ast$-extensions
  which may or may not be classes of gammoids --- so extending and backtracking may not be our best approach here.

  \noindent We have seen before that there is no easy way to decide whether $G$ or $G^\ast$ is a gammoid, therefore we start with the
  valid tableaux $$\Tbf_G = \left( G,\emptyset,\SET{G},\emptyset,\langle \, \rangle \right)
  \txtand \Tbf_{G^\ast} = \left( G^\ast,\emptyset,\SET{G^\ast},\emptyset,\langle \, \rangle \right),$$
  where $\langle .\rangle$ denotes the generated equivalence relation defined on the set of matroids occurring in the respective tableau.
  We may derive the extended joint tableau 
  $$ \Tbf_1 = [\Tbf_G \cup \Tbf_{G^\ast}]_\equiv = \left( G, \emptyset,\SET{G,G^\ast}, \langle G\simeq G^\ast\rangle  \right).$$
  Now observe that although $G$ is deflated, $G^\ast$ is not deflated. We have
  	 \begin{align*}C^\ast_8 & = \SET{\vphantom{A^A}F\in \Fcal\left( G^\ast\restrict \SET{1,2,\ldots,7}  \right)~\middle|~ 8 \in \cl_{G^\ast}(F)}
  	 \\ & = \SET{\vphantom{A^A}F\in \Fcal\left( G^\ast\restrict \SET{1,2,\ldots,7}  \right)~\middle|~\SET{1,2,3}\subseteq F}.
  \end{align*}
  Let $G^\ast_7 = G^\ast\restrict \SET{1,2,\ldots,7}$. We have $\alpha_{G^\ast_7}(\SET{1,2,\ldots,7}) = -1$, thus $G^\ast_7$ is not a 
  strict gammoid, and thus
  \[ \Tbf_{G^\ast_7} = \left(G^\ast_7, \emptyset, \SET{G^\ast_7}, \emptyset, \langle\,\rangle\right) \]
  is a valid tableau. Since $G^\ast_7$ is a deflate of $G^\ast$, each of them is an induced matroid with respect to the other. Therefore
  we may identify $G^\ast$ and $G^\ast_7$ in the joint tableau
  \[ \Tbf_2 = \left( \Tbf_1\cup \Tbf_{G^\ast_7} \right)(G^\ast\simeq G^\ast_7) = \left( G,\emptyset,\SET{G,G^\ast,G^\ast_7},\emptyset,\langle G \simeq G^\ast \simeq G^\ast_7\rangle \right).\]
  Now let $G_7 = \left( G^\ast_7 \right)^\ast$, and we have $\alpha_{G_7} \geq 0$. Thus
  \[ \Tbf_{G_7} = \left( G_7, \SET{G_7},\emptyset, \emptyset, \langle\,\rangle \right) \]
  is a valid tableau. We now may derive the decisive tableau
  \[ \Tbf_3 = [\Tbf_2 \cup \Tbf_{G_7}]_\equiv = \left( G, \SET{G_7}, \SET{G,G^\ast,G^\ast_7,G_7},\emptyset, \langle G\simeq G^\ast \simeq G^\ast_7 \simeq G_7\rangle \right) \]
  where case {\em (i)} of Definition~\ref{def:decisiveTableau} holds. Consequently, $G$ is a gammoid.
\end{example}

\begin{figure}
\includegraphics[width=\textwidth]{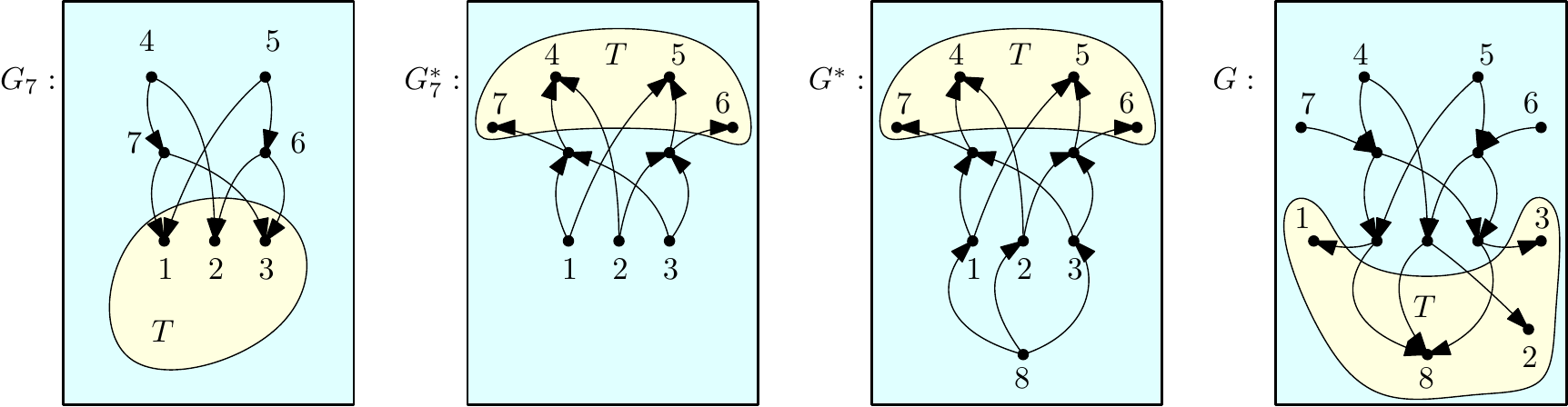}
\caption{\label{fig:tab}Reconstruction of a representation of $G_{8,4,1}$ from the matroid tableaux in Example~\ref{ex:tab}.}
\end{figure}

\PRFR{Mar 31st}
\noindent The representation of $G_{8,4,1}$ given in Figure~\ref{fig:tab} can obviously be reduced by two vertices 
if we move both the vertices $6$ and $7$ one step along their only incident arcs and delete the now superfluous sources. 
So $G_{8,4,1}$ may be represented with $11$ vertices,
and it is still possible that there is a representation of $G_{8,4,1}$ with $10$ vertices. Clearly, $9$ vertices do not suffice since no
single-element extension of $G_{8,4,1}$ is a strict gammoid.

\bigskip 
\PRFR{Mar 31st}
\noindent
Based on our experience, let us provide our best procedure for determining whether a given matroid $G=(E,\Ical)$ is a gammoid.
We start the procedure with the valid initial tableau $\Tbf := (G,\emptyset,\emptyset,\emptyset,\langle\,\rangle)$.

\begin{step}\label{step:start}\PRFR{Mar 31st}
	If $\Tbf$ is decisive, stop.
\end{step}

\begin{step}\PRFR{Mar 31st}
	Choose an intermediate goal
	$M\in \left( \SET{G' ~\middle|~ G'\text{~is a minor of~}G} \cup \Mcal  \right)\BS \left( \Gcal \cup \Xcal \right)$,
	preferably one with $M\simeq G$ which is small both in rank and cardinality.
\end{step}

\begin{step}\PRFR{Mar 31st}
	If $\Tbf_M = (M,\emptyset,\emptyset,\emptyset,\langle\,\rangle)\cup \Tbf$ is decisive,
	then set $\Tbf := \left[\vphantom{A^A}[\Tbf \cup \left( \Tbf_M!  \right)]_\equiv\right]_\simeq$
	 and continue with Step~\ref{step:start}.
\end{step}

\begin{step}\label{step:choice}
	Determine whether $M$ has a minor that is isomorphic to $M(K_4)$. If this is the case,  then 
	$\Tbf_M = \left( M, \emptyset, \emptyset, \SET{M,M^\ast}, \langle\,\rangle  \right)$ is valid, we set
	 $\Tbf := \left[\vphantom{A^A}[\Tbf\cup \Tbf_M]_\equiv \right]_\simeq$ and then
	continue with Step~\ref{step:start}.
\end{step}

\noindent Since $M(K_4) = \left( M(K_4) \right)^\ast$, we have that $M(K_4)$ is neither a minor of $M$ nor of $M^\ast$
when reaching the next step.

\begin{step}\label{step:chosen}\PRFR{Mar 31st}
	Determine whether $M$ has a minor that is isomorphic to $U_{2,4}$. If this is not the case,  then 
	$\Tbf_M = \left( M, \SET{M,M^\ast}, \emptyset, \emptyset,  \langle\,\rangle  \right)$ is valid, we set
	 $\Tbf := \left[\vphantom{A^A}[\Tbf\cup \Tbf_M]_\equiv \right]_\simeq$ and then
	continue with Step~\ref{step:start}.
\end{step}

\noindent Since $U_{2,4} = \left( U_{2,4} \right)^\ast$, we have that $U_{2,4}$ is neither a minor of $M$ nor of $M^\ast$
when reaching the next step.

\begin{step}\PRFR{Mar 31st}
	If $M\in \Mcal$, continue immediately with Step~\ref{step:nonStrict}.
	Determine whether $\alpha_M \geq 0$. If this is the case, then 
$\Tbf_M = \left( M, \SET{M,M^\ast}, \emptyset, \emptyset,  \langle\,\rangle  \right)$ is valid, we set
	 $\Tbf := \left[\vphantom{A^A}[\Tbf\cup \Tbf_M]_\equiv \right]_\simeq$ and
	continue with Step~\ref{step:start}.
\end{step}

\begin{step}\label{step:nonStrict}\PRFR{Mar 31st}
	If $M^\ast\in \Mcal$, continue immediately with Step~\ref{step:baseorderable}.
	Determine whether $\alpha_{M^\ast} \geq 0$. If this is the case, then
$\Tbf_{M^\ast} = \left( M^\ast, \SET{M,M^\ast}, \emptyset, \emptyset,  \langle\,\rangle  \right)$ is valid, we set
	 $\Tbf := \left[\vphantom{A^A}[\Tbf\cup \Tbf_{M^\ast}]_\equiv \right]_\simeq$ and
	continue with Step~\ref{step:start}.
\end{step}

\begin{step}\label{step:baseorderable}\PRFR{Mar 31st}
	Determine whether $M$ is strongly base-orderable. If this is not the case, then 
	$\Tbf_M = \left( M, \emptyset, \emptyset, \SET{M,M^\ast}, \langle\,\rangle  \right)$ is valid, we set
	 $\Tbf := \left[\vphantom{A^A}[\Tbf\cup \Tbf_M]_\equiv \right]_\simeq$ and then
	continue with Step~\ref{step:start}.
\end{step}

\noindent The class of strong base-orderable matroids is closed under duality and minors \cite{Ingleton1971}, therefore $M^\ast$ and
all minors of $M$ and $M^\ast$ are strongly base-orderable upon reaching the next step.

\begin{step}\PRFR{Mar 31st}
	Let $M=(E,\Ical)$. Determine whether there is some $X\in \Ical$ with $\left| X \right| = \rk_M(E) - 3$
	and some $Y\subseteq E\BS X$ such that $\alpha_{M\contract\left( E\BS X \right)}(Y) < 0$. If this is the
	case, then the tableau 
	$\Tbf_M = \left( M, \emptyset, \emptyset, \SET{M,M^\ast}, \langle\,\rangle  \right)$ is valid, we set
	 $\Tbf := \left[\vphantom{A^A}[\Tbf\cup \Tbf_M]_\equiv \right]_\simeq$ and then
	continue with Step~\ref{step:start}.
\end{step}

\begin{step}\PRFR{Mar 31st}
	Let $M^\ast=(E,\Ical^\ast)$. Determine whether there is some $X\in \Ical^\ast$ with \linebreak 
	 $\left| X \right| = \rk_{M^\ast}(E) - 3$
	and some $Y\subseteq E\BS X$ such that $\alpha_{M^\ast\contract\left( E\BS X \right)}(Y) < 0$. If this is the
	case, then the tableau 
	$\Tbf_{M^\ast} = \left( M^\ast, \emptyset, \emptyset, \SET{M,M^\ast}, \langle\,\rangle  \right)$ is valid, we set
	 $\Tbf := \left[\vphantom{A^A}[\Tbf\cup \Tbf_{M^\ast}]_\equiv \right]_\simeq$ and then
	continue with Step~\ref{step:start}.
\end{step}

\PRFR{Mar 31st}
\noindent The next step may be omitted or carried out sloppily\footnote{It clearly would be sloppy to just consider $W$, $X$, $Y$, and $Z$
with $\max\SET{\left| W \right|,\left| X \right|,\left| Y \right|,\left| Z \right|} \leq k$ for some $k\in \Z\BSET{0,1}$, or even
to just check whether $M$ has a Vámos-matroid as a minor.},
 because it may take a considerable amount of time for larger matroids and it does
not seem to be worth the computational effort in practice. 

\begin{step} \PRFR{Mar 31st}
	Let $M=(E,\Ical)$. Determine whether there are $W,X,Y,Z\in \Ical$ such that
	\begin{align*}
		\rk(W) + &\rk(X) + \rk(W\cup X\cup Y) + \rk(W\cup X\cup Z) + \rk(Y\cup Z)  \\
		& > \rk(W\cup X) + \rk(W\cup Y) + \rk(W\cup Z) + \rk(X\cup Y) + \rk(X\cup Z).
	\end{align*}
	If this is the
	case, then the tableau 
	$\Tbf_{M^\ast} = \left( M^\ast, \emptyset, \emptyset, \SET{M,M^\ast}, \langle\,\rangle  \right)$ is valid, we set
	 $\Tbf := \left[\vphantom{A^A}[\Tbf\cup \Tbf_{M^\ast}]_\equiv \right]_\simeq$ and then
	continue with Step~\ref{step:start}.
\end{step}

\begin{step}\PRFR{Mar 31st}
	Determine whether $M$ is deflated. If not, then find a deflate $N$ of $M$ with a ground set of minimal cardinality,
	set $\Tbf := \left[\vphantom{A^A}\left[\left( \Tbf\cup\Tbf_N \right)(M \simeq N)\right]_\equiv \right]_\simeq$
	where $$\Tbf_N = \begin{cases}[r]
		 \left( N,\SET{N,N^\ast},\emptyset,\emptyset,\langle\,\rangle \right)& \quad \text{if~}\alpha_N \geq 0,\\
		  \left( N,\emptyset,\SET{N},\emptyset,\langle\,\rangle \right) & \quad \text{otherwise,}
		  \end{cases}$$ 
		and continue with Step~\ref{step:start}.
\end{step}

\begin{step}\PRFR{Mar 31st}\label{lastStep}
	Determine whether $M^\ast$ is deflated. If not, then find a deflate $N$ of $M^\ast$ with a ground set of minimal cardinality,
	set $\Tbf := \left[\vphantom{A^A}\left[\left( \Tbf\cup\Tbf_N \right)(M^\ast \simeq N)\right]_\equiv \right]_\simeq$
	where $$\Tbf_N = \begin{cases}[r]
		 \left( N,\SET{N,N^\ast},\emptyset,\emptyset,\langle\,\rangle \right)& \quad \text{if~}\alpha_N \geq 0,\\
		  \left( N,\emptyset,\SET{N},\emptyset,\langle\,\rangle \right) & \quad \text{otherwise,}
		  \end{cases}$$ 
		and continue with Step~\ref{step:start}.
\end{step}

\PRFR{Mar 31st}
\noindent This is the point where we may try creative ways of determining whether $M$ is a gammoid or not.
If we are successful, then we augment the tableau $\Tbf$ accordingly, and continue with Step~\ref{step:start}.
In the best case, we might guess a representation of $M$, or find a considerably smaller matroid $M'$ such that $M$ is induced from $M'$
by some digraph $D$. In theory, it is also possible to find a known non-gammoid $X'$ that is induced from $M$ by some digraph $D$, which
then implies that $M$ must be a non-gammoid.  Since the class of strongly base-orderable matroids is closed under matroid 
induction by digraphs,
$X'\notin \SET{M(K_4), P_7}$, because we know since Step~\ref{step:baseorderable} that $M$ is strongly base-orderable.
In practice, we never managed to successfully show that some known excluded minor may be induced from a candidate matroid $M$ 
under examination.
Currently, $P_8^=$ (\cite{Ox11}, p.651) is the only 
excluded minor for the class of gammoids (J.~Bonin, \cite{joeP8}) that we know of, which is strongly base-orderable 
yet neither has rank or co-rank $3$. Furthermore,
$P_8^=$ is isomorphic to its dual $\left(  P_8^=\right)^\ast$ which makes it a rather special matroid. 
Therefore we think it is reasonable to assume that the odds are clearly in favor of $M$ being a gammoid upon reaching the 
next step.\footnote{
Or, more pessimistically, we might not know sufficiently general excluded minors for
the class of gammoids to assess the situation here more realistically.}

\needspace{4\baselineskip}
\begin{step}\label{step:exhaustion}\PRFR{Mar 31st}
	Try to find an extension $N$ of $M$ with at most $\rk_G(E)^2\cdot \left| E \right| + \rk_G(E) + \left| E \right|$ elements
	such that $N$ is not isomorphic to any $M'\in \Gcal \cup \Mcal \cup \Xcal$.
	Set $\Tbf := \left[\vphantom{A^A}\left[\Tbf\cup\Tbf_N\right]_\equiv \right]_\simeq$
	where $$\Tbf_N = \begin{cases}[r]
		 \left( N,\SET{N,N^\ast},\emptyset,\emptyset,\langle\,\rangle \right)& \quad \text{if~}\alpha_N \geq 0,\\
		  \left( N,\emptyset,\SET{N},\emptyset,\langle\,\rangle \right) & \quad \text{otherwise,}
		  \end{cases}$$ 
		and continue with Step~\ref{step:start}.
	If no such extension of $M$ exists, then set $M:= G$ and continue with Step~\ref{step:chosen}.
\end{step}

\PRFR{Mar 31st}
\noindent Clearly, if we continue this process long enough, then Step~\ref{step:exhaustion} ensures 
that the tableau $\Tbf$ will eventually
become decisive for $G$ by exhausting all isomorphism classes of extensions of $G$
with at most $\rk_G(E)^2\cdot \left| E \right| + \rk_G(E) + \left| E \right|$ elements.

\clearpage
% -*- root: ../thesis.tex -*-

\section{Representation over $\R$}

\noindent 
\PRFR{Feb 15th}
There are many ways to arrive at the fact that every gammoid can be represented by a matrix
over a field $\Kbm$ whenever $\Kbm$ has enough elements. Or, to be more precise, for every field $\Fbm$ and 
every gammoid $M$ there
is an extension field $\Kbm$ of $\Fbm$, such that $M$ can be represented by a matrix over $\Kbm$.
For the sake of simplicity, we only consider representations of gammoids over the field of the reals $\R$.
In \cite{Ar06}, F.~Ardila points out that the Lindström Lemma yields an easy method to construct a matrix $\mu\in \R ^{E\times B}$ from the digraph $D=(V,A)$ such that $\Gamma(D,T,E) = M(\mu)$;
the construction is universal in the sense that it works with indeterminates and thus yields a representation over $\Fbm$ whenever these indeterminates can be replaced with elements from $\Fbm$
without zeroing out any nonzero subdeterminants of $\mu$.

\begin{definition}\PRFR{Feb 15th}
	Let $D=(V,A)$ be a digraph and $w\colon A\maparrow \R$.
	Then $w$ shall be called
	 \deftext[indeterminate weighting of D@indeterminate weighting of $D$]{indeterminate weighting of $\bm D$},
	whenever the set $\SET{w(a)\mid a\in A}$ is $\Z$-independent.
\end{definition}

\begin{example}\PRFR{Feb 15th}
	Let $D=(V,A)$ be any digraph, then $\left| A \right| < \infty$. Thus there is a set $X\subseteq \R$
	that is $\Z$-independent with $\left| X \right| = \left| A \right|$ (Lemma~\ref{lem:enoughZindependents}). Then
	any bijection $\sigma\colon A\maparrow X$ induces an indeterminate weighting $w\colon X\maparrow \R$
	with $w(x) = \sigma(x)$, thus indeterminate weightings exist for all digraphs.
\end{example}

\begin{notation}\label{n:prodp}\PRFR{Feb 15th}
	Let $D=(V,A)$ be a digraph and $w\colon A\maparrow \R$ be an indeterminate weighting of $D$. Let $q=(q_i)_{i=1}^n\in \Wbf(D)$,
	we shall write \[ \prod q = \prod_{i=1}^{n-1} w\left( \vphantom{A^A}(q_i,q_{i+1}) \right). \qedhere \]
\end{notation}

%\needspace{8\baselineskip}
 \begin{lemma}[Lindström \cite{Li73}]\label{lem:lindstrom}\PRFR{Feb 15th}
 	Let $D=(V,A)$ be an acyclic digraph, $n\in \N$ a natural number, $S=\dSET{s_1,s_2,\ldots,s_n}\subseteq V$ and $T=\dSET{t_1,t_2,\ldots,t_n}\subseteq V$ be equicardinal subsets of $V$, and let $w\colon A\maparrow \R$ be an indeterminate weighting of $D$. Furthermore,
 	$\mu\in \R^{V\times V}$ shall be the matrix with
 	\[ \mu(u,v) = \sum_{p\in \Pbf(D;u,v)} \prod p .\]
 	%where $P_{(u,v)} = \SET{p\in \Pbf(D)\mid  p_1 = u \txtand p_{-1}=v}$.
 	Then
 	\[ \det \left( \mu\restrict S\times T \right) = \sum_{L\colon S\routesto T} \left( \sgn(L)  
 	\prod_{p\in L} \left( \prod p \right) \right) \]
 	where $\sgn(L) = \sgn(\sigma)$ for the unique permutation $\sigma\in \mathfrak{S}_n$ with
 	the property that for every $i\in\SET{1,2,\ldots,n}$ there is a path $p\in L$ with $p_1 = s_i$ and 
 	$p_{-1} = t_{\sigma(i)}$.
 	Furthermore, \[  \det \left( \mu\restrict S\times T \right) = 0 \] if and only if
 	there is no linking from $S$ to $T$ in $D$.
 \end{lemma}
%\remred{TODO: Ist doch nicht so einfach, wie der Ardila es schreibt... Man braucht irgendwie, dass $D$ ein acyclic dirgaph ist; woher kommt so eine darstellung..}
\noindent As suggested by F.~Ardila, we present the following bijective proof given by I.M.~Gessel and X.G.~Viennot \cite{GV89p}. 
\begin{proof}\PRFR{Feb 15th} The Leibniz formula (Definition~\ref{def:det}) yields
\begin{align*}
\det \left( \mu\restrict S\times T \right) 
& = \sum_{\sigma\in \mathfrak{S}_n} \sgn(\sigma) \prod_{i=1}^{n} \mu(s_i,t_{\sigma(i)}) \\
& = \sum_{\sigma\in \mathfrak{S}_n} \sgn(\sigma) \prod_{i=1}^{n} \left( 
				\sum_{p\in \Pbf{\left(D;s_i,t_{\sigma(i)}\right)}} \prod p \right) \\
& = \sum_{\sigma\in \mathfrak{S}_n} \sgn(\sigma)\left( \sum_{K\in Q_\sigma} \,\,
			\prod_{p\in K} \left( \prod p\right) \right),
\end{align*}
where $$Q_\sigma = \SET{\left. K \in \binom{\Pbf(D)}{n} \,\,\right|\,\, \forall i\in \SET{1,2,\ldots,n}\colon\,\exists p\in K\colon\,p_1=s_i \txtand p_{-1}=t_{\sigma(i)} }\,$$ consists of all families of paths connecting $s_i$ with $t_{\sigma(i)}$ for all $i\in \SET{1,2,\ldots,n}$.
%,  and
 %$ P_{\left(s_i,t_{\sigma(i)}\right)}$ is defined as above, consisting of all paths from
 %$s_i$ to $t_{\sigma(i)}$ in $D$.
 Clearly, for $\sigma,\tau\in \mathfrak{S}_n$ with $\sigma\not= \tau$, the sets $Q_\sigma\cap Q_\tau = \emptyset$ are disjoint, therefore the following map with the domain $Q=\bigcup_{\sigma\in \mathfrak{S}_n} Q_\sigma$ is well defined:
 \[ \sgn\colon Q \maparrow \SET{-1,1},\quad K\mapsto \sgn(\sigma) 
 \quad\text{~where~}\sigma\in\mathfrak{S}_n\text{~such that~}K\in Q_\sigma .\]
 %is well-defined. 
 Thus we may write
 \begin{align*}
\det \left( \mu\restrict S\times T \right) & = \sum_{K\in Q} \sgn(K) \prod_{p\in K} \left( \prod p \right).
\end{align*}
Furthermore, if $L\colon S\routesto T$ is a linking from $S$ to $T$ in $D$, then $L\in Q_\sigma$ 
where $\sigma\in\mathfrak{S}_n$ is the unique permutation mapping the indexes of the initial vertices of the paths in $L$ 
to the indexes of the terminal vertices of the paths in $L$. Let us denote the routings in $Q$ by
\[ R = \SET{L\in Q\mid L\text{~is a routing}}.\]
We prove the first statement of the lemma by showing that there is a bijection
$\phi\colon Q\BS R \maparrow Q\BS R$, such that for all $K\in Q\BS R$,
$$\prod_{p\in K}\left( \prod p \right) = \prod_{p\in \phi(K)} \left( \prod p \right)$$
and $\sgn(K) = -\sgn(\phi(K))$. We construct such a map $\phi$ now.
Let $$K' = \SET{p\in K~\middle|~\vphantom{A^A} \exists q\in K\BSET{p}\colon\,\left| p \right|\cap \left| q \right|\not=\emptyset}$$
be the set of paths in $K$ that meet a vertex of another path, clearly $\left| K' \right| \geq 2$
since $K$ is not a routing. There is a total order on $K'$:
 let $p,q\in K'$, then $p\leq q$ if and only if $i\leq j$ where $p_1 = s_i$ and $q_1 = s_j$.
Now let $p=(p_i)_{i=1}^{n(p)}\in K'$ be chosen 
such that $p$ is the minimal element with respect to the above order.
Let $j(p)\in \SET{1,2,\ldots,n(p)}$ be the smallest index, such that there is some $q\in K'\BSET{p}$
with $p_{j(p)}\in \left| q \right|$. Now let
 $q=(q_i)_{i=1}^{n(q)}\in\SET{k\in K'\BSET{p}~\middle|~\vphantom{A^A} p_{j(p)}\in \left| q \right|}$
  be the minimal choice with respect to the above order on $K'$, and let $j(q)\in \SET{1,2,\ldots,n(q)}$
  such that $q_{j(q)} = p_{j(p)}$. Now let
  $p' = p_1 p_2\ldots p_{j(p)} q_{j(q) + 1} q_{j(q) + 2} \ldots q_{n(q)}$
  and $q' = q_1 q_2 \ldots q_{j(q)} p_{j(p)+1} p_{j(p)+2}\ldots p_{n(p)}$. Since $D$ is acyclic, all walks are paths in $D$, so $\Wbf(D) = \Pbf(D)$.
   Therefore we may set
  $\phi(K) = \left( K\BSET{p,q}\right)\cup\SET{p',q'} \in Q\BS R$.
  Clearly, $\phi(\phi(K)) = K$, therefore $\phi$ is bijective and self-inverse. Furthermore,
  if $K\in Q_\sigma$, then $\phi(K) \in Q_{\sigma \cdot (x y)}$ for a suitable cycle
   $(x y)\in \mathfrak{S}_n$. Thus $\sgn(\phi(K)) = \sgn(\sigma)\sgn\left( (x y) \right) = - \sgn(\sigma) = -\sgn(K)$. Clearly, $K$ and $\phi(K)$ traverse the same arcs, therefore $\prod_{p\in K} (\prod p) = \prod_{p\in \phi(K)} (\prod p)$.
The bijection $\phi$ implies that the summands $K\in Q\BS R$ add up to zero, thus we have $$\det \left( \mu\restrict S\times T \right)  = \sum_{L\in R} \sgn(L) \prod_{p\in L} \left( \prod p \right).$$
The second statement of the lemma follows from the fact that for two routings $L_1,L_2\in R$,
we have $L_1 = L_2$ if and only if $\bigcup_{p\in L_1} \left| p \right|_A = \bigcup_{p\in L_2} \left| p \right|_A$. For the non-trivial direction: assume we have a set of arcs $L_A$ that are traversed by the paths of a linking, and let $V_A=\SET{u,v~\middle|~(u,v)\in L_A}$.
 Then the initial vertices of that linking are the elements of the set
$S_A = \SET{u\in V_A\mid \forall (v,w)\in L_A\colon\,u\not= w}$. The terminal vertices are the elements of the set
$T_A = \SET{w\in V_A\mid \forall (u,v)\in L_A\colon\,u\not= w}$, and the paths can be reconstructed from
the initial vertices $v\in S_A$ by following the unique arcs $(v,w),(w,x),\ldots \in L_A$ until a vertex $t\in T_A$ is reached. Clearly, for $L\in R$, $\prod_{p\in L} \left( \prod p \right) \not= 0$, and since $w$ is an indeterminate weighting, two summands $L,L'\in R$ can only cancel each other when the corresponding monomials are equal, i.e. $\prod_{p\in L} \left( \prod p \right) = \prod_{p\in L'} \left( \prod p \right)$; but then $L_A = L'_A$ holds, and so $L = L'$. Thus no summand in the determinant formula
 which belongs to a routing from $R$
can be cancelled out by another summand belonging to another routing from $R$. Therefore
$\det\left( \mu\restrict S\times T \right) = 0$ if and only if $R=\emptyset$, i.e. there is no linking from $S$ to $T$ in $D$.
\end{proof}

\begin{corollary}\PRFR{Feb 15th}
	Let $D=(V,A)$ be an acyclic digraph, $T,E\subseteq V$, and $w\colon A\maparrow \R$ be an indeterminate weighting of $D$. Furthermore,
	let
	 	$\mu\in \R^{E\times T}$ be the matrix with
 	\[ \mu(e,t) = \sum_{p\in \Pbf{(D;e,t)}} \left( \prod p \right).\]
 	%where $P_{(e,t)} = \SET{p\in \Pbf(D)\mid  p_1 = e \txtand p_{-1}=t}$.
 	Then $\Gamma(D,T,E) = M(\mu)$.
\end{corollary}
\begin{proof}\PRFR{Feb 15th}
	This is straightforward from the Definition~\ref{def:Mmu} and the Lindström Lemma~\ref{lem:lindstrom}.
\end{proof}

\noindent
\PRFR{Feb 15th}
Clearly, for an arbitrary gammoid $M = \Gamma(D,T,E)$, we cannot assume that $D$ is acyclic (Remark~\ref{rem:weNeedCycles}). 
There are several ways to work around this. Either {\em a)}\footnote{This is what is implied by the rationale given in \cite{Ar06}.} 
we adjust our definition of a routing such that routings with non-path walks are allowed, making the class of routings in $D$ 
infinite whenever there is a cycle in $D$. Then we could use power series to calculate the entries of $\mu$ as well as its sub-determinants, 
where convergence is sufficiently guaranteed if $\prod p \in (0,1)$ holds for every cycle walk $p\in \Wbf(D)$. 
A sufficient condition would be to use a weighting $w$ where $0 < w(a) < 1$ for all $a\in A$.
 The construction of $\phi$ in the proof of the Lindström Lemma would still go through, but for the second 
 statement we would have to choose the indeterminate weights more carefully, since a cycle walk $q\in \Wbf(D)$  
 gives rise to the formal power
series $\sum_{i=0}^\infty \left( \prod q \right)^{i}$ which converges to $\frac{1}{1-\prod q}$. 
Clearly, a similar cardinality-argument as in Lemma~\ref{lem:enoughZindependents} 
guarantees that we can find a sufficient number of carefully chosen indeterminates in $\R$. 
Or {\em b)} we could try to find a construction that removes cycles from $D$, possibly changing the
gammoid represented by the resulting digraph $D'$, then use the Lindström Lemma to obtain a 
matrix $\nu$, and then revert the constructions in order to obtain $\mu$ from $\nu$; which is what we will do now.

\begin{definition}\PRFR{Feb 15th}
	Let $D=(V,A)$ be a digraph, $x,t\notin V$ be distinct new elements, and let $c=(c_i)_{i=1}^n\in \Wbf(D)$ be a cycle walk.
	The \deftext[lifting of c in D@lifting of $c$ in $D$]{lifting of $\bm c$ in $\bm D$ by $\bm(\bm x\bm,\bm t\bm)$} is the digraph
	$D^{(c)}_{(x,t)} = (V\disunion\SET{x,t}, A')$ where
	\[ A' = A \BSET{(c_1,c_2)} \cup \SET{(c_1,t),(x,c_2),(x,t)}. \qedhere\]
\end{definition}

\noindent Observe that the cycle walk $c\in\Wbf(D)$ is not a walk in the lifting of $c$ in $D$ anymore.

\begin{example}\PRFR{Feb 15th}
	Consider $D=(\SET{c_1,c_2,c_3,c_4},\SET{(c_1,c_2),(c_2,c_3),(c_3,c_4),(c_4,c_1)})$. Then
	$c_1c_2c_3c_4c_1\in\Wbf(D)$ is a cycle. The lifting of $c$ in $D$ by $(x,t)$ is then defined to be
	the digraph $D'=(\SET{c_1,c_2,c_3,c_4,x,t},\SET{(c_1,t),(c_2,c_3),(c_3,c_4),(c_4,c_1),(x,c_2),(x,t)})$. \\[5mm]
	\hspace*{4cm}
	\includegraphics{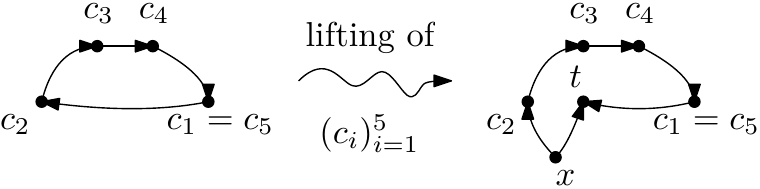} \qedhere
\end{example}

\begin{lemma}\label{lem:liftingNoNewCycles}\PRFR{Feb 15th}
	Let $D=(V,A)$ be a digraph, $x,t\notin V$, and $c=(c_i)_{i=1}^n\in \Wbf(D)$ a cycle walk, and let
	$D'=D^{(c)}_{(x,t)}$ be the lifting of $c$ in $D$ by $(x,t)$.
	If $c'\in \Wbf(D')$ is a cycle walk, then
	$c'\in \Wbf(D)$. In other words, the lifting of cycle walks does not introduce new cycle walks.
\end{lemma}
\begin{proof}\PRFR{Feb 15th}
	Let $D'=(V',A')$.
	Clearly, $x$ is a source in $D^{(c)}_{(x,t)}$ and $t$ is a sink in $D^{(c)}_{(x,t)}$. Thus $x,t\notin \left| c' \right|$.
	But then $\left| c' \right|_A \subseteq A'\cap \left( V\times V \right)$ and therefore $c'$ is also a cycle walk in $D$.
\end{proof}

\begin{definition}\label{def:completeLifting}\PRFR{Feb 15th}
	Let $D=(V,A)$ be a digraph. A \deftext[complete lifting of D@complete lifting of $D$]{complete lifting of $\bm D$}
	is an acyclic digraph $D'=(V',A')$ for which there is a suitable $n\in \N$ such that there is a
	set $X=\dSET{x_1,t_1,x_2,t_2,\ldots,x_n,t_n}$ with $X\cap V = \emptyset$,
	a family of digraphs $D^{(i)} = (V^{(i)},A^{(i)})$ for $i\in \SET{0,1,\ldots,n}$
	where $D' = D^{(n)}$, $D^{(0)} = D$, and for all $i\in\SET{1,2,\ldots,n}$
	 $$D^{(i)} = \left( D^{(i-1)}\right)^{(c_i)}_{(x_i,t_i)}$$
	with respect to a cycle walk $c_i\in \Wbf\left( D^{(i-1)}\right)$.
	In this case, we say that the set $$R = \SET{(x_i,t_i)\mid i\in\SET{1,2,\ldots,n}}$$ \deftextX{realizes} 
	the complete lifting $D'$ of $D$.
\end{definition}

\begin{lemma}\label{lem:completelifting}\PRFR{Feb 15th}
	Let $D=(V,A)$ be a digraph. Then $D$ has a complete lifting.
\end{lemma}
\begin{proof}\PRFR{Feb 15th}
	By induction on the number of cycle walks in $D$. If $D$ has no cycle walk, $D$ is a complete lifting of $D$.
	Now let $c\in \Wbf(D)$ be a cycle walk, and let $x,t\notin V$.
	Let $D' = D^{(c)}_{(x,t)}$. By construction $c\notin \Wbf(D')$,
	thus
	 $D'$ has strictly fewer cycle walks than $D$ (Lemma~\ref{lem:liftingNoNewCycles}), therefore there is a complete lifting $D''$ of $D'$ by induction hypothesis.
	  Since $D'$ is a lifting of $D$, $D''$ is also a complete lifting of $D$.
\end{proof}

\needspace{3\baselineskip}
\begin{lemma}\label{lem:cyclelifting}\PRFR{Feb 15th}
	Let $D=(V,A)$, $E,T\subseteq V$, $c\in \Wbf(D)$ a cycle, $x,t\notin V$, and let $D'=D^{(c)}_{(x,t)}$ be the lifting of $c$ in $D$.
	Then $\Gamma(D,T,E) = \Gamma(D',T\cup\SET{t},E\cup\SET{x})\contract E$.
\end{lemma}
\begin{proof}\PRFR{Feb 15th}
	Let $M=\Gamma(D,T,V)$ be the strict gammoid induced by the representation $(D,T,E)$ 
	of the not necessarily strict gammoid $\Gamma(D,T,E)$, and let \linebreak
	$M' = \Gamma(D',T\cup\SET{t},V')$ be the strict gammoid obtained from the lifting of $c$.
	Then $M'' = \left( M' \right)\contract\left( V\cup\SET{t} \right)$ is a strict gammoid that is represented by $(D'',T,V\cup\SET{t})$
	where the digraph
	 $D'' = (V_0\BSET{x},A_0\BS\left( V_0\times\SET{x} \right))$
	is induced from the $x$-$t$-pivot $D_0$ of $D'$, i.e.
	 $D_0 = D'_{x\leftarrow t} = (V_0,A_0)$. This follows from the proof of Lemma~\ref{lem:contractionStrictGammoid}
	 along with the single-arc routing $\SET{xt}\colon \SET{x}\routesto T\cup\SET{t}$ in $D'$.
	Let $A''$ denote the arc set of $D''$. 
	It is easy to see from the involved constructions (Fig.~\ref{fig:liftingcycles}), 
	that $A'' = \left(  A\BSET{(c_1,c_2)} \right) \cup\SET{(c_1, t), (t, c_2)}$.
	 Clearly, a routing $R$ in $D$ can have at most one path $p\in R$ such that $(c_1,c_2)\in \left| p \right|_A$, 
	 and since $t\notin V$, we obtain a routing $R'= \left(R\BSET{p}\right)\cup\SET{q t r}$ 
	 for $q,r\in \Pbf(D)$ such that 
	$p=qr$  with $q_{-1}=c_1$ and $r_1=c_2$.
	 Clearly, $R'$ routes $X$ to $Y$ in $D''$ whenever $R$ routes $X$ to $Y$ in $D$.
	Conversely, let $R'\colon X'\routesto Y'$ be a routing in $D''$ with $t\notin X'$.
	 Then there is at most one $p\in R'$ with $t\in \left| p \right|$.
	  We can invert the construction and let $R''= \left( R'\BSET{p}  \right)\cup\SET{qr}$
	   for the appropriate paths $q,r\in \Pbf(D'')$ with $p=qtr$. 
	Then $R''$ is a routing from $X'$ to $Y'$ in $D'$. Thus we have shown that $M''\restrict V = M$,
	and consequently, with $E\subseteq V$ and Lemma~\ref{lem:contractrestrictcommutes}, it follows that
	\begin{align*}
	\Gamma(D,T,E) = M\restrict E = \left( M''  \right)\restrict E & = \left( \Gamma(D',T\cup\SET{t},V') \contract \left( V\cup\SET{t} \right) \right) \restrict E \\& = \Gamma(D',T\cup\SET{t},E\cup\SET{x})\contract E. \qedhere
\end{align*}
\end{proof}

\begin{figure}[t]
\begin{center}
\includegraphics{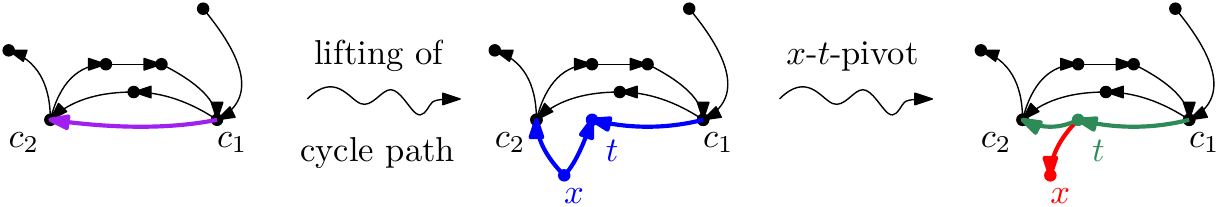}
\end{center}
\caption{\label{fig:liftingcycles}Constructions involved in Lemma~\ref{lem:cyclelifting}.}
\end{figure}

\begin{corollary}\label{cor:acyclicQuasiRepresentationOfGammoids}\PRFR{Feb 15th}
	Let $M=(E,\Ical)$ be a gammoid. Then there is an acyclic digraph $D=(V,A)$ and sets $T,E'\subseteq V$ such that
	$M = \Gamma \left( D,T,E' \right)\contract E$
	and such that $$\left| T \right| = \rk_M(E) + \left| E'\BS E \right|.$$
\end{corollary}
\begin{proof}\PRFR{Feb 15th}
	Let $M=\Gamma(D',T',E)$ with $\left| T' \right|=\rk_M(E)$.
	Then let $D$ be a complete lifting of $D'$ (Lemma~\ref{lem:completelifting}),
	and let $D^{(0)},D^{(1)},\ldots, D^{(n)}$ be the family of digraphs and $c_1,c_2,\ldots,c_n$ be the cycle walks that correspond to 
	the complete lifting $D$ of $D'$ 
	as required by Definition~\ref{def:completeLifting},
	and let $\dSET{x_1,t_1,\ldots,x_n,t_n}$ denote the new elements such that
	\[ D^{(i)} = \left( D^{(i-1)}  \right)^{(c_i)}_{(x_i,t_i)} \]
	holds for all $i\in\SET{1,2,\ldots,n}$.
	Induction on the index $i$ with Lemma~\ref{lem:cyclelifting} yields that
	\[ \Gamma(D',T,E) = \Gamma(D^{(i)},T\cup\SET{t_1,t_2,\ldots,t_i},E\cup\SET{x_1,x_2,\ldots,x_i})\contract E\]
	holds for all $i\in\SET{1,2,\ldots,n}$.
	Clearly,
	\[ \left| T\cup\SET{t_1,t_2,\ldots,t_n} \right| = \left| T \right| + n = \rk_M(E) + n = \rk_M(E) + \left| \SET{x_1,x_2,\ldots,x_n} \right|. \qedhere \]
%	we prove the statement by induction on the cardinality of the sets that realize the complete lifting $D'$ of $D$.
%	In the base case, the empty set realizes the complete lifting. Then $D=D'$ is acyclic,
%	so we can set
%	$E'=E$ and $T=T'$. For the induction step, let $2n$ be the cardinality of the set $R$ that realizes the complete lifting $D'$ of $D$.
%	Let $D^{(0)},D^{(1)},\ldots, D^{(n)}$ be the family of digraphs and $c_1,c_2,\ldots,c_n$ be the cycle walks that correspond to $D'$ 
%	as required by Definition~\ref{def:completeLifting}.
%	Then the induction step follows from Lemma~\ref{lem:cyclelifting} applied to the digraph $D''$, which is obtained from $D^{(n-1)}$
%	by induction hypothesis, and the cycle walk $c_n$.
\end{proof}

\begin{theorem}\label{thm:gammoidOverR}\PRFR{Feb 15th}
	Let $M=(E,\Ical)$ be a gammoid, $T=\dSET{t_1,t_2,\ldots,t_{\rk_M(E)}}$. Then there is a matrix $\mu\in \R^{E\times T}$ such
	that $M= M(\mu)$.
\end{theorem}
\begin{proof}\PRFR{Feb 15th}
	By Corollary~\ref{cor:acyclicQuasiRepresentationOfGammoids}, there is an acyclic digraph $D=(V,A)$ and
	there are sets $E',T' \subseteq V$,
	such that $M = N\contract E$ where $N=\Gamma(D,T',E')$
	and $\left| T' \right| = \rk_M(E) + \left| E' \BS E \right|$.
	Remember that $E'\BS E$ is independent in $N$,
	and every base $B$ of $M$ induces a base $B\cup \left( E'\BS E \right)$ of $N$.
	The Lindström Lemma~\ref{lem:lindstrom} yields a matrix $\nu\in \R^{E'\times T'}$ such that 
	$N = M(\nu)$. In Lemma~\ref{lem:contractequalspivot} and Remark~\ref{rem:contraction} we have seen that we can pivot in the
	independent set $E'\BS E$
	in $\nu$, which yields a new matrix $\nu'\in \R^{E'\times T'}$. Let $T_0 = \SET{t'\in T' ~\middle|~ \forall e'\in E'\BS E\colon\, \nu'(e',t') = 0 }$ denote the remaining columns of $\nu'$ that have not been used to pivot in an element of $E'\BS E$.
	We set $\mu = \nu'\restrict E\times T_0$. Thus $M(\mu) = M(\nu)\contract E = N\contract E = M$.
\end{proof}

\noindent Let us compare the two methods {\em a)} and {\em b)} mentioned above. In our opinion, 
both methods are connected to aspects of the same underlying phenomenon that cycle paths do not interfere with 
the existence of linkings between given sets of vertices in a digraph.

% -*- root: ../../thesis.tex -*-

\needspace{6\baselineskip}

\vspace*{-\baselineskip} %Remove the line space created by the tilde below
\begin{wrapfigure}{r}{4cm}
\vspace{1\baselineskip}
\begin{centering}%move the picture slightly to the right
~~\includegraphics{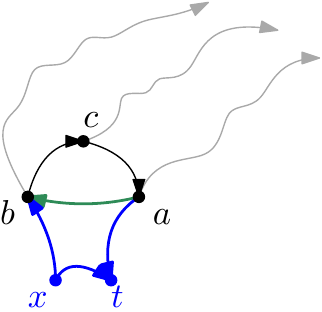}
\end{centering}%
\vspace*{-1\baselineskip} %make the picture more tightly cropped
\end{wrapfigure}
~ %The tilde creates a new dummy paragraph. WHY IS THAT NEEDED? -> would increase the space %
  % before the ex. environment. THE NEXT FREE LINE IS ESSENTIAL!

\begin{example}\PRFR{Feb 15th}
	Let $D=(V,A)$ be a digraph, such that the only cycle walk in $D$ is $abca\in \Wbf(D)$,
	and let $x,t\notin V$. Now chose an arbitrary target node $t_0\in V$. Let $D'=(V',A')=D^{(abca)}_{(x,t)}$
	be the lifting of $abca$ in $D$, clearly $D'$ is acyclic. Let $w\colon A'\maparrow \R$ be an
	indeterminate weighting of $D'$. We introduce the following
	sums
	\[
		\alpha  = \sum_{p\in P_{\left(a,t_0\right)}} \,\prod p, \quad
		\beta  = \sum_{p\in P_{\left(b,t_0\right)}} \,\prod p, \text{~and} \quad
		\gamma  = \sum_{p\in P_{\left(c,t_0\right)}} \,\prod p, \\
	\]
	where for $u\in\SET{a,b,c}$, $$P_{(u,t_0)} = \SET{p\in \Pbf(D') ~\middle|~ p_1=u,\, p_{-1}=t_0, \txtand \SET{(b,c),(c,a)}\cap \left| p \right|_A = \emptyset},$$
	i.e. $P_{(u,t_0)} $
	consists of the paths from $u$ to $t_0$ not visiting another element from $\SET{a,b,c}$.
	Now let $\mu\in \R^{V'\times \SET{t_0,t}}$ be the matrix obtained from the Lindström Lemma~\ref{lem:lindstrom}
	for the strict gammoid $M = \Gamma(D',\SET{t_0,t},V')$.
	We set $w_b = w(b,c)$, $w_c = w(c,a)$, $w_x = w(x,b)$.
	%\remred{TODO: $w_\bullet$ indexe prüfen!}
	Clearly, we have 
	\begin{align*}
		\mu(a,t_0) &= \alpha, \\
		\mu(c,t_0) &= \gamma + w_c \cdot \alpha,\\
		\mu(b,t_0) &= \beta + w_b \cdot \gamma + w_b \cdot w_c \cdot \alpha, \txtand \\
		\mu(x,t_0) &= w_x \cdot \beta + w_x \cdot w_b \cdot\gamma + w_x \cdot w_b \cdot w_c \cdot\alpha .\\
	\end{align*}
	Furthermore, we set $w_x' = w(x,t)$ and $w_a' = w(a,t)$. With respect to the new target $t$
	introduced by the lifting of $abca$ in $D$, we have
	 $\mu(x,t) = {w_x'} $, $\mu(a,t) = {w_a'}$, 
	\linebreak
	 $\mu(c,t) = w_c \cdot {w_a'}$, and 
	  $\mu(b,t) = w_b \cdot w_c \cdot {w_a'}$. Let $N=M\contract\left( V'\BSET{t} \right)$
	  and $\nu\in \R^{ \left(V'\BSET{t}\right)\times \SET{t_0}}$ be as in Lemma~\ref{lem:contractequalspivot}
	  with $M(\nu) = N$. Then
	  \begin{align*}
	  	\nu(a,t_0) & = \mu(a,t_0) - \frac{\mu(a,t)}{\mu(x,t)}\mu(x,t_0) \\
	  		       & = \alpha - \frac{{w_a'} \cdot w_x }{{w_x'} } \left(\beta + w_b \cdot\gamma + w_b \cdot w_c \cdot\alpha \right) \\
	  		       %& = \left(\frac{{w_a'}\cdot w_x }{{w_x'} } \right)\cdot \left(\, \left(w_x'-w_b \cdot w_c \right) \cdot \alpha - \beta - w_b \cdot \gamma \right)
	  \end{align*}

	 \bSep
	  Now let $w'\colon A\maparrow \R$ be an indeterminate weighting of $D$ where $w'(q) = w(q)$ for all $q\in A\BSET{(a,b)}$.
	  We set $w_a = w'(a,b)$.
	  We calculate
	  $$ \alpha' = \sum_{p\in \Pbf(D;a, t_0)} \prod p =
	  %where the sum ranges over $P'=\SET{p\in \Pbf(D) \mid p_1 = a \txtand p_{-1}=t_0}$.
	  %Then $\alpha' = 
	  \alpha + w_a  \cdot \beta + w_a \cdot w_b \cdot \gamma.$$
	  If we further assume that $0 < w_a \cdot w_b  \cdot w_c  < 1$, then we have convergence in the following equation
	  \begin{align*}
	   \alpha'' = \sum_{w\in \Wbf(D;a,t_0)} \left( \prod w  \right) & = \sum_{i=0}^\infty (w_a \cdot w_b \cdot w_c )^i \cdot \alpha'% \\
	  % &
	    = \frac{\alpha + w_a  \cdot \beta + w_a \cdot w_b \cdot \gamma}{1-w_a \cdot w_b \cdot w_c }
	  \end{align*}
	  % where the sum ranges over $P'' = \SET{p\in \Pbf(D) \mid p_1 = a \txtand p_{-1}=t_0} $, the set of all paths from $a$ to $t_0$ in $D$.
	   The second equation holds because $abca$ is the only cycle walk in $D$, therefore all non-path walks from $a$ to $t_0$ must be of the form $(abc)^ip$ for $i\in \N\BSET{0}$ and $p\in \Pbf(D;a,t_0)$; 
	   the summand where $i=0$ corresponds to the paths $\Pbf(D;a,t_0) \subseteq \Wbf(D;a,t_0)$.
	   Therefore $\alpha'' = \mu'(a,t_0)$, where $\mu'$ is the matrix that we would have obtained from $D$ using the Lindström Lemma method, operating with formal power series
		 and a convergent indeterminate weighting as in version {\em a)} above.

	  \bSep
	   We argue that that for a given indeterminate weighting $w$ of $D$ ---
	   which has been chosen such that every formal power series involved in the construction of the Lindström Lemma matrix converges;
	   and such that whenever the power series of a sub-determinant of the matrix converges to zero, then the power series of that determinant is the zero series
	   ---
	   there is an indeterminate weighting $w'$ of $D'$,
	   with $w(q) = w'(q)$ for all $q\in A\cap A'$ 
	   such that $\nu(a,t_0) = \mu'(a,t_0)$.
	   The formal equation $\nu(a,t_0) = \mu'(a,t_0)$
	  % \begin{align*}   {{w_x'}} {{\alpha}} + {w_a} {{w_x'}} {{\beta}}& + {w_a} {w_b} {{w_x'}} {{\gamma}}\\ & =\hphantom{+}
	  % {\left({w_a} {{w_a'}} {w_b}^{2} {w_c}^{2} {w_x} - {w_a} {{w_a'}} {w_b} {w_c} {w_x} - {{w_a'}} {w_b} {w_c} {w_x} + {{w_a'}} {w_x}\right)} {{\alpha}} 
	  % \\ & \hphantom{ = }\,\, + {\left({w_a} {{w_a'}} {w_b} {w_c} {w_x} - {{w_a'}} {w_x}\right)} {{\beta}} + {\left({w_a} {{w_a'}} {w_b}^{2} {w_c} {w_x} - {{w_a'}} {w_b} {w_x}\right)} {{\gamma}},
	   %\end{align*}
	   %which can be solved for
	   %\[ {{w_a'}} = \frac{{w_a} {w_b} {{w_x'}} {{\gamma}} + {w_a} {{w_x'}} {{\beta}} + {{w_x'}} {{\alpha}}}{{\left({w_a} {w_b}^{2} %{w_c}^{2} - {\left({w_a} + 1\right)} {w_b} {w_c} + 1\right)} {w_x} {{\alpha}} + {\left({w_a} {w_b} {w_c} - 1\right)} {w_x} {{\beta}} %+ {\left({w_a} {w_b}^{2} {w_c} - {w_b}\right)} {w_x} {{\gamma}}}\]
	   %and for
	   %\[ {w_a} = \frac{{{w_a'}} {w_b} {w_x} {{\gamma}} + {{w_a'}} {w_x} {{\beta}} + {\left({\left({{w_a'}} {w_b} {w_c} - {{w_a'}}\right)} {w_x} + {{w_x'}}\right)} {{\alpha}}}{{\left({{w_a'}} {w_b}^{2} {w_c}^{2} - {{w_a'}} {w_b} {w_c}\right)} {w_x} {{\alpha}} + {\left({{w_a'}} {w_b} {w_c} {w_x} - {{w_x'}}\right)} {{\beta}} + {\left({{w_a'}} {w_b}^{2} {w_c} {w_x} - {w_b} {{w_x'}}\right)} {{\gamma}}}.\]
	   may be solved for \[
	   w_{a}' =  \frac{-w_x'\cdot w_a}{w_x\cdot(1 - w_a\cdot w_b\cdot w_c) } = \frac{P'}{Q'} \]
	   yielding the non-trivial polynomial equation $Q'\cdot w_a' - P' = 0$ where $P'$ and $Q'$ are 
	   integer-coefficient polynomials. Therefore we may extend and restrict $w$ to an indeterminate weighting
	    $\tilde w\colon (A\cup A')\BSET{(a,t),(a,b)}\maparrow \R$,
	   and then set $w'(x) = \tilde w(x)$ for $x\in A' \BSET{(a,t)}$, and $w'((a,t)) = w_a'$ as in the equation above, calculated with respect to $\tilde w$.
	 	This yields the desired indeterminate weighting, because $w'((a,t))$ $\Z$-depends on $\tilde w((a,b)) = w((a,b))$ which is
	 	$\Z$-independent of $\tilde w[A\BSET{(a,b)}]$.
%	   Since $w[A]$ is $\Z$-independent, we claim that there is a indeterminate weighting $w'$ of $D'$ 
%	   with $w'(a) = w(a)$ for $a\in A\cap A'$,
%	   and where $w'(a,t) = w_a'$ defined as in the equation above. We give an indirect argument.
%	   Take any indeterminate weighting $w''$ that extends $w$ to $A'$. If we set $w'(x) = w''(x)$ for $x\in A'\BSET{(a,t)}$
%	   and $w'(a,t) = w_a''$, then $w'$ is still $\Z$-independent, since otherwise the solution of the equation above for $w_a$ would contradict the fact that $w$ is an indeterminate weighting of $D$. 
\end{example}

\PRFR{Mar 7th}
\noindent
In the paper {\em A parameterized view on matroid optimization problems} \cite{Marx09}, D.~Marx shows that there is a randomized polynomial time algorithm with respect to the size of the ground set of a gammoid, 
that constructs a matrix $\mu$ from $(D,T,E)$ such that $M(\mu) = \Gamma(D,T,E)$.
The method of D.~Marx starts with the construction of the dual $N^\ast$ of the underlying strict gammoid $N=\Gamma(D,T,V)$ 
for a given representation $(D,T,E)$
with $D=(V,A)$ through the linkage system of $D$ to $T$ (Definition~\ref{def:linkageSystem} and Lemma~\ref{lem:linkage}). 
Then a matrix $\nu$ with $M(\nu) = N^\ast$ is constructed with a small probability of failure (see Proposition~\ref{prop:randompolytimeTransversalMatroid} below), 
which in turn is converted into a standard representation (Remark~\ref{rem:stdRep})
of the form $(I_r\,\, A^\top)^\top$ using Gaussian Elimination. Then $(-A \,\, I_{n-r})^\top$ is the desired representation of $M$.
Before we present the main proposition that leads to this result, we need the following lemma.

\begin{lemma}[\cite{Sch80}, Corollary 1]\PRFR{Mar 7th}\label{lem:SchwartzNumberZeros}\PRFR{Mar 7th}
	Let $\Fbm$ be a field, $X = \dSET{x_1,x_2,\ldots,x_n}$,
	 let $p \in \Fbm[X]$ be a polynomial
	with $p \not= 0$.
	Furthermore, let $F\subseteq \Fbm$ be a finite subset of elements of the coefficient field with $\left| F \right| \geq c\cdot \deg(p)$
	for some $c\in \Qbm$ with $c > 0$. Then
	  \[	\left| \SET{\xi \in F^X ~\middle|~ p[X=\xi] = 0} \right|  \leq \frac{\left| F \right|^n}{c}.\]
\end{lemma}

\PRFR{Mar 7th}
\noindent For a formal proof, we refer the reader to J.T.~Schwartz's 
{\em Fast Probabilistic Algorithms for Verification of Polynomial Identities} \cite{Sch80}. 
The proof idea is to do induction on the number of variables involved. The base case is the fact that a polynomial
in a single variable of degree $d$ can have at most $d$ different roots. In the induction step, we fix the values of all but one variable,
if the resulting polynomial in a single variable is the zero polynomial, we may choose any value from $F$ for that variable. Otherwise,
there are at most the degree of the resulting polynomial many choices for the last variable such that the polynomial evaluates to zero.

\begin{lemma}[\cite{Marx09}, Lemma 1, \cite{Sch80}, \cite{Zi79}]\label{lem:probOfZero}
	Let $\Fbm$ be a field, \linebreak
	let $X = \dSET{x_1,x_2,\ldots,x_n}$ be a finite set,
	 let $p \in \Fbm[X]$ be a polynomial
	with $p \not= 0$, and let $F\subseteq \Fbm$ be a finite set.
	Let $\xi$ be a random variable sampled from a uniform distribution on the set $F^X$.
	Then the probability that $\xi$ is a zero of $p$ may be estimated by
	\[ \Pr\left(\vphantom{A^A}p[X=\xi] = 0\right) \leq \frac{\deg(p)}{\left| F \right|} .\]
\end{lemma}

\begin{proof}\footnote{D.~Marx omits the proof and instead cites \cite{Sch80} and \cite{Zi79}.} \PRFR{Mar 7th}
	In Lemma~\ref{lem:SchwartzNumberZeros} we set $c = \frac{\left| F \right|}{\deg(p)}$ and get
	\[ \frac{\left| \SET{\xi \in F^X ~\middle|~ p[X=\xi] = 0} \right|}{ \left| F^X \right|} \leq \frac{\left| F \right|^n}{c\cdot \left| F \right|^n} = \frac{1}{c} = \frac{\deg(p)}{\left| F \right|}. \qedhere\]
\end{proof}

\begin{proposition}[\cite{Marx09}, Proposition 3.11]\label{prop:randompolytimeTransversalMatroid}\PRFR{Mar 7th}
	Let $E$ be a finite set, $r\in \N$, and $\Acal = (A_i)_{i=1}^r \subseteq E$ be a family of subsets of $E$.
	Then a matrix $\mu\in \R^{E\times \SET{1,2,\ldots,r}}$ with $M(\mu) = M(\Acal)$ can be constructed in randomized polynomial time.
\end{proposition}
\begin{proof}\PRFR{Mar 7th}
	For all $k\in \N$ with $k>1$, we write $\mathrm{unif}(k)$
	in order to denote an integer that has been randomly sampled from a uniform distribution on $\SET{1,2,\ldots,k}$. Several instances
	of $\mathrm{unif}(k)$ shall denote independently sampled random variables.
	%
	%\noindent
	Let $p\in \N$ be an arbitrary parameter.
	We define the random matrix $\mu\in \R^{E\times \SET{1,2,\ldots,r}}$ by\footnote{D.~Marx uses samples 
	from $\mathrm{unif}\left( 2^p\cdot \left| E \right|\cdot 2^{\left| E \right|} \right)$ and uses the argument that there are at most
	$2^{\left| E \right| }$ independent sets. This line of arguments is valid, 
	yet it does not use the fact that if $X$ is independent in $M(\mu)$, then all subsets of $X$ are independent in $M(\mu)$, too; 
	consequently, the probability of failure is overestimated.}
	\[ \mu(e,i) = \begin{cases}[r]
					\mathrm{unif} \left( 2^p\cdot r\cdot Q \right) & \quad\text{if~} e\in A_i, \\
					0 & \quad\text{otherwise,}
	\end{cases} 
	\]
	where $$Q = \binom{\left| E \right|}{\left\lceil \frac{\left| E \right|}{2}\right\rceil}.$$
	Clearly, $Q \leq 2^{\left| E \right|}$ with equality if $\left| E \right| = 1$.
	Observe that sampling $\mathrm{unif}(2^k)$ can be done by sampling $k$ bits from a uniform distribution. Thus $\mu$ can be obtained
	by sampling at most
	 $\left| E \right|\cdot r \cdot \left(p + \lceil \log_2\left( Q + r \right) \rceil \right)$ uniform random bits.
	We show that \mbox{$\Pr\left(  M(\mu) \not= M(\Acal)  \vphantom{A^A}\right)  \leq \frac{1}{2^p}$.}
	Let $X\subseteq E$ be an independent set with respect to  $M(\mu)$. Then $\idet(M\restrict X\times \SET{1,2,\ldots,r}) = 1$,
	so there is an injective map $\phi\colon X\maparrow \SET{1,2,\ldots,r}$ such that $\mu(x,\phi(x)) \not= 0$
	for all $x\in X$. By construction of $\mu$ we obtain that in this case $x\in A_{\phi(x)}$. Therefore $X$ is a partial transversal
	of $\Acal$, and so $X$ is independent in $M(\Acal)$, too.

	\PRFR{Mar 7th}
	\noindent
	Now let $X\subseteq E$ be a base of $M(\Acal)$. Thus $X$ is a maximal partial transversal of $\Acal$ and
	there is an injective map $\phi\colon X\maparrow \SET{1,2,\ldots,r}$ such that $x\in A_{\phi(x)}$ for all $x\in X$.
	Let $X=\dSET{x_1,x_2,\ldots,x_k}$, then we may define the matrix $\nu \in \R[X]^{X\times \phi[X]}$
	where
		\[ \nu(x,i) = \begin{cases}[r]
						x_i & \quad\text{if~} i = \phi(x),\\
						\mu(x,i) & \quad\text{otherwise.}
		\end{cases}\]
	Then $\det(\nu)$ is a polynomial of degree $\left| X \right| = k \leq r$ with leading monomial $x_1x_2\cdots x_k$ in $\R[X]$,
	and if $\xi\in \R^X$ is the vector
	where $\xi(x) = \mu(x,\phi(x))$ for all $x\in X$, we have the equality
	\[ \det(\mu\restrict X\times \phi[X]) = \left( \det(\nu) \right)[X=\xi] .\]
	Remember that each value $\xi(x)$ has been uniformly sampled from a set with cardinality $ 2^p\cdot r \cdot Q$,
	thus Lemma~\ref{lem:probOfZero} yields
	\[ \Pr\left( \vphantom{A^A}\det(\mu\restrict X\times\phi[X]) = 0 \right) \leq \frac{\left| X \right|}{2^p\cdot r \cdot Q} \leq \frac{1}{2^p\cdot Q}.\]
	There are at most $\binom{\left| E \right|}{\rk_{M(\Acal)}(E)}$ different bases in $M(\Acal)$,
	and the family of all subsets of $E$ with cardinality $\left\lceil  \frac{\left| E \right|}{2} \right\rceil$
	is a maximal-cardinality anti-chain in the power set lattice of $E$.
	Therefore, there are at most $Q$ different bases
	in $M(\Acal)$
	needed to detect failure of $M(\Acal)=M(\mu)$.
	Thus we obtain
	\[ \Pr\left( M(\mu) \not= M(\Acal) \vphantom{A^A}\right) \leq \sum_{B\in \Bcal(M(\Acal))} \frac{1}{2^p\cdot Q} \leq \frac{1}{2^p} . \qedhere\]
\end{proof}

\noindent Clearly, if the rank of $M(\Acal)$ is known, we may use the better factor
$Q = \binom{\left| E \right|}{\rk_M(E)}$ in the probabilistic construction of $\mu$
given in the proof of Proposition~\ref{prop:randompolytimeTransversalMatroid}. 
 If we also know the number of bases, we may even use $Q= \left| \Bcal(M(\Acal)) \right|$.

\cleardoublepage
% -*- root: ../thesis.tex -*-

\chapter{Oriented Matroids}\label{ch:OMs}
\remblue{
\remred{TODO} Oriented matroids have been studied by....

\noindent Historically, oriented matroids have been defined by J.~Folkman and J.~Lawrence \cite{FoLa78} and by R.G.~Bland and M.~Las~Vergnas \cite{BlV78}
as structures that arise when the circuits of a matroid are turned into signed subsets of the ground set such that certain properties mimicking those of the
sign patterns associated with the coefficients of non-trivial linear combinations of the zero vector are fulfilled. Both papers show that
 many aspects of ordinary matroid theory carry over to oriented matroids, most notably the concepts of duality and minors are reflected by oriented matroids
 in a way that is compatible with the structure of the underlying matroids. Therefore we are at liberty to merge aspects of the theory of oriented matroids
 into our definition of oriented matroids.
}

\section{Quick Introduction to Oriented Matroids}

\PRFRC
Let us consider a matroid $M(\mu)$ where $\mu\in \R^{E\times \SET{1,2,\ldots,r}}$ is a finite matrix over the reals
with full column rank, i.e. such that $\rk_{M(\mu)}(E) = r$.
Whenever $C\in \Ccal(M(\mu))$ is a circuit, there are coefficients $\alpha\colon C\maparrow \R$ such that
$\alpha(c)\not= 0$ for all $c\in C$ and such that
\[ \sum_{c\in C} \alpha(c) \cdot \mu_c = 0 \]
holds in the vector space $\R^r$. Furthermore, $\alpha$ is uniquely determined by $\SET{\mu_c\mid c\in C}$
up to a homogeneous factor $\lambda \in \R\BSET{0}$, i.e. whenever the
equality $\sum_{c\in C} \beta(c) \cdot \mu_c = 0$ holds for $\beta\colon C\maparrow \R$ with $\beta$ not
constantly zero on $C$, then there is some $\lambda \in \R\BSET{0}$ with
$\alpha(c) = \lambda \beta(c)$ for all $c\in C$. Therefore, the signs of the coefficients are determined 
up to a possible negation of all signs by
the circuit $C$ and the matrix $\mu$. 

\begin{definition}\PRFRC
	Let $E$ be a set. A \deftext[signed subset]{signed subset of $\bm E$} is a map\label{n:signedsubset}
	\[ X\colon E\maparrow \SET{-1,0,1}.\]
	We denote the \deftext[positive elements of a signed subset]{positive elements of $\bm X$} by
	\( X_+ = \SET{x\in E\mid X(x) = 1}, \)\label{n:xplus}
	the \deftext[negative elements of a signed subset]{negative elements of $\bm X$} shall be denoted by
	\( X_- = \SET{x\in E\mid X(x)=-1},\)\label{n:xminus}
	the \deftext[support of a signed subset]{support of $\bm X$} is defined as
	\( X_\pm = \SET{x\in E\mid X(x)\not= 0},\)\label{n:xpm}
	and the \deftext[zero-set of a signed subset]{zero-set of $\bm X$} is 
	denoted by $X_0 = E\BS X_\pm$.\label{n:xzero} The \deftext[negation of a signed subset]{negation of $\bm X$}
	is the signed subset \(-X\)\label{n:minusx} where $-X\colon E\maparrow \SET{-1,0,1},$ $e\mapsto -X(e)$. 
	Let $C\subseteq E$ and $\alpha \colon C\maparrow \R$ be a vector of coefficients. 
	The \deftext[signs of a over E@signs of $\alpha$ over $E$]{signs of $\bm \alpha$ over $\bm E$} shall be denoted by $E_\alpha$, which is defined
	to be the map
	\[ E_\alpha \colon E\maparrow \SET{-1,0,1},\,e\mapsto \begin{cases} \hphantom{-}0 & \quad \text{if }e\notin C \txtor \alpha(e) = 0,\\
																		-1 & \quad \text{if }\alpha(e) < 0,\\
																		\hphantom{-}1 & \quad \text{if }\alpha(e) > 0.
																		\end{cases}\]
	The \deftextX{class of all signed subsets of $\bm E$}\label{n:signedsubsets} shall be denoted by $\sigma E$.
	Let $X,Y\in \sigma E$ be signed subsets of $E$. We say that \deftext[signed subset]{$\bm X$ is a signed subset of $\bm Y$},
	if $X_+\subseteq Y_+$ and $X_- \subseteq Y_-$. We denote this fact by writing $X\subseteq_\sigma Y$.\label{n:subsetsigma}
	Furthermore, we write $X\subsetneq_\sigma Y$ whenever $X\subseteq_\sigma Y$ and $X_\pm \subsetneq Y_\pm$ holds.
	The \deftext{empty signed subset} of $E$ is the map\label{n:emptysigma}
\( \emptyset_{\sigma E} \colon E\maparrow \SET{-1,0,1},\,e\mapsto 0 .\)
\end{definition}

\begin{notation}\PRFR{Apr 5th}
	Let $E$ be a finite set, and let $C\in \sigma E$ such that $C_+ = \dSET{p_1,p_2,\ldots,p_m}$ and $C_- = \dSET{n_1,n_2,\ldots,n_k}$.
	We shall denote $C$ by both
	\[ \SET{p_1,p_2,p_3,\ldots,p_m, -n_1, -n_2,\ldots,-n_k} \]
	and
	\[ \SET{+p_1,+p_2,+p_3,\ldots,+p_m, -n_1, -n_2,\ldots,-n_k},\]
	i.e. we write a list of $C_\pm$ where every element from $C_+$ has either no prefix or a $+$-sign, and where every element of $C_-$
	has a $-$-sign as prefix. The elements of $C_0$ are not listed. As with normal sets, we disregard the order in which the elements of $C_\pm$ are listed.
\end{notation}

\begin{example}\PRFRC
With regard to a fixed representation $\mu\in \R^{E\times\SET{1,2,\ldots,r}}$,
every circuit $C\in \Ccal(M(\mu))$ gives rise to two different signed subsets of $E$:
Let $\alpha\colon C\maparrow \R$ be not constantly zero on $C$ with $\sum_{c\in C}\alpha(c)\cdot\mu_c=0$,
then $E_\alpha$ and $-E_\alpha$ are the signed subsets of $E$ 
that correspond to the signs of non-trivial coefficients $\alpha\colon C\maparrow \R$ with $\sum_{c\in C}\alpha(c)\cdot \mu_c = 0$.
\end{example}

\begin{definition}\PRFRC
	Let $E$ be a finite set, $C,D\in \sigma E$ be signed subsets of $E$. We define the \deftext[separator of signed subsets]{separator of $\bm C$ and $\bm D$}
	to be the set\label{n:sep} \[ \sep(C,D) = \left( C_+ \cap D_- \right) \cup \left( C_- \cap D_+ \right).
	\qedhere \]
%	The \deftext[merger of signed subsets]{merger of $\bm C$ and $\bm D$} shall be
%	the signed subset\label{n:merger} \[ C\circ D \colon E\maparrow \SET{-1, 0, 1},\quad e\mapsto \begin{cases}
%																				0 & \quad \text{if~} e\in \sep(C,D),\\
%																				C(e) & \quad \text{if~} e\in C_\pm \BS \sep(C,D),\\
%																				D(e) & \quad \text{otherwise}. \end{cases} \]
%	The \deftext[composition of signed subsets]{composition of $\bm C$ with $\bm D$} is the signed subset\label{n:composition}
%	\[ C + \eps D \colon E\maparrow \SET{-1, 0, 1},\quad e\mapsto \begin{cases}
%																				C(e) & \quad \text{if~} e\in C_\pm,\\
%																				D(e) & \quad \text{otherwise}. \end{cases} \]
\end{definition}

\noindent There is a notion of orthogonality for signed subsets which generalizes the ordinary orthogonality in vector spaces
(see \cite{BlVSWZ99}, p.115; \cite{Ni12}, p.27).

\begin{definition}\label{def:XorthoY}\PRFRC
	Let $E$ be a finite set, $C,D\in\sigma E$ be signed subsets of $E$
	Then $C$ and $D$ shall be called \deftext{orthogonal signed subsets},
	if either
	\begin{enumerate}\ROMANENUM
	\item there are $e,f\in E$, such that
	\[ C(e)\cdot D(e) = - C(f)\cdot D(f) \not= 0 \]
	holds; or
	\item for all $e\in E$, the equation
	\[ C(e)\cdot D(e) = 0 \]
	holds.
\end{enumerate}
	We write $X\bot Y$\label{n:XorthoY} in order to denote that $X$ and $Y$ are orthogonal,
	and $X\!\!\not\!\!\bot\, Y$ to denote that $X$ and $Y$ are not orthogonal. In the latter case $X_\pm\cap Y_\pm \not=\emptyset$ and 
	the common elements of the supports of $X$ and $Y$ 
	all have the same relative sign with respect to $X$ and $Y$, i.e. $X(e) = \alpha\cdot Y(e)$ for all $e\in X_\pm \cap Y_\pm$
	and some $\alpha\in \SET{-1,1}$ that does not depend on the choice of $e$.
\end{definition}

\needspace{3\baselineskip}

\begin{lemma}\label{lem:MinusCOrthoD}\PRFRC
	Let $E$ be a finite set, $C,D\in \sigma E$. Then
	$C\bot D$ if and only if $\left( -C \right)\bot D$ if and only if $C\bot\left( -D \right)$ if and only if $\left( -C \right)\bot \left( -D \right)$.
\end{lemma}
\begin{proof}\PRFRC
	Since $\bot$ is obviously a symmetric relation, it suffices to show that $C\bot D$ implies $\left( -C \right)\bot D$.
	But for every $e\in E$, $\left( -C \right)(e) = -C(e)$, 
	therefore both properties {\em (i)} and {\em (ii)} of Definition~\ref{def:XorthoY} carry over from $C$ to $-C$.
\end{proof}

% \studyremark{
% \begin{lemmaX}
% 	Let $E$ be a finite set, $C,D\in\sigma E$. Then
% 	\begin{enumerate}\ROMANENUM
% 		\item $\left| C_\pm \cap D_\pm \right| = 1$ implies $C\not\bot D$; and
% 		\item $\left| C_\pm \cap D_\pm \right| \in \SET{2,3}$ and $C\bot D$ implies that {\em (i)} of Definition~\ref{def:XorthoY} holds.
% 	\end{enumerate}
% \end{lemmaX}
% \begin{proof}
% Trivial.
% \end{proof}}

\begin{lemma}\PRFRC
	Let $E$ be a finite set, $\alpha,\beta \in \R^E$ with $\langle \alpha,\beta \rangle = 0$.
	Then $E_\alpha \bot E_\beta$.
\end{lemma}
\begin{proof}\PRFRC
	If $\left( E_\alpha \right)_\pm \cap \left( E_\beta \right)_\pm = \emptyset$, then
	{\em (ii)} of Definition~\ref{def:XorthoY} holds, thus $E_\alpha \bot E_\beta$.
	Otherwise, there is some $e\in E$ with $\alpha(e)\cdot \beta(e) \not= 0$.
	Let $$E_e = \SET{e'\in E ~\middle|~ \sgn\left(\alpha(e)\cdot \beta(e)\right) = \sgn\left(\alpha(e')\cdot \beta(e')\right)}.$$
	Since $\langle a,b \rangle=0$, we have
	\[ -\sum_{e'\in E_e} \alpha(e') \cdot \beta(e')\,\,\, =\,\,\, \langle \alpha,\beta\rangle  -\sum_{e'\in E_e} \alpha(e') \cdot \beta(e')
	\,\,\, =\,\,\, \sum_{f\in E\BS E_{e}} \alpha(f)\cdot \beta(f).\]
	We give an indirect argument and assume that {\em (i)} does not hold. Then for all $f\in E\BS{E_e}$,
	we have $\alpha(f)\cdot \beta(f) = 0$. Thus $-\sum_{e'\in E_e} \alpha(e') \cdot \beta(e') = 0$, but the sign of $\alpha(e')\cdot \beta(e')$
	is the same for every $e'\in E_e$. Therefore $\alpha(e')\cdot \beta(e')=0$ for all $e\in E_e$, 
	contradicting the assumption that $\alpha(e)\cdot \beta(e) \not= 0$, so {\em (i)} must hold. Thus $E_\alpha\bot E_\beta$.
\end{proof}

\needspace{6\baselineskip}
\begin{definition}\label{def:OrientedMatroid}\PRFRC
	Let $E$ be a finite set, $\Ccal\subseteq \sigma E$ and $\Ccal^\ast \subseteq \sigma E$.
	The triple $\Ocal = (E,\Ccal,\Ccal^\ast)$\label{n:Ocal} is called \deftext{oriented matroid},
	if the following properties hold:
	\begin{enumerate}\label{n:Cx}
		\item[($\Ccal$1)] $\emptyset_{\sigma E}\notin \Ccal$,
		\item[($\Ccal$2)] for all $C\in\sigma E$, $C\in \Ccal$ if and only if $-C \in \Ccal$,
		\item[($\Ccal$3)] for all $X,Y\in \Ccal$, $X_\pm \subseteq Y_\pm$ implies $X = Y$ or $X = -Y$,
		\item[($\Ccal$4)] for all $X,Y\in \Ccal$ with $X\not= -Y$ and all $e\in X_+\cap Y_-$ and $f\in X_\pm \BS \sep(X,Y)$, there is some $Z\in\Ccal$ such that
				$e\notin Z_\pm$, $Z(f) = X(f)$, $Z_+ \subseteq X_+\cup Y_+$ and $Z_- \subseteq X_-\cup Y_-$;

		\item[($\Ccal^\ast$1)] $\emptyset_{\sigma E}\notin \Ccal^\ast$,
		\item[($\Ccal^\ast$2)] for all $C'\in\sigma E$, $C'\in \Ccal^\ast$ if and only if $-C' \in \Ccal^\ast$,
		\item[($\Ccal^\ast$3)] for all $X',Y'\in \Ccal^\ast$, $X'_\pm \subseteq Y'_\pm$ implies $X' = Y'$ or $X' = -Y'$,
		\item[($\Ccal^\ast$4)] for all $X',Y'\in \Ccal^\ast$ with $X'\not= -Y'$ and all $e\in X'_+\cap Y'_-$ and $f\in X'_\pm \BS \sep(X',Y')$, there some is $Z'\in\Ccal^\ast$ such that
				$e\notin Z'_\pm$, $Z'(f) = X'(f)$, $Z'_+ \subseteq X'_+\cup Y'_+$ and $Z'_- \subseteq X'_-\cup Y'_-$;

		\item[($\Ocal$1)] for all $C\in \Ccal$ and $C'\in\Ccal^\ast$, we have $C\bot C'$,
		\item[($\Ocal$2)] there is a matroid $M=(E,\Ical)$, such that $$\SET{C_\pm ~\middle|~\vphantom{C'} C\in \Ccal} = \Ccal(M) \,\,\txtand\,\,
		\SET{C'_\pm ~\middle|~ C'\in \Ccal^\ast} = \Ccal(M^\ast). $$

	\end{enumerate}
	In this case, the elements $C \in \Ccal$ shall be called \deftext[signed circuits of O@signed circuits of $\Ocal$]{signed circuits of $\bm \Ocal$},
	and $\Ccal$ shall be the \deftextX{family of signed circuits of $\bm \Ocal$}. 
	Likewise, the elements $C' \in \Ccal^\ast$ shall be called \deftext[signed cocircuits of O@signed cocircuits of $\Ocal$]{signed cocircuits of $\bm \Ocal$},
	and $\Ccal^\ast$ shall be the \deftextX{family of signed cocircuits of $\bm \Ocal$}.
	Furthermore, $M(\Ocal)$\label{n:MOcal} shall denote the \deftext[underlying matroid of O@underlying matroid of $\Ocal$]{underlying matroid of $\Ocal$}, whose existence and uniqueness is guaranteed by {\em ($\Ocal$2)}.
\end{definition}

\begin{remark}\label{rem:sufficientConditionOM}\PRFRC
The above definition of an oriented matroid is redundant in the sense that some of the properties follow from other properties easily. For
instance, {\em ($\Ocal$1)} and {\em($\Ocal$2)} imply all other properties from Definition~\ref{def:OrientedMatroid} 
(see \cite{Ox11}, p.401).
 We give a quick overview over the most common cryptomorphic ways to define oriented matroids via signed circuits.
For full disclosure on these
less redundant definitions of oriented matroids, we refer the reader to \cite{BlVSWZ99}, \cite{BlV78}, and \cite{FoLa78}. 

In \cite{BlV78}, R.G.~Bland and M.~Las~Vergnas define oriented matroids to be pairs $(E,\Ccal)$ such that $\Ccal\subseteq \sigma E$
 has the properties {\em ($\Ccal$1)}, {\em ($\Ccal$2)}, {\em ($\Ccal$3)}, and the property
\begin{enumerate}\label{n:Ccal4p}
		\item[($\Ccal$4')] for all $X,Y\in \Ccal$ with $X\not= -Y$ and all $e\in X_+\cap Y_-$, there is some $Z\in\Ccal$ such that
				$e\notin Z_\pm$, $Z_+ \subseteq X_+\cup Y_+$ and $Z_- \subseteq X_-\cup Y_-$.
\end{enumerate}
	In Theorem~2.1 \cite{BlV78}, R.G.~Bland and M.~Las~Vergnas prove that if we assume {\em ($\Ccal$1)}, {\em ($\Ccal$2)}, {\em ($\Ccal$3)}, then 
	the properties {\em ($\Ccal$4)} and {\em ($\Ccal$4')} are equivalent. In Theorem~2.2 \cite{BlV78}, they prove
	that if $(E,\Ccal,\Ccal^\ast)$ is an oriented matroid as in our Definition~\ref{def:OrientedMatroid}, then $\Ccal$ uniquely
	determines $\Ccal^\ast$ and vice versa.
	Therefore, in order to define an oriented matroid on the ground set $E$, it suffices to determine $\Ccal$, and show that 
	{\em ($\Ccal$1)}, {\em ($\Ccal$2)}, {\em ($\Ccal$3)}, and {\em ($\Ccal$4')} hold. Since the underlying matroid $M(\Ocal)$
	is already uniquely determined by the supports of the elements of $\Ccal$, we can reconstruct the supports of $\Ccal^\ast$
	by examining the cocircuits of $M(\Ocal)$.
	In order to find the correct signatures of $D\in\Ccal^\ast$,
	we can set the sign $D(d)$ for an arbitrarily chosen $d\in D_\pm$ to $+1$, or to $-1$ in order to generate the corresponding 
	negation $-D$. If $D_\pm =\SET{d}$, we are done. 
	If $\left| D_\pm \right| > 1$, 
	then Lemma~\ref{lem:CircuitCocircuitIntersectInTwo}
	with respect to the dual matroid $M(\Ocal)^\ast$ 
	indicates that for every $c\in D_\pm\BSET{d}$, there is a circuit $C\in \Ccal$ such that $C_\pm \cap D_\pm = \SET{c,d}$.
	We let $D(c) = D(d)$ if $C(c) \not= C(d)$, and $D(c) = -D(d)$ if $C(c) = C(d)$. Clearly, this is the only possibility which
	yields $D\bot C$.

In \cite{FoLa78}, J.~Folkman and J.~Lawrence independently defined
 their version of oriented matroids to be triples $(E_\sigma,\Ccal,-)$ subject to basically 
the same properties as R.G.~Bland's and M.~Las~Vergnas's version of oriented matroids, but where the latter used explicit signs $+$ and $-$, the former used a fix-point free involution on $E_\sigma$ that maps $+e$ to $-e$ and vice versa. Thus the oriented matroids of
 J.~Folkman and J.~Lawrence correspond to {\em reorientation classes} of oriented matroids in this work.
\end{remark}

\begin{example}\PRFR{Apr 5th}
	Let $E$ be a finite set and $M=(E,2^E)$ be the free matroid on $E$. Then
	$$\Ocal = \left(E,\emptyset,\SET{\vphantom{E^E}\SET{e},\SET{-e}~\middle|~ e\in E}\right)$$ is the only possible oriented matroid with $M(\Ocal) = M$.
\end{example}

\noindent It is straightforward, that applying an $M(\Ocal)$-automorphism to the signed circuits $\Ccal$ and cocircuits $\Ccal^\ast$
of an oriented matroid $\Ocal$ yields another oriented matroid.

\needspace{4\baselineskip}
\begin{definition}\PRFR{Apr 5th}
	Let $\Ocal=(E,\Ccal,\Ccal^\ast)$ be an oriented matroid, and let $\phi\colon E\maparrow E$ be an $M(\Ocal)$-automorphism.
	The \deftext[relabeling of O by p@relabeling of $\Ocal$ by $\phi$]{relabeling of $\bm \Ocal$ by $\bm \phi$}
	shall be the triple \[ \phi[\Ocal] = \left(E,\Ccal_\phi, \Ccal_\phi^\ast\right) \]
	where \(
	\Ccal_\phi = \SET{C\circ \phi \in \sigma E ~\middle|~ C \in \Ccal}
	\) and \( \Ccal_\phi^\ast = \SET{C'\circ \phi \in \sigma E ~\middle|~ C'\in \Ccal^\ast} \).
\end{definition}

\begin{lemma}
	Let $\Ocal$ be an oriented matroid and $\phi$ and $M(\Ocal)$-automorphism. Then $\phi[\Ocal]$ is an oriented matroid.
\end{lemma}
\begin{proof}
	For $X\in\sigma E$, we have $(X\circ\phi)_+ = \phi[X_+]$, $(X\circ\phi)_- = \phi[X_-]$,
	$(X\circ\phi)_\pm = \phi[X_\pm]$, and $-(X\circ\phi) = (-X)\circ \phi$. As a consequence, for $C,D\in \sigma E$,
	we have that $C\bot D$ if and only if $C\circ \phi \bot D\circ \phi$ as well as $\sep(C\circ \phi,D\circ \phi) = \phi[\sep(C,D)]$.
	Furthermore, for $Y\subseteq E$,
	we have $Y\in \Ccal(M(\Ocal))$ if and only if $\phi[Y] \in \Ccal(M(\Ocal))$, as well as $Y\in \Ccal(M(\Ocal)^\ast)$
	if and only if $\phi[Y]\in \Ccal(M(\Ocal)^\ast)$.
	With these properties, it is straightforward yet tiresome to verify using the definition of $\phi[\Ocal]$,
	that the axioms for $\Ocal$ carry over to $\phi[\Ocal]$.
\end{proof}

\begin{definition}
	Let $\Ocal=(E,\Ccal,\Ccal^\ast)$ be an oriented matroid. The \deftext[dual oriented matroid]{dual oriented matroid of $\bm \Ocal$}
	is the triple $\Ocal^\ast = (E,\Ccal^\ast,\Ccal)$.\label{n:OcalBot}
\end{definition}

\needspace{3\baselineskip}
\begin{lemma}
	Let $\Ocal$ be an oriented matroid. Then $\Ocal^\ast$ is an oriented matroid, too.
\end{lemma}
\begin{proof}
	Observe that the axioms {\em ($\Ccal i$)} for $\Ocal$ are equivalent to the axioms {\em ($\Ccal^\ast i$)} for $\Ocal^\ast$,
	analogously {\em ($\Ccal^\ast i$)} for $\Ocal$ are equivalent to {\em ($\Ccal i$)} for $\Ocal^\ast$; where $i\in \SET{1,2,3,4}$. Furthermore,
	{\em ($\Ocal$1)} is symmetric in itself, so it holds for $\Ocal$ if and only if it holds for $\Ocal^\ast$. Moreover,
	every witness $M(\Ocal)$, that certifies {\em ($\Ocal$2)} for $\Ocal,$ yields a witness $\left( M(\Ocal) \right)^\ast$,
	that certifies {\em ($\Ocal$2)} for $\Ocal^\ast$, and vice versa. Therefore, the triple $\Ocal$ is an oriented matroid
	if and only if the triple $\Ocal^\ast$ is an oriented matroid, so we obtain that $\Ocal^\ast$ is an oriented matroid from the premises of this lemma.
\end{proof}

\needspace{6\baselineskip}
\begin{definition}\label{def:Omu}\PRFRC
	Let $E$ be a finite set and $\mu\in \R^{E\times \SET{1,2,\ldots,r}}$ be a matrix.
	The \deftext[oriented matroid represented by mu@oriented matroid represented by $\mu$]{oriented matroid represented by $\bm \mu$}
	is the uniquely determined oriented matroid \label{n:OcalMu}$\Ocal(\mu) = (E,\Ccal_\mu,\Ccal_\mu^\ast)$ where
	\[ \Ccal_\mu = \SET{C\in \Dcal_\mu ~\middle|~ \forall C'\in\Dcal_\mu \colon \,C'_\pm \subseteq C_\pm \Rightarrow C'_\pm = C_\pm} \]
	and 
	\[ \Dcal_\mu = \SET{E_\alpha \in \sigma E\BSET{\emptyset_{\sigma E}} ~~\middle|~~ \alpha\in \R^E,\,\alpha\not\equiv 0,\,\sum_{e\in E}\alpha(e)\cdot \mu_e = 0 } .\]
\end{definition}

\begin{lemma}\label{lem:OcalMu}\PRFRC
	Let $E$ be a finite set and $\mu\in \R^{E\times \SET{1,2,\ldots,r}}$ be a matrix. Then $\Ocal(\mu)$ is indeed an oriented matroid.
\end{lemma}
\begin{proof}\PRFRC
By Remark~\ref{rem:sufficientConditionOM}, it suffices to show that {\em ($\Ccal$1)}, {\em ($\Ccal$2)}, {\em ($\Ccal$3)}, and {\em ($\Ccal$4')} hold for $\Ccal_\mu$
 in Definition~\ref{def:Omu}.
%\noindent
{\em ($\Ccal$1)} is obvious from the construction.
Let $C\in \Ccal_\mu$, then there is a vector $\alpha\in \R^E$ such that
 $\sum_{e\in E}\alpha(e)\cdot \mu_e = 0$, thus there is $\alpha'\in \R^E$ with $\alpha'(e)=-\alpha(e)$
 such that  $$\sum_{e\in E}\alpha'(e)\cdot \mu_e = -\sum_{e\in E}\alpha(e)\cdot \mu_e = 0.$$ Clearly, $E_{\alpha'} = -E_\alpha$ and therefore $\left( E_\alpha \right)_\pm = \left( E_{\alpha'}  \right)_\pm$. Thus $-C\in\Ccal_\mu$, so
 {\em ($\Ccal$2)} holds.

 \bSep\PRFRC
 We give an indirect argument for {\em ($\Ccal$3)}: Let $X,Y\in \Ccal_\mu$ with $X_\pm = Y_\pm$ and
 \linebreak $Y\notin \SET{X,-X}$, such that
 $X_\pm$ is minimal in $\Ccal_\mu$ with respect to set-inclusion $\subseteq$. There is an element %\linebreak
$f\in \left( X_+\cap Y_+ \right)\cup \left( X_-\cap Y_- \right)$
and an element $f'\in \sep(X,Y)$. Now let $\alpha,\alpha'\in \R^E$ with $\sum_{e\in E}\alpha(e)\cdot\mu_e
=\sum_{e\in E}\alpha'(e)\cdot\mu_e = 0$ such that $E_\alpha = X$ and $E_{\alpha'} = Y$. Let $\beta\in\R^E$ where
 $$\beta(e) = \alpha(e) - \frac{\alpha(f)}{\alpha'(f)}\alpha'(e)$$ for all $e\in E$. Then
 \begin{align*}
 	\sum_{e\in E}\beta(e)\cdot \mu_e & = \sum_{e\in E}\left(\alpha(e) - \frac{\alpha(f)}{\alpha'(f)}\alpha'(e)  \right)\cdot \mu_e \\
 	& = \left( \sum_{e\in E}\alpha(e)\cdot\mu_e\right) - \frac{\alpha(f)}{\alpha'(f)}\left(\sum_{e\in E}\alpha'(e)\cdot\mu_e\right) \\
 	& = 0,
 \end{align*}
 with $\beta(f) = 0$ and $$\beta(f') = \alpha(f') - \frac{\alpha(f)}{\alpha'(f)}\cdot \alpha'(f) \not= 0$$
 since $f'\in \sep(E_\alpha,E_{\alpha'})$, therefore
  $\emptyset \not= \left( E_\beta \right)_\pm \subsetneq \left( E_\alpha \right)_\pm$ which contradicts the minimality 
 of $X_\pm$. Therefore our assumption must be wrong and $Y\in \SET{X,-X}$.

 \bSep\PRFRC
 The proof of {\em ($\Ccal$4')} is similar: Let $X,Y\in\Ccal_\mu$ with $X\not= -Y$, and let $f\in X_+\cap Y_-$. Again
  let $\alpha,\alpha'\in \R^E$ with $\sum_{e\in E}\alpha(e)\cdot\mu_e
=\sum_{e\in E}\alpha'(e)\cdot\mu_e = 0$ such that $E_\alpha = X$ and $E_{\alpha'} = Y$. Define $\beta\in \R^E$ as above,
then again $\sum_{e\in E}\beta(e)\cdot\mu_e = 0$ and $\beta(f) = 0$. Since $X\not= -Y$, and obviously $X\not= Y$,
we obtain that $X_\pm \not= Y_\pm$ from {\em ($\Ccal$3)}. Thus there is an element $g\in \left( X_\pm \cap Y_0  \right) \cup \left( X_0\cap Y_\pm \right)$. Since either $\alpha(g) = 0$ or $\alpha'(g) = 0$, we obtain that
 $$\beta(g) = \alpha(g) -\frac{\alpha(f)}{\alpha'(f)}\alpha'(g) \not= 0.$$
 %because both summands have the same sign and therefore cannot cancel.
 So $\emptyset_{\sigma E} \not= E_\beta \in \Dcal_\mu$. 
 Furthermore, for all $g\in \left( X_\pm \cap Y_0  \right)$,
 we have $\beta(g) = \alpha(g)$, and thus $E_\beta(g) = X(g)$.
 Also,
 for all $g\in \left(X_0\cap Y_\pm\right)$ we have $$\beta(g) = -\frac{\alpha(f)}{\alpha'(f)}\alpha'(g)$$ and 
 since $\sgn\left(-\frac{\alpha(f)}{\alpha'(f)}\right) = 1$, 
 we have $E_\beta(g) = Y(g)$. Finally, for all $g\in X_0\cap Y_0$, we clearly have $\beta(g) = 0$.
 Thus we found $Z=E_\beta \in \Dcal_\mu$ with $Z(f) = 0$, $Z_+ \subseteq X_+ \cup Y_+$, and $Z_- \subseteq X_- \cup Y_-$.
 We claim that there is some $Z'\in\Ccal_\mu$ with $Z' \subseteq_\sigma Z$, yielding the desired signed circuit for {\em ($\Ccal$4')}.
 We give a constructive argument for this claim.
 Assume that $Z\notin \Ccal_\mu$, then there is an $Z'\in \Ccal_\mu$ such that $Z'_\pm \subsetneq Z_\pm$.
 Let $\gamma\in \R^E$ with $\sum_{e\in E}\gamma(e)\cdot\mu_e = 0$ such that $E_\gamma = Z'$.
 Let $f\in Z'_\pm \cap Z_\pm$ such that $\left| \frac{\beta(f)}{\gamma(f)} \right|$ is minimal, thus for all $f'\in Z'_\pm \cap Z_\pm$ we have
 \[ \left| \beta(f')  \right| \geq \left| \frac{\beta(f)}{\gamma(f)} \gamma(f') \right|.\]
 Therefore if we let $\delta\in \R^E$ such that $\delta(e) = \beta(e) - \frac{\beta(f)}{\gamma(f)}\gamma(e)$,
 we have $\sum_{e\in E}\delta(e)\cdot \mu_e = 0$ and $\emptyset_{\sigma E} \not= E_\delta \subsetneq_\sigma E_\beta$,
 guaranteed by the choice of $f$.
 So we found some $Z'' = E_\delta \in \Dcal_\mu$ with $Z'' \subsetneq_\sigma Z$. % and $Z' \not\subseteq_\sigma Z''$. <-- stimmt, aber egal.
If $Z''\in \Ccal_\mu$ we are done, otherwise we continue the last construction where $Z''$ takes on the role of $Z$.
Since $E$ is finite and our construction
strictly reduces the cardinality of the support of the signed subset in question, we finally 
construct a signed subset $Z^{(2n)'}$ with minimal possible support, and therefore
we eventually find some $Z^{(2n)'}\in \Ccal_\mu$
with $Z^{(2n)'} \subseteq_\sigma Z$.
\end{proof}

\begin{corollary}\label{cor:MOmuEQMmu}\PRFRC
	Let $E,C$ be finite sets, and $\mu\in \R^{E\times C}$ a real matrix.
	Then $$M(\Ocal(\mu)) = M(\mu),$$ i.e. 
	the underlying matroid of the oriented matroid represented by $\mu$ is the matroid represented by $\mu$.
\end{corollary}
\begin{proof}\PRFRC
	Obvious from Definition~\ref{def:Mmu} and Definition~\ref{def:Omu}. 
\end{proof}

\begin{definition}\PRFRC
	Let $M=(E,\Ical)$ be a matroid. Then $M$ shall be called \deftext{orientable matroid},
	if there is an oriented matroid $\Ocal=(E,\Ccal,\Ccal^\ast)$, such that $M = M(\Ocal)$,
	or equivalently
	$\SET{C_\pm \mid C\in \Ccal} = \Ccal(M)$. Furthermore, every oriented matroid $\Ocal$ with this property
	shall be called \deftext[orientation of M@orientation of $M$]{orientation of $\bm M$}.
\end{definition}

\begin{corollary}\label{cor:Rorientable}\PRFRC
	Every matroid that can be represented over $\R$ is orientable.
\end{corollary}
\begin{proof}\PRFRC
	Follows from Lemma~\ref{lem:OcalMu} and Corollary~\ref{cor:MOmuEQMmu}.
\end{proof}

%begin{corollary}\label{cor:gammoidOrientable}\PRFRC
%	Every gammoid is orientable.
%\end{corollary}
%\begin{proof}\PRFRC
%Every gammoid is representable over $\R$ (Theorem~\ref{thm:gammoidOverR}), thus we may use Corollary~\ref{cor:Rorientable}.
%\end{proof}

\noindent Thus every gammoid is orientable (Lemma~\ref{lem:gammoidOrientable}).

\begin{definition}\PRFRC
	Let $\Ocal=(E,\Ccal,\Ccal^\ast)$ be an oriented matroid. We say that $\Ocal$ is \deftext[realizable oriented matroid]{realizable},
	if there is a finite set $Q$ and a matrix $\mu\in \R^{E\times Q}$, such that $\Ocal = \Ocal(\mu)$.
	If there is no such matrix, we shall call $\Ocal$ \deftext[non-realizable oriented matroid]{non-realizable}.
\end{definition}

\begin{remark} \PRFRC
Of course, not every oriented matroid arises in this way from a matrix over a linearly ordered field. 
 H.~Miyata, S.~Moriyama, and K.~Fukuda published
a listing of all non-realizable oriented matroids\footnote{See: \url{http://www-imai.is.s.u-tokyo.ac.jp/~hmiyata/oriented_matroids/} \cite{OM2}} of rank $4$ with $\left| E \right| = 8$, and of rank $3$ and rank $6$
 with $\left| E \right| = 9$; the results are based on the oriented matroid database\footnote{See: \url{http://www.om.math.ethz.ch/} \cite{OM1}} 
 by L.~Finschi and K.~Fukuda.
 \end{remark}

\begin{example}[\cite{BlVSWZ99}, p.20]\label{ex:NonRepresentableOrientationOfUniform}\label{ex:UniformNonRealizable}\index{RS(8)@$\mathtt{RS}(8)$}\PRFR{Apr 5th}
 The oriented matroid we want to present now has been named $\mathtt{RS}(8)$. 
 It is an orientation of the rank $4$ uniform matroid with $8$ elements, and therefore clearly an orientation of a gammoid.
 It has $2\cdot \binom{8}{5} = 112$ signed circuits as well as $112$ signed cocircuits.
 Since these signed subsets come in pairs $X$ and $-X$, we only have to list half of them.
 Let $E=\SET{1,2,\ldots,8}$. Then $\mathtt{RS}(8) = (E,\Ccal, \Ccal^\ast)$ where
 %\allowdisplaybreaks
 \begin{align*}
 	\Ccal = \pm \{ & 
\SET{ \hphantom{-} 1 , - 2 , - 3 , \hphantom{-} 4 , - 5 },\,
\SET{ \hphantom{-} 1 , \hphantom{-} 2 , - 3 , \hphantom{-} 4 , - 6 },\,
\SET{ \hphantom{-} 1 , - 2 , \hphantom{-} 3 , \hphantom{-} 4 , - 7 },\,
 \\ &
\SET{ \hphantom{-} 1 , - 2 , - 3 , \hphantom{-} 4 , - 8 },\,
\SET{ - 1 , \hphantom{-} 2 , - 3 , \hphantom{-} 5 , - 6 },\,
\SET{ - 1 , \hphantom{-} 2 , \hphantom{-} 3 , \hphantom{-} 5 , - 7 },\,
 \\ &
\SET{ - 1 , \hphantom{-} 2 , \hphantom{-} 3 , \hphantom{-} 5 , - 8 },\,
\SET{ \hphantom{-} 1 , - 2 , \hphantom{-} 3 , \hphantom{-} 6 , - 7 },\,
\SET{ - 1 , - 2 , \hphantom{-} 3 , \hphantom{-} 6 , - 8 },\,
 \\ &
\SET{ - 1 , \hphantom{-} 2 , - 3 , \hphantom{-} 7 , - 8 },\,
\SET{ - 1 , \hphantom{-} 2 , - 4 , \hphantom{-} 5 , - 6 },\,
\SET{ \hphantom{-} 1 , - 2 , \hphantom{-} 4 , - 5 , - 7 },\,
 \\ &
\SET{ - 1 , - 2 , \hphantom{-} 4 , \hphantom{-} 5 , - 8 },\,
\SET{ \hphantom{-} 1 , \hphantom{-} 2 , \hphantom{-} 4 , - 6 , - 7 },\,
\SET{ - 1 , - 2 , \hphantom{-} 4 , \hphantom{-} 6 , - 8 },\,
 \\ &
\SET{ - 1 , \hphantom{-} 2 , - 4 , \hphantom{-} 7 , \hphantom{-} 8 },\,
\SET{ - 1 , \hphantom{-} 2 , \hphantom{-} 5 , - 6 , - 7 },\,
\SET{ - 1 , \hphantom{-} 2 , \hphantom{-} 5 , - 6 , - 8 },\,
 \\ &
\SET{ \hphantom{-} 1 , - 2 , - 5 , - 7 , \hphantom{-} 8 },\,
\SET{ \hphantom{-} 1 , \hphantom{-} 2 , - 6 , - 7 , \hphantom{-} 8 },\,
\SET{ - 1 , \hphantom{-} 3 , - 4 , \hphantom{-} 5 , \hphantom{-} 6 },\,
 \\ &
\SET{ \hphantom{-} 1 , - 3 , - 4 , - 5 , \hphantom{-} 7 },\,
\SET{ \hphantom{-} 1 , - 3 , - 4 , - 5 , \hphantom{-} 8 },\,
\SET{ \hphantom{-} 1 , \hphantom{-} 3 , \hphantom{-} 4 , - 6 , - 7 },\,
 \\ &
\SET{ \hphantom{-} 1 , - 3 , \hphantom{-} 4 , - 6 , - 8 },\,
\SET{ - 1 , \hphantom{-} 3 , - 4 , - 7 , \hphantom{-} 8 },\,
\SET{ - 1 , \hphantom{-} 3 , \hphantom{-} 5 , \hphantom{-} 6 , - 7 },\,
 \\ &
\SET{ - 1 , \hphantom{-} 3 , \hphantom{-} 5 , \hphantom{-} 6 , - 8 },\,
\SET{ \hphantom{-} 1 , - 3 , - 5 , \hphantom{-} 7 , - 8 },\,
\SET{ \hphantom{-} 1 , \hphantom{-} 3 , - 6 , - 7 , \hphantom{-} 8 },\,
 \\ &
\SET{ \hphantom{-} 1 , \hphantom{-} 4 , - 5 , - 6 , - 7 },\,
\SET{ - 1 , \hphantom{-} 4 , \hphantom{-} 5 , - 6 , - 8 },\,
\SET{ \hphantom{-} 1 , - 4 , - 5 , - 7 , \hphantom{-} 8 },\,
 \\ &
\SET{ \hphantom{-} 1 , \hphantom{-} 4 , - 6 , - 7 , - 8 },\,
\SET{ - 1 , \hphantom{-} 5 , \hphantom{-} 6 , \hphantom{-} 7 , - 8 },\,
\SET{ - 2 , \hphantom{-} 3 , - 4 , - 5 , \hphantom{-} 6 },\,
 \\ &
\SET{ \hphantom{-} 2 , - 3 , - 4 , - 5 , \hphantom{-} 7 },\,
\SET{ \hphantom{-} 2 , \hphantom{-} 3 , - 4 , - 5 , \hphantom{-} 8 },\,
\SET{ \hphantom{-} 2 , - 3 , \hphantom{-} 4 , - 6 , \hphantom{-} 7 },\,
 \\ &
\SET{ \hphantom{-} 2 , \hphantom{-} 3 , - 4 , - 6 , \hphantom{-} 8 },\,
\SET{ - 2 , \hphantom{-} 3 , - 4 , - 7 , \hphantom{-} 8 },\,
\SET{ - 2 , \hphantom{-} 3 , \hphantom{-} 5 , \hphantom{-} 6 , - 7 },\,
 \\ &
\SET{ - 2 , \hphantom{-} 3 , - 5 , \hphantom{-} 6 , - 8 },\,
\SET{ \hphantom{-} 2 , - 3 , - 5 , \hphantom{-} 7 , - 8 },\,
\SET{ \hphantom{-} 2 , - 3 , - 6 , \hphantom{-} 7 , \hphantom{-} 8 },\,
 \\ &
\SET{ \hphantom{-} 2 , \hphantom{-} 4 , \hphantom{-} 5 , - 6 , - 7 },\,
\SET{ - 2 , \hphantom{-} 4 , - 5 , \hphantom{-} 6 , - 8 },\,
\SET{ \hphantom{-} 2 , - 4 , - 5 , \hphantom{-} 7 , \hphantom{-} 8 },\,
 \\ &
\SET{ \hphantom{-} 2 , - 4 , - 6 , \hphantom{-} 7 , \hphantom{-} 8 },\,
\SET{ - 2 , - 5 , \hphantom{-} 6 , \hphantom{-} 7 , - 8 },\,
\SET{ \hphantom{-} 3 , \hphantom{-} 4 , \hphantom{-} 5 , - 6 , - 7 },\,
 \\ &
\SET{ - 3 , \hphantom{-} 4 , \hphantom{-} 5 , - 6 , - 8 },\,
\SET{ \hphantom{-} 3 , - 4 , - 5 , - 7 , \hphantom{-} 8 },\,
\SET{ \hphantom{-} 3 , - 4 , - 6 , - 7 , \hphantom{-} 8 },\,
 \\ &
\SET{ - 3 , - 5 , \hphantom{-} 6 , \hphantom{-} 7 , - 8 },\,
\SET{ - 4 , - 5 , \hphantom{-} 6 , \hphantom{-} 7 , \hphantom{-} 8 }
 		\}
 \end{align*}
and where $\Ccal^\ast$ is uniquely determined by $\Ccal$.\footnote
{There has not been a complete example of an oriented matroid in this work so far, thus we present $\Ccal^\ast$
in order for the reader to check their understanding of the signed cocircuits of an oriented matroid with a non-trivial example.
%\allowdisplaybreaks
 \begin{align*}
 	\Ccal^\ast = \pm \{ &
\SET{ - 1 , \hphantom{-} 2 , \hphantom{-} 3 , \hphantom{-} 4 , - 5 },\,
\SET{ - 1 , - 2 , - 3 , - 4 , - 6 },\,
\SET{ \hphantom{-} 1 , \hphantom{-} 2 , \hphantom{-} 3 , \hphantom{-} 4 , \hphantom{-} 7 },\,
\SET{ - 1 , - 2 , - 3 , \hphantom{-} 4 , \hphantom{-} 8 },\,
 \\ &
\SET{ - 1 , - 2 , - 3 , - 5 , - 6 },\,
\SET{ - 1 , \hphantom{-} 2 , - 3 , - 5 , - 7 },\,
\SET{ - 1 , \hphantom{-} 2 , \hphantom{-} 3 , - 5 , - 8 },\,
\SET{ \hphantom{-} 1 , - 2 , \hphantom{-} 3 , - 6 , \hphantom{-} 7 },\,
 \\ &
\SET{ - 1 , - 2 , - 3 , - 6 , \hphantom{-} 8 },\,
\SET{ \hphantom{-} 1 , \hphantom{-} 2 , \hphantom{-} 3 , \hphantom{-} 7 , - 8 },\,
\SET{ \hphantom{-} 1 , \hphantom{-} 2 , - 4 , \hphantom{-} 5 , \hphantom{-} 6 },\,
\SET{ \hphantom{-} 1 , - 2 , - 4 , \hphantom{-} 5 , \hphantom{-} 7 },\,
 \\ &
\SET{ \hphantom{-} 1 , \hphantom{-} 2 , - 4 , \hphantom{-} 5 , - 8 },\,
\SET{ - 1 , - 2 , - 4 , - 6 , \hphantom{-} 7 },\,
\SET{ - 1 , - 2 , - 4 , - 6 , - 8 },\,
\SET{ \hphantom{-} 1 , \hphantom{-} 2 , - 4 , - 7 , - 8 },\,
 \\ &
\SET{ - 1 , - 2 , - 5 , - 6 , \hphantom{-} 7 },\,
\SET{ \hphantom{-} 1 , \hphantom{-} 2 , \hphantom{-} 5 , \hphantom{-} 6 , \hphantom{-} 8 },\,
\SET{ \hphantom{-} 1 , - 2 , \hphantom{-} 5 , \hphantom{-} 7 , \hphantom{-} 8 },\,
\SET{ \hphantom{-} 1 , \hphantom{-} 2 , \hphantom{-} 6 , - 7 , - 8 },\,
 \\ &
\SET{ - 1 , \hphantom{-} 3 , \hphantom{-} 4 , - 5 , - 6 },\,
\SET{ \hphantom{-} 1 , \hphantom{-} 3 , \hphantom{-} 4 , \hphantom{-} 5 , \hphantom{-} 7 },\,
\SET{ - 1 , \hphantom{-} 3 , \hphantom{-} 4 , - 5 , \hphantom{-} 8 },\,
\SET{ \hphantom{-} 1 , \hphantom{-} 3 , \hphantom{-} 4 , - 6 , \hphantom{-} 7 },\,
 \\ &
\SET{ - 1 , - 3 , \hphantom{-} 4 , \hphantom{-} 6 , \hphantom{-} 8 },\,
\SET{ \hphantom{-} 1 , - 3 , - 4 , - 7 , - 8 },\,
\SET{ - 1 , - 3 , - 5 , - 6 , - 7 },\,
\SET{ - 1 , \hphantom{-} 3 , - 5 , - 6 , - 8 },\,
 \\ &
\SET{ \hphantom{-} 1 , \hphantom{-} 3 , \hphantom{-} 5 , \hphantom{-} 7 , - 8 },\,
\SET{ - 1 , - 3 , \hphantom{-} 6 , - 7 , \hphantom{-} 8 },\,
\SET{ \hphantom{-} 1 , - 4 , \hphantom{-} 5 , \hphantom{-} 6 , \hphantom{-} 7 },\,
\SET{ - 1 , \hphantom{-} 4 , - 5 , \hphantom{-} 6 , \hphantom{-} 8 },\,
 \\ &
\SET{ \hphantom{-} 1 , - 4 , \hphantom{-} 5 , \hphantom{-} 7 , - 8 },\,
\SET{ \hphantom{-} 1 , - 4 , - 6 , - 7 , - 8 },\,
\SET{ \hphantom{-} 1 , \hphantom{-} 5 , \hphantom{-} 6 , \hphantom{-} 7 , \hphantom{-} 8 },\,
\SET{ - 2 , - 3 , - 4 , \hphantom{-} 5 , - 6 },\,
 \\ &
\SET{ \hphantom{-} 2 , \hphantom{-} 3 , \hphantom{-} 4 , - 5 , \hphantom{-} 7 },\,
\SET{ - 2 , - 3 , \hphantom{-} 4 , \hphantom{-} 5 , \hphantom{-} 8 },\,
\SET{ - 2 , \hphantom{-} 3 , - 4 , - 6 , \hphantom{-} 7 },\,
\SET{ \hphantom{-} 2 , - 3 , \hphantom{-} 4 , \hphantom{-} 6 , \hphantom{-} 8 },\,
 \\ &
\SET{ \hphantom{-} 2 , \hphantom{-} 3 , \hphantom{-} 4 , \hphantom{-} 7 , \hphantom{-} 8 },\,
\SET{ - 2 , \hphantom{-} 3 , - 5 , - 6 , \hphantom{-} 7 },\,
\SET{ - 2 , - 3 , \hphantom{-} 5 , - 6 , \hphantom{-} 8 },\,
\SET{ - 2 , - 3 , \hphantom{-} 5 , - 7 , \hphantom{-} 8 },\,
 \\ &
\SET{ \hphantom{-} 2 , - 3 , \hphantom{-} 6 , - 7 , - 8 },\,
\SET{ - 2 , - 4 , \hphantom{-} 5 , - 6 , \hphantom{-} 7 },\,
\SET{ \hphantom{-} 2 , \hphantom{-} 4 , - 5 , \hphantom{-} 6 , \hphantom{-} 8 },\,
\SET{ \hphantom{-} 2 , - 4 , - 5 , - 7 , - 8 },\,
 \\ &
\SET{ \hphantom{-} 2 , \hphantom{-} 4 , \hphantom{-} 6 , \hphantom{-} 7 , \hphantom{-} 8 },\,
\SET{ - 2 , \hphantom{-} 5 , - 6 , \hphantom{-} 7 , \hphantom{-} 8 },\,
\SET{ - 3 , - 4 , \hphantom{-} 5 , \hphantom{-} 6 , - 7 },\,
\SET{ - 3 , \hphantom{-} 4 , \hphantom{-} 5 , \hphantom{-} 6 , \hphantom{-} 8 },\,
 \\ &
\SET{ - 3 , - 4 , - 5 , - 7 , - 8 },\,
\SET{ - 3 , - 4 , \hphantom{-} 6 , - 7 , - 8 },\,
\SET{ \hphantom{-} 3 , - 5 , - 6 , \hphantom{-} 7 , - 8 },\,
\SET{ - 4 , - 5 , - 6 , - 7 , - 8 } \}.
 \end{align*}}
%
% \noindent
	Proposition~1.5.1 in \cite{BlVSWZ99} states that the oriented matroid $\mathtt{RS}(8)$ 
	is a non-realizable orientation of the 
	underlying uniform matroid, thus we may expect gammoids to have non-realizable orientations. 
	For the full proof, we refer the reader to p.~23 in \cite{BlVSWZ99}. 
	The idea of the proof is the following: Assume that $\mathtt{RS}(8)$ is realizable, 
	then there is a matrix $\mu\in \R^{\SET{1,2,\ldots, 8}\times\SET{1,2,3,4}}$ 
	such that $\mu\restrict\SET{1,2,3,4}\times\SET{1,2,3,4}$ is an identity matrix. This leaves us with a variable matrix
	$\mu\restrict \SET{5,6,7,8}\times \SET{1,2,3,4}$, for which we would have to find values that yield the 
	correct signed circuits of $\mathtt{RS}(8)$. The signed circuits of $\mathtt{RS}(8)$ can be translated to
	strict inequalities that $\mu$ must obey. For instance, the signed circuit $\SET{1,-2,-3,4,-5}$ states that
	$\mu_5 = \alpha \mu_1 - \beta \mu_2 - \gamma \mu_3 + \delta \mu_4$ must have a solution with $\alpha,\beta,\gamma,\delta > 0$. 
	In this particularly easy case, we obtain the inequalities $\mu(5,1) > 0$, $\mu(5,2) < 0$, $\mu(5,3) < 0$, $\mu(5,4) > 0$.\footnote{In general, the strict inequalities derived are not linear, as Cramer's rule yields polynomial terms eliminating the coefficients of the non-trivial linear combinations of the zero that correspond to signed circuits of $\Ocal$.}
	 The proof of non-realizability is completed by the observation that the constructed
	system of inequalities has no solutions, therefore there is no matrix $\mu$ with $\Ocal(\mu)=\mathtt{RS}(8)$, and $\mathtt{RS}(8)$ is non-realizable.
\end{example}

\PRFRC
\noindent For every realizable oriented matroid of the form $\Ocal(\mu)$, we obtain another realizable oriented matroid $\Ocal(\mu')$ 
where $\mu'$ is obtained from $\mu$
by multiplying an arbitrary set of rows with $-1$.
 Clearly, for the underlying matroids we have $M(\mu) = M(\mu')$. 
 It is easy to see that the next definition carries this operation over to all oriented matroids (\cite{BlVSWZ99}, p.3).

\begin{definition}\label{def:Cflip}\PRFRC
 	Let $E$ be a set, $X\subseteq E$, and $C\in \sigma E$ a signed subset of $E$.
 	The \deftext[flip of a signed subset]{$\bm X$-flip of $\bm C$} is defined to be
 	the signed subset\label{n:Xflip}
 	\[ C_{-X} \colon E\maparrow \SET{-1, 0, 1},\quad e\mapsto \begin{cases} -C(e)&\quad \text{if~} e\in X,\\
 																			\hphantom{-}C(e)&\quad\text{otherwise.} \end{cases} \]
\end{definition}

\begin{definition}\PRFRC
	Let $\Ocal=(E,\Ccal,\Ccal^\ast)$ be an oriented matroid, and let $X\subseteq E$.
	The \deftext[flip reorientation of O@flip reorientation of $\Ocal$]{$\bm X$-flip reorientation of $\bm \Ocal$}
	is the triple $\Ocal_{-X} = (E,\Ccal_{-X},\Ccal_{-X}^\ast)$ where
	 $$\Ccal_{-X} = \SET{C_{-X}~\middle|~\vphantom{C'} C\in\Ccal} \quad\text{and}\quad
	\Ccal_{-X}^\ast = \SET{C'_{-X}~\middle|~ C'\in\Ccal^\ast}.$$
	Let $\Ocal'$ also be an oriented matroid. We say that 
	$\Ocal'$ is a \deftext[reorientation of O@reorientation of $\Ocal$]{reorientation of $\bm \Ocal$}, if
	there is a subset $X\subseteq E$ with $\Ocal' = \Ocal_{-X}$.
\end{definition}

\begin{lemma}\label{lem:reorientationIsOM}\PRFRC
	Let $\Ocal=(E,\Ccal)$ be an oriented matroid, and let $X\subseteq E$. Then $\Ocal_{-X}$ is an oriented matroid.
\end{lemma}
\begin{proof}\PRFRC
	For every $X\subseteq E$ and $C\in \sigma E$, it is clear from Definition~\ref{def:Cflip}, that
	\linebreak $\left( C_{-X} \right)_{-X} = C$,
	therefore the map $\phi_X \colon \sigma E\maparrow \sigma E, C\mapsto C_{-X}$ is an involution on $\sigma E$.
	Since $\left( \emptyset_{\sigma E}  \right)_{-X} =  \emptyset_{\sigma E}$, we obtain that $\emptyset_{\sigma E} \notin \Ccal_{-X}$ from $\emptyset_{\sigma E} \notin \Ccal$, thus {\em ($\Ccal$1)} holds.  Furthermore, for any $C\in \sigma E$
	we have
	$\left( C_{-X}  \right)_+ = \left( C_+ \BS X \right) \cup \left( C_- \cap X \right)$ and
	$\left( C_{-X}  \right)_- = \left( C_- \BS X \right) \cup \left( C_+ \cap X \right)$. In particular we have for $C,D\in \sigma E$ that
	$C = -D$ if and only if $C_{-X} = -D_{-X}$.  We also have $C\bot D$ if and only $C_{-X}\bot D_{-X}$: 
%	\linebreak 
	Case {\em (ii)} of Definition~\ref{def:XorthoY} is oblivious of any sign flips in $C$ and $D$ since $C_0 = \left( C_{-X}  \right)_0$
	and $D_0 = \left( D_{-X} \right)_0$; whereas in case {\em (i)} the passage from $C$ and $D$ to $C_{-X}$ and $D_{-X}$, respectively,
	introduces an even amount of sign flips. 
	  So $C(e)D(e) = C_{-X}(e)D_{-X}(e)$ and $C(f)D(f) = C_{-X}(f)D_{-X}(f)$.
	 Therefore {\em ($\Ccal$2)}, {\em ($\Ccal$3)}, and {\em ($\Ccal$4)} carry over from $\Ccal$ to $\Ccal_{-X}$.
	 Since $\SET{C_\pm ~\middle|~ C\in \Ccal} = \SET{C_\pm~\middle|~ C\in \Ccal_{-X}}$,
	 $\SET{C'_\pm ~\middle|~ C'\in \Ccal^\ast} = \SET{C_\pm'~\middle|~ C'\in \Ccal_{-X}^\ast}$,
	 and for all $C\in\Ccal_{-X}$ and $D\in \Ccal_{-X}^\ast$, we have $C\bot D$; 
	  we obtain that $\Ccal_{-X}^\ast$ is
	 indeed the unique family of signed cocircuits of the oriented matroid on $E$ with the family of signed circuits $\Ccal_{-X}$,
	 thus $\Ocal_{-X}$ is an oriented matroid (Remark~\ref{rem:sufficientConditionOM}).
\end{proof}

\begin{corollary}\PRFRC
	Let $\Ocal=(E,\Ccal,\Ccal^\ast)$ be an oriented matroid, and let $X\subseteq E$. Then $M(\Ocal_{-X}) = M(\Ocal)$.
\end{corollary}
\begin{proof}\PRFRC
	For $C\in \sigma E$, we have $C_\pm = \left( C_{-X} \right)_\pm$.
\end{proof}

\begin{definition}\PRFRC
	Let $\Ocal=(E,\Ccal,\Ccal^\ast)$ be an oriented matroid.
	The \deftext[reorientation class of O@reorientation class of $\Ocal$]{reorientation class of $\bm \Ocal$}
	is defined to be\label{n:reorientationclass}
	\( [\Ocal] = \SET{\Ocal_{-X}\mid X\subseteq E} .\)
\end{definition}

\begin{example}\label{ex:nonStrictGammoidOrientations}\PRFR{Apr 5th}
	%\allowdisplaybreaks
	According to L.~Finschi's database \cite{OM1}, the gammoid from Example~\ref{ex:nonStrictGammoid} has exactly one equivalence class
	with respect to relabeling and reorientation of oriented matroids.
	 Thus it is easy to check that it has two reorientation classes, $\Ocal_1 = (E,\Ccal_1,\Ccal_1^\ast)$ and $\Ocal_2 = (E,\Ccal_2,\Ccal_2^\ast)$
	 where
\begin{align*}
 \Ccal_1 = \pm \{ &
\SET{ a , b , - c , e } ,
\SET{ a , b , - d , - f } ,
\SET{ b , - c , - d , g } ,
\SET{ d , e , f , - g } ,
 \\ &
\SET{ - a , \hphantom{-} b , - c , \hphantom{-} f , \hphantom{-} g } ,
\SET{ - a , - c , \hphantom{-} d , \hphantom{-} e , \hphantom{-} f } ,
\SET{ - a , - c , \hphantom{-} d , \hphantom{-} f , \hphantom{-} g } ,
 \\ &
\SET{ \hphantom{-} a , \hphantom{-} b , \hphantom{-} d , \hphantom{-} e , - g } ,
\SET{ \hphantom{-} a , \hphantom{-} b , \hphantom{-} e , - f , - g } ,
\SET{ \hphantom{-} a , \hphantom{-} c , \hphantom{-} d , \hphantom{-} e , - g } ,
 \\ &
\SET{ \hphantom{-} a , \hphantom{-} c , \hphantom{-} e , - f , - g } ,
\SET{ \hphantom{-} b , - c , \hphantom{-} d , \hphantom{-} e , \hphantom{-} f } ,
\SET{ \hphantom{-} b , - c , \hphantom{-} e , \hphantom{-} f , \hphantom{-} g } \}
\end{align*}
and
\begin{align*}
 \Ccal_2 = \pm \{ &
\SET{ - a , b , c , e } ,
\SET{ a , - b , d , - f } ,
\SET{ b , c , - d , - g } ,
\SET{ d , e , - f , g } ,
 \\ &
\SET{ - a , \hphantom{-} b , - c , \hphantom{-} f , \hphantom{-} g } ,
\SET{ - a , \hphantom{-} b , \hphantom{-} d , \hphantom{-} e , \hphantom{-} g } ,
\SET{ - a , \hphantom{-} b , \hphantom{-} e , \hphantom{-} f , \hphantom{-} g } ,
 \\ &
\SET{ - a , - c , \hphantom{-} d , \hphantom{-} e , \hphantom{-} g } ,
\SET{ - a , - c , \hphantom{-} d , \hphantom{-} f , \hphantom{-} g } ,
\SET{ \hphantom{-} a , \hphantom{-} c , \hphantom{-} d , \hphantom{-} e , - f } ,
 \\ &
\SET{ \hphantom{-} a , \hphantom{-} c , \hphantom{-} e , - f , - g } ,
\SET{ \hphantom{-} b , \hphantom{-} c , \hphantom{-} d , \hphantom{-} e , - f } ,
\SET{ \hphantom{-} b , \hphantom{-} c , \hphantom{-} e , - f , - g } \}
\end{align*}
Here, $[\Ocal_2] = [\phi[\Ocal_1]]$ where $\phi = (ac)(fg)$ is a corresponding relabeling.
\end{example}

\begin{lemma}\PRFRC
	Let $E,Y$ be finite sets, $\mu\in \R^{E\times Y}$, and $X\subseteq E$.
	Let $\nu\in \R^{E\times Y}$ be the matrix where for every $e\in E$ and $y\in Y$,
	\[ \nu(e,y) = \begin{cases} -\mu(e,y)& \quad \text{if~} e\in X,\\
								\hphantom{-}\mu(e,y)&\quad\text{otherwise.}
	\end{cases} \]
	Then $\left( \Ocal(\mu) \right)_{-X} = \Ocal(\nu)$.
\end{lemma}

\begin{proof}\PRFRC
	Let $\alpha\in \R^E$. Let $\beta\in \R^E$ be defined such that
	\[ \beta(e) = \begin{cases} -\alpha(e)&\quad \text{if~}e \in X,\\
						\hphantom{-}\alpha(e)& \quad \text{otherwise.}
					\end{cases} \]
	Clearly,
	$(E_\alpha)_{-X} =  E_\beta$ and \( \sum_{e\in E} \beta(e)\cdot \nu(e) = \sum_{e\in E}\alpha(e)\cdot \mu(e)\).
	Thus $\sum_{e\in E}\alpha(e)\cdot \mu(e) = 0$ if and only if $\sum_{e\in E}\beta(e)\cdot \nu(e) = 0$,
	and consequently $E_\alpha \in \Ccal_\mu$ if and only if $E_\beta \in \Ccal_\nu$.
	Therefore $\left( \Ocal(\mu) \right)_{-X} = \Ocal(\nu)$.
\end{proof}

\begin{corollary}\PRFRC
	Let $\Ocal=(E,\Ccal,\Ccal^\ast)$ be an oriented matroid, and let $X\subseteq E$. Then $\Ocal$ is realizable if and only if $\Ocal_{-X}$ is realizable.
\end{corollary}

%\clearpage
% -*- root: ../thesis.tex -*-

%\subsection{Minors}

\begin{definition}\label{def:OMrestriction}\PRFRC
	Let $\Ocal=(E,\Ccal,\Ccal^\ast)$ be a oriented matroid, $R\subseteq E$.
	The \deftext[restriction of O to R@restriction of $\Ocal$ to $R$]{restriction of $\bm \Ocal$ to $\bm R$}
	is the triple\label{n:OrestrictR}
	\( \Ocal \restrict R = \left(R,\Ccal_R,\Ccal_R^\ast\right) \)
	where
	\[ \Ccal_R = \SET{C\in \sigma R ~\middle|~ \exists D\in \Ccal\colon\, D_\pm \subseteq R \text{~s.t.~} D\restrict_R = C}\]
	and
	\[ \Ccal_R^\ast = \SET{C'\in \Dcal_R^\ast ~\middle|~ \nexists D'\in \Dcal_R^\ast\colon\, D'_\pm \subsetneq C'_\pm},\]
	and where
	\[ \Dcal_R^\ast = \SET{C'\in \sigma R \BSET{\emptyset_{\sigma R}} ~\middle|~ \exists D'\in \Ccal^\ast \colon \, D'\restrict_R = C'} .\]
	\label{def:OMcontraction}
	Let $Q\subseteq E$.
	The \deftext[contraction of O to Q@contraction of $\Ocal$ to $Q$]{contraction of $\bm \Ocal$ to $\bm Q$}
	is the triple\label{n:OcontractQ}
	\( \Ocal \contract Q = \left(Q,\Ccal_{'Q},\Ccal^\ast_{'Q}\right) \)
	where 
	\[ \Ccal^\ast_{'Q} = \SET{C'\in \sigma R ~\middle|~ \exists D'\in \Ccal^\ast\colon\, D'_\pm \subseteq R \text{~s.t.~} D'\restrict_R = C'}\]
	and
	\[ \Ccal_{'Q} = \SET{C\in \Dcal_{'Q} ~\middle|~ \nexists D\in \Dcal_{'Q}\colon\, D_\pm \subsetneq C_\pm},\]
	and where
	\[ \Dcal_{'Q} = \SET{C\in \sigma R \BSET{\emptyset_{\sigma R}} ~\middle|~ \exists D\in \Ccal \colon \, D\restrict_R = C} . \qedhere\]
\end{definition}

\begin{lemma}\label{lem:OMminors}\PRFRC
Let $\Ocal=(E,\Ccal,\Ccal^\ast)$ be a oriented matroid, $X\subseteq E$.
Then $\Ocal\restrict X$ and $\Ocal\contract X$ are oriented matroids,
and further \[
	\left( \Ocal^\ast \restrict X \right)^\ast = \Ocal\contract X
	\quad\text{as well as}\quad
	\left( \Ocal^\ast \contract X \right)^\ast = \Ocal\restrict X
\]
holds.
\end{lemma}

\noindent For a proof, please refer to Propositions 3.3.1 and 3.3.2 (\cite{BlVSWZ99}, p.110) in {\em Oriented Matroids} by A.~Björner, M.~Las~Vergnas, B.~Sturmfels, N.~White,~and G.~Ziegler.

%\clearpage
% -*- root: ../thesis.tex -*-

\clearpage
% -*- root: ../thesis.tex -*-

\section{Colorings}

\PRFR{Mar 7th}
The notion of colorings used in this work originates from {\em Antisymmetric Flows in Matroids} 
by J.~Nešetřil and W.~Hochstättler \cite{HN06}. A recommended source for the properties and bearings of this notion
is R.~Nickel's thesis {\em Flows and Colorings in Oriented Matroids} \cite{Ni12}.

\begin{remark}\PRFR{Mar 7th}
The first appearance of a notion of a coloring of general oriented matroids
can be tracked down to the   paper {\em On $(k,d)$-Colorings and Fractional Nowhere-Zero Flows} 
by L.A.~Goddyn, M.~Tarsi, and C.-Q.~Zhang \cite{GTZ98},
who define the {\em star flow index} of an reorientation class $[\Ocal']$ of oriented matroids
to be \[
	\xi^\ast\left( \left[\Ocal'\right] \right) = \min_{\Ocal \in [\Ocal']} \max_{D\in \Ccal^\ast_\Ocal} \frac{\left| D_\pm \right|}{\left| D_+ \right|}
\]
where $\Ccal^\ast_\Ocal$ denotes the family of signed cocircuits of $\Ocal$. The star flow index is closely related to
the chromatic number of a graph $G$ through a result of J.G.~Minty \cite{Mi62}:
 If $G=(V,E)$ is a graph, and $\mu\in \R^{E\times V}$ is the signed edge-vertex-incidence matrix of an orientation of the edges of $G$, then for $\Ocal' = \left( \Ocal(\mu) \right)^\ast$ we have $\lceil \xi^\ast([\Ocal'])\rceil = \chi(G)$ where $\chi(G)$ is the 
 well-known
 chromatic number of $G$.
 {\em{We would like to point out that this is \textbf{not the chromatic number} of oriented matroids that \textbf{we are concerned with} in this work.}}\footnote{For further details on the differences between the oriented flow number of L.A.~Goddyn, M.~Tarsi, and C.-Q.~Zhang \cite{GTZ98} and
 the chromatic number of J.~Nešetřil and W.~Hochstättler \cite{HN06}, see \cite{Ni12}, pp.~98f.}
 \end{remark}

%
%the albeit from our perspective it is defined
%with regard to the dual oriented matroid.
%
%.developed by \remred{TODO...} L.A.~Goddyn, W.~Hochstättler, J.~Nešetřil, N.A.~Neudauer, R.~Nickel.
%
%\remred{CITATIONS!}

\begin{definition}\PRFRC
	Let $E$ be a set.
	A \deftext[signed multiset]{signed multi-subset of $\bm E$} -- or shorter: \deftextX{signed multiset} --
	is a map $S\colon E\maparrow \Z$. The \deftextX{family of signed multi-subsets of E}
	shall be denoted by $\Z .E$.\label{n:ZE}
	Since $\SET{-1,0,1}\subseteq \Z$, we shall identify the signed subsets with the corresponding signed multisets,
	i.e. $\sigma E \equiv \SET{F\in \Z. E\mid \forall e\in E\colon F(e) \in \SET{-1,0,1}}$.
	The \deftext[empty signed multiset]{empty signed multi-subset of $\bm E$}\label{n:emptysetZE} is the map
	\[ \emptyset_{\Z. E} \colon E\maparrow \Z,\, e\mapsto 0. \qedhere\]
\end{definition}

\needspace{5\baselineskip}
\begin{definition}[Dual of Definition 1, \cite{HN06}]\PRFRC
	Let $\Ocal=(E,\Ccal,\Ccal^\ast)$ be an oriented matroid. The \deftext[coflow lattice of O@coflow lattice of $\Ocal$]{coflow lattice of $\bm \Ocal$}
	shall consist of all integral linear combinations of cocircuits of $\Ocal$, i.e.
	\[ \Z .\Ccal^\ast = \SET{F\in \Z.E ~~\middle|~~ \exists \alpha\in \Z^{\Ccal^\ast}\colon\,\forall e\in E\colon\, F(e) = \sum_{C'\in\Ccal^\ast} \alpha(C')\cdot C'(e)}. \]
	Each element $F\in \Z.\Ccal^\ast$ shall be called a \deftextX{coflow of $\bm \Ocal$}.
	A \deftext[nowhere-zero coflow of O@nowhere-zero coflow of $\Ocal$]{nowhere-zero coflow of $\bm \Ocal$} is a coflow $F\in \Z.\Ccal^\ast$
	where $F(e)\not= 0$ for all $e\in E$.
\end{definition}

\begin{definition}\label{def:chiO}\PRFR{Mar 7th}
	Let $\Ocal=(E,\Ccal,\Ccal^\ast)$ be an oriented matroid. We define the \deftext[chromatic number of $\Ocal$]{chromatic number of $\bm \Ocal$}
	to be\label{n:chiO} \[ \chi(\Ocal) = \min \SET{\vphantom{A^{A^A}}\max \SET{\vphantom{A^A} \left| F( e ) \right| + 1~\middle|~ e\in E}  ~~ \middle| ~~ 
	F\in \Z.\Ccal^\ast,\,\forall e\in E\colon F( e )\not= 0}. \]
	By convention, we set $\chi(\Ocal) = \infty$ if there is no nowhere-zero coflow in $\Z.\Ccal^\ast$.
\end{definition}

\noindent
The only oriented matroid $\Ocal$ with $\chi(\Ocal) = 1$ is the trivial oriented matroid,\footnote{That is, if we set $\min \SET{\max\SET{\left| F(e) \right| + 1 ~\middle|~ e\in \emptyset} ~\middle|~ F \in \Z.\emptyset } = 1$. The rationale behind this is that a matroid $M=(E,\Ical)$ with $E=\emptyset$ is the graphical matroid of every edge-less graph; and the trivial oriented matroid is an orientation of $M$.}
i.e. the oriented matroid $\Ocal = (E,\Ccal,\Ccal^\ast)$ where $E = \Ccal = \Ccal^\ast =\emptyset$.

\begin{remark}\PRFR{Mar 7th}
	Let $\Ocal=(E,\Ccal,\Ccal^\ast)$ be an oriented matroid. 
	We have $\chi(\Ocal) = \infty$ if and only if there is an element $e\in E$ such that $D(e) = 0$ for all $D\in \Ccal^\ast$:
	Let $\Ccal^\ast = \pm \dSET{D_1, D_2,\ldots, D_k}$, i.e. for each pair $D\in \Ccal^\ast$ we choose precisely one element from $\SET{D,-D}$.
	Then there is no cancellation of non-zero summands in 
		\[ F(e) = \sum_{i=1}^{k} 2^{i-1} \cdot D_i(e)\]
	for $e\in E$, thus $F$ is a nowhere-zero coflow of $\Ocal$ if and only if for every $e\in E$ there is some $D\in \Ccal^\ast$ with $e\in D_\pm$.
	This is the case if and only if $M(\Ocal)$ has no loop.
\end{remark}

\noindent We give  a quick tour justifying why the name {\em chromatic number} is appropriate in this context. For a more detailed introduction, we refer the reader to Chapter~4 in \cite{Ni12}.

\needspace{5\baselineskip}

\vspace*{-\baselineskip} %Remove the line space created by the tilde below
\begin{wrapfigure}{r}{3cm}
\vspace{\baselineskip}
\begin{centering}~~%move the picture slightly to the right
\includegraphics{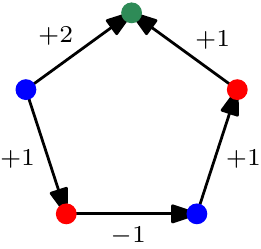}
\end{centering}%
\vspace*{-1\baselineskip} %make the picture more tightly cropped
\end{wrapfigure}
~ %The tilde creates a new dummy paragraph. WHY IS THAT NEEDED? -> would increase the space %
  % before the ex. environment. THE NEXT FREE LINE IS ESSENTIAL!

\begin{example}\PRFR{Mar 7th}
	Consider the undirected graph $G=(V,E)$ where $V = \SET{1,2,3,4,5}$ and $E = \{\SET{1,2},$ $\SET{1,5},$ $\SET{2,3},$ $\SET{3,4},$ $\SET{4,5}\}$.
	A proper coloring of $G$ is a map $\phi\colon V \maparrow \Z$ such that $\phi(v)\not=\phi(w)$ whenever $\SET{v,w}\in E$.
	The chromatic number of $G$ is defined as $\chi(G) = \min\SET{\left| \phi[V] \right| \vphantom{A^A}~\middle|~ \phi \text{~proper coloring of~}G}$,
	in this case $\chi(G) = 3$ and $\phi(1)=\phi(3)=1$, $\phi(2)=\phi(4)=2$, $\phi(5) = 3$ is a corresponding proper coloring.
	An orientation of $G$ is a digraph $D=(V,A)$ such that for every $\SET{u,v}\in E$ we have the equivalency $(u,v)\in A\Leftrightarrow (v,u)\notin A$,
	and such that $E = \SET{\SET{u,v}~\middle|~ (u,v)\in A}$. 
	Every orientation $D$ of $G$ gives rise to a map $\sigma\colon V\times V\maparrow \SET{-1,0,1}$ 
	where $\sigma(u,v) = +1$ if $(u,v)\in A$, $\sigma(u,v) = -1$ if $(v,u)\in A$, and $\sigma(u,v) = 0$ if $\SET{u,v}\notin E$. A nowhere-zero-coflow on $G$
	with respect to the orientation $D$	is a map $f\colon E\maparrow \Z$, such that $f(e)\not= 0$ for all $e\in E$, and such that for every closed 
	walk $v_1 v_2 \ldots v_k$ in $G$, i.e. $v_1=v_k$, we have $$\sum_{i=1}^{k-1}\left(  \sigma(v_i,v_{i+1}) \cdot f\left( \SET{v_i,v_{i+1}} \right) \right) = 0.$$
	Every coloring $\phi\colon V\maparrow \Z$ induces a coflow $\hat{\phi}\colon E\maparrow \Z$ on $G$ with respect to an orientation $D$ by setting
	$\hat \phi(\SET{u,v}) = \sigma(u,v)\cdot \left( \phi(v) - \phi(u) \right)$. Furthermore, $\hat \phi$ is a nowhere-zero-coflow if and only if $\phi$ is a
	proper coloring of $G$. Conversely, if $f$ is a nowhere-zero-coflow on $G$ with respect to $D$,
	 then we may reconstruct a proper coloring $\tilde f \colon V\maparrow \Z$ from it by choosing a vertex $v$ per component of $G$ and setting $\tilde f(v) = 0$. For every other vertex $x$,
	 let $w_1 w_2 \ldots w_k$ be a walk from the chosen vertex $v=w_1$ of the component containing $x$ to $x=w_k$ in $G$. 
	 We set $$\tilde f(x) = \left( \sum_{i=1}^{k-1} \left( \sigma(w_i,w_{i+1})\cdot f(\SET{w_i,w_{i+1}}) \right) \right)\mathrm{mod}_{\max\SET{\left| f(e) \right|+1~\middle|~\vphantom{A^A} e\in E}.}$$
	 Then $\tilde f(x)$ is a proper coloring of $G$ which uses at most $\max\SET{\vphantom{A^A}\left| f(e) \right|+1~\middle|~ e\in E}$ colors.
	 Furthermore, every graph $G=(V,E)$ gives rise to a cycle matroid $M(G) = (E,\Ical)$ where a set of edges $X\subseteq E$ is independent, if and only if
	 $(V,X)$ does not contain a cycle walk\footnote{Remember that the trivial walk $v$ is not a cycle walk.} (see \cite{We76}, p.10). The cycle matroid associated with the above graph $G$ is
	 $M(G) = \left( E,\SET{X\subseteq E~\middle|~\vphantom{A^A} \left| X \right| \leq 4} \right)$, the uniform matroid of rank $4$ on $E$. The cocircuits of cycle matroids
	 are the $\subseteq$-minimal subsets $D\subseteq E$, such that the graph $G\BS X = (V,E\BS X)$ has more components than $G=(V,E)$. Furthermore, every
	 orientation $D$ of $G$ yields an oriented matroid $\Ocal=(E,\Ccal,\Ccal^\ast)$ with $M(\Ocal) = M(G)$ by the following construction: $C\in \Ccal$ if
	 and only if there is a cycle walk $x_1x_2\ldots x_k$ in $G$ with $C_\pm = \SET{x_1,x_2,\ldots,x_k}$ and
	  $C(\SET{x_i,x_{i+1}}) = \sigma(x_i,x_{i+1})$ for all $i\in\SET{1,2,\ldots,k-1}$. In other words, an edge in the support of a signed circuit
	  $C$ of $\Ocal$ is assigned $+1$ if its orientation agrees with the corresponding arc in the cycle walk, and $-1$ otherwise. Dually, we have $D\in \Ccal^\ast$
	 if and only if there is a minimal edge-cut $X\subseteq E$ and a partition $L,R$ of $V$ such that every edge $e\in X$ has
	 the property $\left| e\cap L  \right| = 1$; with $D_\pm = X$ and
	  $D(\SET{l,r}) = \sigma(l,r)$ for all $l\in L$ and $r\in R$ with $\SET{l,r}\in X$.
	 Let $\Ocal=(E,\Ccal,\Ccal^\ast)$ be the oriented matroid that corresponds to the orientation $D$ of $G$ as above. We have
	  $\Ccal = \pm \SET{ \vphantom{\sum_A^{A^A}}\SET{\vphantom{A^A}\SET{1,2},-\SET{1,5},\SET{2,3},\SET{3,4},\SET{4,5}} }$ and
	  \begin{align*}
	  	\Ccal^\ast = \pm \left\{\vphantom{A^{A^A}}\right. & 
	  	\SET{\vphantom{A^A}\SET{1,2},\SET{1,5}}, \SET{\vphantom{A^A}\SET{1,2},-\SET{2,3}}, 
	  	  \SET{\vphantom{A^A}\SET{1,2},-\SET{3,4}}, \SET{\vphantom{A^A}\SET{1,2},-\SET{4,5}},	  	\\ 
	  	  &
	  	\SET{\vphantom{A^A}\SET{2,3},\SET{1,5}}, 
	  	  \SET{\vphantom{A^A}\SET{2,3},-\SET{3,4}}, \SET{\vphantom{A^A}\SET{2,3},-\SET{4,5}},	  	\\ 
	  	  &
	  	\SET{\vphantom{A^A}\SET{3,4},\SET{1,5}}, \SET{\vphantom{A^A}\SET{3,4},-\SET{4,5}},	\SET{\vphantom{A^A}\SET{4,5},\SET{1,5}}
	  	  \left.\vphantom{A^{A^A}}\right\} .
	  \end{align*}
	 Furthermore, all coflows of $G$ with respect to $D$ are integral linear combinations of the signed cocircuits of the oriented matroid
	 $\Ocal$ corresponding to the orientation $D$ of $G$; the coflow $\hat \phi$ may be written as the a linear combination of cocircuits of $\Ocal$
	 \( \hat \phi = \SET{\vphantom{A^A}\SET{1,2},-\SET{2,3}} + \SET{\vphantom{A^A}\SET{3,4},-\SET{4,5}} + 2\cdot\SET{\vphantom{A^A}\SET{4,5},\SET{1,5}}.\)
	 For all graphs $G=(V,A)$, the equation $\chi(G) = \chi(\Ocal)$ holds, 
	 where $\Ocal$ corresponds to an orientation $D$ of the cycle matroid of $G$. Thus the chromatic number of 
	 oriented matroids is a generalization of the chromatic number of graphs.
\end{example}

\begin{example}\PRFR{Mar 7th}
	Consider the oriented matroids $\Ocal_1$ and $\Ocal_2$ given in Example~\ref{ex:nonStrictGammoidOrientations}.
	$M(\Ocal_1)=M(\Ocal_2)=(E,\Ical)$ is the matroid given in Example~\ref{ex:nonStrictGammoid}.
	For both orientations, the corresponding coflow lattice is the free integer module $\Z^E$. Therefore we have
	$\chi(\Ocal_1) = \chi(\Ocal_2) = 2$.
\end{example}

\PRFR{Mar 7th}
\noindent
Let $\Ocal$ and $\Ocal'$ be two oriented matroids such that $M(\Ocal) = M(\Ocal')$. We would like to mention that it is still an open problem whether
in this case
the equation  $\chi(\Ocal) = \chi(\Ocal')$ holds in general (\cite{Ni12} (Q4), p.69). This question clearly is beyond the scope of this work.
However, if $\Ocal' = \Ocal_{-X}$ is the reorientation of $\Ocal$ with respect to some set $X\subseteq E$, then 
$$\Z.\Ccal^\ast_{-X} = \SET{ F\in \Z^{E} ~\middle|~ \exists F'\in \Z.\Ccal^\ast \colon \, \forall e\in E\colon\, F(e)=(-1)^{\chi_X(e)} F'(e) }$$
where $\chi_X$ is the characteristic function of $X\subseteq E$, i.e. $\chi_X(e) = 1$ if $e\in X$ and $\chi_X(e) = 0$ if $e\notin X$.
Therefore the nowhere-zero coflows of $\Ocal$ are in a $\left| \cdot \right|$-preserving one-to-one correspondence with the nowhere-zero coflows of $\Ocal'$,
thus $\chi(\Ocal) =\chi(\Ocal')$ whenever $\Ocal'$ is a reorientation of $\Ocal$.
%\footnote{Our 
%best educated guess would be that the equation holds in general, because it holds when the underlying matroid is uniform. Our intuition suggests that
%if two different orientations of the same underlying matroid 
%have respective coflow lattices that differ with respect to some property like the chromatic number, then this
%difference should manifest in orientations of uniform matroids.}

\begin{theorem}[\cite{HN06}, Theorem 1]\label{thm:chiOuniform}\PRFR{Mar 7th}
	Let $r\in \N$ and $\Ocal = (E,\Ccal,\Ccal^\ast)$ be an oriented matroid 
	such that $M(\Ocal) = \left( E,\SET{\vphantom{A^A}X\subseteq E~\middle|~\left| X \right| \leq r }\right)$ is the uniform matroid of rank $r$ on $E$.
	Then 
	\[ \chi(\Ocal) = \begin{cases}[r] 2 & \quad\text{if~} n\cdot (n-r)\text{~is even},\\
	                                 3 &  \quad\text{if~} n\cdot (n-r)\text{~is odd}.\\
	\end{cases} \]
\end{theorem}

\noindent
See \cite{HN06} for the proof. Theorem~\ref{thm:chiOuniform} is the generalization of the fact that the chromatic number of a cycle graph -- with at least $3$ vertices -- is $2$ if the cycle graph consists of an even number of vertices, and $3$ if the cycle graph consists of an odd number of vertices.

\needspace{6\baselineskip}

\begin{theorem}[\cite{HN08}, Theorem 3]\PRFR{Mar 7th}
	Let $\Ocal= (E,\Ccal,\Ccal^\ast)$ be an oriented matroid such that $M(\Ocal)$ has no loops, no parallel edges, and $\rk_{M(\Ocal)}(E) \geq 3$.
	%such that for all $e,f\in E$ with
	%$e\not= f$, $\rk_{M(\Ocal)}\left( \SET{e,f} \right) = 2$.
	Then \[ \chi(\Ocal) \leq \rk_{M(\Ocal)}(E) + 1 \]
	where equality holds if and only if $M(\Ocal)$ is isomorphic 
	to the cycle matroid $M(K)$ of the complete graph $K=\left(V,\binom{V}{2}\right)$ with $\left| V \right| =\rk_{M(\Ocal)}(E)+1$ vertices.
\end{theorem}

\PRFR{Mar 7th}
\noindent
See \cite{HN08} for the proof. If $M(\Ocal)$ has no loops but it has two parallel edges, i.e. some $e,f\in E$ with $e\not=f$ and
 $\rk_{M(\Ocal)}\left( \SET{e,f} \right) = 1$,
then $\SET{e,f}$ is a circuit of $M(\Ocal)$. By Lemma~\ref{lem:CircuitCocircuitOrthogonality} we obtain that
the equality $e\in D_\pm \Leftrightarrow f\in D_\pm$ holds for every $D\in \Ccal^\ast$.
Let $C\in \Ccal$ be the signed circuit with $C_\pm = \SET{e,f}$, and let $D_1,D_2\in \Ccal^\ast$ with $e\in \left( {D_1}_\pm \cap {D_2}_\pm \right)$.
Since $C\bot D_1$ and $C\bot D_2$, we obtain that $D_1(e)D_2(e) = D_1(f) D_2(f)$, i.e. the sign of $e$ uniquely determines the sign of $f$ 
in any cocircuit of $\Ocal$.
 In the proof of Lemma~\ref{lem:CircuitCocircuitOrthogonality} we
established that cocircuits are the complements of hyperplanes, therefore every signed cocircuits $D'$ of the restriction
$\Ocal \restrict \left( E\BSET{f} \right)$
corresponds to a signed cocircuit $D$ of $\Ocal$ 
where $\SET{e,f}\subseteq D_\pm$ if and only if $e\in D'_\pm$. 
Thus a nowhere-zero coflow $\phi'$ of $\Ocal\restrict \left( E\BSET{f} \right)$ 
extends naturally to a nowhere-zero coflow $\phi$ of $\Ocal$ with $\phi'(f) \in \SET{-\phi(e),\phi(e)}$ by taking any integer linear combination of cocircuits 
of the restriction $\Ocal\restrict \left( E\BSET{f} \right)$  with respect to the corresponding cocircuits of $\Ocal$.
Consequently, $\chi(\Ocal) = \rk_{M(\Ocal)}(E) + 1$
if and only if $M(\Ocal)$ is isomorphic to the cycle matroid of a multi-graph on $\rk_{M(\Ocal)}(E) + 1$ vertices that has at least one edge between every pair
of distinct vertices.

\clearpage
% -*- root: ../thesis.tex -*-

\section{Lattice Path Matroids are 3-Colorable}

\PRFR{Mar 7th}
The results presented in this section %, except for the duality respecting representations of lattice path matroids, 
have been presented in the technical report  {\em Lattice Path Matroids are 3-Colorable}
by I.~Albrecht and W.~Hochstättler \cite{Al15}.

\begin{definition}[\cite{GoHoNe15}, Definition~4]\PRFR{Mar 7th}
 Let $M=(E,\Ical)$ be
a matroid. A flat $X\in\Fcal(M)$ is called \deftext[coline of M@coline of $M$]{coline of $\bm M$},
 if $\rk_M(X)=\rk_M(E)-2$.
 A flat $Y\in\Fcal(M)$ is called
\deftext[copoint of M on X@copoint of $M$ on $X$]{copoint of $\bm M$ on $\bm X$},
 if $X\subseteq Y$ and
$\rk_M(Y)=\rk_M(E)-1$. 
If further $\left| Y\backslash X \right|=1$,
we say that $Y$ is a \deftext[simple copoint on X@simple copoint on $X$]{simple copoint on $\bm X$}. 
If otherwise $\left| Y\backslash X \right|>1$, we
say that $Y$ is a \deftext[multiple copoint on X@multiple copoint on $X$]{multiple copoint on $\bm X$}\footnote{In \cite{GoHoNe15} and \cite{Al15}, multiple copoints are called {\em fat copoints}.\index{fat copoint}}.
 A \deftext{quite simple coline}\footnote{In \cite{GoHoNe15} and \cite{Al15}, quite simple colines are called {\em positive colines}.\index{positive coline}} is a coline $X\in \Fcal(M)$,
 such that there are more simple copoints on $X$ than there are multiple copoints on $X$.
\end{definition}

\noindent The following definitions are basically those found in J.E.~Bonin and A.~deMier's paper {\em Lattice path matroids: Structural properties} \cite{Bonin2006701}.
\begin{definition}\PRFR{Mar 7th}
Let $n\in\mathbb{N}$. A \deftext{lattice path} of length $n$ is a tuple\label{n:latticePath}
$(p_{i})_{i=1}^{n}\in\{\mathrm{N},\mathrm{E}\}^{n}$. 
We say
that the \deftext[i-th step of a lattice path@$i$-th step of a lattice path]{$\bm i$-th step} of $(p_{i})_{i=1}^{n}$ is towards the North if $p_{i}=\mathrm{N}$,
and towards the East if $p_{i}=\mathrm{E}$.
\end{definition}

\begin{definition}\PRFR{Mar 7th}
Let $n\in\mathbb{N}$, and let $p = (p_{i})_{i=1}^{n}$ and $q = (q_{i})_{i=1}^{n}$
be lattice paths of length $n$. We say that $p$
is \deftext[p south of q@$p$ south of $q$]{south of $\bm q$}
% with common endpoints 
if for all $k\in\SET{1,2,\ldots,n}$, 
\[
\left|\left\{ i\in \N \BSET {0} \vphantom{A^A}~\middle|~ i\leq k \txtand p_{i}=\mathrm{N}\right\} \right|\leq\left|\left\{ i\in \N \BSET {0} \vphantom{A^A}~\middle|~ i\leq k \txtand q_{i}=\mathrm{N}\right\} \right|.
\]
We say that $p$ and $q$ have \deftext[lattice paths with common endpoints]{common endpoints}, if
 $$\left|\left\{ i\in \N \BSET {0} \vphantom{A^A}~\middle|~ i\leq n \txtand p_{i}=\mathrm{N}\right\} \right| = \left|\left\{ i\in \N \BSET {0} \vphantom{A^A}~\middle|~ i\leq n \txtand q_{i}=\mathrm{N}\right\} \right|$$ holds. 
 We say that the \deftextX{lattice path $\bm p$
is south of  $\bm q$ with common endpoints},
 if $p$ and $q$ have common endpoints and $p$ is south of $q$.
 In this case, we write $p \preceq q$.\label{n:neverabove}
\end{definition}

\begin{definition}\PRFR{Mar 7th}
Let $n\in \N$, and  let $p,q\in \SET{\mathrm{E},\mathrm{N}}^n$ be lattice paths such
that $p\preceq q$. We define the set
of \deftext[lattice paths between p and q@lattice paths between $p$ and $q$]{lattice paths between
 $\bm p$ and $\bm q$}
to be\label{n:LPbetweenPQ}
\[
\mathrm{P}\left[p,q\right]=\left\{ r \in\{\mathrm{N},\mathrm{E}\}^{n}
\vphantom{A^A}~\middle|~
p \preceq r \preceq q\right\} . \qedhere
\]
\end{definition}

\needspace{6\baselineskip}
\begin{definition}\label{def:LPmatroid}\PRFR{Mar 7th}
A matroid $M=(E,\Ical)$ is called \deftext{strong lattice path matroid}, if
its ground set has the property 
$E = \SET{1,2,\ldots,\left| E \right|}$ and if
there are lattice paths $p,q\in \SET{\mathrm{E},\mathrm{N}}^{\left| E \right|}$
with $p\preceq q$,
such that $M = M[p,q],$
where $M[p,q]$ denotes the transversal matroid  presented by
the family
$\Acal_{[p,q]}=(A_i)_{i=1}^{\rk_M(E)} \subseteq E$
with
\[
A_i = \SET{j\in E~\middle|~\exists (r_j)_{j=1}^{\left| E \right|}\in \mathrm{P}[p,q]\colon\,
 r_j = \mathrm{N}\txtand \left| \SET{k\in E\mid k\leq j,\,r_k=\mathrm{N} }\right| = i },
\]
i.e. each $A_i$ consists of those $j\in E$, such that there is a 
lattice path $r$ between $p$ and $q$ such that the $j$-th step of $r$ is towards the North for the $i$-th time in total.
Furthermore,
a matroid $M=(E,\Ical)$ is called \deftext{lattice path matroid}, 
if there is a bijection $\phi \colon E\maparrow \SET{1,2,\ldots,\left| E \right|}$ 
such that $\phi[M] = \left(\phi[E],\SET{\phi[X]\vphantom{A^A}~\middle|~ X\in\Ical}\right)$ is a strong lattice path matroid.
\end{definition}

\vspace*{-\baselineskip} %Remove the line space created by the tilde below
\begin{wrapfigure}{r}{5cm}
\vspace{1\baselineskip}
\begin{centering}%move the picture slightly to the right
~~%\includegraphics[scale=1]{ex}
\begin{tikzpicture}
\node[inner sep=0pt,circle,minimum size=4pt,fill] at (0,0) {};

\node[inner sep=0pt,circle,minimum size=4pt,fill] at (1,0) {};

\node[inner sep=0pt,circle,minimum size=4pt,fill] at (2,0) {};

\node[inner sep=0pt,circle,minimum size=4pt,fill] at (0,1) {};

\node[inner sep=0pt,circle,minimum size=4pt,fill] at (1,1) {};

\node[inner sep=0pt,circle,minimum size=4pt,fill] at (2,1) {};

\node[inner sep=0pt,circle,minimum size=4pt,fill] at (3,1) {};
\node[inner sep=0pt,circle,minimum size=4pt,fill] at (0,2) {};

\node[inner sep=0pt,circle,minimum size=4pt,fill] at (1,2) {};

\node[inner sep=0pt,circle,minimum size=4pt,fill] at (2,2) {};

\node[inner sep=0pt,circle,minimum size=4pt,fill] at (3,2) {};
\node[inner sep=0pt,circle,minimum size=4pt,fill] at (1,3) {};

\node[inner sep=0pt,circle,minimum size=4pt,fill] at (2,3) {};

\node[inner sep=0pt,circle,minimum size=4pt,fill] at (3,3) {};

\draw[thin] (1,0) -- (1,3);
\draw[thin] (2,0) -- (2,3);
\draw[thin] (0,1) -- (3,1);
\draw[thin] (0,2) -- (3,2);
\draw[very thick] (0,0) -- (2,0) -- (2,1) -- (3,1) -- (3,3);
\node at (3.5,.5) {$p$};
\node at (-.5,.5) {$q$};

\draw[very thick] (0,0) -- (0,2) -- (1,2) -- (1,3) -- (3,3);
\begin{scope}[shift={(-.1,-.15)}]
\node at (.3,.5) {$1$};
\node at (1.3,.5) {$2$};
\node at (2.3,.5) {$3$};
\node at (.3,1.5) {$2$};
\node at (1.3,1.5) {$3$};
\node at (2.3,1.5) {$4$};
\node at (3.3,1.5) {$5$};
\node at (1.3,2.5) {$4$};
\node at (2.3,2.5) {$5$};
\node at (3.3,2.5) {$6$};
\end{scope}
\end{tikzpicture}
\end{centering}%
\vspace*{-1\baselineskip} %make the picture more tightly cropped
\end{wrapfigure}
~ %The tilde creates a new dummy paragraph. WHY IS THAT NEEDED? -> would increase the space %
  % before the ex. environment. THE NEXT FREE LINE IS ESSENTIAL!

\begin{example}\PRFR{Mar 7th}
Let us consider the two lattice paths $p=(\mathrm{E},\mathrm{E},\mathrm{N},\mathrm{E},\mathrm{N},\mathrm{N})$
and $q=(\mathrm{N},\mathrm{N},\mathrm{E},\mathrm{N},\mathrm{E},\mathrm{E})$. 
We have $p\preceq q$ and the strong lattice path matroid $M[p,q]$ is the transversal matroid $M(\Acal)$ presented
by the family \linebreak $\Acal=(A_i)_{i=1}^3$ of subsets of $\SET{1,2,\ldots,6}$ where $A_{1}=\{1,2,3\}$, $A_{2}=\{2,3,4,5\}$,
and 
$A_{3}=\{4,5,6\}$.%
\end{example}

\begin{theorem}[\cite{Bonin2006701}, Theorem~2.1]\PRFR{Mar 7th}
\label{LPMthm:P} Let $p$, $q$ be lattice
paths of length $n$, such that $p\preceq q$.
Let $\Bcal \subseteq 2^{\SET{1,2,\ldots,n}}$ consist of the bases of the
strong lattice path matroid $M=M[p,q]$ on the ground set $\SET{1,2,\ldots,n}$.
Let
$$\phi \colon \mathrm{P}[p,q] \maparrow \Bcal,\quad (r_i)_{i=1}^n \mapsto \SET{j\in \N ~\middle|~ 1\leq j\leq n,\,r_j = \mathrm{N}} .$$
Then $\phi$ is a bijection between 
the family of lattice paths $\mathrm{P}[p,q]$ between $p$ and $q$ and the family of bases of $M$.
\end{theorem}
\begin{proof}\PRFR{Mar 7th}
	Clearly, $\phi$ is well-defined: let $r=(r_i)_{i=1}^n\in \mathrm{P}[p,q]$, and let $m = \rk_M(\SET{1,2,\ldots,n})$,
	then there are $j_1 < j_2 < \ldots < j_m$ such that $r_i = \mathrm{N}$ if and only if $i\in \SET{j_1,j_2,\ldots,j_m}$.
	Thus the map $$\iota_r\colon \phi(r) \maparrow \SET{1,2,\ldots,m},$$ where
	$\iota_r(i) = k$ for $k$ such that $i = j_k$, witnesses that the set $\phi(r) \subseteq \SET{1,2,\ldots,n}$ is indeed a transversal of $\Acal_{[p,q]}$,
	and therefore a base of $M[p,q]$. It is clear from Definition~\ref{def:LPmatroid} that $\phi$ is surjective.
	It is obvious that if we consider only lattice paths of a fixed given length $n$, then the indexes of the steps towards the North
	 uniquely determine such a lattice path. Thus $\phi$ is also injective.
\end{proof}

\needspace{6\baselineskip}
\begin{proposition}\PRFR{Mar 7th}
\label{LPMprop:X}Let $p=(p_{i})_{i=1}^{n}$, $q=(q_{i})_{i=1}^{n}$
be lattice paths of length $n$ such that $p\preceq q$. Let $j\in E=\SET{1,2,\ldots,n}$ and $M=M[p,q]$. Then
\begin{enumerate} \ROMANENUM
\item $\rk_M\left( \SET{1,2,\ldots,j} \right)=\left|\left\{ i\in \SET{1,2,\ldots,j} \mid q_{i}=\mathrm{N}\right\} \right|$.
\item \begin{minipage}[t]{12cm} The element $j$ is a loop in $M$ if and only if 
\[
\left|\left\{ i\in \SET{1,2,\ldots,j-1} \vphantom{A^A}~\middle|~ p_{i}=\mathrm{N}\right\} \right|=\left|\left\{ i\in \SET{1,2,\ldots,j}\vphantom{A^A}~\middle|~ q_{i}=\mathrm{N}\right\} \right|,
\]
i.e. the $j$-th step is forced to go towards East for all $r\in \mathrm{P}[p,q]$.\end{minipage}
\vtop{%
  \vskip0pt
  \vspace*{.8cm}
  \hbox{%
    \includegraphics[scale=.75]{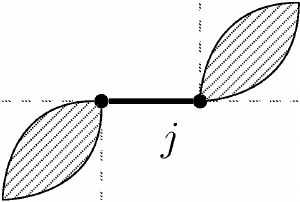}%
  }%
}
\item For all $k\in E$ with $j<k$, $j$ and $k$ are parallel edges in $M$
if and only if 
\begin{eqnarray*}
\left|\left\{ i\in\SET{1,2,\ldots,j-1}\vphantom{A^A}~\middle|~p_{i} = \mathrm{N}\right\} \right| & = & \left|\left\{ i\in\SET{1,2,\ldots,k-1}\vphantom{A^A}~\middle|~ p_{i}=\mathrm{N}\right\} \right|
\\ & =&
\left|\left\{ i\in\SET{1,2,\ldots,j}\vphantom{A^A}~\middle|~q_{i} = \mathrm{N}\right\} \right| -1
\\& = & \left|\left\{ i\in\SET{1,2,\ldots,k}\vphantom{A^A}~\middle|~q_{i}=\mathrm{N}\right\} \right| -1, \end{eqnarray*}
\begin{minipage}[t]{8cm}
i.e. the $j$-th and $k$-th steps of any $r\in \mathrm{P}[p,q]$ are in a common corridor towards
the East that is one step wide towards the North.\end{minipage}~~~~~~\vtop{%
  \vskip0pt
  \vspace*{-1.0cm}
  \hbox{%
  \includegraphics[scale=.75]{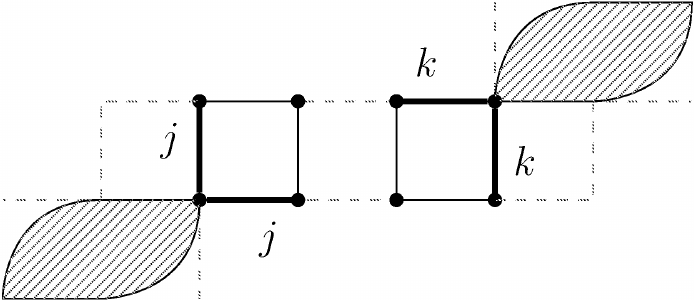}}%
}
\end{enumerate}
\end{proposition}
\begin{proof}\PRFR{Mar 7th}
For every $r\in \mathrm{P}[p,q]$, we have $r \preceq q$,
therefore $r$ is south of $q$, thus for
\linebreak all $k\in E$,
$\left| \SET{j\in \SET{1,2,\ldots,k} \vphantom{A^A}~\middle|~r_k = \mathrm{N}} \right| \leq \left| \SET{j\in \SET{1,2,\ldots,k}\vphantom{A^A}~\middle|~ q_k = \mathrm{N}} \right|$.
Consequently, 
\linebreak
$\left\{ i\in \SET{1,2,\ldots,j} \vphantom{A^A}~\middle|~ q_{i}=\mathrm{N}\right\}$ is a maximal independent subset of $\SET{1,2,\ldots,j}$
and so statement {\em (i)} holds.
%
%\noindent
An element  $j\in E$ is a loop in $M$, if and only if $\mathrm{rk}_M(\{j\})=0$,
which is the case if and only $\SET{j}$ is not independent in $M$.
This is the case if and only if for all bases $B$ of $M$, $j\notin B$ holds, 
because every independent set is a subset of a base (Lemma~\ref{lem:augmentation}).
The latter holds
 if and only if for all $(r_{i})_{i=1}^{n}\in\mathrm{P}[p,q]$ the $j$-th step is towards the East, i.e. $r_{j}=\mathrm{E}$.
 This, in turn, is the case
if and only if $\left|\left\{ i\in\SET{1,2,\ldots,j-1}\vphantom{A^A}~\middle|~p_{i}=\mathrm{N}\right\} \right|=\left|\left\{ i\in\SET{1,2,\ldots,j}\vphantom{A^A}~\middle|~q_{i}=\mathrm{N}\right\} \right|$.
Thus statement {\em (ii)} holds, too.
Let $j,k\in E$ with $j< k$. It is easy to see that if $j$ and $k$ are in a common corridor, then every lattice path $r=(r_i)_{i=1}^{n}$ of length $n$
with $r_j = r_k = \mathrm{N}$ cannot be between $p$ and $q$, i.e. $p\preceq r \preceq q$ cannot hold: a lattice path $r$ with $r_j = r_k = \mathrm{N}$
is either below $p$ at $j-1$ or above $q$ at $k$. 
Thus $\SET{j,k}$ cannot be independent in $M$. By {\em (i)}, neither $j$ nor $k$ can be a loop in $M$, thus $j$ and $k$ must be parallel edges in $M$.
Conversely, let $j < k$ be parallel edges in $M$. 
Then $j$ is not a loop in $M$, so there is a path $r^1 = (r_{i}^{1})_{i=1}^{n}\in\mathrm{P}[p,q]$
with $r_{j}^{1}=\mathrm{N}$ which is minimal with regard to $\preceq$,
and then
\[
\left|\left\{ i\in \SET{1,2,\ldots,j-1}\vphantom{A^A}~\middle|~ r_{i}^{1}=\mathrm{N}\right\} \right|=\left|\left\{ i\in\SET{1,2,\ldots,j-1}\vphantom{A^A}~\middle|~p_{i}=\mathrm{N}\right\} \right|.
\]
Since $j$ and $k$ are parallel edges, $\SET{j,k}\not\subseteq B$ for all bases $B$ of $M$.
Therefore there is no $r=(r_{i})_{i=1}^{n}\in \mathrm{P}[p,q]$ such that $r_i = r_k = \mathrm{N}$.
This yields the equation
\begin{align*}
\left|\left\{ i\in\SET{1,2,\ldots,k}\vphantom{A^A}~\middle|~q_{i} = \mathrm{N}\right\} \right|  \,\,=\,\, & \left| \left\{ i\in\SET{1,2,\ldots,j}\vphantom{A^A}~\middle|~ r_{i}^{1}=\mathrm{N}\right\}\right|
\\   = \,\,& \left|\left\{ i\in\SET{1,2,\ldots,j-1}\mid r_{i}^{1}=\mathrm{N}\right\} \right|+1.
\end{align*}
Since $k$ is not a loop in $M$, it follows that \[
\left|\left\{ i\in\SET{1,2,\ldots,j-1}\vphantom{A^A}~\middle|~ p_{i}=\mathrm{N}\right\} \right|=\left|\left\{ i\in\SET{1,2,\ldots,j}\vphantom{A^A}~\middle|~ q_{i}=\mathrm{N}\right\} \right|-1.
\] Thus {\em (iii)} holds.
\end{proof}

\begin{lemma}\PRFR{Mar 7th}
\label{LPMlem:A}Let $p=(p_{i})_{i=1}^{n}$ and $q=(q_{i})_{i=1}^{n}$ be
lattice paths of length $n$, such that $p\preceq q$, and such that $M=M[p,q]$ is
a strong lattice path matroid on $E=\SET{1,2,\ldots,n}$ which has no loops.
Let $j\in E$ such that $q_j = \mathrm{N}$. Then 
\[
\SET{1,2,\ldots,j-1}=\cl_M\left(\SET{1,2,\ldots,j-1}\right).
\]
Furthermore, for all $k\in E$ with $k \geq j$, 
\[
\rk_M\left( \SET{1,2,\ldots,j-1}\cup\{k\} \right)= \rk_M\left( \SET{1,2,\ldots,j-1} \right)+1.
\]
 \end{lemma}
 
\vspace*{-\baselineskip} %Remove the line space created by the tilde below
\begin{wrapfigure}{r}{5cm}
\vspace{1.6\baselineskip}
\begin{centering}%move the picture slightly to the right
~~%\includegraphics[width=50mm]{qrjlk-crop.tif}
\begin{tikzpicture}[scale=0.8]\selectcolormodel{cmyk}
\node[inner sep=0pt,circle,minimum size=4pt,fill] (j) at (0,0) {};
\node[inner sep=0pt,circle,minimum size=4pt,fill] at (1,0) {};
\node[inner sep=0pt,circle,minimum size=4pt,fill] at (0,1) {};
\node[inner sep=0pt,circle,minimum size=4pt,fill] at (1,1) {};
\begin{scope}[shift={(1,2)}]
\draw[thick]   (-1,-1) .. controls (-1,-1) and (-1,1) .. (0,1) -- (3,1);
\draw  (0,1) -- (3,1) .. controls (3,-1) and (2,-1) .. (2,-1)  -- (-1,-1) .. controls (-1,-1) and (-1,1) .. (0,1);
\fill[pattern color = lightgray, pattern = north east lines] (0,1) -- (3,1) .. controls (3,-1) and (2,-1) .. (2,-1)  -- (-1,-1) .. controls (-1,-1) and (-1,1) .. (0,1);
\fill[color=white] (0,-.5) -- (2,-.5) -- (2, .5) -- (0, .5) -- (0,-.5);
\draw[dashed] (0,-.5) -- (2,-.5) -- (2, .5) -- (0, .5) -- (0,-.5);
\draw[ultra thick] (0,-1) .. controls (0,-0.5) and (1,-0.5).. (1,-0.5) -- (1,.5) .. controls (1,.5) and (2,0.5) .. (3,1);
\node at (1.3,0) {$k$}; 
\end{scope}

\node[inner sep=0pt,circle,minimum size=4pt,fill] at (2,1.5) {};
\node[inner sep=0pt,circle,minimum size=4pt,fill] at (2,2.5) {};

%\draw [dashed,color=gray] (j) -- (-2,0);
%\draw [dashed,color=gray] (j) -- (0,-1);
%\draw [dashed,color=gray] (k) -- (5,1);
%\draw [dashed,color=gray] (k) -- (3,2);
%\draw [dashed,color=gray] (3,0) -- (4,0) -- (4,1);
%\draw [dashed,color=gray] (0,1) -- (-1,1) -- (-1,0);

%\fill[pattern color = gray, pattern = north east lines] (-1,0) .. controls (-1,-1) and (-2,-1) .. (-2,-1) -- (-1,-1) .. controls (-1,-1) and (-1,0) .. (0,0);
\begin{scope}[shift={(-1,0)}]
\draw[ultra thick]   (-1,-1) .. controls (-1,-1) and (-1,0) .. (0,0) -- (1,0);
\draw  (0,0) -- (3,0) .. controls (3,-1) and (2,-1) .. (2,-1)  -- (-1,-1) .. controls (-1,-1) and (-1,0) .. (0,0);
\fill[pattern color = lightgray, pattern = north east lines] (0,0) -- (3,0) .. controls (3,-1) and (2,-1) .. (2,-1)  -- (-1,-1) .. controls (-1,-1) and (-1,0) .. (0,0);
\draw[thick] (-1,-1) .. controls (0,0) and (2,-1).. (2,0);
\end{scope}

%\draw[] (j)--(0,1)--(1,1)--(1,0)--(j);
\draw[thick] (j)--(0,1);
\draw[ultra thick] (j)--(1,0);

\draw[ultra thick] (1,0)--(1,1);

%\draw[ultra thick] (k)--(3,0);
%\draw[ultra thick] (k)--(2,1);

\draw[dashed] (0,0)--(0,1)--(1,1)--(1,0)--(0,0);
\begin{scope}[shift={(2,0)}]
\draw[dashed,color=gray] (0,0)--(0,1)--(1,1)--(1,0)--(0,0);
\end{scope}

%\node at (.7,-.4) {$j$}; 
%\node at (2.3,1.4) {$k$}; 
\node at (.3,.5) {$j$}; 
%\node at (1.3,.5) {$l$}; 
\node at (-1.5,.2) {$q$};
\node at (-1,-.8) {$r$};
\end{tikzpicture}
\end{centering}%
%\vspace*{-3\baselineskip} %make the picture more tightly cropped
\end{wrapfigure}
~ %The tilde creates a new dummy paragraph. WHY IS THAT NEEDED? -> would increase the space %
  % before the ex. environment. THE NEXT FREE LINE IS ESSENTIAL!

\begin{proof}\PRFR{Mar 7th}
By Proposition~\ref{LPMprop:X}~{\em (i)}, we have
 $$\mathrm{rk}_M(\SET{1,2,\ldots,j-1})=\left|\left\{ i\in\SET{1,2,\ldots,j-1}\vphantom{A^A}~\middle|~ q_{i}=\mathrm{N}\right\} \right|.$$
Now fix some $k\in E$ with $k\geq j$. Since $M$ has no loop, 
 there is a base $B$ of $M$ with $k\in B$ and thus a lattice path $r=(r_{i})_{i=1}^{n}\in\mathrm{P}[p,q]$
with $r_{k}=\mathrm{N}$ (Theorem~\ref{LPMthm:P}).
We can construct a lattice path $s=(s_{i})_{i=1}^{n}\in\mathrm{P}[p,q]$
that follows $q$ for the first $j-1$ steps, then goes towards the
East until it meets $r$, and then goes on as $r$ does. 
The base $B_s = \SET{i\in E\mid s_i = \mathrm{N}}$
 that corresponds to the constructed path yields
\begin{align*}
\mathrm{rk}_M(\SET{1,2,\ldots,j-1}\cup\{k\}) & \geq \left|\left(\SET{1,2,\ldots,j-1}\vphantom{A^A}\cup\{k\}\right)\cap B_s\right|
\\ & = 1+\left|\left\{ i\in\SET{1,2,\ldots,j-1}\vphantom{A^A}~\middle|~ q_{i}=\mathrm{N}\right\} \right|
\\ & = 1+\mathrm{rk}_M(\SET{1,2,\ldots,j-1}).
\end{align*}
Since $\rk_M$ is unit increasing (Theorem \ref{thm:rankAxioms}, {\em (R2')}),
adding a single element to a set can increase the rank by at most one,
 thus the inequality in the above formula is indeed an equality.
This implies that $k\notin\mathrm{cl}_M(\SET{1,2,\ldots, j-1})$ (Lemma \ref{lem:clKeepsRank}).
Since $k$ was arbitrarily chosen with $k \geq j$,
we obtain
 $\SET{1,2,\ldots, j-1}=\mathrm{cl}_M(\SET{1,2,\ldots, j-1})$.\end{proof}

\begin{theorem}\label{thm:WesternColine}\PRFR{Mar 7th}
\label{LPMthm:B}Let $p=(p_{i})_{i=1}^{n}$, $q=(q_{i})_{i=1}^{n}$ be
lattice paths, such that $p\preceq q$ and such that $M=M[p,q]=(E,\Ical)$ 
has no loop and no parallel edges, and such that $\rk_M(E) \geq 2$.
Let $j_{1}=\max\SET{i\in E\mid q_i = \mathrm{N}}$ and $j_{2}=\max\SET{i\in E\mid i\not= j_1 \txtand q_i = \mathrm{N}}.$
Then the following holds
\begin{enumerate}\ROMANENUM
\item $\SET{1,2,\ldots, j_{2}-1}$ is a coline of $M$, we shall call it the \deftext[Western coline]{Western coline of $\bm M$}.
\item $\SET{1,2,\ldots,j_{1}-1}$ is a copoint on the Western coline of $M$, which is a multiple copoint whenever $j_{1}-j_{2}\ge 2$. 
\item For every $k \geq j_1$ the set $\SET{1,2,\ldots,j_{2}-1}\cup\{k\}$ is a simple copoint on the Western coline of $M$.
\end{enumerate}
\end{theorem}

\begin{proof}\PRFR{Mar 7th}
Lemma~\ref{LPMlem:A} provides that the set $W = \SET{1,2,\ldots,j_{2}-1}$ as well as the set $X = \SET{1,2,\ldots,j_{1}-1}$ 
is a flat of $M$. By construction of $j_1$ and $j_2$ we have that $\rk(W) = \rk(E) -2$ and $\rk(X) = \rk(E) - 1$.
Thus $W$ is a coline of $M$ --- so {\em (i)} holds --- and $X$ is a copoint of $M$,
which follows from  and the construction of $j_{2}$
and $j_{1}$. Since $\left| X\BS W \right| = \left| \SET{j_2, j_2+1,\dots, j_1-1} \right| = j_1 - j_2$
we obtain statement {\em (ii)}.
Let $k\geq j_1$, and let $X_k = \SET{1,2,\ldots,j_{2}-1}\cup\{k\}$.
 Lemma~\ref{LPMlem:A} yields that $\mathrm{rk}(X_k)=\mathrm{rk}(E)-1$, thus $\cl(X_k)$ is a copoint on the Western coline $W$.
 It remains to show that $\cl(X_k) = X_k$, which implies that $X_k$ is indeed a simple copoint on $W$.
 We prove this fact by showing that for all $k' \geq j_1$, $\rk(X_k\cup\SET{k'}) = \rk(E)$
  by constructing a lattice path. Without loss of generality we may assume that $k < k'$.
  Since $M$ has no loops and no parallel edges, there is a lattice path $r=(r_{i})_{i=1}^{n}\in\mathrm{P}[p,q]$
  with $r_k = r_{k'} = \mathrm{N}$.
% The only thing left to show is that for any $k\in\SET{1,2,\ldots,n}$ with $k\geq j_{1}$,
% $\mathrm{cl}(\SET{1,2,\ldots,j_{2}-1}\cup\{k\})=\SET{1,2,\ldots,j_{2}-1}\cup\{k\}$, i.e. that those
% sets are simple copoints. Again, it suffices to construct a path that
% shows that for any $k'\in\SET{1,2,\ldots,n}\backslash(\SET{1,2,\ldots,j_{2}-1}\cup\{k\})$, $\mathrm{rk}\left(\SET{1,2,\ldots,j_{2}-1}\cup\{k,k'\}\right)=\mathrm{rk}(\SET{1,2,\ldots,j_{2}-1})+2$.
% Possibly switching $k$ and $k'$, we may assume that $k<k'$. Since
% $M$ has neither parallel edges nor loops, there is a lattice path
% $R=(r_{i})_{i=1}^{n}\in\mathrm{P}[p,q]$ such that $\{k,k'\}\subseteq R^{-1}(\mathrm{N})$. 
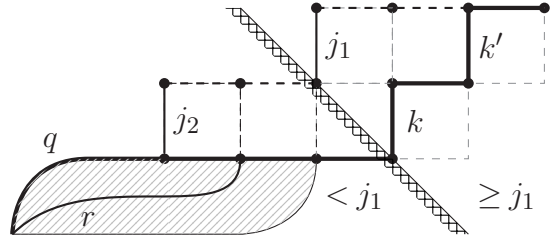
\begin{figure}[t]
\begin{center}
\begin{tikzpicture}
\node[inner sep=0pt,circle,minimum size=4pt,fill] (j) at (0,0) {};
\node[inner sep=0pt,circle,minimum size=4pt,fill] at (1,0) {};
\node[inner sep=0pt,circle,minimum size=4pt,fill] at (0,1) {};
\node[inner sep=0pt,circle,minimum size=4pt,fill] at (1,1) {};
\node[inner sep=0pt,circle,minimum size=4pt,fill] at (2,1) {};
\node[inner sep=0pt,circle,minimum size=4pt,fill] at (3,0) {};
\node[inner sep=0pt,circle,minimum size=4pt,fill] at (3,1) {};
\node[inner sep=0pt,circle,minimum size=4pt,fill] at (4,1) {};

\node[inner sep=0pt,circle,minimum size=4pt,fill] at (2,0) {};
\node[inner sep=0pt,circle,minimum size=4pt,fill] at (2,2) {};
\node[inner sep=0pt,circle,minimum size=4pt,fill] at (5,2) {};
\node[inner sep=0pt,circle,minimum size=4pt,fill] at (4,2) {};
\node[inner sep=0pt,circle,minimum size=4pt,fill] at (3,2) {};

%\draw [dashed,color=gray] (j) -- (-2,0);
%\draw [dashed,color=gray] (j) -- (0,-1);
%\draw [dashed,color=gray] (k) -- (5,1);
%\draw [dashed,color=gray] (k) -- (3,2);
%\draw [dashed,color=gray] (3,0) -- (4,0) -- (4,1);
%\draw [dashed,color=gray] (0,1) -- (-1,1) -- (-1,0);

%\fill[pattern color = gray, pattern = north east lines] (-1,0) .. controls (-1,-1) and (-2,-1) .. (-2,-1) -- (-1,-1) .. controls (-1,-1) and (-1,0) .. (0,0);
\begin{scope}[shift={(-1,0)}]
\draw[ultra thick]   (-1,-1) .. controls (-1,-1) and (-1,0) .. (0,0) -- (1,0);
\draw  (0,0) -- (3,0) .. controls (3,-1) and (2,-1) .. (2,-1)  -- (-1,-1) .. controls (-1,-1) and (-1,0) .. (0,0);
\fill[pattern color = lightgray, pattern = north east lines] (0,0) -- (3,0) .. controls (3,-1) and (2,-1) .. (2,-1)  -- (-1,-1) .. controls (-1,-1) and (-1,0) .. (0,0);
\draw[thick] (-1,-1) .. controls (0,0) and (2,-1).. (2,0);
\end{scope}

%\draw[] (j)--(0,1)--(1,1)--(1,0)--(j);
\draw[thick] (j)--(0,1);

%\draw[ultra thick] (k)--(3,0);
%\draw[ultra thick] (k)--(2,1);
\begin{scope}[shift={(2,0)}]
\draw[dashed,color=gray] (0,0)--(0,1)--(1,1)--(1,0)--(0,0);
\end{scope}
\begin{scope}[shift={(2,1)}]
\draw[dashed,color=gray] (0,0)--(0,1)--(1,1)--(1,0)--(0,0);
\end{scope}
\begin{scope}[shift={(3,0)}]
\draw[dashed,color=gray] (0,0)--(0,1)--(1,1)--(1,0)--(0,0);
\end{scope}
\begin{scope}[shift={(3,1)}]
\draw[dashed,color=gray] (0,0)--(0,1)--(1,1)--(1,0)--(0,0);
\end{scope}
\begin{scope}[shift={(4,1)}]
\draw[dashed,color=gray] (0,0)--(0,1)--(1,1)--(1,0)--(0,0);
\end{scope}

\draw[dashed] (0,0)--(0,1)--(1,1)--(1,0)--(0,0);
\begin{scope}[shift={(1,0)}]
\draw[dashed] (0,0)--(0,1)--(1,1)--(1,0)--(0,0);
\end{scope}

%\node at (.7,-.4) {$j$}; 
%\node at (2.3,1.4) {$k$}; 
\node at (.3,.5) {$j_2$}; 
\node at (2.3,1.5) {$j_1$}; 

\node at (4.3,1.5) {$k'$}; 

\node at (3.3,.5) {$k$};
\draw [thick,dashed] (2,1) -- (0,1);
\draw [thick] (2,1)--(2,2);
\draw [thick,dashed] (2,2) --(5,2);

\draw [ultra thick] (0,0) -- (3,0) -- (3,1) -- (4,1) -- (4,2) -- (5,2);

\node at (-1.5,.2) {$q$};
\node at (-1,-.8) {$r$};
\draw (1,2) -- (4,-1);
\fill[pattern color=black,pattern = grid] (1,2) -- (4,-1) -- (3.8,-1) -- (.8,2);
\node at (2.5,-.5) {$<j_1$};
\node at (4.5,-.5) {$\ge j_1$};
\end{tikzpicture}
\end{center}\label{fig:WesternColine}
\caption{Construction of the lattice path $s$ in the proof of Theorem~\ref{thm:WesternColine}.}
\end{figure}
There is a lattice path $s=(s_{i})_{i=1}^{n}\in\mathrm{P}[p,q]$  that follows $q$ for
the first $j_{2}-1$ steps, then goes towards the East until it meets
$r$, and then goes on as $r$ does. The constructed
path $s$ yields that 
\begin{align*}
  \mathrm{rk}(X_k\cup\{k'\}) & \geq\left|\left(W\cup\{k,k'\}\right)\cap \SET{i\in E\mid s_i=\mathrm{N}} \right|
  \\ & =2+\left|W\cap \SET{i\in E\mid q_i =\mathrm{N}}\right|
\\& 
  =2+\mathrm{rk}(W) = 1 + \rk(X_k) = 1+ \rk(X_{k'}),
\end{align*}
where $X_k' = W\cup\SET{k'}$. Thus $k'\notin \cl(X_k)$ and $k\notin \cl(X_k')$. This
 completes the proof of statement {\em (iii)}.
\end{proof}

\begin{theorem}\PRFR{Mar 7th}\label{thm:simplelpmqsc}
Let $M=(E,\Ical)$ be a strong lattice path matroid with $\rk_M(E) \geq 2$ such that 
$\left| E \right| = n$ and such that $M$ has neither a loop nor a pair of parallel edges.
Then either the Western
coline is quite simple, or the element $n\in E$ is a coloop, and in the latter case
there is either another
coloop or $\rk_M(E) \ge3$.\end{theorem}

\begin{proof}\PRFR{Mar 7th}
If $j_{1}\leq n-1$ as defined in Theorem~\ref{LPMthm:B}, $W=\SET{1,2,\ldots,j_{2}-1}$
has at most a single multiple copoint and at least two simple copoints, therefore
it is quite simple. Otherwise $j_{1}=n$ is a coloop. If there is another
coloop $e_{1}$, then $\SET{1,2,\ldots,n-1}\backslash\{e_{1}\}$ is a quite simple coline
with two simple copoints. If $n$ is the only coloop, the rank of $M$ is $2$, and
there is no other coloop, then this would imply that there are parallel edges --- a contradiction to the assumption that $M$ is a simple matroid.
\end{proof}

\begin{corollary}\label{cor:LPMqscoline}\PRFR{Mar 7th}
Every simple lattice path matroid $M=(E,\Ical)$ with $\rk_M(E)\geq 2$ has a quite simple
coline.\end{corollary}
\begin{proof}
Without loss of generality, we may assume that $M$ is a strong lattice path matroid on $E=\SET{1,2,\ldots,n}$,
and we may use $j_1$ and $j_2$ as defined in Theorem~\ref{LPMthm:B}.
From Theorem~\ref{thm:simplelpmqsc}, we obtain the following:
If $j_1 < n$, the Western coline is quite simple. Otherwise, if $j_1=n$, then $n$ is a coloop.
If there is another coloop $e_{1}$, then $\SET{1,2,\ldots,n-1}\backslash\{e_{1}\}$ is a quite simple coline.
If there is no other coloop, then we have $\rk_M(E) \geq 3$,
 and the contraction $M' = M\contract E\BSET{n}$ is a strong lattice path matroid
without loops, without parallel edges, and
without coloops, such that $\rk_{M'}(E\BSET{n})=\rk_M(E) - 1 \geq 2$. Thus the corresponding $j_1' < n -1$ 
and the Western coline $W'$ of $M'$ is quite simple in $M'$ (Theorem~\ref{thm:simplelpmqsc}).
But then $\tilde W = W'\cup\SET{n}$ is a coline of $M$, and $\tilde X$ is a copoint on $\tilde W$ with respect to $M$
if and only if $X' = \tilde X \BSET{n}$ is a copoint on $W'$ with respect to $M'$. 
Since $\left| \tilde W \BS \tilde X \right| = \left| W' \BS X' \right|$, we obtain that $\tilde W$ is a quite simple coline of $M$.
\end{proof}

\needspace{4\baselineskip}
\begin{definition}[\cite{GoHoNe15}, Definition~2] Let $\mathcal{O}$ be an oriented
matroid. We say that $\mathcal{O}$ is \deftext[generalized series-parallel oriented matroid]{generalized series-parallel},
if every non-trivial minor $\Ocal'$ of $\mathcal{O}$ with a simple underlying matroid $M(\Ocal')$ has a $\{0,\pm1\}$-valued
coflow which has exactly one or two nonzero-entries.\end{definition}

\begin{lemma}[\cite{GoHoNe15}, Lemma~5]\label{lem:quiteSimpleColineQGSP}\PRFR{Mar 7th}
If an orientable matroid $M$ has a quite simple
coline, then every orientation $\mathcal{O}$ of $M$ has a $\{0,\pm1\}$-valued
coflow which has exactly one or two nonzero-entries.
\end{lemma}

\noindent For a proof, see \cite{GoHoNe15}.

\begin{remark}\label{rem:GPSloRank}\PRFR{Mar 7th}
  A simple matroid of rank $1$ has only one element, no circuit and a single cocircuit consisting of the sole element of the matroid;
  so every rank-$1$ oriented matroid is generalized series-parallel.
  Observe that every simple matroid $M=(E,\Ical)$ with $\rk_M(E) = 2$ is a lattice path matroid,
  as it is isomorphic to the
  strong lattice path matroid $M[p,q]$ where $p=(p_i)_{i=1}^{\left| E \right|}$ 
  with \[ p_i = \begin{cases} \mathrm{E} &\quad \text{if~} i < \left| E \right| - 2, \\
                              \mathrm{N} &\quad \text{otherwise,} \end{cases}\]
  and where $q=(q_i)_{i=1}^{\left| E \right|}$ 
  with \[ q_i = \begin{cases} \mathrm{N} &\quad \text{if~} i \leq 2,\\
                              \mathrm{E} &\quad \text{otherwise.} \end{cases}\]
  Therefore Lemma~\ref{lem:quiteSimpleColineQGSP} and Corollary~\ref{cor:LPMqscoline} yield that $\Ocal$ has a $\{0,\pm1$\}-valued
coflow which has exactly one or two nonzero-entries.
Consequently, every oriented matroid $\Ocal =(E,\Ccal,\Ccal^\ast)$ with
  $\rk_{M(\Ocal)}(E) \leq 2$ is generalized series-parallel.    
\end{remark}

\begin{theorem}[\cite{Bonin2006701}, Theorem~3.1]\label{thm:LPMclosedUnderStuff}\PRFR{Mar 7th}
The class of lattice path matroids
is closed under minors, duals and direct sums.\end{theorem}

\PRFR{Mar 7th}
\noindent For a proof, see \cite{Bonin2006701}, pp. 5ff; furthermore see Figure~\ref{fig:LPMC} for
examples of the constructions involved.
The main observation is that a base $B$ of a strong lattice path matroid $M$ corresponds to
a lattice path $r$ for which the $k$th step is to the North if and only if $k\in B$. So the base $E\BS B$
of $M^\ast$ corresponds to a lattice path $r^\ast$ for which the $k$th step is to the East if and only if $k\in B$,
but this lattice path $r^\ast$ is the path $r$ mirrored at South-West-to-North-East line through the origin. So the dual of
a lattice path matroid $M$ is the lattice path matroid $M^\ast$ where the upper bound lattice path
of $M^\ast$ is the SW-NE-mirror image of the lower bound lattice path of $M$,
and the lower bound lattice path of $M^\ast$ is the SW-NE-mirror image of the upper bound lattice path of $M$.
The direct sum of lattice path matroids has lower and upper bound lattice paths that correspond to the concatenation of the lower bound lattice paths -- and upper bound lattice paths,
respectively -- of the summand lattice path matroids.
In order to obtain the restriction $N=M \restrict E\BSET{e}$ of a lattice path matroid $M=(E,\Ical)$ with $e\in E$,
we have to take all lattice paths between the lower and upper bound lattice path of $M$, and remove the step corresponding to $e$.
Thus the lower and upper bound lattice paths of $N$ arise from the step-$e$-omissions of lattice paths 
that may differ from the upper and lower bound lattice paths of $M$ only at the step corresponding to $e$ and at most one other step.

\begin{figure}[t]
\begin{center}
\includegraphics[scale=.9]{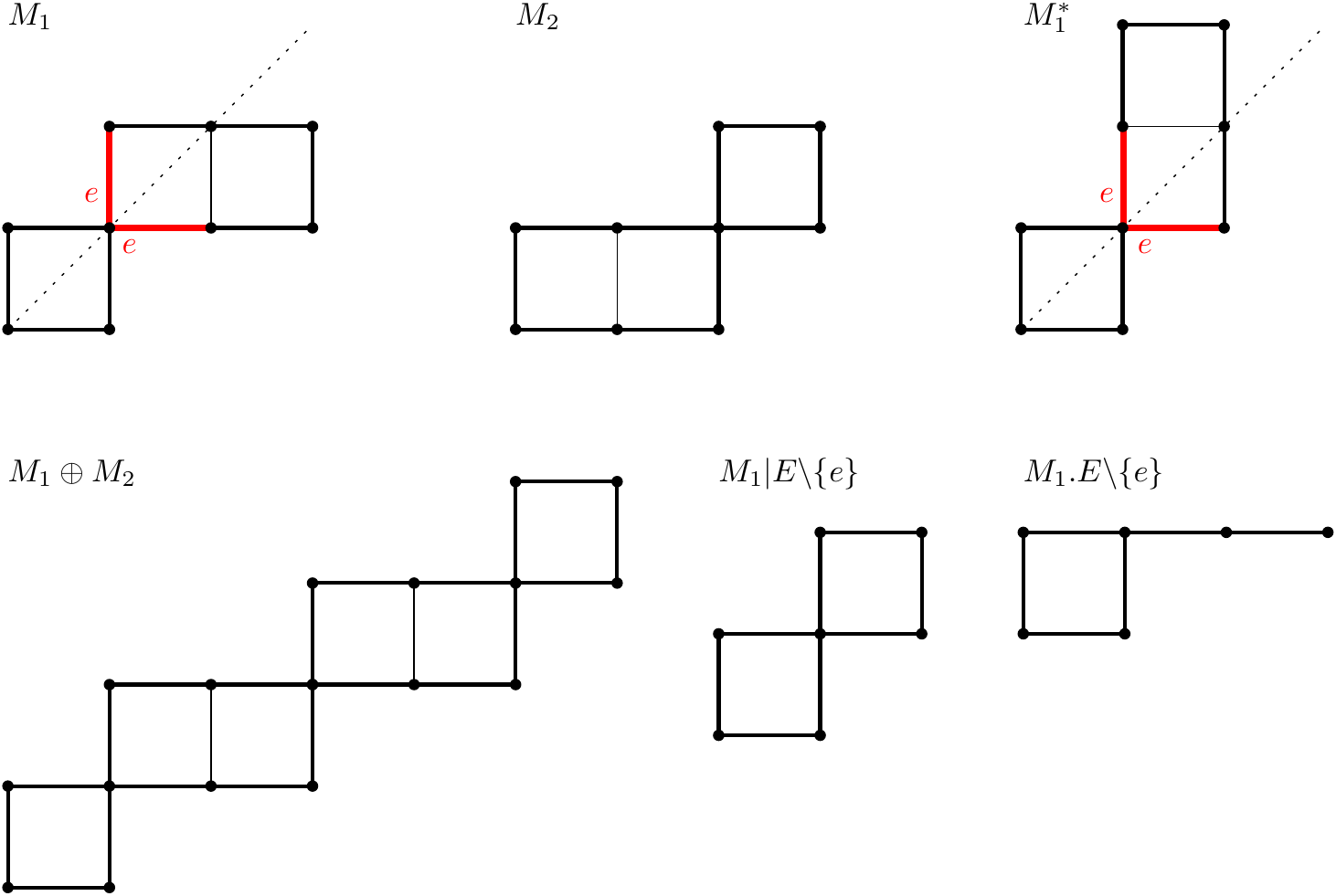}
\end{center}
\caption{\label{fig:LPMC}Constructions on Lattice Path Matroids}
\end{figure}

\needspace{6\baselineskip}

\begin{corollary}\label{cor:LPMGSP}\PRFR{Mar 7th}
All orientations of lattice path matroids are generalized series-parallel.
\end{corollary}
\begin{proof}
Lemma~\ref{lem:quiteSimpleColineQGSP}, Remark~\ref{rem:GPSloRank}, Theorem~\ref{thm:LPMclosedUnderStuff} and Corollary~\ref{cor:LPMqscoline}.
\end{proof}

\needspace{6\baselineskip}
\begin{theorem}[\cite{GoHoNe15}, Theorem 3]\PRFR{Mar 7th}
\label{thm:GPS3col} Let $\Ocal=(E,\Ccal,\Ccal^\ast)$ be a generalized series-parallel oriented
matroid such that $M(\Ocal)$ has no loops.
Then there is a nowhere-zero coflow $F\in \Z.\Ccal^\ast$ such that $\left| F(e) \right| < 3$ for all $e\in E$.
Thus $\chi(\Ocal) \leq 3$.
\end{theorem}

\noindent For a proof, see \cite{GoHoNe15}.

\begin{corollary}\PRFR{Mar 7th}
  Let $\Ocal$ be an oriented matroid such that $M(\Ocal)$ is a lattice path matroid without loops.
  Then \( \chi(\Ocal) \leq 3 \).
\end{corollary}
\begin{proof}
Theorem~\ref{thm:GPS3col} and Corollary~\ref{cor:LPMGSP}.
\end{proof}
\clearpage
% -*- root: ../thesis.tex -*-

\needspace{6\baselineskip}

\section{Oriented Gammoids}

In this section, we examine the class of oriented matroids whose underlying matroids are gammoids.

\begin{lemma}\PRFR{Mar 7th}\label{lem:gammoidOrientable}
	Let $M=(E,\Ical)$ be a gammoid. Then $M$ is orientable.
\end{lemma}
\begin{proof}\PRFR{Mar 7th}
	By Theorem~\ref{thm:gammoidOverR} there is a set $T$ with $\left| T \right| = \rk_M(E)$ and there is a matrix $\mu\in \R^{E\times T}$,
	such that $M=M(\mu)$. Then the oriented matroid $\Ocal(\mu)$ is an orientation of $M(\mu)=M$ (Corollary~\ref{cor:MOmuEQMmu}).
\end{proof}

\PRFR{Mar 7th}
\noindent Given a gammoid $M=\Gamma(D,T,E)$, the oriented matroid, whose existence is guaranteed by the previous lemma,
 depends on the actual values of the
 indeterminate weighting $w\colon A\maparrow \R$ of $D$, and obtaining the signatures of the circuits from the matrix
 $\mu$ requires some computational effort. The same applies to the integer-valued representation obtained 
 from the probabilistic method
 described in Proposition~\ref{prop:randompolytimeTransversalMatroid}, where it is possible to use E.H.~Bareiss's variant of 
 Gaussian Elimination \cite{Ba68} that works without division, and which has polynomial computational complexity and a 
 polynomial bound on the absolute of the intermediate values that may occur during the calculation.
 We further point out that Example~\ref{ex:UniformNonRealizable} indicates that there are orientations of gammoids
 which cannot be represented by a real matrix.

\begin{lemma}\label{lem:CRAMERsrule}\PRFR{Mar 7th}
	Let $E$ and $T$ be finite sets, and let $\mu\in \R^{E\times T}$ be a matrix, 
	and $M = M(\mu)$ be the matroid represented by $\mu$ over $\R$. 
	Further, let $\Ocal = (E,\Ccal,\Ccal^\ast) = \Ocal(\mu)$
	be the oriented matroid obtained from $\mu$,
	let $C\in \Ccal(M)$ be a circuit of $M$ and let $c\in C$ be an arbitrary element of that circuit.
	Let $T_0\subseteq T$ such that $\idet(\mu \restrict (C\BSET{c})\times T_0) = 1$.
	Consider the signed subset $C_c \in \sigma E$ with
	 \[ C_c(e) = \begin{cases}[r]
	 	0 & \quad \text{if~} e\notin C,\\
	 	-1 & \quad \text{if~} e = c,\\
	 	\sgn\left( \frac{\det (\nu_e)}{\det(\mu \restrict (C\BSET{c})\times T_0)}\right) & \quad \text{otherwise}
	 \end{cases} \]
	 where \[ \nu_e \colon C\BSET{c}\times T_0 \maparrow \R,\quad (x,t)\mapsto \begin{cases}
	 		\mu(c,t) & \quad \text{if~} x = e,\\
	 		\mu(x,t) & \quad \text{otherwise.}
	 \end{cases} 
	 \]
	 Then $C_c \in \Ccal$.
\end{lemma}

\begin{proof}\PRFR{Mar 7th}
	By \textsc{Cramer}'s rule we obtain that
	\[ \mu_c = \sum_{e\in C\BSET{c}} \frac{\det (\nu_e)}{\det(\mu \restrict (C\BSET{c})\times T_0)} \cdot \mu_e. \]
	Therefore, 
	\[ -\mu_c + \sum_{e\in C\BSET{c}} \frac{\det (\nu_e)}{\det(\mu \restrict (C\BSET{c})\times T_0)} \cdot \mu_e = 0 \]
	is a non-trivial linear combination of the zero vector. Clearly $C_c$ consists of the signs of the corresponding coefficients and therefore $C_c\in \Ccal$ is an orientation of $C$ with respect to $\Ocal(\mu)$. 
\end{proof}

% -*- root: ../thesis.tex -*-

 \subsection{Heavy Arc Orientations}

\PRFR{Mar 7th}
 In this section, we develop a notion of orientations of gammoids which stem from indeterminate weightings with
 a special property, that allows us to determine the signed circuits of the orientation without carrying out any computations in $\R$.
 Instead, we only have to inspect a given representation $(D,T,E)$
 with respect to the given linear order on $A$ and the given signs of the arc weighting.

 %\begin{definition}\PRFR{Mar 7th}
 %	A tuple $(D,T,E,S_A,\ll)$\label{n:rCOG} is called \deftext{realizable combinatorial orientation of a gammoid},
 %	if
 %	\begin{enumerate}\ROMANENUM
 %		\item $D=(V,A)$ is a\remred{n acyclic?} digraph, $T\subseteq V$, $E\subseteq V$,
 %		\item \( S_A\colon A\maparrow \SET{-1, 1} \) is a signature of $A$, and
 %  	\item $\ll\,\,\subseteq A\times A$ is a linear order on $A$.\label{n:ll}
 %	\end{enumerate}
%	We further assert that every realizable combinatorial orientation of a gammoid
%	has implicit orders on the sets $E=\dSET{e_1,e_2,\ldots,e_s}$ and
%	$T = \dSET{t_1,t_2,\ldots,t_r}$ which can be chosen arbitrarily as long as
%	this is done in a consistent way across different realizable combinatorial orientations and gammoids.%
%	\footnote{We could make these two orders part of the definition, but a realizable combinatorial orientation of a gammoid
%	already has five components, and changes in the two orders do not yield any genuinely new orientations of gammoids.}
 %\end{definition}
 \begin{definition}\label{def:heavyArcSignature}\PRFR{Mar 7th}
 	Let $D=(V,A)$ be a digraph, let $\sigma\colon A\maparrow \SET{-1,1}$ be a map and let $\ll$ be a binary relation on $A$.
 	We shall call $(\sigma, \ll)$ \label{n:sigmaLL} a \deftext[heavy arc signature of a digraph]{heavy arc signature of $\bm D$}, 
 	if $\ll$ is a linear order on $A$.
 \end{definition}

 %\begin{definition}
 %	Let $D=(V,A)$ be a digraph, $X,Y\subseteq V$, $\ll\subseteq A\times A$ be a linear order on the arcs of $D$, and
 %	let $L_1,L_2\colon X\routesto Y$ be two distinct linkings from $X$ to $Y$ in $D$.
 %	Then $L_1$ shall be called \deftext[lighter linking]{lighter} than $L_2$,
 %	if there is some $a_2 \in A_2\BS A_1$ such that for all
 %	$a_1 \in A_1\BS A_2$, $a_1 \ll a_2$ holds;
 %	where $A_1 = \bigcup_{p\in L_1} \left| p \right|_A$ and
 %	$A_2 = \bigcup_{p\in L_2} \left| p \right|_A$.
 %	In other words, $L_1$ is lighter than $L_2$ if there is a path in $L_2$, which traverses an arc $a_2$ that is 
 %	not
 %	traversed by a path in $L_1$, such that every arc $a_1$ traversed by a path in $L_1$
 %	but not by a path in $L_2$ is smaller with respect to $\ll$.
 %	In this case, we shall write $L_1 \llless L_2$.\label{n:llless}
 %\end{definition}
 \begin{definition}\label{def:routingOrder}\PRFR{Mar 7th}
 	Let $D=(V,A)$ be a digraph and $(\sigma, \ll)$ be a heavy arc signature of $D$.
 	The \deftext[induced routing order]{$\bm( \bm \sigma \bm, \bm \ll \bm)$-induced routing order of $\bm D$}
 	shall be the linear order $\lll$ on the family of routings of $D$, where $Q \lll R$ holds\label{n:llless}
 	if and only if the $\ll$-maximal element $x$ of the symmetric difference $Q_A \bigtriangleup R_A$ has
 	the property $x\in R_A$, where $Q_A = \bigcup_{p\in Q} \left| p \right|_A$ and $R_A = \bigcup_{p\in R} \left| p \right|_A$.
 \end{definition}

 \begin{remark}\PRFR{Mar 7th}
 	Clearly, $\llless$ is a linear order on all routings in $D$, because every routing $R$ in $D$ is uniquely determined by its set of
 	traversed arcs $R_A$ (Lemmas~\ref{lem:linkage} and \ref{lem:routage}).
 \end{remark}

\begin{definition}\PRFR{Mar 7th}
	Let $D=(V,A)$ be a digraph, and let $(\sigma,\ll)$ be a heavy arc signature of $D$.
	Let $R\colon X\routesto Y$ be a routing in $D$ where $X=\dSET{x_1,x_2,\ldots,x_n}$
	and $Y=\dSET{y_1,y_2,\ldots,y_m}$ are implicitly ordered.
	The \deftext[sign of a routing]{sign of $\bm R$ with respect to $\bm ( \bm \sigma\bm,\bm \ll \bm)$} shall be\label{n:sgnsigma}
	\[ \sgn_\sigma (R) = \sgn(\phi) \cdot \left( \prod_{p\in R,\,a\in \left| p \right|_A} \sigma (a)  \right) \]
	where $\phi\colon \SET{1,2,\ldots,n}\maparrow \SET{1,2,\ldots,m}$ is the unique map such that 
	for all $i \in \SET{1,2,\ldots,n}$ there is a path $p\in R$
	with $p_1 = x_i$ and $p_{-1} = y_{\phi(x)}$; and where
	\[ \sgn(\phi) = {(-1)}^{\left| \SET{\vphantom{A^A}(i,j)~\middle|~i,j\in \SET{1,2,\ldots,n}\colon\, i < j \txtand \phi(i) > \phi(j)}\right|} . \qedhere\]
\end{definition}

\begin{definition}\label{def:Csigmac}\PRFR{Mar 7th}
	Let $D=(V,A)$ be a digraph such that $V=\dSET{v_1,v_2,\ldots,v_n}$ is implicitly ordered,
	 $(\sigma,\ll)$ be a heavy arc signature of $D$, and let $T,E\subseteq V$
	be subsets that inherit the implicit order of $V$.
	Furthermore, let  $M=\Gamma(D,T,E)$ be the corresponding gammoid, and let $C\in \Ccal(M)$
	be a circuit of $M$ 
	 such that $C = \dSET{c_1,c_2,\ldots,c_m}$ inherits its implicit order from $V$;
	 and let $i\in\SET{1,2,\ldots,m}$. 
%	order of $E$ with respect to $\sigma$,
%	and $c_i\in C$. 
%
	The \deftext[heavy arc circuit signature]{signature of $\bm C$ with respect to $\bm M$, $\bm i$, and $\bm (\bm \sigma \bm, \bm \ll \bm)$}
	shall be the signed subset $C_{(\sigma,\ll)}^{(i)}$ of $E$ where
%	is defined to be $C_{\sigma c_i} \in \sigma E$ where for all $e\in E$\label{n:Csigmac}
	\[ C_{(\sigma,\ll)}^{(i)}(e) = \begin{cases}[r]
							0 &\quad \text{if~}e\notin C,\\
							- \sgn_{\sigma}(R_{i}) &\quad \text{if~}e = c_i,\\
							(-1)^{i-j+1} \cdot \sgn_{\sigma}(R_{j})   &\quad \text{if~}e = c_j\not= c_i,
						\end{cases} \]
	and where for all $k\in \SET{1,2,\ldots,m}$
	\[ R_k = \max_{\llless} \SET{R \mid R\colon C\BSET{c_k}\routesto T\text{~in~}D} \]
	denotes the unique $\llless$-maximal routing from $C\BSET{c_k}$ to $T$ in $D$.
\end{definition}

\begin{remark}\PRFR{Mar 7th}
	The factors $(-1)^{i-j+1}$ in Definition~\ref{def:Csigmac} do not appear explicitly in Lemma~\ref{lem:CRAMERsrule},
	where
	$\nu_e$ is obtained from the restriction $\mu\restrict (C\BSET{c})\times T_0$ 
	by replacing the values in row $e$ with the values of $\mu_c$.
	We have to account for the number of row transpositions that are needed to turn $\nu_e$ into the restriction $\mu\restrict(C\BSET{e})\times T_0$, which depends on the position of $e=c_j$ relative to $c=c_i$ with respect to the implicit order of $V$.
\end{remark}

\begin{definition}\label{def:heavyArcWeighting}\PRFR{Mar 7th}
	Let $D=(V,A)$ be a digraph and $(\sigma, \ll)$ a heavy arc signature of $D$, and let $w\colon A\maparrow \R$ be an indeterminate weighting of $D$.
	We say that $w$ is a \deftext[heavy arc weighting]{$\bm( \bm\sigma\bm,\bm\ll\bm)$-weighting of $\bm D$} if, for all $a\in A$,
	the inequality $\left| w(a) \right| \geq 1$,
	the strict inequality
	\[\sum_{L\subseteq \SET{x\in A~\middle|~ x \ll a,\,x\not=a}} \left( \prod_{x\in L} \left| w(x)  \right| \right)  < \left| w(a) \right|,  \]
	and the equality 
	\( \sgn(w(a)) = \sigma(a) \)
	hold.
\end{definition}

\needspace{4\baselineskip}
\begin{lemma}\label{lem:ExistenceOfHeavyArcWeighting}\PRFR{Mar 7th}
	Let $D=(V,A)$ be a digraph and $(\sigma, \ll)$ be a heavy arc signature of $D$.
	There is a $(\sigma,\ll)$-weighting of $D$.
\end{lemma}
\begin{proof}\PRFR{Mar 7th}
	Let $w\colon A\maparrow \R$ be an indeterminate weighting of $D$, which exists due to Lemma~\ref{lem:enoughZindependents}.
	It is clear from Definition~\ref{def:Zindependent} that for every $\zeta \in \Z^A$ and every $\tau \in \SET{-1,1}^A$, the map
	$w_{\zeta,\tau}\colon A\maparrow \R$, which has $$w_{\zeta,\tau}(a) = \tau(a)\cdot \frac{w(a)}{\sgn(w(a))} + \tau(a) \cdot \zeta(a)$$ for all $a\in A$, is an indeterminate weighting of $D$, too.
	Now, let $\zeta\in \Z^A$, such that for all $a\in A$ we have the following recurrence relation
	\[ \zeta(a) =  \left\lceil  \sum_{L\subseteq \SET{x\in A ~\middle|~ x \ll a,\,x\not=a}} \left( \prod_{x\in L} \left(\vphantom{A^1} \left| w(x) \right| + \zeta(x) \right) \right) \right\rceil.\]
	The map $\zeta$ is well-defined by this recurrence relation because $\left| A \right| < \infty$ and therefore there is a $\ll$-minimal element $a_0$ in $A$, and we have $\zeta(a_0) = \prod_{x\in \emptyset} \left(\vphantom{A^1} \left| w(x) \right| + \zeta(x) \right) = 1$.
	Then $w_{\zeta,\sigma}$ is a $(\sigma,\ll)$-weighting of $D$. Clearly,
	\begin{align*}
		  \sgn\left( w_{\zeta,\sigma}(a)  \right) & = \sgn\left( \sigma(a) \cdot \frac{w(a)}{\sgn(w(a))} + \sigma(a) \cdot \zeta(a) \right) 
		  \\ & 
		  = \sgn\left( \vphantom{a^1}\sigma(a) \right)\cdot\sgn\left( \frac{w(a)}{\sgn(w(a))} +  \zeta(a) \right)\\& 
		  = \sigma(a) \cdot 1 = \sigma(a)
		  \end{align*}
	holds for all $a\in A$. Furthermore, we have
	\begin{align*}
		\left| w_{\zeta,\sigma}(a) \right| & = \left| \sigma(a) \cdot \frac{w(a)}{\sgn(w(a))} + \sigma(a) \cdot \zeta(a) \right| \\
		& > \left| \zeta(a) \right| %\\
		 = \left\lceil  \sum_{L\subseteq \SET{x\in A ~\middle|~ x \ll a,\,x\not=a}} \left( \prod_{x\in L} \left(\vphantom{A^1} \left| w(x) \right| + \zeta(x) \right) \right) \right\rceil \\
		& \geq  \sum_{L\subseteq \SET{x\in A~\middle|~ x \ll a,\,x\not=a}} \left( \prod_{x\in L} \left|  w_{\zeta,\sigma}(x) \right| \right). \qedhere
	\end{align*}
\end{proof}

\needspace{5\baselineskip}

\begin{lemma}\label{lem:lllMaximalRoutingsHaveCommonEnd}\PRFR{Mar 7th}
	Let $D=(V,A)$ be a digraph, $(\sigma,\ll)$ be a heavy arc weighting of $D$, $E,T\subseteq V$,
	$C\in \Ccal(\Gamma(D,T,E))$ be a circuit in the corresponding gammoid, and let $c,d\in C$.
	Furthermore, let $R_c \colon C\BSET{c} \routesto T$ and $R_d \colon C\BSET{d} \routesto T$ be 
	the $\lll$-maximal routings in $D$.
	Then \[ \SET{p_{-1} ~\middle|~ p\in R_c} = \SET{p_{-1} ~\middle|~ p\in R_d} \]
	holds.
\end{lemma}
\begin{proof}\PRFR{Mar 7th}
	Let $S$ be a $C$-$T$-separator of minimal cardinality in $D$, i.e. a $C$-$T$-separator with
	$\left| S \right| = \left| C \right| -1$. Since $R_c$ and $R_d$ are both $C$-$T$-connectors with maximal cardinality,
	we obtain that for every $s\in S$ there is $p_c^s \in R_c$ and a $p_d^s \in R_d$ such that $s\in \left| p_c^s \right|$ and
	$s\in \left| p_d^s \right|$ (Corollary~\ref{cor:Menger}), thus there are paths $l_c^s,l_d^s,r_c^s,r_d^s\in \Pbf(D)$ such that
	$p_c^s = l_c^s . r_c^s$ and
	\linebreak
	 $p_d^s = l_d^s . r_d^s$ with $\left( {r_c^s} \right)_1 = \left( {r_d^s} \right)_1 = s$. Now let $R_c^S = \SET{r_c^s ~\middle|~ s\in S}$
	and $R_d^S = \SET{r_d^s ~\middle|~ s\in S}$, clearly both $R_c^S$ and $R_d^S$ are routings from $S$ to $T$ in $D$.
	Assume that $R_c^S \not= R_d^S$, then we have $R_c^S \lll R_d^S$ --- without loss of generality, by possibly
	switching names for $c$ and $d$. Then $Q = \SET{l_c^s . r_d^s ~\middle|~ s\in S}$ is a routing from $C\BSET{c}$ to $T$ in $D$.
	But for the symmetric differences we have the equality 
	\[ \left( \bigcup_{p\in Q} \left| p \right|_A \right) \bigtriangleup \left( \bigcup_{p\in R_c} \left| p \right|_A \right) 
	= \left( \bigcup_{p\in R_d^S} \left| p \right|_A \right) \bigtriangleup \left( \bigcup_{p\in R_c^S} \left| p \right|_A \right),
	\]
	which implies $R_c \lll Q$, a contradiction to the assumption that $R_c$ is the $\lll$-maximal routing from $C\BSET{c}$ to $T$.
	Thus $R_c^S = R_d^S$ and the claim of the lemma follows.
\end{proof}

\PRFR{Mar 7th}
\noindent Now we have amassed all ingredients that we need in order to show that every heavy arc signature of $D$ corresponds to an orientation of a gammoid
whenever $D$ is an acyclic digraph. Thus heavy arc signatures yield orientations of cascade matroids.

\needspace{4\baselineskip}
\begin{lemma}\label{lem:acyclicOrientation}\PRFR{Mar 7th}
	Let $D=(V,A)$ be an acyclic digraph where $V$ is implicitly ordered, $(\sigma,\ll)$ be a heavy arc signature of $D$, and $T,E\subseteq V$.
	Then there is a unique oriented matroid $\Ocal=(E,\Ccal,\Ccal^\ast)$ where
	 \[ \Ccal = \SET{ \pm C_{(\sigma,\ll)}^{(1)} ~\middle|~ C\in \Ccal(\Gamma(D,T,E))}.\]
\end{lemma}
\begin{proof}\PRFR{Mar 7th}
	Let $M=\Gamma(D,T,E)$, and let $w\colon A\maparrow \R$ be a $(\sigma,\ll)$-weighting of $D$ which exists due to Lemma~\ref{lem:ExistenceOfHeavyArcWeighting}.
	Furthermore, let $\mu\in \R^{E\times T}$ be the matrix defined as in the Lindström Lemma~\ref{lem:lindstrom}, with respect to the
	$(\sigma,\ll)$-weighting $w$ and the implicit order on $V$. Theorem~\ref{thm:gammoidOverR} along with its proof yields that we have $M = M(\mu)$.
	Let $\Ocal = \Ocal(\mu) = (E,\Ccal_\mu,\Ccal_\mu^\ast)$ be the oriented matroid that arises from $\mu$, so $M(\Ocal) = M(\mu)$ holds (Corollary~\ref{cor:MOmuEQMmu}). We show that $\Ccal_\mu = \Ccal$. It suffices to prove
	 that for all $C\in \Ccal(M)$, all $D \in \Ccal_\mu$ with $D_\pm = C$, and all $D'\in \Ccal$
	with $D_\pm = C$ we have $D \in \SET{D',-D'}$. 
	Now, let $C\in \Ccal(M)$ and let $C=\dSET{c_1,c_2,\ldots,c_k}$ implicitly ordered respecting the implicit order of $V$.
	The claim follows if  $D(c_1)D(c_j) = D'(c_1)D'(c_j)$ holds for all $j\in \SET{2,3,\ldots,k}$.
	Let $T_0 \subseteq T$ be the target vertices onto which the $\lll$-maximal 
	and $\left| \cdot \right|$-maximal $C$-$T$-connectors link in $D$ 
	(Lemma~\ref{lem:lllMaximalRoutingsHaveCommonEnd}).
	From Lemma~\ref{lem:CRAMERsrule} we obtain that
	\begin{align*}
		D(c_1)D(c_j) & = -1\cdot \sgn\left( \frac{\det (\nu_j)}{\det(\mu \restrict (C\BSET{c_1})\times T_0)}\right)\\
		& = -\sgn(\det(\nu_j))\cdot \sgn(\mu \restrict (C\BSET{c_1})\times T_0)) 
	\end{align*}
	where \[ \nu_j \colon C\BSET{c_1}\times T_0 \maparrow \R,\quad (x,t)\mapsto \begin{cases}[r]
	 		\mu(c_1,t) & \quad \text{if~} x = c_j,\\
	 		\mu(x,t) & \quad \text{otherwise.}
	 \end{cases} 
	 \]
	 Observe that $\nu_j$ arises from the restriction $\mu\restrict C\BSET{c_j}\times T_0$ by a 
	 row-permutation, which has at most one non-trivial cycle,
	 and this cycle then has the length $j-1$, therefore $$\det(\nu_j) = (-1)^{j-2} \det \left( \mu\restrict C\BSET{c_j}\times T_0 \right)$$ holds,
	 so
	 \begin{align*}
		D(c_1)D(c_j) & =  (-1)^{1-j}\sgn\left( \det \left( \mu\restrict C\BSET{c_j}\times T_0 \right) \right)\cdot \sgn(\mu \restrict (C\BSET{c_1})\times T_0)).
	\end{align*}
	 We further have
	 \begin{align*}
	 	D'(c_1)D'(c_j) & = (-1)^{j+1}\cdot \sgn_{\sigma}(R_1)\cdot \sgn_{\sigma}(R_j)
	 \end{align*}
	 where for all $i\in \SET{1,2,\ldots,k}$
	\[ R_i = \max_{\llless} \SET{R \mid R\colon C\BSET{c_i}\routesto T\text{~in~}D} \]
	denotes the unique $\llless$-maximal routing from $C\BSET{c_i}$ to $T$ in $D$.
	By the Lindström Lemma~\ref{lem:lindstrom} we obtain that for all $i\in\SET{1,2,\ldots,k}$
	the equation
	\begin{align*}
		\det\left( \mu\restrict C\BSET{c_i}\times T_0 \right) & =  \sum_{R\colon C\BSET{c_i}\routesto T_0} \left( \sgn(R)  
 		\prod_{p\in R} \left( \prod_{a\in \left| p \right|_A} w(a) \right) \right) 
 	\end{align*}
 	holds,
	where $\sgn(R)$ is the sign of the permutation implicitly given by the start and end vertices of the paths in $R$, 
	both with respect to the implicit order on $V$.
	Since $w$ is a $(\sigma,\ll)$-weighting, we have
	\[ \left| \sum_{R\colon C\BSET{c_i}\routesto T_0,\,R\not= R_i} \left( \sgn(R)  
 		\prod_{p\in R} \left( \prod_{a\in \left| p \right|_A} w(a) \right) \right) \right| < \left| w(a_i)  \right| \]
 	where $a_i \in \bigcup_{p\in R_i} \left| p \right|_A$ is the $\ll$-maximal arc in the $\lll$-maximal routing $R_i$ from
 	$C\BSET{c_i}$ to $T_0$ in $D$. Therefore the sign of $\det\left( \mu\restrict C\BSET{c_i}\times T_0 \right)$ is determined
 	by the sign of the summand that contains $w(a_i)$ as a factor, which is the summand that corresponds to $R = R_i$.
	Therefore
	\begin{align*}
		\sgn \left( \det\left( \mu\restrict C\BSET{c_i}\times T_0 \right) \right) & = 
		 \sgn \left( \sgn(R_i)\prod_{p\in R_i,\,a\in \left| p \right|_A} w(a) \right) \\
		 & = \sgn(R_i) \prod_{p\in R_i,\,a\in \left| p \right|_A} \sgn(w(a)) \\
		 & = \sgn(R_i) \prod_{p\in R_i,\,a\in \left| p \right|_A} \sigma(a) \\
		 & = \sgn_\sigma (R_i).
	\end{align*}
	So we obtain
	\begin{align*}
		D(c_1)D(c_j) & =  (-1)^{1-j} \sgn_\sigma(R_1) \cdot \sgn_{\sigma}(R_j) \\
		             & =   (-1)^{j+1}\cdot \sgn_{\sigma}(R_1)\cdot \sgn_{\sigma}(R_j) = D'(c_1)D'(c_j). \qedhere
	\end{align*}
\end{proof}

\noindent Unfortunately, we cannot omit the assumption that $D$ is an acyclic digraph.

\needspace{6\baselineskip}

\vspace*{-\baselineskip} %Remove the line space created by the tilde below
\begin{wrapfigure}{r}{6cm}
\vspace{\baselineskip}
\begin{centering}~~%move the picture slightly to the right
\includegraphics{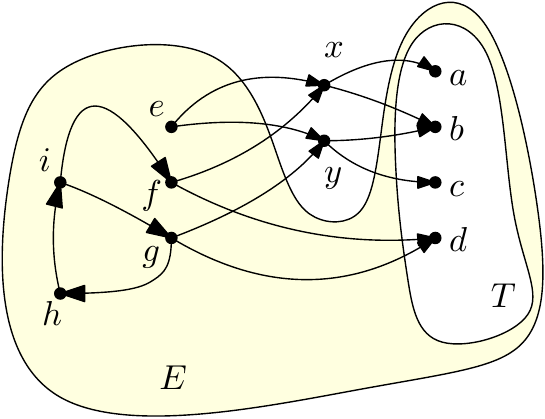}
\end{centering}%
\vspace*{-1\baselineskip} %make the picture more tightly cropped
\end{wrapfigure}
~ %The tilde creates a new dummy paragraph. WHY IS THAT NEEDED? -> would increase the space %
  % before the ex. environment. THE NEXT FREE LINE IS ESSENTIAL!

\begin{example}
	We consider the digraph \linebreak $D=(V,A)$ with the implicitly ordered vertex set $V=\dSET{a,b,c,d,e,f,g,h,i,x,y}$, 
	and $A$ as depicted on the right. Let $T=\SET{a,b,c,d}$.
	Clearly, $\Wbf(D)$ contains the cycle walk $ghig$. Let $(\sigma,\ll)$ be the heavy arc signature of $D$ where
	$\sigma(a) = 1$ for all $a\in A$, and where $a_1 \ll a_2$ if the tuple $a_1$ is less than the tuple $a_2$ with respect to the
	lexicographic order on $V\times V$ derived from the implicit order of the vertex set.
	Let $C_1 = \SET{f,g,i}$, $C_2 = \SET{d,e,f,i}$, $C_{f} = \SET{d,e,g,i}$. Clearly $C_1,C_2,C_{f}\in \Ccal(\Gamma(D,T,E))$.
	The following routings are $\lll$-maximal among all routings in $D$ with the same set of initial vertices and with targets in $T$.
	\begin{align*}
		R_{\SET{f,g}} & = \SET{fxb, gyc}  & \sgn_{\sigma}\left( R_{\SET{f,g}} \right) & = +1\\
		R_{\SET{f,i}} & = \SET{fxb, igyc} & \sgn_{\sigma}\left( R_{\SET{f,i}} \right) & = +1\\
		R_{\SET{g,i}} & = \SET{gyc, ifxb} & \sgn_{\sigma}\left( R_{\SET{g,i}} \right) & = -1\\
		R_{\SET{d,e,f}} & = \SET{d,eyc,fxb} & \sgn_{\sigma}\left( R_{\SET{d,e,f}} \right) & = -1\\
		R_{\SET{d,e,i}} & = \SET{d,exb,igyc} & \sgn_{\sigma}\left( R_{\SET{d,e,i}} \right) & = +1\\
		R_{\SET{d,f,i}} & = \SET{d,fxb,igyc} & \sgn_{\sigma}\left( R_{\SET{d,f,i}} \right) & = +1\\
		R_{\SET{e,f,i}} & = \SET{exb,fd,igyc} & \sgn_{\sigma}\left( R_{\SET{e,f,i}} \right) & = -1\\
		R_{\SET{d,e,g}} & = \SET{d,eyc,ghifxb} & \sgn_{\sigma}\left( R_{\SET{d,e,g}} \right) & = -1\\
		R_{\SET{d,g,i}} & = \SET{d,gyc,ifxb} & \sgn_{\sigma}\left( R_{\SET{d,g,i}} \right) & = -1\\
		R_{\SET{e,g,i}} & = \SET{exb,gyc,ifd} & \sgn_{\sigma}\left( R_{\SET{e,g,i}} \right) & = +1\\
	\end{align*}
	Now let us calculate the signatures of $C_1$, $C_2$, and $C_f$ according to Definition~\ref{def:Csigmac}.
	We obtain \[	
		\left( C_{1} \right)^{(1)}_{(\sigma,\ll)} = \SET{f,g,-i}
		,\,
		\left( C_{2} \right)^{(1)}_{(\sigma,\ll)} = \SET{d,e,-f,-i}
		, \txtand 
		\left( C_{f} \right)^{(1)}_{(\sigma,\ll)} = \SET{-d,-e,-g,-i}.
		\]
	This clearly violates axiom {\em ($\Ccal$4)}: if we eliminate $f$ from $\left( C_{1} \right)^{(1)}_{(\sigma,\ll)}$ 
	and $\left( C_{2} \right)^{(1)}_{(\sigma,\ll)}$, then the resulting signed circuit must have opposite signs for $d$ and $i$,
	but $d$ and $i$ have the same sign with respect to $\left( C_{f} \right)^{(1)}_{(\sigma,\ll)}$. Therefore we see that
	the assumption, that $D$ is
	acyclic, cannot be dropped from Lemma~\ref{lem:acyclicOrientation}.
\end{example}

\noindent We can still use the construction involved in Lemma~\ref{lem:acyclicOrientation} for every representation $(D,T,E)$,
but we first have to construct a complete lifting of $D$ (Lemma~\ref{lem:completelifting}).
We then may use Lemma~\ref{lem:acyclicOrientation}
together with a heavy arc orientation of the lifted digraph in order to obtain an orientation of the lifted representation,
and then use the contraction formula from Lemma~\ref{lem:cyclelifting} in order to obtain the orientation of $(D,T,E)$ (Lemma \ref{lem:OMminors}).
Thus we have found a purely combinatorial way to determine an orientation of a gammoid from its representation,
and the proof of Lemma~\ref{lem:acyclicOrientation} yields that every orientation obtained in this way is realizable.

\cleardoublepage
% -*- root: ../thesis.tex -*-

\chapter{Conclusions and Open Problems}

\noindent In this section, we demonstrate the significance of this work by summing up and contextualizing the main new results and concepts presented
in this work.
We introduced the notion of {\em duality respecting representations} of gammoids and proved that every gammoid has such a representation.
For a long time, it has been a well-known fact that the class of gammoids is closed under duality, and the classical proofs of this
property employ the important insight that strict gammoids are precisely the duals of transversal matroids. By shifting our focus away from
strict gammoids, we were able to reveal that a gammoid and its dual are tightly related to each other: they are represented by special pairs of
opposite digraphs\footnote{We would like to call such directed graphs dual to each other. This may be justified by observing that for
lattices and partial orders, which may be considered special binary relations, the concept of duality merely swaps the
first and the second component of that relation. Since directed graphs may be considered binary relations as well, 
it is quite natural to call the opposite digraph dual. Unfortunately, there are other notions of duality with respect 
to directed graphs that have equally good justifications.},
and these special pairs are easily obtained from any representation. This discovery lead us to the concept of 
a {\em standard representation},
which allowed us to define complexity measures for gammoids as minimal complexity measures of the digraphs that may appear in a standard
representation of a gammoid. We defined the {\em arc-complexity} and {\em vertex-complexity} of gammoids and showed that
the derived classes of gammoids with bounded arc- or vertex-complexity are closed under duality and minors, and that these classes are characterized by finitely many excluded minors. But in general, these classes are not closed under direct sums. In order to derive subclasses of gammoids that are closed under direct sums, we defined the {\em $f$-width} of gammoids. We were able to show 
that the subclasses of gammoids
with bounded $f$-width are closed under direct sums for super-additive functions $f$.

\medskip
\noindent
Regarding our investigation into the problem of deciding whether a given matroid is a gammoid,
the starting point
was Mason's $\alpha$-criterion for strict gammoids. Naturally we were more interested in the situation where the matroid
under consideration is not a strict gammoid, and therefore we defined the concept of an {\em $\alpha$-violation} that
captures minimal situations in a matroid that are not ``strictly gammoidal''. Unfortunately, we showed that it is not possible
to classify $\alpha$-violations into violations that correspond to gammoids and violations that correspond to non-gammoids --- we saw
that there are non-gammoids that have two copies of a violation, that may occur in a gammoid as its unique violation.
We condensed our gathered experience with the recognition problem of gammoids into the notion of a {\em matroid tableau} and
the corresponding \ref{lastStep}-step directions for the derivation of a decisive matroid tableau.

\medskip
\noindent
We introduced the concept of {\em lifting} cycles in digraphs of representations of gammoids, in order to find
acyclic representations of gammoids that have the original gammoid as a contraction minor. 
This concept may be used to
avoid technicalities with the non-acyclic generalizations of the Lindström Lemma, but it may generally be applied in
situations where the presence of cycles in digraphs complicates matters. One such situation arises when we try to use {\em heavy arc signatures}
in order to orient a gammoid. We provided a way to determine orientations of gammoids without having to carry out actual calculations in
$\Qbm$ or $\Rbm$ as long as the corresponding digraph of the representation of the gammoid has no cycle walk. We also gave an
example that this condition may not be dropped. Apart from that, we were able to show that the class of {\em lattice path matroids} is generalized series-parallel, and therefore $3$-colorable.

\bigskip
\noindent In the following sections, we give some starting points for further research in the field of gammoids.

\section{Other Complexity Measures}

\noindent Let $\mu$ be a measure that assigns every digraph $D=(V,A)$ a value $\mu(D) \in \R$. Analogously to the definitions of
the arc-complexity and vertex-complexity of a gammoid, we may define the $\mu$-complexity of a gammoid $M$ to be
\[ \hat{\mu}(M) = \min\SET{\vphantom{A^A} {\mu(D)} ~\middle|~ \left( D, T, E \right)\text{~is a standard representation of~}M}. \]
If $\mu(D) = \mu(D^\opp)$ for all digraphs $D$, then the $\mu$-complexity has the property that $\hat\mu(M) = \hat\mu(M^\ast)$.
Obviously, all complexity measures for directed graphs have this property as soon as they are obtained from measures for undirected graphs by ignoring
the orientation of the arcs. This yields a variety of new research questions about the properties of
subclasses of gammoids with bounded $\mu$-complexity: for which measures $\mu$ are the the classes consisting
of gammoids $M$ with $\hat\mu(M) \leq k$ closed under minors, under duality, and under direct-sums? If such a class is closed under minors, then what are the excluded minors for that class? Which of these excluded minors are gammoids? Is the class characterized by finitely many excluded minors?
Interesting choices of $\mu$ include arboricity, star-arboricity, thickness, degeneracy, girth, tree number, DAG-width, and many more.
Furthermore, we should consider the same questions with respect to the $f$-$\mu$-width which may be defined as
\[ \hat\mu_f(M) = \max\SET{\frac{\hat\mu\left( \left( M\contract Y \right)\restrict X \right))}{f\left( \left| X \right|
	 \right) }
		 ~\middle|~ X\subseteq Y\subseteq E} \]
where $M=(E,\Ical)$.

\section{Arc Complexity of Uniform Matroids}

\noindent The following is the most fundamental open problem that we encountered in the course of this work.
It is most promising to be answered positively in the next few years
-- possibly by utilizing some results from the theory of digraphs -- but the solution of this problem is unfortunately out of
reach to the author within the schedule of this work. 
We are not able to show the following conjecture, but we are convinced that it is true.

\begin{conjecture}\label{conj:uniformArcs}
	Let $r\in \N$, $U=(E,\Ical)$ be a uniform matroid of rank $r$ on the ground set $E$, i.e.
	$\Ical = \SET{\vphantom{A^A}X\subseteq E~\middle|~ \left| X \right| \leq r}$.
	Then $\arcC(U) = r\cdot \left( \left| E \right| -r \right)$.
\end{conjecture}

\noindent There is the following reformulation with respect to directed graphs.

\begin{conjecture}
	Let $D=(V,A)$ be a digraph, and let $X,Y\subseteq V$ with $X\cap Y =\emptyset$.
	If for every $X'\subseteq X$ and every $Y'\subseteq Y$ with $\left| X' \right| = \left| Y' \right|$
	there is a routing $R\colon X'\routesto Y'$ in $D$ with 
	\[ Y \cap \bigcup_{p\in R} \left| p \right| = Y',\] then
	 $$\left| A  \right|\geq \left| X \right|\cdot \left| Y \right|.$$
\end{conjecture}

\begin{proof}[Relative proof]
	Let $M=\Gamma(D,Y,X\cup Y)$, and let $Q\subseteq X\cup Y$ with $\left| Q \right| \leq \left| Y \right|$.
	Clearly \linebreak $\left| Q \right| = \left| Q_X \right| + \left| Q_Y \right|$
	where $Q_X = Q\BS Y$ and $Q_Y = Q\cap Y$. Thus $\left| Q \right|\leq \left| Y \right|$ implies
	\linebreak
	$\left| Q_X \right| \leq \left| Y\BS Q_Y \right|$. Therefore there is a subset $Q_Y' \subseteq Y\BS Q$ with $\left| Q_Y' \right| = \left| Q_X \right|$. By hypothesis we obtain a routing $R\colon Q_X\routesto Q_Y'$ in $D$
	which avoids $Q_Y \subseteq Y\BS Q_Y'$, thus $R\cup\SET{q~\middle|~q\in Q_Y}$ is a routing from $Q$ to $Y$ in $D$.
	Consequently, $Q$ is independent in $M$.
	We showed that every independent subset of $X\cup Y$ with at most $\left| Y \right|$ elements is independent in $M$,
	and it is clear that no subset of $X\cup Y$ with more than $\left| Y \right|$ elements is independent in $M$,
	therefore $M$ is a uniform matroid of rank $\left| Y \right|$ with
	$\left| X \cup Y \right| = \left| X \right| + \left| Y \right|$ 
	elements. The statement of this conjecture then follows from Conjecture~\ref{conj:uniformArcs}.
\end{proof}

\noindent Closely related is the following conjecture which would be implied by the previous two conjectures.

\begin{conjecture}
	For all $k\in \N$ there is a gammoid $G=(E,\Ical)$ such that $$\arcC(G) \geq k\cdot \left| E \right|.$$
\end{conjecture}

\noindent This conjecture might be easier to proof, and it still would imply that the classes $\Wcal^k$ of gammoids $G$ 
with $\arcW^k(G) \leq 1$ for $k\in \N$ contain an infinite sequence of strictly bigger subclasses of gammoids.
% Finally, let us sketch the implications
%of any of the above conjectures: In that case, it would be possible to take the canonical orientation of
%a complete biparite graph $(X\disunion Y, X\times Y)$ and replace its arc set $X\times Y$ with a weirdly complicated gadget
%that has less than $\left| X\times Y \right|$ arcs, but more vertices, and still exhibits the same routing-connectivity between $X$ and $Y$,
%and this gadget ceases work for small choices of $X$ and $Y$.

\section{$\alpha$-Violations}

\noindent
In Example~\ref{ex:BiApexMatroid} and Remark~\ref{rem:BiApex} we saw that $\alpha$-violations,
which may be resolved into a strict gammoid
by extension, may overlap in a common matroid. It is possible that this common matroid may not be resolved into a strict gammoid,
although each restriction that encompasses only a single violation may be extended to a strict gammoid. It is an interesting
 open research problem to
investigate in what ways $\alpha$-violations may overlap, and to determine under which circumstances overlapping obstructs
the simultaneous resolution of the respective $\alpha$-violations into strict gammoid extensions of the matroid exhibiting the overlapping
$\alpha$-violations.

\section{Excluded Minors}

\noindent Since every matroid of rank $\leq 2$ is a gammoid, and every gammoid of rank $3$ is a strict gammoid, we may
use the formulas from Section~\ref{sec:alphaNExt} in order to compute all small\footnote{In this case: up to $10$ elements.} excluded minors of rank $3$ for the class of gammoids,
as well as the number of isomorphism classes of small gammoids of rank $3$ with $n$-elementary ground sets \cite{OEIS}. For $n\leq 10$, this takes less than $4$ hours on modern hardware. The following
statements are up to isomorphy:
The smallest excluded minor of rank $3$ is $M(K_4)$ with $6$ elements, the second smallest excluded minor is $P_7$ with $7$ elements.
There are $3$ excluded minors with $8$ elements, $11$ excluded minors with $9$ elements, and $96$ excluded minors with $10$ elements.
The difficult part in obtaining excluded minors with rank and corank greater than $3$ is to prove, that the considered matroid is not a 
gammoid, and apart from $P_8^=$, we do not know any excluded minor with rank and corank greater than $3$, that is representable over $\Rbm$
and strongly base-orderable. Therefore we ask: Are there other $\Rbm$-representable and strongly base-orderable excluded minors for the class of gammoids with rank and corank greater than $3$? Are there infinitely many such excluded minors?

\noindent
Furthermore, we do not know the excluded minors for $\Wcal^k$, the classes of gammoids $G$ with $\arcW^k(G) \leq 1$.
Is $\Wcal^k$ characterized by finitely many excluded minors? Moreover, let $\Wcal_f$ be the class of gammoids $G$
with $\arcW_f(G) \leq 1$ for a super-additive function $f$. What is the smallest growing behavior of $f$, such that $\Wcal_f$
has infinitely many excluded minors? And, conversely, what is the biggest growing behavior of $f$, such that $\Wcal_f$ has finitely many excluded minors?

\section{Complexity Class of Recognition Problems}

\noindent V.~Chandru, C.R.~Coullard, and D.K.~Wagner showed in \cite{CHANDRU198575} that the problem of
deciding,
whether a given matroid $M$ is a bicircular matroid, is NP-hard.
The proof involves deciding, whether the frame matroid constructed from a given gain-network
is a bicircular matroid or not, in order to answer an instance of the {\em Subset Product Problem}.
Clearly, there are frame matroids which are not gammoids, for instance all Dowling geometries of rank $\geq 3$ (\cite{Ox11}, p.663)
as well as  all graphical matroids with an $M(K_4)$ minor.
On the other hand, every bicircular matroid is a gammoid,
and thus it is possible that the additional information, that a given frame matroid is a gammoid, helps to decide
whether the given matroid is bicircular within polynomial time.
If this is the case, then ruling out that a given matroid is a gammoid must be NP-hard --- which is the most likely scenario
considering D.~Mayhew's result  that every gammoid is a minor of an excluded minor for the class of gammoids \cite{Ma16}.

\noindent Closely related to the open problem of the complexity class of recognizing gammoids are the open problems
regarding the complexity classes of finding a representation, finding a representation with
the minimal number of arcs when this number is already known, and determining the minimal number of arcs needed to represent a gammoid;
and all of the above for each of the subclasses $\Wcal^k$ for $k\in\N$, too.

\section{Coloring}

Every gammoid, that is also a binary matroid, is the polygon matroid of a series-parallel network \cite{In77},
therefore graphic gammoids are $3$-colorable.
The following conjecture  motivated our studies of gammoids in the first place.

\begin{conjecture}[\cite{GoHoNe15}, Conjecture~14]\label{conj:winfried}
	Every simple gammoid of rank $2$ or greater has a quite simple coline.
\end{conjecture}

\noindent
If the conjecture holds, then all gammoids are generalized series-parallel matroids and therefore $3$-colorable.
Although we were not able to resolve this conjecture at this point, we are convinced that the newly developed theory in this work
will prove helpful for future approaches. A related open problem is the question, whether every gammoid is generalized series-parallel,
which may still be the case even if the Conjecture~\ref{conj:winfried} is wrong: although no coline of the non-gammoid $P_7$ 
is quite simple, all of its 
orientations are still generalized series-parallel.

\noindent
We provided a method of obtaining an orientation of a gammoid from its representation by combinatorial means, but all orientations obtained
in this way are representable. It has been known for long that there are oriented matroids whose orientations are not representable.
One example of a non-representable orientation is the orientation $\mathtt{RS}(8)$ of the uniform matroid $U_{4,8}$ --- which is a strict gammoid.
Is there a way to obtain some or all non-representable orientations of gammoids in a purely combinatorial way from their representations, possibly by generalizing the notion of {\em heavy arc signatures}? And finally, is there a way to deal with cycle walks in the digraph of a given representation of a gammoid other than first lifting all the cycle walks, then orienting using an extended heavy arc signature, and then contracting the oriented extension?

\vspace*{3cm}
\begin{flushright}
\textit{Who questions much, shall learn much, and retain much.  ~~~~~ \\
\hfill{~} --- Sir Francis Bacon.}
\end{flushright}

% ********************************** Back Matter *******************************
% Backmatter should be commented out, if you are using appendices after References
%\backmatter

\cleardoublepage
% -*- root: ../thesis.tex -*-
\stepcounter{chapter}
\setcounter{section}{0}
\chapter*{Listings}
\addcontentsline{toc}{chapter}{Listings}
\fancyhead[RE]{Listings}

\section{Digraph Backtracking Algorithm}\label{lst:isGammoidSage}

\PRFR{Mar 7th}
The routine ``{\ttfamily isGammoid}'' performs a backtracking search in the domain of all digraphs with a fixed number of vertices
in order to determine whether a given input matroid $M$ is a
gammoid (Algorithm~\ref{alg:gammoidBackTrack}) --
or, optionally, whether $M$ is a gammoid representable with certain upper bounds on the number of arcs and vertices occurring. It
has been tested with {\ttfamily SageMath} version 8 running on macOS 10.13.3. We present some runtime measurements of various inputs
in order to convey a sense of how slow this algorithm actually is. 
We measured the performance using three matroids, one is a non-gammoid, one is a strict gammoid, and one is a non-strict gammoid, each with
different upper bounds for the number 
of vertices allowed in a representing digraph candidate. We stopped the each measurement once a time-limit of $48$ hours was reached. 
Here are the results:

\begin{tabularx}{\textwidth}{XX}
(Example~\ref{ex:MK4}, $M(K_4)$) &
\ttfamily sage: time(isGammoid(MK4()))\\*
& \ttfamily CPU time total: 27.4 ms\\
($M(K_4)$, $\left| V \right|=\left| E \right|+1$)&\ttfamily sage: time(isGammoid(MK4(),7))\\*
& \ttfamily CPU time total: 1.95 s\\
($M(K_4)$, $\left| V \right|=\left| E \right|+2$)&\ttfamily sage: time(isGammoid(MK4(),8))\\*
& \ttfamily CPU time total: 8min 3s\\
($M(K_4)$, $\left| V \right|=\left| E \right|+3$)&\ttfamily sage: time(isGammoid(MK4(),9))\\*
& \ttfamily CPU time total: > 48h\\
(Example~\ref{ex:nonStrictGammoid}, $\Gamma(D,T,V)$, $\left| V \right|=9$)& \ttfamily sage: time(isGammoid(strictG,9))\\*
& \ttfamily CPU time total: 46min 57s\\
(Example~\ref{ex:nonStrictGammoid}, $\Gamma(D,T,V)$, $\left| V \right|=10$)&\ttfamily sage: time(isGammoid(strictG,10)) \\*
& \ttfamily CPU time total: > 48h \\
(Example~\ref{ex:nonStrictGammoid}, $\Gamma(D,T,E)$, $\left| V \right|=9$)&\ttfamily sage: time(isGammoid(G),9) \\*
& \ttfamily CPU time total: 4h 28min 52s\\
(Example~\ref{ex:nonStrictGammoid}, $\Gamma(D,T,E)$, $\left| V \right|=10$)&\ttfamily sage: time(isGammoid(G,10)) \\*
& \ttfamily CPU time total: > 48h \\
\end{tabularx}

\PRFR{Mar 7th}
\noindent
Those times suggest that the digraph backtracking method is not suitable for deciding the value of $\Gamma_\Mcal(M)$ 
for $M$ defined on ground sets larger than a few elements within a reasonable time frame. The generic bound derived from Remark~\ref{rem:upperBoundForV}
is useless in practice, for instance, the matroid defined in Example~\ref{ex:nonStrictGammoid} has an upper bound of at most $123$ vertices in
a representing digraph.
 We might achieve some  slight improvement in performance  by utilizing a tree
 structure to store the set of essential paths and the family of maximal essential routings, but this measure would not have any influence on the rapid growth of 
 number of digraph 
 candidates that have to be traversed by the backtracking method (Remark~\ref{rem:backTrackSlow}).
 \vspace*{1\baselineskip}

\noindent
\begin{lstlisting}[language=Python,backgroundcolor=\color{black!5},frame=tlb,basicstyle=\footnotesize]
from itertools import chain, combinations
from sage.all import DiGraph

def vertexBound(n,r):
    if r <= 3: # either strict gammoid or no gammoid
        return n
    if r >= n-3: # either transversal matroid or no gammoid
        return n + r
    return r*r*n + r + n

def arcBoundStrict(n,r):
    return (n-r)*(n-1) # Mason '72 Corollary 2.5

def augmentList(L,n):
    L = list(L)
    i = int(1)
    for j in range(n):
        while i in L:
            i += 1
        L.append(i)
    return L

class NonLoopArcIterator:
    def __init__(self, V):
        self.V = list(V)
        self.lenV = len(self.V)
        self.x = 0
        self.y = 0
    def move2(self, x,y):
        self.x = x
        self.y = y
    def next(self):
        if self.x < self.y-1:
            self.x += 1
        elif self.x == self.y -1:
            self.x = self.y
            self.y = 0
        elif self.y < self.x-1:
            self.y += 1
        else: # y == x-1 or y == x
            if self.lenV <= self.x+1:
                raise StopIteration
            self.y = self.x+1
            self.x = 0
        return (self.V[self.x], self.V[self.y])
    def done(self):
        if self.x < self.y-1:
            return False
        elif self.x == self.y -1:
            return False
        elif self.y < self.x-1:
            return False
        else: # y == x-1 or y == x
            if self.lenV <= self.x+1:
                return True
            return False

def isGammoid(M, vertexLimit=None, arcLimit=None):
    T = frozenset(M.bases()[0])
    BM = frozenset(M.bases())
    if vertexLimit is None:
        vertexLimit = vertexBound(len(M.groundset()),M.rank())
    if arcLimit is None:
        arcLimit = arcBoundStrict(vertexLimit,M.rank())
    V = sorted(T) + sorted(M.groundset().difference(T))
    E = frozenset(V)
    if len(V) > vertexLimit:
        return False
    if len(V) < vertexLimit:
        V = augmentList(V,vertexLimit - len(V))
    sink_routing = frozenset([(t,t,frozenset([t])) for t in T])
    essentialMaximalRoutings = set([( T,sink_routing, T )])
    newEssentialMaximalRoutings = []
    rollbackIndex = [-1]
    noLongerEssentialMaximalRoutings = []
    rollbackIndex1 = [0]
    independentBases = [T]
    rollbackIndex2 = [0]
    newEssentialPaths = []
    rollbackIndex3 = [0]
    noLongerEssentialPaths = []
    rollbackIndex4 = [0]
    rollbackArc = [None]
    rollbackArcIdx = [(0,0)]
    Vmax = len(E)
    rollbackVmax = [0]
    requestRollback = False
    nbrOfBases = len(M.bases())
    arcCount = 0
    D = DiGraph()
    D.add_vertices(E)
    arcs = NonLoopArcIterator(V)
    essentialPaths = {}
    for u in V:
        for v in V:
            essentialPaths[(u,v)] = set()
    for v in V:
        essentialPaths[(v,v)] = set([frozenset([v])])
    while 1:
        if len(independentBases) == nbrOfBases and not requestRollback:
            return D, T, sorted(M.groundset()) # ''return 1''
        if arcs.done() or arcCount >= arcLimit or requestRollback:
            # pop state from stack
            idx = rollbackIndex.pop()
            idx1 = rollbackIndex1.pop()
            idx2 = rollbackIndex2.pop()
            idx3 = rollbackIndex3.pop()
            idx4 = rollbackIndex4.pop()
            arc = rollbackArc.pop()
            Vmax = rollbackVmax.pop()
            arcx,arcy = rollbackArcIdx.pop()
            if idx < 0: # ''d = 0''
                return False # ''return 0''
            essentialMaximalRoutings.difference_update(
                newEssentialMaximalRoutings[idx:])
            essentialMaximalRoutings.update(
                noLongerEssentialMaximalRoutings[idx1:])
            del newEssentialMaximalRoutings[idx:]
            del noLongerEssentialMaximalRoutings[idx1:]
            del independentBases[idx2:]
            for v,w,p in newEssentialPaths[idx3:]:
                essentialPaths[(v,w)].discard(p)
            for v,w,p in noLongerEssentialPaths[idx4:]:
                essentialPaths[(v,w)].add(p)
            del newEssentialPaths[idx3:]
            del noLongerEssentialPaths[idx4:]
            arcCount -= 1
            D.delete_edge(arc[0],arc[1])
            arcs.move2(arcx,arcy)
            requestRollback = False
        else:
            arc = arcs.next()
            if arc[0] in T: # skip arcs with sink-tails
                continue
            if arcs.y > Vmax:
                requestRollback = True
                continue
            # push state to stack
            rollbackArc.append(arc)
            rollbackArcIdx.append((arcs.x,arcs.y))
            rollbackIndex.append(len(newEssentialMaximalRoutings))
            rollbackIndex1.append(len(noLongerEssentialMaximalRoutings))
            rollbackIndex2.append(len(independentBases))
            rollbackIndex3.append(len(newEssentialPaths))
            rollbackIndex4.append(len(noLongerEssentialPaths))
            rollbackVmax.append(Vmax)
            # update digraph
            if arcs.y == Vmax:
                Vmax += 1
            left_part = []
            right_part = []
            for v in V:
                for p in essentialPaths[(v,arc[0])]:
                    left_part.append( (v, p) )
                for p in essentialPaths[(arc[1],v)]:
                    right_part.append( (v, p) )
            new_paths = set()
            for l0,l in left_part:
                for r0,r in right_part:
                    if l.intersection(r):
                        continue
                    non_essential_path = False
                    for x in l:
                        for y in r:
                            if D.has_edge(x,y):
                                non_essential_path = True
                                break
                        if non_essential_path:
                            break
                    if non_essential_path:
                        continue
                    p = l.union(r)
                    new_paths.add((l0,r0,p))
            idx4 = len(noLongerEssentialPaths)
            for v,w,p in new_paths:
                superseeded = []
                for q in essentialPaths[(v,w)]:
                    if p.issubset(q):
                        superseeded.append(q)
                        noLongerEssentialPaths.append((v,w,q))
                newEssentialPaths.append((v,w,p))
                essentialPaths[(v,w)].difference_update(superseeded)
                essentialPaths[(v,w)].add(p)
            criterion = [x for x in noLongerEssentialPaths[idx4:]
                                      if x[0] in E and x[1] in T]
            idx = len(newEssentialMaximalRoutings)
            idx1 = len(noLongerEssentialMaximalRoutings)
            new_E_paths = [x for x in new_paths if x[1]in T and x[0]in E]
            for X,R,P in essentialMaximalRoutings:
                if R.intersection(criterion):
                    noLongerEssentialMaximalRoutings.append((X,R,P))
                else:
                    for v,w,p in new_E_paths:
                        if not w in X:
                            continue
                        if len(p.intersection(P)) > 1:
                            continue
                        X1 = X.difference([w]).union([v])
                        if not X1 in independentBases:
                            if not X1 in BM:
                                # found excess base
                                requestRollback = True
                                break
                            independentBases.append(X1)
                        R1 = R.difference([(w,w,frozenset([w]))]).union(
                                                              [(v,w,p)])
                        if R1.intersection(sink_routing):
                            newEssentialMaximalRoutings.append(
                                (X1,R1,P.union(p)))
                    if requestRollback:
                        break
            essentialMaximalRoutings.difference_update(
                noLongerEssentialMaximalRoutings[idx1:])
            essentialMaximalRoutings.update(
                newEssentialMaximalRoutings[idx:])
            D.add_edge(arc[0],arc[1])
            arcCount += 1
\end{lstlisting}

\section{Calculating $\alpha_N$ for $N\in \Xcal(M,e)$}\label{lst:measureDeltaAlpha}

\begin{figure}[H]
\begin{center}
\includegraphics[width=.7\textwidth]{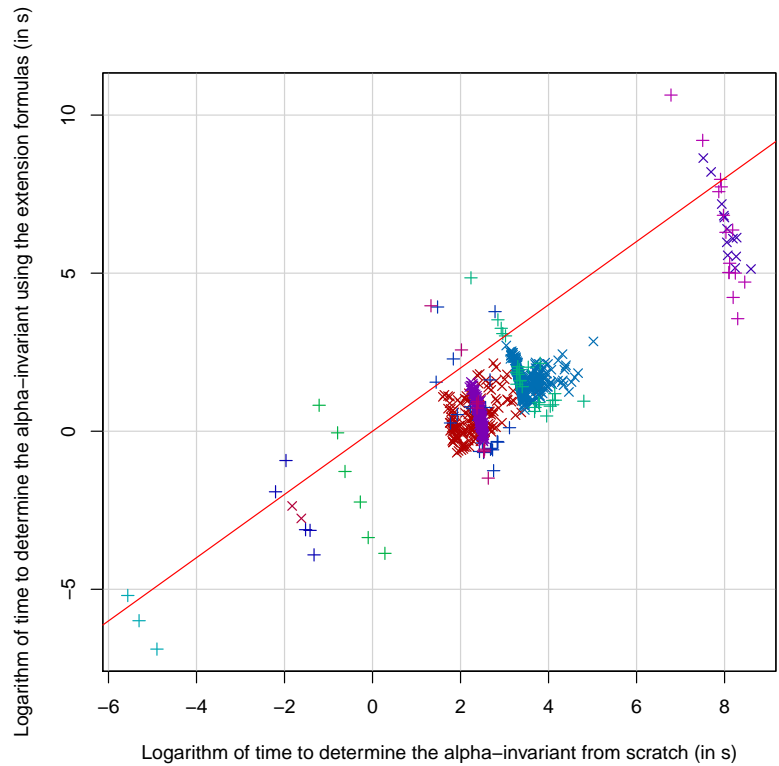}
\end{center}
\caption{\label{fig:Scat1}Scatter plot of runtime measurements for Listing~\ref{lst:measureDeltaAlpha}. Each $+$-mark corresponds to a principal extension, each $\times$-mark to a non-principal extension. The red line indicates equal runtime. The logarithms are dyadic.} 
\end{figure}

In this section, compare the  performance of determining the $\alpha_N$-vector of a single element extension $N\in \Xcal(M,e)$
with the same rank as $M$
obtained through the formulas 
derived in Section~\ref{sec:alphaNExt} with the performance of determining the $\alpha_N$-vector from scratch. The
performance has been measured with {\ttfamily SageMath} version 8 running on macOS 10.13.3. The program code is listed at the end of this section,
the code has been compiled with the built-in Cython compiler of {\ttfamily SageMath} prior to the measurements. 
For a fixed initial matroid $M=(E,\Ical)$, 
we measured one representative $N$ of each isomorphism class of the single element extensions of $M$ that have the same rank as $N$.
The runtime for each representative has been measured with $3$ repetitions per method. 
In total we performed runtime  measurements on $822$ different single-element extensions,
the median of the extension-formula-runtime to from-scratch-runtime ratio is approximately $0.2635$,
at least $95\%$ of the measured ratios are smaller than $0.5972$. 
Therefore we expect the method using the formulas from Section~\ref{sec:alphaNExt} to be almost four times faster on average, 
and to be at least $1\frac 1 2$ times faster in the usual case, than computation of $\alpha_N$ from scratch (Figure~\ref{fig:Scat1} on p.\pageref{fig:Scat1}).

\noindent
We give an overview over the measurements with respect to the specific initial matroids in the following table, where $M_0$ is the initial matroid, $k$ is the number of non-isomorphic same-rank single element extensions of $M_0$, $k_1$ is the number of non-isomorphic same-rank single element extensions of $M_0$ that have a principal modular cut, $r_{.5}$ is the median extension-formula-runtime to from-scratch-runtime ratio, and $r_{.95}$ is the $95$th percentile extension-formula-runtime to from-scratch-runtime ratio.
The scatter plot depicts the dyadic logarithm of the run-time in seconds using the extension-formula (vertical axis) versus the dyadic logarithm of the run-time in seconds calculating from scratch (horizontal axis).
The blue $+$-marks correspond to single element extensions of $M_0$ with principal modular cuts, and the black $\times$-marks correspond
to single element extensions of $M_0$ that have no principal modular cuts. The red line indicates the locations that correspond to equal run-time for both methods.
%
%\needspace{6\baselineskip}
\begin{tabularx}{\textwidth}{p{3cm} |l|l|l|l|l}
$M_0$ & $k$ & $k_1$ & $ r_{.5}$ & $ r_{.95}$ &  \\ 
\hline 
\hline
Ex.~\ref{ex:nonStrictGammoid},
 $\Gamma(D,T,E)$ & $177$ & $22$ & $.24$ & $.65$ & 
 \begin{minipage}[c]{7cm} \includegraphics[trim={1cm 1cm 1cm 1cm},width=7cm]{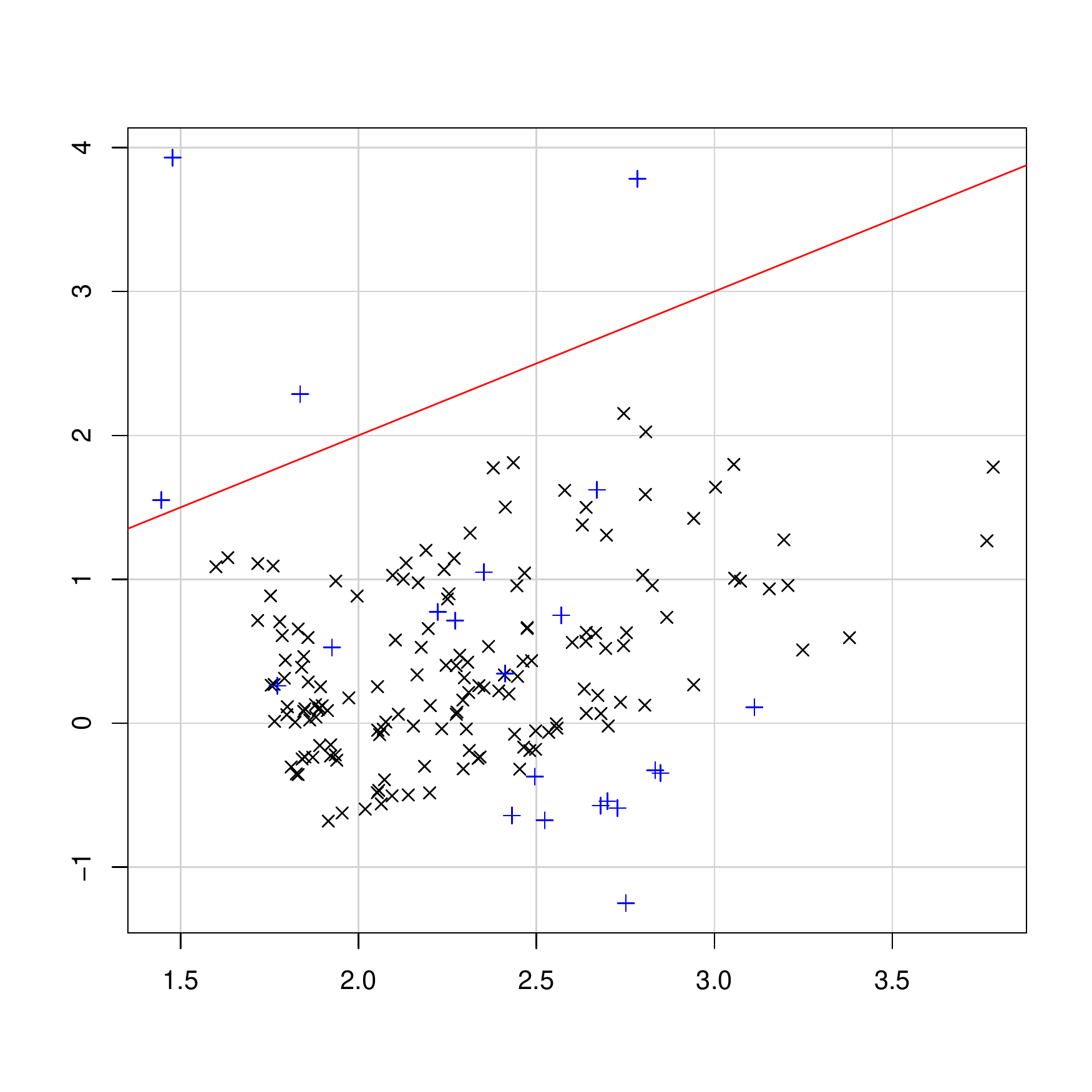} \end{minipage}
 \\
\hline
Ex.~\ref{ex:nonStrictGammoid},
 $\Gamma(D,T,V)$ & $367$ & $26$ & $.26$ & $.55$ & 
 \begin{minipage}[c]{7cm} \includegraphics[trim={1cm 1cm 1cm 1cm},width=7cm]{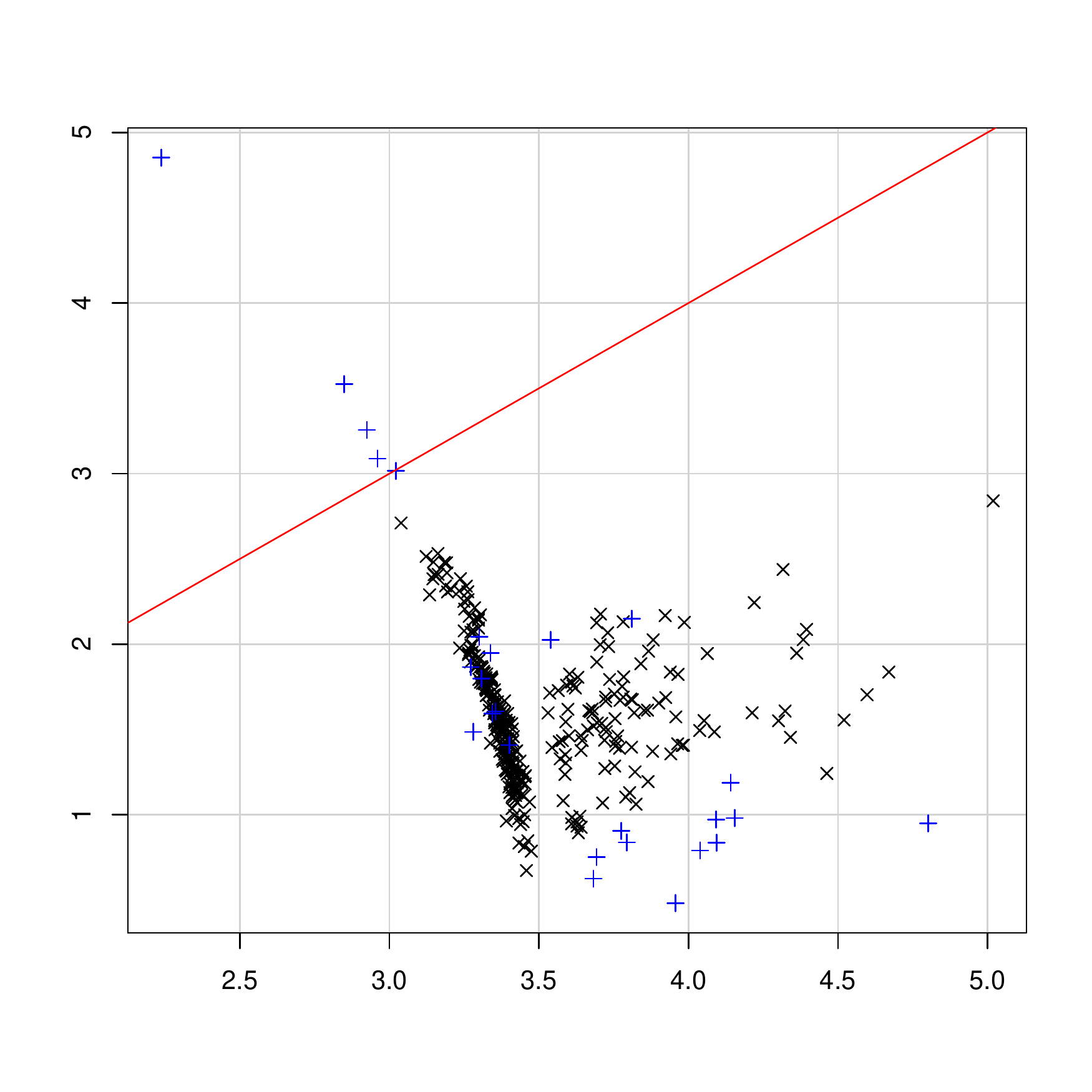} \end{minipage}
 \\
\hline
$\left(E,2^E\right)$ with $\left| E \right| = 5$ & $6$ & $6$ & $.45$ & $3.49$ & 
 \begin{minipage}[c]{7cm} \includegraphics[trim={1cm 1cm 1cm 1cm},width=7cm]{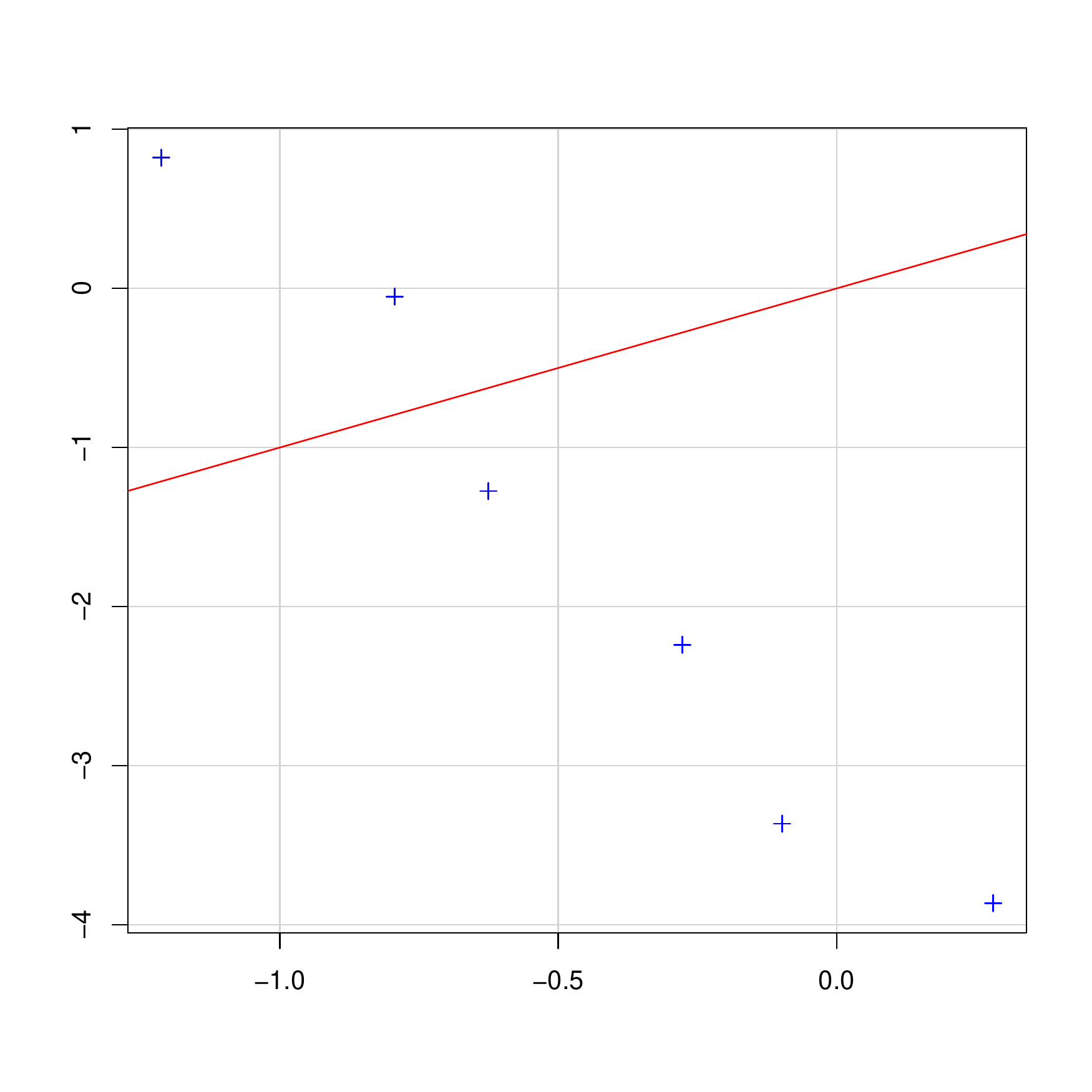} \end{minipage}
 \\
\hline
$M(K_4)$,
Ex.~\ref{ex:violationNonGammoid} & $7$ & $5$ & $.45$ & $1.81$ & 
 \begin{minipage}[c]{7cm} \includegraphics[trim={1cm 1cm 1cm 1cm},width=7cm]{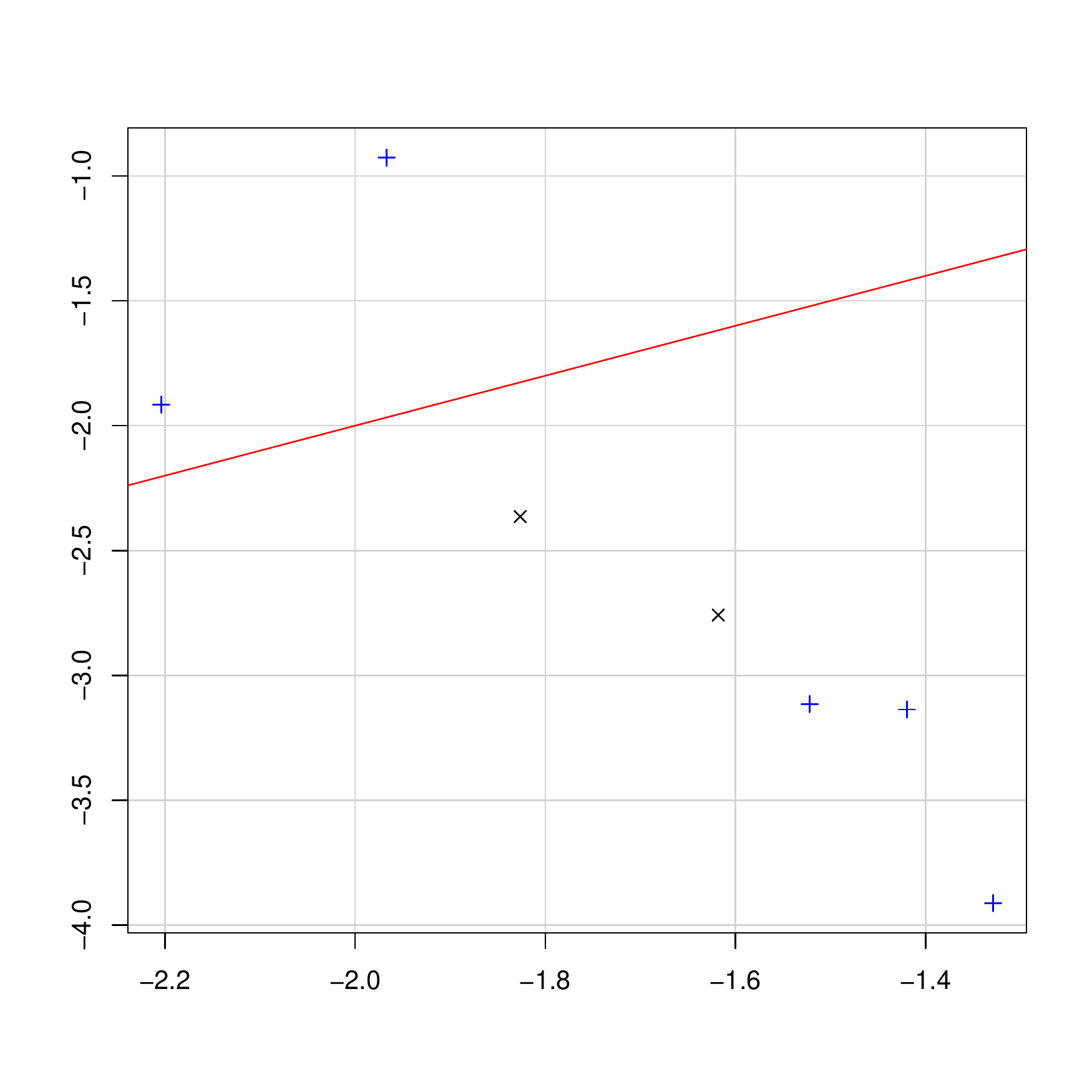} \end{minipage}
 \\
\hline
$\!\!\!\!\!M(K_4)\oplus M(K_4)$,
Ex.~\ref{ex:violationNonGammoid} & $28$ & $15$ & $.26$ & $2.87$ & 
 \begin{minipage}[c]{7cm} \includegraphics[trim={1cm 1cm 1cm 1cm},width=7cm]{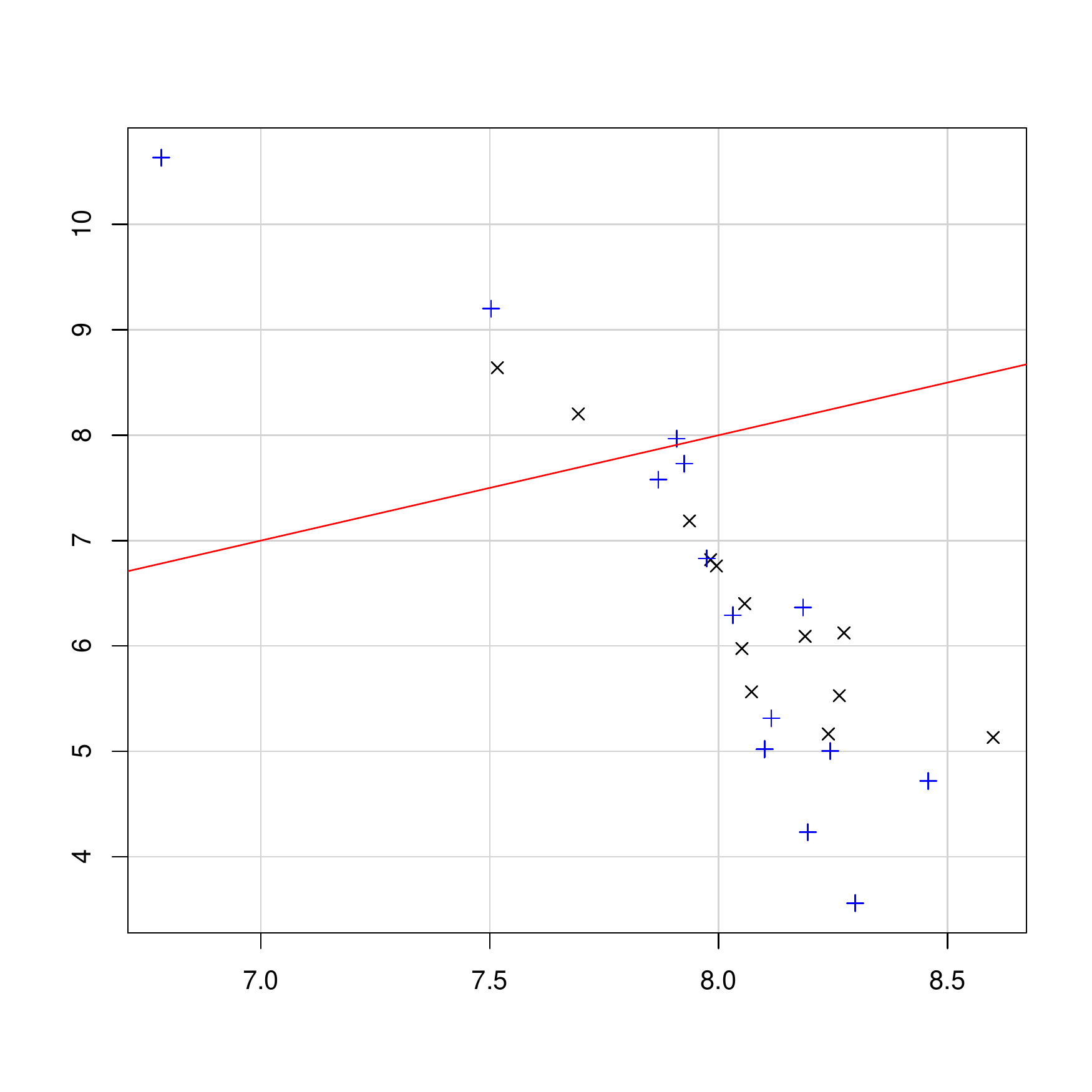} \end{minipage}
 \\
\hline
$U_{4,2}$,
\cite{Ox11},~p.639 & $3$ & $3$ & $.62$ & $1.22$ & 
 \begin{minipage}[c]{7cm} \includegraphics[trim={1cm 1cm 1cm 1cm},width=7cm]{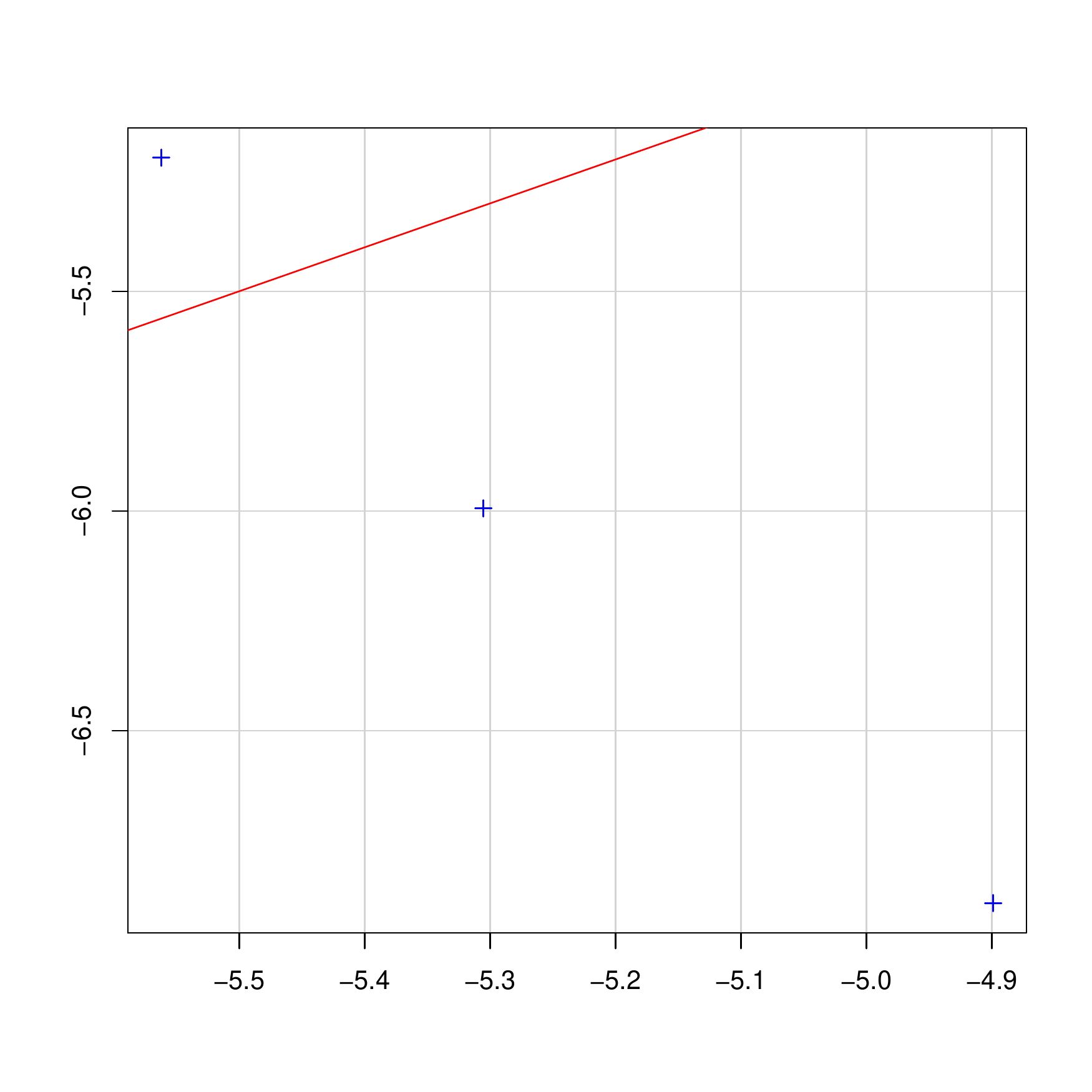} \end{minipage}
 \\
\hline
$P_8^=$,
\cite{Ox11},~p.651 & $234$ & $9$ & $.27$ & $.52$ & 
 \begin{minipage}[c]{7cm} \includegraphics[trim={1cm 1cm 1cm 1cm},width=7cm]{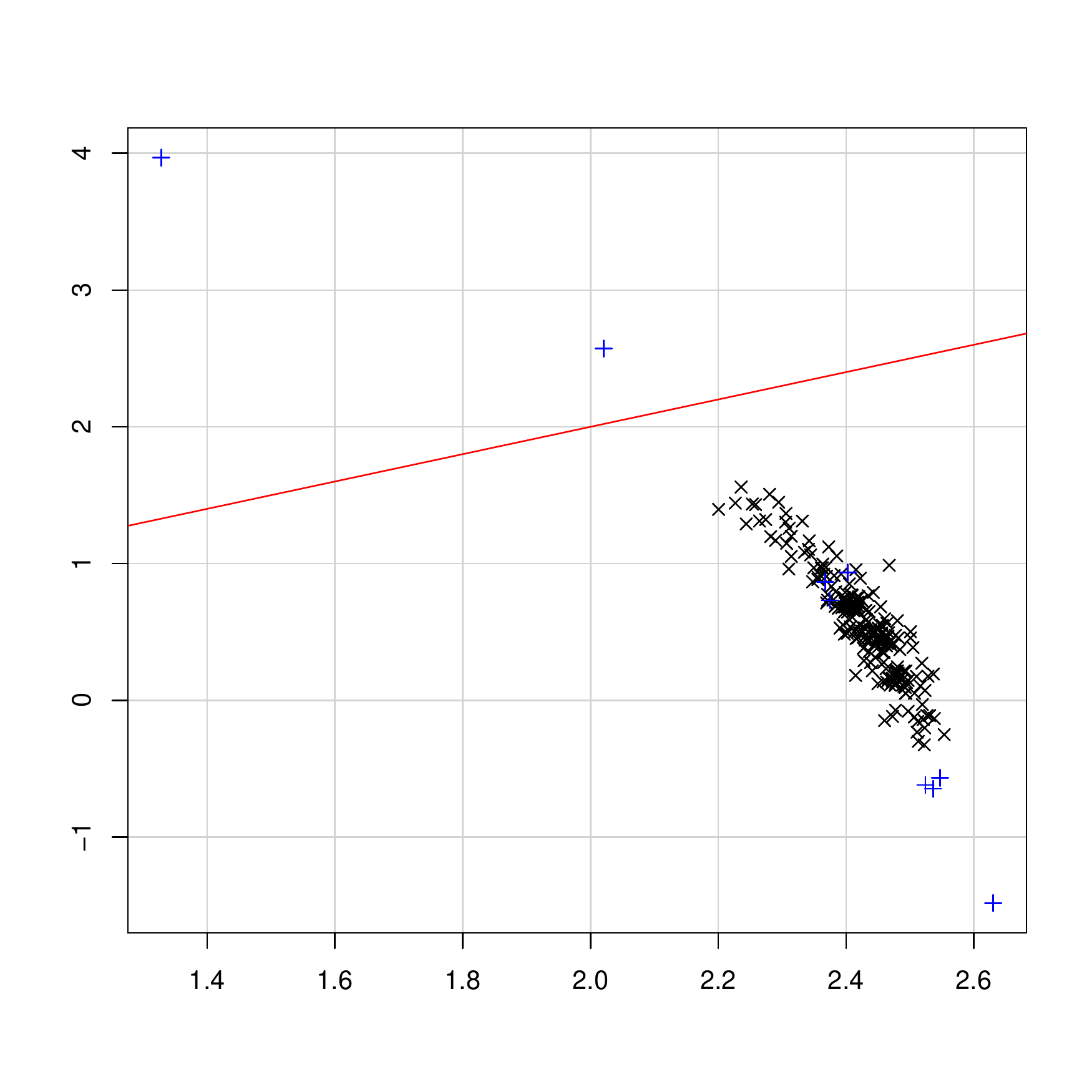} \end{minipage}
 \\
\end{tabularx}

\vspace*{1\baselineskip}

\noindent
\begin{lstlisting}[language=Python,backgroundcolor=\color{black!5},frame=tlb,basicstyle=\footnotesize]
from itertools import chain,combinations
from sage.all import IncidenceStructure,Matrix,binomial
import timeit,sys
def powerset(s):
    return chain.from_iterable(combinations(s, r)
                                for r in xrange(len(s)+1))
def powersetSubsetIndex(S, X):
    if type(X) != list:
        X = sorted(X)
    idx0,x0,s0 = 0,0,0
    for i in range(len(S)):
        idx0 += binomial(len(X),i)
    idxs = sorted((X.index(s) for s in S))
    for i in idxs:
        for j in range(x0,i):
            idx0 += binomial(len(X) - j -1, len(S) - s0 - 1)
        s0 += 1
        x0 = i + 1
    return idx0
def allFlats(M):
    return chain.from_iterable(M.flats(r) for r in xrange(0,M.rank()+1))
def leqWrtFlats(l,r,flats):
    l = frozenset(l)
    r = frozenset(r)
    if l == r:
        return True
    if not l.issubset(r):
        return False
    if not l in flats:
        return False
    return True
def downsetWrtFlats(r,flats):
    r = frozenset(r)
    return frozenset([r] + [l for l in flats if l.issubset(r)])
def alphaPosetLeq(M):
    return lambda l,r,f=frozenset(allFlats(M)): leqWrtFlats(l,r,f)
def alphaPosetDownsets(M):
    return lambda r,f=frozenset(allFlats(M)): downsetWrtFlats(r,f)
def moebiusMatrixForAlphaPoset(M):
    G = sorted(M.groundset())
    nG = len(G)
    mu = Matrix(2**nG,2**nG,sparse=True)
    idx1 = -1
    leq = alphaPosetLeq(M)
    downset = alphaPosetDownsets(M)
    for L in powerset(G):
        idx1 += 1
        idx2 = -1
        for R in powerset(G):
            idx2 += 1
            if L == R:
                mu[idx1,idx2] = 1
            elif leq(L,R):
                s = 0
                for P in downset(R):
                    if P == R:
                        continue
                    s += mu[idx1,powersetSubsetIndex(P,G)]
                mu[idx1,idx2] = -s
    return mu
def alphaVector(M):
    G = sorted(M.groundset())
    nG = len(G)
    alpha = Matrix(1,2**nG,sparse=True)
    idx1 = -1
    leq = alphaPosetLeq(M)
    downset = alphaPosetDownsets(M)
    for L in powerset(G):
        idx1 += 1
        aval = len(L) - M.rank(L)
        for R in downset(L):
            if R == L:
                continue
            aval -= alpha[0,powersetSubsetIndex(R,G)]
        if aval:
            alpha[0,idx1] = aval
    return alpha
def calculateAlphaOfExtension(e, C, M, alpha=None, mu=None):
    G0 = sorted(M.groundset())
    G1 = sorted(G0+[e])
    if alpha is None:
        alpha = alphaVector(M)
    if mu is None:
        mu = moebiusMatrixForAlphaPoset(M)
    alphaE = Matrix(1,2**len(G1),sparse=True)
    idxE = -1
    deltaAlphaE = {}
    for X in powerset(G1):
        X = frozenset(X)
        idxE += 1
        val = 0
        if e in X:
            X0 = X.difference([e])
            clX0 = M.closure(X0)
            if not clX0 in C:
                if clX0 == X0:
                    val = 0
                else:
                    val = alpha[0,powersetSubsetIndex(X0,G0)]
            else:
                d = 1
                for Z in C:
                    if Z.issubset(X0):
                        if Z == X0:
                            continue
                        #print setStr(Z), "<=", setStr(X0)
                        d -= deltaAlphaE[Z]
                deltaAlphaE[X0] = d
                val = alpha[0,powersetSubsetIndex(X0,G0)] + d
        else:
            s = 0
            for Y in powerset(X):
                Y = frozenset(Y)
                if Y == X:
                    continue
                nuY = len(Y) - M.rank(Y)
                for Z in C:
                    Z = frozenset(Z)
                    if not Y.issubset(Z):
                        continue
                    if not Z.issubset(X):
                        continue
                    if Z == X:
                        continue
                    s += nuY * mu[powersetSubsetIndex(Y,G0),
                                  powersetSubsetIndex(Z,G0)]
            val = alpha[0,powersetSubsetIndex(X,G0)] + s
        if val:
            alphaE[0,idxE] = val
    return alphaE
def canonicalMatroidLabel(M):
    IS = IncidenceStructure(M.bases())
    phi = IS.canonical_label()
    return (frozenset((frozenset((phi[x] for x in B))
              for B in M.bases())), len(M.groundset()))
def measureAlphaMRunTime(M,count=-1,repeats=3,f=sys.stdout,name="M"):
    head = [(M,alphaVector(M),moebiusMatrixForAlphaPoset(M))]
    visited = set([canonicalMatroidLabel(M)])
    eidx = 1
    while eidx in M.groundset():
        eidx += 1
    while (head):
        new_head = []
        for M,alpha,mu in head:
            for M0 in M.extensions(element=eidx):
                B0 = canonicalMatroidLabel(M0)
                if B0 in visited:
                    continue
                visited.add(B0)
                C0 = frozenset([F for F in allFlats(M)
                                  if eidx in M0.closure(F)])
                isPrincipal = False
                F0 = set(M.groundset())
                for F in C0:
                    F0.intersection_update(F)
                    if not F0 in C0:
                        break
                isPrincipal = F0 in C0
                alphas = []
                def f0():
                    alphas.append(calculateAlphaOfExtension
                                     (eidx,C0,M,alpha,mu))
                def f1():
                    alphas.append(alphaVector(M0))
                t1 = timeit.timeit(f1,number=repeats)
                t0 = timeit.timeit(f0,number=repeats)
                for a0,a1 in zip(alphas,alpha[1:]):
                    if a0 != a1:
                        raise Exception("Calculation Error!")
                print >>f, ",".join([str(name),
                            str(1 if isPrincipal else 0), \
                            "%.02f"%((t0*100)/t1),str(t0),str(t1)])
                f.flush()
                new_head.append((M0,alphas[0],
                        moebiusMatrixForAlphaPoset(M0)))
                count -= 1
                if count == 0:
                    return
        if count < 0:
            return
        head = new_head
        eidx += 1
    raise Exception("ERROR! out of extensions!")
\end{lstlisting}

% ********************************** Bibliography ******************************
\begin{spacing}{0.9}

% To use the conventional natbib style referencing
% Bibliography style previews: http://nodonn.tipido.net/bibstyle.php
% Reference styles: http://sites.stat.psu.edu/~surajit/present/bib.htm

\bibliographystyle{alpha}

\cleardoublepage
\fancyhead[RE]{References}
\fancyhead[LO]{References}

\bibliography{References/references} % Path to your References.bib file

\newcommand{\etalchar}[1]{$^{#1}$}
\begin{thebibliography}{{Cam}14}

\bibitem[AH15]{Al15}
Immanuel Albrecht and Winfried Hochstättler.
\newblock {Lattice Path Matroids are 3-Colorable}.
\newblock Technical report, Lehrgebiet Diskrete Mathematik und Optimierung,
  FernUniversität in Hagen, 2015.

\bibitem[Alb17]{Al17}
Immanuel Albrecht.
\newblock {On Finding Some New Excluded Minors for Gammoids}.
\newblock {\em Electronic Notes in Discrete Mathematics}, 61:37--43, 2017.

\bibitem[{Ard}06]{Ar06}
F.~{Ardila}.
\newblock {Transversal and cotransversal matroids via the Lindstrom lemma}.
\newblock {\em ArXiv Mathematics e-prints}, May 2006.

\bibitem[Bar68]{Ba68}
Erwin~H. Bareiss.
\newblock {Sylvester's Identity and Multistep Integer-Preserving Gaussian
  Elimination}.
\newblock {\em Mathematics of Computation}, 22(103):565--578, 1968.

\bibitem[BdM06]{Bonin2006701}
Joseph~E. Bonin and Anna de~Mier.
\newblock {Lattice path matroids: Structural properties}.
\newblock {\em European Journal of Combinatorics}, 27(5):701 -- 738, 2006.

\bibitem[Bir67]{Bi67}
Garrett Birkhoff.
\newblock {\em {Lattice Theory}}.
\newblock American Mathematical Society, Providence, 3rd edition, 1967.

\bibitem[BJG09]{BJG09}
Jørgen Bang-Jensen and Gregory Gutin.
\newblock {\em {Digraphs: Theory, Algorithms and Applications}}.
\newblock Springer, London, 2nd edition, 2009.

\bibitem[BLS{\etalchar{+}}99]{BlVSWZ99}
A.~Björner, M.~LasVergnas, B.~Sturmfels, N.~White, and G.~Ziegler.
\newblock {\em Oriented Matroids}.
\newblock Cambridge University Press, second edition, 1999.

\bibitem[Bon]{joeP8}
Joseph~E. Bonin.
\newblock {$P_8^=$ is not a gammoid}.
\newblock \url{http://matroidunion.org/?p=2044}.

\bibitem[Bry71]{Brylawski1971}
Thomas~H. Brylawski.
\newblock A combinatorial model for series-parallel networks.
\newblock {\em Transactions of the American Mathematical Society}, 154:1--22,
  1971.

\bibitem[Bry75]{Brylawski1975}
Thomas~H. Brylawski.
\newblock {An Affine Representation for Transversal Geometries}.
\newblock {\em Studies in Applied Mathematics}, 54(2):143--160, 1975.

\bibitem[BV78]{BlV78}
Robert~G Bland and Michel~Las Vergnas.
\newblock Orientability of matroids.
\newblock {\em Journal of Combinatorial Theory, Series B}, 24(1):94--123, 1978.

\bibitem[{Cam}14]{Ca14Msc}
A.~{Cameron}.
\newblock {Kinser inequalities and related matroids}.
\newblock {\em ArXiv e-prints}, January 2014.

\bibitem[CCW85]{CHANDRU198575}
Vijaya Chandru, Collette~R Coullard, and Donald~K Wagner.
\newblock On the complexity of recognizing a class of generalized networks.
\newblock {\em Operations Research Letters}, 4(2):75 -- 78, 1985.

\bibitem[CR70]{CR70}
Henry Crapo and Gian-Carlo Rota.
\newblock {\em On the Foundations of Combinatorial Theory: Combinatorial
  Geometries}.
\newblock M.I.T. Press, Cambridge, preliminary edition, 1970.

\bibitem[Cra65]{C65}
H.H. Crapo.
\newblock Single-element extensions of matroids.
\newblock {\em J. Res. Nat. Bur. Standards Sect. B}, 69B(1--2):55--65, 1965.

\bibitem[Dun76]{Dun76}
F.D.J. Dunstan.
\newblock {Matroids and Submodular Functions}.
\newblock {\em The Quarterly Journal of Mathematics}, 27(3):339--348, 1976.

\bibitem[Fin]{OM1}
Lukas Finschi.
\newblock {Homepage of Oriented Matroids}.
\newblock \url{http://www.om.math.ethz.ch/}.

\bibitem[FL78]{FoLa78}
Jon Folkman and Jim Lawrence.
\newblock Oriented matroids.
\newblock {\em Journal of Combinatorial Theory, Series B}, 25(2):199--236,
  1978.

\bibitem[GCG11]{GCG11}
L.~Guille, T.~Chan, and A.~Grant.
\newblock {The Minimal Set of Ingleton Inequalities}.
\newblock {\em IEEE Transactions on Information Theory}, 57(4):1849--1864,
  April 2011.

\bibitem[GGW14]{GGW14}
Jim Geelen, Bert Gerards, and Geoff Whittle.
\newblock {Solving Rota's Conjecture}.
\newblock {\em Notices of the AMS 61}, 61:736--743, 2014.

\bibitem[GHN16]{GoHoNe15}
Luis Goddyn, Winfried Hochst{\"a}ttler, and Nancy~Ann Neudauer.
\newblock {\em Discrete Mathematics}, 339(5):1425--1429, 2016.

\bibitem[G{\"o}r00]{Go00}
F.~G{\"o}ring.
\newblock {Short proof of Menger's Theorem}.
\newblock {\em Discrete Mathematics}, 219(1):295--296, 2000.

\bibitem[GTZ98]{GTZ98}
Luis~A. Goddyn, Michael Tarsi, and Cun-Quan Zhang.
\newblock {On $(k, d)$-Colorings and Fractional Nowhere-Zero Flows}.
\newblock {\em Journal of Graph Theory}, 28(3):155--161, 1998.

\bibitem[GV89]{GV89p}
Ira~M. Gessel and X.~G. Viennot.
\newblock {Determinants, Paths, and Plane Partitions}, 1989.

\bibitem[HN06]{HN06}
Winfried Hochstättler and Jaroslav Nešetřil.
\newblock {Antisymmetric Flows in Matroids}.
\newblock {\em European Journal of Combinatorics}, 27(7):1129 -- 1134, 2006.
\newblock Eurocomb ’03 - Graphs and Combinatorial Structures.

\bibitem[HN08]{HN08}
Winfried Hochstättler and Robert Nickel.
\newblock On the chromatic number of an oriented matroid.
\newblock {\em Journal of Combinatorial Theory, Series B}, 98(4):698 -- 706,
  2008.

\bibitem[Ing71a]{Ingleton1971}
A.~W. Ingleton.
\newblock {A Geometrical Characterization of Transversal Independence
  Structures}.
\newblock {\em Bulletin of the London Mathematical Society}, 3(1):47--51, 1971.

\bibitem[Ing71b]{Ingleton69}
A.W. Ingleton.
\newblock {Representation of Matroids}.
\newblock In D.J.A. Welsh, editor, {\em Combinatorial Mathematics and its
  Applications}, pages 149--167, Proc. Conf. math. Inst., Oxford 1969, 1971.
  Academic Press.

\bibitem[Ing77]{In77}
A.W. Ingleton.
\newblock {\em {Transversal Matroids and Related Structures}}, pages 117--131.
\newblock Springer Netherlands, Dordrecht, 1977.

\bibitem[IP73]{IP73}
A.W. Ingleton and M.J. Piff.
\newblock Gammoids and transversal matroids.
\newblock {\em Journal of Combinatorial Theory, Series B}, 15(1):51--68, 1973.

\bibitem[KW12]{KW12}
Stefan Kratsch and Magnus Wahlstr{\"o}m.
\newblock {Representative Sets and Irrelevant Vertices: New Tools for
  Kernelization}.
\newblock In {\em Proceedings of the 2012 IEEE 53rd Annual Symposium on
  Foundations of Computer Science}, FOCS '12, pages 450--459, Washington, DC,
  USA, 2012. IEEE Computer Society.

\bibitem[Lin73]{Li73}
Bernt Lindström.
\newblock On the vector representations of induced matroids.
\newblock {\em Bull. London Math. Soc}, (5), 1973.

\bibitem[Mar09]{Marx09}
Dániel Marx.
\newblock A parameterized view on matroid optimization problems.
\newblock {\em Theoretical Computer Science}, 410(44):4471 -- 4479, 2009.
\newblock Automata, Languages and Programming (ICALP 2006).

\bibitem[Mas70]{Ma70thesis}
J.H. Mason.
\newblock {\em Representations of independence spaces}.
\newblock University of Wisconsin--Madison, 1970.

\bibitem[Mas72]{M72}
J.H. Mason.
\newblock On a class of matroids arising from paths in graphs.
\newblock {\em Proceedings of the London Mathematical Society}, 3(1):55--74,
  1972.

\bibitem[Mat77]{Ma77}
Laurence~R. Matthews.
\newblock {Bicircular Matroids}.
\newblock {\em The Quarterly Journal of Mathematics}, 28(2):213--227, 1977.

\bibitem[{May}16]{Ma16}
D.~{Mayhew}.
\newblock {The antichain of excluded minors for the class of gammoids is
  maximal}.
\newblock {\em ArXiv e-prints}, January 2016.

\bibitem[McD72]{McD72}
Colin McDiarmid.
\newblock Strict gammoids and rank functions.
\newblock {\em Bulletin of the London Mathematical Society}, 4(2):196--198,
  1972.

\bibitem[Men27]{Me27}
Karl Menger.
\newblock Zur allgemeinen {Kurventheorie}.
\newblock {\em Fundamenta Mathematicae}, 10:96--115, 1927.

\bibitem[Min62]{Mi62}
George~J. Minty.
\newblock {A Theorem on n-Coloring the Points of a Linear Graph}.
\newblock {\em The American Mathematical Monthly}, 69(7):623--624, 1962.

\bibitem[MMF]{OM2}
Hiroyuki Miyata, Sonoko Moriyama, and Komei Fukuda.
\newblock {Classification of Oriented Matroids}.
\newblock \url{http://www-imai.is.s.u-tokyo.ac.jp/~hmiyata/oriented_matroids/}.

\bibitem[MNW09]{MNW09}
Dillon Mayhew, Mike Newman, and Geoff Whittle.
\newblock On excluded minors for real-representability.
\newblock {\em Journal of Combinatorial Theory, Series B}, 99(4):685--689,
  2009.

\bibitem[Nic12]{Ni12}
Robert Nickel.
\newblock {\em {Flows and Colorings in Oriented Matroids}}.
\newblock PhD thesis, FernUniversität in Hagen, 2012.

\bibitem[NvdP17]{NvdP17}
P.~{Nelson} and J.~van~der Pol.
\newblock {Doubly exponentially many Ingleton matroids}.
\newblock {\em ArXiv e-prints}, October 2017.

\bibitem[{OEI}]{OEIS}
{OEIS Foundation Inc.}
\newblock {The On-Line Encyclopedia of Integer Sequences}.
\newblock \url{http://oeis.org/A300985}.

\bibitem[Oxl11]{Ox11}
James Oxley.
\newblock {\em Matroid theory}, volume~21 of {\em Oxford Graduate Texts in
  Mathematics}.
\newblock Oxford University Press, Oxford, second edition, 2011.

\bibitem[PvdP15]{PvdP15}
Rudi Pendavingh and Jorn van~der Pol.
\newblock {On the Number of Matroids Compared to the Number of Sparse Paving
  Matroids}.
\newblock {\em The Electronic Journal of Combinatorics}, 22(1):\#P2.51, 2015.

\bibitem[PvdP17]{PP17}
Rudi Pendavingh and Jorn van~der Pol.
\newblock Enumerating matroids of fixed rank.
\newblock {\em The Electronic Journal of Combinatorics}, 24(1):\#P1.8, 2017.

\bibitem[Rot64]{Ro64}
Gian-Carlo Rota.
\newblock {On the Foundations of Combinatorial Theory I. Theory of M{\"o}bius
  Functions}.
\newblock {\em Z. Wahrscheinlichkeitstheorie}, (2):340--368, 1964.

\bibitem[Sch80]{Sch80}
J.~T. Schwartz.
\newblock {Fast Probabilistic Algorithms for Verification of Polynomial
  Identities}.
\newblock {\em J. ACM}, 27(4):701--717, October 1980.

\bibitem[SS10]{SS10}
Paul Seymour and Blair~D. Sullivan.
\newblock Counting paths in digraphs.
\newblock {\em European Journal of Combinatorics}, 31(3):961--975, 2010.

\bibitem[Wel71]{We71}
D.J.A. Welsh.
\newblock Generalized versions of {Hall's} theorem.
\newblock {\em Journal of Combinatorial Theory, Series B}, 10(2):95--101, 1971.

\bibitem[Wel76]{We76}
D.J.A. Welsh.
\newblock {\em Matroid theory}.
\newblock Academic Press Inc., London, 1976.

\bibitem[Zip79]{Zi79}
Richard Zippel.
\newblock Probabilistic algorithms for sparse polynomials.
\newblock In Edward~W. Ng, editor, {\em Symbolic and Algebraic Computation},
  pages 216--226, Berlin, Heidelberg, 1979. Springer Berlin Heidelberg.

\end{thebibliography}

% If you would like to use BibLaTeX for your references, pass `custombib' as
% an option in the document class. The location of 'reference.bib' should be
% specified in the preamble.tex file in the custombib section.
% Comment out the lines related to natbib above and uncomment the following line.

%\printbibliography[heading=bibintoc, title={References}]

\end{spacing}

\cleardoublepage
% -*- root: ../thesis.tex -*-

\newcommand{\hrf}{~\hrulefill}
\stepcounter{chapter}
\setcounter{section}{0}
\chapter*{Index of Symbols and Notation}
\fancyhead[RE]{Index of Symbols and Notation}
\fancyhead[LO]{Index of Symbols and Notation}
\addcontentsline{toc}{chapter}{Index of Symbols and Notation}
%\begin{longtable}[l]{p{4cm}l}
\begin{tabularx}{\textwidth}{p{4.4cm}X}
%ONLY SYMBOLS
$\dSET{\ldots}$ \hrf & set consisting of distinct elements given in list, p.\pageref{n:dset}.\\
$\subseteq_\sigma$ \hrf & denotes the signed subset relation, p.\pageref{n:signedsubset}.\\
$\emptyset_{\sigma E}$ \hrf & denotes the empty signed subset of $E$, p.\pageref{n:emptysigma}.\\
$\emptyset_{\Z . E}$ \hrf & denotes the empty signed multi-subset of $E$, p.\pageref{n:emptysetZE}.\\
$\DclD{\bullet}{D}$ \hrf & outer-extension operator in $D$, p.\pageref{n:DclD}.\\
$\partial_D\bullet$ \hrf & outer-margin operator in $D$, p.\pageref{n:partD}.\\
$\prod p$ \hrf & product of the weights of all traversed arcs of $p$, p.\pageref{n:prodp}.\\ 
%$\ll \,\,\subseteq A\times A$ \hrf& linear order on the arcs of $D=(V,A)$, p.\pageref{n:ll}.\\
$\llless $ \hrf& $(\sigma,\ll)$-induced order on the routings of $D$, p.\pageref{n:llless}.\\
%GREEK-LETTERS
$\alpha_M\colon 2^E \maparrow \Z$ \hrf & $\alpha$-invariant of $M$, p.\pageref{n:alphaM}.\\
$\left(\Alpha_M, \sqsubseteq_M\right)$ \hrf & $\alpha_M$-poset, p.\pageref{n:alphaposet}.\\
$\left(\Beta_M^C, \sqsubseteq_M^C \right)$ \hrf & extension poset of $C \in \Mcal(M)$, p.\pageref{n:BetaMC}.\\
$\Delta\alpha_M$ \hrf& $\Delta\alpha$-invariant of $M$, p.\pageref{n:Deltaalphainvariant}.\\
$\DeltaP\alpha_M$ \hrf& $\DeltaP\alpha$-invariant of $M$, p.\pageref{n:DeltaPalphainvariant}.\\
$\delta_D(X,T)$ \hrf & barrier between $X$ and $T$ in $D$, p.\pageref{n:barrier}.\\
$\Gamma(D,T,E)$ \hrf & gammoid represented by $(D,T,E)$, p.\pageref{n:GTDE}.\\
$\Gamma(D,T,V)$ \hrf & strict gammoid, p.\pageref{n:GDTV}.\\
$\Gamma_\Mcal$ \hrf& class map for recognizing gammoids in $\Mcal$, p.\pageref{n:GammaM}.\\
$\mu \in K^{R\times C}$ \hrf & $R\times C$-matrix over $K$, p.\pageref{n:matrix}.\\
$\mu^\top \in K^{C\times R}$ \hrf & transpose of $\mu$, p.\pageref{n:transposed}.\\
$\mu^\top_c$ \hrf & $c$-th column of $\mu$, p.\pageref{n:matrix}.\\
$\mu_{(P,\leq)}$ \hrf & Möbius-function of the poset $(P,\leq)$, p.\pageref{n:moebius}.\\
$\mu_r$ \hrf & $r$-th row of $\mu$, p.\pageref{n:matrix}.\\
$\mu\restrict R_0$ \hrf & restriction of $\mu$ to the rows $R_0$, p.\pageref{n:matrestrict}.\\
$\mu\restrict R_0\times C_0$ \hrf & restriction of $\mu$ to the rows $R_0$ and columns $C_0$, p.\pageref{n:matrestrict}.\\
%$\mu@\left(x_i,\phi(x_i)\right)_{i=1}^n$ \hrf & $(X,\phi)$-pivot of $\mu\in R^{E\times T}$, p.\pageref{n:Xphi-pivot}.\\
%$\sigma = (D,T,E,S_A,\ll)$ \hrf & realizable combinatorial orientation of a gammoid, p.\pageref{n:rCOG}.\\
$(\sigma,\ll)$ \hrf & heavy arc signature of $D=(V,A)$, p.\pageref{n:sigmaLL}.\\
$\sigma E$ \hrf & class of signed subsets of $E$, p.\pageref{n:signedsubset}.\\
$\zeta_{(P,\leq)}$ \hrf & zeta-matrix of the poset $(P,\leq)$, p.\pageref{n:zeta}.\\
%NORMAL-LETTERS
$\bigcap \Acal $ \hrf & denotes the intersection $\bigcap_{A\in\Acal} A$, p.\pageref{n:bigcap}.\\
$\bigcup \Acal $ \hrf & denotes the union $\bigcup_{A\in\Acal} A$, p.\pageref{n:bigcup}.\\
$\Acal=(A_i)_{i\in I}\subseteq E $ \hrf & family of subsets of $E$ indexed by $I$, p.\pageref{n:Afam}.\\
$\Acal_{\Delta} = (A_i)_{i\in B} \subseteq A$ \hrf & arc system of $(A\disunion B,\Delta)$, p.\pageref{n:ArcSystem}.\\
$\Acal_{D,T} = (A^{(D,T)}_i)_{i\in V\BS T}$ \hrf& linkage system of $D$ to $T$, p.\pageref{n:ADT}.\\
$\mathrm{AC}(D) = (V_{D},A_{D})$ \hrf & denotes the arc-cut digraph for $D$, p.\pageref{n:arcCutDigraph}.\\
$\mathrm{AC}(D,T,E)$ \hrf & denotes the arc-cut matroid for $(D,T,E)$, p.\pageref{n:ACDTE}.\\
$\Acal_M = (A_i)_{i\in I} \subseteq E$ \hrf& denotes the $\alpha$-system of $M=(E,\Ical)$, p.\pageref{n:alphaSystem}.\\
{\em (B1)} \hrf & (axiom) a base of $M$ exists, p.\pageref{n:Bx}.\\
{\em (B2)} \hrf & (axiom) equicardinality of bases, p.\pageref{n:Bx}.\\
{\em (B3)} \hrf & (axiom) base exchange, p.\pageref{n:Bx}.\\
{\em (B3')} \hrf & (axiom) strong base exchange, p.\pageref{n:B3p}.\\
$(B,\rho)$ \hrf & denotes an $M$-black box, p.\pageref{n:Brho}.\\
$\bits(M) = (\bits(M,i))_{i=1}^{N}$ \hrf & binary encoding of $M$, p.\pageref{n:bMenc}.\\
$\Bcal(M)$ \hrf & family of all bases of $M=(E,\Ical)$, p.\pageref{n:BcalM}.\\
$\Bcal_M(F)$ \hrf & family of all bases of $F$ in $M$, p.\pageref{n:BcalMF}.\\
{\em ($\Ccal$1)} \hrf & (o.m. axiom) $\emptyset_{\sigma E}$ is not a circuit, p.\pageref{n:Cx}.\\
{\em ($\Ccal$2)} \hrf & (o.m. axiom) circuits closed under negation, p.\pageref{n:Cx}.\\
{\em ($\Ccal$3)} \hrf & (o.m. axiom) incomparability of circuits, p.\pageref{n:Cx}.\\
{\em ($\Ccal$4)} \hrf & (o.m. axiom) strong circuit elimination, p.\pageref{n:Cx}.\\
{\em ($\Ccal$4')} \hrf & (o.m. axiom) weak circuit elimination, p.\pageref{n:Ccal4p}.\\
{\em ($\Ccal^\ast$1)} \hrf & (o.m. axiom) $\emptyset_{\sigma E}$ is not a cocircuit, p.\pageref{n:Cx}.\\
{\em ($\Ccal^\ast$2)} \hrf & (o.m. axiom) cocircuits closed under negation, p.\pageref{n:Cx}.\\
{\em ($\Ccal^\ast$3)} \hrf & (o.m. axiom) incomparability of cocircuits, p.\pageref{n:Cx}.\\
{\em ($\Ccal^\ast$4)} \hrf & (o.m. axiom) strong cocircuit elimination, p.\pageref{n:Cx}.\\
%$C_{\sigma c}$ \hrf& $\sigma$-signed circuit $C$ solved for $c$, p.\pageref{n:Csigmac}.\\
$\arcC(M)$ \hrf & arc-complexity of the gammoid $M$, p.\pageref{n:ArcCompl}.\\
$C\bot D$\hrf& orthogonality of signed subsets $C,D\in \sigma E$, p.\pageref{n:XorthoY}.\\
%$C\circ D$ \hrf & denotes the merger of $C,D\in \sigma E$, p.\pageref{n:merger}.\\
%$C+ \epsilon D$ \hrf & denotes the composition of $C\in\sigma E$ with $D\in\sigma E$, p.\pageref{n:merger}.\\
$\cl_M $ \hrf & closure operator of $M$, p.\pageref{n:clM}.\\
$\Ccal(M)$ \hrf & circuit set of the matroid $M$, p.\pageref{n:CM}.\\
$\vK(M)$ \hrf & vertex-complexity of the gammoid $M$, p.\pageref{n:VertexCompl}.\\
$C_{-X} \in \sigma E$ \hrf & the $X$-flip of $C\in\sigma E$, p.\pageref{n:Xflip}.\\
$D=(A\disunion B, \Delta)$ \hrf & directed bipartite graph for $\Delta$ from $A$ to $B$, p.\pageref{n:DABD}.\\
$D=(V,A)$ \hrf & directed graph, p.\pageref{n:DVA}.\\
$D^{\opp}=(V,A^\opp)$ \hrf & opposite digraph, p.\pageref{n:Dopp}.\\
$D_{r\leftarrow s} = (V,A_{r\leftarrow s})$ \hrf & $r$-$s$-pivot of the digraph $D=(V,A)$, p.\pageref{n:digraphpivot}.\\
%$D_{s\rightarrow r} = (V,A_{s\rightarrow r})$ \hrf & $s$-$r$-dual-pivot of the digraph $D=(V,A)$, p.\pageref{n:digraphdualpivot}.\\
%$(D,T,E,S_A,\ll)$ \hrf & realizable combinatorial orientation of a gammoid, p.\pageref{n:rCOG}.\\
%$ \eval_{\Sbm} (p)$ \hrf & evaluation map of $p\in \R[X]$ with respect to $\Sbm \supseteq \R$, p.\pageref{n:evalSP}.\\
$\Fcal(\alpha)$ \hrf & family of $\alpha$-flats for $\alpha\colon 2^E\maparrow \Z$, p.\pageref{n:FcalAlpha}.\\
$\Fcal(M)$ \hrf & family of all flats of $M$, p.\pageref{n:FM}.\\
$\Fcal(M,X)$ \hrf & flats of $M$ that are proper subsets of $X$, p.\pageref{n:FcalMX}.\\
%$\Fcal_{\pm}(M,X)$ \hrf & family of all non-vanishing flats below $X$, p.\pageref{n:nonvanishing}.\\
%$\Fcal_+(M,X)$ \hrf & family of all hot flats below $X$ in $M$, p.\pageref{n:hotflats}.\\
$f[X]$ \hrf & set of images of $x\in X$ under $f$, p.\pageref{n:fsquareX}.\\
$f\restrict_{X'}$ \hrf & restriction of the map $f\colon X\maparrow Y$ to $X'\subseteq X$, p.\pageref{n:frestrictX'}.\\
%$(E,\Ical)$ \hrf & (independence) matroid, p.\pageref{n:EI}.\\
$\Ical_\alpha$ \hrf & zero-family of $\alpha\colon 2^E \maparrow \Z$, p.\pageref{n:Ialpha}.\\
{\em (I1)} \hrf & (axiom) $\emptyset$ is independent, p.\pageref{n:Is}.\\
{\em (I2)} \hrf & (axiom) independence carries over to subsets, p.\pageref{n:Is}.\\
{\em (I3)} \hrf & (axiom) augmentation of independent sets, p.\pageref{n:Is}.\\
$\idet \mu$ \hrf & determinant-indicator of $\mu$, p.\pageref{n:idet}.\\
$I(D,T,E)$ \hrf & matroid on $E$ induced by $D$ from $T=(T_0,\Tcal)$, p.\pageref{n:IDTE}.\\
%$\charW(f)$ \hrf & characteristic of $\arcW_f$, p.\pageref{n:kWf}.\\
$K^{m\times n}$ \hrf & class of all $m\times n$-matrices over $K$, p.\pageref{n:matrix}.\\
$K^{R\times C}$ \hrf & class of all $R\times C$-matrices over $K$, p.\pageref{n:matrix}.\\
$\kth(n,r)$ \hrf & bijection that enumerates all $r$-element subsets of $\SET{1,2,\ldots,n}$, p.\pageref{n:kth}.\\
$M(\alpha) = (E,\Ical_\alpha)$\hrf & matroid corresponding to the matroid invariant $\alpha$, p.\pageref{n:MAlpha}.\\
$M(\mu)$ \hrf & matroid on $E$ represented by $\mu\in \K^{E\times C}$, p.\pageref{n:matMmu}\\
$M(\Acal) = (E,\Ical_\Acal)$\hrf & transversal matroid presented by $\Acal$, p.\pageref{n:MAEIA}.\\
$M\contract C$ \hrf & contraction of $M$ to $C$, p.\pageref{n:MC}.\\
$M(\Delta,M_0)$ \hrf & matroid induced by $\Delta\subseteq D\times E$ from $M_0$, p.\pageref{n:MDM0}.\\
$M=(E,\Ical)$ \hrf & (independence) matroid, p.\pageref{n:EI}.\\
$M^\ast=(E,\Ical^\ast)$ \hrf & dual matroid of $M$, p.\pageref{n:Mdual}.\\
$\Mcal(M)$ \hrf & class of all modular cuts of $M$, p.\pageref{n:MM}.\\
%$\Mcal(M,E')$\hrf& class of all $\left| E' \right|$-fold modular cuts of $M$, p.\pageref{n:McalMEp}.\\
$M(K_4)$ \hrf & polygon matroid of the complete graph on $4$ vertices, p.\pageref{ex:MK4}.\\
$M(\Ocal)$ \hrf & underlying matroid of the oriented matroid $\Ocal$, p.\pageref{n:MOcal}.\\
$M\restrict R$ \hrf & restriction of $M$ to $R$, p.\pageref{n:MR}.\\
$\Nbf(M)$ \hrf & encoding length of $M$, p.\pageref{n:encM}.\\
$\N^X$ \hrf & multi-sets over $X$, p.\pageref{n:multiset}.\\
$\N^{(X)}$ \hrf & finite multi-sets over $X$, p.\pageref{n:fin-multiset}.\\
%$O_\alpha$ \hrf & $\alpha$-zero flat for $\alpha\colon 2^E\maparrow \Z$, p.\pageref{n:Oalpha}.\\
{\em ($\Ocal$1)} \hrf & (o.m. axiom) orthogonality, p.\pageref{n:Cx}.\\
{\em ($\Ocal$2)} \hrf & (o.m. axiom) underlying matroid, p.\pageref{n:Cx}.\\
$\Ocal=(E,\Ccal,\Ccal^\ast)$ \hrf & oriented matroid, p.\pageref{n:Ocal}.\\
$\Ocal^\ast=(E,\Ccal^\ast,\Ccal)$ \hrf & dual oriented matroid of $\Ocal$, p.\pageref{n:OcalBot}.\\
$[\Ocal]$ \hrf & reorientation class of $\Ocal$, p.\pageref{n:reorientationclass}.\\
$\Ocal(\mu)= (E,\Ccal_\mu,\Ccal_\mu^\ast)$ \hrf & oriented matroid represented by $\mu\in \R^{E\times C}$, p.\pageref{n:OcalMu}.\\
$\Ocal\restrict R$ \hrf & restriction of $\Ocal$ to $R$, p.\pageref{n:OrestrictR}.\\
$\Ocal\contract Q$ \hrf & contraction of $\Ocal$ to $Q$, p.\pageref{n:OcontractQ}.\\
$(p_{i})_{i=1}^{n}\in\{\mathrm{N},\mathrm{E}\}^{n}$ \hrf & lattice path on an $\SET{\mathrm{N},\mathrm{E}}$-grid, p.\pageref{n:latticePath}.\\
 $\left| p \right|$ \hrf & set of vertices visited by $p$,  p.\pageref{n:walk}.\\
 $\left| p \right|_A$ \hrf & set of arcs traversed by $p$,  p.\pageref{n:walk}.\\
 $\Pbf(D)$ \hrf & set of paths in $D$, p.\pageref{n:simplePath}.\\
 $\Pbf(D; u,v)$ \hrf & set of paths from $u$ to $v$ in $D$, p.\pageref{n:SPathUV}.\\
 $p\preceq q$ \hrf &  $p$ is never above $q$ with common endpoints, p.\pageref{n:neverabove}.\\
 $\mathrm{P}\left[p,q\right]$ \hrf & lattice walks between $p$ and $q$, p.\pageref{n:LPbetweenPQ}.\\
 $\downarrow_{(P,\leq)} y$ \hrf & the down-set of $y\in P$ with respect to the poset $(P,\leq)$, p.\pageref{n:Pdownset}.\\
{\em (R1')} \hrf & (axiom) $\rk(\emptyset) = 0$, p.\pageref{n:Rxp}.\\
{\em (R2')} \hrf & (axiom) $\rk$ is unit-increasing, p.\pageref{n:Rxp}.\\
{\em (R3')} \hrf & (axiom) if two points are dependent, so is their line, p.\pageref{n:Rxp}.\\
{\em (R2'')} \hrf & unit-increment propagates to subsets, p.\pageref{n:R2pp}.\\
{\em (R1)} \hrf & (axiom) $\rk$ is non-negative and subcardinal, p.\pageref{n:Rx}.\\
{\em (R2)} \hrf & (axiom) $\rk$ is non-decreasing, p.\pageref{n:Rx}.\\
{\em (R3)} \hrf & (axiom) $\rk$ is submodular, p.\pageref{n:Rx}.\\
{\em (R4)} \hrf & there is a unique maximal superset of same rank, p.\pageref{n:R4}.\\
{\em (R5)} \hrf & there are rank-cardinality independent subsets, p.\pageref{n:R5}.\\
$\mathrm{Rec}\Gamma_\Mcal$ \hrf & gammoid recognition problem for $\Mcal$, p.\pageref{n:RecGM}.\\
$\rk_M $ \hrf & rank function of $M$, p.\pageref{n:rkM}.\\
$\R[X]$ \hrf & polynomials over commutative $\R$ with variables $X$, p.\pageref{n:polynomring}.\\ 
$R\colon X\routesto Y \subseteq \Wbf(D)$ \hrf & routing from $X$ to $Y$ in $D$, p.\pageref{n:routing}.\\
$\sep(C,D)$ \hrf & separator of signed subsets $C,D\in \sigma E$, p.\pageref{n:sep}.\\
$\sgn_\sigma (R)$ \hrf & sign of routing w.r.t. $(\sigma,\ll)$, p.\pageref{n:sgnsigma}.\\
$\Tbf = (G,\Gcal,\Mcal,\Xcal,\simeq)$ \hrf & matroid tableau, p.\pageref{n:mattab}.\\
$\bigcup_{i=1}^n \Tbf_i$ \hrf & joint matroid tableau, p.\pageref{n:jointTableau}.\\
$[\Tbf]_\simeq$ \hrf & expansion matroid tableau, p.\pageref{n:expTab}.\\
$[\Tbf]_\equiv$ \hrf & extended matroid tableau, p.\pageref{n:extTab}.\\
$\Tbf!$ \hrf & conclusion matroid tableau, p.\pageref{n:concTab}.\\
$\Tbf(M_1\simeq M_2)$ \hrf & identified matroid tableau, p.\pageref{n:idTab}.\\
$w=(w_i)_{i=1}^n \in
    V^{n}$ \hrf& walk in $D=(V,A)$, p.\pageref{n:walk}.\\
$\Wbf(D)$ \hrf & set of walks in $D$, p.\pageref{n:PbfD}.\\
$\Wbf(D; u,v)$ \hrf & set of walks from $u$ to $v$ in $D$, p.\pageref{n:PathUV}.\\
$\arcW_f(M)$ \hrf & $f$-width of the gammoid $M$, p.\pageref{n:arcWfM}.\\
 $(w_1w_2\ldots w_n)^i$ \hrf & shorthand for $i$-iterations of $w_1w_2\ldots w_n\in V^n$, p.\pageref{n:walk}.\\ 
 $w.q$ \hrf & concatenated walk $w_{1}w_{2}\ldots w_{n}q_{2}q_{3}\ldots q_{m}$,  p.\pageref{n:pdotq}.\\
$X\colon E\maparrow \SET{-1,0,1}$\hrf & signed subset of $E$, p.\pageref{n:signedsubset}.\\
$X_+$\hrf & positive elements of $X$, p.\pageref{n:xplus}.\\
$X_-$\hrf & negative elements of $X$, p.\pageref{n:xminus}.\\
$X_\pm$\hrf & support of $X$, p.\pageref{n:xpm}.\\
$X_0$\hrf & zero-set of $X$, p.\pageref{n:xzero}.\\
$-X$\hrf & negation of $X$, p.\pageref{n:minusx}.\\
$\Xcal(M,e)$ \hrf & class of single-element extensions of $M$ by $e$, p.\pageref{n:XMe}.\\
%$\Xcal(M,E')$ \hrf & class of $\left| E' \right|$-fold extensions of $M$ by $E'$, p.\pageref{n:XcalMEp}.\\
$\Vcal(M)$\hrf & family of all $\alpha_M$-violations, p.\pageref{n:VM}.\\
$\Z .E$ \hrf & family of signed multisets $S\colon E\maparrow \Z$, p.\pageref{n:ZE}.\\
%\studyremark{$\Zcal(M)$ \hrf & family of all cyclic flats of $M$, p.\pageref{n:ZM}.\\}
%\end{longtable}
\end{tabularx}
%Dummy:
%$ $ \hrf & , p.\pageref{n:}.\\

% ********************************** Appendices ********************************

%\begin{appendices} % Using appendices environment for more functunality
%
%\include{Appendix1/appendix1}
%\include{Appendix2/appendix2}
%
%\end{appendices}

\cleardoublepage
\fancyhead[RE]{Index}
\fancyhead[LO]{Index}

% *************************************** Index ********************************
\printthesisindex % If index is present

\cleardoublepage
\thispagestyle{empty}
~~

\end{document}